\documentclass{article}
\usepackage[textwidth=18cm,textheight=23cm,centering]{geometry}
\usepackage[all]{xy}
\SelectTips{cm}{}
\usepackage{amsthm, amsmath, bbold, stmaryrd, MnSymbol, amsxtra}
\usepackage{enumerate}
\usepackage{xspace}

\newdir{ >}{{}*!/-5pt/\dir{>}}
\newdir{ (}{{}*!/-10pt/\dir^{(}}

\theoremstyle{plain}
\newtheorem{theorem}{Theorem}[subsection]
\newtheorem{lemma}[theorem]{Lemma}
\newtheorem{corollary}[theorem]{Corollary}
\newtheorem{proposition}[theorem]{Proposition}

\theoremstyle{definition}
\newtheorem{definition}[theorem]{Definition}

\newtheorem{example}[theorem]{Example}

\newtheorem{note}[theorem]{Note}

\newtheorem{notation}[theorem]{Notation}

\def\epsilon{\varepsilon}
\def\At{\mathsf{At}}

\renewcommand{\rho}{\varrho}

\newcommand{\takeout}[1]{\empty}

\newcommand{\SmallCat}{\mathsf{Cat}}

\newcommand{\Set}{\mathsf{Set}}

\newcommand{\BA}{\mathsf{BA}}

\newcommand{\DL}{\mathsf{DL}}

\newcommand{\Poset}{\mathsf{Poset}}

\newcommand{\Cat}{\mathcal{C}}
\newcommand{\DCat}{\mathcal{D}}

\newcommand{\xto}[1]{\xrightarrow{#1}}
\newcommand{\hookto}{\hookrightarrow}

\newcommand{\epito}{\twoheadrightarrow}

\newcommand{\up}{\uparrow}

\newcommand{\down}{\downarrow}

\newcommand{\monoto}{\rightarrowtail}

\newcommand{\R}{\mathcal{R}}

\newcommand{\Pred}{\mathsf{Pred}}

\newcommand{\Up}{\mathsf{Up}}

\newcommand{\Nat}{\mathbb{N}}

\newcommand{\two}{\mathbb{2}}

\newcommand{\three}{\mathbb{3}}

\newcommand{\ang}[1]{\langle #1 \rangle}
\newcommand{\sem}[1]{\llbracket #1 \rrbracket}
\newcommand{\Land}{\bigwedge}
\newcommand{\Lor}{\bigvee}
\newcommand{\To}{\Rightarrow}
\newcommand{\oT}{\Leftarrow}

\usepackage{mathrsfs}
\usepackage[draft]{fixme} 

\usepackage{xcolor}
\usepackage{hyperref}
\hypersetup{
  colorlinks,
  linkcolor={red!50!black},
  citecolor={blue!50!black},
  urlcolor={blue!80!black},
  hypertexnames=false
}

\CompileMatrices

\newtheorem{motivation}[theorem]{Motivation}

\makeatletter
\newcommand{\xTo}[2][]{\ext@arrow 0359\Rightarrowfill@{#1}{#2}}
\makeatother

\newcommand{\DeMc}[1]{\mathbb{DM}(#1)}

\newcommand{\Ch}[1]{\mathtt{C}_{#1}}
\newcommand{\jslCh}[1]{\mathbb{C}_{#1}}
\newcommand{\dmCh}[1]{\mathbb{C}_{#1}^\neg}

\newcommand{\Diagj}{\mathtt{Diag}_j}
\newcommand{\Diagm}{\mathtt{Diag}_m}

\newcommand{\eqnOR}{(\mathrm{Rev}\sigma)}
\newcommand{\eqnEx}{(\mathrm{Ex}\sigma^2)}
\newcommand{\eqnCEx}{(\mathrm{Cx}\sigma^2)}
\newcommand{\eqnInv}{(\mathrm{Inv}\sigma)}

\newcommand{\ctos}[1]{\mathrm{c2s}_{#1}}
\newcommand{\stoc}[1]{\mathrm{s2c}_{#1}}

\newcommand{\trep}[1]{\mathrm{trep}_{#1}}

\newcommand{\siff}{\Leftrightarrow}

\newcommand{\upsp}{\upfilledspoon}
\newcommand{\downsp}{\downfilledspoon}

\newcommand{\Upa}{\Uparrow}
\newcommand{\Dna}{\Downarrow}

\newcommand{\Tight}[1]{\mathsf{Ti}(#1)}
\newcommand{\jslTight}[1]{\mathsf{Ti}[#1]}
\newcommand{\biTight}[1]{\mathsf{Ti}[#1]}

\newcommand{\Nleq}{\mathtt{Nleq}}

\newcommand{\BId}[1]{\mathsf{BId}(#1)}
\newcommand{\jslBId}[1]{\mathbb{BId}(#1)}

\newcommand{\tenp}{\otimes}
\newcommand{\syncp}{\owedge}
\newcommand{\ttenp}{\otimes_t}

\newcommand{\BiMor}[1]{\mathsf{BiMor}(#1)}

\newcommand{\rG}{\mathcal{G}}
\newcommand{\rH}{\mathcal{H}}
\newcommand{\rI}{\mathcal{I}}

\newcommand{\rJ}{\mathcal{J}}

\newcommand{\BC}[1]{\mathcal{K}(#1)}
\newcommand{\UBC}[1]{\mathcal{K}_u{(#1)}}
\newcommand{\URC}[1]{\mathcal{K}_r(#1)}
\newcommand{\UIC}[1]{\mathcal{K}_i(#1)}

\newcommand{\rR}{\mathcal{R}}
\newcommand{\rS}{\mathcal{S}}
\newcommand{\rT}{\mathcal{T}}

\newcommand{\rM}{\mathcal{M}}

\newcommand{\rTS}{\mathcal{TS}}

\newcommand{\rE}{\mathcal{E}}
\newcommand{\rF}{\mathcal{F}}

\newcommand{\aA}{\mathbb{A}}
\newcommand{\aB}{\mathbb{B}}
\newcommand{\aC}{\mathbb{C}}
\newcommand{\aD}{\mathbb{D}}
\newcommand{\aE}{\mathbb{E}}
\newcommand{\aF}{\mathbb{F}}
\newcommand{\aQ}{\mathbb{Q}}
\newcommand{\aR}{\mathbb{R}}
\newcommand{\aS}{\mathbb{S}}
\newcommand{\aT}{\mathbb{T}}

\newcommand{\latL}{\mathcal{L}}
\newcommand{\latOp}[1]{\mathcal{O}(#1)}
\newcommand{\latCl}[1]{\mathcal{C}(#1)}

\newcommand{\pP}{\mathtt{P}}
\newcommand{\pQ}{\mathtt{Q}}

\newcommand{\pTwoC}{\mathtt{2}}
\newcommand{\pTwoA}{\mathtt{2}_a}

\newcommand{\pCovHom}[2]{\mathtt{Cover}(#1,#2)}

\newcommand{\prCong}[2]{\mathcal{PC}_{#1}^{#2}}
\newcommand{\genCong}[2]{\mathcal{GC}_{#1}(#2)}
\newcommand{\mirrCong}[2]{\mathcal{MC}_{#1}^{#2}}

\newcommand{\genSubset}[2]{\mathrm{GS}_{#1}(#2)}
\newcommand{\genElem}[2]{\mathrm{GS}_{#1}(\{ #2 \})}
\newcommand{\mirrSub}[2]{{\mathbb{MS}}_{#1}^{#2}}
\newcommand{\mirrSubset}[2]{\mathrm{MS}_{#1}(#2)}

\newcommand{\setCong}[1]{Con(#1)}
\newcommand{\setSub}[1]{Sub(#1)}
\newcommand{\latCong}[1]{\mathcal{CON}(#1)}
\newcommand{\latSub}[1]{\mathcal{SUB}(#1)}

\newcommand{\jslM}[1]{\mathbb{M}_{#1}}
\newcommand{\jslElem}[1]{\mathbb{Elem}(#1)}
\newcommand{\jslIdeal}[1]{\mathbb{Ideal}(#1)}
\newcommand{\jslTrelem}[1]{\mathbb{Trelem}(#1)}
\newcommand{\jslTrideal}[1]{\mathbb{Trideal}(#1)}

\newcommand{\cl}{\mathrm{\bf cl}}
\newcommand{\inte}{\mathrm{\bf in}}

\newcommand{\JSL}{{\mathsf{JSL}}}

\newcommand{\SAJ}{\mathsf{SAJ}}
\newcommand{\SAM}{\mathsf{SAM}}
\newcommand{\SAI}{\mathsf{SAI}}

\newcommand{\UGJ}{\mathsf{UG}_j}
\newcommand{\UGM}{\mathsf{UG}_m}
\newcommand{\UG}{\mathsf{UG}}

\newcommand{\BiCliq}{\mathsf{Dep}}
\newcommand{\Cover}{\mathsf{Cover}}

\newcommand{\Dep}{\mathsf{Dep}}

\newcommand{\Rel}{\mathsf{Rel}}

\newcommand{\Id}{\mathsf{Id}}
\newcommand{\Pow}{\mathcal{P}}
\newcommand{\FPow}{\Pow_f}

\newcommand{\JFPow}{\mathbb{P}_f}
\newcommand{\JPow}{\mathbb{P}}
\newcommand{\BPow}{\mathbb{P}_b}
\newcommand{\DPow}{\mathbb{P}_d}

\newcommand{\Ji}{\mathsf{Ji}}

\newcommand{\SetToPos}{F_\leq}
\newcommand{\PosToJSL}{F_\lor}
\newcommand{\JSLToDL}{F_\land}
\newcommand{\DLToBA}{F_\neg}
\newcommand{\BAToDL}{U_\neg}
\newcommand{\DLToJSL}{U_\land}
\newcommand{\JSLToPos}{U_\lor}
\newcommand{\PosToSet}{U_\leq}

\newcommand{\Pirr}{\mathtt{Pirr}}
\newcommand{\Open}{\mathtt{Open}}

\newcommand{\GJOpen}{\Open_{\mathrm{j}}}
\newcommand{\GJPirr}{\Pirr_{\mathrm{j}}}
\newcommand{\GMOpen}{\Open_{\mathrm{m}}}
\newcommand{\GMPirr}{\Pirr_{\mathrm{m}}}
\newcommand{\GOpen}{\Open_{\mathrm{g}}}
\newcommand{\GPirr}{\Pirr_{\mathrm{g}}}

\newcommand{\pOp}{\mathsf{op}}

\newcommand{\ideal}[2]{\mathrm{idl}_{#1}\ang{#2}}
\newcommand{\elem}[2]{\mathrm{el}_{#1}\ang{#2}}
\newcommand{\idealInv}[2]{\mathrm{idl}^{\bf-1}_{#1}\ang{#2}}
\newcommand{\elemInv}[2]{\mathrm{el}^{\bf-1}_{#1}\ang{#2}}
\newcommand{\trideal}[2]{\mathrm{idl}^3_{#1}\ang{#2}}
\newcommand{\trelem}[2]{\mathrm{el}^3_{#1}\ang{#2}}

\newcommand{\swap}{\mathrm{swap}}
\newcommand{\rot}{\mathrm{rot}}

\def\endbox{\hfill$\blacksquare$}

\newcommand{\rep}{\mathrm{rep}}
\newcommand{\red}{\mathrm{red}}
\newcommand{\jrep}{\mathrm{{\bf j}rep}}
\newcommand{\jred}{\mathrm{{\bf j}red}}
\newcommand{\mrep}{\mathrm{{\bf m}rep}}
\newcommand{\mred}{\mathrm{{\bf m}red}}
\newcommand{\grep}{\mathrm{{\bf g}rep}}
\newcommand{\gred}{\mathrm{{\bf g}red}}

\newcommand{\OD}{\mathtt{OD}}

\CompileMatrices
\begin{document}

\title{Representing Semilattices as Relations}
\author{Robert S.\ R. Myers  \\[1ex] \href{me.robmyers@gmail.com}{me.robmyers@gmail.com} }
\maketitle

\CompileMatrices

\section{Note for readers}

\begin{itemize}
  \item[--]
  The first few sections provide background concerning order-theory and semilattices.
  \item[--] 
  The new category $\Dep$ is introduced and studied in Sections \ref{sec:the_cat_dep} and \ref{sec:tensors}.
  \begin{itemize}
    \item 
    Its objects $\rG$ are the relations between finite sets.
    It morphisms $\rG \to \rH$ are those relations $\rR$ factoring via relational composition through $\rG$ (resp.\ $\rH$) on the left (resp.\ right).
    \item
    In subsection \ref{subsec:dep_equiv_jsl} we prove it is categorically equivalent to $\JSL_f$.
  \end{itemize}
  \item[--]
  Section \ref{sec:tensors} describes the tensor product and tight tensor product of finite join-semilattices, also in $\Dep$.
  \item[--]
  Section \ref{sec:graphs_and_de_morgan} extends the main result from binary relations to symmetric relations (undirected graphs) and from finite join-semilattices to finite De Morgan algebras.
  \item[--]
  Finally, the \hyperref[appendix:appendix]{Appendix} describes and proves a number of relevant categorical dualities and free constructions.
\end{itemize}


\section{Conventions and background}

\subsection{Conventions regarding relations and functions}

It is worth clarifying the definition of functions and relations because:
\begin{enumerate}[(a)]
\item
algorithms requires specific representations.
\item
it avoids `clutter' e.g.\ we don't want to distinguish between a functional relation and a function.
\end{enumerate}

After these basic definitions we introduce notation to avoid a cumbersome presentation.

\begin{definition}[Relations and functions]
\label{def:rel_and_func}
\item
\begin{enumerate}
\item
The \emph{cartesian product} of two sets $X$ and $Y$ is defined $X \times Y := \{ (x,y) : x \in X, y \in Y\}$.

\item
A \emph{relation} is a triple $(\rR,X,Y)$ where $\rR \subseteq X \times Y$ is any subset. Then $X$ is called the \emph{domain} of $\rR$ (also denoted $\rR_s$) whereas $Y$ is called the \emph{codomain} of $\rR$ (also denoted $\rR_t$).

\item
\begin{enumerate}
\item
The \emph{identity relation} on a set $X$ is defined $(\Delta_X,X,X)$ where $\Delta_X := \{(x,x) \in X \times X : x \in X\}$.

\item
The \emph{converse} of a relation $(\rR,X,Y)$ is the relation $(\breve{\rR},Y,X)$ where $\breve{\rR} := \{ (y,x) : (x,y) \in \rR \}$.

\item
The \emph{complement} of a relation $(\rR,X,Y)$ is the relation $(\overline{\rR},X,Y)$ where $\overline{\rR} := (X \times Y) \backslash \rR$.
\end{enumerate}

\item
For any relation $(\rR, X, Y)$, subset $S \subseteq X$ and domain element $d \in X$, define:
\[
\rR[S] := \{ y \in Y : \exists x \in S. (x,y) \in \rR \}
\qquad
\rR[d] := \rR[\{d\}]
\]
i.e.\ the image of a subset of the domain, and the image of a domain element.

\item
Given any two compatible relations $(\rR,X,Y)$ and $(\rS,Y,Z)$ then their \emph{composite relation} is defined:
\[
(\rR,X,Y) ; (\rS,Y,Z) := (\{(x,z) \in X \times Z, \exists y \in Y.( (x,y) \in \rR \text{ and } (y,z) \in \rS ) \},\,X,\,Z)
\]

\item
A relation $(\rR,X,Y)$ is \emph{functional} if $\forall x \in X. \exists \text{ unique } y \in Y. (x,y) \in \rR$.
Then a \emph{function} is another word for a functional relation, so the identity relation $(\Delta_X,X,X)$ is also $X$'s \emph{identity function}, written $id_X$.

\item
For any $X$, its \emph{powerset} $\Pow X$ is the set containing precisely the subsets $S \subseteq X$ including the empty set $\emptyset$. If $X$ is finite let $|X| \in \Nat := \{0,1,2,\dots\}$ denote its number of elements.

\item
For each set $X$, subset $S \subseteq X$ and element $z \in X$, we define:
\[
\overline{S} := X\backslash S = \{ x \in X : x \notin S\}
\qquad
\overline{z} := \overline{\{z\}} = X \backslash \{z\}
\qquad
\neg_X := (\{ (S,\overline{S}) : S \in \Pow X\},\Pow X,\Pow X)
\]
i.e.\ the relative-complement of subsets or elements, and the involutive relative-complement function.

\takeout{
\item
Given any relation $(\rR,X,Y)$ then:
\begin{tabular}{ll}
its \emph{image function} is defined: 
&
$(\{(S,\rR[S]) : S \in \Pow X \} ,\Pow X,\Pow Y)$
\\
its \emph{preimage function} is defined:
&
$(\{(S,\breve{\rR}[S]) : S \in \Pow Y \},\Pow Y,\Pow X)$
\end{tabular}
}

\item
The notions of injective, surjective and bijective functions are as usual. These concepts \emph{only apply to functions} so e.g.\ if we say a relation is injective we mean that it is an injective function. For any function $f : X \to Y$  its \emph{preimage function} $(f^{-1},\Pow Y,\Pow X)$ has action $f^{-1}(S) = \{ x \in X : f(x) \in S \}$.

\item
Given a relation $(\rR,X_1,X_2)$ and subsets $Y_i \subseteq X_i$ for $i = 1,2$ then its \emph{restriction} $(\rR,X_1,X_2) |_{Y_1 \times Y_2}$ is the relation $(\rR \cap Y_1 \times Y_2, Y_1,Y_2)$.


\endbox

\end{enumerate}
\end{definition}

\smallskip
We now list standard notational conventions which we shall henceforth adopt.

\begin{notation}[Relations and functions]
\label{nota:rels_and_funcs}
\item
\begin{enumerate}
\item
For any relation $(\rR,X,Y)$ let $\rR_s := X$ ((s)ource = domain) and $\rR_t := Y$ ((t)arget = codomain). We may denote a relation by the symbol $\rR$ as long as $\rR_s$ and $\rR_t$ have been specified, either by defining them directly, or via words such as `the relation $\rR \subseteq X \times Y$' where it is understood that $\rR_s = X$ and $\rR_t = Y$.

\item
To indicate a relation $(\rR,X,Y)$ or $\rR \subseteq X \times Y$ we may also write $\rR : X \to Y$. This is more usual for functions, but is perfectly acceptable for relations too.

\item
$\rR(x,y)$ indicates that $(x,y) \in \rR$. Sometimes it is more natural to write $x \rR y$, for example if $\rR$ is a partial order.  The converse relation of $\rR$ may be written as $\breve{\rR}$ or $\rR\spbreve$, observing that $\breve{\rR}_s = \rR_t$ and $\breve{\rR}_t = \rR_s$.

\item
We denote general relations by upper-case calligraphic symbols e.g.\ $\rR$, $\rS$ and $\rT$, and general functions by lower-case standard type symbols e.g.\ $f$, $g$, $h$. Then we may write $f(x) = y$ to mean $f(x,y)$ where the latter $y$ is necessarily unique. Since we understand functions to be functional relations there will be some overlap in symbols, but usually only if a certain relation turns out (or is restricted) to be functional. 

\item
By the above remarks, given relations $\rR \subseteq X \times Y$ and $\rS \subseteq Y \times Z$ we may write their composite relation as $\rR ; \rS \subseteq X \times Z$. Following the usual convention, the composite of two functions $f : X \to Y$ and $g : Y \to Z$ is written as $g \circ f : X \to Z$ i.e.\ the other way around to relational composition.

\item
We write the restriction of a relation as $\rR \cap Y_1 \times Y_2$ or alternatively as $\rR |_{Y_1 \times Y_2}$. \endbox

\end{enumerate}
\end{notation}

\subsection{Order Theory}

We shall need various basic concepts from order theory i.e.\ posets, join-semilattices, lattices, bounded lattices, De Morgan algebras (which needn't be distributive), distributive lattices, boolean lattices and algebras, join and meet-irreducible elements,  join and meet-prime elements, and also closure and interior operators on an arbitrary poset. We also prove a number of (mostly) standard results e.g.\ every finite join-semilattice is a finite lattice in a unique way, a finite lattice is distributive iff every join-irreducible element is join-prime, and we also describe the canonical order-isomorphism between join and meet-irreducibles of a finite distributive lattice.


\begin{definition}[Basic order theory]
\label{def:std_order_theory}
\item
\begin{enumerate}
\item
A \emph{poset} is a pair $\pP = (P,\leq_\pP)$ where $P$ is a set and the relation $\leq_\pP \; \subseteq P \times P$ is reflexive, transitive and anti-symmetric. An \emph{order relation} is a relation $\rR \subseteq X \times X$ with these three properties. A poset $(P,\leq_\pP)$ is \emph{finite} if $P$ is a finite set.  

\item
Given a poset $\pP$, then its \emph{opposite poset} is defined $\pP^{\pOp} := (P,\leq_{\pP^{\pOp}})$ where $\leq_{\pP^{\pOp}} \; := \; \breve{\leq_\pP}$ is the converse relation, and is more usually written as $\geq_\pP$. A \emph{monotone} (or \emph{monotonic}) function from $\pP$ to $\pQ$ is a function $f : P \to Q$ such that $p_1 \leq_\pP p_2$ implies $f(p_1) \leq_\pQ f(p_2)$ for all elements $p_1,p_2 \in P$. We may indicate monotone morphisms by writing $f : (P,\leq_\pP) \to (Q,\leq_\pQ)$. We also have the \emph{opposite} monotone morphism $f^{\pOp} : \pP^{\pOp} \to \pQ^{\pOp}$ which acts in the same way i.e.\ $f^{\pOp}(p) = f(p)$ for all elements $p \in P$, see Note \ref{note:tale_of_two_orders} below.

\item
We'll use the other standard symbols and their converses i.e.\ $<_\pP$ means strictly less than (with converse $>_\pP$), $\nleq_\pP$ means not less than or equal to (with converse $\ngeq_\pP$), $\nless_\pP$ means not strictly less than (with converse $\ngtr_\pP$). We also have the irreflexive and symmetric \emph{incomparibility relation} $\parallel_\pP \; := \; \nleq_\pP \, \cap \, \ngeq_\pP$ which also equals $\parallel_{\pP^{\pOp}}$. Finally, $\pP$'s \emph{covering relation} $\prec_\pP \, \subseteq P \times P$ is defined: 
\[
p_1 \prec_\pP p_2
:\iff p_1 <_\pP p_2 \text{ and } \neg\exists p \in P.(p_1 <_\pP p <_\pP p_2).
\]

\item
A \emph{chain} is a non-empty totally ordered poset $\pP$ i.e.\ such that $\parallel_\pP \, = \, \emptyset$. An \emph{antichain} is a non-empty poset $\pP$ where distinct elements are incomparable i.e.\ $\parallel_\pP \, = P \times P \backslash \Delta_P$. A poset which is either empty or an antichain is called a \emph{discrete poset}. Let us denote the \emph{2-chain} by $\pTwoC := (\{0,1\},\Delta_2 \cup \{(0,1)\})$ and the \emph{2-antichain} by $\pTwoA := (\{0,1\},\Delta_2)$.

We say that $\pP = (P,\leq_\pP)$ is a \emph{subposet} of $\pQ = (Q,\leq_\pQ)$ if $P \subseteq Q$ and $\leq_\pP \, = \, \leq_\pQ \, \cap \, P \times P$. Then a subposet must inherit the order, so that $\pTwoA$ is not a subposet of $\pTwoC$. If a chain $(P,\leq_\pP)$ is finite then its \emph{length} is defined $|P| - 1$ e.g.\ the two element chain $\pTwoC$ has length $1$. Then the \emph{length} $l(\pP) \in \Nat \cup \{\omega \}$ of an arbitrary poset $\pP$ is defined as the supremum of the lengths of all finite chains arising as subposets of $\pP$. That is, if the length of such chains is bounded then it is the maximum length of any chain, otherwise it is $\omega$.

\item
A subset $S \subseteq P$ of a poset $\pP$ is \emph{up-closed} (or \emph{upwards-closed}) if whenever $p \in S$ and $p \leq_\pP p'$ then $p' \in S$. A subset $S \subseteq P$ of a poset $\pP$ is \emph{down-closed} (or \emph{downwards-closed}) if it is up-closed in $\pP^{\pOp}$. Equivalently, the up-closed (resp.\ down-closed) sets of $\pP$ are precisely those sets of the form $f^{-1}(\{1\})$ (resp.\ $f^{-1}(\{0\})$) for monotone functions $f : \pP \to \pTwoC$. 

\item
Given a poset $\pP = (P,\leq_\pP)$ then its \emph{join-irreducible} elements $J(\pP) \subseteq P$ and \emph{meet-irreducible} elements $M(\pP) \subseteq P$ are defined as follows:
\[
\begin{tabular}{c}
$J(\pP) := \{ p \in P : \exists! q \in P. \; q \prec_\pP p \}$
\\[1ex]
$M(\pP) := \{ p \in P : \exists! q \in P. \; p \prec_\pP q  \}$
\end{tabular}
\]
where $\exists!$ should be read as `there exists a unique'. If $\pP$ has a minimum element $\bot_\pP \in P$ then its \emph{atoms} are defined $At(\pP) := \{ p \in P : \bot_\pP \prec_\pP p \}$ i.e.\ those elements covering $\bot_\pP$. On the other hand, if $\pP$ has a maximum element $\top_\pP \in P$ then its \emph{coatoms} are defined $CoAt(P) := \{ p \in P : p \prec_\pP \top_\pP \}$ i.e.\ those elements covered by $\top_\pP$. These are order-dual concepts i.e.\ $M(\pP) = J(\pP^{\pOp})$ holds generally, and $CoAt(\pP) = At(\pP^{\pOp})$ holds whenever $\pP$ has a top element. 

See Lemma \ref{lem:std_order_theory}.5 below for the way we usually think about join/meet-irreducibles.

\item
A \emph{join-semilattice} $\aQ = (Q,\lor_\aQ,\bot_\aQ)$ is a commutative and idempotent monoid i.e.\ $Q$ is the carrier set, $\lor_\aQ : Q \times Q \to Q$  is the associative binary operation and $\bot_\aQ \in Q$ is the unit. Equivalently, a join-semilattice is a poset $(Q,\leq_\aQ,\lor_\aQ,\bot_\aQ)$ where all finite suprema exist, the empty supremum being $\bot_\aQ$ and the supremum of $\{q_1,q_2\}$ being $q_1 \lor_\aQ q_2$. In particular, one can define $x \leq_\aQ y :\iff x \lor_\aQ y = y$. We usually describe $\bot_\aQ$ as the \emph{bottom element} or \emph{empty join}, $\lor_\aQ$ as the \emph{binary join}, and the suprema of finitely many elements as \emph{joins}. The latter are often denoted via the symbol $\Lor_\aQ$ which inductively generalises $\bot_\aQ$ and $\lor_\aQ$.

\item
A \emph{lattice} $\latL = (L,\lor_\latL,\land_\latL)$ is a poset $(L,\leq_\latL)$ with  binary joins $\lor_\latL$ and binary meets $\land_\latL$. A \emph{bounded lattice} $\latL = (L,\lor_\latL,\bot_\latL,\land_\latL,\top_\latL)$ is a poset $(L,\leq_\latL)$ with all finite joins (suprema) and all finite meets (infima). Notice that the join-structure is always written before the meet-structure. We may also speak of a \emph{lattice with bottom} or \emph{lattice with top}. Finite lattices always have a bottom and top, although they may not preserved by the morphisms under consideration.

Bounded lattices may be equationally axiomatised by specifying two commutative idempotent monoids $(\lor,\bot)$ and $(\land,\top)$ (hence join-semilattices), as well as the absorption laws. The latter laws ensure that their respective order relations are the converse of one another.

\item
Given a bounded lattice $\latL$ and elements $x,y \in L$ then \emph{$y$ is a complement of $x$} if:
\[
x \land_\latL y = \bot_{\latL}
\qquad\text{and}\qquad
x \lor_\latL y = \top_{\latL}.
\]
There exist lattices where an element may have no complement (the midpoint of a $3$-chain) or many of them (add a bottom and top to a $3$-antichain). However, an element of a distributive lattice may have at most one complement by Lemma \ref{lem:std_order_theory}.9 below.

\item
A \emph{distributive lattice} $\aD = (D,\lor_\aD,\land_\aD)$ is a lattice where the two distributive laws hold i.e.\
\[
\begin{tabular}{c}
$x \land (y \lor z) \approx (x \land y) \lor (x \land z)$
\\[1ex]
$x \lor (y \land z) \approx (x \lor y) \land (x \lor z)$
\end{tabular}
\]
In practice we'll mostly deal with \emph{bounded distributive lattices} i.e.\ bounded lattices which are distributive. A \emph{boolean lattice} $\aB$ is a complemented bounded distributive lattice i.e.\ for every element $x \in B$ there exists $y \in B$ such that $x \land_\aB y = \bot_\aB$ and $x \lor_\aB y = \top_\aB$. Elements of distributive lattices may have at most one complement, see Lemma \ref{lem:std_order_theory}.9 below. A \emph{boolean algebra} $\aA = (A,\lor_\aA,\bot_\aA,\land_\aA,\top_\aA,\neg_\aA)$ is a boolean lattice $(A,\lor_\aA,\bot_\aA,\land_\aA,\top_\aA)$ endowed with its unique complement operation $\neg_\aA : A \to A$. Equivalently they may be equationally axiomatised via the equational axiomatisation of lattices, the distributive laws and the two complement laws i.e.\ $x \land \neg x \approx \bot$ and $x \lor \neg x \approx \top$. In particular, the two De Morgan laws follow from the latter equational axiomatisation of boolean algebras, which we also verify syntactically in Lemma \ref{lem:std_order_theory}.9 below.

\item
Regarding morphisms, recall that we've already defined the notion of `monotonic function' between posets in item (2) above.

\begin{enumerate}
\item
Given join-semilattices $\aQ_i = (Q_i,\lor_{\aQ_i},\bot_{\aQ_i})$ for $i = 1,2$ then a \emph{join-semilattice morphism} $f : \aQ_1 \to \aQ_2$ is a monotonic function $f : (Q_1,\leq_{\aQ_1}) \to (Q_2,\leq_{\aQ_2})$ which also preserves the bottom and binary join. Equivalently they are the monoid morphisms (since join-preservation implies monotonicity), and if both join-semilattices are finite they are precisely those functions $f : Q_1 \to Q_2$ preserving \emph{all} joins.

\item
Given lattices $\latL_i = (L_i,\lor_{\latL_i},\land_{\latL_i})$ for $i = 1,2$ then a \emph{lattice morphism} $f : \latL_1 \to \latL_2$ is a monotonic function $f : (L_1,\leq_{\latL_1}) \to (L_2,\leq_{\latL_i})$ which also preserves the binary join and binary meet. A \emph{bounded lattice morphism} between bounded lattices is a lattice morphism which additionally preserves the top and bottom element.

\item
Finally, a \emph{boolean algebra morphism} $f : \aA_1 \to \aA_2$ is a bounded lattice morphism between their underlying boolean lattices $(A_i,\lor_{\aA_i},\bot_{\aA_i},\land_{\aA_i},\top_{\aA_i})$ for $i = 1,2$. Such morphisms automatically preserve negation via the uniqueness of complements in boolean algebras, see Lemma \ref{lem:std_order_theory}.9 below. One can define boolean algebra morphisms in a number of equivalent ways e.g.\ via preservation of bottom, meet and complement.

\end{enumerate}

\item
The opposite poset construction restricts to lattices, bounded lattices, boolean algebras, and also finite join-semilattices. Furthermore it preserves distributivity and the existence of complements.

\begin{enumerate}
\item
Each bounded lattice $\latL$ has an \emph{opposite bounded lattice}:
\[
(L,\lor_\latL,\bot_\latL,\land_\latL,\top_\latL)^{\pOp} := (L,\land_\latL,\top_\latL,\lor_\latL,\bot_\latL)
\]
We also have the \emph{opposite lattice}. Observe that $\latL^{\pOp}$'s underlying poset is the opposite of $\latL$'s underlying poset.

\item
If a possibly bounded lattice $\latL$ is distributive then so is $\latL^{\pOp}$, since the two distributive laws arise from one another by swapping meets and joins. Likewise, if a bounded lattice $\latL$ is a boolean lattice then so is $\latL^{\pOp}$ by inspecting the two complementation laws.

\item
Each boolean algebra $\aA$ has an \emph{opposite boolean algebra} i.e.\ 
\[
(A,\lor_\aA,\bot_\aA,\land_\aA,\top_\aA,\neg_\aA)^{\pOp} 
:= (A,\land_\aA,\top_\aA,\lor_\aA,\bot_\aA,\neg_\aA)
\]
That negation is well-defined follows because the two standard equational laws expressing complementation are order-dual statements. $\aA^{\pOp}$'s underlying poset is the opposite of $\aA$'s underlying poset.

\item
Each finite join-semilattice $\aQ$ is a lattice in a unique fashion i.e.\
\[
\top_\aQ := \Lor_\aQ Q
\qquad
q_1 \land_\aQ q_2 := \Lor_\aQ \{ q \in Q : q \leq_\aQ q_1, q \leq_\aQ q_2 \}
\]
Thus in the finite case we also have the \emph{opposite join-semilattice} defined  $\aQ^{\pOp} := (Q,\land_\aQ,\top_\aQ)$, noting that $\aQ^{\pOp}$'s underlying poset is the opposite of $\aQ$'s underlying poset. The existence of meets fails for infinite join-semilattices e.g.\
\begin{enumerate}[(i)]
\item
Concerning boundedness, given any infinite set $X$ then the join-semilattice $(\FPow X,\cup,\emptyset)$ of finite subsets has no top element.

\item
Below we depict the `usual' example of a join-semilattice which fails to have binary meets.
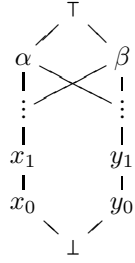
\[
\xymatrix@=5pt{
& \top
\\
\alpha \ar@{-}[ur] && \beta \ar@{-}[ul]
\\
\vdots \ar@{-}[u] \ar@{-}[urr] && \vdots \ar@{-}[u] \ar@{-}[ull]
\\
x_1 \ar@{-}[u] && y_1 \ar@{-}[u]
\\
x_0 \ar@{-}[u] && y_0 \ar@{-}[u]
\\
& \bot \ar@{-}[ul] \ar@{-}[ur]
}
\]
That is, $\aQ$ has the order relation:
\[
\leq_\aQ \;
= 
\{\bot\} \times Q \,\cup\, \{ (x_i,x_j) : 0 \leq i \leq j \} \,\cup\, \{ (y_i,y_j) : 0 \leq i \leq j \} \,\cup\, (X \cup Y) \times \{\alpha,\beta\} \,\cup\, Q \times \{\top\}
\]
where $X = \{ x_i : i \in \Nat \}$ and $Y = \{ y_i : i \in \Nat \}$. It has all finite joins, whereas $\alpha \land \beta$ doesn't exist.

\end{enumerate}

\end{enumerate}

Recall that given any monotone function $f : \pP \to \pQ$ we also have the opposite monotone function $f^{\pOp} : \pP^{\pOp} \to \pQ^{\pOp}$ which acts in exactly the same way as $f$. Then this restricts to bounded lattice morphisms and also boolean algebra morphisms. However it does \emph{not} restrict to join-semilattice morphisms between finite join-semilattices, since the latter needn't preserve the top or the binary meet.

\item
A \emph{preorder} is a relation $\rR \subseteq X \times X$ which is reflexive and transitive. There is an associated equivalence relation $\rE \subseteq X \times X$ defined $\rE(x,y) :\iff \rR(x,y) \land \rR(y,x)$. Then the \emph{poset induced by the preorder $\rR$} consists of the $\rE$-equivalence-classes equipped with the well-defined order relation $\sem{x}_\rE \leq_{\rR\backslash\rE} \sem{y}_\rE :\iff \rR(x,y)$.

\item
A monotonic function $f : \pP \to \pQ$ is an \emph{order-embedding} if:
\[
p_1 \leq_\pP p_2
\iff
f(p_1) \leq_\pQ f(p_2)
\qquad
\text{for every $p_1,\,p_2 \in P$.}
\]
Every order-embedding defines an injective function, although there are injective monotone maps which are not order-embeddings e.g.\ from the $2$-antichain to the $2$-chain. However, every injective join-semilattice morphism $f : \aQ \to \aR$ defines an order-embedding $f : (Q,\leq_\aQ) \to (R,\leq_\aR)$, as we now show:
\[
f(q_1) \leq_\aR f(q_2)
\iff
f(q_1) \lor_\aR f(q_2) = f(q_2)
\iff
f(q_1 \lor_\aQ q_2) = f(q_2)
\iff
q_1 \lor_\aQ q_2 = q_2
\iff
q_1 \leq_\aQ q_2.
\]

\item
We say that a join-semilattice $\aQ$ is a \emph{join-semilattice retract} of a join-semilattice $\aR$ if there exist join-semilattice morphisms:
\[
\xymatrix@=15pt{
\aQ \ar@/^5pt/@{>->}[rr]^e  && \aR \ar@/^5pt/@{->>}[ll]^r
}
\qquad\text{such that $r \circ e = id_\aQ$}.
\]
Note that $e$ is necessarily injective and hence an order-embedding, whereas $r$ is necessarily surjective.


\item
We list various specific algebras and some associated terminology, where $Z$ is any set.

\begin{enumerate}
\item
$\JPow Z := (\Pow Z,\cup,\emptyset)$ is called a \emph{powerset join-semilattice}, observing that $(\JPow Z)^{\pOp} = (\Pow Z,\cap,Z)$. Furthermore $\JFPow Z := (\FPow Z,\cup,\emptyset)$ is called the \emph{free join-semilattice on $Z$}, observing that it needn't have a top if $Z$ is infinite. Usually $Z$ will be finite, in which case these two concepts coincide.
\item
$\DPow Z := (\Pow Z,\cup,\emptyset,\cap,Z)$ is called the \emph{powerset bounded distributive-lattice}, and $\BPow Z := (\Pow Z,\cup,\emptyset,\cap,Z,\neg_Z)$ is called the \emph{powerset boolean-algebra}.

\item
If $\aQ$ is a finite join-semilattice then it possesses a unique bounded lattice structure by (12).d above. Then we say that:
\begin{enumerate}
\item
\emph{$\aQ$ is a boolean join-semilattice} if its associated bounded lattice is boolean, 
\item
\emph{$\aQ$ is a distributive join-semilattice} if its associated lattice is distributive, noting that it will also be bounded by finiteness.
\end{enumerate}

\item
$\jslM{Z} := ( \{\emptyset,Z\} \, \cup \, \{ \overline{z} : z \in Z \}, \cup, \emptyset)$ may be viewed as the join-semilattice version of an antichain. That is, its order structure amounts to viewing $Z$ as the $\subseteq$-antichain $\{ \overline{z} : z \in Z \}$, and then adding a bottom $\emptyset$ and top $Z$ to obtain a lattice.

\endbox
\end{enumerate}

\end{enumerate}
\end{definition}

\begin{note}[Order opposite $(-)^{\pOp}$ versus categorical opposite $(-)^{op}$]
\label{note:tale_of_two_orders}
\item
The opposite poset/monotone map construction $(-)^{\pOp}$ is distinct from the standard categorical notation $(-)^{op}$. Actually, they do align by viewing posets as thin skeletal categories (so that the opposite poset is the opposite category), and moreover the functors between these categories are the monotone maps (so that the opposite functor corresponds to the opposite monotone map). However, we'll also use $f^{op} : Y \to X$ to indicate a morphism in the opposite of other categories, so we distinguish the notations to avoid confusion. \endbox
\end{note}

We now list some basic facts concerning these definitions.

\begin{lemma}[Standard order-theoretic results]
\label{lem:std_order_theory}
\item
\begin{enumerate}
\item
If a poset $\pP$ has all finite suprema then it is the underlying poset of a unique join-semilattice $\aQ$.
\item
If a poset $\pP$ has all finite suprema and infima then it is the underlying poset of a unique bounded lattice $\latL$.
\item
Viewed as their underlying posets, the finite lattices, the finite bounded lattices, and the finite join-semilattices coincide.

\item
If a poset $\pP$ has finite length then $\leq_\pP$ is the reflexive transitive closure of $\prec_P$.


\item
Let $\latL$ be a locally finite lattice i.e.\ every interval is finite.
\begin{enumerate}
\item
If $\latL$ has a bottom element then $J(\latL)$ contains precisely those elements which are not a finite join of other elements. Equivalently:
\begin{quote} 
$x \in J(\latL)$ iff (i) $x \neq \bot_\latL$, and (ii) $\forall x_1,x_2 \in L$. $x = x_1 \lor_\latL x_2$ implies $\exists i. x = x_i$.
\end{quote}

\item
If $\latL$ has a top element then $M(\latL)$ contains precisely those elements which are not a finite meet of other elements. Equivalently:
\begin{quote} 
$x \in M(\latL)$ iff (i) $x \neq \top_\latL$, and (ii) $\forall x_1,x_2 \in L$. $x = x_1 \land_\latL x_2$ implies $\exists i. x = x_i$.
\end{quote}

\item
$\latL$ has both a bottom and top element iff it is a finite lattice (hence bounded), in which case both of the above statements hold.
\end{enumerate}

\item
Each finite lattice $\latL$ is join-generated by $J(\latL)$ and every join-generating set $S \subseteq Q$ contains $J(\latL)$. Order-dually, $\latL$ is meet-generated by $M(\latL)$ and every meet-generating set contains $M(\latL)$.

\item
Given any finite join-semilattice $\aQ$ and elements $q_1,\, q_2 \in Q$, the following statements hold:
\[
\begin{tabular}{ll}
$q_1 \leq_\aQ q_2$ & $\iff \forall j \in J(\aQ).[ j \leq_\aQ q_1 \implies j \leq_\aQ q_2  ]$
\\[1ex]
& $\iff \forall m \in M(\aQ).[ q_2 \leq_\aQ m \implies q_1 \leq_\aQ m  ]$
\end{tabular}
\]

\item
The atoms of any boolean lattice $\latL$ are precisely $J(\latL)$, the co-atoms are precisely $M(\latL)$.

\item
Let $\aD$ be a bounded distributive lattice.
\begin{enumerate}
\item
Each element $d \in D$ can have at most one complement. 
\item
If every element $d \in D$ has a complement then it is a boolean lattice.
\end{enumerate}
Concerning the second point: syntactically speaking, the equational axiomatisation of boolean algebras described in Definition \ref{def:std_order_theory}.9 above is correct. That is, the De Morgan laws are deducible from it.

\item
For any finite lattice $\latL$ the following statements are equivalent.
\begin{enumerate}
\item
$\latL$ is distributive.
\item
For all $j \in J(\latL)$ and $x_1,\, x_2 \in L$ we have:
 \[
j \leq_\latL x_1 \lor_\latL x_2 
\iff
j \leq_\latL x_1 \text{ or } j \leq_\latL x_2
\]
noting that one may equivalently restrict to $x_1, x_2 \in L \backslash \{\bot_\latL,\top_\latL\}$.
\item
For all $j \in J(\latL)$ and subsets $X \subseteq L$:
\[
j \leq_\latL \Lor_\latL X
\iff
\exists x \in X. j \leq_\latL x
\]
\item
For all $j \in J(\latL)$ and subsets $X \subseteq J(\latL)$:
\[
j \leq_\latL \Lor_\latL X
\iff
\exists x \in X. j \leq_\latL x
\]
\end{enumerate}

\takeout{
If $\aD$ is a finite distributive lattice, then for any $j \in J(\aD)$ and any $d_1,d_2 \in D$ we have:
\[
j \leq_\aD d_1 \lor_\aD d_2 
\iff
j \leq_\aD d_1 \text{ or } j \leq_\aD d_2
\]
More generally, for all subsets $X \subseteq D$ we have $j \leq_\aD \Lor_\aD X$ iff there exists $d \in X$ such that $j \leq_\aD d$. The latter has an equivalent formulation where one restricts $X \subseteq J(\aD) \subseteq D$.
}

\item
A lattice $\latL$ is distributive iff $M_3$ and $N_5$ do not arise as sublattices, recalling that:
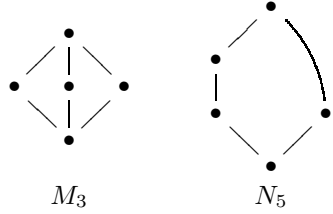
\[
\begin{tabular}{ccc}
$\vcenter{\vbox{\xymatrix@=10pt{
& \bullet
\\
\bullet \ar@{-}[ur] & \bullet \ar@{-}[u] & \bullet \ar@{-}[ul]
\\
& \bullet \ar@{-}[u] \ar@{-}[ul] \ar@{-}[ur]
}}}$
&&
$\vcenter{\vbox{\xymatrix@=10pt{
& \bullet
\\
\bullet \ar@{-}[ur] &&
\\
\bullet \ar@{-}[u] && \bullet \ar@{-}@/_5pt/[uul]
\\
& \bullet \ar@{-}[ul] \ar@{-}[ur]
}}}$
\\
$M_3$ && $N_5$
\end{tabular}
\]

\item
If $\aD$ is a finite distributive lattice and $h : D \to 2$ is any function, the following statements are equivalent:
\begin{enumerate}
\item
$h$ defines a bounded distributive lattice morphism of type $\aD \to \two$.
\item
$h^{-1}(\{1\}) = \; \up_\aD j$ for some $j \in J(\aD)$.
\item
$h^{-1}(\{0\}) = \; \down_\aD m$ for some $m \in M(\aD)$.
\end{enumerate}

\item
For every finite distributive lattice $\aD$ we have the bijective monotone and order-reflecting function:
\[
\begin{tabular}{lll}
$\tau_\aD : (J(\aD),\leq_\aD) \to (M(\aD),\leq_\aD)$
&
$\tau_\aD(j)$ & $:= \Lor_\aD \overline{\up_\aD j}$
\\
$\tau_\aD^{\bf-1} : (M(\aD),\leq_\aD) \to (J(\aD),\leq_\aD)$
&
$\tau_\aD^{\bf-1}(m)$ &  $:= \Land_\aD \overline{\down_\aD m}$
\end{tabular}
\]

\item
For every finite distributive lattice $\aD$ we have:
\[
\nleq_\aD |_{J(\aD) \times M(\aD)} = (\leq_{\aD^{\pOp}} |_{J(\aD) \times J(\aD)}) ; \tau_\aD
\]
where $\tau_\aD : J(\aD) \to M(\aD)$ is the canonical bijection from the previous statement. This is equivalent to either of the statements below:
\[
\forall j_1, j_2 \in J(\aD).( j_1 \nleq_\aD \tau_\aD(j_2) \iff j_2 \leq_\aD j_1)
\qquad
\forall j \in J(\aD), m \in M(\aD).( j \nleq_\aD m \iff \tau_\aD^{\bf-1}(m) \leq_\aD j).
\]

\item
Given any finite join-semilattice $\aQ$ then the following statements are equivalent:
\begin{enumerate}
\item
$\aQ$ is a distributive join-semilattice.
\item
$\aQ$ is a join-semilattice retract of a finite boolean join-semilattice.
\item
$\aQ$ is a join-semilattice retract of a finite distributive join-semilattice.
\end{enumerate}

\end{enumerate}
\end{lemma}

\begin{proof}
\item
\begin{enumerate}
\item
Given $\pP = (P,\leq_\pP)$ with all finite suprema we have the join-semilattice $\aQ := (P,\lor_\pP,\Lor_\pP \emptyset)$ where $\leq_\aQ \; = \; \leq_\pP$. Regarding uniqueness, given $\aQ = (Q,\lor_\aQ,\bot_\aQ)$ where $\leq_\aQ \; = \; \leq_\pP$, the ordering uniquely determines the finite joins.

\item
Same argument as previous item, noting that the ordering also uniquely determines the finite meets.

\item
Any finite lattice $\latL$ is bounded, so we can take its underlying join-semilattice structure $\aQ$ i.e.\ the binary join and the bottom element, which determines the ordering. Conversely any finite join-semilattice defines a lattice with the same ordering, because we have all finite joins, hence all joins, hence all meets too.

\item
If $\pP$ has finite length then any $p \leq_\pP p'$ is witnessed by a finite chain of covers $p = p_1 \prec_\pP \dots \prec_\pP p_n = p'$, so that $\leq_\pP$ is the reflexive transitive closure of $\prec_\pP$.

\item
Let $\latL$ be a locally finite lattice.
\begin{enumerate}[(a)]
\item
Let $j \in J(\latL)$ so that it covers precisely one element $x \in L$. If $j = \Lor_{i \in I} x_i$ for some (finite) set $I$ then there must exist $i \in I$ such that $x <_\latL x_i$, and since $x_i \leq_\latL j$ we deduce that $x_i = j$. Conversely, suppose we have $j \in L$ which is not the finite join of any other elements. Since $\latL$ is locally finite and has a bottom element $\bot_\latL$, we have the finite set $[\bot_\latL,j)$ whose join $x := \Lor_\latL [\bot_\latL,j)$ cannot be $j$, so that $x \prec_\latL j$. Given any $y \prec_\latL j$ then since $y <_\latL j$ we have $y \leq_\latL x$ so that $y = x$. We also have the inductive description i.e.\ instead of finite joins we consider empty joins and binary joins.

\item
Follows from (a) by order-duality, noting that $M(\latL) = J(\latL^{\pOp})$ in every lattice -- see Definition \ref{def:std_order_theory}.6.

\item
Every finite lattice (automically bounded) is locally finite, and every bounded locally finite lattice $\latL$ is finite since $L = [\bot_\latL,\top_\latL]$ is finite. Then both (a) and (b) apply.

\end{enumerate}

\item
Concerning irreducibles, instead of finite lattices $\latL$ we may consider finite join-semilattices $\aQ$. Recall the finite join-semilattice $\JPow J(\aQ) = (\Pow J(\aQ),\cup,\emptyset)$. There is a surjective join-semilattice morphism $f : (\Pow J(\aQ),\cup,\emptyset) \epito \aQ$ defined $f(S) := \Lor_\aQ S$, as we now verify. It is clearly a well-defined function, and also a morphism because:
\[
\begin{tabular}{c}
$f(\bot_{\JPow J(\aQ)}) = f(\emptyset) = \Lor_\aQ \emptyset = \bot_\aQ$
\\[1ex]
$f(S_1 \lor_{\JPow J(\aQ)} S_2) = f(S_1 \cup S_2) = \Lor_\aQ S_1 \cup S_2 = (\Lor_\aQ S_1) \lor_\aQ (\Lor_\aQ S_2) = f(S_1) \lor_\aQ f(S_2)$
\end{tabular}
\]
using generalised associativity. By standard universal algebra we obtain a sub join-semilattice $f[\Pow J(\aQ)] \subseteq \aQ$. Concerning surjectivity, suppose for a contradiction that we have some $q \in Q$ such that $q \notin f[\Pow J(\aQ)]$. Since the latter is the closure of $J(\aQ)$ under $\aQ$-joins, we deduce by (5).(a) that $q \in J(\aQ)$ which is a contradiction. Thus $J(\aQ)$ join-generates $\aQ$.

To see that every join-generating subset contains $J(\aQ)$, assume that $\aQ$ is join-generated by $S \subseteq Q$ i.e.\ $\ang{S}_\aQ = \aQ$. For a contradiction assume there exists $j \in J(\aQ)$ such that $j \notin S$. By (5).(a) we know that $j$ is not the join of other elements, yielding the contradiction $j \notin Q$. The statements involving meet-irreducibles follow by order-duality i.e.\ apply the above statements to $\aQ^{\pOp}$.

\item
The first equivalence follows because each $q \in Q$ is the join of those join-irreducibles below it by (6). The second equivalence follows by order-duality.

\item
Let $\aA$ be any possibly infinite boolean algebra. We have $At(\aA) \subseteq J(\aA)$ using the definitions. Take any join-irreducible $j \in J(\aA)$ and for a contradiction assume $j \notin At(\aA)$, hence $\bot <_\aA a \prec_\aA j$ for a unique $a \in A$. Observe that:
\[
(\neg_\aA a \land_\aA j) \lor_\aA a
= (\neg_\aA a \lor_\aA a) \land_\aA (j \lor_\aA a)
= \top_\aA \land_\aA j
= j
\]
using a distributive law, a complement law and a unit law. Because $j$ covers precisely one element we deduce that $j = \neg_\aA a \land_\aA j$ and hence $a \leq_\aA j \leq_\aA \neg_\aA a$, so that $a = a \land_\aA \neg_\aA a = \bot_\aA$, this being a contradiction.

\item
Let $\aD = (D,\lor_\aD,\bot_\aD,\land_\aD,\top_\aD)$ be a bounded distributive lattice.

\begin{enumerate}
\item
Given any $d \in D$ suppose we have two complements $d_1$, $d_2 \in D$, then:
\[
d_1 
= d_1 \land_\aD \top_\aD 
= d_1 \land_\aD (d \lor_\aD d_2)
= (d_1 \land_\aD d) \lor_\aD (d_1 \land_\aD d_2)
= \bot_\aD \lor_\aD (d_1 \land_\aD d_2)
= d_1 \land_\aD d_2
\]
so that $d_1 \leq_\aD d_2$, and by the symmetric argument $d_1 = d_2$.

\item
We'll show how to equationally deduce $\neg(x\land y) \approx \neg x \lor \neg y$. Firstly, $(x \land y) \land (\neg x \lor \neg y) \approx \bot$ is deduced  using a distributive law and a complement law twice followed by idempotence, whereas $(x \land y) \lor (\neg x \lor \neg y) \approx \top$ is deduced using the other distributive law and the other complement law twice followed by idempotence. One can then instantiate a general procedure i.e.\ the syntactic version of the uniqueness of complements in boolean lattices, having already specified their existence via the complement laws. Given $x \land y \approx \bot$ and $x \lor y  \approx \top$, one (i) applies $\neg x \lor (-)$ to deduce $\neg x \land y \approx \neg x$, (ii) applies $\neg x \land (-)$ to deduce $\neg x \lor y \approx \neg x$ and therefore $y \approx y \land (\neg x \lor y) \approx \neg x \land y$ using an absorption law. Combining (i) and (ii) yields $y \approx \neg x \land y \approx \neg x$. Then applying this general procedure to the earlier equalities we deduce $\neg(x\land y) \approx \neg x \lor \neg y$.
\end{enumerate}

\item

\begin{enumerate}[(i)]
\item
First suppose that (a) holds i.e.\ $\latL = \aD$ is a finite distributive lattice. We'll show that (b), (c) and (d) all hold and are equivalent.

Concerning (b), suppose $j \in J(\aD)$, $d_1,d_2 \in D$ are such that $j \leq_\aD d_1 \lor_\aD d_2$. By distributivity $j = j \land_\aD (d_1 \lor_\aD d_2) = (j \land_\aD d_1) \lor_\aD (j \land_\aD d_2)$, so by join-irreducibility $\exists i.\, j = j \land_\aD d_i$ and hence $j \leq_\aD d_i$, as required.  The more general formulation (c) involving any subset $X \subseteq D$ follows by induction:
\begin{enumerate}
\item
If $X = \emptyset$ recall that $j \neq_\aD \bot_\aD$ by definition.
\item
If $X = \{d\} \cup Y$ where $|Y| < |X|$ then we have $j \leq_\aD d \lor_\aD \Lor_\aD Y$, so that either $j \leq_\aD d$ and we are done, or $j \leq_\aD \Lor_\aD Y$ and we may apply induction.
\end{enumerate}

Finally, (c) is equivalent to the more specific condition (d) i.e.\  we restrict to subsets $X \subseteq J(\aD) \subseteq D$. This follows because:
\[
\Lor_\aD X = \Lor_\aD J(\aD) \; \cap \down_\aD X
\qquad
\text{for any subset $X \subseteq D$}
\]
i.e.\ every element in a finite join-semilattice arises as the join of those join-irreducibles below it.

\item
The proof of the previous item shows that (b), (c) and (d) are all equivalent, even without knowing that the finite lattice $\latL$ is distributive. Then it suffices to show that (b) implies (a). We'll achieve this by embedding $\latL$ into a set-theoretic bounded distributive lattice, recalling that sublattices of bounded distributive lattices are distributive. So define:
\[
e : \latL \to (\Pow J(\latL),\cup,\emptyset,\cap,J(\latL))
\text{\quad with action \quad}
e(x) :=  J(\latL) \; \cap \down_\latL X.
\]
This is a well-defined function and also injective because elements of a finite (join-semi)lattice arise as the join of those join-irreducibles beneath them. We  have $e(\bot_\latL) = \emptyset$ and $e(\top_\latL) = J(\latL)$, and also $e(x_1 \land_\latL x_2) = e(x_1) \cap e(x_2)$ by virtue of the defining property of meets. Finally we use (b) to show preservation of joins:
\[
e(x_1 \lor_\latL x_2)
= \{ j \in J(\latL) : j \leq_\latL x_1 \lor_\latL x_2 \}
\stackrel{!}{=} \{ j \in J(\latL) : \exists i. j \leq_\latL x_i \}
= e(x_1) \cup e(x_2).
\]

\end{enumerate}

\item
See \cite[Chapter II, Theorem 1]{GratzerGeneralLattice1998}.

\item
$\mathrm{(a)} \implies \mathrm{(b)}$.

Assume (a) and let $X := h^{-1}(\{1\})$. Then $X$ is upwards-closed (since $h$ is monotone) and closed under meets (since $h$ preserves meets), so that  $X = \; \up_\aD \Land_\aD X$ by finiteness. Moreover $d := \Land_\aD X \in J(\aD)$ because if $d = d_1 \lor_\aD d_2$ then $h(d_1) \lor_\two h(d_2) = h(d_1 \lor_\aD d_2) = h(d) = 1$, so that $h(d_i) = 1$ for some $i$, hence $d_i \leq_\aD d_1 \lor d_2 = d \leq_\aD d_i$.

\smallskip
$\mathrm{(a)} \implies \mathrm{(c)}$.

Given $h : \aD \to \two$ then we also have the bounded distributive lattice morphism $h^{\pOp} : \aD^{\pOp} \to \two^{\pOp}$ with the very same action. Thus we also have $g := \swap \circ h^{\pOp} : \aD^{\pOp} \to \two$. Applying the previous argument we deduce that $g^{-1}(\{1\}) = \; \up_{\aD^{\pOp}} j = \; \down_\aD j$ for some $j \in J(\aD^{\pOp}) = M(\aD)$. Finally observe that $h^{-1}(\{0\}) = g^{-1}(\{1\})$.

\smallskip
$\mathrm{(b)} \implies \mathrm{(a)}$.

We have a function $h : D \to 2$ such that $h(d) = 1$ iff $j \leq_\aD d$ where $j \in J(\aD)$. Firstly $h(\bot_\aD) = 0$ because $j$ is join-irreducible, and also $h(\top_\aD) = 1$ by virtue of being the top element. Moreover $h(d_1 \land_\aD d_2) = 1$ iff $j \leq_\aD d_1 \land_\aD d_2$ iff $\forall i.j \leq_\aD d_i$ iff $\forall i. h(d_i) = 1$. Finally $h(d_1 \lor d_2) = 1$ iff $h(d_1) = 1$ or $h(d_2) = 1$ by Lemma \ref{lem:std_order_theory}.10. 

\smallskip
$\mathrm{(c)} \implies \mathrm{(a)}$.

We have a function $h : D \to 2$ such that $h(d) = 0$ iff $d \leq_\aD m$. Therefore $g := \swap \circ h : D \to 2$ is such that $g(d) = 1$ iff $m \leq_{\aD^{\pOp}} d$ where $m \in J(\aD^{\pOp})$. By the previous statement we deduce that $g$ has type $\aD^{\pOp} \to \two$, so that $h = \swap \circ g^{\pOp}$ has type $\aD \to \two$.

\item
The functions are well-defined, bijective and the inverse of each other by the previous statement, since distinct join/meet-irreducibles have distinct principal up/downsets. Finally,
\[
j_1 \leq_\aD j_2
\iff \up_\aD j_2 \; \subseteq \; \up_\aD j_1
\iff \overline{\up_\aD j_1} \; \subseteq \; \overline{\up_\aD j_2}
\iff \Lor_\aD \overline{\up_\aD j_1} \leq_\aD \Lor_\aD \overline{\up_\aD j_2}
\iff \tau_\aD(j_1) \leq_\aD \tau_\aD(j_2).
\]

\item
The following calculation:
\[
\begin{tabular}{lll}
$\nleq_\aD |_{J(\aD) \times M(\aD)} \, ;\, \tau_\aD^{\bf-1}(j_1,j_2)$
&
$\iff \exists m \in M(\aD).( j_1 \nleq_\aD m \text{ and } j_2 = \tau_\aD^{\bf-1}(m) )$
\\&
$\iff \exists m \in M(\aD).( j_1 \nleq_\aD m \text{ and } \tau_\aD(j_2) = m)$
\\&
$\iff j_1 \nleq_\aD \tau_\aD(j_2)$
\\&
$\iff \neg(j_1 \leq_\aD \Lor_\aD \overline{\up_\aD j_2})$
\\&
$\iff \neg(j_1 \leq_\aD \Lor_\aD \{ d \in D : j_2 \nleq_\aD d \})$
\\&
$\iff \neg\exists d \in D.(j_1 \leq_\aD d \text{ and } j_2 \nleq_\aD d )$
& $j_1$ is join-prime
\\&
$\iff \forall d \in D.( j_1 \leq_\aD d \To j_2 \leq_\aD d)$
\\&
$\iff j_2 \leq_\aD j_1$
& by Lemma \ref{lem:std_order_theory}.7
\\&
$\iff \leq_{\aD^{\pOp}} |_{J(\aD) \times J(\aD)} (j_1,j_2)$
\end{tabular}
\]
proves that:
\[
\nleq_\aD |_{J(\aD) \times M(\aD)} \, ; \, \tau_\aD^{\bf-1}
= \; \leq_{\aD^{\pOp}} |_{J(\aD) \times J(\aD)}
\]
so post-composing the bijection yields:
\[
\nleq_\aD |_{J(\aD) \times M(\aD)}
= \; \leq_{\aD^{\pOp}} |_{J(\aD) \times J(\aD)} \, ; \, \tau_{\aD}.
\]

\item
Let $\aQ$ be a finite join-semilattice. 

\begin{enumerate}
\item
If $\aQ$ is distributive then we have the retract:
\[
\begin{tabular}{ll}
$e : \aQ \monoto \JPow J(\aQ)$
& $r : \JPow J(\aQ) \epito \aQ$
\\
$e(a) := \{ j \in J(\aQ) : j \leq_\aQ q \}$
& $r(S) := \Lor_\aQ S$
\\[0.5ex]
\end{tabular}
\]
That $e$ is well-defined follows because join-irreducibles in finite distributive lattices are join-prime (see (10) above), and $r$ is well-defined by freeness of $\JPow J(\aQ)$ (or is easily directly verified). That $r \circ e = id_\aQ$ follows because elements of a finite join-semilattice are the join of those join-irreducibles beneath them.

\item
To finish the proof, it suffices to establish that:
\begin{quote}
if $\aQ$ is a join-semilattice retract of a finite distributive join-semilattice $\aR$ then it is distributive.
\end{quote}
By assumption $r \circ e = id_\aQ$ for some injective morphism $e : \aQ \to \aR$ and surjective morphism $r : \aR \to \aQ$. If $S \subseteq R$ is the closure of $e[Q]$ under binary $\aR$-meets, then $S$ is also closed under binary $\aR$-joins because $e[Q]$ is closed under them and $\aR$ is distributive. Now, since $\bot_\aR = e(\bot_\aQ) \in e[Q] \subseteq S$, it follows that $S$ defines a sub join-semilattice $\iota : \aS \hookto \aR$ which is also a sublattice of $(R,\lor_\aR,\land_\aR)$ and hence distributive. Since the join-semilattice morphism $r' := r \circ \iota : \aS \to \aQ$ is surjective because $e[Q] \subseteq S$, it suffices to establish that $r'$ preserves binary meets.

\smallskip
Given any $s_1,\,s_2 \in S$ then by construction there exist subsets $X_i \subseteq Q$ such that $s_i = \Land_\aR e[X_i]$ for $i = 1,2$. Now, for any subset $X \subseteq Q$ and element $x \in X$ we have:
\[
\Land_\aQ X = r \circ e(\Land_\aQ X)
\leq_\aQ r(\Land_\aR e[X])
\leq_\aQ r \circ e(x) = x
\]
using monotonicity, and since $x$ is arbitrary it follows that $r(\Land_\aR e[X]) = \Land_\aQ X$. Finally,
\[
\begin{tabular}{lll}
$r'(s_1 \land_\aS s_2)$
&
$= r(s_1 \land_\aR s_2)$
& $r'$ restricts $r$, $\land_\aS$ restricts $\land_\aR$
\\&
$= r(\Land_\aR e[X_1] \land_\aR \Land_\aR e[X_2])$
& 
\\&
$= r(\Land_\aR e[X_1 \cup X_2])$
& associativity
\\&
$= \Land_\aQ (X_1 \cup X_2)$
& see above
\\&
$= \Land_\aQ X_1 \land_\aQ \Land_\aQ X_2$
& associativity
\\&
$= r(\Land_\aQ e[X_1]) \land_\aQ r(\Land_\aQ e[X_2])$
& see above
\\&
$= r'(s_1) \land_\aQ r'(s_2)$
& 
\end{tabular}
\]

\end{enumerate}
\end{enumerate}
\end{proof}

We shall also need the concepts of a closure operator and interior operator on a poset.

\begin{definition}
\label{def:cl_inte}
Let $\pP = (P,\leq_\pP)$ be any poset.
\begin{enumerate}
\item
A \emph{closure operator on $\pP$} is a function $\cl : P \to P$ such that:
\begin{enumerate}
\item
$x \leq_\pP \cl(x)$,
\item
$x \leq_\pP y$ implies $\cl(x) \leq_\pP \cl(y)$,
\item
$\cl \circ \cl = \cl$.
\end{enumerate}
for all $x,y \in P$. That is, a closure operator is a monotone endomorphism $\pP \to \pP$ which is \emph{extensive} (property 1) and idempotent (property 3). Its fixpoints are those $P$ where $\cl[P] \subseteq P$ and are called the \emph{closed elements}.

\item
An \emph{interior operator on $\pP$} is a function $\inte : P \to P$ such that:
\begin{enumerate}
\item
$\inte(x) \leq_\pP x$,
\item
$x \leq_\pP y$ implies $\inte(x) \leq_\pP \inte(y)$,
\item
$\inte \circ \inte = \inte$.
\end{enumerate}
Only the first property is different: the \emph{co-extensive} property. Fixpoints of $\inte$ are those where $\inte[P] \subseteq P$ and are called the \emph{open} elements.
\end{enumerate}
\end{definition}

Observe that a closure operator on $\pP$ is really the same thing as an interior operator on $\pP^{\pOp}$.

\begin{lemma}[Open and closed elements as substructures]
\label{lem:clo_co_basic}
\item
Take any closure operator $\cl : \pP \to \pP$ and interior operator $\inte : \pP \to \pP$.
\begin{enumerate}
\item
$\cl[P]$ is closed under all meets that exist in $\pP$. In particular, it contains $\top_\pP$ if the latter exists.
\item
If $\pP$ has a lattice structure $\latL$ then $\cl[P]$ forms a sub-join-semilattice of $\latL^{\pOp}$. Moreover, if $\pP$ is finite and has a join-semilattice structure $\aQ$ then $\cl[P]$ is the carrier of a sub join-semilattice of $\aQ^{\pOp}$.
\item
$\inte[P]$ is closed under all joins that exist in $\pP$. In particular, it contains $\bot_\pP$ if the latter exists.
\item
If $\pP$ has a lattice structure $\latL$ then $\inte[P]$ forms a sub-join-semilattice of $\latL$. Moreover, if $\pP$ is finite and has a join-semilattice structure $\aQ$ then $\inte[P]$ is the carrier of a sub join-semilattice of $\aQ$.
\end{enumerate}
\end{lemma}

\begin{proof}
Regarding the first statement, suppose that $x_i \in \cl[\pP]$ for all $i \in I$ and that $z = \Land_{i \in I} x_i$ exists in $\pP$. Then $\cl(z) \leq_\pP \cl(x_i)$ for all $i \in I$ by monotonicity of $\cl$. Thus $\cl(z) \leq_\pP \Land_{i\in I} x_i = z$, so by the extensivity of $\cl$ we deduce that $\cl(z) = z$. The second statement follows immediately from the first. The final statements are order-duals of the first two.
\end{proof}

\begin{note}[Closure/interior operators needn't preserve meets/joins]
\item
Although the meet of closed sets is closed whenever it exists, $\cl$  \emph{needn't preserve meets} i.e.\ we may have $\cl(x \land y) \neq \cl(x) \land \cl(y)$. For example, let $X = \{x,y_1,y_2,z\}$ have four distinct elements, take the two binary relations $\rR_i = \{(x,y_i),(y_i,z)\} \subseteq X \times X$ for $i = 1,2$, and let $\cl$ construct the transitive closure on the respective inclusion-ordered lattice of binary relations. Then although $\cl(\rR_1 \cap \rR_2) = \emptyset$ and $\cl(\rR_1) \cap \cl(\rR_2) = \{(x,z)\}$ are both closed under transitivity, they are not equal. Of course, this also means that interior operators needn't preserve joins. \endbox
\end{note}

\begin{lemma}[Properties of adjoint monotone morphisms]
\label{lem:adjoint_cl_in}
\item
Given two monotone functions $f : \pP \to \pQ$ and $g : \pQ \to \pP$ such that:
\[
f(p) \leq_\pQ q
\iff
p \leq_\pP g(q)
\qquad
\text{for all $p \in P$ and $q \in Q$}
\]
then the following statements hold:
\begin{enumerate}
\item
$g \circ f : \pP \to \pP$ is a closure operator and $f \circ g : \pQ \to \pQ$ is an interior operator.

\item
For any subset $X \subseteq P$ such that $\Lor_\pP X$ exists in $\pP$, we have:
\[
f(\Lor_\pP X) = \Lor_\pQ f[X]
\]
so in particular the latter join exists in $\pQ$.

\item
For any subset $Y \subseteq Q$ such that $\Land_\pQ Y$ exists in $\pQ$, we have:
\[
g(\Land_\pQ Y) = \Land_\pP g[Y]
\]
so in particular the latter meet exists in $\pP$.

\end{enumerate}
\end{lemma}

\begin{proof}
\item
\begin{enumerate}
\item
Defining $\cl := g \circ f$ then it is certainly a monotone morphism $\pP \to \pP$. Regarding extensivity, $p \leq_\pP \cl(p)$ iff $p \leq_\pP g(f(p))$ iff $f(p) \leq_\pQ f(p)$ and hence always holds. Similarly $f \circ g$ is co-extensive because $f \circ g(q) \leq_\aQ q$ iff $g(q) \leq_\aQ g(q)$. Regarding the idempotence of $\cl$, applying monotonicity to extensivity yields $\cl(p) \leq_\pP \cl \circ \cl(p)$ and the converse follows by the co-extensivity of $f \circ g$ and the monotonicity of $g$:
\[
g \circ f \circ g \circ f(p)
= g(f \circ g(f(p)))
\leq_\pP g (f(p)).
\]
Thus $\cl$ is a well-defined closure operator. Regarding the interior operator $\inte := f \circ g$, since $q \leq_{\pQ^{\pOp}} f(p)$ iff $g(q) \leq_{\pP^{\pOp}} p$ we can apply the above argument to deduce that $f^{\pOp} \circ g^{\pOp} : \pQ^{\pOp} \to \pQ^{\pOp}$ is a closure operator on $\pQ^{\pOp}$, hence $\inte$ is an interior operator on $\pQ$.

\item
Given any $X \subseteq P$ such that $\Lor_\pP X$ exists in $\pP$ we are going to show that $\Lor_\pQ f[X]$ exists in $\pQ$, and in fact equals $q_X := f(\Lor_\pP X)$. For every $x \in X$ we have $f(x) \leq_\pQ q_X$ by monotonicity i.e.\ it is an upper-bound for $f[X] \subseteq Q$. Given any other $q_0 \in Q$ such that $\forall x \in X. f(x) \leq_\pQ q_0$, then by adjointness we have $x \leq_\pP g(q_0)$ and hence $\Lor_\pP X \leq_\pP g(q_0)$. Applying monotonicity and the co-extensitivity of $f \circ g$ proved in (1), we deduce that:
\[
q_X = f(\Lor_\pP X) \leq_\pQ f(g(q_0)) \leq_\pQ q_0.
\]

\item
This follows from (2) via order-duality i.e.\ $g(q) \leq_{\pP^{\pOp}} p \iff q \leq_{\aQ^{\pOp}} f(p)$ for every $(p,q) \in P \times Q$.

\end{enumerate}
\end{proof}

\begin{note}
This instantiates a well-known categorical result i.e.\ given an adjunction $G \vdash F : \Cat \to \DCat$ between categories then  $G \circ F$ is the functorial component of a \emph{monad} (closure operator) and $F \circ G$ is the functorial component of a \emph{comonad} (interior operator). \endbox
\end{note}

\section{Finite join-semilattices and their self-duality}

\begin{definition}[The category of finite join-semilattices]
$\JSL_f$ is the category whose objects are the finite join-semilattices $\aQ = (Q,\lor_\aQ,\bot_\aQ)$ and whose morphisms $f : \aQ \to \aR$ are the join-semilattice morphisms between them, see Definition \ref{def:std_order_theory}.10. Composition is the usual functional composition, and the identity morphism $id_\aQ$ is the identity function $\Delta_Q$. \endbox
\end{definition}

Thus $\JSL_f$ consists of all finite join-semilattices with its usual algebra homomorphisms. Viewing the finite join-semilattices as the finite commutative and idempotent monoids, the latter are precisely the monoid morphisms. Alternatively  they may be described as those functions preserving \emph{all} joins i.e.\ such that:
\[
f(\Lor_\aQ S) = \Lor_\aR f[S]
\qquad
\text{(for all $S \subseteq Q$ )}
\]
due to the finiteness of the join-semilattices involved. We are going to describe the self-duality of $\JSL_f$, which is of fundamental importance to our approach. It restricts two distinct dualities:

\begin{enumerate}
\item
The self-duality of complete join-semilattices i.e.\ complete lattices equipped with those functions which preserve all joins. This variety (with infinitary signature) consists of the Eilenberg-Moore algebras for the powerset functor $\Pow : \Set \to \Set$ \cite{CatsMaclane71}.

\item
The Stone-type duality between the variety of join-semilattices with bottom (with finitary signature) and the Stone topological join-semilattices \cite{StoneSpaces}.
\end{enumerate}

But it also follows directly from the adjoint functor theorem restricted to posets i.e.\ thin skeletal categories. Each finite join-semilattice has a bottom element and binary joins, hence all joins by finiteness, hence all meets by completeness. Each $\JSL_f$-morphism $f : \aQ \to \aR$ preserves all joins (= colimits), thus by the adjoint functor theorem it has a unique left adjoint i.e.\ a function  $f_* : R \to Q$ such that:
\[
f(q) \leq_\aR r
\;\iff\;
q \leq_\aQ f_*(r)
\qquad
(\text{for all $q \in Q$, $r \in R$})
\]
which preserves all meets i.e.\ sends meets in $\aR$ to meets in $\aQ$. It follows that:
\[
f_*(r) = \Lor_\aQ f^{-1}(\down_\aR r) = \Lor_\aQ \{ q \in Q : f(q) \leq_\aR r \}
\qquad
\text{(for all $r \in R$)}
\]
and defines a $\JSL_f$-morphism of type $\aR^{\pOp} \to \aQ^{\pOp}$.

\begin{theorem}
\label{thm:jsl_self_dual}
$\JSL_f$ is self-dual via the functor $\OD_j : \JSL_f^{op} \to \JSL_f$ defined:
\[
\OD_j\aQ := \aQ^{\pOp}
\qquad
\dfrac
{f: \aQ \to \aR}
{\OD_j f^{op} := \lambda r \in R.\Lor_\aQ f^{-1}(\down_\aR r) : \aR^{\pOp} \to \aQ^{\pOp}}
\]
with natural isomorphism $rep : Id_{\JSL_f} \To \OD_j \circ \OD_j^{op}$ whose components are the identity morphisms $rep_\aQ := id_\aQ$.
\end{theorem}

\begin{proof}
We first verify that $\OD_j$ is a well-defined functor. Recall that each finite join-semilattice $\aQ = (Q,\lor_\aQ,\bot_\aQ)$ defines a finite lattice $\latL = (Q,\lor_\aQ,\bot_\aQ,\land_\aQ,\top_\aQ)$ in a unique fashion, so that $\OD_j\aQ = \aQ^{\pOp} = (Q,\land_\aQ,\top_\aQ)$ is a well-defined finite join-semilattice. We have already explained why $\OD_j f^{op} = f_*$ is a well-defined $\JSL_f$-morphism of type $\aR^{\pOp} \to \aQ^{\pOp}$, but let us directly verify this anyway. It is certainly a well-defined function because all joins exist in $\aQ$, so let us verify the `adjointness' i.e.\ $f(q) \leq_\aR r$ iff $q \leq_\aQ f_*(r)$, for any $q \in Q$ and $r \in R$.
\begin{enumerate}
\item
$(\To)$ by definition of $f_*$.
\item
$(\oT)$ because if $q$ is the join of all $q_i \in Q$ such that $f(q_i) \leq_\aR r$ then $f(q) = f(\Lor_i q_i) = \Lor_i f(q_i) \leq_\aR r$.
\end{enumerate}
We now use this to verify preservation of the bottom element and binary join:
\[
f_*(\bot_{\aR^{\pOp}}) 
= f(\top_\aR) 
= \Lor_\aQ f^{-1}(\down_\aR \top_\aR) 
= \Lor_\aQ f^{-1} (R)
= \Lor_\aQ Q
= \top_\aQ
= \bot_{\aQ^{\pOp}}
\]
\[
\begin{tabular}{lll}
$f_*(r_1 \lor_{\aR^{\pOp}} r_2)$
&
$= f_*(r_1 \land_\aR r_2)$
\\&
$= \Lor_\aQ f^{-1}(\down_\aR r_1 \land_\aQ r_2)$
\\&
$= \Lor_\aQ f^{-1}(\down_\aR r_1 \, \cap \down_\aR r_2)$
& property of binary meets in posets
\\&
$= \Lor_\aQ f^{-1}(\down_\aR r_1) \cap f^{-1}(\down_\aR r_2)$
& $f^{-1}$ preserves intersections
\\&
$= \Lor_\aQ \{ q \in Q : f(q) \leq_\aR  r_1, r_2 \}$
\\&
$= \Lor_\aQ \{ q \in Q : q \leq_\aQ  f_*(r_1), f_*(r_2) \}$
& applying adjoint relationship
\\&
$= f_*(r_1) \land_\aQ f_*(r_2)$
& induced meet in join-semilattice
\\&
$= f_*(r_1) \lor_{\aQ^{\pOp}} f_*(r_2)$
\end{tabular}
\]
Thus $\OD_j$'s action is well-defined. Regarding the preservation of identity morphisms:
\[
\OD_j id_\aQ^{op} 
= (id_\aQ)_*
= \lambda q \in Q.\Lor_\aQ \Delta_Q^{-1}(\down_\aQ q)
= \lambda q \in Q.\Lor_\aQ \down_\aQ q
= \lambda q \in Q.q
= id_{\OD_j \aQ}
\]
Regarding composition of morphisms, given $f : \aQ \to \aR$ and $g : \aR \to \aS$ then:
\[
\begin{tabular}{lll}
$\OD_j (f^{op} \circ g^{op})$
&
$= \OD_j (g \circ f)^{op}$
\\&
$=(g \circ f)_*$
\\&
$= \lambda s \in S. \Lor_\aQ (g \circ f)^{-1}(\down_\aS s)$
\\&
$= \lambda s.\Lor_\aQ f^{-1} \circ g^{-1} (\down_\aS s)$
& functoriality of preimage
\\&
$= \lambda s.\Lor_\aQ f^{-1} (\down_\aR \Lor_\aR g^{-1}(\down_\aS s) )$
& $g^{-1}(\down_\aS s)$ down-closed, join-closed in $\aR$
\\&
$= \lambda s. f_*(\down_\aR \Lor_\aR g^{-1}(\down_\aS s))$
\\&
$= f_* \circ g_*$
\\&
$= \OD_j f^{op} \circ \OD_j g^{op}$
\end{tabular}
\]
Finally let us verify that $rep : \Id_{\JSL_f} \To \OD_j \circ \OD_j^{op}$ where $rep_\aQ = id_\aQ$ is a natural isomorphism i.e.\ for all morphisms $f : \aQ \to \aR$ we must show the following diagram commutes:
\[
\xymatrix@=15pt{
\aQ \ar[rr]^{rep_\aQ} \ar[d]_f && \aQ \ar[d]^{\OD_j \OD_j^{op} f}
\\
\aR \ar[rr]_{rep_\aR} && \aR
}
\]
or equivalently that $(f_*)_* = f$. This already follows from the uniqueness of adjoints, but we'll verify it anyway:
\[
\begin{tabular}{lll}
$(f_*)_*$
&
$= \lambda q \in Q. \Lor_{\aR^{\pOp}} (f_*)^{-1}(\down_{\aQ^{\pOp}} q)$
\\&
$= \lambda q \in Q. \Land_\aR \{ r \in R : q \leq_\aQ  f_*(r) \}$
\\&
$= \lambda q \in Q. \Land_\aR \{ r \in R : f(q) \leq_\aR r \}$
& by adjoint relationship
\\ &
$= \lambda q \in Q. f(q)$
\\ &
$= f$
\end{tabular}
\]
\end{proof}

Let us make a few basic observations concerning these adjoint morphisms.


\begin{lemma}
\label{lem:adj_obs}
\item
\begin{enumerate}
\item
For all $\JSL_f$-morphisms $f : \aQ \to \aR$ we have the adjoint relationship:
\[
f(q) \leq_\aR r
\iff q \leq_\aQ f_*(r)
\qquad
\text{(for every $q \in Q$ and $r \in R$)}
\]

\item
The adjoint of an isomorphism between finite join-semilattices acts like its inverse. That is, if $f : \aQ \to \aR$ is a $\JSL_f$-isomorphism then $f_* = (f^{\bf-1})^{\pOp}$, where it is permissible to take the order-dual monotone mapping because we are dealing with bounded lattice isomorphisms. Moreover:
\[
(f^{\bf-1})_* = (f_*)^{\bf-1} = f^{\pOp}
\]

\item
The image function is the adjoint of the preimage function. That is, for any function $f : X \to Y$ between finite sets, the adjoint of $\Pow f : (\Pow X,\cup,\emptyset) \to (\Pow Y,\cup,\emptyset)$ is the preimage $f^{-1} : (\Pow Y,\cap,Y) \to (\Pow X,\cap,X)$.

\item
Each $\JSL_f(\aQ,\aR)$ admits a join-semilattice structure, the ordering $\leq_{(\aQ,\aR)}$ being the pointwise-ordering. The mapping $f \mapsto f_*$ defines a $\JSL_f$-isomorphism from $(\JSL_f(\aQ,\aR),\leq_{(\aQ,\aR)})$ to $(\JSL_f(\aR^{\pOp},\aQ^{\pOp}),\leq_{(\aR^{\pOp},\aQ^{\pOp})})$.

\item
A $\JSL_f$-morphism is a section iff its adjoint is a retract.

\end{enumerate}
\end{lemma}

\begin{proof}
\item
\begin{enumerate}
\item
See the proof of Theorem \ref{thm:jsl_self_dual}.

\item
Given a $\JSL_f$-isomorphism $f : \aQ \to \aR$ then the adjoint $f_* : \aR^{\pOp} \to \aQ^{\pOp}$ has action:
\[
f_*(r) 
= \Lor_\aQ f^{-1}(\down_\aR r)
= \Lor_\aQ \down_\aQ f^{\bf-1}(r)
= f^{\bf-1}(r)
\]
using the fact that $f^{\bf-1} : \aR \to \aQ$ is a monotone bijection. Then it has the same typing and action as $(f^{\bf-1})^{\pOp}$, so these are the same $\JSL_f$-morphisms. Consequently:
\[
(f_*)^{\bf-1} 
= ((f^{\bf-1})^{\pOp})^{\bf-1}
= ((f^{\bf-1})^{\bf-1})^{\pOp}
= f^{\pOp}
\qquad
(f^{\bf-1})_*
= ((f^{\bf-1})^{\bf-1})^{\pOp}
= f^{\pOp}
\]
where in the left derivation we have used the general fact that inverses commute with $(-)^{\pOp}$.

\item
We calculate:
\[
\begin{tabular}{ll}
$(\Pow f)_* (S)$
&
$= \Lor_{(\Pow X,\cup,\emptyset)} (\Pow f)^{-1}(\down_{(\Pow Y,\cup,\emptyset)} S)$
\\&
$= \bigcup (\Pow f)^{-1}(\{ K \subseteq Y : Z \subseteq S\})$
\\&
$= \bigcup \{ K \subseteq X : f[X] \subseteq S \}$
\\&
$= f^{-1}(S)$
\end{tabular}
\]
since the $f$-preimage of $S$ is the largest subset whose image under $f$ lies inside $S$.

\item
The bottom element is $\lambda r \in R. \bot_\aR$, and the pointwise-join of morphisms is again a $\JSL_f$-morphism:
\[
\begin{tabular}{ll}
$f_1 \lor_{(\aQ,\aR)} f_2 (\bot_\aQ)$ 
&
$= f_1(\bot_\aQ) \lor_\aR f_2(\bot_\aR) 
= \bot_\aR \lor_\aR \bot_\aR 
= \bot_\aR$
\\\\
$f_1 \lor_{(\aQ,\aR)} f_2(q_1 \lor_\aQ q_2)$
&
$= f_1(q_1 \lor_\aQ q_2) \lor_\aR f_2(q_1 \lor_\aQ q_2)$
\\ &
$= f_1(q_1) \lor_\aR f_1(q_2) \lor_\aR f_2(q_1) \lor_\aR f_2(q_2)$
\\ &
$= f_1(q_1) \lor_\aR f_2(q_1) \lor_\aR f_1(q_2) \lor_\aR f_2(q_2)$
\\ &
$= f_1 \lor_{(\aQ,\aR)} f_2(q_1) \, \lor_\aR \, f_1 \lor_{(\aQ,\aR)} f_2(q_2)$
\end{tabular}
\]
The mapping $f \mapsto f_*$ is bijective by the self-duality theorem. Given $f \leq_{(\aQ,\aR)} g$ we first show that $f_* \leq_{(\aR^{\pOp},\aQ^{\pOp})} g_*$. Given any $r \in R$, then $g_*(r)$ is the $\aQ$-join of all elements $a \in Q$ such that $g(a) \leq_\aR r$, and since $f(a) \leq_\aQ g(a) \leq_\aR r$ we deduce that $f_*(r)$ is the $\aQ$-join of a larger set. Therefore $g_*(r) \leq_\aQ f_*(r)$ and thus $f_*(r) \leq_{\aQ^{\pOp}} g_*(r)$, and since $r$ was arbitrary we have $f_* \leq_{(\aR^{\pOp},\aQ^{\pOp})} g_*$. This proves monotonicity. Order-reflection follows by applying the adjunction in the opposite direction i.e.\ by the same argument $f_* \leq_{(\aR^{\pOp},\aQ^{\pOp})} g_*$ implies that $f = (f_*)_* \leq_{(\aQ,\aR)} (g_*)_* = g$ using the naturality of $rep$.

\item
Recall that an algebra morphism $s : \aQ \to \aR$ is a section (resp.\ $r : \aR \to \aQ$ is a retract) iff there exists an algebra morphism $r : \aR \to \aS$ (resp.\ $s : \aS \to \aR$) such that $r \circ s = id_\aQ$. Then since $s_* \circ r_* = (r \circ s)_* = (id_{\aQ})_* = id_{\aQ^{\pOp}}$ the statement is clear.

\end{enumerate}
\end{proof}

\takeout{
\begin{note}
Regarding the second point, a certain natural generalisation of it fails. That is, given $f : \aQ \to \aR$ and a domain/codomain restriction $g : \aS \to \aT$  which is a join-semilattice isomorphism, then it needn't be the case that $g^{\bf-1}$ restricts $f_*$. A counterexample is obtained by considering the example given in the proof of Lemma  \ref{lem:jsl_mor_inj_fail}.2 below. That is, it restricts to an isomorphism of a four element boolean lattice but its adjoint does not restrict to the inverse of this isomorphism.
\end{note}
}

The fourth point above and its proof lead naturally to the following standard definition.

\begin{definition}
\label{def:hom_functor_jsl}
To any two finite join-semilattices $\aQ$ and $\aR$ we associate the finite join-semilattice: \[
\begin{tabular}{c}
$\JSL_f[\aQ,\aR] = (\JSL_f(\aQ,\aR),\lor_{(\aQ,\aR)},\bot_{(\aQ,\aR)})$
\\[0.5ex]
where $f_1 \lor_{(\aQ,\aR)} f_2 := \lambda q \in Q. f_1(q) \lor_\aR f_2(q)$ and also $\bot_{(\aQ,\aR)} := \lambda q \in Q. \bot_\aR$.
\end{tabular}
\]
That is, the join is the pointwise-join and the bottom is the constantly bottom map. \endbox
\end{definition}

\begin{lemma}
\label{lem:jsl_mor_iso}
\item
\begin{enumerate}
\item
The self-duality of join-semilattices restricts to a join-semilattice isomorphism:
\[
\JSL_f[\aQ,\aR] \xto{(-)_*} \JSL_f[\aR^{\pOp},\aQ^{\pOp}]
\]
for each $\aQ$, $\aR \in \JSL_f$.
\item
Any $\JSL_f$-morphism $\theta : \aR \to \aS$ induces a join-semilattice morphism:
\[
\JSL_f[\aQ,\aR] \xto{\theta \circ (-)} \JSL_f[\aQ,\aS]
\qquad\mathrm{with\;action}\qquad
g \mapsto \theta \circ g
\]
\end{enumerate}
\end{lemma}

\begin{proof}
The first statement follows from the statement and proof of Lemma \ref{lem:adj_obs}.4 above. Regarding the second statement, we certainly have a well-defined function and:
\[
\begin{tabular}{llll}
$\theta \circ (-)(\bot_{\JSL_f[\aQ,\aR]})$
&
$= \theta \circ (\lambda q \in Q.\bot_\aR)$
&\quad
$\theta \circ (-)(g \lor_{\JSL_f[\aQ,\aR]} f)$
& $= \theta \circ (\lambda q \in Q. g(q) \lor_\aR f(q))$
\\&
$= \lambda q \in Q.\theta(\bot_\aR)$
&&
$= \lambda q \in Q. (\theta \circ g(q) \lor_\aS \theta \circ f(q))$
\\&
$= \lambda q \in Q. \bot_\aS$
&&
$= \theta \circ g \lor_{\JSL_f[\aQ,\aS]} \theta \circ f$
\\&
$= \bot_{\JSL_f[\aQ,\aS]}$
&&
$= \theta \circ (-)(g) \lor_{\JSL_f[\aQ,\aS]} \theta \circ (-)(f)$
\end{tabular}
\]
\end{proof}

Recall that the elements of a possibly infinite join-semilattice $\aQ$ biject with the join-semilattice morphisms of type $\two \to \aQ$ i.e.\ consider the action on the latter on $\top_\two = 1$. Restricting to the finite level, the self-duality yields a correspondence between elements of $\aQ$ and the ideals $\ideal{\aQ^{op}}{q}$ i.e.\ the morphisms of type $\aQ^{\pOp} \to \two$.

\begin{definition}[Elements and ideals as morphisms]
\label{def:elem_ideal_mor}
\item
Let $\aQ$ be any finite join-semilattice.
\begin{enumerate}
\item
Each element $q \in Q$ has an associated join-semilattice morphism:
\[
\elem{\aQ}{q} :=  \; \lambda b \in \{0,1\}. (b = 1) \; ? \; q : \bot_\aQ \; :  \two \to \aQ
\]
and we define the join-semilattice $\jslElem{\aQ} := \JSL_f[\two,\aQ]$.

\item
Each element $q_0 \in Q$ has an associated join-semilattice morphism:
\[
\ideal{\aQ}{q_0} := \; \lambda q \in Q. (q \leq_\aQ q_0) \; ? \; 0 : 1 \; : \aQ \to \two
\]
and we define the join-semilattice $\jslIdeal{\aQ} := \JSL_f[\aQ,\two]$. \endbox

\end{enumerate}
\end{definition}

\begin{lemma}
\label{lem:jsl_elem_iso}
For each finite join-semilattice $\aQ$ the following statements hold.
\begin{enumerate}
\item
We have the join-semilattice isomorphism:
\[
\begin{tabular}{c}
$\elem{\aQ}{-} : \aQ \to \jslElem{\aQ} = \JSL_f[\two,\aQ]$
\\
$\elem{\aQ}{q} := \lambda b \in 2. b \; ? \; q : \bot_\aQ$
\qquad
$\elemInv{\aQ}{h : \two \to \aQ} := h(1)$
\end{tabular}
\]

\item
We have the join-semilattice isomorphism:
\[
\begin{tabular}{c}
$\ideal{\aQ}{-} : \aQ^{\pOp} \to \jslIdeal{\aQ} = \JSL_f[\aQ,\two]$
\\
$\ideal{\aQ}{q_0} = \lambda q \in Q. (q \leq_\aQ q_0) \; ? \; 0 : 1$
\qquad
$\idealInv{\aQ}{h : \aQ \to \two} := \Lor_\aQ h^{-1}(\{0\})$
\end{tabular}
\]

\item
Regarding the adjoints of these special morphisms, 
\[
\begin{tabular}{lll}
$(\elem{\aQ}{q})_*$ & $=$ & $\aQ^{\pOp} \xto{\ideal{\aQ^{\pOp}}{q}} \two \xto{\swap^{\bf-1}} \two^{\pOp}$
\\
$(\ideal{\aQ}{q})_*$ & $=$ & $\two^{\pOp} \xto{\swap} \two \xto{\elem{\aQ^{\pOp}}{q}} \aQ^{\pOp}$
\end{tabular}
\]

\takeout{
\item
$\aQ \cong \jslElem{\aQ} \cong \jslIdeal{\aQ^{\pOp}}$ via the following composition of join-semilattice isomorphisms:
\[
\begin{tabular}{c}
$\aQ \xto{\elem{\aQ}{-}}
\jslElem{\aQ} = \JSL_f[\two,\aQ] \xto{(-)_*} 
\JSL_f[\aQ^{\pOp},\two^{\pOp}] \xto{\swap \circ (-)} 
\JSL_f[\aQ^{\pOp},\two] 
= \jslIdeal{\aQ^{\pOp}}$
\\[0.5ex]
with action $q_0 \mapsto \ideal{\aQ^{\pOp}}{q_0} = \lambda q \in Q.((q_0 \leq_{\aQ} q) \; ? \; 0 : 1)$
\end{tabular}
\]
where $\elem{\aQ}{-}$ is from Lemma \ref{lem:jsl_elem_iso}, $(-)_*$ is from Corollary \ref{lem:jsl_mor_iso}, and $\swap \circ (-)$ is post-composition with the unique join-semilattice isomorphism of type $\two^{\pOp} \to \two$, see Lemma \ref{lem:jsl_mor_iso}.
}

\end{enumerate}
\end{lemma}

\begin{proof}
\item
\begin{enumerate}
\item
A join-semilattice morphism $f : \two \to \aQ$ must map $0 = \bot_\two$ to $\bot_\aQ$ and may send $1$ to any element of $\aQ$. Thus $\elem{\aQ}{-}$ is a well-defined bijective function, and preserves joins because:
\[
\begin{tabular}{l}
$\elem{\aQ}{\bot_\aQ} 
= \lambda b \in 2.b \;?\; \bot_\aQ : \bot_\aQ 
= \lambda b. \bot_\aQ 
= \bot_{\jslElem{\aQ}}$
\\[0.5ex]
$\elem{\aQ}{q_1 \lor_\aQ q_2}
= \lambda b \in 2. b \;?\; q_1 \lor_\aQ q_2 : \bot_\aQ \lor \bot_\aQ 
= \elem{\aQ}{q_1} \lor_{\jslElem{\aQ}} \elem{\aQ}{q_2}$
\end{tabular}
\]
The correctness of its inverse is clear.

\item
Each $\ideal{\aQ}{q_0} : \aQ \to \two$ is a well-defined join-semilattice morphism because it is the composite:
\[
\ideal{\aQ}{q_0} = \; \aQ \xto{(\elem{\aQ^{\pOp}}{q_0})_*} \two^{\pOp} \xto{\swap} \two
\]
where $\swap$ is the unique join-semilattice morphism of type $\two^{\pOp} \to \two$ (it flips the bit). To see this, let us describe the action of $(\elem{\aQ^{\pOp}}{q_0})_*$.
\[
\begin{tabular}{lll}
$q$
& $\mapsto$ &
$\Lor_\two \{ b \in \{0,1\} : \elem{\aQ^{\pOp}}{q_0}(b) \leq_{\aQ^{\pOp}} q \}$
\\& $=$ &
$\Lor_\two \{ 1 : q \leq_\aQ \elem{\aQ^{\pOp}}{q_0}(1)  \}$
\\& $=$ &
$\Lor_\two \{ 1 : q \leq_\aQ q_0  \}$
\\& $=$ &
$\begin{cases} 
1 & \text{if $q \leq_\aQ q_0$}
\\
0 & \text{otherwise}
\end{cases}$
\end{tabular}
\]
Applying $\swap$ yields the desired action. It also follows from (1) that $\ideal{\aQ}{-}$ is a bijection, and regarding preservation of $\aQ^{\pOp}$-joins:
\[
\ideal{\aQ}{\top_\aQ} 
= \lambda q \in Q. (q \leq_\aQ \top_\aQ) \;?\; 0 : 1 
= \lambda q \in Q. 0 = \bot_{\jslIdeal{\aQ}}
\]
\[
\begin{tabular}{lll}
$\ideal{\aQ}{q_1 \land_\aQ q_2}$
&
$= \lambda q \in Q. (q \leq_\aQ q_1 \land_\aQ q_2) \;?\; 0 : 1$
\\&
$= \lambda q \in Q. (q \nleq_\aQ q_1 \text{ or } q \nleq_\aQ q_2) \;?\; 1 : 0$
\\&
$= \ideal{\aQ}{q_1} \lor_{\jslIdeal{\aQ}} \ideal{\aQ}{q_2}$
\end{tabular}
\]

\item
Follows from the proof of the previous statement, noting that $\swap : \two^{\pOp} \to \two$ is self-adjoint, $\swap^{\bf-1} : \two \to \two^{\pOp}$ is self-adjoint, and they are the same underlying functions (although distinct $\JSL_f$-morphisms).

\takeout{
\item
We need only verify the action of the composite isomorphism:
\[
\begin{tabular}{ll}
$\swap \circ (\elem{\aQ}{q_0})_*$
&
$= \swap \circ (\lambda b \in 2.b \;?\; q_0 : \bot_\aQ)_* $
\\&
$= \swap \circ (\lambda q \in Q. \Lor_\two \{ b \in 2: (b \;?\; q_0 : \bot_\aQ) \leq_\aQ q \})$
\\&
$= \swap \circ (\lambda q \in Q. q_0 \leq_\aQ q \;?\; 1 : 0 )$
\\&
$= \lambda q \in Q. (q_0 \leq_\aQ q) \;?\; 0 : 1$
\\&
$= \lambda q \in Q. (q \leq_{\aQ^{\pOp}} q_0) \;?\; 0 : 1$
\\&
$= \ideal{\aQ^{\pOp}}{q_0}$
\end{tabular}
\]
}

\end{enumerate}

\end{proof}

We shall spend the rest of this section discussing embeddings and quotients of finite join-semilattices. Later on we shall again consider the structural properties of $\JSL_f$ e.g.\ we define the tensor product and prove its universality using the category $\Dep$.

\begin{lemma}
\label{lem:jsl_mono_epi_iso}
Let $f : \aQ \to \aR$ be any $\JSL_f$-morphism.
\begin{enumerate}
\item
$f$ is a monomorphism iff it is injective.
\item
$f$ is an epimorphism iff it is surjective.
\item
$f$ is injective iff $f_*$ is surjective, and equivalently $f$ is surjective iff $f_*$ is injective.
\item
$f$ is an isomorphism iff it is monic and epic iff it is bijective.
\end{enumerate}
\end{lemma}

\begin{proof}
\item
\begin{enumerate}
\item
That $f$ is monic means precisely that $f \circ \alpha = f \circ \beta$ implies $\alpha = \beta$ for any $\JSL_f$-morphisms $\alpha,\beta : \aS \to \aQ$. Given that $f$ is injective then $f$ is monic because $f(\alpha(q)) = f(\beta(q))$ implies $\alpha(q) = \beta(q)$. Conversely if $f$ is monic and $f(q_1) = f(q_2)$ then $f \circ \elem{\aQ}{q_1} = f \circ \elem{\aQ}{q_2}$ and hence $\elem{\aQ}{q_1} = \elem{\aQ}{q_2}$, so that $q_1 = q_2$.
\item
That $f$ is epic means precisely that $\alpha \circ f = \beta \circ f$ implies $\alpha = \beta$ for any $\JSL_f$-morphisms $\alpha,\beta : \aR \to \aS$. Given that $f$ is surjective then $f$ is epic because $\alpha(r) = \alpha(f(q)) = \beta(f(q)) = \beta(r)$ by choosing a suitable $q$. Conversely assume $f$ is epic, so that $f$ is the dual of an injective function by using Theorem \ref{thm:jsl_self_dual} and the previous statement. Then it suffices to show that $f_* : \aR^{\pOp} \to \aQ^{\pOp}$ is surjective whenever $f : \aQ \monoto \aR$ is injective. So given any $q \in Q$ we need to find some $r \in R$ such that $f_*(r) = q$, and the obvious choice is $r := f(q)$.
\[
f_*(f(q))
= \Lor_\aQ f^{-1}(\down_\aR f(q))
= \Lor_\aQ \{ q' \in Q : f(q') \leq_\aR f(q) \}
\]
Certainly $q$ is one of the summands. Conversely if $f(q') \leq_\aR f(q)$ then $f(q' \lor_\aQ q) = f(q') \lor_\aR f(q) = f(q)$, so by injectivity $q \lor_\aQ q' = q$ and thus $q' \leq_\aQ q$. Therefore $f_*(f(q)) = q$ and we are finished.

\item
$f$ is injective iff $f$ is $\JSL_f$-monic by the first statement, iff $f_*$ is $\JSL_f$-epic by the duality of Theorem \ref{thm:jsl_self_dual}, iff $f_*$ is surjective by the second statement. Since $f = (f_*)_*$ by the naturality of $rep$ we obtain the other statement.

\item
That $f$ is an isomorphism means that there exists a $\JSL_f$-morphism $g : \aR \to \aQ$ such that $g \circ f = id_\aQ$ and $f \circ g = id_\aR$. Then if $f$ is an isomorphism it is split-monic hence monic hence injective, and split-epic hence epic hence surjective. Thus $f$ is bijective. Conversely suppose $f$ is injective and surjective, hence bijective. Then its functional inverse is a well-defined $\JSL_f$-morphism, a well-known fact that holds in any variety of algebras.

\end{enumerate}
\end{proof}

\begin{note}
Although a bijective homomorphism defines an algebra isomorphism in any variety of algebras, this fails in the ordered setting. For example, a bijective monotone function from a discretely ordered two element set to a $2$-chain does not have a monotone inverse. Moreover, algebra homomorphisms can be epic and yet not surjective, as is the case in the variety of distributive lattices. For example, each of two embeddings of a $3$-chain into a $4$-element boolean algebra are not surjective. However they are both epic using the fact that complements in distributive lattices are unique whenever they exist, see Lemma \ref{lem:std_order_theory}.9. \endbox
\end{note}

We have more to say regarding injective and surjective join-semilattice morphisms. Let us start with some negative results and their duals.

\begin{lemma}
\label{lem:jsl_mor_inj_fail}
Let $f : \aQ \to \aR$ be any join-semilattice morphism between finite join-semilattices. 
\begin{enumerate}
\item
If $f$'s restriction to $J(\aQ) \subseteq Q$ is injective then $f$ need not be injective.
\item
If $f$'s restriction to $M(\aQ) \subseteq Q$ is injective then $f$ need not be injective.
\item
Moreover even if both these restrictions are injective then $f$ needn't be.

\end{enumerate}
\end{lemma}

\begin{proof}
\item
\begin{enumerate}
\item
As a counter-example, first recall that the join-semilattice $\jslM{3}$ is obtained by adding a new top and bottom element to the discrete poset with elements $X = \{x_1,x_2,x_3\}$. Then we have the join-semilattice morphism $f : \JPow X \to \jslM{3}$ where  $f(\{x_i\}) = x_i$ for each of three join-irreducibles (atoms). It is clearly injective on the join-irreducibles, yet maps every meet-irreducible (coatom) to $\top_{\jslM{3}}$.

\item
We illustrate a counter-example below.
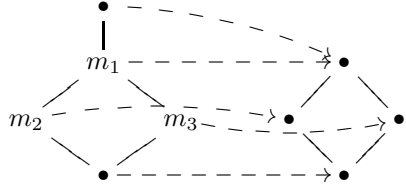
\[
\xymatrix@=10pt{
& \bullet \ar@/^5pt/@{-->}[drrrr] &  
\\
& m_1 \ar@{-}[u] \ar@{-->}[rrrr] &  &&   & \bullet
\\
m_2 \ar@{-}[ur] \ar@/^5pt/@{-->}[rrrr] && m_3 \ar@{-}[ul] \ar@/_5pt/@{-->}[rrrr] && \bullet  \ar@{-}[ur] && \bullet  \ar@{-}[ul]
\\
& \bullet \ar@{-}[ul] \ar@{-}[ur] \ar@{-->}[rrrr] &  &&   & \bullet \ar@{-}[ul]  \ar@{-}[ur]
}
\]
It is easily seen to be a well-defined join-semilattice morphism $f : \aQ \to \aR$ i.e.\ we are essentially extending the identity function on $\two^2$ with an identification. Then it is injective on the meet-irreducibles $\{m_1,m_2,m_3\}$ but it is not an injective function.

\item
The third statement follows from the second example above, noting that $J(\aQ) = \{m_2,m_3,\top_{\aQ}\}$.

\end{enumerate}
\end{proof}

We now dualise the above observations item-by-item. Recall that the ideal associated to an element $q_i \in Q$ of a join-semilattice $\aQ$ is the join-semilattice morphism $\ideal{\aQ}{q_0} : \aQ \to \two$ defined $\lambda q \in Q.(q \leq_\aQ q_0) \;?\; 0 : 1$. Then one says $f$ \emph{separates} a collection of ideals $\{ \ideal{\aQ}{q_i} : i \in I \}$ if whenever $q_i \neq q_j$ then $\ideal{\aQ}{q_i} \circ f \neq \ideal{\aQ}{q_j} \circ f$. 

\begin{lemma}
Let $f : \aQ \to \aR$ be any join-semilattice morphism between finite join-semilattices. 
\begin{enumerate}
\item
If $f$ separates the ideals $\{ \ideal{\aQ}{m} : m \in M(\aQ) \}$ then $f$ needn't be surjective.
\item
If $f$ separates the ideals $\{ \ideal{\aQ}{j} : j \in J(\aQ) \}$ then $f$ needn't be surjective.
\item
If $f$ separates the ideals $\{ \ideal{\aQ}{q} : q \in J(\aQ) \cup M(\aQ) \}$ then $f$ needn't be surjective.

\end{enumerate}
\end{lemma} 

Now for some positive results and their dual statements. This time we shall start with the surjective morphisms.

\begin{lemma}
\label{lem:jsl_surj_char}
A morphism $f : \aQ \to \aR$ of finite join-semilattices is surjective iff $J(\aR) \subseteq f[J(\aQ)]$.
\end{lemma}

\begin{proof}
Generally speaking, an algebra homomorphism is surjective iff the image of any subset generating the domain generates the codomain. Assume that $f$ is surjective. By Lemma \ref{lem:std_order_theory}.10 we know (i) $J(\aQ)$ generates $\aQ$, and moreover (ii) $J(\aR)$ is contained in any subset generating $\aR$, thus in particular $J(\aR) \subseteq f[J(\aQ)]$. Conversely the latter inclusion implies $f$ is surjective via (i).
\end{proof}

Dualising yields the following characterisation of embeddings.

%
%

\begin{lemma}
\label{lem:jsl_embed_irr_char}
A morphism $f : \aQ \to \aR$ of finite join-semilattices is injective iff:
\[
\forall m_q \in M(\aQ). \exists m_r \in M(\aR). \forall j_q \in J(\aQ).
(f(j_q) \leq_\aR m_r \iff j_q \leq_\aQ m_q)
\]
\end{lemma}

\begin{proof}
$f : \aQ \to \aR$ is injective iff $f_* : \aR^{\pOp} \to \aQ^{\pOp}$ is surjective by Lemma \ref{lem:jsl_mono_epi_iso}.3, or equivalently:
\[
M(\aQ) = J(\aQ^{\pOp}) \subseteq f_*[J(\aR^{\pOp})] = f_*[M(\aR)]
\]
by Lemma \ref{lem:jsl_surj_char}. Then we observe that:
\[
\begin{tabular}{lll}
$f_*[M(\aR)]$
&
$= \{ f_*(m_r) : m_r \in M(\aR) \}$
\\&
$= \{ \Lor_\aQ \{q \in Q : q \leq_\aQ f_*(m_r) \} : m_r \in M(\aR) \}$
\\&
$= \{ \Lor_\aQ \{j_q \in J(\aQ) : j_q \leq_\aQ f_*(m_r) \} : m_r \in M(\aR) \}$
\\&
$= \{ \Lor_\aQ \{j_q \in J(\aQ) : f(j_q) \leq_\aR m_r \} : m_r \in M(\aR) \}$
& by adjoint relationship
\end{tabular}
\]
\end{proof}

We also have the following related well-known facts.

\begin{lemma}
\label{lem:canon_surj_inj_mor_tight}
Let $\aQ$ be any finite join-semilattice.
\begin{enumerate}
\item
Given any surjective join-semilattice morphism $\sigma : \JPow Z \epito \aQ$ then $|J(\aQ)| \leq |Z|$.
\item
Given any injective join-semilattice morphism $e : \aQ \monoto \JPow Z$ we have $|M(\aQ)| \leq |Z|$.
\takeout{
\item
$e_\aQ[Q] \subseteq \Pow M(\aQ)$ generates the set-theoretic boolean algebra of all subsets of $M(\aQ)$.
}
\end{enumerate}
\end{lemma}

\begin{proof}
The first statement holds because by Lemma \ref{lem:jsl_surj_char} we know that $\sigma[J(\JPow Z)] \supseteq J(\aQ)$ and therefore $|J(\aQ)| \leq |J(\JPow Z)| = |Z|$. The second statement follows from the first by the self-duality of $\JSL_f$ and the fact that surjections dualise injections via Lemma \ref{lem:jsl_mono_epi_iso}. That is, given $e$ we obtain the surjective morphism $e_* : (\JPow Z)^{\pOp} \epito \aQ^{\pOp}$ and thus also $e_* \circ (\neg_Z)^{\bf-1} : \JPow Z \epito \aQ^{\pOp}$, so that $|M(\aQ)| = |J(\aQ^{\pOp})| \leq |Z|$.
\takeout{
  Regarding the third statement, given the canonical embedding $e_\aQ : \aQ \monoto \JPow M(\aQ)$ we may close the codomain under set-theoretic boolean operations yielding another embedding $e'_\aQ : \aQ \monoto \aR \cong \JPow At(\aR)$. If the boolean semilattice $\aR$ fails to contain all subsets of $M(\aQ)$ then $At(\aR) < |M(\aQ)|$ via the duality between finite boolean algebras and finite sets, which contradicts the second statement.
  }
\end{proof}

\subsection{Congruence lattices of finite join-semilattices}

\begin{definition}[Congruence and subalgebra lattices]
\label{def:jsl_cong_sub_lattices}
\item
Let $\aQ$ be a finite join-semilattice.
\begin{enumerate}
\item
A \emph{congruence of $\aQ$} (also called a \emph{$\aQ$-congruence}) is an equivalence relation $\rR \subseteq Q \times Q$ closed under the rule:
\[
\dfrac{\rR(q_1,q_2) \quad \rR(q_3,q_4)}{\rR(q_1 \lor_\aQ q_3, q_2 \lor_\aQ q_4)}
\qquad
(\lor cg)
\qquad\
\text{for every $q_1,q_2,q_3,q_4 \in Q$.}
\]
Letting $\setCong{\aQ}$ be the collection of all $\aQ$-congruences and ordering by inclusion yields:
\[
\latCong{\aQ} 
:= (\setCong{\aQ},\lor_{\latCong{\aQ}},\Delta_Q,\cap,Q \times Q)
\qquad
\text{i.e.\ the \emph{bounded lattice of $\aQ$-congruences}.}
\]
For general universal algebraic reasons, it is a sub bounded lattice of the  lattice of all equivalence relations on $Q$. In particular, the binary join $\lor_{\latCong{\aQ}}$ constructs the transitive closure of the binary union. Given any relation $\rS \subseteq Q \times Q$, let:
\[
\genCong{\aQ}{\rS}
:= \bigcap \{ \rR \in \setCong{\aQ} : \rS \subseteq \rR \}
\qquad
\text{be the \emph{$\aQ$-congruence generated by $S$}.}
\]
In the case where $\rS = \{(q_1,q_2)\}$ is a singleton we instead write $\prCong{\aQ}{q_1,q_2}$ (which equals $\prCong{\aQ}{q_2,q_1}$), these being the \emph{principal $\aQ$-congruences}. By universal algebra, the principal congruences where $q_1 \neq q_2$ are precisely the join-irreducible elements of $\latCong{\aQ}$. On the other hand, we also have the \emph{meet-irreducible $\aQ$-congruences}:
\[
\mirrCong{\aQ}{q} := 
(\down_\aQ q)\times(\down_\aQ q) \,\cup\, (\overline{\down_\aQ q})\times(\overline{\down_\aQ q})
\quad \subseteq Q \times Q
\qquad
\text{for each $q \in Q \backslash \{\top_\aQ\}$.}
\]
We also permit $\mirrCong{\aQ}{\top_\aQ}$ under the above definition, observing that it equals $\top_{\latCong{\aQ}}$ and thus is the \emph{maximum $\aQ$-congruence} or alternatively the \emph{trivial $\aQ$-congruence}.

\item
The $\aQ$-subalgebras also define a finite inclusion-ordered bounded lattice:
\[
\latSub{\aQ} := (\setSub{\aQ},\lor_{\latSub{\aQ}},\{\bot_\aQ\},\cap,Q)
\]
where $\setSub{\aQ} := \{ S : (S,\lor_\aS,\bot_\aQ) \subseteq \aQ \} \subseteq \Pow Q$. Notice that we collect the underlying sets of $\aQ$'s subalgebras, rather than the subalgebras themselves. The binary join $\lor_{\latSub{\aQ}}$ constructs all possibly-empty finite joins of the binary union i.e.\ the elements of the $\aQ$-subalgebra generated by the binary union.

Recall the usual notation for generated subalgebras i.e.\ $\ang{X}_\aQ \subseteq \aQ$ is the sub join-semilattice generated by $X \subseteq Q$. Let us denote the carrier of this algebra by $\genSubset{\aQ}{X}$, so it is the closure of $X \subseteq Q$ under all possibly-empty $\aQ$-joins. In the case where $X = \{q\}$ is a singleton we have:
\[
\genElem{\aQ}{q} = \{ \bot_\aQ, q\}.
\]
Excluding the $0$-generated subalgebra with carrier $\genElem{\aQ}{\bot_\aQ} = \{\bot_\aQ\} = \bot_{\latSub{\aQ}}$, it follows by universal algebra that these $1$-generated subalgebras  are precisely the join-irreducible elements of $\latSub{\aQ}$. In fact they are clearly atoms so that $\latSub{\aQ}$ is atomistic: every element is a join of atoms. Finally, for each $q_1,  q_2 \in Q$ we have the $\aQ$-subalgebra:
\[
\mirrSub{\aQ}{q_1,q_2} \subseteq \aQ
\qquad
\text{with carrier}
\qquad
\mirrSubset{\aQ}{q_1,q_2} := \{ q \in Q : q \leq_\aQ q_1 \siff q \leq_\aQ q_2 \} \subseteq Q.
\]
Observing that $\mirrSubset{\aQ}{q,q} = Q = \top_{\latSub{\aQ}}$, then the \emph{meet-irreducible $\aQ$-subalgebras} are those where $q_1 \neq q_2$. \endbox
\end{enumerate}
\end{definition}

\bigskip

The above definitions will soon be clarified. Let us start by describing the bounded lattice isomorphism between $\aQ^{\pOp}$-subalgebras and $\aQ$-congruences, after which we provide a Corollary describing the connection in terms of $\aQ$-quotients. Then in Lemma \ref{lem:jsl_cong_sub_act_on_gen} we'll describe the irreducible $\aQ$-congruences and $\aQ$-subalgebras, and the action of the bounded lattice isomorphisms upon them.

\begin{theorem}[Representing congruences as subalgebras and conversely]
\label{thm:cong_sub_dual_iso}
\item
For each finite join-semilattice $\aQ$ we have the bounded lattice isomorphism:
\[
\begin{tabular}{llll}
$\ctos{\aQ} : (\latCong{\aQ})^{\pOp} \to \latSub{\aQ^{\pOp}}$
&&
$\ctos{\aQ}(\rR)$ & $:= \{ \Lor_\aQ \sem{q}_\rR : q \in Q \}$
\\[1ex]&&
$\stoc{\aQ}(S)(q_1,q_2)$ & $:\iff \forall s \in S.(q_1 \leq_\aQ s \siff q_2 \leq_\aQ s)$ 
\end{tabular}
\]
where $\stoc{\aQ} = \ctos{\aQ}^{\bf-1}$.
\end{theorem}

\begin{proof}
By universal algebra, the $\aQ$-congruences $\rR \in \setCong{\aQ}$ are precisely the kernels $\ker{f}$ of all surjective join-semilattice morphisms $f : \aQ \epito \aR$ where $\aR \in \JSL_f$. Let us recall that:
\[
\ker{f} := \{ (q_1,q_2) \in Q \times Q : f(q_1) = f(q_2) \}.
\]
Certainly each such kernel is a $\aQ$-congruence. Conversely we have the canonical surjective function $\sem{\cdot}_\rR : Q \epito Q \backslash \rR$ because $\rR$ is an equivalence relation, and this actually defines a  join-semilattice morphism $\aQ \epito \aQ\backslash\rR = (Q \backslash\rR,\lor_{\aQ\backslash\rR},\sem{\bot_\aQ}_\rR)$ where of course $\sem{q_1}_\rR \lor_{\aQ\backslash\rR} \sem{q_2}_\rR := \sem{q_1 \lor_\aQ q_2}_\rR$. Importantly, we note that every $\rR$-equivalence class is non-empty and closed under binary $\aQ$-joins. Then by finiteness $\Lor_\aQ \sem{q}_\rR \in \sem{q}_\rR$ i.e.\ each $\rR$-equivalence class always contains a maximum element.

\smallskip
Given any $\aQ$-congruence $\rR$, take the adjoint of its associated canonical surjective morphism:
\[
\dfrac{\sem{\cdot}_\aQ : \aQ \epito \aQ \backslash \rR }
{(\sem{\cdot}_\aQ)_* : (\aQ \backslash \rR)^{\pOp} \monoto \aQ^{\pOp} }
\]
The latter is necessarily injective by Lemma \ref{lem:jsl_mono_epi_iso}, so define $\aS_\rR := (\sem{\cdot}_\aQ)_*[\aQ \backslash \rR] \subseteq \aQ^{\pOp}$ to be the image of this embedding. To understand its elements, consider the following  calculation:
\[
\begin{tabular}{lll}
$(\sem{\cdot}_\aQ)_*(\sem{q}_\rR)$
&
$= \Lor_\aQ \{ q' \in Q : \sem{q'}_\rR \leq_{\aQ\backslash\rR} \sem{q}_\rR \}$
& by definition
\\&
$= \Lor_\aQ \{ q' \in Q : \sem{q'}_\rR \lor_{\aQ\backslash\rR} \sem{q}_\rR = \sem{q}_\rR \}$
\\&
$= \Lor_\aQ \{ q' \in Q : \sem{q' \lor_\aQ q}_\rR = \sem{q}_\rR \}$
\\&
$= \Lor_\aQ \{ q' \in Q: \sem{q' \lor_\aQ \Lor_\aQ \sem{q}_\rR}_\rR = \sem{\Lor_\aQ \sem{q}_\rR}_\rR  \}$
& by well-definedness
\\&
$= \Lor_\aQ \{ q' \in Q: q' \leq_\aQ \Lor_\aQ \sem{q}_\rR  \}$
& by maximality
\\&
$= \Lor_\aQ \sem{q}_\rR$.
\end{tabular}
\]
Thus $S_\rR$ is obtained by taking the maximum element from each $\rR$-equivalence class. It follows that $\ctos{\aQ}$ is a well-defined function. For injectivity it suffices to show that $\sem{\cdot}_\rR : \aQ \epito \aQ\backslash\rR$ and the (surjective) adjoint of $\iota : \aS_\rR \hookto \aQ$ have the same kernel. We have $\iota \circ \beta = (\sem{\cdot}_\rR)_*$ for some isomorphism $\beta$, and thus $\beta_* \circ \iota_* = \sem{\cdot}_\rR$ where $\beta_* = (\beta^{\bf-1})^{\pOp}$ is also an isomorphism. It follows that $\ker{\iota_*} = \rR$, as required. Concerning surjectivity, take any sub join-semilattice $\iota : \aS \hookto \aQ^{\pOp}$ and define $\rR_\aS := \ker \iota_*$. We are going to show that $\ctos{\aR}(\rR_\aS) = S$. First observe,
\[
\iota_*(q) 
= \Lor_\aS \{ s \in S : s \leq_{\aQ^{\pOp}} q \}
= \Land_\aQ \{ s \in S : q \leq_\aQ s \}
\]
so if we assume $\iota_*(q_1) = \iota_*(q_2)$ for any fixed $q_1, q_2 \in Q$, then if $q_1 \leq_\aQ s \in S$ we have $q_2 \leq_\aQ \iota_*(q_2) = \iota_*(q_1) \leq_\aQ s$, and symmetrically $q_2 \leq_\aQ s$ implies $q_1 \leq_\aQ s$. It follows that:
\[
\rR_\aS(q_1,q_2) \iff \forall s \in S.(q_1 \leq_\aQ s \siff q_2 \leq_\aQ s)
\]
also because the latter condition implies that $\iota_*(q_1)$ and $\iota_*(q_2)$ have the same summands. Since  $\iota_*(s) = s$ for each $s \in S$, it follows that $\iota_*(q) = \iota_*(\iota_*(q))$ and hence every $\rR_\aS$-equivalence class contains some $s \in S$. Furthermore they may contain at most one element of $S$ via anti-symmetry. Then it follows that:
\[
\ctos{\aR}(\rR_\aS)
= \{ \Lor_\aQ \sem{q}_{\rR_\aS} : q \in Q \}
= S
\]
because if $s \in \sem{q}_{\rR_\aS}$ then it is necessarily the maximum element relative to $\leq_\aQ$. Then we have proved that $\ctos{\aQ}$ is bijective and have also described its inverse $\stoc{\aQ}$ as desired. To establish that they are bounded lattice isomorphisms we'll show that $\stoc{\aQ}$ preserves and reflects the given orderings.
\begin{enumerate}
\item
Assuming $\aS_1 \subseteq \aS_2 \subseteq \aQ^{\pOp}$ we need to show that $\stoc{\aQ}(S_2) \subseteq \stoc{\aQ}(S_1)$. Since $S_1 \subseteq S_2$ this follows immediately by restricting the universal quantification from $S_2$ to $S_1$.

\item
Now suppose that $\stoc{\aQ}(S_2) \subseteq \stoc{\aQ}(S_1)$ i.e.\ for every $(q_1,q_2) \in Q \times Q$ we know:
\[
\forall s \in S_2.(q_1 \leq_\aQ s \siff q_2 \leq_\aQ s)
\quad\text{implies}\quad
\forall s \in S_1.(q_1 \leq_\aQ s \siff q_2 \leq_\aQ s).
\]
For a contradiction assume $S_1 \nsubseteq S_2$ so we have some $s_1 \in S_1 \cap \overline{S_2}$, and define $s_2 := \Land_\aQ \{ s \in S_2 : s_1 \leq_\aQ s \}$ observing that $s_2 \in S_2$ (because $\aS_2 \subseteq \aQ^{\pOp}$) and also $s_1 \leq_\aQ s_2$. Setting $(q_1,q_2) := (s_1,s_2)$, one can see that the premise of the above correspondence holds. Instantiating the deduced conclusion with $s := s_1$ we obtain $s_1 \leq_\aQ s_1 \iff s_2 \leq_\aQ s_1$, and hence derive the contradiction $s_1 = s_2 \in S_2$.
\end{enumerate}
\end{proof}

\begin{corollary}[Subalgebra/quotient correspondence]
\item
For any $\aS \subseteq \aQ^{\pOp}$ and any $\aQ$-congruence $\rR$ there are associated isomorphisms:
\[
\begin{tabular}{lllll}
\emph{1.} & $\alpha : \aS^{\pOp} \to \aQ\backslash\rT$
&
$\alpha(s) := \sem{s}_\rT$
&
$\alpha^{\bf-1}(\sem{q}_\rT) := \Land_\aQ \{ s \in S : q \leq_\aQ s \}$
&
$\rT := \stoc{\aQ}(S) \in \setCong{\aQ}$.
\\[1ex]
\emph{2.} & $\beta : (\aQ\backslash\rR)^{\pOp} \to \aR$
&
$\beta(\sem{q}_\rR) := \Lor_\aQ \sem{q}_\rR$
&
$\beta^{\bf-1}(r) := \sem{r}_\rR$
&
$\aR \subseteq \aQ^{\pOp}$ has carrier $\ctos{\aQ}(\rR)$.
\end{tabular}
\]
\end{corollary}

\begin{proof}
Fixing any $\aS \subseteq \aQ^{\pOp}$ and $\rR \in \setCong{\aQ}$ let us verify the claimed isomorphisms $\alpha$ and $\beta$.

\begin{enumerate}
\item
The carrier of the subalgebra $\aS$ yields the $\aQ$-congruence $\rT := \stoc{\aQ}(S)$, and the inclusion join-semilattice morphism $\iota : \aS \hookto \aQ^{\pOp}$ yields the $\aQ$-congruence $\ker{\iota_*}$. It follows directly from the proof of Theorem \ref{thm:cong_sub_dual_iso} that these two kernels coincide. Consequently there exists a unique $\JSL_f$-isomorphism such that:
\[
\begin{tabular}{ll}
$\vcenter{\vbox{\xymatrix@=15pt{
\aQ \ar@{>>}[rr]^{\sem{\cdot}_\rT} \ar@{>>}[drr]_{\iota_*} && \aQ\backslash\rT
\\
&& \aS^{\pOp} \ar[u]_{\alpha}^{\cong}
}}}$
&
using the appropriate homomorphism theorem from universal algebra.
\end{tabular}
\]
By definition $\iota_*(q) = \Lor_\aS \{ s \in S : s \leq_{\aQ^{\pOp}} q \} =   \Land_\aQ \{ s \in S : q \leq_\aQ s \}$ so that $\iota_*(s) = s$ for any $s \in S$. Thus  $\alpha(s) = \alpha(\iota_*(s)) = \sem{s}_\rT$ as expected, and finally $\alpha^{\bf-1}(\sem{q}_\rT) = \iota_*(q) = \Land_\aQ \{ s \in S : q \leq_\aQ s \}$.

\item
The $\aQ$-congruence $\rR$ yields the subalgebra $\iota : \aR \hookto \aQ^{\pOp}$ with carrier $R := \ctos{\aQ}(\rR)$. By the proof of Theorem \ref{thm:cong_sub_dual_iso} we know that the latter is precisely the image of the embedding  $(\sem{\cdot}_\rR)_* : (\aQ\backslash\rR)^{\pOp} \monoto \aQ^{\pOp}$, so the action of the latter defines a $\JSL_f$-isomorphism $\beta$ as follows:
\[
\xymatrix@=15pt{
\aR \ar@{>->}[rr]^\iota && \aQ^{\pOp}
\\
(\aQ\backslash\rR)^{\pOp}  \ar[u]^{\beta}_{\cong} \ar@{>->}[urr]_{(\sem{\cdot}_\rR)_*}
}
\]
Then $\beta(\sem{q}_\rR) = \iota(\beta(\sem{q}_\rR)) = (\sem{\cdot}_\rR)_*(q) = \Lor_\aQ \sem{q}_\rR$, where the final equality was established in the proof of the Theorem. Finally, since we always know that $\Lor_\aQ \sem{q}_\rR$ actually lies inside $\sem{q}_\rR$, it follows that $\beta^{\bf-1}(r) = \sem{r}_\rR$. 

\end{enumerate}
\end{proof}

\begin{lemma}
\label{lem:jsl_cong_sub_act_on_gen}
Concerning the isomorphism $\ctos{\aQ} : (\latCong{\aQ})^{\pOp} \to \latSub{\aQ^{\pOp}}$ from Theorem \ref{thm:cong_sub_dual_iso}.
\begin{enumerate}
\item
For any $q_1, q_2 \in Q$, the isomorphism $\ctos{\aQ}$ acts as follows:
\[
\begin{tabular}{llll}
$\prCong{\aQ}{q_1,q_2}$
& ${\mapsto}$ 
& $\mirrSubset{\aQ^{\pOp}}{q_1,q_2}$
& $= \{ q \in Q : q_1 \leq_\aQ q \siff q_2 \leq_\aQ q \}$
\\[1ex]
$\prCong{\aQ}{q_1,\bot_\aQ}$
& ${\mapsto}$ 
& $\mirrSubset{\aQ^{\pOp}}{q_1,\bot_\aQ}$
& $= \{ q \in Q : q_1 \leq_\aQ q \}$
\\[1ex]
$\prCong{\aQ}{q_1,\top_\aQ}$
& ${\mapsto}$ 
& $\mirrSubset{\aQ^{\pOp}}{q_1,\top_\aQ}$
& $= \{\top_\aQ \} \cup \{ q \in Q : q_1 \nleq_\aQ q \}$
\\[1ex]
$\mirrCong{\aQ}{q_1}$
& ${\mapsto}$ 
& $\genElem{\aQ^{\pOp}}{q_1}$
& $= \{ \top_\aQ, q_1 \}$
\end{tabular}
\]
Finally, for any  relation $\rS \subseteq Q \times Q$ we have:
\[
\begin{tabular}{lll}
$\ctos{\aQ}(\genCong{\aQ}{\rS}) $
&
$= \{ q \in Q : \forall (q_1,q_2) \in \rS.( q_1 \leq_\aQ q \siff q_2 \leq_\aQ q ) \}$
\end{tabular}
\]

\item
For any $q_1, q_2 \in Q$, the isomorphism $\stoc{\aQ^{\pOp}}$ acts as follows:
\[
\begin{tabular}{lllll}
$\mirrSubset{\aQ}{q_1,q_2}$
& ${\mapsto}$ 
& $\prCong{\aQ^{\pOp}}{q_1,q_2}(q'_1,q'_2)$
& $\iff$ & $\forall q \in Q.( (q \leq_\aQ q_1 \siff q \leq_\aQ q_2) \To (q \leq_\aQ q'_1 \siff q \leq_\aQ q'_2) )$
\\[1ex]
$\mirrSubset{\aQ}{q_1,\bot_\aQ}$
& ${\mapsto}$ 
& $\prCong{\aQ^{\pOp}}{q_1,\bot_\aQ}(q'_1,q'_2)$
& $\iff$ & $\forall q \in Q.( q \nleq_\aQ q_1 \To (q \leq_\aQ q'_1 \siff q \leq_\aQ q'_2))$
\\[1ex]
$\mirrSubset{\aQ}{q_1,\top_\aQ}$
& ${\mapsto}$ 
& $\prCong{\aQ^{\pOp}}{q_1,\top_\aQ}(q'_1,q'_2)$
& $\iff$ & $\forall q \in Q.( q \leq_\aQ q_1 \To (q \leq_\aQ q'_1 \siff q \leq_\aQ q'_2))$
\\[1ex]
$\genElem{\aQ}{q_1}$
& ${\mapsto}$ 
& $\mirrCong{\aQ^{\pOp}}{q_1}(q'_1,q'_2)$
& $\iff$ & $(q_1 \leq_\aQ q'_1 \siff q_1 \leq_\aQ q'_2)$
\end{tabular}
\]
Finally, for any subset $X \subseteq Q$ we have:
\[
\begin{tabular}{lll}
$\stoc{\aQ^{\pOp}}(\genSubset{\aQ}{X})$
&
$= \{ (q'_1,q'_2) \in Q \times Q : \forall x \in X.( x \leq_\aQ q'_1 \siff x \leq_\aQ q'_2 ) \}$
\\[1ex]&
$= \bigcup \{ \BC{(\up_\aQ \Lor_\aQ A) \cap \overline{\up_\aQ X\backslash A}} : A \subseteq X \}$
\end{tabular}
\]
recalling that $\BC{Z} := Z \times Z$.

\item
Concerning irreducible elements,
\[
\begin{tabular}{lllll}
$J(\latCong{\aQ})$ & $= \{ \prCong{\aQ}{q_1,q_2} : q_1 \neq q_2 \in Q \}$
&&
$M(\latCong{\aQ})$ & $= \{ \mirrCong{\aQ}{q} : q \in Q \backslash \{\top_\aQ\} \}$
\\[1ex]
$J(\latSub{\aQ})$ & $= \{ \genSubset{\aQ}{q} : q \in Q \backslash \{\bot_\aQ\} \}$
&&
$M(\latSub{\aQ})$ & $= \{ \mirrSubset{\aQ}{q_1,q_2} : q_1 \neq q_2 \in Q \}$
\end{tabular}
\]
and consequently:
\[
|J(\latCong{\aQ})| = |M(\latSub{\aQ})| = \frac{1}{2} \cdot |Q| \cdot (|Q| - 1)
\qquad
|M(\latCong{\aQ})| = |J(\latSub{\aQ})| = |Q| - 1
\]

\end{enumerate}
\end{lemma}

\begin{proof}
\item
\begin{enumerate}
\item
Fixing any $q_1, q_2 \in Q$, we are going to establish that $\ctos{\aQ}(\prCong{\aQ}{q_1,q_2}) = \mirrSubset{\aQ^{\pOp}}{q_1,q_2}$. Observe that:
\[
\mirrSubset{\aQ^{\pOp}}{q_1,q_2} = \{ q \in Q : q_1 \leq_\aQ q \siff q_2 \leq_\aQ q \}
\]
using Definition \ref{def:jsl_cong_sub_lattices} and the fact that the ordering is reversed. Let us first verify that $S := \mirrSubset{\aQ^{\pOp}}{q_1,q_2}$ defines a sub join-semilattice of $\aQ^{\pOp}$. Certainly $\bot_{\aQ^{\pOp}} =  \top_\aQ \in S$, so given any $s_1, s_2 \in S$ we need to show that $s_1 \land_\aQ s_2 \in S$. To this end, define the predicates $\phi(s) := (q_1 \leq_\aQ s \;\land\; q_2 \leq_\aQ s)$ and $\psi(s) := (q_1 \nleq_\aQ s \;\land\; q_2 \nleq_\aQ s)$, and proceed case-by-case:
\begin{enumerate}
\item
if $\phi(s_1) \land \phi(s_2)$ then $\phi(s_1 \land_\aQ s_2)$,

\item
if $\phi(s_1) \land \psi(s_2)$ then $\psi(s_1 \land_\aQ s_2)$ else we obtain at least one of the contradictions $q_1, q_2 \leq_\aQ s_2$,

\item
finally if $\psi(s_1) \land \psi(s_2)$ then $\psi(s_1 \land_\aQ s_2)$ lest we obtain contradictions $q_i \leq_\aQ s_j$.

\end{enumerate}

and we are done. Now, we are going to establish that:
\[
\rR(q_1,q_2) \iff \ctos{\aQ}(\rR) \subseteq \mirrSubset{\aQ^{\pOp}}{q_1,q_2}
\qquad
\text{for any $\aQ$-congruence $\rR$}.
\]
This suffices because the principal $\aQ$-congruence generated by $(q_1,q_2)$ is characterised by the property that $\prCong{\aQ}{q_1,q_2} \subseteq \rR \iff \rR(q_1,q_2)$, so via the order-isomorphism we'd deduce that $\ctos{\aQ}(\prCong{\aQ}{q_1,q_2}) = \mirrSubset{\aQ^{\pOp}}{q_1,q_2}$. Using the definition of $\ctos{\aQ}$ and $\mirrSubset{\aQ^{\pOp}}{q_1,q_2}$, the desired equivalence can be rewritten as follows:
\[
\rR(q_1,q_2) 
\quad\stackrel{?}{\iff}\quad
\forall q \in Q.( q_1 \leq_\aQ \Lor_\aQ \sem{q}_\rR \iff q_2 \leq_\aQ \Lor_\aQ \sem{q}_\rR  )
\]
and we also recall that $\Lor_\aQ \sem{q}_\rR \in \sem{q}_\rR$ for every $\aQ$-congruence $\rR$ and every element $q \in Q$.
\begin{enumerate}
\item
$(\To)$ Assume $\rR(q_1,q_2)$. Recalling the join-semilattice morphism $(\sem{\cdot}_\rR)_* : (\aQ \backslash \rR)^{\pOp} \to \aQ^{\pOp}$ from Theorem \ref{thm:cong_sub_dual_iso}, its monotonicity informs us that:
\[
\mathrm{(\star)}
\qquad
\forall q_a, q_b \in Q.( \sem{q_a}_\rR \leq_{\aQ\backslash\rR} \sem{q_b}_\rR \;\To\; \Lor_\aQ \sem{q_a}_\rR \leq_\aQ \Lor_\aQ \sem{q_b}_\rR  )
\]
also using the description of its action from the proof of Theorem \ref{thm:cong_sub_dual_iso}. If we assume that $q_1 \leq_\aQ \Lor_\aQ \sem{q}_\rR$ for any fixed $q \in Q$, then we have $\sem{q_1}_\rR \leq_{\aQ\backslash\rR} \sem{\Lor_\aQ \sem{q}_\rR}_\rR = \sem{q}_\rR$ via the monotonicity of $\sem{\cdot}_\rR$, and consequently  $\Lor_\aQ \sem{q_1}_\rR \leq_\aQ \Lor_\aQ \sem{q}_\rR$ by applying $\mathrm{(\star)}$. Thus we have:
\[
\begin{tabular}{lll}
$q_2$
&
$\leq_\aQ \Lor_\aQ \sem{q_2}_\rR$
\\&
$= \Lor_\aQ \sem{q_1}_\rR $
& since $\rR(q_1,q_2)$
\\& 
$\leq_\aQ \Lor_\aQ \sem{q}_\rR$
& by above
\end{tabular}
\]
Via the symmetric argument when assuming $q_2 \leq_\aQ \Lor_\aQ \sem{q}_\rR$, we are done.

\item
$(\oT)$ Conversely, assume that $q_1 \leq_\aQ \Lor_\aQ \sem{q}_\rR \iff q_2 \leq_\aQ \Lor_\aQ \sem{q}_\rR$ for every $q \in Q$. Then the two particular cases where $q := q_1$ and $q := q_2$ yield:
\[
q_1 \leq_\aQ \Lor_\aQ \sem{q_2}_\rR
\qquad\text{and}\qquad
q_2 \leq_\aQ \Lor_\aQ \sem{q_1}_\rR.
\]
By the monotonicity of $\sem{\cdot}_\rR$ we deduce $\sem{q_1}_\rR = \sem{q_2}_\rR$, so that $\rR(q_1,q_2)$ as required.

\end{enumerate}

Having proved the first claim of (1), the next two claims follow because they are instantiations of the first where $q_2 := \bot_\aQ$ and $q_2 := \top_\aQ$, respectively. Concerning the third claim, we point out that $\mirrSubset{\aQ^{\pOp}}{q_1,\top_\aQ}$ necessarily contains $\top_\aQ$ by well-definedness, and whenever $q \neq \top_\aQ$ then $q_2 := \top_\aQ \nleq_\aQ q$. Regarding the fourth claim, let us verify that:
\[
\ctos{\aQ}(\mirrCong{\aQ}{q}) \stackrel{?}{=} \genSubset{\aQ^{\pOp}}{q} = \{ \top_\aQ, q \}
\qquad
\text{for every $q \in Q$.}
\]
Indeed, since $\rR := \mirrCong{\aQ}{q} = \BC{\down_\aQ q} \cup \BC{\overline{\down_\aQ q}}$ we deduce that:
\begin{enumerate}
\item
If $q = \top_\aQ$ then $Q\backslash\rR = \{ \sem{\top}_\rR \}$ and hence by definition $\ctos{\aQ}(\rR) = \{ \Lor_\aQ \sem{\top_\aQ}_\rR \} = \{ \top_\aQ \}$ as required.
\item
If $q \neq \top_\aQ$ then $Q\backslash\rR = \{ \sem{q}_\rR, \sem{\top}_\rR \}$ where $\sem{q}_\rR = \down_\aQ q$ and $\sem{\top_\aQ}_\rR$, so that $\ctos{\aQ}(\rR) = \{ q, \top_\aR \}$.
\end{enumerate}

As for the fifth and final claim, it follows directly from the first:
\[
\begin{tabular}{lll}
$\ctos{\aQ}(\genCong{\aQ}{\rS})$
&
$= \ctos{\aQ}(\Lor_{\latCong{\aQ}} \{ \prCong{\aQ}{q_1,q_2} : \rS(q_1,q_2)\} )$
\\[1ex]&
$= \bigcap \{ \ctos{\aQ}(\prCong{\aQ}{q_1,q_2}) : \rS(q_1,q_2) \}$
& apply order-isomorphism
\\[1ex]&
$= \bigcap \{ \{ q \in Q : q_1 \leq_\aQ q \siff q_2 \leq_\aQ q  \} : \rS(q_1,q_2) \}$
& by first claim
\\[1ex]&
$= \{ q \in Q : \forall (q_1,q_2) \in \rS.( q_1 \leq_\aQ q \siff q_2 \leq_\aQ q) \}$
\end{tabular}
\]
Here we have used the fact that every $\aQ$-congruence $\rR$ is the $\latCong{\aQ}$-join of the principal $\aQ$-congruences it contains. This follows because whenever $\rR(q_1,q_2)$ we necessarily have $\prCong{\aQ}{q_1,q_2} \subseteq \rR$ by definition of principal congruences.

\item
The second statement mirrors the first, and is mostly directly deducible from it by virtue of the isomorphisms at hand. However, additional information is provided by describing e.g.\ the principal $\aQ^{\pOp}$-congruences explicitly. On the other hand, all of these descriptions can be readily verified by directly unwinding the definitions. The final claim follows because:
\[
\begin{tabular}{lll}
$\stoc{\aQ^{\pOp}}(\genSubset{\aQ}{X})$
&
$= \stoc{\aQ^{\pOp}}(\Lor_{\latSub{\aQ}} \{ \genElem{\aQ}{x} : x \in X  \})$
& 
\\[1ex]&
$= \bigcap \{ \ctos{\aQ^{\pOp}}(\genElem{\aQ}{x}) : x \in X \}$
\\[1ex]&
$= \bigcap \{ \mirrCong{\aQ^{\pOp}}{x} : x \in X \}$
& by first claim
\\[1ex]&
$= \bigcap \{ \BC{\up_\aQ x} \,\cup\, \BC{\overline{\up_\aQ x}} : x \in X  \}$
\\&
$= \bigcup \{ \bigcap_{x \in A} \BC{\up_\aQ x} \cap \bigcap_{x \in X\backslash A} \BC{\overline{\up_\aQ x}} : A \subseteq X \}$
& by set-theoretic distributivity
\\&
$= \bigcup \{ \BC{\up_\aQ \Lor_\aQ A} \,\cup\, \overline{\BC{\up_\aQ X\backslash A}} : A \subseteq X \}$
& see below
\end{tabular}
\]
Regarding the final equality, observe that:
\[
\BC{I} \cap \BC{J} = I \times I \,\cap\, J \times J = (I \cap J) \times (I \cap J) = \BC{I \cap J}
\]
and also the general equalities:
\[
\up_\aQ x_1 \,\cap\, \up_\aQ x_2 = \; \up_\aQ (x_1 \lor_\aQ x_2)
\qquad
\overline{\up_\aQ x_1} \cap \overline{\up_\aQ x_2} = \overline{\up_\aQ x_1 \,\cup\, \up_\aQ x_2} = \overline{\up_\aQ \{x_1,x_2\}}.
\]

\item
The description of $J(\latCong{\aQ})$ follows by universal algebra i.e.\ is a general statement concerning the lattice of congruences of a finite algebra. Likewise, the description of $J(\latSub{\aQ})$ follows for the subalgebra lattice of any (possibly infinite) algebra. Then the descriptions of the meet-irreducibles follow via (1) and (2), seeing as $\ctos{\aQ} : (\latCong{\aQ})^{\pOp} \to \latSub{\aQ^{\pOp}}$ is a bounded lattice isomorphism, and hence induces bijections between join/meet-irreducibles.

\end{enumerate}
\end{proof}

\section{The category $\Dep$}
\label{sec:the_cat_dep}

\subsection{Introducing $\Dep$ and its self-duality}
\label{subsec:intro_dep}

We describe a category and its self-duality. We'll prove it is equivalent to $\JSL_f$ in the next section. It is based on the work of Moshier and Jipsen \cite{ContextJipsen2012} (see \href{http://math.chapman.edu/~jipsen/summerschool/Jipsen%202012%20Categories%20of%20algebraic%20contexts%20equivalent%20to%20idempotent%20semirings%20and%20domain%20semirings.pdf}{here}). We reuse their notation $(-)^\up$ and $(-)^\down$, and our category $\Dep$ is a variation of Moshier's category \textbf{Ctxt} restricted to finite relations.

\begin{definition}[The category $\Dep$]
\label{def:category_bicliq}
Its objects are the relations between finite sets $\rG \subseteq \rG_s \times \rG_t$. Its morphisms $\rR : \rG \to \rH$ are those relations $\rR \subseteq \rG_s \times \rH_t$ such that the $\Rel$-diagram:
\[
\xymatrix@=15pt{
\rG_t \ar@{..>}[rr]^{\rR_r\spbreve} && \rH_t
\\
\rG_s \ar[urr]^{\rR} \ar[u]^{\rG} \ar@{..>}[rr]_{\rR_l} && \rH_s \ar[u]_{\rH}
}
\]
commutes for some relations $\rR_l$ and $\rR_r$\footnote{We use the converse relation $\rR_r\spbreve$ to make the self-duality of this category `nicer' later on.}. Equivalently, a morphism $\rR : \rG \to \rH$ is a relation $\rR \subseteq \rG_s \times \rH_t$ which factors through $\rG$ (on the left) and $\rH$ (on the right).

The identity morphism $id_\rG : \rG \to \rG$ is the relation $\rG \subseteq \rG_s \times \rG_t$. Given $\rR : \rG \to \rH$ and $\rS : \rH \to \rI$, their composite $\rR \fatsemi \rS \subseteq \rG_s \times \rI_t$ is defined by the following commuting $\Rel$-diagram:
\[
\xymatrix@=15pt{
\rG_t \ar@{..>}[rr]^{\rR_r\spbreve} && \rH_t \ar@{..>}[rr]^{\rS_r\spbreve} && \rI_t
\\
\rG_s \ar[urr]^{\rR} \ar[u]^{\rG} \ar@{..>}[rr]_{\rR_l} && \rH_s \ar[u]_{\rH} \ar[urr]^{\rS} \ar@{..>}[rr]_{\rS_l} && \rI_s \ar[u]_{\rI} 
}
\]
e.g. $\rR \fatsemi \rS$ is the relational composite $\rR_l ; \rS$. \endbox
\end{definition}

\begin{example}[$\Dep$-morphisms]
  \label{ex:dep_morphisms}
  \item
  \begin{enumerate}
    \item \emph{$\Dep$-morphisms are closed under converse and union}.

    Given $\rR : \rG \to \rH$ then $\breve{\rR} : \breve{\rH} \to \breve{\rG}$ by taking the converse of the commutative square, which actually swaps the witnessing relations. We have $\emptyset : \rG \to \rH$ via empty witnessing relations. Given $\rR, \rS: \rG \to \rH$ then $\rR \cup \rS: \rG \to \rH$ by (i) unioning the respective witnessing relations, (ii) the bilinearity of relational composition w.r.t.\ union.

    \item \emph{Bipartite graph isomorphisms $\beta : G_1 \to G_2$ induce $\Dep$-isomorphisms.}
    
    Suppose we have a bipartite graph isomorphism $\beta : G_1 \to G_2$ where each $G_i = (V_i, \rE_i)$, so $\rE_1(x,y) \iff \rE_2(\beta(x),\beta(y))$. Given any bipartition $(X, Y)$ of $G_1$ we obtain a bipartition $(\beta[X], \beta[Y])$ of $G_2$. Changing notation provides the $\Dep$-morphism below left:
    \[
      \begin{tabular}{ccc}
        $\xymatrix@=15pt{
          Y \ar@{->}[rr]^{\beta |_{Y \times \beta[Y]}}  && \beta[Y]
          \\
          X \ar[u]^-{\rG_1} \ar@{..>}[urr]^-\rR \ar@{->}[rr]_{\beta |_{X \times \beta[X]} } && \beta[X] \ar[u]_-{\rG_2}
        }$
        &&
        $\xymatrix@=15pt{
          \beta[Y] \ar@{->}[rr]^{\breve{\beta} |_{\beta[Y] \times Y}}  && Y
          \\
          \beta[X] \ar[u]^-{\rG_2} \ar@{..>}[urr]^-\rS \ar@{->}[rr]_{\breve{\beta} |_{\beta[X] \times X} } && X \ar[u]_-{\rG_1}
        }$
      \end{tabular}
    \]
    where each $\rG_i := \rE_i |_{X \times Y}$. The bijective inverse $\beta^{-1} = \breve{\beta}$ provides witnessing relations in the opposite direction
    i.e.\ the $\Dep$-morphism $\rS : \rG_2 \to \rG_1$ above right. These morphisms are mutually inverse: $\rG_1$ is $\Dep$-isomorphic to $\rG_2$.

    \item \emph{The canonical quotient poset of a preorder defines a $\Dep$-isomorphism.}
    
    Let $\rG \subseteq X \times X$ be a transitive and reflexive relation. There is a canonical way to construct a poset $\pP = (X/{\rE}, \leq_{\pP})$ via the equivalence relation $\rE(x_1, x_2) :\iff \rG(x_1, x_2) \land \rG(x_2, x_1)$, where $\sem{x_1}_\rE \leq_{\pP} \sem{x_2}_\rE :\iff \rG(x_1, x_2)$.
    
    Consider the $\Rel$-diagram:
    \[
      \xymatrix@=15pt{
        X \ar@{->}[rr]^-{\nin}
        && \{ \overline{\breve{\rG}[x]} : x \in X \} \ar@{=}[r]
        &
        \{ \overline{\bigcup \down_{\pP} \sem{x}_\rE} : x \in X \} \ar[rrr]^-{(\lambda \sem{x}_\rE. \overline{\breve{\rG}[x]})\spbreve}
        &&& X / \rE
        \\
        X \ar[u]^-{\rG} \ar@{->}[rr]_-{\lambda x.\rG[x]}
        && \{ \rG[x] : x \in X \} \ar@{=}[r] \ar[u]_-{\nsubseteq}
        &
        \{ \bigcup \up_{\pP} \sem{x}_\rE : x \in X \}
        \ar@{->}[rrr]_-{(\lambda \sem{x}_\rE. \rG[x])\spbreve}
        &&& X / \rE \ar[u]_{\leq_{\pP}}
      }
    \]
    Note that $\rG[x]$ is the `upwards closure' i.e.\ the union of the upwards closure $\up_{\pP} \sem{x}_{\rE}$, whereas $\breve{\rG}[x]$ is the `downwards closure' in a similar manner. The left square commutes for completely general reasons, defining the $\Dep$-morphism:
    \[
      \rR(x_1, \overline{\breve{\rG}[x_2]})
      :\iff \exists x \in X.[\rG(x_1, x) \text{ and } \rG(x, x_2) ]
      \iff \rG(x_1, x_2).
    \]
    The right square involves bijections via (i) identifying elements of $\pP$ with principal up/downsets, (ii) the disjointness of equivalence classes. It also commutes:
    \[
      \begin{tabular}{lll}
        $\bigcup \up_{\pP} \sem{x_1}_\rE \nsubseteq \overline{\bigcup \down_{\pP} \sem{x_2}_\rE}$
        &
        $\iff \bigcup \up_{\pP} \sem{x_1}_\rE \;\cap\; \bigcup \down_{\pP} \sem{x_2}_\rE \neq \emptyset$
        \\&
        $\iff \exists x \in X. \sem{x_1}_\rE \leq_\pP \sem{x}_\rE \leq_\pP \sem{x_2}_\rE$
        \\&
        $\iff \sem{x_1}_\rE \leq_{\pP} \sem{x_2}_\rE$.
      \end{tabular}
    \]
    In fact, $\rR : \rG \to \;\nsubseteq$ is an instance of the natural isomorphism $\red_\rG$ from Theorem \ref{thm:bicliq_jirr_equivalent} further below, and the right square defines a $\Dep$-isomorphism by Example 2 above. Thus $\rG \cong \;\leq_\pP$, although whenever $|X| > |X / \rE|$ this isomorphism \emph{cannot arise from a bipartite graph isomorphism}.

    \item \emph{Monotonicity can be characterised by $\Dep$-morphisms.}
    
    Given finite posets $\pP$ and $\pQ$, a function $f: P \to Q$ is monotonic iff the following $\Rel$-diagram commutes:
    \[
    \xymatrix@=15pt{
    P \ar[r]^f & Q \ar[r]^{\leq_{\pQ}} & Q 
    \\
    P \ar[u]^{\leq_\pP} \ar[rr]_f && Q \ar[u]_{\leq_\pQ}
    }
    \]
    as the reader may verify. Actually, $f$ is monotonic iff $f ; \leq_\pQ  \,:\, \leq_\pP \,\to\, \leq_\pQ$ is a $\Dep$-morphism.

    \item \emph{Biclique edge-coverings amount to $\Dep$-monos}.

    Generally speaking, $\Dep$-morphisms represent two edge-coverings of a bipartitioned graph. A \emph{single edge-covering} amounts to a $\Dep$-mono of a special kind:
      \[
        \xymatrix@=15pt{
          \rG_t \ar[rr]^{\Delta_{\rG_t}}  && \rG_t
          \\
          \rG_s \ar[u]^-{\rG} \ar[urr]^-\rG \ar@{->}[rr]_{\rG_l} && \rH_s \ar[u]_-{\rH}
        }
      \]
      i.e.\ morphisms $\rG: \rG \to \rH$ where additionally $\rG_t = \rH_t$. Later we'll see that any mono $\rR : \rG \to \rI$ induces such a $\rG : \rG \to \rH$ where $|\rH_s| \leq |\rI_s|$ and $|\rH_t| \leq |\rI_t|$. 
      
    \item \emph{Biclique edge-coverings amount to $\Dep$-epis}.

    Analogous to the previous example, a single edge-covering can be represented as a $\Dep$-epi $\rG : \rH \to \rG$ where $\rG_s = \rH_s$. This will follow from self-duality i.e.\ epis are precisely the converses of monos.
    
    \endbox

  \end{enumerate}
\end{example}

\begin{lemma}
$\Dep$ is a well-defined category
\end{lemma}

\begin{proof}
$id_\rG := \rG : \rG \to \rG$ is well-defined via witnesses $\Delta_{\rG_s} ; \rG = \rG = \rG ; \Delta_{\rG_t}\spbreve$. Each composite $\rR \fatsemi \rS$ is well-defined via the composite witnesses, see the diagram in Definition \ref{def:category_bicliq}. Composites are independent of the witnesses of their components since $\rR \fatsemi \rS = \rR_l ; \rS = \rR ; \rS_r\spbreve$. Composition is associative because the respective composition of $\Rel$-diagrams is associative. Finally, given $\rR : \rG \to \rH$ then $id_{\rG} \fatsemi \rR = \Delta_{\rG_s} ; \rR = \rR$ and similarly $\rR \fatsemi id_{\rH} = \rR ; \Delta_{\rH_t}\spbreve = \rH$.
\end{proof}

\smallskip
We'll introduce further notation and auxiliary results. In particular, we'll prove that each $\Dep$-morphism $\rR$ has canonical inclusion-maximum witnesses.

\begin{definition}[$(-)^\up$ and $(-)^\down$]
\label{def:up_down}
Given any relation $\rR \subseteq \rR_s \times \rR_t$ between finite sets we define two functions:
\[
\begin{tabular}{lcl}
$\rR^\up : \Pow \rR_s \to \Pow \rR_t$
&&
$\rR^\down : \Pow \rR_t \to \Pow \rR_s$
\\
$\rR^\up (X) := \rR[X]$
&&
$\rR^\down (Y) := \{ x \in \rR_s : \rR[x] \subseteq Y \}$.
\end{tabular}
\]
Then $\rR^\up$ is called the \emph{$\rR$-image function} whereas $\rR^\down$ is called the \emph{$\rR$-preimage function}. They induce a closure operator and an interior operator (co-closure operator) as follows:
\[
\begin{tabular}{c}
$\cl_\rR := \rR^\down \circ \rR^\up : (\Pow \rR_s,\subseteq) \to (\Pow \rR_s,\subseteq)$
\\
$\inte_\rR := \rR^\up \circ \rR^\down : (\Pow \rR_t,\subseteq) \to (\Pow \rR_t,\subseteq)$
\end{tabular}
\]
See Definition \ref{def:cl_inte} for background.
\endbox
\end{definition}

\smallskip

\begin{note}
$\rR^\down$ is called the \emph{$\rR$-preimage function} because it generalises the usual preimage function of a function. That is, given any function (= functional relation) $f : X \to Y$ then $f^\down(B) := \{ x \in X : f[x] \subseteq B \} = \{ x \in X : f(x) \in B \}$. 
\end{note}

\smallskip

\begin{note}
The operators $(-)^\up$ and $(-)^\down$ faithfully represent relational composition as functional composition.
\begin{enumerate}
\item
$(-)^\up$ defines an equivalence functor (in fact, isomorphism) from the category of finite sets and relations $\Rel_f$ to the full subcategory of  $\JSL_f$  with objects $\JPow X = (\Pow X,\cup,\emptyset)$ where $X$ is a finite set.
\item
$(-)^\down$ defines an equivalence functor (in fact, isomorphism) from $\Rel_f^{op}$ to the full subcategory of $\JSL_f$ with objects $(\JPow X)^{\pOp} = (\Pow X,\cap,X)$ where $X$ is a finite set. \endbox
\end{enumerate}
\end{note}

\smallskip

\begin{lemma}[Relating $(-)^\up$ and $(-)^\down$]
\label{lem:up_down_basic}
\item
Let $\rG$, $\rH$ be relations between finite sets, $\rR \subseteq \rG_s \times \rH_t$, $\rS \subseteq \rH_s \times \rI_t$ any relations and $X$ any finite set.
\begin{enumerate}
\item
We have the adjoint relationship:
\[
(\up\dashv\down)
\qquad
\rR^\up(X) \subseteq Y
\iff
X \subseteq \rR^\down(Y)
\qquad
\qquad
\text{for all subsets $X \subseteq \rG_s$, $Y \subseteq \rH_t$}
\]
hence they actually define adjoint $\JSL_f$-morphisms:
\[
\begin{tabular}{llll}
$\rR^\up$ & $:(\Pow \rG_s,\cup,\emptyset)$ & $\to$ & $(\Pow \rH_t,\cup,\emptyset)$
\\
$\rR^\down$ & $:(\Pow \rH_t,\cap,\rH_t)$ & $\to$ & $(\Pow \rG_s,\cap,\rG_s)$
\end{tabular}
\]

\item
$\cl_\rR$ is a well-defined closure operator and $\inte_\rR$ is a well-defined interior operator.

\item
The following labelled equalities hold:
\[
\begin{tabular}{ccccc}
$(\up\Delta)$ & $\Delta_X^\up = id_{\Pow X}$
&&
$\Delta_X^\down = id_{\Pow X}$ & $(\down\Delta)$
\\[1ex]
$(\up\circ)$ & $(\rR;\rS)^\up = \rS^\up \circ \rR^\up$
&&
$(\rR;\rS)^\down = \rR^\down \circ \rS^\down$ & $(\down\circ)$
\\[1ex]
$(\up\down\up)$ & $\rR^\up \circ \rR^\down \circ \rR^\up = \rR^\up$
&&
$\rR^\down \circ \rR^\up \circ \rR^\down = \rR^\down$ & $(\down\up\down)$
\\[1ex]
$(\neg\up\neg)$ & $\neg_{\rG_t} \circ \rR^\up \circ \neg_{\rG_s} = \breve{\rR}^\down$
&&
$\neg_{\rG_s} \circ \rR^\down \circ \neg_{\rG_t} = \breve{\rR}^\up$ & $(\neg\down\neg)$
\end{tabular}
\]
The rules $(\neg\up\neg)$ and $(\neg\down\neg)$ are referred to as `De Morgan dualities'.

\item
We have two sets of four equivalent statements:
\[
\begin{tabular}{c|c|c}
equivalent statements &\qquad& equivalent statements
\\ \hline
$\rR^\up = \rR^\up \circ \cl_\rG$
&&
$\rR^\up = \inte_\rH \circ \rR^\up$.
\\
$\rR^\down = \cl_\rG \circ \rR^\down$
&&
$\rR^\down = \rR^\down \circ \inte_\rH$.
\\
$\breve{\rR}^\down = \breve{\rR}^\down \circ \inte_{\breve{\rG}}$
&&
$\breve{\rR}^\down = \cl_{\breve{\rH}} \circ \breve{\rR}^\down$
\\
$\breve{\rR}^\up = \inte_{\breve{\rG}} \circ \breve{\rR}^\up$
&&
$\breve{\rR}^\up  = \breve{\rR}^\up \circ \cl_{\breve{\rH}}$
\end{tabular}
\]

\end{enumerate}
\end{lemma}

\begin{proof}
\item
\begin{enumerate}
\item
$\rR^\up(X) \subseteq Y \iff \rR[X] \subseteq Y \iff \forall x \in X. \rR[x] \subseteq Y \iff X \subseteq \rR^\down(Y)$ establishes the adjunction. Thus $\rR^\up$ preserves all colimits = joins in $(\Pow \rG_s,\subseteq)$ = unions, and also $\rR^\down$ preserves all limits = meets in $(\Pow \rH_t,\subseteq)$ = intersections.

\item
First observe that $\rR^\up$ defines a monotone function on $(\Pow \rR_s,\subseteq)$, and $\rR^\down$ defines a monotone function on $(\Pow \rR_t,\subseteq)$. Then this follows from (1) via Lemma \ref{lem:adjoint_cl_in}.

\item
Regarding the topmost rules:
\[
\Delta_X^\up 
= \lambda A \subseteq X.\Delta_X[A] 
= id_{\Pow X}
= \lambda A \subseteq X.\{ x \in X : \{x\} \subseteq A \}
= \Delta_X^\down
\]
Next we prove $(\up\circ)$ and $(\up\down\up)$:
\[
\begin{tabular}{c}
$(\rR ; \rS)^\up(X) = (\rR ; \rS)[X] = \rS[\rR[X]] = \rS^\up \circ \rR^\up (X)$
\\[1ex]
$\rR^\up \circ \rR^\down \circ \rR^\up(X)
= \rR[\{ x \in X : \rR[x] \subseteq \rR[X] \}]
= \rR[X]
= \rR^\up(X)$
\end{tabular}
\]
and now $(\down\circ)$ and $(\down\up\down)$:
\[
\begin{tabular}{lll}
$(\rR ; \rS)^\down(Z)$
& $= \{ x \in X : \rR;\rS[x] \subseteq Z \}$
\\&
$= \{ x \in X : \rS[\rR[x]] \subseteq Z \} $
\\&
$= \{ x \in X : \rR[x] \subseteq \rS^\down(Z) \}$
& by adjoint relationship
\\&
$= \rR^\down(\rS^\down(Z))$
\\
\\
$\rR^\down \circ \rR^\up \circ \rR^\down(Y)$
&
$= \rR^\down(\rR[\{x \in X : \rR[x] \subseteq Y \}])$
\\&
$= \{ x \in X : \rR[x] \subseteq \{ y \in Y : \exists x \in X.\, y \in \rR[x] \subseteq Y \} \}$
\\&
$= \{ x \in X : \rR[x] \subseteq Y  \}$
\\&
$= \rR^\down(Y)$
\end{tabular}
\]
Finally we prove the `De Morgan dualities'. Firstly, $(\neg\up\neg)$  holds because:
\[
\begin{tabular}{lll}
$\neg_{\rG_t} \circ \rR^\up \circ \neg_{\rG_s}(X)$
&
$= \overline{\rR[\overline{X}]}$
\\&
$= \{ h_t \in \rH_t : \neg\exists g_s \in \overline{X}.\rR(g_s,h_t) \}$
\\&
$= \{ h_t \in \rH_t : \neg\exists g_s \in \overline{X}.\breve{\rR}(h_t,g_s) \}$
\\&
$= \{ h_t \in \rH_t : \breve{\rR}[h_t] \subseteq X \}$
\\&
$= \breve{\rR}^\down(X)$
\end{tabular}
\]
and $(\neg\down\neg)$ follows by setting $\rR := \breve{\rR}$ and cancelling involutions.

\item
If the left-hand set of four statements are equivalent, then so are the right-hand set of four statements. This follows by substituting $\rR \mapsto \breve{\rR}$. Also, in the left-hand statements, the last two follow from the first two by applying De Morgan duality. Then it suffices to prove that the first two statements on the left are equivalent.

The pointwise-inclusion-orderings $\rR^\down \leq \cl_\rG \circ \rR^\down$ and $\rR^\up \leq \rR^\up \circ \cl_\rG$ always hold because $\cl_\rG$ is extensive and both $\rR^\down$ and $\rR^\up$ are monotone. We must prove that:
\[
\rG^\down \circ \rG^\up \circ \rR^\down \leq \rR^\down 
\quad\iff\quad
\rR^\up \circ \rG^\down \circ \rG^\up  \leq \rR^\up
\]
\begin{enumerate}
\item
$(\To)$ Applying the adjoint relationship yields $\rR^\up \circ \rG^\down \circ \rG^\up \circ \rR^\down \leq id_{\Pow \rH_t}$, so precomposing with the monotone function $\rR^\up$ yields $\rR^\up \circ \rG^\down \circ \rG^\up \circ \rR^\down \circ \rR^\up \leq \rR^\up$. Finally observe that:
\[
\rR^\up \circ \rG^\down \circ \rG^\up
\leq
\rR^\up \circ \rG^\down \circ \rG^\up \circ \rR^\down \circ \rR^\up
\leq \rR^\up
\]
because $\rR^\down \circ \rR^\up$ is a closure operator by (2) and hence extensive, and $\rR^\up \circ \rG^\down \circ \rG^\up$ is monotone.

\item
$(\oT)$ Applying the adjoint relationship yields $\cl_\rG \leq \rR^\down \circ \rR^\up$, so precomposing with $\rR^\down$ we obtain:
\[
\rG^\down \circ \rG^\up \circ \rR^\down 
\leq 
\rR^\down \circ \rR^\up \circ \rR^\down 
= \rR^\down
\]
where the final equality is by $(\down\up\down)$.
\end{enumerate}

\end{enumerate}
\end{proof}

We are now ready to formalise the canonical maximum witnesses of $\Dep$-morphisms. We'll also prove an important functional characterisation of $\Dep$-morphisms.

\begin{definition}
\label{def:bicliq_mor_components}
The \emph{component relations} of a $\Dep$-morphism $\rR : \rG \to \rH$ are defined:
\[
\begin{tabular}{ll}
$\rR_-$ & $:= \{ (g_s,h_s) \in \rG_s \times \rH_s : h_s \in \rH^\down(\rR[g_s]) \}$
\\
$\rR_+$ & $:= \{ (h_t,g_t) \in \rH_t \times \rG_t : g_t \in \breve{\rG}^\down(\breve{\rR}[h_t]) \}$
\end{tabular}
\]
\endbox
\end{definition}

\begin{example}[Component relations of $id_{\leq_\pP}$]
Given a finite poset $\pP$ we compute the component relations of the  identity-morphism $id_{\leq_\pP} = \; \leq_\pP$. Firstly:
\[
\begin{tabular}{c}
$\leq_\pP^\down(\leq_\pP[p_s])
= \; \leq_\pP^\down(\up_\pP p_s)
= \{ p \in P : \; \up_\pP p \subseteq \; \up_\pP p_s)
= \; \up_\pP p_s$
\\[1ex]
$\breve{\leq_\pP}^\down(\breve{\leq_\pP}[p_t])
= \; \leq_{\pP^{\pOp}}^\down(\up_{\pP^{\pOp}} p_t)
\stackrel{*}{=} \; \up_{\pP^{\pOp}} p_t
= \; \down_\pP p_t$
\end{tabular}
\]
where (*) follows from the 1st line. Consequently $(\leq_\pP)_- =  \; \leq_\pP$ and $(\leq_\pP)_+ = \; \leq_{\pP^{\pOp}}$. These are witnesses  because:
\[
\leq_\pP \; ; \; \leq_\pP \; = \; \leq_\pP \; = \; \leq_\pP ; \leq_{\pP^{\pOp}}\spbreve
\]
by reflexivity and transitivity. Concerning maximality, if $\rR_l \subseteq P \times P$ satisfies  $\rR_l ;  \leq_\pP \; = \; \leq_\pP$ then $\rR_l  \subseteq \; \leq_\pP$ because if $\rR_l(p_1,p_2)$ then $p_1 \leq_\pP p_2$ by reflexivity of order-relations. Similarly if $\leq_\pP \; = \; \leq_\pP ; \rR_r\spbreve$ then $\rR_r \subseteq \; \leq_{\pP^{\pOp}}$. \endbox
\end{example}

\smallskip

\begin{lemma}[Morphism characterisation and maximum witnesses]
  \label{lem:bicliq_mor_char_max_witness}
  \item
  \begin{enumerate}
  \item
  A relation $\rR \subseteq \rG_s \times \rH_t$ defines a $\Dep$-morphism $\rG \to \rH$ iff
  \[
  \rR^\up \circ \cl_\rG = \rR^\up = \inte_\rH \circ \rR^\up,
  \]
  or equivalently $\rR^\up \circ \cl_\rG = \inte_\rH \circ \rR^\up$.
  
  \item
  Each $\Dep$-morphism $\rR : \rG \to \rH$ has the maximum witness $(\rR_-,\rR_+)$ i.e.\ 
  \begin{enumerate}
  \item
  $\rR_- ; \rH = \rR = \rG ; \rR_+\spbreve$, and 
  \item
  for any $(\rR_l,\rR_r)$ such that $\rR_l ; \rH = \rR = \rG ; \rR_r\spbreve$ we have both $\rR_l \subseteq \rR_-$ and $\rR_r \subseteq \rR_+$.
  \end{enumerate}
  \end{enumerate}
  \end{lemma}
  
  \begin{proof}
  \item
  \begin{enumerate}[i.]
  \item
  Let us prove half of the first statement. Assuming that $\rR : \rG \to \rH$ is a $\Dep$-morphism then we have some witnessing relations $(\rR_l,\rR_r)$ such that $\rR_l ; \rH = \rR = \rG ; \rR_r\spbreve$. Consequently:
  \[
  \begin{tabular}{lll}
  $\rR^\up \circ \cl_\rG$
  &
  $= (\rG ; \rR_+\spbreve)^\up \circ \rG^\down \circ \rG^\up$
  & by assumption and definition
  \\&
  $= (\rR_+\spbreve)^\up \circ \rG^\up \circ \rG^\down \circ \rG^\up$
  & by Lemma \ref{lem:up_down_basic}.$(\up\circ)$
  \\&
  $= (\rR_+\spbreve)^\up \circ \rG^\up$
  & by Lemma \ref{lem:up_down_basic}.$(\up\down\up)$
  \\&
  $= (\rG ; \rR_+\spbreve)^\up$
  & by Lemma \ref{lem:up_down_basic}.$(\up\circ)$
  \\&
  $= \rR^\up$
  & by assumption
  \\
  \\
  $\inte_\rH \circ \rR^\up$
  &
  $= \rH^\up \circ \rH^\down \circ (\rR_- ; \rH)^\up$
  & by assumption and definition
  \\&
  $= \rH^\up \circ \rH^\down \circ \rH^\up \circ \rR_-$
  & by Lemma \ref{lem:up_down_basic}.$(\up\circ)$
  \\&
  $= \rH^\up \circ \rR_-^\up$
  & by Lemma \ref{lem:up_down_basic}.$(\up\down\up)$
  \\&
  $= (\rR_- ; \rH)^\up$
  & by Lemma \ref{lem:up_down_basic}.$(\up\circ)$
  \\&
  $= \rR^\up$
  & by assumption
  \end{tabular}
  \]
  
  \item
  Before proving the other half of the first statement, let us first prove the second statement i.e.\ we again assume $\rR : \rG \to \rH$ is a $\Dep$-morphism and now know that $\rR^\up \circ \cl_\rG = \rR^\up = \inte_\rH \circ \rR^\up$. We first show that the `associated component relations' $(\rR_-,\rR_+)$ witness the fact that $\rR$ is a morphism.
  \[
  \begin{tabular}{lll}
  $\rR_- ; \rH(g_s,h_t)$
  & 
  $\iff \exists h_s \in \rH_s.[ h_s \in \rH^\down(\rR[g_s]) \text{ and } \rH(h_s,h_t) ]$
  & by definition of $\rR_-$
  \\&
  $\iff h_t \in \rH^\up \circ \rH^\down \circ \rR^\up(\{g_s\})$
  & by definition of $\rH^\up$
  \\&
  $\iff h_t \in \rR^\up(\{g_s\})$
  & since $\rR^\up = \inte_\rH \circ \rR^\up$
  \\&
  $\iff \rR(g_s,h_t)$
  \\
  \\
  $\rG ; \rR_+\spbreve(g_s,h_t)$
  &
  $\iff \exists g_t \in \rG_t.[ \rG(g_s,g_t) \text{ and } g_t \in \breve{\rG}^\down(\breve{\rR}[h_t])  ]$
  & by definition of $\rR_+$
  \\&
  $\iff \exists g_t \in \rG_t.[ \breve{\rG}(g_t,g_s) \text{ and } g_t \in \breve{\rG}^\down(\breve{\rR}[h_t])  ]$
  \\&
  $\iff g_s \in \breve{\rG}^\up \circ \breve{\rG}^\down \circ \breve{\rR}^\up(\{h_t\})$
  & by definition of $\breve{\rG}^\up$
  \\&
  $\iff g_s \in \breve{\rR}^\up(\{h_t\})$
  & since $\breve{\rR}^\up = \inte_{\breve{\rG}} \circ \breve{\rR}^\up$
  \\&
  $\iff \rR(g_s,h_t)$
  \end{tabular}
  \]
  The penultimate equivalence follows because we know that $\rR^\up \circ \cl_\rG = \rR^\up$ and may apply Lemma \ref{lem:up_down_basic}.4. To show that $(\rR_-,\rR_+)$ is maximum, take any other witnesses i.e.\ $\rR_l ; \rH = \rR = \rG ; \rR_r\spbreve$. Then:
  \[
  \begin{tabular}{lll}
  $\rR_l(g_s,h_s)$
  &
  $\implies \forall h_t \in \rH_t.[ \rH(h_s,h_t) \To \rR(g_s,h_t) ]$
  &
  since $\rR_l ; \rH = \rR$
  \\&
  $\iff \rH[h_s] \subseteq \rR[g_s]$
  \\&
  $\iff h_s \in \rH^\down(\rR[g_s])$
  & by definition of $\rH^\down$
  \\&
  $\iff \rR_-(g_s,h_s)$
  & by definition of $\rR_-$
  \\
  \\
  $\rR_r(h_t,g_t)$
  &
  $\iff \rR_r\spbreve(g_t,h_t)$
  \\&
  $\implies \forall g_s \in \rG_s.[ \rG(g_s,g_t) \To \rR(g_s,h_t) ]$
  & since $\rR = \rG ; \rR_r\spbreve$
  \\&
  $\iff \breve{\rG}[g_t] \subseteq \breve{\rR}[h_t]$
  \\&
  $\iff g_t \in \breve{\rG}^\down (\breve{\rR}[h_t])$
  & by definition of $\breve{\rG}^\down$
  \\&
  $\iff \rR_+(h_t,g_t)$
  & by definition of $\rR_+$
  \end{tabular}
  \]
  
  \item
  Let us prove the remaining part of the first statement:
  \begin{quote}
  given a relation $\rR \subseteq \rG_s \times \rH_t$ such that $\rR^\up \circ \cl_\rG = \rR^\up = \inte_\rH \circ \rR^\up$ we must establish that $\rR$ defines a $\Dep$-morphism of type $\rG \to \rH$.
  \end{quote}

   Even though we don't yet know that $\rR$ is a $\Dep$-morphism, we can apply Definition \ref{def:bicliq_mor_components} to obtain the two relations $(\rR_-,\rR_+)$. Then we can reuse the first proof in (ii) above to deduce that $\rR = \rR_- ; \rH$. Furthermore we can also reuse the proof that $\rR = \rG ; \rR_+\spbreve$ because the assumption $\rR^\up \circ \cl_\rG = \rR^\up$ implies that $\breve{\rR}^\up = \inte_{\breve{\rG}} \circ \breve{\rR}^\up$ by Lemma \ref{lem:up_down_basic}.4.
  
  \item
  Finally, the first statement can be weakened to $\rR^\up \circ \cl_\rG = \inte_\rH \circ \rR^\up$ because this already implies both composites are equal to $\rR^\up$. Indeed, since $\cl_\rG$ is a closure operator and $\inte_\rG$ is an interior operator,
  \[
  \inte_\rH \circ \rR^\up \subseteq \rR^\up \subseteq \rR^\up \circ \cl_\rG
  \]
  so we can replace the inclusions by equalities.
\end{enumerate}
\end{proof}

Here is yet another useful result.

\begin{lemma}[Computing composites]
\label{lem:bicliq_func_comp}
For any $\Dep$-morphisms $\rR : \rG \to \rH$ and $\rS : \rH \to \rI$,
\[
\begin{tabular}{cc}
$(\up\fatsemi)$  & $(\rR \fatsemi \rS)^\up = \rS^\up \circ \rH^\down \circ \rR^\up$
\\
$(\down\fatsemi)$ & $(\rR \fatsemi \rS)^\down = \rR^\down \circ \rH^\up \circ \rS^\down$.
\end{tabular}
\]
Finally, $
\rR_- ; \rS 
= \rR_- ; \rS_- ; \rI 
= \rR \fatsemi \rS 
= \rG ; \rR_+\spbreve ; \rS_+\spbreve
= \rR ; \rS_+\spbreve 
$.
\end{lemma}

\begin{proof}
Recalling that $\rR$ has canonical witnesses $(\rR_-, \rR_+\spbreve)$, let us prove $(\up\fatsemi)$.
\[
\begin{tabular}{lll}
$(\rR \fatsemi \rS)^\up$
&
$= (\rR_- ; \rS)^\up$
& by definition
\\&
$= \rS^\up \circ \rR_-^\up$
& by $(\up\circ)$
\\&
$= \rS^\up \circ \rH^\down \circ \rH^\up \circ \rR_-^\up$
& by Lemma \ref{lem:bicliq_mor_char_max_witness}
\\&
$= \rS^\up \circ \rH^\down \circ (\rR_- ; \rH)^\up$
& by $(\up\circ)$
\\&
$= \rS^\up \circ \rH^\down \circ \rR^\up$
& since $\rR = \rR_- ; \rH$
\end{tabular}
\]
We infer $(\down\fatsemi)$ because $(\up\fatsemi)$ is an equality of $\JSL_f$-morphisms, so we can take adjoints, flipping the composition and also the direction of the arrows $\up$ and $\down$. The final claim follows by the definition of $\Dep$-composition.
\end{proof}

\begin{definition}
\label{def:bicliq_self_duality}
The \emph{self-duality} $(-)\spcheck : \Dep^{op} \to \Dep$ takes the converse of both objects and morphisms, and moreover flips the component relations.
\[
\rG\spcheck := \breve{\rG}
\qquad
\dfrac{\rR : \rG \to \rH}{(\rR^{op})\spcheck := \breve{\rR} : \breve{\rH} \to \breve{\rG}}
\qquad
(\rR\spcheck)_- = \rR_+
\qquad
(\rR\spcheck)_+ = \rR_-
\]
\endbox
\end{definition}

\begin{theorem}[Self-duality of $\Dep$]
\label{thm:bicliq_self_dual}
$(-)\spcheck : \Dep^{op} \to \Dep$ is a well-defined equivalence functor with respective natural isomorphism:
\[
\alpha : \Id_{\Dep} \To (-)\spcheck \circ ((-)\spcheck)^{op}
\qquad
\alpha_\rG := id_\rG = \rG
\]
\end{theorem}

\begin{proof}
$(-)\spcheck$'s action on objects is certainly well-defined. Regarding its action on morphisms, given a morphism $\rR : \rG \to \rH$ we have a relation $\rR \subseteq \rG_s \times \rH_t$ so that $\breve{\rR} \subseteq \rH_t \times \rG_s = \breve{\rH}_s \times \breve{\rG}_t$ has the correct type $\breve{\rH} \to \breve{\rG}$. To establish that $\breve{\rR}$ is a well-defined morphism we must show that:
\[
\breve{\rR}^\up \circ \cl_{\breve{\rH}} = \breve{\rR}^\up = \inte_{\breve{\rG}} \circ \breve{\rR}^\up
\]
by Lemma \ref{lem:bicliq_mor_char_max_witness}. But by the same Lemma we already know that $\rR^\up \circ \cl_\rG = \rR^\up = \inte_\rH \circ \rR^\up$, so by Lemma \ref{lem:up_down_basic}.4 we deduce that the above equivalent statements hold. Preservation of identity morphisms follows because: 
\[
(id_\rG)\spcheck 
= (\rG : \rG \to \rG)\spcheck = \breve{\rG} : \breve{\rG} \to \breve{\rG} = id_{\breve{\rG}} = id_{\rG\spcheck}
\]
Next we show preservation of composition, i.e.\ given compatible $\Dep$-morphisms $\rR : \rG \to \rH$ and $\rS : \rH \to \rI$ we must show that $((\rR \fatsemi \rS)^{op})\spcheck = \rS\spcheck \fatsemi \rR\spcheck$. We first point out the typing $(\rR \fatsemi \rS)\spbreve : \breve{\rI} \to \breve{\rG}$ so that $(\rR \fatsemi \rS)\spbreve \subseteq \breve{\rI}_s \times \breve{\rG}_t = \rI_t \times \rG_s$. Then we calculate as follows:
\[
\begin{tabular}{lll}
$((\rR\fatsemi\rS)\spbreve)^\up$
&
$= \neg_{\rG_s} \circ (\rR\fatsemi\rS)^\down \circ \neg_{\rI_t}$
& by $(\neg\down\neg)$
\\&
$= \neg_{\rG_s} \circ (\rR^\down \circ \rH^\up \circ \rS^\down) \circ \neg_{\rI_t}$
& by $(\down\fatsemi)$
\\&
$= (\neg_{\rG_s} \circ \rR^\down \circ \neg_{\rH_t}) \circ (\neg_{\rH_t} \circ \rH^\up \circ \neg_{\rH_s}) \circ (\neg_{\rH_s} \circ \rS^\down \circ \neg_{\rI_t})$
\\&
$= \breve{\rR}^\up \circ \breve{\rH}^\down \circ \breve{\rS}^\up$
& by $(\neg\down\neg)$ and $(\neg\up\neg)$
\\&
$= (\breve{\rS} \fatsemi \breve{\rR})^\up$
& by $(\up\fatsemi)$
\\&
$= (\rS\spcheck \fatsemi \rR\spcheck)^\up$
\end{tabular}
\]
where we have implicitly used the fact that $\breve{\rR}$ and $\breve{\rS}$ are well-defined morphisms. Then $(-)\spcheck$ is a well-defined functor. Each component $\alpha_\rG = id_\rG$ is certainly an isomorphism. Naturality comes down to the equality $\alpha_\rG \fatsemi \breve{\rR}\spbreve = \rR \fatsemi \alpha_\rH$ for each morphism $\rR : \rG \to \rH$, which follows because $\Dep$ is a well-defined category and relational converse is involutive.

\smallskip
Finally we establish that $(\breve{\rR}_-,\breve{\rR}_+) = (\rR_+,\rR_-)$. Since $(\rR_-,\rR_+)$ is a witness for $\rR$ we have $\rR_- ; \rH = \rR = \rG ; \rR_+\spbreve$.  Applying relational converse yields $\rR_+ ; \breve{\rG} = \breve{\rR} = \breve{\rH} ; \rR_-\spbreve$ so that $\breve{\rR}$ has the witness $(\rR_+,\rR_-)$. By maximality of $\breve{\rR}$'s associated components we deduce that $\rR_+ \subseteq \breve{\rR}_-$ and $\rR_- \subseteq \breve{\rR}_+$. The reverse inclusions follow by the symmetric argument i.e.\ by starting with $\breve{\rR}$'s associated components. 
\end{proof}

\subsection{$\Dep$ is categorically equivalent to $\JSL_f$}
\label{subsec:dep_equiv_jsl}

Each $\rG$ has an associated interior operator $\inte_\rG$ by Definition \ref{def:up_down}. Let $O(\rG) \subseteq \Pow \rG_t$ be its fixpoints, which we will also refer to as the \emph{$\rG$-open sets}.

\smallskip

\begin{definition}[Equivalence functors between $\Dep$ and $\JSL_f$]
\label{def:open_pirr}
\item
\begin{enumerate}
\item
$\Open : \Dep \to \JSL_f$ is defined:
\[
\Open \rG := (O(\rG),\cup,\emptyset)
\qquad
\dfrac
{\rR : \rG \to \rH}
{\Open \rR := \lambda Y. \rR_+\spbreve[Y] : \Open(\rG) \to \Open(\rH) }
\]
recalling $\rR_+$ from Definition \ref{def:bicliq_mor_components}. Equivalently $\Open\rR := \lambda Y \in O(\rG). \rR^\up \circ \rG^\down(Y)$ (see below).

\item
$\Pirr : \JSL_f \to \Dep$ is defined:
\[
\begin{tabular}{c}
$\Pirr \aQ := \; \nleq_{\aQ}\; \subseteq J(\aQ) \times M(\aQ)
\qquad
\dfrac{f : \aQ \to \aR}
{\Pirr f := \{ (j,m) \in J(\aQ) \times M(\aR) : f(j) \nleq_\aR m \} : \Pirr\aQ \to \Pirr\aR}$
\end{tabular}
\]
with component relations:
\[
\begin{tabular}{ll}
$(\Pirr f)_-$ & $:= \{ (j_1,j_2) \in J(\aQ) \times J(\aR) : j_2 \leq_\aR f(j_1) \}$
\\
$(\Pirr f)_+$ & $:= \{ (m_1,m_2) \in M(\aR) \times M(\aQ) : f_*(m_1) \leq_\aQ m_2 \}$.
\end{tabular}
\]
It constructs the \emph{poset of irreducibles} introduced by Markowsky \cite{MarkowskyLat1975}.
\endbox
\end{enumerate}
\end{definition}

\bigskip
The two definitions of $\Open\rR$ above are consistent.

\begin{lemma}
\label{lem:open_r_two_defs}
For any $\Dep$-morphism $\rR : \rG \to \rH$ and $Y \in O(\rG)$ we have $\rR_+\spbreve[Y] = \rR^\up \circ \rG^\down(Y)$.
\end{lemma}

\begin{proof}
\[
\begin{tabular}{lll}
$\Open \rR(Y)$
&
$=\rR_+\spbreve[Y]$
\\&
$= (\rR_+\spbreve)^\up \circ \rG^\up \circ \rG^\down(Y) $
& $Y$ is $\rG$-open
\\&
$= (\rG;\rR_+\spbreve)^\up \circ \rG^\down(Y) $
& by $(\up\circ)$
\\&
$=  \rR^\up \circ \rG^\down(Y)$
& by Lemma \ref{lem:bicliq_mor_char_max_witness}.2 
\end{tabular}
\]
\end{proof}

\begin{example}
\item
\begin{enumerate}
\item
$\Open\Delta_X = \JPow X = (\Pow X,\cup,\emptyset)$ is a boolean semilattice. Recalling that every relation $\rR \subseteq X \times Y$ defines a $\Dep$-morphism $\rR : \Delta_X \to \Delta_Y$, then $\Open\rR : \JPow X \to \JPow Y$ has action:
\[
\Open\rR(A) := \rR^\up \circ (\Delta_X)^\down(A) = \rR^\up(A).
\]

\item
By (1), $\Open$ sends identity relations to boolean join-semilattices. This generalizes to bijections i.e.\ bijective functional relations. But there exist non-functional relations with this property too:
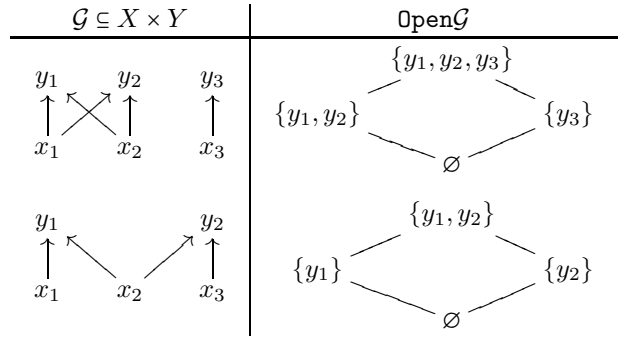
\[
\begin{tabular}{c|c}
$\rG \subseteq X \times Y$ & $\Open\rG$
\\ \hline
$\vcenter{\vbox{\xymatrix@=15pt{
y_1 & y_2 & y_3
\\
x_1 \ar[u]\ar[ur] & x_2 \ar[u]\ar[ul] & x_3 \ar[u]
\\
y_1 &  & y_2
\\
x_1 \ar[u] & x_2 \ar[ul]\ar[ur] & x_3 \ar[u]
}}}$
&
$\vcenter{\vbox{\xymatrix@=5pt{
& \{y_1,y_2,y_3\}
\\
\{y_1,y_2\}  \ar@{-}[ur] && \{y_3\}  \ar@{-}[ul]
\\
& \emptyset \ar@{-}[ul]\ar@{-}[ur]
\\
& \{y_1,y_2\}
\\
\{y_1\}  \ar@{-}[ur] && \{y_2\}  \ar@{-}[ul]
\\
& \emptyset \ar@{-}[ul]\ar@{-}[ur]
}}}$
\end{tabular}
\]

\item
Recall that each boolean semilattice $\aQ = (Q,\lor_\aQ,\bot_\aQ)$ has a unique bijective complementation operation:
\[
\neg_\aQ : Q \to Q
\qquad
\neg_\aQ (q) := \Lor_\aQ \overline{\up_\aQ q}.
\]
It turns out that $\Pirr\aQ = \;\nleq_\aQ \; \subseteq J(\aQ) \times M(\aQ)$ is precisely the domain/codomain restriction $\neg_\aQ : At(\aQ) \to CoAt(\aQ)$. This restriction is also bijective: an atom $a$ is not less than or equal to a coatom $c$ iff $c = \neg_\aQ a$.

\end{enumerate}
\end{example}

Before proving the well-definedness of $\Open$ and $\Pirr$ we provide a number of helpful results.

\begin{definition}[The finite lattice of $\rG$-open sets and its isomorphic lattice of $\rG$-closed sets]
\label{def:bip_cl_inte_lattice}
\item
Let $\rG$ be a relation between finite sets.
\begin{enumerate}
\item
Define two sets of subsets:
\[
\begin{tabular}{lll}
$O(\rG)$ & $:= O(\inte_\rG) \subseteq \Pow \rG_t$
& the \emph{$\rG$-open sets}.
\\
$C(\rG)$ & $:= C(\cl_\rG) \subseteq \Pow \rG_s$
& the \emph{$\rG$-closed sets}.
\end{tabular}
\]

\item
There are two inclusion-ordered bounded lattice structures on these sets:
\[
\begin{tabular}{lll}
$\latOp{\rG}$ 
& 
$:= (O(\rG),\cup,\emptyset,\land_{\latOp{\rG}},\inte_\rG(\rG_t))$
&
where $Y_1 \land_{\latOp{\rG}} Y_2 := \inte_\rG(Y_1 \cap Y_2)$,
\\
$\latCl{\rG}$
&
$:= (C(\rG),\lor_{\latCl{\rG}},\cl_\rG(\emptyset),\cap,\rG_s)$
&
  where $X_1 \lor_{\latCl{\rG}} X_2 := \cl_\rG(X_1 \cup X_2)$,
  \end{tabular}
  \]
  recalling that $\inte_\rG(\rG_t) = \rG[\rG_s]$ and $\cl_\rG(\emptyset) = \rG^\down(\emptyset)$.
  
\item
There is a lattice isomorphism $\theta_\rG : \latCl{\rG} \to \latOp{\rG}$:
  \[
  \theta_\rG(X) := \rG^\up(X) = \rG[X]
  \qquad
  \theta_\rG^{\bf-1}(Y) := \rG^\down(Y)
  \]
\item
There is a self-inverse lattice isomorphism $\kappa_\rG : (\latCl{\rG})^{\pOp} \to \latOp{\breve{\rG}}$ with action $\kappa_\rG(X) := \overline{X}$.
\endbox
\end{enumerate}
\end{definition}

\bigskip

\begin{lemma}[The bounded lattices of $\rG$-open/closed sets and their irreducibles]
\label{lem:lat_op_cl}
\item
\begin{enumerate}
\item
The bounded lattices $\latOp{\rG}$ and $\latCl{\rG}$ are well-defined.
\item
Each $\theta_\rG : \latCl{\rG} \to \latOp{\rG}$ and $\kappa_\rG : (\latCl{\rG})^{\pOp} \to \latOp{\breve{\rG}}$ are well-defined bounded lattice isomorphisms.

\item
The $\rG$-open sets and $\rG$-closed sets can be described as follows:
\[
\begin{tabular}{lll}
$O(\rG)$ & $= \{ \rG[S] : S \subseteq \rG_s \}$
& i.e.\ the closure of $\{ \rG[g_s] : g_s \in \rG_s \}$ under unions.
\\
$C(\rG)$ & $= \{ \rG^\down(S) : S \subseteq \rG_t \}$
& i.e.\ the closure of $\{ \rG^\down(\overline{g_t}) : g_t \in \rG_t \}$ under intersections.
\end{tabular}
\]
Finally, we have the following inclusions:
\[
J(\latOp{\rG}) \subseteq \{ \rG[g_s] : g_s \in \rG_s \}
\qquad
M(\latOp{\rG}) \subseteq \{ \inte_\rG(\overline{g_t}) : g_t \in \rG_t \}.
\]
\end{enumerate}
\end{lemma}

\begin{proof}
\item
\begin{enumerate}
\item
$\latOp{\rG}$ is the `standard' bounded inclusion-ordered lattice one obtains from an interior operator defined on the underlying poset of a bounded lattice. In detail,  $\inte_\rG$ is defined on $(\Pow\rG_t,\subseteq)$ and the latter has all joins = unions, hence the $\rG$-open sets $O(\rG)$ are closed under all possibly-empty unions (see Lemma \ref{lem:clo_co_basic}.3). Since $(O(\rG),\cup,\emptyset)$ is a finite join-semilattice it is also a bounded lattice in a unique way. The induced top is $\inte_\rG(\rG_t) = \rG^\up \circ \rG^\down(\rG_t) = \rG^\up(\rG_s) = \rG[\rG_s]$. The induced meet is:
\[
Y_1 \land Y_2 
:= \bigcup \{ Y \in O(\rG) : Y \subseteq Y_1 \cap Y_2 \}
= \inte_\rG(Y_1 \cap Y_2)
\]
where $\subseteq$ follows because $Y = \inte_\rG(Y) \subseteq \inte_\rG(Y_1 \cap Y_2)$ by monotonicity, and $\supseteq$ follows by co-extensivity. The argument concerning $\latCl{\rG}$ is analogous, noting that $\bot_{\latCl{\rG}} =  \cl_\rG(\emptyset) = \rG^\down \circ \rG^\up(\emptyset) = \rG^\down(\emptyset)$.

\item
We first show that both $\theta_\rG : \latCl{\rG} \to \latOp{\rG}$ and $\theta_\rG^{\bf-1} : \latOp{\rG} \to \latCl{\rG}$ are well-defined functions.
\[
\begin{tabular}{cc}
\begin{tabular}{lll}
$\theta_\rG(X)$
&
$= \theta_\rG(\cl_\rG(X))$
& $X$ is $\rG$-closed
\\&
$= \rG^\up \circ \rG^\down \circ \rG^\up(X)$
& by definition
\\&
$= \inte_\rG(\rG^\up(X))$
& by definition
\end{tabular}
&
\begin{tabular}{lll}
$\theta_\rG^{\bf-1}(Y)$
&
$= \theta_\rG^{\bf-1}(\inte_\rG(Y))$
& $Y$ is $\rG$-open
\\&
$= \rG^\down \circ \rG^\up \circ \rG^\down(Y)$
& by definition
\\&
$= \cl_\rG(\rG^\down(Y))$
& by definition
\end{tabular}
\end{tabular}
\]
Then for any $\rG$-open set $Y$ and $\rG$-closed set $X$ we have:
\[
\theta_\rG \circ \theta_\rG^{\bf-1}(Y) 
= \rG^\up \circ \rG^\down(Y) 
= \inte_\rG(Y) = Y
\qquad
\theta_\rG^{\bf-1} \circ \theta_\rG(X)
= \rG^\down \circ \rG^\up(X)
= \cl_\rG(X) = X
\]
so they are mutually inverse bijections. Finally they inherit monotonicity from $\rG^\up$ and $\rG^\down$, so they are order-isomorphisms and hence also bounded lattice isomorphisms.

Next we show that $\kappa_\rG : (\latCl{\rG})^{\pOp} \to \latOp{\breve{\rG}}$ is a well-defined bounded lattice isomorphism. Since relative complement is involutive and flips the inclusion-ordering, we need only show it is a well-defined surjective function.
\[
\begin{tabular}{lll}
$\kappa_\rG(X)$
&
$= \neg_{\rG_s}(X)$
\\&
$= \neg_{\rG_s} \circ \rG^\down \circ \rG^\up(X)$
& since $X$ is $\rG$-closed
\\&
$= \breve{\rG}^\up \circ \breve{\rG}^\down \circ \neg_{\rG_s} (X)$
& by De Morgan duality
\\&
$= \inte_{\breve{\rG}}(\overline{X})$
& by definition
\\
\\
$\kappa_\rG^{\bf-1}(Y)$
&
$= \neg_{\rG_s}(Y)$
\\&
$= \neg_{\rG_s} \circ \breve{\rG}^\up \circ \breve{\rG}^\down(Y)$
& since $Y$ is $\breve{\rG}$-open
\\&
$= \rG^\down \circ \rG^\up \circ \neg_{\rG_s} (Y)$
& by De Morgan duality
\\&
$= \cl_{\rG}(\overline{Y})$
& by definition
\end{tabular}
\]
The first equality establishes well-definedness and the second surjectivity i.e.\ each $\breve{\rG}$-open set $Y$ is the relative complement of the $\rG$-closed set $\cl_\rG(\overline{Y})$.

\item
We establish $O(\rG) = \{ \rG[S] : S \subseteq \rG_s \}$. Given any $Y \in O(\rG)$ then since $\theta_\rG : \latCl{\rG} \to \latOp{\rG}$ is bijective we deduce $Y = \rG[X]$ for some $X \in C(\rG) \subseteq \Pow\rG_s$. Conversely, given any subset $S \subseteq \rG_s$,
\[
\rG[S] 
= \rG^\up(S) 
\stackrel{(\up\down\up)}{=} \rG^\up \circ \rG^\down \circ \rG^\up(S) 
= \inte_\rG(\rG[S])
\]
Every $\rG[S]$ is clearly a possibly-empty union of the sets $\rG[s]$. Next we show that $C(\rG) = \{ \rG^\up(S) : S \subseteq \rG_t \}$. Given any $X \in C(\rG)$ then since $\theta_\rG^{\bf-1}$ is bijective we deduce $X = \rG^\down(Y)$ for some $Y \in O(\rG) \subseteq \Pow\rG_t$. Conversely, given any $S \subseteq \rG_t$ then:
\[
\rG^\down(S) 
\stackrel{!}{=} \rG^\down \circ \rG^\up \circ \rG^\down(S) 
= \cl_\rG(\rG^\down(S))
\]
where the marked equality follows by $(\down\up\down)$. Recalling that $\rG^\down$ preserves all possibly-empty intersections (it is right adjoint to $\rG^\up$), it follows that every $\rG^\down(S)$ is a possibly-empty intersections of the special sets $\rG^\down(\overline{g_t})$.

The inclusion $J(\latOp{\rG}) \subseteq \{ \rG[g_s] : g_s \in \rG_s \}$ follows because the latter sets join-generate $\latOp{\rG}$, and thus must contain the join-irreducibles by Lemma \ref{lem:std_order_theory}.6. Concerning the inclusion $M(\latOp{\rG}) \subseteq \{ \inte_\rG(\overline{g_t}) : g_t \in \rG_t \}$, the preceding inclusion informs us that every $J \in J(O(\breve{\rG}))$ takes the form $\breve{\rG}[g_t]$. Then the composite bounded lattice isomorphism:
\[
\begin{tabular}{c}
$\latOp{\breve{\rG}}
\xto{\kappa_\rG^{\bf-1}} (\latCl{\rG})^{\pOp}
\xto{\theta_\rG^{\pOp}} (\latOp{\rG})^{\pOp}$
\end{tabular}
\]
necessarily restricts to a bijection $J(\latOp{\breve{\rG}}) \to M(\latOp{\rG})$, with action:
\[
\begin{tabular}{lll}
$\theta_\rG^{\pOp} \circ \kappa_\rG^{\bf-1}(\breve{\rG}[g_t])$
&
$= \theta_\rG(\kappa_\rG^{\bf-1}(\breve{\rG}[g_t]))$
&
\\&
$= \rG^\up \circ \neg_{\rG_s}(\breve{\rG}[g_t])$
&
\\&
$= \rG^\up \circ \rG^\down (\overline{g_t})$
& by De Morgan duality
\\&
$= \inte_\rG(\overline{g_t})$
\end{tabular}
\]
and we are finished.
\end{enumerate}
\end{proof}

We have associated two isomorphic finite bounded lattices $\latOp{\rG}$ and $\latCl{\rG}$ to each relation $\rG$. Next we describe two pairs of adjoint morphisms between the underlying join-semilattices of these bounded lattices, parametric in any $\Dep$-morphism.

\begin{lemma}
\label{lem:composite_adjoints}
Let $\rR : \rG \to \rH$ be any $\Dep$-morphism.
\begin{enumerate}
\item
For all subsets $X_1 \subseteq \rG_s$ and $\rH$-closed subsets $X_2 \subseteq \rH_s$:
\[
\rH^\down \circ \rR^\up(X_1) \subseteq X_2
\iff
X_1 \subseteq \rR^\down \circ \rH^\up(X_2).
\]
\item
Restricting the domain and codomain yields the adjoint $\JSL_f$-morphisms:
\[
\begin{tabular}{l}
$\rR^\down \circ \rH^\up : (C(\rH),\cap,\rH_s) \to (C(\rG),\cap,\rG_s)$
\\[1ex]
$\rH^\down \circ \rR^\up : (C(\rG),\lor_{\latCl{\rG}},\cl_\rG(\emptyset)) \to (C(\rH),\lor_{\latCl{\rH}},\cl_\rH(\emptyset))$
\end{tabular}
\]
\item
For all subsets $Y_1 \subseteq \rH_t$ and $\rG$-open subsets $Y_2 \subseteq  \rG_t$:
\[
\rR^\up \circ \rG^\down (Y_2) \subseteq Y_1
\iff
Y_2 \subseteq \rG^\up \circ \rR^\down (Y_1)
\]
\item
Restricting the domain and codomain yields the adjoint $\JSL_f$-morphisms:
\[
\begin{tabular}{l}
$\rR^\up \circ \rG^\down : (O(\rG),\cup,\emptyset) \to (O(\rH),\cup,\emptyset)$
\\[1ex]
$\rG^\up \circ \rR^\down : (O(\rH),\land_{\latOp{\rH}},\inte_\rH(\rH_t)) \to (O(\rG),\land_{\latOp{\rG}},\inte_\rG(\rG_t))$
\end{tabular}
\]
\end{enumerate}
\end{lemma}

\begin{proof}
\item
\begin{enumerate}

\item
Regarding the first statement:
\[
\begin{tabular}{lll}
$X_1 \subseteq \rR^\down \circ \rH^\up(X_2)$
&
$\iff \rR^\up(X_1) \subseteq \rH^\up(X_2)$
& by $(\up\dashv\down)$
\\&
$\iff \rH^\up \circ \rH^\down \circ \rR^\up(X_1) \subseteq \rH^\up(X_2)$
& by Lemma \ref{lem:bicliq_mor_char_max_witness}.1
\\&
$\iff \rH^\down \circ \rR^\up(X_1) \subseteq \rH^\down \circ \rH^\up(X_2)$
& by $(\up\dashv\down)$
\\&
$\iff \rH^\down \circ \rR^\up(X_1) \subseteq X_2$
& since $X_2$ is closed
\end{tabular}
\]

\item
As for the second statement, $\rR^\down \circ \rH^\up$ sends all subsets (in particular $\rH$-closed subsets) to $\rG$-closed subsets because (i) $\rR^\down = \cl_\rG \circ \rR^\down$ by Lemma \ref{lem:bicliq_mor_char_max_witness}.1 and Lemma \ref{lem:up_down_basic}.4, and (ii) every subset of the form $\rG^\down(S)$ is $\rG$-closed by Lemma \ref{lem:lat_op_cl}.3. Furthermore $\rH^\down \circ \rR^\up$ sends all subsets (in particular $\rG$-closed subsets) to $\rH$-closed subsets because $\rH^\down = \cl_\rH \circ \rH^\down$ by $(\down\up\down)$. Thus when restricted to closed subsets in both their domain and codomain they define an adjunction. Consequently, the right adjoint $\rR^\down \circ \rH^\up$ preserves all meets (= intersections) and the left adjoint preserves all joins (which needn't be unions).

\item
(3) and (4) follow from (1) and (2) by applying duality. Take the dual morphism $\rR\spcheck = \breve{\rR} : \breve{\rH} \to \breve{\rG}$ and apply the first statement, yielding:
\[
X_1 \subseteq \breve{\rR}^\down \circ \breve{\rG}^\up(X_2) 
\iff
\breve{\rG}^\down \circ \breve{\rR}^\up(X_1) \subseteq X_2
\]
for all subsets $X_1 \subseteq \breve{\rH}_s = \rH_t$ and all $\breve{\rG}$-closed subsets $X_2 \subseteq \breve{\rG}_s = \rG_t$. Since relative complement defines a lattice isomorphism $\kappa_{\breve{\rG}} : (\latCl{\breve{\rG}})^{\pOp} \to \latOp{\rG}$, we can substitute both $X_1 := \overline{Y_1}$ and $X_2 := \overline{Y_2}$ and rewrite so that:
\[
\neg_{\rH_t} \circ \breve{\rR}^\down \circ \breve{\rG}^\up \circ \neg_{\rG_t}(Y_2) \subseteq Y_1 
\iff
Y_2 \subseteq \neg_{\rG_t} \circ \breve{\rG}^\down \circ \breve{\rR}^\up \circ \neg_{\rH_t}(Y_1)
\]
for all subsets $Y_1 \subseteq \rH_t$ and all $\rG$-open subsets $Y_2 \subseteq \rG_t$. Applying De Morgan duality yields the desired equivalence. Regarding the fourth statement, this follows via the restriction of $\breve{\rR}$ via the second statement.
\end{enumerate}
\end{proof}
  
The De Morgan dualities in Lemma \ref{lem:up_down_basic}.3 extend to the closure and interior operators associated to $\rG$. Furthermore they satisfy the usual characterisations as intersections/unions of closed/open sets, and in particular $\cl_{\Pirr\aQ}$ is the `usual' closure structure associated to a finite lattice with underlying join-semilattice $\aQ$.

\begin{lemma}
\label{lem:cl_inte_of_pirr}
\item
\begin{enumerate}
\item
For each relation $\rG$:
\[
\begin{tabular}{llll}
$\inte_\rG = \neg_{\rG_t} \circ \cl_{\breve{\rG}} \circ \neg_{\rG_t}$
&\qquad&
$\inte_\rG(Z) = \bigcup \{ Y \in O(\rG) : Y \subseteq Z \}$
& for any subset $Z \subseteq \rG_t$.
\\
$\cl_\rG = \neg_{\rG_s} \circ \inte_{\breve{\rG}} \circ \neg_{\rG_s}$
&\qquad&
$\cl_\rG(Z) = \bigcap \{ X \in C(\rG) : Z \subseteq X \}$
& for any subset $Z \subseteq \rG_s$.
\end{tabular}
\]
Moreover,
\[
Y \subseteq \inte_\rG(\overline{g_t})
\iff
g_t \nin Y
\]
for every $\rG$-open $Y \subseteq \rG_t$ and every $g_t \in \rG_t$.
 

\item
For each finite join-semilattice $\aQ$:
\[
\begin{tabular}{ll}
$\cl_{\Pirr\aQ} (S) = \{ j \in J(\aQ) : j \leq_\aQ \Lor_\aQ S \}$
&
$\cl_{(\Pirr\aQ)\spbreve} (S) = \{ m \in M(\aQ) : \Land_\aQ S \leq_\aQ m \}$
\\
$\inte_{\Pirr\aQ} (S) = \{ m \in M(\aQ) : \Land_\aQ \overline{S} \nleq_\aQ m \}$
&
$\inte_{(\Pirr\aQ)\spbreve} (S) = \{ j \in J(\aQ) : j \nleq_\aQ \Lor_\aQ \overline{S} \}$
\end{tabular}
\]
recalling that $\Pirr\aQ := \; \nleq_\aQ \; \subseteq J(\aQ) \times M(\aQ)$.

\end{enumerate}
\end{lemma}

\begin{proof}
\item
\begin{enumerate}
\item
The two left-hand statements follow by applying De Morgan duality i.e.\ $(\neg\up\neg)$ and $(\neg\down\neg)$. Next:
\[
\begin{tabular}{lll}
$\bigcup \{ Y \in O(\rG) : Y \subseteq Z \}$
&
$= \bigcup \{ \rG[g_s] : g_s \in \rG_s,\; \rG[g_s] \subseteq Z \}$
& restrict to join-irreducibles
\\&
$= \{ g_t \in \rG_t : \exists g_s \in \rG_s. g_t \in \rG[g_s] \subseteq Z \}$
\\&
$= \rG^\up \circ \rG^\down(Z)$
\\&
$= \inte_\rG(Z)$
\end{tabular}
\]
The description of $\cl_\rG(Z)$ follows by (i) the De Morgan duality above, and (ii) because $\breve{\rG}$-open sets are the relative complements of $\rG$-closed sets by Lemma \ref{lem:lat_op_cl}.2. The final equivalence follows because $Y \subseteq \overline{g_t}$ iff $g_t \nin Y$, and moreover $\inte_\rG$ is co-extensive, monotone and idempotent.

\item
Regarding the second statement, for all subsets $S \subseteq J(\aQ)$ we have:
\[
\begin{tabular}{lll}
$\cl_{\Pirr\aQ} (S)$
& $= \; \nleq_\aQ^\down \circ \nleq_\aQ^\up (S)$
\\&
$= \; \nleq_\aQ^\down (\{ m \in M(\aQ) : \exists j' \in S. j' \nleq_\aQ m \}$
\\&
$= \{ j \in J(\aQ) : \;\nleq_\aQ[j] \subseteq \{ m \in M(\aQ) : \exists j' \in S. j' \nleq_\aQ m \} \}$
\\&
$= \{ j \in J(\aQ) : \forall m \in M(\aQ).( j \nleq_\aQ m \To \exists j' \in S. j' \nleq_\aQ m ) \}$
\\&
$= \{ j \in J(\aQ) : \forall m \in M(\aQ).( \neg\exists j' \in S. j' \nleq_\aQ m \To j \leq_\aQ m  ) \}$
\\&
$= \{ j \in J(\aQ) : \forall m \in M(\aQ).( (\forall j' \in S. j' \leq_\aQ m) \To j \leq_\aQ m ) \}$
\\&
$= \{ j \in J(\aQ) : \forall m \in M(\aQ).( \Lor_\aQ S \leq_\aQ m \To j \leq_\aQ m \}$
\\&
$= \{ j \in J(\aQ) : j \leq_\aQ \Lor_\aQ S \} $
\end{tabular}
\]
Next observe $(\Pirr\aQ)\spbreve = \Pirr\aQ^{\pOp}$, so that:
\[
\cl_{(\Pirr\aQ)\spbreve}(S) 
= \{ m \in J(\aQ^{\pOp}) : m \leq_{\aQ^{\pOp}} \Lor_{\aQ^{\pOp}} S \}
= \{ m \in M(\aQ) : \Land_\aQ S \leq_\aQ m \}
\]
The descriptions of the interior operators follow by the De Morgan duality exhibited in the first statement.

\end{enumerate}
\end{proof}

We now have enough structure to prove the well-definedness of $\Open$ and $\Pirr$.

\begin{lemma}
\label{lem:open_well_defined}
$\Open : \Dep \to \JSL_f$ is a well-defined faithful functor.
\end{lemma}

\begin{proof}
$\Open \rG = (O(\rG),\cup,\emptyset)$ is a finite join-semilattice. Regarding its action on morphisms, recall that $\Open\rR(Y) := \rR^\up \circ \rG^\down(Y)$ for every $\rG$-open $Y$. Then by Lemma \ref{lem:composite_adjoints}.4 this is a well-defined $\JSL_f$-morphism of the desired type. $\Open$ preserves identity morphisms because:
\[
\Open id_\rG 
= \Open (\rG : \rG \to \rG)
= \lambda Y.\rG^\up \circ \rG^\down(Y)
\stackrel{!}{=} \lambda Y.Y
= id_{\Open\rG}
\]
since $Y$ is $\rG$-open. Next we show preservation of composition i.e.\ for any $\rR : \rG \to \rH$ and $\rS : \rH \to \rI$:
\[
\Open \rS \circ \Open \rR
= \Open(\rR \fatsemi \rS).
\]
We have not reversed the sense of the morphisms, the difference in ordering is due to the different order in which one writes functional and $\Dep$-composition, the latter in keeping with relational composition. It follows via:
\[
\begin{tabular}{lll}
$\Open(\rR\fatsemi\rS)(Y)$
&
$= (\rR\fatsemi\rS)^\up \circ \rG^\down(Y)$
& 
\\&
$= \rS^\up \circ \rH^\down \circ \rR^\up \circ \rG^\down(Y)$
& by $(\up\fatsemi)$
\\&
$= \rS^\up \circ \rH^\down(\Open \rR(Y))$
\\&
$= \Open\rS(\Open \rR(Y))$
\end{tabular}
\]
We now establish that $\Open$ is faithful. Given $\rR,\, \rS : \rG \to \rH$ such that $\Open\rR = \Open\rS$ then for all $\rG$-open sets $Y$ we have $\rR_+\spbreve[Y] = \rS_+\spbreve[Y]$. Then for each $g_s \in \rG_s$ we have:
\[
\begin{tabular}{lll}
$\rR[g_s]$
&
$= (\rG ; \rR_+\spbreve)[g_s]$
& by Lemma \ref{lem:bicliq_mor_char_max_witness}.2 
\\&
$= \rR_+\spbreve[\rG[g_s]]$
\\&
$= \rS_+\spbreve[\rG[g_s]]$
& since $\rG[g_s]$ is $\rG$-open
\\&
$= (\rG;\rS_+\spbreve)[g_s]$
& 
\\&
$= \rS[g_s]$
& by Lemma \ref{lem:bicliq_mor_char_max_witness}.2 
\end{tabular}
\]
so that $\rR = \rS$ as required.
\end{proof}

\begin{lemma}
\label{lem:pirr_well_defined}
$\Pirr : \JSL_f \to \Dep$ is a well-defined functor.
\end{lemma}

\begin{proof}
$\Pirr$ is clearly well-defined on objects. Take any join-semilattice morphism $f : \aQ \to \aR$ and recall that:
\[
\Pirr f := \{ (j,m) \in J(\aQ) \times M(\aR) : f(j) \nleq_\aR m \}.
\]
We must show that it is a $\Dep$-morphism of type $\Pirr\aQ \to \Pirr\aR$. By Lemma \ref{lem:cl_inte_of_pirr}.2 we know that:
\[
\cl_{\Pirr\aQ} (S) = \{ j \in J(\aQ) : j \leq_\aQ \Lor_\aQ S \}
\qquad\text{and also}\qquad
\inte_{\Pirr\aR} (S) = \{ m \in M(\aR) : \Land_\aR \overline{S} \nleq_\aR m \}.
\]
Then we calculate:
\[
\begin{tabular}{lll}
$(\Pirr f)^\up \circ \cl_{\Pirr\aQ}(S)$
&
$= (\Pirr f)^\up (\{ j \in J(\aQ) : j \leq_\aQ \Lor_\aQ S \})$
& see above
\\&
$= \{ m \in M(\aR) : \exists j \in J(\aQ).[ j \leq_\aQ \Lor_\aQ S \text{ and } f(j) \nleq_\aR m  ] \}$
\\&
$= \{ m  : \neg\forall j.[ j \leq_\aQ \Lor_\aQ S \To f(j) \leq_\aR m  ] \} $
\\&
$= \{ m : \neg\forall j.[ j \leq_\aQ \Lor_\aQ S \To j \leq_\aR f_*(m)  ] \}$
& take adjoint
\\&
$= \{ m : \Lor_\aQ S \nleq_\aR f_*(m) \}$
\\&
$= \{ m : f(\Lor_\aQ S) \nleq_\aR m \}$
& take adjoint
\\&
$= \{ m : \Lor_\aQ f[S] \nleq_\aR m \}$
& $f$ preseves joins
\\&
$= \{ m  : \exists j \in S. f(j) \nleq_\aR m \}$
\\&
$= (\Pirr f)^\up(S)$
\\
\\
$\inte_{\Pirr\aR} \circ (\Pirr f)^\up (S)$
&
$= \inte_{\Pirr\aR}(\{ m_1 \in M(\aR) : \exists j \in S. f(j) \nleq_\aR m_1 \})$
\\&
$= \{ m_2 \in M(\aR) : \Land_\aR \{ m_1 \in M(\aR) : \neg\exists j \in S. f(j) \nleq_\aR m_1 \} \nleq_\aR m_2 \}$
\\&
$= \{ m_2 \in M(\aR) : \Land_\aR \{ m_1 \in M(\aR) : \forall j \in S. f(j) \leq_\aR m_1 \} \nleq_\aR m_2  \}$
\\&
$= \{ m_2 \in M(\aR) : \Land_\aR \{ m_1 \in M(\aR) : \Lor_\aR f[S] \leq_\aR m_1 \} \nleq_\aR m_2  \}$
\\&
$= \{ m_2 \in M(\aR) : \Lor_\aR f[S]  \nleq_\aR m_2  \}$
\\&
$= \{ m_2 \in M(\aR) : \exists j \in S. f(j)  \nleq_\aR m_2  \}$
\\&
$= (\Pirr f)^\up(S)$
\end{tabular}
\]
Then $\Pirr f : \Pirr \aQ \to  \Pirr\aR$ is a well-defined $\Dep$-morphism. Concerning preservation of identity morphisms:
\[
\Pirr id_\aQ 
= \{ (j,m) \in J(\aQ) \times M(\aQ) : j \nleq_\aQ m \} 
= \Pirr\aQ : \Pirr\aQ \to \Pirr\aQ
= id_{\Pirr\aQ}.
\]
Next, given any join-semilattice morphisms $f : \aQ \to \aR$ and $g : \aR \to \aS$ we must show that $\Pirr( g \circ f) = \Pirr f \fatsemi \Pirr g$. First we verify the definitions of $\Pirr f$'s component relations.
\[
\begin{tabular}{lll}
$(\Pirr f)_- (j_1,j_2)$
&
$\iff j_2 \in (\Pirr\aR)^\down(\Pirr f[j_1])$
\\&
$\iff \Pirr\aR[j_2] \subseteq \{ m_1 \in M(\aR) : f(j_1) \nleq_\aR m_1 \}$
\\&
$\iff \{ m \in M(\aR) : f(j_1) \leq_\aR m \} \subseteq \{ m \in M(\aR) : j_2 \leq_\aR m \}$
\\&
$\iff \forall m \in M(\aR).( f(j_1) \leq_\aR m \To j_2 \leq_\aR m  )$
\\&
$\iff j_2 \leq_\aR f(j_1)$
\\
\\
$(\Pirr f)_+ (m_1,m_2)$
&
$\iff  m_2 \in (\Pirr\aQ\spbreve)^\down((\Pirr f)\spbreve[m_1]) $
\\&
$\iff \Pirr\aQ\spbreve[m_2] \subseteq (\Pirr f)\spbreve[m_1]$
\\&
$\iff \{ j \in J(\aQ) : j \nleq_\aQ m_2 \} \subseteq \{ j \in J(\aQ) : f(j) \nleq_\aR m_1 \}  $
\\&
$\iff \{ j \in J(\aQ) : f(j) \leq_\aR m_1 \} \subseteq \{ j \in J(\aQ) : j \leq_\aQ m_2 \}   $
\\&
$\iff \forall j \in J(\aQ).( f(j) \leq_\aR m_1 \To j \leq_\aQ m_2)$
\\&
$\iff \forall j \in J(\aQ).( j \leq_\aQ f_*(m_1) \To j \leq_\aQ m_2 )$
\\&
$\iff f_*(m_1) \leq_\aQ m_2 $
\end{tabular}
\]
Then finally:
\[
\begin{tabular}{lll}
$\Pirr f \fatsemi \Pirr g$
&
$= \Pirr f ; (\Pirr g)_+\spbreve$
\\&
$= \{ (j_q,m_s) \in J(\aQ) \times M(\aS) : \exists m_r \in M(\aR).( f(j_q) \nleq_\aR m_r \,\land\, g_*(m_s) \leq_\aR m_r  ) \}$
\\&
$= \{ (j_q,m_s) \in J(\aQ) \times M(\aS) : \neg\forall m_r \in M(\aR).( g_*(m_s) \leq_\aR m_r \To  f(j_q) \leq_\aR m_r ) \}$
\\&
$= \{ (j_q,m_s) \in J(\aQ) \times M(\aS) : f(j_q) \nleq_\aR g_*(m_s) \}$
\\&
$= \{ (j_q,m_s) \in J(\aQ) \times M(\aS) : g \circ f(g_q) \nleq_\aS m_s \}$
\\&
$= \Pirr (g \circ f)$
\end{tabular}
\]
\end{proof}


We are now finally ready to prove the categorical equivalence of $\JSL_f$ and $\Dep$. We do this very explicitly i.e.\ the two natural isomorphisms and their inverses are provided, as well as the associated component relations of the $\Dep$-isomorphisms.

\bigskip

\begin{theorem}[$\Dep$ is equivalent to $\JSL_f$]
\label{thm:bicliq_jirr_equivalent}
The functors $\Open : \Dep \to \JSL_f$ and $\Pirr : \JSL_f \to \Dep$ define an equivalence of categories, with respective natural isomorphisms:
\[
\begin{tabular}{llll}
$rep : \Id_{\JSL_f} \To \Open \circ \Pirr$
&&
$rep_\aQ$ & $:= \lambda q \in Q.\{ m \in M(\aQ) : q \nleq_\aQ m \}$
\\
&& $rep_\aQ^{\bf-1}$ & $:= \lambda Y. \Land_\aQ M(\aQ) \setminus Y$
\\[1ex]
$red : \Id_{\Dep} \To \Pirr \circ \Open$
&&
$red_\rG$ & $:= \{ (g_s,Y) \in \rG_s \times M(\Open\rG) : \rG[g_s] \nsubseteq Y \}$ 
\\
&&
$red_\rG^{\bf-1}$ & $:= \breve{\in} \; \subseteq J(\Open\rG) \times \rG_t$
\end{tabular}
\]
where $\red_\rG$ and its inverse have associated component relations:
\[
\begin{tabular}{ll}
$(red_\rG)_- := \{ (g_s,X) \in \rG_s \times J(\Open\rG) : X \subseteq \rG[g_s] \}$
&
$(red_\rG)_+ := \breve{\nin} \; \subseteq  M(\Open\rG) \times \rG_t$
\\
$(red_\rG^{\bf-1})_- := \{ (X,g_s) \in J(\Open\rG)\times\rG_s : \rG[g_s] \subseteq X \}$
&
$(red_\rG^{\bf-1})_+ := \{ (g_t,Y) \in \rG_t \times M(\Open\rG) : \inte_\rG(\overline{g_t}) \subseteq Y  \}$
\end{tabular}
\]
\end{theorem}

\begin{proof}
\item
\begin{enumerate}
\item

We verify that $rep$ is a natural isomorphism. Each $rep_\aQ$ is a well-defined function because:
\[
\begin{tabular}{lll}
$\inte_{\Pirr\aQ}(rep_\aQ(q))$
&
$= \inte_{\Pirr\aQ} (\{ m \in M(\aQ) : q \nleq_\aQ m \})$
\\&
$= \{ m' \in M(\aQ) : \Land_\aQ \{ m : q \leq_\aQ m \} \nleq_\aQ m' \}$
& by Lemma \ref{lem:cl_inte_of_pirr}
\\&
$= \{ m' \in M(\aQ) : q \nleq_\aQ m' \}$
\\&
$= rep_\aQ(q)$
\end{tabular}
\]
It is well-known that $rep_\aQ$ defines a $\JSL_f$-morphism, usually described as an embedding into $(\Pow M(\aQ),\cup,\emptyset)$. We verify this explicitly:  $rep_\aQ(\bot_\aQ)= \{ m \in M(\aQ) : \bot_\aQ \nleq_\aQ m \} = \emptyset = \bot_{\Open(\Pirr \aQ)}$, and:
\[
rep_\aQ(q_1 \lor_\aQ q_2) 
= \{ m \in M(\aQ) : q_1 \lor_\aQ q_2 \nleq_\aQ m \}
= \{ m \in M(\aQ) : q_1 \nleq_\aQ m \text{ or } q_2 \nleq_\aQ m \}
= rep_\aQ(q_1) \cup rep_\aQ(q_2).
\]
Next we show that $rep_\aQ$ is bijective. It is injective because distinct elements $q_1 \neq_\aQ q_2$ necessarily have distinct sets of meet irreducibles above them, seeing as they are the respective meet of them, so their complements relative to $M(\aQ)$ are also distinct. For surjectivity, we first observe that for any subset $Y \subseteq M(\aQ)$ we have:
\[
\begin{tabular}{lll}
$\Land_\aQ \overline{Y}$
&
$= \Lor_\aQ \{q \in Q : \forall m \in \overline{Y}.q \leq_\aQ m \}$
& meet in terms of join
\\&
$= \Lor_\aQ \{ j \in J(\aQ) : \forall m \in \overline{Y}. j \leq_\aQ m \}$
& restrict to join-irreducibles
\\&
$= \Lor_\aQ \{ j \in J(\aQ) : \forall m \in M(\aQ) (m \in \overline{Y} \To j \leq_\aQ m) \}$
& recall $\overline{Y} = M(\aQ)\backslash Y$
\\&
$= \Lor_\aQ \{ j \in J(\aQ) : \forall m \in M(\aQ) (j \nleq_\aQ m \To m \in Y) \}$
\\&
$= \Lor_\aQ (\Pirr\aQ)^\down(Y)$.
\end{tabular}
\]
Then for any $\Pirr\aQ$-open $Y \subseteq M(\aQ)$ we now show that $rep_\aQ(\Land_\aQ M(\aQ)\backslash Y) = Y$.
\[
\begin{tabular}{lll}
$rep_\aQ(\Land_\aQ \overline{Y})$
&
$= rep_\aQ(\Lor_\aQ \nleq_\aQ^\down(Y))$
& see above
\\&
$= \{ m \in M(\aQ) : \Lor_\aQ \nleq_\aQ^\down(Y) \nleq_\aQ m \}$
\\&
$= \{ m \in M(\aQ) : \exists j \in \; \nleq_\aQ^\down(Y). j \nleq_\aQ m \}$
\\&
$= \{ m \in M(\aQ) : \exists j \in J(\aQ).(\nleq_\aQ[j] \subseteq Y \text{ and } j \nleq_\aQ m  ) \} $
\\&
$= \; \nleq_\aQ^\up \circ \nleq_\aQ^\down(Y)$
\\&
$= Y$.
& since $Y$ is $\Pirr\aQ$-open
\end{tabular}
\]
Thus we have shown that each $rep_\aQ : \aQ \to \Open(\Pirr \aQ)$ is a $\JSL_f$-isomorphism, and furthermore the inverse is necessarily $rep_\aQ^{\bf-1}(Y) = \Land_\aQ M(\aQ)\backslash Y$ by the above argument. Then it only remains to prove naturality i.e.\
\[
\xymatrix@=15pt{
\aQ \ar[rr]^-{rep_\aQ} \ar[d]_f && (O(\nleq_\aQ |_{J(\aQ) \times M(\aQ)} ),\cup,\emptyset) \ar[d]^{\Open(\Pirr f)}
\\
\aR \ar[rr]_-{rep_\aR} && (O(\nleq_\aR |_{J(\aR) \times M(\aR)} ),\cup,\emptyset)
}
\]
for all $\JSL_f$-morphisms $f : \aQ \to \aR$. Unwinding the definitions,  $\rep_\aR \circ f(q) = \{ m \in M(\aR) : f(q) \nleq_\aR m \}$ and:
\[
\begin{tabular}{lll}
$\Open(\Pirr f) \circ \rep_\aQ(q)$
&
$= \Open(\Pirr f)(\{m_1 \in M(\aQ) : q \nleq_\aQ m_1\})$
\\&
$= (\Pirr f)_+[\{m_1 \in M(\aQ) : q \nleq_\aQ m_1\}]$
\\&
$= \{ m \in M(\aR) : \exists m_1 \in M(\aQ).(q \nleq_\aQ m_1 \text{ and } f_*(m) \leq_\aQ m_1) \}$
\\&
$= \{ m : \neg\forall m_1 \in M(\aQ).(f_*(m) \leq_\aQ m_1 \To q \leq_\aQ m_1 ) \}$
\\&
$= \{ m : q \nleq_\aQ f_*(m) \}$
\\&
$= \{ m : f(q) \nleq_\aR m \}$
& via adjoints
\\&
$= rep_\aQ(q)$
\end{tabular}
\]

\item
We verify that $red$ is a natural isomorphism. Let $\aQ := \Open\rG$. We start by showing that each $red_\rG = \{ (g_s,Y) \in \rG_s \times M(\aQ) : \rG[g_s] \nsubseteq Y \}$ is a well-defined $\Dep$-morphism of type $\rG \to \Pirr\aQ$.
\[
\begin{tabular}{lll}
$red_\rG^\up \circ \cl_\rG(S)$
&
$= \{ Y \in M(\aQ) : \exists g_s \in \cl_\rG(S). \rG[g_s] \nsubseteq Y \}$
\\&
$= \{ Y \in M(\aQ) : \exists g_s \in \cl_\rG(S). g_s \nin \rG^\down(Y) \}$
& by definition of $\rG^\down$
\\&
$= \{ Y \in M(\aQ) : \cl_\rG(S) \nsubseteq \rG^\down(Y) \}$
\\&
$= \{ Y \in M(\aQ) : \rG^\up \circ \rG^\down \circ \rG^\up(S) \nsubseteq Y \}$
& by $(\up\dashv\down)$
\\&
$= \{ Y \in M(\aQ) : \rG^\up(S) \nsubseteq Y \}$
& by $(\up\down\up)$
\\&
$= \{ Y \in M(\aQ) : \exists s \in S.\rG[s] \nsubseteq Y \}$
\\&
$= red_\rG^\up(S)$
\end{tabular}
\]
\[
\begin{tabular}{lll}
$\inte_{\Pirr\aQ} \circ red_\rG^\up(S)$
&
$= \inte_{\Pirr\aQ}(\{ Y \in M(\aQ) : \exists s \in S. \rG[s] \nsubseteq Y \})$
\\&
$= \inte_{\Pirr\aQ}(\{ Y \in M(\aQ) : \rG[S] \nsubseteq Y \})$
\\&
$= \{ Y' \in M(\aQ) : \Land_\aQ \{ Y \in M(\aQ) : \rG[S] \leq_\aQ Y \} \nsubseteq Y'  \}$
& by Lemma \ref{lem:cl_inte_of_pirr}
\\&
$= \{ Y' \in M(\aQ) : \rG[S] \nsubseteq Y' \}$
& since $\rG[S] \in Q$
\\&
$= \red_\rG^\up(S)$
\end{tabular}
\]
To show that $red_\rG$ is an isomorphism, we first show that its proposed inverse $\red_\rG^{\bf-1} := \breve{\in} \subseteq J(\aQ) \times \rG_t$ is a well-defined $\Dep$-morphism of type $\Pirr\aQ \to \rG$.
\[
\begin{tabular}{lll}
$\breve{\in}^\up \circ \cl_{\Pirr\aQ}(S)$
&
$= \{ g_t \in \rG_t : \exists Y \in \cl_{\Pirr\aQ}(S). g_t \in Y \}$
\\&
$= \{ g_t \in \rG_t : \exists Y \in J(\aQ).[ Y \leq_\aQ \Lor_\aQ S \text{ and } g_t \in Y   ] \}$
& by Lemma \ref{lem:cl_inte_of_pirr}
\\&
$= \{ g_t \in \rG_t : \exists Y \in J(\aQ). g_t \in Y \subseteq \bigcup S \}$
\\&
$= \{ g_t \in \rG_t : g_t \in \bigcup S  \}$
& using Lemma \ref{lem:lat_op_cl}.3
\\&
$= \breve{\in}^\up(S)$
\end{tabular}
\]
Furthermore $\inte_\rG \circ \breve{\in}^\up(S) = \inte_\rG(\bigcup S) = \bigcup S = \breve{\in}^\up(S)$ because $S \subseteq J(\aQ)$ is a collection of $\rG$-open sets. Now we show they that these two morphisms are the inverse of one another.
\[
\begin{tabular}{lll}
$(\red_\rG \fatsemi \red_\rG^{\bf-1})^\up(S)$
&
$= (\red_\rG^{\bf-1})^\up \circ (\Pirr\aQ)^\down \circ (\red_\rG)^\up(S)$
& by $(\up\fatsemi)$
\\&
$= \breve{\in}^\up(\{ X \in J(\aQ) :\; \nleq_\aQ[X] \subseteq \{ Y \in M(\aQ) : \rG[S] \nsubseteq Y \})$
\\&
$= \breve{\in}\,[\{ X : \forall Y \in M(\aQ).( \rG[S] \subseteq Y \To X \subseteq Y  ) \}]$
\\&
$= \breve{\in}\,[\{ X : X \subseteq \rG[S] \}]$
\\&
$= \bigcup \{ X  : X \subseteq \rG[S] \}$
\\&
$= \Lor_{\aQ} \{ X \in J(\aQ) : X \leq_\aQ \rG[S] \}$
\\&
$= \rG[S]$
\\&
$= id_\rG^\up(S)$
\end{tabular}
\]
\[
\begin{tabular}{lll}
$(red_\rG^{\bf-1} \fatsemi red_\rG)^\up(S)$
&
$= red_\rG^\up \circ \rG^\down \circ (red_\rG^{\bf-1})^\up(S)$
& by $(\up\fatsemi)$
\\&
$= red_\rG[\rG^\down(\bigcup S)]$
\\&
$= \{ Y \in M(\aQ) : \exists g_s \in \rG^\down(\bigcup S). \rG[g_s] \nsubseteq Y \}$
\\&
$= \{ Y  : \rG^\down(\bigcup S) \nsubseteq \rG^\down(Y) \}$
\\&
$= \{ Y  : \rG^\up \circ \rG^\down(\bigcup S) \nsubseteq Y \}$
& by $(\up\dashv\down)$
\\&
$= \{ Y  : \bigcup S \nsubseteq Y \}$
& since $\bigcup S$ is $\rG$-open
\\&
$= \{ Y  : \exists X \in S. X \nsubseteq Y \}$
\\&
$= \{ Y  : \exists X \in S. X \nleq_\aQ Y \}$
\\&
$= \;\nleq_\aQ[S]$
\\&
$= (id_{\Pirr\aQ})^\up(S)$
\end{tabular}
\]

Thus each $\rep_\rG$ is indeed a $\Dep$-isomorphism with inverse $rep_\rG^{\bf-1}$. Let us also verify one of $rep_\rG$'s associated components:
\[
\begin{tabular}{lll}
$(red_\rG)_+$
&
$= \{ (Y,g_t) \in M(\Open\rG) \times \rG_t : g_t \in \breve{\rG}^\down(\red_\rG\spbreve[Y]) \}$
& by definition
\\&
$= \{ (Y,g_t) : \breve{\rG}[g_t] \subseteq \{ g_s \in \rG_s : \rG[g_s] \nsubseteq Y \} \}$
\\&
$= \{ (Y,g_t) : \forall g_s \in \rG_s.( g_t \in \rG[g_s] \To \rG[g_s] \nsubseteq Y  )\}$
\\&
$= \{ (Y,g_t) : \forall g_s \in \rG_s.( \rG[g_s] \subseteq Y \To g_t \nin \rG[g_s]   \}$
\\&
$= \{ (Y,g_t) : \forall g_s \in \rG_s.( \rG[g_s] \subseteq Y \To \rG[g_s] \subseteq \inte_\rG(\overline{g_t}) \}$
& by Lemma \ref{lem:cl_inte_of_pirr}.1
\\&
$= \{ (Y,g_t) :  Y \subseteq \inte_\rG(\overline{g_t}) \}$
& by Lemma \ref{lem:lat_op_cl}.3
\\&
$= \{ (Y,g_t) : g_t \notin Y  \}$
& by Lemma \ref{lem:cl_inte_of_pirr}.1
\\&
$= \breve{\notin} \; \subseteq M(\Open\rG) \times \rG_t$
\end{tabular}
\]
It remains to verify naturality i.e.\ the diagram below on the left commutes for all $\Dep$-morphisms $\rR : \rG \to \rH$.
\[
\xymatrix@=15pt{
\rG \ar[rr]^-{red_\rG} \ar[d]_{\rR} && \Pirr(\Open\rG) \ar[d]^{\Pirr \circ \Open\rR}
\\
\rH \ar[rr]_-{red_\rH} && \Pirr(\Open\rH) 
}
\qquad
\xymatrix@=15pt{
\Open\rG \ar[rr]^-{\Open red_\rG} \ar[d]_{\Open\rR} && \Pirr(\Open\rG) \ar[d]^{\Open(\Pirr \circ \Open\rR)}
\\
\Open\rH \ar[rr]_-{\Open red_\rH} && \Open(\Pirr(\Open\rH))
}
\]
Applying $\Open$ to this diagram yields the diagram above on the right, and since $\Open$ is faithful it suffices to show the latter commutes. In fact it is an instance of $\rep$'s naturality because $\Open red_\rG = \rep_{\Open \rG}$ for any relation $\rG$, as we now show.
\[
\begin{tabular}{lll}
$\Open red_\rG$
&
$= \lambda Y \in O(\rG).(red_\rG)_+\spbreve[Y]$
\\&
$= \lambda Y. (\breve{\nin})\spbreve[Y]$
& by above calculation
\\&
$= \lambda Y. \nin[Y]$
\\&
$= \lambda Y. \{ M \in M(\Open\rG) : \exists g_t \in Y. g_t \nin M  \}$
\\&
$= \lambda Y. \{ M  : Y \nsubseteq M \}$
\\&
$= \lambda Y. \{ M  : Y \nleq_{\Open\rG} M \}$
\\&
$= \rep_{\Open\rG}$
\end{tabular}
\]

\item
Having proved the main result, we finally mechanically verify the associated component relations of $red_\rG$ and its inverse $red_\rG^{\bf-1}$, starting with $(red_\rG)_- \subseteq \rG_s \times J(\aQ)$.
\[
\begin{tabular}{lll}
$(red_\rG)_-(g_s,X)$
&
$\iff \rG_s \times J(\aQ) : X \in\; \nleq_\aQ^\down(red_\rG[g_s])$
\\&
$\iff \;\nleq_\aQ[X] \subseteq \{ Y \in M(\aQ) : \rG[g_s] \nsubseteq Y \}$
\\&
$\iff \forall Y \in M(\aQ).(X \nsubseteq Y \To \rG[g_s] \nsubseteq Y)$
\\&
$\iff \forall Y \in M(\aQ).(\rG[g_s] \subseteq Y \To X \subseteq Y)$
\\&
$\iff X \subseteq \rG[g_s]$
\end{tabular}
\]
In (2) we established that $(red_\rG)_+ = \breve{\nin} \; \subseteq M(\aQ) \times \rG_t$, so next consider $(red_\rG^{\bf-1})_- \subseteq J(\aQ) \times \rG_s$:
\[
\begin{tabular}{lll}
$(red_\rG^{\bf-1})_- (X,g_s)$
&
$\iff g_s \in \rG^\down(red_\rG^{\bf-1}[X])$
\\&
$\iff \rG[g_s] \subseteq \breve{\in}[X]$
\\&
$\iff \rG[g_s] \subseteq X $
\end{tabular}
\]
and finally $(red_\rG^{\bf-1})_+ \subseteq \rG_t \times M(\aQ)$:
\[
\begin{tabular}{lll}
$(red_\rG^{\bf-1})_+(g_t,Y)$
&
$\iff Y \in (\nleq_\aQ\spbreve)^\down((red_\rG^{\bf-1})\spbreve[g_t])$
\\&
$\iff \;\nleq_\aQ\spbreve[Y] \subseteq \{ X \in J(\aQ) : g_t \in X \}$
\\&
$\iff \forall X \in J(\aQ).(X \nsubseteq Y \To g_t \in X )$
\\&
$\iff \forall X \in J(\aQ).(g_t \nin X \To X \subseteq Y)$
\\&
$\iff \forall X \in J(\aQ).(X \subseteq \inte_\rG(\overline{g_t}) \To X \subseteq Y)$
& by Lemma \ref{lem:cl_inte_of_pirr}.1
\\&
$\iff \inte_\rG(\overline{g_t}) \subseteq Y$
\end{tabular}
\]

\end{enumerate}
\end{proof}

Then we have proved the claimed equivalence. It will be helpful to further clarify the fullness of $\Open$.

\begin{lemma}[Explicit fullness of $\Open$]
\label{lem:open_explicit_fullness}
\item
Given any $\JSL_f$-morphism $f : \Open\rG \to \Open\rH$ then $f = \Open \rR$ where the $\BiCliq$-morphism:
\[
\rR : \rG \to \rH
\qquad\text{is defined}\qquad
\rR(g_s,h_t) :\iff h_t \in f(\rG[g_s]).
\]
\end{lemma}

\begin{proof}
Consider:
\[
\xymatrix@=15pt{
\rG_t \ar[rr]^-{\rS_r} && M(\Open\rG) \ar[rr]^{(\Pirr f)_+\spbreve} && M(\Open\rH) \ar[rr]^-{\rT_r} && \rH_t
\\
\rG_s \ar[u]^\rG \ar[rr]_-{\rS_l} && J(\Open\rG) \ar[urr]^{\Pirr f} \ar[u]^{\Pirr\Open\rG} \ar[rr]_{(\Pirr f)_-} && J(\Open\rH) \ar[u]_{\Pirr\Open\rH} \ar[rr]_-{\rT_l} && \rH_s \ar[u]_\rH
\\
}
\]
where the respective relations are defined:
\[
\begin{tabular}{llllll}
$\rS_l(g_s,X)$ & $ :\iff X \subseteq \rG[g_s]$
&&
$\rT_l(X,h_s)$ & $ :\iff \rH[h_s] \subseteq X$
\\[1ex]
$\rS_r(Y,g_t)$ & $ :\iff g_t \nin Y$
&&
$\rT_r(Y,h_t)$ & $ :\iff \inte_\rH(\overline{h_t}) \subseteq Y$.
\end{tabular}
\]
The left and right squares commute because:
\[
\begin{tabular}{c}
$\rS_l ; \Pirr\Open\rG = \{ (g_s,Y) \in \rG_s \times M(\Open\rG) : \rG[g_s] \nsubseteq Y \} = \rG ; \rS_r$
\\[1ex]
$\rT_l ; \rH = \{ (X,h_t) \in J(\Open\rH) \times \rH_t : h_t \in X \} = \Pirr\Open\rH ; \rT_r$,
\end{tabular}
\]
as the reader may verify. Composing together these three $\BiCliq$-morphisms yields $\rR := \rS_l ; \Pirr f; \rT_r : \rG \to \rH$, where: 
\[
\begin{tabular}{lll}
$\rR(g_s,h_t)$
\\ \quad
$\iff \exists X \in J(\Open\rG), Y \in M(\Open\rH).( X \subseteq \rG[g_s] \text{ and } \Pirr f(X,Y) \text{ and } \inte_\rH(\overline{h_t}) \subseteq Y )$
\\ \quad
$\iff \exists X, Y.( X \subseteq \rG[g_s] \text{ and } f(X) \nsubseteq Y \text{ and } \inte_\rH(\overline{h_t}) \subseteq Y)$
\\ \quad
$\iff f(\rG[g_s]) \nsubseteq \inte_\rH(\overline{h_t})$
& see below
\\ \quad
$\iff h_t \in f(\rG[g_s])$
& by Lemma \ref{lem:cl_inte_of_pirr}.1.
\\ \quad
\end{tabular}
\]
Regarding the marked equivalence, $\To$ holds because $f$ is monotonic and hence preserves inclusions, and the converse follows because $\rG[x]$ is a union of join-irreducibles and $\inte_\rH(\overline{h_t})$ is a meet of meet-irreducibles. Then $\rR$ is a well-defined $\BiCliq$-morphism and:
\[
\begin{tabular}{lll}
$\Open\rR(\rG[g_s])$
&
$= \rR^\up \circ \rG^\down(\rG[g_s])$
\\&
$= \rR[g_s]$
& since $\rR^\up = \rR^\up \circ \cl_\rG$
\\&
$= f(\rG[g_s])$.
\end{tabular}
\]
Thus $f = \Open \rR$ because their action on join-irreducibles is the same.
\end{proof}


\smallskip
We finish off this subsection by using the above equivalence theorem to characterise all morphisms between finite boolean and distributive join-semilattices.
\smallskip

\begin{theorem}[Characterisation of $\JSL_f$-morphisms between boolean and distributive join-semilattices]
\label{thm:char_jsl_mor_bool_distr}
\item
\begin{enumerate}
\item
Each finite boolean join-semilattice $\aQ$ is isomorphic to $\JPow Z = \Open\Delta_Z$ for some finite set $Z$.
\item
The $\JSL_f$-morphisms $\JPow Z_1 \to \JPow Z_2$ are precisely the functions $\rR^\up$ where  $\rR \subseteq Z_1 \times Z_2$ is an arbitrary relation.
\item
Each finite distributive join-semilattice $\aQ$ is iso to $(Up(\pP),\cup,\emptyset) = \Open \leq_\pP$ for some finite poset $\pP = (P,\leq_\pP)$.
\item
The $\BiCliq$-morphisms $\rR : \,\leq_\pP\, \to \,\leq_\pQ$ are precisely those relations $\rR \subseteq P \times Q$ such that:
\[
\begin{tabular}{ccc}
$\forall p \in P. \; \rR[p] \in Up(\pQ)$
&and&
$\forall q \in Q.\; \breve{\rR}[q] \in Dn(\pP)$.
\end{tabular}
\]
Moreover, the $\JSL_f$-morphisms $\Open\leq_\pP \to \Open\leq_\pQ$ are precisely the functions $\rR^\up |_{Up(\pP) \times Up(\pQ)}$ where $\rR \subseteq P \times Q$ satisfies the above two conditions.

\end{enumerate}
\end{theorem}

\begin{proof}
\item
\begin{enumerate}
\item
Recall that a finite join-semilattice $\aQ$ is said to be \emph{boolean} if its associated bounded lattice is. Then by Lemma \ref{lem:std_order_theory}.13 and the fact that $(J(\aQ),M(\aQ)) = (At(\aQ),CoAt(\aQ))$ by Lemma \ref{lem:std_order_theory}.8, we have:
\[
\Pirr\aQ
= \; \nleq_\aQ |_{At(\aQ) \times CoAt(\aQ)} 
\stackrel{!}{=} (\leq_{\aQ^{\pOp}} |_{At(\aQ) \times At(\aQ)}) ; \tau_\aQ
\]
where $\tau_\aQ : At(\aQ) \to CoAt(\aQ)$ is the canonical bijection. Since atoms are incomparable we see that $\Pirr\aQ \subseteq At(\aQ) \times CoAt(\aQ)$ is a functional composite of bijections and hence a bijection itself. It follows that every $X \subseteq CoAt(\aQ)$ is $\Pirr\aQ$-open, and we may use the canonical $\JSL_f$-isomorphism:
\[
\aQ \xto{rep_\aQ} \Open\Pirr\aQ = \JPow CoAt(\aQ) = \Open \Delta_{CoAt(\aQ)}.
\]

\item
Every relation $\rR \subseteq Z_1 \times Z_2$ between finite sets defines a $\BiCliq$-morphism of type $\rR : \Delta_{Z_1} \to \Delta_{Z_2}$. Then since $\Open\Delta_Z = \JPow Z$ and by the equivalence theorem, the $\JSL_f$-morphisms of type $\JPow Z_1 \to \JPow Z_2$ are precisely the functions:
\[
\Open\rR
= \lambda S \in \Pow Z_1. \rR^\up \circ \Delta_{Z_1}^\down(S)
= \lambda S \subseteq Z_1. \rR[S]
= \rR^\up,
\]
where $\rR \subseteq Z_1 \times Z_2$ is arbitrary.

\item
A finite join-semilattice $\aQ$ is \emph{distributive} if its associated lattice is. By Lemma \ref{lem:std_order_theory}.13,
\[
\Pirr\aQ
= \; \nleq_\aQ |_{J(\aQ) \times M(\aQ)} 
\stackrel{!}{=} (\leq_{\aQ^{\pOp}} |_{J(\aQ) \times J(\aQ)}) ; \tau_\aQ
\]
where $\tau_\aQ : J(\aQ) \to M(\aQ)$ is the canonical order-isomorphism. For brevity, let $\pP := (J(\aQ),\leq_\aQ \cap\, {J(\aQ) \times J(\aQ)} )$ so we have the bipartite graph isomorphism:
\[
\xymatrix@=15pt{
J(\aQ) \ar[rr]^-{\tau_\aQ}_-\cong && M(\aQ)
\\
J(\aQ) \ar[u]^-{\leq_{\pP^{\pOp}}} \ar[rr]_-{\Delta_{J(\aQ)}}^-{\cong} && J(\aQ) \ar[u]_-{\Pirr\aQ}
}
\]
This witnesses a $\BiCliq$-morphism $\rR := \Pirr\aQ : \; \leq_{\pP^{\pOp}} \to \Pirr\aQ$ and we now show that $\Open\rR$ is a $\JSL_f$-isomorphism. First observe that:
\begin{enumerate}
\item
The $\leq_{\pP^{\pOp}}$-open sets $Y \subseteq J(\aQ)$ are precisely the down-closed subsets $Dn(\pP)$.
\item
Given any $Y \in Dn(\pP)$,
\[
\leq_{\pP^{\pOp}}^\down(Y)
= \{ j \in J(Q) : \; \leq_{\pP^{\pOp}}[j] \subseteq Y \}
= \{ j \in J(Q) : \; \down_{\pP} j \subseteq Y \}
= Y,
\]
and similarly $\leq_{\pP^{\pOp}}^\up(Y) \; = \; \down_\pP Y = Y$.
\end{enumerate}
Then the join-semilattice morphism $\Open\rR$ has action:
\[
\Open\rR(Y) 
= \Pirr\aQ^\up \circ \leq_{\pP^{\pOp}}^\down(Y) 
= \Pirr\aQ[Y]
= \tau_\aQ[\leq_{\pP^{\pOp}}[Y]]
= \tau_\aQ[Y]
\]
It is injective because $\tau_\aQ$ is, and surjective because $\Pirr\aQ[X] = \tau_\aQ[\leq_{\pP^{\pOp}}[X]]$ for every $X \subseteq J(\aQ)$, so that every $\Pirr\aQ$-open set is the image of some $\leq_{\pP^{\pOp}}$-open set. Then we have the composite isomorphism:
\[
\Open \leq_{\pP^{\pOp}} \xto{\Open\rR} \Open\Pirr\aQ \xto{rep_\aQ^{\bf-1}} \aQ.
\]

\item
First observe that for every finite poset $\pP = (P,\leq_\pP)$ we have:
\[
\Open\pP = (Up(\pP),\cup,\emptyset)
\]
since the $\pP$-open sets are precisely the images $\leq_\pP[X]$ where $X \subseteq P$. By Lemma \ref{lem:bicliq_mor_char_max_witness}.1 and Lemma \ref{lem:up_down_basic}.4, the $\BiCliq$-morphisms $\rR : \; \leq_\pP \; \to \; \leq_\pQ$ are precisely those relations $\rR \subseteq P \times Q$ such that:
\[
 \rR^\up = \rR^\up \circ \cl_{\leq_\pP}
 \qquad\text{and}\qquad
 \breve{\rR}^\up = \breve{\rR}^\up \circ \cl_{\leq_\pQ\spbreve}.
\]
Regarding these closure operators, we have:
\[
\begin{tabular}{lll}
$\cl_{\leq_\pP}$
&
$= \; \leq_\pP^\down \circ \leq_\pP^\up$
\\&
$= \lambda S \subseteq P. \leq_\pP^\down(\up_\pP S)$
\\&
$= \lambda S \subseteq P. \{ p \in P : \; \up_\pP p \; \subseteq \; \up_\pP S \; \}$
\\&
$= \lambda S \subseteq P. \up_\pP S$
\end{tabular}
\]
and thus $\cl_{\leq_\pP}$ constructs the up-closure in $\pP$, so that $\cl_{\leq_\pQ\spbreve} = \cl_{\leq_{\pQ^{\pOp}}}$ constructs the down-closure in $\pQ$. Now, by monotonicity and the fact that downwards-closure preserves unions, the above equalities may equivalently be written:
\[
\forall p \in P. \; \rR[\up_\pP p] \subseteq \rR[p]
 \qquad\text{and}\qquad
\forall q \in Q. \; \breve{\rR}[\down_\pQ q] \subseteq \breve{\rR}[q].
\]
The left-hand equality says that whenever $p \leq_\pP p'$ and $\rR(p',q)$ then $\rR(p,q)$, or equivalently that $\breve{\rR}[q]$ is down-closed in $\pP$ for every $q \in Q$. As for the right-hand equality, it equivalently asserts that $\rR[p]$ is up-closed in $\pQ$ for every $p \in P$.

Finally let us apply the categorical equivalence, so that the $\JSL_f$-morphisms of type $\Open\leq_\pP \to \Open\leq_\pQ$ are precisely those of the form $\Open\rR$ where $\rR$ is restricted as above. Concerning its action,
\[
\Open\rR (Y) 
= \rR^\up \circ \leq_\pP^\down(Y)
= \rR^\up(Y)
\]
because $\{ p \in P : \; \up_\pP p \; \subseteq Y \} = Y$ whenever $Y \in Up(\pP)$. In conclusion, $\Open\rR = \rR^\up \cap Up(\leq_\pP) \times Up(\leq_\pQ)$ where it is both necessary and sufficient that the relation $\rR \subseteq P \times Q$ satisfies the claimed conditions.

\end{enumerate}
\end{proof}

\begin{note}
\item
\begin{enumerate}
\item
Concerning Lemma \ref{thm:char_jsl_mor_bool_distr}.1, the proof can be contrasted with another method i.e.\ use the duality between finite boolean algebras and finite sets, and view the representing boolean algebra isomorphism as a $\JSL_f$-isomorphism, see Theorem \ref{thm:bool_fin_duality} in the Appendix.

\item
Concerning Lemma \ref{thm:char_jsl_mor_bool_distr}.3, an alternative proof would use Birkhoff's duality between finite bounded distributive lattices and finite posets, viewing the representing bounded distributive lattice isomorphism as a $\JSL_f$-isomorphism -- see Theorem \ref{thm:birkhoff_duality} in the Appendix.

\item
Regarding the bipartite graph isomorphism in Lemma \ref{thm:char_jsl_mor_bool_distr}.3, such isomorphisms always induce $\BiCliq$-isos -- see Example \ref{ex:dep_morphisms}.2 \endbox

\end{enumerate}
\end{note}

\subsection{The equivalence $\JSL_f \cong \BiCliq$ without using irreducibles}

In this short subsection we describe a functor $\Nleq$ which is naturally isomorphic to $\Pirr : \JSL_f \to \BiCliq$. On objects, $\Nleq \aQ = \; \nleq_\aQ \; \subseteq Q \times Q$ is the full unrestricted relation i.e.\ makes no mention of join/meet-irreducibles. We'll describe the equivalence between $\JSL_f$ and $\BiCliq$ in terms of $\Nleq$ and $\Open$ i.e.\ explicitly describe their respective natural isomorphisms.

\begin{lemma}
\label{lem:pirr_to_nleq_iso}
Let $\aQ$ be any finite join-semilattice.
\begin{enumerate}
\item
$\cl_{\nleq_\aQ}(S) = \{ q \in Q : q \leq_\aQ \Lor_\aQ S \}$.

\item
We have the $\BiCliq$-isomorphism $\rE_\aQ : \Pirr\aQ \to \; \nleq_\aQ$ defined:
\[
\begin{tabular}{lllll}
$\rE_\aQ$ & $:= \{ (j,q) \in J(\aQ) \times Q : j \nleq_\aQ q \}$
&&
$\rE_\aQ^{\bf-1}$ & $:= \{ (q,m) \in Q \times M(\aQ) : q \nleq_\aQ m \}$
\\[0.5ex]
$(\rE_\aQ)_-$ & $:= \{ (j,q) \in J(\aQ) \times Q : q \leq_\aQ j \}$
&&
$(\rE_\aQ^{\bf-1})_-$ & $:= \{ (q,j) \in Q \times J(\aQ) : j \leq_\aQ q \}$
\\
$(\rE_\aQ)_+$ & $:= \{ (q,m) \in Q \times M(\aQ) : q \leq_\aQ m \}$
&&
$(\rE_\aQ^{\bf-1})_+$ & $:= \{ (m,q) \in M(\aQ) \times Q : m \leq_\aQ q \}$
\end{tabular}
\]
\end{enumerate}
\end{lemma}

\begin{proof}
\item
\begin{enumerate}
\item
This follows by a simple calculation:
\[
\begin{tabular}{lll}
$\cl_{\nleq_\aQ} (S)$
& $= \; \nleq_\aQ^\down \circ \nleq_\aQ^\up (S)$
\\&
$= \; \nleq_\aQ^\down (\{ q \in Q : \exists s \in S. s \nleq_\aQ q \}$
\\&
$= \{ q \in Q : \;\nleq_\aQ[q] \subseteq \{ q' \in Q : \exists s \in S. s \nleq_\aQ q' \} \}$
\\&
$= \{ q \in Q : \; \{ q' \in Q : \forall s \in S. s \leq_\aQ q' \} \subseteq \; \leq_\aQ[q] \}$
\\&
$= \{ q \in Q : \; \forall q' \in Q. ( \Lor_\aQ S \leq_\aQ q' \To q \leq_\aQ q' ) \}$
\\&
$= \{ q \in Q : q \leq_\aQ \Lor_\aQ S \}$
\end{tabular}
\]

\item
$\rE_\aQ$ is a well-defined $\BiCliq$-morphism because $(\rE_\aQ)_- ; \; \nleq_\aQ \; =  \rE_\aQ = \Pirr\aQ ; (\rE_\aQ)_+\spbreve$ as is easily verified. These are $\rE_\aQ$'s associated component relations because each $(\rE_\aQ)_-[j]$ is closed via $\cl_{\nleq_\aQ}$ (see (1)), and each $(\rE_\aQ)_+[q]$ is closed via $\cl_{(\Pirr\aQ)\spbreve} = \cl_{\Pirr(\aQ^{\pOp})}$ (see Lemma \ref{lem:cl_inte_of_pirr}.2). Similarly $\rI_{\aQ}^{\bf-1}$ is a well-defined $\BiCliq$-morphism and its associated components are correct, observing that they are \emph{not} the converses of $\rE_\aQ$'s components. Finally:
\[
\begin{tabular}{lll}
$\rE_\aQ \fatsemi \rE_\aQ^{\bf-1}$
&
$= \rE_\aQ ; (\rE_\aQ^{\bf-1})_+\spbreve$
\\&
$= \{ (j,m) \in J(\aQ) \times M(\aQ) : \exists q \in Q. (j \nleq q \,\land\, m \leq_\aQ q)  \}$
\\&
$= \{ (j,m) \in J(\aQ) \times M(\aQ) : \neg\forall q \in Q. (m \leq_\aQ q \To j \leq q )  \}$
\\&
$= \{ (j,m) \in J(\aQ) \times M(\aQ) : j \nleq_\aQ m \}$
\\&
$= id_{\Pirr \aQ}$  
\\
\\
$\rE_\aQ^{\bf-1} \fatsemi \rE_\aQ$
&
$= \rE_\aQ^{\bf-1} ; (\rE_\aQ)_+\spbreve$
\\&
$= \{ (q_1,q_2) \in Q \times Q : \exists m \in M(\aQ).( q_1 \nleq_\aQ m \,\land\, q_2 \leq_\aQ m  ) \}$
\\&
$= \{ (q_1,q_2) \in Q \times Q : \neg\forall m \in M(\aQ).( q_2 \leq_\aQ m \To  q_1 \leq_\aQ m  ) \}$
\\&
$= \{ (q_1,q_2) \in Q \times Q : q_1 \nleq_\aQ q_2 \}$
\\&
$= id_{\nleq_\aQ}$
\end{tabular}
\]
using the definition of $\BiCliq$-composition.

\end{enumerate}
\end{proof}

\begin{definition}[The equivalence functor $\Nleq : \JSL_f \to \BiCliq$]
\label{def:nleq_functor_nat_iso}
\item
The functor $\Nleq : \JSL_f \to \BiCliq$ is defined:
\[
\Nleq \aQ := \; \nleq_\aQ \; \subseteq Q \times Q
\qquad
\dfrac{f : \aQ \to \aR}
{\Nleq f := \{ (q,r) \in Q \times R : f(q) \nleq_\aR r \} : \; \nleq_\aQ \; \to \; \nleq_\aR }
\]
We also have the natural isomorphism $\rE : \Pirr \to \Nleq$ whose components $\rE_\aQ$  are described in Lemma \ref{lem:pirr_to_nleq_iso}.2. \endbox
\end{definition}

\begin{lemma}
\label{lem:nleq_e_well_def}
\item
\begin{enumerate}
\item
$\Nleq : \JSL_f \to \BiCliq$ is a well-defined functor.
\item
$\rE : \Pirr \to \Nleq$ is a well-defined natural isomorphism. 
\end{enumerate}
\end{lemma}

\begin{proof}
Given any join-semilattice morphism $f : \aQ \to \aR$, let us show that:
\[
\Nleq f = \rE_\aQ^{\bf-1} \fatsemi \Pirr f \fatsemi \rE_\aR
\]
Before doing so, we first compute:
\[
\begin{tabular}{lll}
$(\rE_\aR)_+ ; (\Pirr f)_+$
&
$= \{ (r,m_q) \in R \times M(\aQ) : \exists m_r \in M(\aR).( r \leq_\aR m_r \,\land\, f_*(m_r) \leq_\aR m_q )  \}$
\\&
$= \{ (r,m_q) \in R \times M(\aQ) : f_*(r) \leq_\aR m_q )  \}$
\end{tabular}
\]
Regarding the final equality, $\subseteq$ follows because $f_* : \aR^{\pOp} \to \aQ^{\pOp}$  also defines a monotonic function from $(R,\leq_\aR)$ to $(Q,\leq_\aQ)$, and $\supseteq$ follows because $M(\aR) = J(\aR^{\pOp})$ so that $r$ arises as a $\lor_{\aR^{\pOp}}$-join of join-irreducibles $m_r$.
\[
\begin{tabular}{lll}
$\rE_\aQ^{\bf-1} \fatsemi \Pirr f \fatsemi \rE_\aR$
&
$= (\rE_\aQ^{\bf-1} \fatsemi \Pirr f) ; (\rE_\aR)_+\spbreve$
\\&
$= \rE_\aQ^{\bf-1} ; (\Pirr f)_+\spbreve ; (\rE_\aR)_+\spbreve$
\\&
$= \rE_\aQ^{\bf-1} ; ((\rE_\aR)_+ ; (\Pirr f)_+)\spbreve$
\\&
$= \rE_\aQ^{\bf-1} ; ( \{ (r,m_q) \in R \times M(\aQ) : f_*(r) \leq_\aR m_q )  \})\spbreve$
& see above
\\&
$= \{ (q,r) \in Q \times R: \exists m \in M(\aQ).( q \nleq_\aQ m \,\land\, f_*(r) \leq_\aR m  ) \}$
\\&
$= \{ (q,r) \in Q \times R: \neg\forall m \in M(\aQ).( f_*(r) \leq_\aR m \To q \leq_\aQ m  ) \}$
\\&
$= \{ (q,r) \in Q \times R : q \nleq_\aQ f_*(r) \}$
\\&
$= \{ (q,r) \in Q \times R : f(q) \nleq_\aR r \}$
& adjoint relationship
\\&
$= \Nleq f$
\end{tabular}
\]
Thus the action of $\Nleq$ is well-defined. In fact for completely general reasons it inherits functorality from $\Pirr$. Firstly we have $\Nleq id_\aQ = \rE_\aQ^{\bf-1} \fatsemi \Pirr id_\aQ \fatsemi \rE_\aQ= \rE_\aQ^{\bf-1} \fatsemi \rE_\aQ = id_{\Nleq \aQ}$, and secondly:
\[
\begin{tabular}{lll}
$\Nleq(g \circ f)$
&
$= \rE_\aQ^{\bf-1} \fatsemi \Pirr (g \circ f) \fatsemi \rE_\aS$
\\&
$= \rE_\aQ^{\bf-1} \fatsemi \Pirr f \fatsemi \Pirr g \fatsemi \rE_\aS$
\\&
$= (\rE_\aQ^{\bf-1} \fatsemi \Pirr f \fatsemi \rE_\aR) \fatsemi (\rE_\aR^{\bf-1} \fatsemi \Pirr g \fatsemi \rE_\aR)$
\\&
$= \Nleq f \fatsemi \Nleq g$
\end{tabular}
\]
Finally, the fact that each $\rE_\aQ$ is a $\BiCliq$-isomorphism and $\Nleq f = \rE_\aQ^{\bf-1} \fatsemi \Pirr f \fatsemi \rE_\aR$ immediately implies that $\rE : \Pirr \to \Nleq$ defines a natural isomorphism.
\end{proof}

\begin{theorem}[Equivalence between $\JSL_f$ and $\BiCliq$ involving $\Nleq$]
\label{thm:jsl_bicliq_equiv_without_irr}
\item
The functors $\Nleq : \JSL_f \to \BiCliq$ and $\Open$ define an equivalence of categories with associated natural isomorphisms:
\[
\begin{tabular}{lll}
$\alpha :  \Id_{\JSL_f} \To \Open \circ \Nleq$
&
$\alpha_\aQ := \lambda q \in Q. \overline{\up_\aQ q}$
&
$\alpha_\aQ^{\bf-1} := \lambda Y. \Land_\aQ \overline{Y}$
\\[0.5ex]
$\beta : \Id_{\BiCliq} \To \Nleq \circ \Open$
&
$\beta_\rG := \{ (g_s,Y) \in \rG_s \times O(\rG) : \rG[g_s] \nsubseteq Y \}$
&
$\beta_\rG^{\bf-1} := \; \breve{\in} \; \subseteq O(\rG) \times \rG_t$
\end{tabular}
\]
where $\beta_\rG$ and its inverse have associated components:
\[
\begin{tabular}{llll}
$(\beta_\rG)_-$ & $:= \{ (g_s,X) \in \rG_s \times O(\rG) : X \subseteq \rG[g_s] \}$
&
$(\beta_\rG)_+$ & $:=  \; \breve{\nin} \; \subseteq O(\rG) \times \rG_t$
\\
$(\beta_\rG^{\bf-1})_-$ & $:= \{ (X,g_s) \in O(\rG) \times \rG_s : \rG[g_s] \subseteq X \}$
&
$(\beta_\rG^{\bf-1})_+$ & $:= \{ (g_t,Y) \in \rG_t \times O(\rG) : \inte_\rG(\overline{g_t}) \subseteq Y \}$
\end{tabular}
\]
\end{theorem}

\begin{proof}
We'll combine Theorem \ref{thm:bicliq_jirr_equivalent} with the natural isomorphism $\rE : \Pirr \To \Nleq$. That is, we define:
\[
\begin{tabular}{ll}
$\alpha :=$ & $ \Id_{\JSL_f} \xto{rep} \Open \circ \Pirr \xto{\Open \rE_-} \Open \circ \Nleq$
\\[0.5ex]
$\beta :=$ & $\Id_{\BiCliq} \xto{red} \Pirr \circ \Open \xto{\rE_{\Open-}} \Nleq \circ \Open$
\end{tabular}
\]
Since they are built from natural isomorphisms and functors, they are themselves natural isomorphisms i.e.\ we have have an equivalence of categories. Let us now verify their action:
\[
\begin{tabular}{lll}
$\alpha_\aQ (q)$
&
$= \Open \rE_\aQ \circ rep_\aQ (q)$
\\&
$= \Open \rE_\aQ(\{ m \in M(\aQ) : q \nleq_\aQ m \})$
\\&
$= (\rE_\aQ)_+\spbreve[\{ m \in M(\aQ) : q \nleq_\aQ m \}]$
\\&
$= \{ q' \in Q : \exists m \in M(\aQ).( q' \leq_\aQ m \,\land\, q \nleq_\aQ m ) \}$
\\&
$= \{ q' \in Q : \neg\forall m \in M(\aQ).( q' \leq_\aQ m \To q \leq_\aQ m) \}$
\\&
$= \{ q' \in Q : q \nleq_\aQ q' \}$
\\&
$= \overline{\up_\aQ q}$
\end{tabular}
\]
Then since we know $\alpha_\aQ$ is an isomorphism it follows that $\alpha_\aQ^{\bf-1}$ is the inverse. Recalling that $\leq_{\Open\rG}$ is the inclusion relation on the $\rG$-open sets $O(\rG)$, then:
\[
\begin{tabular}{lll}
$\beta_\rG$
&
$= \red_\rG \fatsemi \rE_{\Open\rG}$
\\&
$= \red_\rG ; (\rE_{\Open\rG})_+\spbreve$
\\&
$= \{ (g_s,Y) \in \rG_s \times O(\rG) : \exists M \in M(\Open\rG).( \rG[g_s] \nleq_{\Open\rG} M \,\land\, Y \leq_{\Open\rG} M  ) \}$
\\&
$= \{ (g_s,Y) \in \rG_s \times O(\rG) : \neg\forall M \in M(\Open\rG).( Y \leq_{\Open\rG} M \To  \rG[g_s] \leq_{\Open\rG} M) \}$
\\&
$= \{ (g_s,Y) \in \rG_s \times O(\rG) : \rG[g_s] \nsubseteq Y \}$
\\
\\
$\beta_\rG^{\bf-1}$
&
$= \rE_{\Open\rG}^{\bf-1} \fatsemi red_\rG^{\bf-1}$
\\&
$= \rE_{\Open\rG}^{\bf-1} ; (red_\rG^{\bf-1})_+\spbreve$
\\&
$= \{ (Y,g_t) \in O(\rG) \times \rG_t : \exists M \in M(\Open\rG).( Y \nleq_{\Open\rG} M \,\land\, \inte_\rG(\overline{g_t}) \leq_{\Open\rG} Y)  \}$
\\&
$= \{ (Y,g_t) : \neg\forall M \in M(\Open\rG).(\inte_\rG(\overline{g_t}) \leq_{\Open\rG} Y \To Y \leq_{\Open\rG} M)  \}$
\\&
$= \{ (Y,g_t)  : Y \nsubseteq \inte_\rG(\overline{g_t}) \}$
\\&
$= \{ (Y_,g_t)  : g_t \in Y\}$
& by Lemma \ref{lem:cl_inte_of_pirr}.1
\\&
$= \breve{\in} \; \subseteq O(\rG) \times \rG_t$
\end{tabular}
\]
The descriptions of $\beta_\rG$ and $\beta_\rG^{\bf-1}$'s associated components follows via similar simple computations.
\end{proof}

\subsection{$\BiCliq$ as a canonical construction}

In this subsection we provide an alternative description of $\BiCliq$.

\begin{quote}
We introduce the category $\Cover$ which is essentially the arrow category of $\Rel_f$. Its hom-sets admit a natural closure structure, so that $\BiCliq$ is the restriction of $\Cover$ to \emph{closed morphisms}.
\end{quote}

This closure structure is really just a more detailed explanation of the maximum $\rR$-witnesses $(\rR_-,\rR_+)$, revealing that their construction is functorial in nature. Given what we already know, it is not hard to prove. However it is useful because it allows us to work with morphisms `modulo closure' in a precise sense. 

\begin{motivation}
Of particular importance is the following basic fact. Given a finite set of $\BiCliq$-endomorphisms $\{ (\rR_-^a,\rR_+^a) : \rG \to \rG : a \in \Sigma \}$ then:
\[
(\rR_-^{a_1},\rR_+^{a_1}) \fatsemi \cdots \fatsemi (\rR_-^{a_n},\rR_+^{a_n})
\text{\quad is the closure of \quad}
(\rR_-^{a_1} ; \cdots ; \rR_-^{a_n},  \; \rR_+^{a_n} ; \cdots ; \rR_+^{a_1})
\]
That is, we may use the usual relational composition in each component and then close once. The restriction to endomorphisms is unnecessary. The reason we emphasise it stems from our interest in nondeterministic acceptance of regular languages. Later on, the endomorphisms $(\rR_-^a,\rR_+^a) : \rG \to \rG$ will be viewed as the $a$-transitions of two classical nondeterministic automata, one with states $\rG_s$ (the `lower one') and the other with states $\rG_t$ (the `upper one'). These paired nondeterministic automata naturally accept a single regular language, using only the definition of $\BiCliq$-composition. Then using the above fact, this language is precisely the language accepted by the lower nfa, or equivalently the reverse of the language accepted by the upper nfa. That is, our `categorical' notion of language acceptance corresponds to the classical notion of nondeterministic acceptance. \endbox
\end{motivation}

\begin{definition}[The category $\Cover$] 
The objects of $\Cover$ are the relations between finite sets i.e.\ the objects of $\BiCliq$. A morphism $(\rR_l,\rR_r) : \rG \to \rH$ is a pair of relations $\rR_l \subseteq \rG_s \times \rH_s$ and $\rR_r \subseteq \rH_t \times \rG_t$ such that:
\[
\rR_l ; \rH = \rG ; \rR_r\spbreve
\]
Then $id_\rG := (\Delta_{\rG_s},\Delta_{\rG_t})$ and composition is defined $(\rR_l,\rR_r);(\rS_l,\rS_r) := (\rR_l;\rS_l,\rS_r;\rR_r)$. \endbox
\end{definition}

\begin{definition}[$\BiCliq$-morphism associated to a $\Cover$-morphism]
A $\Cover$-morphism $(\rR_l,\rR_r) : \rG \to \rH$ has an \emph{associated $\BiCliq$-morphism} $\rR : \rG \to \rH$, namely $\rR_l ; \rH =: \rR := \rG ; \rR_r\spbreve$. \endbox
\end{definition}

\begin{note}[$\Cover$ is isomorphic to the arrow category of $\Rel_f$]
\label{note:category_cover}
Arguably the most natural category whose objects are relations between finite sets is the \emph{arrow category} $Arr(\Rel_f)$ of the category of finite sets and relations $\Rel_f$. It is isomorphic to $\Cover$ by (i) reversing the type of $\rR_+$, and (ii) changing composition appropriately (use pairwise relational composition). In fact, $Arr(\Rel_f)$ corresponds to the comma category $\Id_{\Rel_f} \down \Id_{\Rel_f}$, whereas $\Cover$ corresponds to $\Id_{\Rel_f} \down (-)\spbreve : \Rel_f^{op} \to \Rel_f$. Small comma categories always arise as natural pullbacks in $\SmallCat$, the category of small categories. \endbox
\end{note}

$\Cover$ is a well-defined category by Note \ref{note:category_cover} above. Given $(\rR_l,\rR_r) : \rG \to \rH$ s.t.\ $\rR_l ; \rH = \rG ; \rR_r\spbreve$, taking the relational converse yields $\rR_r ; \breve{\rG} = \breve{\rH} ; \rR_l\spbreve$ i.e.\ a $\Cover$-morphism $(\rR_r,\rR_l) : \breve{\rH} \to \breve{\rG}$. This defines a self-duality and acts in the same way as $\BiCliq$'s self-duality (Definition \ref{def:bicliq_self_duality}); we denote it by the same symbol.

\begin{lemma}[Self-duality of $\Cover$]
$\Cover$ is well-defined and self-dual via
$(-)\spcheck : \Cover^{op} \to \Cover$,
\[
\rG\spcheck := \breve{\rG}
\qquad
\dfrac{(\rR_l,\rR_r) : \rG \to \rH}{(\rR_l,\rR_r)\spcheck := (\rR_r,\rR_l) : \breve{\rH} \to \breve{\rG}}
\]
with witnessing natural isomorphism $\alpha : \Id_{\Cover} \To (-)\spcheck \circ ((-)\spcheck)^{op}$ with action $\alpha_\rG := id_\rG = (\Delta_{\rG_s},\Delta_{\rG_t})$.
\end{lemma}

$\Cover$'s hom-sets admit a natural ordering i.e.\ pairwise inclusion. We now define a natural closure operator uniformly on each such poset.

\begin{definition}[The poset $\Cover(\rG, \rH)$]
For each pair of relations $(\rG,\rH)$ we define the finite poset:
\[
\pCovHom{\rG}{\rH} := (\Cover(\rG,\rH),\leq_{(\rG,\rH}))
\qquad
\text{where}
\qquad
(\rR_1,\rR_2) \leq_{(\rG,\rH)} (\rS_1,\rS_2) :\iff \rR_1 \subseteq \rS_1
\text{ and }
\rR_2 \subseteq \rS_2.
\]
This poset of morphisms admits a natural closure operator $\cl_{\rG,\rH} : \pCovHom{\rG}{\rH} \to \pCovHom{\rG}{\rH}$ defined:
\[
\begin{tabular}{c}
$\cl_{\rG,\rH}(\rR_l,\rR_r) := (\rR_l^\bullet,\,\rR_r^\bullet)$ where:
\\[1ex]
$\rR_l^\bullet := \{ (g_s,h_s) \in \rG_s \times \rH_s : h_s \in \cl_\rH(\rR_l[g_s]) \}
\qquad
\rR_r^\bullet := \{ (h_t,g_t) \in \rH_t \times \rG_t : g_t \in \cl_{\breve{\rG}}(\rR_r[h_t]) \}$
\end{tabular}
\]
using the closure operators $\cl_\rH = \rH^\down \circ \rH^\up$ and $\cl_{\breve{\rG}} = \breve{\rG}^\down \circ \breve{\rG}^\up$ from Definition \ref{def:up_down}. \endbox
\end{definition}

\smallskip
Each finite poset $\pCovHom{\rG}{\rH}$ is actually a finite lattice: the bottom is $(\emptyset,\emptyset) : \rG \to \rH$ and the join is pairwise binary union (the meet structure is induced). We now prove that these closure operators are well-defined and construct the associated components $(\rR_-,\rR_+)$. Furthermore they naturally interact with the self-duality and compositional structure.

\begin{lemma}
\label{lem:cl_hom}
\item
\begin{enumerate}
\item
For any $\Cover$-morphism $(\rR_l,\rR_r) : \rG \to \rH$ we have:
\[
(\rR_l,\rR_r) \leq_{(\rG,\rH)} (\rR_l^\bullet,\rR_r^\bullet)
\qquad
\rR_l^\bullet ; \rH = \rR_l ; \rH
\qquad
\rG ; \rR_r\spbreve = \rG ; (\rR_r^\bullet)\spbreve
\]
so that $\cl_{\rG,\rH}(\rR_l,\rR_r) : \rG \to \rH$ is also a $\Cover$-morphism.

\item
$\cl_{\rG,\rH}$ is a well-defined closure operator on the finite poset $\pCovHom{\rG}{\rH}$.

\item

The closure of a $\Cover$-morphism $(\rR_l,\rR_r) : \rG \to \rH$ can be described in the following three ways.
\begin{enumerate}[i.]
\item
The components $(\rR_-,\rR_+)$ of its associated $\BiCliq$-morphism $\rR$.
\item
The pairwise union of all $\Cover$-morphisms $(\rS_l,\rS_r) : \rG \to \rH$ such that $\rS_l ; \rH = \rR_l ; \rH$.
\item
The pairwise union of all $\Cover$-morphisms $(\rS_l,\rS_r) : \rG \to \rH$ such that $\rG ; \rS_r\spbreve = \rG ; \rR_l\spbreve$.
\end{enumerate}

\item
The closure operators $\cl_{\rG,\rH}$ commute with $\Cover$'s self-duality i.e.\
\[
\cl_{\rH\spcheck,\rG\spcheck}((\rR_l,\rR_r)\spcheck)
=
(\cl_{\rG,\rH}(\rR_l,\rR_r))\spcheck
\]

\item
The closure operators $\cl_{\rG,\rH}$ are well-behaved w.r.t.\ $\Cover$-composition i.e.\
\[
\begin{tabular}{c}
$\dfrac{
\cl_{\rG,\rH}(\rR_l,\rR_r) = \cl_{\rG,\rH}(\rS_l,\rS_r)
}{
\cl_{\rG,\rI}((\rR_l,\rR_r) ; (\rT_l,\rT_r)) 
= \cl_{\rG,\rI}((\rS_l,\rS_r);(\rT_l,\rT_r))
}$
\\[3ex]
$\dfrac{
\cl_{\rH,\rI}(\rS_l,\rS_r) = \cl_{\rH,\rI}(\rT_l,\rT_r)
}{
\cl_{\rG,\rI}((\rR_l,\rR_r) ; (\rS_l,\rS_r)) = \cl_{\rG,\rI}((\rR_l,\rR_r) ; (\rT_l,\rT_r))
}$
\end{tabular}
\]
for all appropriately typed morphisms $(\rR_l,\rR_r)$, $(\rS_l,\rS_r)$ and $(\rT_l,\rT_r)$.

\end{enumerate}
\end{lemma}

\begin{proof}
\item
\begin{enumerate}

\item
The left statement follows because $\cl_\rH$ and $\cl_{\breve{\rG}}$ are extensive. The central and right statement follow because for all $g_s \in \rG_s$ and $h_t \in \rH_t$,
\[
\begin{tabular}{ll}
\begin{tabular}{lll}
$\rR_l^\bullet ; \rH[g_s]$
& $= \rH[\rR_l^\bullet[g_s]]$
\\&
$= \rH[\cl_\rH(\rR_l[g_s])]$
\\&
$= \rH^\up \circ \rH^\down \circ \rH^\up(\rR_l[g_s])$
\\&
$= \rH^\up(\rR_l[g_s])$
& by $(\up\down\up)$
\\&
$= \rR_l ; \rH[g_s]$
\end{tabular}
&
\begin{tabular}{lll}
$\rR_r^\bullet ; \breve{\rG}[h_t]$
& $= \breve{\rG}[\rR_r^\bullet[h_t]]$
\\&
$= \breve{\rG}^\up(\cl_{\breve{\rG}}(\rR_r[h_t])) $
\\&
$= \breve{\rG}^\up \circ \breve{\rG}^\down \circ \breve{\rG}^\up(\rR_r[h_t])$
\\&
$= \breve{\rG}^\up(\rR_r[h_t])$
& by $(\up\down\up)$
\\&
$= \rR_r ; \breve{\rG}[h_t]$
\end{tabular}
\end{tabular}
\]
Then $\cl_{\rG,\rH}(\rR_l,\rR_r)$ is a well-defined $\Cover$-morphism using the fact that $(\rR_l,\rR_r)$ is.

\item
That $\cl_{\rG,\rH}$ is a well-defined function follows from the previous statement. That it is monotonic, extensive and idempotent follows because $\cl_\rH$ and $\cl_{\breve{\rG}}$ possess these properties.

\item
Given a $\Cover$-morphism $(\rR_l,\rR_r)$ let $\rR$ be its associated $\BiCliq$-morphism and $(\rR_-,\rR_+)$ the latter's associated component relations. Then:
\[
\begin{tabular}{ll}
\begin{tabular}{lll}
$\rR_-[g_s]$
& $= \rH^\down(\rR[g_s])$
& by definition
\\&
$= \rH^\down \circ (\rR_l;\rH)^\up(\{g_s\})$
& by assumption
\\&
$= \rH^\down \circ \rH^\up \circ \rR_l^\up(\{g_s\})$
& by $(\up\circ)$
\\&
$= \cl_\rH(\rR_l[g_s])$
& by definition
\\&
$= \rR_l^\bullet[g_s]$
& by definition
\end{tabular}
&
\begin{tabular}{lll}
$\rR_+[h_t]$
& $= \breve{\rG}^\down(\breve{\rR}[h_t])$
& by definition
\\&
$= \breve{\rG}^\down \circ (\rR_r;\breve{\rG})^\up(\{h_t\})$
& by assumption
\\&
$= \breve{\rG}^\down \circ \breve{\rG}^\up \circ \rR_r^\up(\{h_t\})$
& by $(\up\circ)$
\\&
$= \cl_{\breve{\rG}}(\rR_r[h_t])$
& by definition
\\&
$= \rR_r^\bullet[h_t]$
& by definition
\end{tabular}
\end{tabular}
\]
for all $g_s \in \rG_s$ and $h_t \in \rH_t$. This proves the first statement. By Lemma \ref{lem:bicliq_mor_char_max_witness}.2 we know that $(\rR_-,\rR_+)$ is the union of all $\Cover$-morphisms $(\rS_l,\rS_r) : \rG \to \rH$ such that $\rS_l ; \rH = \rR = \rG ; \rS_r\spbreve$. Then the second and third statement follow by $\rR_l ; \rH = \rR$ and $\rR = \rG ; \rR_l\spbreve$ respectively.

\item
Follows from the definitions:
\[
\cl_{\rH\spcheck,\rG\spcheck}((\rR_l,\rR_r)\spcheck) 
= \cl_{\breve{\rH},\breve{\rG}}(\rR_r,\rR_l) 
= (\rR_r^\bullet,\rR_l^\bullet) 
= (\rR_l^\bullet,\rR_r^\bullet)\spcheck
= (\cl_{\rG,\rH}(\rR_l,\rR_r))\spcheck
\]

\item

To prove the first rule, assume we have $\Cover$-morphisms $(\rR_l,\rR_r),\,(\rS_l,\rS_r) : \rG \to \rH$ with the same closure, and also a $\Cover$-morphism $(\rT_l,\rT_r) : \rH \to \rI$. We need only show that $(\rR_l ; \rT_l)^\bullet = (\rS_l ; \rT_l)^\bullet$ because the other component relation is uniquely determined. Then we calculate:
\[
\begin{tabular}{lll}
$(\rR_l ; \rT_l)^\bullet [g_s]$
&
$= \cl_\rI(\rR_l ; \rT_l[g_s])$
& by definition
\\&
$= \rI^\down \circ \rI^\up \circ \rT_l^\up \circ \rR_l^\up(\{g_s\})$
& by definition
\\&
$= \rI^\down \circ \rT^\up (\rR_l[g_s])$
& since $\rT = \rT_l ; \rI$
\\&
$= \rI^\down \circ \rT^\up \circ \cl_\rH(\rR_l[g_s])$
& since $\rT$ a $\BiCliq$-morphim
\\&
$= \rI^\down \circ \rT^\up \circ \cl_\rH(\rS_l[g_s])$
& since $\rR_l^\bullet = \rS_l^\bullet$
\\&
$= (\rS_l ; \rT_l)^\bullet [g_s]$
& reasoning in reverse
\end{tabular}
\]
for all $g_s \in \rG_s$. The second rule follows by dualising (using (4)), applying the first rule, and dualising again.

\end{enumerate}
\end{proof}

Each closure operator $\cl_{\rG,\rH}$ induces an equivalence relation on its respective hom-set i.e.\ the kernel:
\[
\ker{\cl_{\rG,\rH}} \subseteq \Cover(\rG,\rH) \times \Cover(\rG,\rH)
\]
which relates those morphisms with the same closure. Then by Lemma \ref{lem:cl_hom}.5 these relations are collectively compatible with $\Cover$-composition and thus induce a `quotient category'. We denote the composition of morphisms in this category by `$\fatsemi$' i.e.\ the same symbol we use to denote $\BiCliq$-composition. This is warranted because these two categories are isomorphic.

\begin{definition}[The category $\Cover / \cl$]
It has the same objects as $\Cover$, whereas its hom-sets are:
\[
\Cover/\cl\;(\rG,\rH) := \Cover(\rG,\rH)/\ker{\cl_{\rG,\rH}}
\]
i.e.\ the equivalence classes of $\Cover$-morphisms relative to $\ker{\cl_{\rG,\rH}}$. Let us denote the associated surjective canonical maps by $\sem{\cdot}_{\rG,\rH} : \Cover(\rG,\rH) \epito \Cover/\cl (\rG,\rH)$. Then identity morphisms and composition are defined:
\[
id_\rG := \sem{id_\rG}_{\rG,\rG} = \sem{(\Delta_{\rG_s},\Delta_{\rG_t})}_{\rG,\rG}
\qquad
\sem{(\rR_l,\rR_r)}_{\rG,\rH} \fatsemi \sem{(\rS_l,\rS_r)}_{\rH,\rI}
:= \sem{(\rR_1,\rR_l) ; (\rS_l,\rS_r)}_{\rG,\rI}
\]
We also define two identity-on-objects functors:
\[
\begin{tabular}{llc}
$\sem{\cdot} : \Cover \to \Cover/\cl$
&
$\sem{\rG} := \rG$
&
$\dfrac{(\rR_l,\rR_r) : \rG \to \rH}{\sem{(\rR_l,\rR_r)} := \sem{(\rR_l,\rR_r)}_{\rG,\rH}  : \rG \to \rH}$
\\[2.5ex]
$I : \Cover/\cl \to \BiCliq$
&
$I\rG := \rG$
&
$\dfrac{\sem{(\rR_l,\rR_r)}_{\rG,\rH} : \rG \to \rH}{\rR := \rR_l ; \rH = \rG ; \rR_r\spbreve : \rG \to \rH}$
\end{tabular}
\]
and also the composite functor $\cl := I \circ \sem{\cdot} : \Cover \to \BiCliq$. Recalling that $\BiCliq$-morphisms $\rR$ may be identified with their associated components $(\rR_-,\rR_+)$, we may abuse notation by equivalently defining:
\[
\cl (\rR_l,\rR_r) := \cl_{\rG,\rH}(\rR_l,\rR_r) = (\rR_-,\rR_+)
\]
\endbox
\end{definition}

\begin{theorem}[$\BiCliq$ as a quotient category of $\Cover$]
\item
\begin{enumerate}
\item
$\Cover/\cl$ is a well-defined category and $\sem{\cdot} : \Cover \to \Cover/\cl$ is a well-defined functor.
\item
$I : \Cover/\cl \to \BiCliq$ is a well-defined isomorphism of categories.
\item
$\cl : \Cover \to \BiCliq$ is a well-defined functor and preserves the  ordering on morphisms i.e.\
\[
(\rR_l,\rR_r) \leq_{(\rG,\rH)} (\rS_l,\rS_r)
\implies \rR \subseteq \rS
\qquad
\text{(or equivalently $(\rR_-,\rR_+) \leq_{\rG,\rH} (\rS_-,\rS_+)$)}
\]
\end{enumerate}
\end{theorem}

\begin{proof}
\item
\begin{enumerate}
\item
Follows via Lemma \ref{lem:cl_hom}.5, also see section II.8 on `Quotient functors' in MacClane's book.

\item
$I$'s action on objects and morphisms is well-defined, noting that elements of the same equivalence class induce the same $\BiCliq$-morphism by definition. Concerning preservation of identity morphisms:
\[
\begin{tabular}{lll}
$I id_\rG$
&
$= I \sem{(\Delta_{\rG_s},\Delta_{\rG_t})}_{\rG,\rG}$
\\&
$= I \sem{(\rG_-,\rG_+)}_{\rG,\rG}$
& $(\Delta_{\rG_s},\Delta_{\rG_t})$ a $\rG$-witness, Lemma \ref{lem:cl_hom}.3
\\&
$= \rG_- ; \rG$
\\&
$= \rG$
\end{tabular}
\]
and regarding preservation of composition:
\[
\begin{tabular}{lll}
$I (\sem{(\rR_l,\rR_r)}_{\rG,\rH} \fatsemi \sem{(\rS_l,\rS_r)}_{\rH,\rI})$
&
$= I \sem{(\rR_l,\rR_r) ; (\rS_l,\rS_r)}_{\rG,\rI}$
& by definition
\\&
$= \rR \fatsemi \rS$
& by Corollary \ref{cor:bicliq_func_comp}
\end{tabular}
\]
Next, $I$ is faithful because distinct equivalence classes induce distinct $\BiCliq$-morphisms. It is full by passing from $\rR : \rG \to \rH$ to $\sem{(\rR_-,\rR_+)}_{\rG,\rH}$. Finally it acts like the identity on objects, so we have an isomorphism of categories.

\item
Consequently the composite functor $\cl := I \circ \sem{\cdot}$ is well-defined. Then it preserves the natural ordering on morphisms: given $\rR_l \subseteq \rS_l$ then $\rR = \rR_l ; \rH \subseteq \rS_l ; \rH = \rS$ because relational composition is monotonic separately in each argument.
\end{enumerate}
\end{proof}

We now deduce an important property, viewing $\BiCliq$-morphisms as components $(\rR_-,\rR_+)$.

\begin{corollary}
For any $n \geq 0$ and any chain of $\BiCliq$-morphisms $((\rR_-^i,\rR_+^i) : \rG_i \to \rG_{i+1})_{1 \leq i \leq n}$,
\[
(\rR_-^1,\rR_+^1) \fatsemi \cdots \fatsemi (\rR_-^n,\rR_+^n)
= \cl_{(\rG_1,\rG_{n+1})}((\rR^l_-,\rR^1_+);\cdots;(\rR^n_-,\rR^n_+))
\]
By the usual convention, the case $n = 0$ is the fact that $(\rG_-,\rG_+) = \cl_{(\rG,\rG)}(\Delta_{\rG_s},\Delta_{\rG_t})$.
\end{corollary}

\begin{proof}
This is simply the action of $\cl : \Cover \to \BiCliq$ on composite morphisms.
\end{proof}

\subsection{Dedekind-MacNeille completions}

\begin{definition}[Dedekind-MacNeille completion of finite posets]
\label{def:dedekind_macneille_completion}
\item
Given any finite poset $\pP$ then:
\begin{enumerate}
\item
its \emph{Dedekind-MacNeille completion} is the finite join-semilattice $\DeMc{\pP} := \Open\nleq_\pP = (O(\nleq_\pP,\cup,\emptyset)$. 
\item
its associated \emph{canonical order-embedding} is defined:
\[
e_\pP : \pP \to (O(\nleq_\pP),\subseteq)
\qquad\text{where}\qquad 
e_\pP(p) := \, \nleq_\pP[p] = \overline{\up_\pP p}.
\]
noting that $\nleq_\pP[p] = \overline{\leq_\pP}[p] = \overline{\leq_\pP[p]} = \overline{\up_\pP p}$.
\endbox
\end{enumerate}
\end{definition}

\smallskip

\begin{theorem}[[Dedekind-MacNeille embedding for finite posets]
\label{thm:dedekind_macneille_construction}
\item
$e_\pP : \pP \to U\DeMc{\pP}$ is a well-defined order-embedding, and preserves all meets and joins which exist in $\pP$.
\end{theorem}

\begin{proof}
$e_\pP$ is a well-defined function because $\DeMc{\pP}$ has carrier $O(\nleq_\pP) = \{ \nleq_\pP[X] : X \subseteq P \}$. Then:
\[
p_1 \leq_\pP p_2
\iff \up_\pP p_2 \subseteq \, \up_\pP p_1
\iff \overline{\up_\pP p_1} \subseteq \overline{\up_\pP p_2}
\iff e_\pP(p_1) \leq_{\Open\nleq_\pP} e_\pP(p_2).
\]
so that $e_\pP$ is an order-embedding. Next, given that $\Lor_\pP X$ exists we'll show that $e_\pP$ preserves this join:
\[
\Lor_{\Open\nleq_\pP} e[X]
= \bigcup_{p \in P} \overline{\up_\pP p} 
= \overline{\bigcap_{p \in P} \up_\pP p }
= \overline{\up_\pP \Lor_\pP X}
= e_\pP(\Lor_\pP X).
\]
Finally suppose that $\Land_\pP X$ exists. Recalling Definition \ref{def:bip_cl_inte_lattice}.3, the join-semilattice of open sets $\Open\nleq_\pP$ is isomorphic to the join-semilattice of closed sets $(C(\nleq_\pP),\lor,P)$ whose meet is intersection. This isomorphism acts on the embedding image as follows:
\[
\nleq_\pP[p]
\quad\mapsto\quad \nleq_\pP^\down(\nleq_\pP[p]) 
= \{ p' \in P : \,\nleq_\pP[p'] \subseteq \,\nleq_\pP[p] \}
\stackrel{!}{=} \{ p' \in P : p' \leq_\pP p \}
= \; \down_\pP p
\]
where in the marked equality we recall that $e_\pP$ is an order-embedding. Then since:
\[
\Land_{(C(\nleq_\pP),\lor,P)} \{ \; \down_\pP p : p \in X \}
= \bigcap_{p \in X} \down_\pP p
= \;\down_\pP \Land_\pP X
\]
applying the inverse join-semilattice isomorphism we deduce that $e_\pP$ preserves the meet $\Land_\pP X$.
\end{proof}


\subsection{Canonical embeddings and quotients}

Every finite join-semilattice $\aQ$ arises canonically as a quotient of $\JPow J(\aQ)$. It also embeds into $\JPow M(\aQ)$. In particular, we have the join-preserving morphisms:
\[
\begin{tabular}{lll}
$e_\aQ : \aQ \monoto \JPow M(\aQ)$
&& $\sigma_\aQ : \JPow J(\aQ) \epito \aQ$
\\
$e_\aQ(q) := \{ m \in J(\aQ) : q \nleq_\aQ m \}$
&& $\sigma_\aQ(S) := \Lor_\aQ S$.
\end{tabular}
\]
In this subsection we:
\begin{enumerate}
\item
Explain that these two constructions are adjoint.
\item
Prove a `tight extension lemma' involving them.
\item
Show how canonical embeddings/quotients can be defined parametric in a relation $\rG$,  generalising $e_\aQ$ and $\sigma_\aQ$.
\end{enumerate}

\smallskip

\begin{lemma}
[Adjoint relationships involving $e_\aQ$ and $\sigma_\aQ$]
\label{lem:sigma_e_q_adjoints}
For every finite join-semilattice $\aQ$ we have the commuting diagram:
\[
\xymatrix@=15pt{
\JPow J(\aQ) \ar[d]_{(\neg_{J(\aQ)})^{\bf-1}} \ar@{->>}[dr]^{\sigma_\aQ}  \ar[rr]^{\Pirr\aQ^\up} && \JPow M(\aQ)
\\
(\JPow J(\aQ))^{\pOp} \ar@{->>}[dr]_{(e_{\aQ^{\pOp}})_*}  & \aQ \ar[d]^-{id_\aQ} \ar@{>->}[ur]^{e_\aQ} & (\JPow M(\aQ))^{\pOp} \ar[u]_-{\neg_{M(\aQ)}} 
\\
 & \aQ \ar@{>->}[ur]_{(\sigma_{\aQ^{\pOp}})_*} &
}
\]
Equivalently, we have the three equalities:
\[
\begin{tabular}{llllll}
$\mathrm{(a)}$ & $(\Pirr\aQ)^\up = e_\aQ \circ \sigma_\aQ$ &
$\mathrm{(b)}$ & $\sigma_\aQ = (e_{\aQ^{\pOp}})_* \circ (\neg_{J(\aQ)})^{\bf-1}  $ &
$\mathrm{(c)}$ & $e_\aQ = \neg_{M(\aQ)} \circ (\sigma_{\aQ^{\pOp}})_*$
\end{tabular}
\]
\end{lemma}

\begin{proof}
\item
\begin{enumerate}[(a)]
\item
Recall $\Pirr\aQ = \; \nleq_\aQ \; \subseteq J(\aQ) \times M(\aQ)$ and observe $e_\aQ \circ \sigma_\aQ(\{j\}) = \{ m \in M(\aQ) : j \nleq_\aQ m \} = \;\nleq_\aQ[j]$ for all $j \in J(\aQ)$.

\item
First observe that $e_{\aQ^{\pOp}} : \aQ^{\pOp} \monoto \JPow M(\aQ^{\pOp}) = \JPow J(\aQ)$ has action:
\[
e_{\aQ^{\pOp}}(q) 
:= \{ m \in M(\aQ^{\pOp}) : q \nleq_{\aQ^{\pOp}} m \}
= \{ j \in J(\aQ) : j \nleq_\aQ m \}
\]
Then for any subset $X \subseteq J(\aQ)$ we calculate:
\[
\begin{tabular}{lll}
$(e_{\aQ^{\pOp}})_* \circ (\neg_{J(\aQ)})^{\bf-1}(X)$
&
$=  (e_{\aQ^{\pOp}})_*(\overline{X})$
\\&
$= \Lor_{\aQ^{\pOp}} \{ q \in Q : e_{\aQ^{\pOp}}(q) \leq_{\JPow J(\aQ)} \overline{X} \}$
\\&
$= \Land_\aQ \{ q \in Q : \{ j \in J(\aQ) : j \nleq_\aQ q \} \subseteq \overline{X} \}$
\\&
$= \Land_\aQ \{ q \in Q : X \subseteq \{ j \in J(\aQ) : j \leq_\aQ q \} \}$
\\&
$= \Land_\aQ \{ q \in Q : \Lor_\aQ X \leq_\aQ q \}$
\\&
$= \Lor_\aQ X$
\\&
$= \sigma_\aQ(X)$
\end{tabular}
\]
as required.

\item
The third equality follows from the second i.e.\ (i) reassign $\aQ \mapsto \aQ^{\pOp}$, (ii) take the adjoints of both sides recalling that $(\neg_{J(\aQ^{\pOp})})^{\bf-1}$ is self-adjoint, and (iii) post-compose by the isomorphism $\neg_{J(\aQ^{\pOp})} = \neg_{M(\aQ)}$.

\end{enumerate}
\end{proof}

\begin{lemma}[Tight extension lemma]
\label{lem:tight_mor_extn}
\item
\begin{enumerate}
\item
Each join-semilattice morphism $f : \JPow Z \to \aQ$ has a canonical compatible morphism:
\[
\xymatrix@=15pt{
\JPow Z \ar[rr]^{\rJ{f}^\up} \ar[drr]_{f} && \JPow J(\aQ) \ar@{->>}[d]^{\sigma_\aQ}
\\
&& \aQ
}
\]
where $\rJ{f} := \{ (z,j) \in Z \times J(\aQ) : j \leq_\aQ f(\{z\}) \}$.

\item
Each join-semilattice morphism $f : \aQ \to \JPow Z$ has a canonical extension:
\[
\xymatrix@=15pt{
\JPow M(\aQ) \ar[rr]^{\rM{f}^\up} && \JPow Z
\\
\aQ \ar@{>->}[u]^{e_\aQ} \ar[urr]_{f}
}
\]
where $\rM{f} := \{ (m,z) \in M(\aQ) \times Z :  f_*(\overline{z}) \leq_\aQ m \}$.
\end{enumerate}

\end{lemma}

\begin{proof}
\item
\begin{enumerate}
\item
Recalling that $\sigma_\aQ(S) := \Lor_\aQ S$, we have:
\[
\sigma_\aQ \circ \rJ{f}^\up(\{z\})
= \sigma_\aQ (\rJ{f}[z])
= \Lor_\aQ \{ j \in J(\aQ) : j \leq_\aQ f(\{z\}) \}
= f(\{z\})
\]
for each $z \in Z$, because every element is the join of those join-irreducibles beneath it. Thus commutativity follows by the freeness of $\JPow Z$.

\item
The second statement follows from the first via duality. That is, given $f$ then we  define $g := f_* \circ (\neg_Z)^{\bf-1} : \JPow Z \to \aQ^{\pOp}$ where the self-adjoint (and self-inverse as a function) isomorphism $(\neg_Z)^{\bf-1} = ((\neg_Z)^{\bf-1})_* : \JPow Z \to (\JPow Z)^{\pOp}$ takes the relative complement. Applying the first statement yields $
\sigma_{\aQ^{\pOp}} \circ \rJ{g}^\up = g = f_* \circ (\neg_Z)^{\bf-1}$ where $\rJ{g} \subseteq Z \times M(\aQ)$. Equivalently $(\neg_Z)^{\bf-1} \circ f = (\rJ{g}^\up)_* \circ (\sigma_{\aQ^{\pOp}})_*$ by taking adjoints, so post-composing with $\neg_Z$ yields:
\[
\begin{tabular}{lll}
$f$ 
&
$= \neg_Z \circ (\rJ{g}^\up)_* \circ (\sigma_{\aQ^{\pOp}})_*$
\\&
$= \neg_Z \circ \rJ{g}^\down \circ (\sigma_{\aQ^{\pOp}})_*$
& by $(\up\dashv\down)$
\\&
$= (\rJ{g}\spbreve)^\up \circ \neg_{M(\aQ)} \circ (\sigma_{\aQ^{\pOp}})_*$
& by De Morgan duality
\\&
$= (\rJ{g}\spbreve)^\up \circ e_\aQ$
& by Lemma \ref{lem:sigma_e_q_adjoints}.(b)
\end{tabular}
\]
Finally we have $\rJ{g}\spbreve = \rM{f}$ because: 
\[
\rJ{g}\spbreve[m]
= \{ z \in Z : m \leq_{\aQ^{\pOp}} g(\{z\}) \}
= \{ z \in Z : f_* (\overline{z}) \leq_\aQ m \}
= \rM{f}[m]
\]
for all $m \in M(\aQ)$.
\end{enumerate}
\end{proof}

We now define `similar' join-semilattice morphisms for any bipartite graph $\rG$.

\begin{definition}[Canonical embedding and quotient arising from a bipartite graph]
\item
\label{def:bip_canon_quo_incl}
For each bipartite graph $\rG$ take the unique (surjection,inclusion) factorisation of the $\JSL_f$-morphism $\rG^\up : \JPow \rG_s \to \JPow \rG_t$:
\[
\xymatrix@=15pt{
\JPow \rG_s \ar@{->>}[dr]_{\sigma_\rG} \ar[rr]^-{\rG^\up} && \JPow \rG_t
\\
& \Open\rG \ar@{>->}[ur]_{\iota_\rG} &
}
\qquad
\begin{tabular}{c}
\\\\
where necessarily $\sigma_\rG (X) := \rG[X]$ and $\iota_\rG(X) := X$
\end{tabular}
\]
recalling that $\Open\rG = (O(\rG),\cup,\emptyset)$ consists precisely of the sets $\rG[X]$ where $X \subseteq \rG_t$ by Lemma \ref{lem:lat_op_cl}.3. \endbox
\end{definition}

We shall see that $e_\aQ$ and $\iota_{\Pirr\aQ}$ are the `same maps', but we also have the maps $\iota_\rG$ for arbitrary $\rG$. We use the symbol `$e$' because $e_\aQ$ is an embedding which is never an inclusion, whereas $\iota_\rG$ is an inclusion so we use the symbol `$\iota$'. Likewise the surjective join-semilattice morphisms $\sigma_\aQ$ and $\sigma_{\Pirr\aQ}$ are essentially the same concepts. This will clarify the sense in which $e_\aQ$ and $\sigma_\aQ$ are `canonical' morphisms.

\begin{note}
\label{note:canon_embed_quo_alt}
One could also view $\rG$ as the join-semilattice morphism $\rG^\down : (\Pow \rG_t,\cap,\rG_t) \to (\Pow \rG_s,\cap,\rG_s)$ and take the unique (surjection,inclusion) factorisaton. The induced factor is then $(C(\rG),\cap,\rG_s)$ recalling that $C(\rG)$ consists of all sets $\rG^\down(Y)$ where $Y \subseteq \rG_t$ by Lemma \ref{lem:lat_op_cl}.3. All our subsequent results can be rephrased in terms of these factorisations via the bounded lattice isomorphisms from Lemma \ref{lem:lat_op_cl}.2:
\[
\begin{tabular}{llll}
$\theta_\rG : (C(\rG),\lor_{\latCl{\rG}},\rG^\down(\emptyset),\cap,\rG_s) \to (O(\rG),\cup,\emptyset,\land_{\latOp{\rG}},\rG[\rG_s])$
& where & $\theta_\rG(X) := \rG[X]$ 
&
$\theta_\rG^{\bf-1}(Y) := \rG^\down(Y)$ 
\\[0.5ex]
$\kappa_\rG : (C(\rG),\cap,\rG_s,\lor_{\latCl{\rG}},\rG^\down(\emptyset)) \to (O(\breve{\rG}),\cup,\emptyset,\land_{\latOp{\breve{\rG}}},\breve{\rG}[\rG_t])$
& where & $\kappa_\rG(X) := \overline{X}$
&
$\kappa_\rG^{\bf-1}(Y) := \overline{Y}$
\end{tabular}
\]
However, the very same isomorphisms allow us to suppress the closure semilattices. \endbox
\end{note}

\smallskip
Just as the morphisms $e_\aQ$ and $\sigma_\aQ$ collectively satisfy an adjoint relationship, so too do the morphisms $\iota_\rG$ and $\sigma_\rG$. In order to describe it, we first need explicit notation for a certain composite  isomorphism.

\begin{definition}
[Isomorphism representing the order-dual of $\Open\rG$]
\label{def:open_dual_iso}
\item
For each bipartite graph $\rG \subseteq \rG_s \times \rG_t$ we have the join-semilattice isomorphism:
\[
\begin{tabular}{lll}
$\partial_\rG : (\Open\rG)^{\pOp} \xto{(\theta_\rG^{\bf-1})^{\pOp}} 
(C(\rG),\cap,\rG_s) \xto{\kappa_\rG}
\Open\breve{\rG}$
&
with action $\partial_\rG(X) = \overline{\rG^\down(X)} = \breve{\rG}[\overline{X}]$
\\[0.3ex]
$\partial_\rG^{\bf-1} : \Open\breve{\rG} \xto{\kappa_{\rG}^{\bf-1}} 
(C(\rG),\cap,\rG_s) \xto{\theta_\rG^{\pOp}}
(\Open\rG)^{\pOp}$
&
with action $\partial_\rG^{\bf-1}(X) = \rG[\overline{X}]$ 
\end{tabular}
\]
where well-definedness follows by restricting Lemma \ref{lem:lat_op_cl}.2 i.e.\ the bounded lattice isomorphisms also described in Note \ref{note:canon_embed_quo_alt} directly above, and also De Morgan duality. \endbox
\end{definition}

The above isomorphisms are collectively closed under adjoints, and also collectively relate the components of the canonical natural isomorphism $rep : \Id_{\JSL_f} \To \Open \circ \Pirr$.

\begin{lemma}[Basic properties of the isomorphisms $\partial_\rG$]
\label{lem:open_dual_iso_adjoints}
\item
\begin{enumerate}
\item
For every bipartite graph $\rG$ we have:
\[
(\partial_\rG)_* = \partial_{\breve{\rG}}
\qquad\text{and}\qquad
(\partial_\rG^{\bf-1})_* = \partial_{\breve{\rG}}^{\bf-1}
\]

\item
For every finite join-semilattice $\aQ$ we have:
\[
\xymatrix@=15pt{
(\Open \circ \Pirr \aQ^{\pOp})^{\pOp} \ar[rr]^-{(rep_{\aQ^{\pOp}})_*} && \aQ
\\
\Open \circ \Pirr \aQ \ar[urr]_-{rep_\aQ^{\bf-1}} \ar[u]^-{\partial_{\Pirr\aQ^{\pOp}}^{\bf-1}} &
}
\qquad
\qquad
\xymatrix@=15pt{
(\Open \circ \Pirr \aQ^{\pOp})^{\pOp} \ar[rr]^-{\partial_{\Pirr\aQ^{\pOp}}} && \Open \circ \Pirr \aQ
\\
\aQ \ar[urr]_-{rep_\aQ} \ar[u]^-{rep_{\aQ^{\pOp}}^{\pOp}} &
}
\]
\end{enumerate}
\end{lemma}

\begin{proof}
\item
\begin{enumerate}
\item
By Lemma \ref{lem:adj_obs}.2 $(\partial_{\rG})_* = (\partial_{\rG}^{\bf-1})^{\pOp} : (\Open\breve{\rG})^{\pOp} \to \Open\rG$, which has the same type as $\partial_{\breve{\rG}}$ and acts in the same way. Similarly we have $(\partial_\rG^{\bf-1})_* = \partial_\rG^{\pOp}$ which has the same type as $\partial_{\breve{\rG}}^{\bf-1}$ and acts in the same way.

\item
We first verify the triangle on the right. Its typing is correct, so consider its action:
\[
\begin{tabular}{lll}
$\partial_{\Pirr\aQ^{\pOp}} \circ \rep_{\aQ^{\pOp}}^{\pOp}(q)$
&
$= \partial_{\Pirr\aQ^{\pOp}}(\rep_{\aQ^{\pOp}}(q))$
\\&
$= \partial_{(\Pirr\aQ)\spbreve}(\{ j \in J(\aQ) : j \nleq_\aQ q \})$
& via definition of $\rep_{\aQ^{\pOp}}$
\\&
$= \Pirr\aQ[\{ j \in J(\aQ) : j \leq_\aQ q \}]$
& by definition of $\partial_{(\Pirr\aQ)\spbreve}$
\\&
$= \{ m \in M(\aQ): \exists j \in J(\aQ).[j \leq_\aQ q \text{ and }  j \nleq_\aQ m \}$
\\&
$= \{ m \in M(\aQ) : \neg\forall j \in J(\aQ).[j \leq_\aQ q \implies j \leq_\aQ m  ] \}$
\\&
$= \{ m \in M(\aQ) : q \nleq_\aQ m \}$
\\&
$= rep_\aQ(q)$
\end{tabular}
\]
Thus $rep_\aQ = \partial_{\Pirr\aQ^{\pOp}} \circ rep_{\aQ^{\pOp}}^{\pOp}$ so by the standard law of composite inverses: 
\[
rep_\aQ^{\bf-1}
= (rep_{\aQ^{\pOp}}^{\pOp})^{\bf-1} \circ \partial_{\Pirr\aQ^{\pOp}}^{\bf-1}
= (rep_{\aQ^{\pOp}})_* \circ \partial_{\Pirr\aQ^{\pOp}}^{\bf-1}
\]
also using Lemma \ref{lem:adj_obs}.2.
\end{enumerate}
\end{proof}

Recalling that $\Pirr f_* = (\Pirr f)\spcheck = (\Pirr f)\spbreve$ for any $\JSL_f$-morphism $f$, the correspondence between adjoints in the other direction is captured precisely by the isomorphisms $\partial_\rG$. To see this, first recall that:
\[
\OD_j : \JSL_f^{op} \to \JSL_f
\qquad\text{and}\qquad
(-)\spcheck : \BiCliq^{op} \to \BiCliq
\]
are the self-duality functors on their respective categories.


\begin{theorem}[$\partial$ defines a natural isomorphism]
\label{thm:partial_nat_iso}
\item
The isomorphisms $\partial_\rG$ collectively define a natural isomorphism:
\[
\partial : \OD_j \circ \Open^{op} \To \Open \circ (-)\spcheck
\]
Consequently, for each $\BiCliq$-morphism $\rR : \rG \to \rH$ we have:
\[
(\Open\rR)_* = \partial_\rG^{\bf-1} \circ \Open\breve{\rR} \circ \partial_\rH
\]
\end{theorem}

\begin{proof}
We already know that each $\partial_\rG : (\Open\rG)^{\pOp} \to \Open\breve{\rG}$ is well-defined $\JSL_f$-isomorphism. Observe that the functors $\OD_j \circ \Open^{op}$ and $\Open \circ (-)\spcheck$ both have type $\BiCliq^{op} \to \JSL_f$. Then we need to verify that the following diagram commutes:
\[
\xymatrix@=15pt{
\OD_j \circ \Open^{op} \rH = (\Open\rH)^{\pOp} \ar[rr]^-{\partial_\rH} \ar[d]_-{(\Open\rR)_*} && \Open \circ (-)\spcheck(\rH) = \Open\breve{\rH} \ar[d]^-{\Open\breve{\rR}}
\\
\OD_j \circ \Open^{op}\rG = (\Open\rG)^{\pOp} \ar[rr]_-{\partial_\rG} &&  \Open \circ (-)\spcheck(\rG) = \Open\breve{\rG}
}
\]
or equivalently that $(\Open\rR)_* = \partial_\rG^{\bf-1} \circ \Open\breve{\rR} \circ \partial_\rH$, where the latter has action:
\[
\begin{tabular}{lll}
$\partial_\rG^{\bf-1} \circ \Open\breve{\rR} \circ \partial_\rH(Y)$
&
$= \partial_\rG^{\bf-1} \circ \Open\breve{\rR}(\breve{\rH}[\overline{Y}])$
& by definition
\\&
$= \partial_\rG^{\bf-1} \circ \breve{\rR}_+\spbreve[\breve{\rH}[\overline{Y}]]$
& by definition
\\&
$= \partial_\rG^{\bf-1} \circ \breve{\rH}; \breve{\rR}_+\spbreve[\overline{Y}]$
\\&
$= \partial_\rG^{\bf-1} \circ \breve{\rR}[\overline{Y}]$
& 
\\&
$= \rG[\overline{\breve{\rR}[\overline{Y}]}$
& by definition
\\&
$= \rG^\up \circ \neg_{\rG_s} \circ \breve{\rR}^\up \circ \neg_{\rH_t}(Y)$
\\&
$= \rG^\up \circ \rR^\down(Y)$
& by De Morgan duality
\end{tabular}
\]
Finally observe that this morphism was already described in Lemma \ref{lem:composite_adjoints}.4, where it was shown to have adjoint $\lambda Y \in O(\rG).\rR^\up \circ \rG^\down(Y) = \Open\rR$, so we are done.
\end{proof}

\begin{lemma}[Adjoint relationships involving $\iota_\rG$ and $\sigma_\rG$]
\label{lem:sigma_iota_g_adjoints}
\item
For every bipartite graph $\rG$ we have the commuting diagram:
\[
\xymatrix@=15pt{
\JPow \rG_s \ar[d]_{(\neg_{\rG_s})^{\bf-1}} \ar@{->>}[dr]^{\sigma_\rG}  \ar[rr]^{\rG^\up} && \JPow \rG_t
\\
(\JPow \rG_s)^{\pOp} \ar@{->>}[d]_{(\iota_{\breve{\rG}})_*}  & \Open\rG \ar[dd]^-{id_{\Open\rG}} \ar@{>->}[ur]^{\iota_\rG} & (\JPow \rG_t)^{\pOp} \ar[u]_-{\neg_{\rG_t}} 
\\
(\Open\breve{\rG})^{\pOp} \ar[dr]_{\partial_{\breve{\rG}}} & & (\Open\breve{\rG})^{\pOp} \ar[u]_{(\sigma_{\breve{\rG}})_*}
\\
& \Open\rG \ar@{>->}[ur]_{\partial_{\breve{\rG}}^{\bf-1}}
}
\]
Equivalently, we have the three equalities:
\[
\begin{tabular}{llllll}
$\mathrm{(a)}$ & $\rG^\up = \iota_\rG \circ \sigma_\rG$ &
$\mathrm{(b)}$ & $\sigma_\rG = \partial_{\breve{\rG}} \circ (\iota_{\breve{\rG}})_* \circ (\neg_{\rG_s})^{\bf-1}  $ &
$\mathrm{(c)}$ & $\iota_\rG = \neg_{\rG_t} \circ (\sigma_{\breve{\rG}})_* \circ \partial_{\breve{\rG}}^{\bf-1}$
\end{tabular}
\]
\end{lemma}

\begin{proof}
\item
\begin{enumerate}[(a)]
\item
This is the unique (surjection,inclusion) factorisation described in Definition \ref{def:bip_canon_quo_incl}.

\item
For any subset $X \subseteq \rG_s$ we have:
\[
\begin{tabular}{lll}
$\partial_{\breve{\rG}} \circ (\iota_{\breve{\rG}})_* \circ (\neg_{\rG_s})^{\bf-1}(X)$
&
$= \partial_{\breve{\rG}} \circ (\iota_{\breve{\rG}})_* (\overline{X})$
\\&
$= \partial_{\breve{\rG}}(\bigcup \{ Y \in O(\breve{\rG}) : \iota_{\breve{\rG}}(Y) \leq_{\JPow \rG_s} \overline{X} \} )$
\\&
$= \partial_{\breve{\rG}}(\bigcup \{ Y \in O(\breve{\rG}) : Y \subseteq \overline{X} \} )$
\\&
$= \partial_{\breve{\rG}}(\inte_{\breve{\rG}}(\overline{X}))$
& by Lemma \ref{lem:cl_inte_of_pirr}.1
\\&
$=  \rG^\up \circ \neg_{\rG_s} \circ \breve{\rG}^\up \circ \breve{\rG}^\down \circ \neg_{\rG_s}(X)$
& action of $\partial_{\breve{\rG}}$ and $\inte_{\breve{\rG}}$
\\&
$= \rG^\up \circ \rG^\down \circ \rG^\up(X)$
& by De Morgan duality
\\&
$= \rG^\up(X)$
& by $(\up\down\up)$
\\&
$= \sigma_\rG(X)$
\end{tabular}
\]

\item
Instantiate the previous statement by assigning $\rG \mapsto \breve{\rG}$ and take the adjoints of both sides to obtain:
\[
(\sigma_{\breve{\rG}})_* 
= (\neg_{\rG_t})^{\bf-1} \circ \iota_{\rG} \circ (\partial_{\breve{\rG}})_*
= (\neg_{\rG_t})^{\bf-1} \circ \iota_{\rG} \circ \partial_{\rG}
\]
using Lemma \ref{lem:open_dual_iso_adjoints}.1. The statement follows by post-composing with $\neg_{\rG_t}$ and pre-composing with $\partial_{\rG}^{\bf-1}$.

\end{enumerate}
\end{proof}

Finally we explain the relationship between $\iota_\rG$ and $e_\aQ$, and also $\sigma_\rG$ and $\sigma_\aQ$.

\begin{lemma}[The relationship between $\iota_\rG$ and $e_\aQ$]
\label{lem:relate_canon_embed}
For every finite join-semilattice $\aQ$ and bipartite graph $\rG$,
\[
\xymatrix@=15pt{
\Open\Pirr\aQ \ar@{>->}[rr]^{\iota_{\Pirr \aQ}} && \JPow M(\aQ)
\\
\aQ \ar[u]^{\rep_\aQ} \ar@{>->}[urr]_{e_\aQ}
}
\qquad\qquad
\xymatrix@=15pt{
\JPow M(\Open\rG) \ar[rr]^{\rM{\iota_\rG}^\up} && \JPow \rG_t
\\
\Open\rG \ar@{>->}[u]^{e_{\Open\rG}} \ar@{>->}[urr]_{\iota_\rG}
}
\]
where:
\begin{enumerate}
\item
$rep_\aQ(q) := \{ m \in M(\aQ) : q \nleq_\aQ m \}$ is a component of the natural isomorphism $rep : \Id_{\JSL_f} \To \Open \circ \Pirr$,
\item
the morphism $\rM{\iota_\rG}^\up$ is the canonical extension of $\iota_\rG$ as defined in Lemma \ref{lem:tight_mor_extn}.2, so that:
\[
\rM{\iota_\rG}(X,g_t) 
:\iff (\iota_\rG)_*(\overline{g_t}) \leq_{\Open \rG} X
\iff \inte_\rG(\overline{g_t}) \subseteq X
\]
\end{enumerate}
\end{lemma}

\begin{proof}
First observe that the inclusion $\iota_{\Pirr\aQ} : \Open\Pirr\aQ \hookto \JPow M(\aQ)$ is well-typed because $(\Pirr\aQ)_t = M(\aQ)$. Then the left triangle clearly commutes because the embedding $e_\aQ$ and the isomorphism $rep_\aQ$ act in the same way. The second triangle commutes via the tight extension Lemma \ref{lem:tight_mor_extn}.2. The final $\iff$ holds because $(\iota_\rG)_*(\overline{g_t}) = \inte_\rG(\overline{g_t})$ by Lemma \ref{lem:cl_inte_of_pirr}.1.
\end{proof}

\begin{lemma}
[The relationship between $\sigma_\rG$ and $\sigma_\aQ$]
For every finite join-semilattice $\aQ$ and bipartite graph $\rG$,
\[
\xymatrix@=15pt{
\Open\Pirr\aQ \ar[rr]^{rep_\aQ^{\bf-1}} && \aQ
\\
\JPow J(\aQ) \ar@{->>}[urr]_{\sigma_\aQ} \ar@{->>}[u]^{\sigma_{\Pirr\aQ}}
}
\qquad\qquad
\xymatrix@=15pt{
\JPow \rG_s \ar[rr]^{\rJ{\sigma_\rG}^\up} \ar@{->>}[rrd]_{\sigma_\rG} && \JPow J(\Open\rG) \ar@{->>}[d]^{\sigma_{\Open\rG}}
\\
&& \Open\rG
}
\]
where:
\begin{enumerate}
\item
$rep_\aQ^{\bf-1}(S) := \Land_\aQ M(\aQ) / S$ is a component of the natural isomorphism $rep^{\bf-1} : \Open \circ \Pirr \To \Id_{\JSL_f}$,
\item
the morphism $\rJ{\sigma_\rG}^\up$ is the canonical morphism compatible with $\sigma_\rG$ as defined in Lemma \ref{lem:tight_mor_extn}.1, so that:
\[
\rJ{e_\rG}(g_s,Y) 
:\iff Y \leq_{\Open \rG} \sigma_\rG(\{g_s\})
\iff Y \subseteq \rG[g_s]
\]
\end{enumerate}
\end{lemma}

\begin{proof}
Although the first triangle is easily shown by considering the action, it can also be formally derived using the above results:
\[
\begin{tabular}{lll}
$\sigma_\aQ$
&
$= (e_{\aQ^{\pOp}})_* \circ \neg_{J(\aQ)}^{\bf-1}$
& by Lemma \ref{lem:sigma_e_q_adjoints}.(b)
\\&
$= (rep_{\aQ^{\pOp}})_* \circ (\iota_{\Pirr\aQ^{\pOp}})_* \circ \neg_{J(\aQ)}^{\bf-1}$
& by Lemma \ref{lem:relate_canon_embed}
\\&
$= (rep_{\aQ^{\pOp}})_* \circ (\iota_{(\Pirr\aQ)\spbreve})_* \circ \neg_{J(\aQ)}^{\bf-1}$
\\&
$= (rep_{\aQ^{\pOp}})_* \circ \partial_{(\Pirr\aQ)\spbreve}^{\bf-1} \circ \sigma_{\Pirr\aQ} \circ \neg_{J(\aQ)} \circ \neg_{J(\aQ)}^{\bf-1}$
& by Lemma \ref{lem:sigma_iota_g_adjoints}.(b)
\\&
$= (rep_{\aQ^{\pOp}})_* \circ \partial_{(\Pirr\aQ)\spbreve}^{\bf-1} \circ \sigma_{\Pirr\aQ}$
\\&
$= rep_\aQ^{\bf-1} \circ \sigma_{\Pirr\aQ}$
& by Lemma \ref{lem:open_dual_iso_adjoints}.2
\end{tabular}
\]
Finally, the right triangle follows via the tight extension Lemma \ref{lem:tight_mor_extn}.1.
\end{proof}

\subsection{Monos, Epis and Isos}
\label{subsec:monos_epis_isos}

We begin with characterisations of $\BiCliq$'s monomorphisms and epimorphisms.

\begin{lemma}
\label{lem:bicliq_mono_epi_char}
Let $\rR : \rG \to \rH$ be any $\BiCliq$-morphism.
\begin{enumerate}
\item
The following statements are equivalent.
\begin{enumerate}
\item $\rR$ is monic.
\item $\Open\rR$ is injective.
\item Any of the four equivalent statements holds:
\[
\begin{tabular}{lcl}
$\cl_\rR = \cl_\rG$ && $\rG^\up \circ \cl_\rR = \rG^\up$
\\
$\cl_\rR \leq \cl_\rG$ && $\rG^\up \circ \cl_\rR \leq \rG^\up$.
\end{tabular}
\]
\end{enumerate}
\item The following statements are equivalent.
\begin{enumerate}
\item $\rR$ is epic.
\item $\Open\rR$ is surjective.
\item Any of the four equivalent statements holds:
\[
\begin{tabular}{lcl}
$\inte_\rR = \inte_\rH$
&& $\rH^\down \circ \inte_\rR = \rH^\down$
\\
$\inte_\rH \leq \inte_\rR$
&& $\rH^\down \leq \rH^\down \circ \inte_\rR$.
\end{tabular}
\]
They could be re-written in terms of $\cl_{\breve{\rR}}$, $\cl_{\breve{\rH}}$ and $(-)^\up$ using De Morgan duality.
\end{enumerate}

\end{enumerate}
\end{lemma}

\begin{proof}
\item
\begin{enumerate}
\item
\begin{itemize}
\item
$(a \iff b)$: Follows because $\Open$ is an equivalence functor by Theorem \ref{thm:bicliq_jirr_equivalent} and $\JSL_f$-monos are precisely the injective morphisms by Lemma \ref{lem:jsl_mono_epi_iso}.1.

\item
$(c \iff c)$: We always have $\cl_\rG \leq \cl_\rR$ because:
\[
\cl_\rR
= \rR^\down \circ \rR^\up 
= \rR^\down \circ \rR^\up \circ \cl_\rG
= \cl_\rR \circ \cl_\rG
\geq \cl_\rG
\]
using Lemma \ref{lem:bicliq_mor_char_max_witness} and that $\cl_\rR$ is extensive. We always have $\rG^\up \leq \rG^\up \circ \cl_\rR$ because $\cl_\rR$ is extensive and $\rG^\up$ is monotonic. Finally $\cl_\rR \leq \cl_\rG$ iff $\rG^\up \circ \cl_\rR \leq \rG^\up$ via the usual adjoint relationship.

\item
$(b \iff c)$: First assume (c). By Lemma \ref{lem:open_r_two_defs} it suffices to show that $\rR^\up \circ \rG^\down$ is injective on the restricted domain $O(\rG) \subseteq \Pow \rG_t$. Given any $Y_1$, $Y_2 \in O(\rG)$ then:
\[
\begin{tabular}{lllll}
& $\rR^\up \circ \rG^\down(Y_1)$ & $=$ & $\rR^\up \circ \rG^\down(Y_2)$
& by assumption
\\
$\implies$ & $\rR^\down \circ \rR^\up \circ \rG^\down(Y_1)$ & $=$ & $\rR^\down \circ \rR^\up \circ \rG^\down(Y_2)$
& apply function
\\
$\iff$ & $\cl_\rG \circ \rG^\down(Y_1)$ & $=$ & $\cl_\rG \circ \rG^\down(Y_2)$
& by assumption
\\
$\iff$ & $\rG^\down(Y_1)$ & $=$ & $\rG^\down(Y_2)$
& by $(\down\up\down)$
\\
$\implies$ & $\inte_\rG(Y_1)$ & $=$ & $\inte_\rG(Y_2)$
& apply function
\\
$\iff$ & $Y_1$ & $=$ & $Y_2$
& by openness
\end{tabular}
\]
so that $\Open\rR$ is injective. Conversely, assuming that $\Open\rR$ is injective it suffices to establish that $\rG^\up \circ \rR^\down \circ \rR^\up = \rG^\up$. By Lemma \ref{lem:open_r_two_defs} we know that $\rR^\up \circ \rG^\down$ restricts to an injection on $O(\rG) \subseteq \Pow\rG_t$. We have the equalities:
\[
\rR^\up 
= \rR^\up \circ \rR^\down \circ \rR^\up
= \rR^\up \circ \rG^\down \circ \rG^\up \circ \rR^\down \circ \rR^\up
\qquad
\rR^\up
= \rR^\up \circ \rG^\up \circ \rG^\down
\]
using $(\up\down\up)$ and also $\rR^\up \circ \cl_\rG = \rR^\up$ because $\rR$ is a $\BiCliq$-morphism. Putting them together yields:
\[
\rR^\up \circ \rG^\down (\rG^\up \circ \rR^\down \circ \rR^\up(X))
= \rR^\up \circ \rG^\down (\rG^\up(X))
\]
for any $X \subseteq \rG_s$. Then by injectivity and the fact that the $\rG$-image of any subset is $\rG$-open (Lemma \ref{lem:lat_op_cl}.3), we deduce that $\rG^\up \circ \rR^\down \circ \rR^\up = \rG^\up$.
\end{itemize}

\item
Since the $\JSL_f$-epis are precisely the surjective morphisms we have $(a \iff b)$. That $(c \iff c)$ follows from the previous argument and De Morgan duality. Furthermore $(b \iff c)$ follows from the previous statement and duality. For example, $\rR : \rG \to \rH$ is epic iff its dual $\breve{\rR} : \breve{\rH} \to \breve{\rG}$ is monic iff $\cl_{\breve{\rR}} = \cl_{\breve{\rH}}$ iff $\inte_\rR = \inte_\rH$ by De Morgan duality i.e.\ Lemma \ref{lem:cl_inte_of_pirr}.2.
\end{enumerate}
\end{proof}

The following result should be compared to Lemma \ref{lem:adj_obs}.2.

\begin{lemma}[Adjoints and inverses commute]
\label{lem:bicliq_iso_adj_inverse_commutes}
If $\rR : \rG \to \rH$ is a $\BiCliq$-isomorphism then $(\rR^{\bf-1})\spcheck = (\rR\spcheck)^{\bf-1}$.
\end{lemma}

\begin{proof}
The functoriality of $(-)\spcheck : \BiCliq^{op} \to \BiCliq$ informs us that $\rR\spcheck \fatsemi (\rR^{\bf-1})\spcheck = (\rR^{\bf-1} \fatsemi \rR)\spcheck = id_\rH\spcheck = id_{\rH\spcheck}$ and similarly $(\rR^{\bf-1})\spcheck \fatsemi \rR\spcheck = (\rR \fatsemi \rR^{\bf-1})\spcheck = id_\rG\spcheck = id_{\rG\spcheck}$.
\end{proof}

\smallskip

\begin{lemma}[$\BiCliq$-isomorphisms via components]
\label{lem:bicliq_iso_two_mor}
\item
Given $\BiCliq$-morphisms $\rR : \rG \to \rH$ and $\rS : \rH \to \rG$ then t.f.a.e.\
\begin{enumerate}
\item
$\rR$ is a $\BiCliq$-isomorphism with inverse $\rS$.
\item
We know either $\rR_- ; \rS = \rG$ or $\rG = \rR ; \rS_+\spbreve$.
We also know either $\rS_- ; \rR = \rH$ or  $\rH = \rS ; \rR_+\spbreve$.
\end{enumerate}
\end{lemma}

\begin{proof}
Their equivalence follows by considering the diagrams:
\[
\begin{tabular}{lll}
$\vcenter{\vbox{\xymatrix@=20pt{
\rG_t \ar[rr]^{\rR_+\spbreve} && \rH_t \ar[rr]^{\rS_+\spbreve} && \rG_t
\\
\rG_s \ar[u]^{\rG} \ar[rr]_{\rR_-} \ar[urr]^{\rR} && \rH_s \ar[u]^{\rH} \ar[rr]_{\rS_-} \ar[urr]^{\rS} && \rG_s \ar[u]_{\rG}
}}}$
&&
$\vcenter{\vbox{\xymatrix@=20pt{
\rH_t \ar[rr]^{\rS_+\spbreve} && \rG_t \ar[rr]^{\rR_+\spbreve} && \rH_t
\\
\rH_s \ar[u]^{\rH} \ar[rr]_{\rS_-} \ar[urr]^{\rS} && \rG_s \ar[u]^{\rG} \ar[rr]_{\rR_-} \ar[urr]^{\rR} && \rH_s \ar[u]_{\rH}
}}}$
\end{tabular}
\]
which commute because $\rR$ and $\rS$ are $\BiCliq$-morphisms. In the left diagram, any composite from $\rG_s$ to $\rG_t$ equals $\rR \fatsemi \rS$. Thus the latter equals $id_\rG = \rG$ iff $\rR_- ; \rS = \rG$ or alternatively $\rR ; \rS_+\spbreve$. Likewise in the right diagram.
\end{proof}

In the previous result, we assumed the candidate inverse was already known to be a $\Dep$-morphism. We then relied upon knowing some of the component relations. The following result avoids the component relations, and makes no assumptions concerning the candidate inverse.

\begin{lemma}[$\BiCliq$-isomorphisms via functional compositions]
\label{lem:bicliq_iso_from_rel}
\item
Given any $\BiCliq$-morphism $\rR : \rG \to \rH$ and relation $\rS \subseteq \rH_s \times \rG_t$, the following statements are equivalent.
\begin{enumerate}
\item
$\rR$ is a $\BiCliq$-isomorphism with inverse $\rS$.
\item
The following four equations hold:
\[
\begin{tabular}{lllll}
(a) & $\rR^\down \circ \rH^\up = \rG^\down \circ \rS^\up$ &
(c) & $\rR^\up \circ \rG^\down = \rH^\up \circ \rS^\down$ &
\\
(b) &  $\rS^\down \circ \rG^\up = \rH^\down \circ \rR^\up$ &
(d) & $\rS^\up \circ \rH^\down = \rG^\up \circ \rR^\down$.
\end{tabular}
\]
\end{enumerate}
\end{lemma}

\begin{proof}
\item
\begin{itemize}
\item
$(1 \To 2)$: Assume that $\rR$ is an isomorphism with inverse $\rS : \rH \to \rG$, so that $\rS$ is also a $\BiCliq$-isomorphism. We first show that (a) holds.
\[
\begin{tabular}{lll}
$\rR^\down \circ \rH^\up$
&
$= (\rR_- ; \rH)^\down \circ \rH^\up$
& associated component
\\&
$= \rR_-^\down \circ \rH^\down \circ \rH^\up$
& by $(;\,\down)$
\\&
$= \rR_-^\down \circ \cl_\rH$
& by definition
\\&
$= \rR_-^\down \circ \cl_\rS$
& $\rS$ is monic, and Lemma \ref{lem:bicliq_mono_epi_char}.1
\\&
$= \rR_-^\down \circ \rS^\down \circ \rS^\up$
& by definition
\\&
$= (\rR_- ; \rS)^\down \circ \rS^\up$
& by $(;\,\down)$
\\&
$= \rG^\down \circ \rS^\up$
& by Lemma \ref{lem:bicliq_iso_two_mor}
\end{tabular}
\]
Since $\rS$ is also a $\BiCliq$-isomorphism we obtain (a) for it, which is actually (b). Finally $\rR\spcheck$ and $\rS\spcheck$ are also $\BiCliq$-isomorphisms so we obtain (a) for each of them, and applying De Morgan duality  yields (c) and (d) respectively.

\item
$(2 \To 1)$: The calculation:
\[
\begin{tabular}{lll}
$\rS^\up \circ \cl_\rH$
&
$= \rS^\up \circ \rH^\down \circ \rH^\up$
\\&
$= \rG^\up \circ \rR^\down \circ \rH^\up$
& by (d)
\\&
$= \rG^\up \circ \rG^\down \circ \rS^\up$
& by (a)
\\&
$= \inte_\rG \circ \rS^\up$
\end{tabular}
\]
and Lemma \ref{lem:bicliq_mor_char_max_witness}.1 imply that $\rS$ defines a $\BiCliq$-morphism $\rH \to \rG$. Furthermore, consider:
\[
\begin{tabular}{lll}
\begin{tabular}{lll}
$(\rR \fatsemi \rS)^\up$
&
$= \rS^\up \circ \rH^\down \circ \rR^\up$
& by $(\fatsemi\up)$
\\&
$= \rG^\up \circ \rR^\up \circ \rR^\down$
& by (d)
\\&
$= \rG^\up \circ \cl_\rR$
\end{tabular}
&&
\begin{tabular}{lll}
$(\rR \fatsemi \rS)^\up$
&
$= \rS^\up \circ \rH^\down \circ \rR^\up$
& by $(\fatsemi\up)$
\\&
$= \rS^\up \circ \rS^\down \circ \rG^\up$
& by (b)
\\&
$= \inte_\rS \circ \rG^\up$
\end{tabular}
\end{tabular}
\]
Then since $\inte_\rS \circ \rG^\up \subseteq \rG^\up \subseteq \rG^\up \circ \cl_\rR$ by co-extensivity, extensitivity and monotonicity, we deduce that $(\rR \fatsemi \rS)^\up = \rG^\up$ and hence $\rR \fatsemi \rS = id_\rG$. Finally consider:
\[
\begin{tabular}{lll}
\begin{tabular}{lll}
$(\rS \fatsemi \rR)^\up$
&
$= \rR^\up \circ \rG^\down \circ \rS^\up$
& by $(\fatsemi\up)$
\\&
$= \rH^\up \circ \rS^\down \circ \rS^\up$
& by (c)
\\&
$= \rH^\up \circ \cl_\rS$
\end{tabular}
&&
\begin{tabular}{lll}
$(\rS \fatsemi \rR)^\up$
&
$= \rR^\up \circ \rG^\down \circ \rS^\up$
& by $(\fatsemi\up)$
\\&
$= \rR^\up \circ \rR^\down \circ \rH^\up$
& by (a)
\\&
$= \inte_\rR \circ \rH^\up$
\end{tabular}
\end{tabular}
\]
which by analogous reasoning implies that $\rS \fatsemi \rR = id_\rH$.

\end{itemize}
\end{proof}

A useful class of isomorphisms arises from bijections.

\begin{definition}[Bipartite $\Dep$-isomorphisms]
  \label{def:bipartite_iso}
A \emph{bipartite $\Dep$-isomorphism} $\rR : \rG \to \rH$ is a $\Dep$-morphism witnessed by bijections i.e.\ $\rG; f = \rR = \rH ; \breve{g}$ for bijective functions $f : \rG_s \to \rH_s$ and $g : \rH_t \to \rG_t$. \endbox
\end{definition}

\smallskip

\begin{lemma}
  \label{lem:bipartite_isos_are_isos}
  Every bipartite $\Dep$-isomorphism is a $\Dep$-isomorphism.
\end{lemma}

\begin{proof}
  If $\rR : \rG \to \rH$ has bijective witnesses $(\lambda, \rho)$, it has an inverse $\rR^{-1} : \rH \to \rG$ via witnesses $(\lambda^{-1}, \rho^{-1})$.
\end{proof}

\begin{note}[$\Dep$-objects as bipartite graphs]
  \item
  \begin{enumerate}
    \item 
    Any relation $\rG \subseteq \rG_s \times \rG_t$ can be viewed as an undirected bipartite graph with vertices $\rG_s + \rG_t$ (disjoint union), edges $\rE(e_1(g_s), e_2(g_t)) :\iff \rG(g_s, g_t)$ (and no others) and bipartition $(e_1[\rG_s], e_2[\rG_t])$. The latter is a pair and is sometimes called an \emph{ordered bipartition}. The dual $\Dep$-object $\breve{\rG} \subseteq \rG_t \times \rG_s$ yields the same bipartite graph modulo-isomorphism, yet  has a distinct ordered bipartition $(e_2[\rG_t], e_1[\rG_s])$ unless $\rG_s = \rG_t = \emptyset$.
    \item
    A bipartite $\Dep$-isomorphism $\rR : \rG \to \rH$ has bijective witnesses $(f, g)$. They induce a bipartite graph isomorphism $f + \breve{g}$ between the underlying undirected bipartite graphs. Since the isomorphisms must respect the bipartitions, not every bipartite graph isomorphism arises in this way. \endbox
  \end{enumerate}
\end{note}

\begin{example}[A bipartite $\Dep$-isomorphism]
The relation $\rG \subseteq X \times Y$ below on the left,
\[
\xymatrix@=15pt{
y_1 & y_2
\\
x_1 \ar[u] & x_2 \ar[u] \ar[ul]
}
\qquad\qquad
\xymatrix@=15pt{
y_1 & y_2
&& 0 \ar@{..>}[ll]^{\rR_+} \ar@/_10pt/@{..>}[lll]_{\rR_+}
& 1 \ar@/_10pt/@{..>}[lll]_{\rR_+} 
\\
x_1 \ar[u] \ar@{..>}@/_10pt/[rrr]_{\rR_-} &
x_2 \ar[u] \ar[ul] \ar@{..>}[rr]^{\rR_-} \ar@{..>}@/_10pt/[rrr]_{\rR_-}
&&
1 \ar[u]
& 2 \ar[u] \ar[ul]
}
\]
is bipartite $\Dep$-isomorphic to $\Pirr\three$ where $\three$ is the $3$-chain. We've depicted the associated components of the $\Dep$-isomorphism, which are not functional. \endbox
\end{example}

\begin{example}[A non-bipartite $\Dep$-isomorphism]
  If $\rG = \{ x_1, x_2 \} \times \{ y \}$, the canonical $\BiCliq$-isomorphism $red_\rG : \rG \to \Pirr\Open\rG$ has associated components:
  
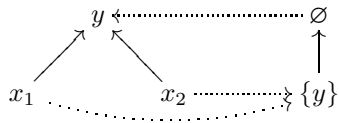
\[
  \xymatrix@=15pt{
  & y \ar@{<..}[rrr] & && \emptyset
  \\
  x_1 \ar@/_10pt/@{..>}[rrrr] \ar[ur] && x_2 \ar@{..>}[rr] \ar[ul] && \{y\} \ar[u]
  }
  \]
  It cannot be a bipartite $\Dep$-iso because $2 = |\rG_s| > |J(\Open\rG)| = 1$. \endbox
\end{example}

Nevertheless we have the following clarifying result.

\begin{lemma}[Bipartite graph isomorphism by restriction]
\label{lem:bip_restrict_to_reduced}
Every $\rG$ has a domain/codomain restriction which is bipartite $\Dep$-isomorphic to $\Pirr\Open\rG$.
\end{lemma}

\begin{proof}
  By Lemma \ref{lem:lat_op_cl}.3 we know that:
  \[
  J(\Open\rG) \subseteq \{ \rG[g_s] : g_s \in \rG_s \}
  \qquad
  M(\Open\rG) \subseteq \{ \inte_\rG(\overline{g_t}) : g_t \in \rG_t \}.
  \]
  Then for each $X \in J(\Open\rG)$ choose $j_X \in \rG_s$ such that $\rG[j_X] = X$, and for each $Y \in M(\Open\rG)$ choose $m_Y \in \rG_t$ such that $\inte_\rG(\overline{m_Y}) = Y$. These chosen elements are necessarily distinct e.g.\ if $X \neq Y$ then $j_X \neq j_Y$, and induce both a restriction $\rG'$ of $\rG$ and a pair of relations $(\rR_l,\rR_r)$ as follows:
  \[
  J := \{ j_X : X \in J(\Open\rG) \}
  \qquad
  M := \{ m_Y : Y \in M(\Open\rG) \}
  \qquad
  \rG' := \rG \;\cap \; J \times M
  \]
  \[
  \begin{tabular}{l}
  $\rR_l := \{ (j_X,X) : X \in J(\Open\rG) \} \subseteq \rG'_s \times J(\Open\rG)$
  \\[0.5ex]
  $\rR_r := \{ (Y,m_Y) : Y \in M(\Open\rG) \} \subseteq M(\Open\rG) \times \rG'_t$
  \end{tabular}
  \]
  We now establish that $(\rR_l,\rR_r) : \rG' \to \Pirr\Open\rG$ is a bipartite $\Dep$-isomorphism. Clearly $\rR_l$ and $\rR_r$ are bijective functions. Further recall that $\Pirr\Open\rG := \; \nsubseteq \; \subseteq J(\Open\rG) \times M(\Open\rG)$ i.e.\ the domain/codomain restriction of the binary relation $\nsubseteq$ on $\Pow \rG_t$. Then: 
  \[
  \begin{tabular}{lll}
  $\rR_l ; \; \nsubseteq (j_X,Y)$
  &
  $\iff \rG[j_X] \;\nsubseteq\; \inte_\rG(\overline{m_Y})$
  & $\rR_l$ functional
  \\&
  $\iff j_X \nin \rG^\down \circ \rG^\up \circ \rG^\down(\overline{m_Y})$
  & by usual adjunction
  \\&
  $\iff j_X \nin \rG^\down(\overline{m_Y})$
  & by $(\down\up\down)$
  \\&
  $\iff \rG(j_X,m_Y)$
  & by definition of $(-)^\down$
  \\&
  $\iff \rG'(j_X,m_Y)$
  & by restriction
  \\
  \\
  $\rG' ; \rR_r\spbreve (j_X,Y)$
  &
  $\iff \exists g_t \in \rG'_t.( \rG(j_X,g_t) \land \rR_r(Y,g_t) )$
  \\&
  $\iff \rG(j_X,m_Y)$
  & $\rR_r$ functional
  \\&
  $\iff \rG'(j_X,m_Y)$
  & by restriction
  \end{tabular}
  \]
  Thus $\rR_l ; \; \nsubseteq \; = \rG' = \rG' ; \rR_+\spbreve$ as required.
\end{proof}

Next we describe certain degenerate yet useful isomorphisms.

\begin{lemma}[Isomorphisms via join/meet generators]
\label{lem:degen_isos}
Given any finite join-semilattice $\aS$ and any subsets  $J(\aS) \subseteq X \subseteq Q$ and $M(\aS) \subseteq Y \subseteq Q$, consider the domain/codomain restriction $\rG := \; \nleq_\aS |_{X \times Y}$. Then $\rR$ and $\rS$ are mutually inverse $\Dep$-isomorphisms:
\[
\begin{tabular}{llllll}
$\rR : \rG \to \Pirr\aS$ 
&& $\rR := \, \nleq_\aS |_{X \times M(\aS)}$
& $\rR_- (x, j) :\iff j \leq_\aS x$
& $\rR_+ (m, y) :\iff m \leq_\aS y$
\\
$\rS : \Pirr\aS \to \rG$  
&& $\rS := \, \nleq_\aS |_{J(\aS) \times Y}$
& $\rS_-(j, x) :\iff x \leq_\aS j$
& $\rS_+ (y, m) :\iff y \leq_\aS m$.
\\ 
\end{tabular}
\]
\end{lemma}

\begin{proof}
We first verify that the following diagram commutes:
\[
\xymatrix@=15pt{
Y \ar[rr]^{\rR_+\spbreve} && M(\aS) \ar[rr]^{\rS_+\spbreve} && Y
\\
X \ar[rr]_{\rR_-} \ar[u]^{\nleq_\aS} && J(\aS) \ar[u]_{\nleq_\aS} \ar[rr]_{\rS_-} && X \ar[u]_{\nleq_\aS}
}
\]
That is, $\rR_- ; \nleq_\aS = \rR$ and $\nleq_\aS ; \rR_+\spbreve = \rR$ follow by Lemma \ref{lem:std_order_theory}.7 and $M(\aS) \subseteq Y$. Similarly, $\rS_- ; \nleq_\aS \; = \rS$ and $\nleq_\aS ; \rS_+\spbreve = \rS$ follow by Lemma \ref{lem:std_order_theory}.7 and $J(\aS) \subseteq X$. Thus both $\rR$ and $\rS$ are well-defined $\BiCliq$-morphisms and the reader can verify that $(\rR_-,\rR_+)$ and $(\rS_-,\rS_+)$ are the associated components. Finally we verify they are mutually inverse:
\[
\begin{tabular}{lll}
$\rR \fatsemi \rS(x,y)$
& 
$\iff \rR ; \rS_+\spbreve(x,y)$
& by Lemma \ref{lem:bicliq_func_comp}
\\&
$\iff \exists m \in M(\aS).( x \nleq_\aS m \text{ and } y \leq_\aS m )$
\\&
$\iff \neg \forall m \in M(\aS).( y \leq_\aS m \To x \leq_\aS m )$
\\&
$\iff x \nleq_\aS y$
\\
\\
$\rS \fatsemi \rR(j,m)$
&
$\iff \rS ; \rR_+\spbreve(j,m)$
& by Lemma \ref{lem:bicliq_func_comp}
\\&
$\iff \exists y \in Y.( j \nleq_\aS y \text{ and } y \leq_\aS m)$
\\&
$\iff \neg\forall y \in Y.( y \leq_\aS m \To j \leq_\aS y)$
\\&
$\iff \neg(j \leq_\aS m)$
& since $m \in Y$
\\&
$\iff j \nleq_\aS m$.
\end{tabular}
\]
\end{proof}

\begin{example}
For each finite join-semilattice $\aS$ we have $\nleq_\aS \; \cong \; \nleq_\aS |_{J(\aS), M(\aS)} = \Pirr\aS$ inside $\Dep$.
\endbox
\end{example}

The previous Lemma permits us to extend $\Pirr\aS$'s domain/codomain whilst remaining isomorphic. Similarly we may extend $\Pirr f$ up to isomorphism.

\begin{lemma}
Let $f : \aS \to \aT$ be a $\JSL_f$-morphism and fix join/meet-generating subsets $(X_S, Y_S)$ and $(X_T, Y_T)$.
Then using isomorphisms from Lemma \ref{lem:degen_isos}:
\[
  \nleq_{\aS} \; : \; \nleq_{\aS} |_{X_S \times Y_S} \to \Pirr\aS
  \qquad
  \nleq_{\aT} \; : \; \Pirr\aT \to \; \nleq_\aT |_{X_T \times Y_T}
\]
 we have the following commuting $\Dep$-diagram where $\rR (s, t) :\iff f(s) \nleq_\aS t$:
\[
  \xymatrix@=15pt{
    \nleq_\aS |_{X_T \times Y_T} \ar@{<-}[rrr]^-{\nleq_{\aS} |_{J(\aS) \times Y_T}} &&& \Pirr\aT
    \\
    \nleq_\aS |_{X_S \times Y_S} \ar[u]^{\rR} \ar[rrr]_-{\nleq_{\aS} |_{X_S \times M(\aS)}} &&& \Pirr\aS \ar[u]^{\Pirr f}
  }  
\]
\end{lemma}


\begin{proof}
By Lemma \ref{lem:degen_isos} we know $(\nleq_\aS |_{X_S \times M(\aS)})_- = \; \geq_\aS |_{X_S \times J(\aS)}$ and also  $(\nleq_\aT |_{J(\aT) \times Y_T})_+ = \; (\geq_\aS |_{M(\aT) \times Y_T})\spbreve$, so composing the compatible arrows in $\Dep$ yields:
\[
  \geq_\aS |_{X_S \times J(\aS)} ; \; \Pirr f ; \; \geq_\aS |_{M(\aT) \times Y_T} (x_s, y_t)
  \iff
  \exists j \in J(\aS). \exists m \in M(\aS). ( j \leq_\aS x_s \,\land\, f(j) \nleq_\aT m \,\land\, y_t \leq_\aT m)
  \qquad
  (\star)
\]
Assuming ($\star$) we'll show $\rR(x_s, y_t)$. If $f(x_s) \leq_\aS m$ then $f(j) \leq_\aS f(x_s)$ by monotonicity, yielding contradition $f(j) \leq_\aS m$, so we know $f(x_s) \nleq_\aS m$. Thus $f(x_s) \nleq_\aS y_t$ for otherwise we obtain the contradiction $f(x_s) \leq_\aS y_t \leq_\aS m$.

Conversely, if $f(x_s) \nleq_\aT y_t$ then $x_s \neq_\aS \bot_\aS$ for otherwise $f(\bot_\aS) = \bot_\aS \nleq_\aS y_t$ is a contradiction. So some $j \in J(\aS)$ satisfies $j \leq_\aS x_s$. If every such join-irreducible satisfied $f(j) \leq_\aT y_t$, then since $f$ preserves joins we'd infer the contradiction $f(x_s) \leq_\aT y_t$. Thus $f(j) \nleq_\aT y_t$ for some $j \leq_\aS x_s$. Finally since $f(j) \nleq_\aT y_t$ we know $y_t \neq_\aS \top_\aT$ hence $\exists m \in M(\aT)$ with $y_t \leq_\aT m$. If $f(j) \leq_\aT m$ for every such meet-irreducible, we'd infer the contradiction $f(j) \leq_\aT m$ by the definition of meets.
\end{proof}

\section{Tensors and tight tensors}
\label{sec:tensors}

\subsection{Hom-functors, irreducible morphisms and the tensor product}

We now investigate the join-semilattice of morphisms $\JSL_f[\aQ,\aR]$. These are the morphisms $\JSL_f(\aQ,\aR)$ equipped with the pointwise-join and the constantly $\bot_\aR$ map. It is extended to a functor in the standard way. We describe its meet-irreducible elements, and in some cases its join-irreducible elements. This is achieved by considering certain special morphisms. We then define the tensor product of finite join-semilattices as a composite functor, whose action on objects is $\aQ \tenp \aR := (\JSL_f[\aQ,\aR^{\pOp}])^{\pOp}$. Bimorphisms are introduced and some basic properties of the tensor product are proved.

\smallskip
However, we leave the proof of the \emph{universality} of the tensor product until the next subsection. We do this because one can prove it in an elegant way using $\BiCliq$. There is a pre-existing inductively defined notion of `bi-ideal' which has been used to define the tensor product of finite join-semilattices \cite{GratzerTensorSemilattices2005}. A bi-ideal over a pair of finite join-semilattices $(\aQ,\aR)$ is precisely the same thing as the relative complement of a $\BiCliq$-morphism of type $\nleq_\aQ \subseteq Q \times Q \to \; \ngeq_\aR \; \subseteq R \times R$.

\begin{definition}[Internal hom-functor]
\label{def:jsl_internal_hom}
\item
For any pair of finite join-semilattices $(\aQ,\aR)$ recall by Definition \ref{def:hom_functor_jsl} that $\JSL_f[\aQ,\aR]$ is the join-semilattice of join-semilattice morphisms $\JSL_f(\aQ,\aR)$. This  extends to a functor $\JSL_f[-,-] : \JSL_f^{op} \times \JSL_f \to \JSL_f$ as follows:
\[
\dfrac{f : \aQ_2 \to \aQ_1 \quad g : \aR_1 \to \aR_2}
{\JSL_f[f^{op},g] := \lambda h. g \circ h \circ f    : \JSL_f[\aQ_1,\aR_1] \to \JSL_f[\aQ_2,\aR_2] }
\]
We refer to this functor as the \emph{internal hom-functor}. \endbox
\end{definition}

\begin{lemma}
$\JSL_f[-,-] : \JSL_f^{op} \times \JSL_f \to \JSL_f$ is a well-defined functor.
\end{lemma}

\begin{proof}
It suffices to show its action is well-defined, since for general categorical reasons we have the well-defined functor $\JSL_f(-,-) : \JSL_f^{op} \times \JSL_f \to \Set$ with the same underlying action as $\JSL_f[-,-]$. Each $\JSL_f[\aQ,\aR]$ is a well-defined finite join-semilattice by Lemma \ref{def:hom_functor_jsl}. Given $f : \aQ_2 \to \aQ_1$ and $g : \aR_1 \to \aR_2$ it remains to show that the action of $\JSL_f[f^{op},g]$ preserves the pointwise join-structure on $\JSL_f(\aQ_1,\aR_1)$.
\[
g \circ \bot_{\JSL_f[\aQ,\aR]} \circ f
= \lambda q_2 \in Q_2. g(\bot_{\JSL_f[\aQ,\aR]}(f(q_2))
= \lambda q_2 \in Q_2. \bot_{\aR_2}
= \bot_{\JSL_f[\aQ_2,\aR_2]}
\]
\[
\begin{tabular}{lll}
$g \circ (h_1 \lor_{\JSL_f[\aQ,\aR]} h_2) \circ f$
&
$= \lambda q_2 \in Q_2. g(h_1 \lor_{\JSL_f[\aQ_1,\aR_1]} h_2(f(q_2))$
\\&
$= \lambda q_2 \in Q_2. g( h_1(f(q_2)) \lor_{\aR_1} h_2(f(q_2) )$
\\&
$= \lambda q_2 \in Q_2. g( h_1(f(q_2))) \lor_{\aR_2} g(h_2(f(q_2))$
\\&
$= (g \circ h_1 \circ f) \lor_{\JSL_f[\aQ_2,\aR_2]} (g \circ h_2 \circ f)$
\end{tabular}
\]
using the fact that $g$ is a join-semilattice morphism.
\end{proof}

\begin{lemma}
\label{lem:hom_meet_bound}
Although $\land_{\JSL_f[\aQ,\aR]}$ needn't be pointwise, the ordering $\leq_{\JSL_f[\aQ,\aR]}$ is. Morever:
\[
\Land_{\JSL_f[\aQ,\aR]} \{ f_i : i \in I\} \, (q) \, \leq_\aR \, \Land_\aR \{ f_i(q) : i \in I \}
\]
for any morphisms $(f_i : \aQ \to \aR)_{i \in I}$ and $q \in Q$.
\end{lemma}

\begin{proof}
We have the injective join-semilattice morphism:
\[
e : \JSL_f[\aQ,\aR] \monoto \aR^Q
\qquad
e(f) := (f(q))_{q \in Q}
\]
because the join in $\aS := \JSL_f[\aQ,\aR]$ is constructed pointwise. Then $\leq_\aS$ is the pointwise-ordering because injective join-semilattice morphisms are order-embeddings. Furthermore since $e$ is monotonic we have $e(\Land_\aS \{ f_i : i \in I\}) 
\leq_{\aR^Q} \Land_{\aR^Q} e(f_i)$ which is precisely the claim above. Finally, Example \ref{ex:hom_meet_not_pw} below provides a specific example where the meet is not constructed pointwise.
\end{proof}

\begin{example}
\label{ex:hom_meet_not_pw}
The meet in $\JSL_f [\aQ , \aR]$ needn't be pointwise.
Let $\aQ = \aR = M_3$ and consider the two endomorphisms:
\[
\begin{tabular}{lll}
$\vcenter{\vbox{\xymatrix@=10pt{
& \top 
\\
x_1 \ar@{-}[ur]  & x_2 \ar@{-}[u] & x_3 \ar@{-}[ul]
\\
& \bot \ar@{-}[ul] \ar@{-}[u] \ar@{-}[ur]
}}}$
&
$f_1(q) := \begin{cases} \bot & \text{if $q = x_1$} \\ q & \text{otherwise} \end{cases}$
&
$f_3(q) := \begin{cases} \bot & \text{if $q = x_3$} \\ q & \text{otherwise} \end{cases}$
\end{tabular}
\]
Let $g := f_1 \land_{\JSL_f[\aQ,\aQ]} f_3$ denote their meet. By Lemma \ref{lem:hom_meet_bound} we have $g(x_1) = g(x_3) = \bot_\aQ$, and hence $g(\top_\aQ) = g(x_1 \lor_\aQ x_3) = g(x_1) \lor_\aQ g(x_3) = \bot_\aQ$. Thus $g$ is the constantly bottom map, so that $g(x_2) <_\aQ x_2 = f_1(x_2) \land_\aQ f_2(x_2)$. \endbox
\end{example}

\smallskip
For any pair of finite join-semilattices $(\aQ,\aR)$ we are going to define two types of special morphisms, where the first meet-generate $\JSL_f[\aQ,\aR]$ generally, and second  join-generate this join-semilattice as long as $\aQ$ or $\aR$ is distributive. A subset of the former will later provide the join-irreducible elements of the \emph{tensor product}, whereas a subset of the latter will induce the join-irreducible elements of the \emph{tight tensor product}. First recall the following basic constructions from Definition \ref{def:elem_ideal_mor} and Lemma \ref{lem:jsl_elem_iso}.

\begin{enumerate}
\item
We have the join-semilattice $\two = (\{0,1\},\lor_\two,0)$ where $\lor_\two$ is the boolean function OR, or equivalently $max(b_1,b_2)$. There is a unique join-semilattice isomorphism of type $\swap : \two^{\pOp} \to \two$. It flips the bit, so that $\swap(b) = 1 - b$.
\item
We have the join-semilattice isomorphisms:
\[
\begin{tabular}{rl}
$\elem{\aQ}{-} : \aQ \to \jslElem{\aQ} = \JSL_f[\two,\aQ]$
&
$\elem{\aQ}{q} := \lambda b \in \{0,1\}. b \;?\; q : \bot_\aQ$
\\[1ex]
$\ideal{\aQ}{-} : \aQ^{\pOp} \to \jslIdeal{\aQ} = \JSL_f[\aQ,\two]$
&
$\ideal{\aQ}{q_0} := \lambda q \in Q. (q \leq_\aQ q_0) \;?\; 0 : 1$
\end{tabular}
\]
\end{enumerate}

\noindent
For what follows, it will be helpful to define similar structures involving a three element chain.

\begin{definition}[Elements and ideals relative to the three-chain]
\label{def:trelem_trideal}
\item
\begin{enumerate}
\item
Define the join-semilattice $\three = (\{0,1,2\},\lor_\three,0)$ where $x_1 \lor_\three x_2 := max(x_1,x_2)$.  There is a unique join-semilattice isomorphism with typing $\rot : \three^{\pOp} \to \three$. It rotates around $1$, so that $\rot(x) = 2 - x$. It is self-adjoint, as is its inverse $\rot^{\bf-1} : \three \to \three^{\pOp}$.
\item
For each finite join-semilattice $\aQ$ define the finite join-semilattice:
\[
\jslTrelem{\aQ} 
= (\{ f \in \JSL_f(\three,\aQ) : f(\top_\three) = \top_\aQ \},\lor_{\jslTrelem{\aQ}},\bot_{\jslTrelem{\aQ}})
\]
where $\lor_{\jslTrelem{\aQ}}$ constructs the pointwise-join in $\aQ$, and $\bot_{\jslTrelem{\aQ}} = \lambda x. (x = 2) \; ? \; \top_\aQ : \bot_\aQ$. That is, $\jslTrelem{\aQ}$ inherits the join structure from $\JSL_f[\three,\aQ]$ but has a different bottom element. There is an associated join-semilattice isomorphism:
\[
\trelem{\aQ}{-} : \aQ \to \jslTrelem{\aQ}
\qquad
\trelem{\aQ}{q} :=  \lambda x \in \{0,1,2\}.
\begin{cases}
\bot_\aQ & \text{if $x = 0$}
\\
q & \text{if $x = 1$}
\\
\top_\aQ & \text{if $x = 2$}
\end{cases}
\]
\item
For each finite join-semilattice $\aQ$, define the finite join-semilattice:
\[
\jslTrideal{\aQ} 
:= (\{ f \in \JSL_f(\aQ,\three) : f_*(\bot_\three) = \bot_\aQ \},\lor_{\jslTrideal{\aQ}},\bot_{\jslTrideal{\aQ}})
\]
where $\lor_{\jslTrideal{\aQ}}$ is the pointwise-join inside $\three$, and $\bot_{\jslTrideal{\aQ}} = \lambda q \in Q.(q = \bot_\aQ)\;?\;0 : 1$. Thus $\jslTrideal{\aQ}$ inherits the join-structure from $\JSL_f[\aQ,\three]$ yet has a different bottom element.  There is an associated join-semilattice isomorphism:
\[
\trideal{\aQ}{-} : \aQ^{\pOp} \to \jslTrideal{\aQ}
\qquad
\trideal{\aQ}{q_0} := \lambda q \in Q.
\begin{cases}
0 & \text{if $q = \bot_\aQ$}
\\
1 & \text{if $\bot_\aQ <_\aQ q \leq_\aQ q_0$}
\\
2 & \text{if $q \nleq_\aQ q_0$}
\end{cases}
\]
which provides a concrete description of this join-semilattice. \endbox
\end{enumerate}

\end{definition}

\begin{lemma}
\label{lem:trelem_trideal_well_def}
\item
\begin{enumerate}
\item
$\trelem{\aQ}{-} : \aQ \to \jslTrelem{\aQ}$ is a well-defined join-semilattice isomorphism.
\item
$\trideal{\aQ}{-} : \aQ^{\pOp} \to \jslTrideal{\aQ}$ is a well-defined join-semilattice isomorphism.
\item
For each $q \in Q$ we have:
\[
(\trelem{\aQ}{q})_* = \rot^{\bf-1} \circ \trideal{\aQ^{\pOp}}{q}
\qquad
(\trideal{\aQ}{q})_* = \trelem{\aQ^{\pOp}}{q} \circ \rot
\]
\end{enumerate}
\end{lemma}

\begin{proof}
\item
\begin{enumerate}
\item
$\jslTrelem{\aQ}$ is a well-defined join-semilattice because the top-preserving morphisms $f : \three \to \aQ$ are closed under pointwise joins, and there is a least such morphism $\bot_{\jslTrelem{\aQ}} = \trelem{\aQ}{\bot_\aQ} = \lambda x.(x = 2)\;?\; \top_\aQ : \bot_\aQ$. Since the only parameter is the value of $f(1)$ and this may be freely chosen, it follows that $\trelem{\aQ}{-}$ is a bijection. We've already observed that the bottom is preserved, and clearly $\trelem{\aQ}{q_1 \lor_\aQ q_2}$ is the pointwise binary join of $(\trelem{\aQ}{q_i})_{i = 1,2}$, and thus also their binary join inside $\jslTrelem{\aQ}$.

\item
We first show that $\jslTrideal{\aQ}$ is a well-defined join-semilattice. Fixing any morphism $f : \aQ \to \three$, then $f_*(\bot_\three)  = \bot_\aQ$ iff $f_* \circ \rot^{\bf-1} : \three \to \aQ^{\pOp}$ preserves the top element. Using the bijective correspondence between adjoints and the fact that $\rot$ is self-adjoint, it follows that the elements of $\jslTrideal{\aQ}$ are precisely those of the form $\rot \circ (\trelem{\aQ^{\pOp}}{q_0})_*$ where $q_0 \in Q$. Since $\jslTrelem{\aQ^{\pOp}}$ is closed under pointwise binary joins, so are their adjoints by Lemma \ref{lem:jsl_mor_iso}, as is their post-composition with the fixed morphism $\rot$ by applying the functor $\JSL_f[\rot^{op},-]$. We have a bottom element because $\jslTrelem{\aQ^{\pOp}}$ has one. Regarding its  description, we first compute $\rot \circ (\trelem{\aQ^{\pOp}}{q_0})_*$ in general.
\[
\begin{tabular}{llll}
$(\trelem{\aQ^{\pOp}}{q_0})_*$
&
$= \lambda q \in Q. \Lor_\three \{ x \in \{0,1,2\} : \trelem{\aQ^{\pOp}}{q_0}(x) \leq_{\aQ^{\pOp}} q \}$
\\&
$= \lambda q \in Q. \Lor_\three \{ x \in \{1,2\} : q \leq_\aQ \trelem{\aQ^{\pOp}}{q_0}(x)  \}$
\\[1ex]&
$= \lambda q \in Q.
\begin{cases}
$2$ & \text{if $q \leq_\aQ \trelem{\aQ^{\pOp}}{q_0}(2) = \top_{\aQ^{\pOp}} = \bot_\aQ$}
\\
$1$ & \text{if $\bot_\aQ <_\aQ q \leq_\aQ \trelem{\aQ^{\pOp}}{q_0}(1) = q_0$}
\\
$0$ & \text{if $q \nleq_\aQ q_0$}
\end{cases}$
\end{tabular}
\]
and hence:
\[
\rot \circ (\trelem{\aQ^{\pOp}}{q_0})_* =
\begin{cases}
$0$ & \text{if $q = \bot_\aQ$}
\\
$1$ & \text{if $\bot_\aQ <_\aQ q \leq_\aQ q_0$}
\\
$2$ & \text{if $q \nleq_\aQ q_0$}
\end{cases}
\]
Thus we have the bottom element $\rot \circ \trelem{\aQ^{\pOp}}{\bot_{\aQ^{\pOp}}}
= \lambda q \in Q. (q = \bot_\aQ) \;?\; 0 : 1$, and have also described a join-semilattice isomorphism:
\begin{quote}
with typing $\jslTrelem{\aQ^{\pOp}} \to \jslTrideal{\aQ}$ and action $\trelem{\aQ^{\pOp}}{q_0} \; \mapsto \; \trideal{\aQ}{q_0}$.
\end{quote}
Then precomposing with the isomorphism $\trelem{\aQ^{\pOp}}{-} : \aQ^{\pOp} \to \jslTrelem{\aQ^{\pOp}}$ from (1) yields the desired join-semilattice isomorphism.

\item
Follows by the previous statement, where it is proved that $\trideal{\aQ}{q} = \rot \circ (\trelem{\aQ^{\pOp}}{q})_*$.

\end{enumerate}
\end{proof}

We now define various special morphisms as compositions of element morphisms and ideal morphisms.

\begin{definition}[Special morphisms]
\label{def:spec_hom_morphisms}
\item
To any pair $(\aQ,\aR)$ of finite join-semilattices and elements $(q_0,r_0) \in Q \times R$, we associate two $\JSL_f$-morphisms:
\[
\begin{tabular}{lll}
$\up_{\aQ,\aR}^{q_0,r_0} \; := \quad
\aQ \xto{\ideal{\aQ}{q_0}} \two \xto{\elem{\aR}{r_0}} \aR$
&
$\up_{\aQ,\aR}^{q_0,r_0}(q) := \begin{cases}
\bot_\aR & \text{if $q \leq_\aQ q_0$}
\\
r_0 & \text{if $q \nleq_\aQ q_0$}
\end{cases}$
\\ \\
$\down_{\aQ,\aR}^{q_0,r_0} \; := \quad
\aQ \xto{\trideal{\aQ}{q_0}} \three \xto{\trelem{\aR}{r_0}} \aR$
&
$\down_{\aQ,\aR}^{q_0,r_0}(q) := \begin{cases}
\bot_\aR & \text{if $q = \bot_\aQ$}
\\
r_0 & \text{if $\bot_\aQ <_\aQ q \leq_\aQ q_0$}
\\
\top_\aR & \text{if $q \nleq_\aQ q_0$}
\end{cases}$
\end{tabular}
\]
\endbox
\end{definition}

\begin{note}[Intuition regarding special morphisms]
\item
\begin{enumerate}
\item
We often think of the special morphisms $\up_{\aQ,\aR}^{q,r} : \aQ \to \aR$ as `approximations from below' i.e.\ we imagine constructing arbitrary morphisms $\aQ \to \aR$ as pointwise joins of these special morphisms. If $(q,r) \in M(\aQ) \times J(\aR)$ then these morphisms are join-irreducible in $\JSL_f[\aQ,\aR]$. In the special case where $\aQ$ or $\aR$ is distributive every join-irreducible morphism takes this form. However this fails in general, and the restriction to those morphisms they join-generate (i.e.\ pointwise-join-generate) yields the previously studied concept of `tight morphism', and also a subfunctor of $\JSL_f[-,-]$ satisfying a universal property relative to tight morphisms. We shall investigate this carefully later on.

\item
The special morphisms $\down_{\aQ,\aR}^{q,r} : \aQ \to \aR$ may be thought of as `approximations from above'. They are used extensively over the next two subsections. As we shall see, they are precisely the meet-irreducible morphisms in $\JSL_f[\aQ,\aR]$, so in particular every morphism $\aQ \to \aR$ arises as a meet (which is rarely  pointwise) of these special morphisms. For the moment, observe that $\down_{\aQ,\aR}^{q,r} : \aQ \to \aR$ is the largest element of $\JSL_f[\aQ,\aR]$ extending:
\[
\elem{\aR}{r} \circ \ideal{[\bot_\aQ,q]}{\bot_\aQ} : [\bot_\aQ,q] \to \aR
\]
where $[\bot_\aQ,q] \subseteq \aQ$ is the interval sub join-semilattice. We should also mention that the equality:
\[
\down_{\aQ,\aR}^{q,r} \; = \; \up_{\aQ,\aR}^{\bot_\aQ,r} \lor_{\JSL_f[\aQ,\aR]} \up_{\aQ,\aR}^{q,\top_\aR}
\]
holds generally. However, this relationship will not be used or proved until the section concerning tight morphisms, although one could already deduce it from Lemma \ref{lem:special_jsl_morphisms}.6 below. \endbox
\end{enumerate}
\end{note}

\begin{lemma}[Properties of special morphisms]
\label{lem:special_jsl_morphisms}
\item
Fix any finite join-semilattices $\aQ$, $\aR$ and elements $(q_0,r_0)\,,(q_1,r_1) \in Q \times R$.
\begin{enumerate}

\item
We have the following symmetric equalities involving $\up_{\aQ,\aR}^{q_0,r_0}$ and $\down_{\aQ,\aR}^{q_0,r_0}$.
\[
\begin{tabular}{rll}
$(\up_{\aQ,\aR}^{q_0,r_0})_* = \; \up_{\aR^{\pOp},\aQ^{\pOp}}^{r_0,q_0}$
&&
$(\down_{\aQ,\aR}^{q_0,r_0})_* = \; \down_{\aR^{\pOp},\aQ^{\pOp}}^{r_0,q_0}$
\\[2ex]
$\up_{\aQ,\aR}^{q_0 \land_\aQ q_1, r_0}
\; = \;
\up_{\aQ,\aR}^{q_0, r_0} \lor_{\JSL_f[\aQ,\aR]} \up_{\aQ,\aR}^{q_1, r_0}$
&&
$\down_{\aQ,\aR}^{q_0 \land_\aQ q_1, r_0}
\; = \;
\down_{\aQ,\aR}^{q_0, r_0} \lor_{\JSL_f[\aQ,\aR]} \down_{\aQ,\aR}^{q_1, r_0}$
\\[1.5ex]
$\up_{\aQ,\aR}^{q_0, r_0 \lor_\aR r_1}
\; = \;
\up_{\aQ,\aR}^{q_0, r_0} \lor_{\JSL_f[\aQ,\aR]} \up_{\aQ,\aR}^{q_0, r_1}$
&&
$\down_{\aQ,\aR}^{q_0, r_0 \lor_\aR r_1}
\; = \;
\down_{\aQ,\aR}^{q_0, r_0} \lor_{\JSL_f[\aQ,\aR]} \down_{\aQ,\aR}^{q_0, r_1}$
\end{tabular}
\]
\[
\begin{tabular}{c}
$ \up_{\aQ,\aR}^{\top_\aQ,\bot_\aR} \;
= \; \bot_{\JSL_f[\aQ,\aR]} \;
= \; \down_{\aQ,\aR}^{\top_\aQ,\bot_\aR}$
\\[1ex]
$\up_{\aQ,\aR}^{\bot_\aQ,\top_\aR} \;
= \; \top_{\JSL_f[\aQ,\aR]} \;
= \; \down_{\aQ,\aR}^{\bot_\aQ,\top_\aR}$
\end{tabular}
\]

\item
Given $q_0 \neq \top_\aQ$ and $r_1 \neq \bot_\aR$, then:
\[
\up_{\aQ,\aR}^{q_0,r_0} \; \leq_{\JSL_f[\aQ,\aR]} \; \up_{\aQ,\aR}^{q_1,r_1}
\quad\iff\quad q_1 \leq_\aQ q_0 \text{ and } r_0 \leq_\aR r_1
\]
Moreover $\up_{\aQ,\aR}^{\top_\aQ,r_0} \;
= \; \bot_{\JSL_f[\aQ,\aR]} \; = \; \up_{\aQ,\aR}^{q_0,\bot_\aR} $ explains the remaining cases.

\item
We also have the following equalities involving $\down_{\aQ,\aR}^{q,r}$ only,
\[
\begin{tabular}{c}
$\down_{\aQ,\aR}^{\bot_\aQ, r_0}\;
= \; \top_{\JSL_f[\aQ,\aR]} \;
= \; \down_{\aQ,\aR}^{q_0,\top_\aR} $
\\[1ex]
$\down_{\aQ,\aR}^{q_0 \lor_\aQ q_1, r_0}
\; = \;
\down_{\aQ,\aR}^{q_0, r_0} \land_{\JSL_f[\aQ,\aR]} \down_{\aQ,\aR}^{q_1, r_0}$
\qquad
$\down_{\aQ,\aR}^{q_0, r_0 \land_\aR r_1}
\; = \;
\down_{\aQ,\aR}^{q_0, r_0} \land_{\JSL_f[\aQ,\aR]} \down_{\aQ,\aR}^{q_0, r_1}$
\end{tabular}
\]

\item
Given $\bot_\aQ \neq q_0$ and $r_1 \neq \top_\aR$, then:
\[
\down_{\aQ,\aR}^{q_0,r_0} \; \leq_{\JSL_f[\aQ,\aR]} \; \down_{\aQ,\aR}^{q_1,r_1}
\quad\iff\quad
\text{$q_1 \leq_\aQ q_0$ and $r_0 \leq_\aR r_1$}
\]
Moreover the remaining cases are explained by the previous statement.

\item
If $q_0 \leq_\aQ q_1$ then the meet $\down_{\aQ,\aR}^{q_0,r_0} \land_{\JSL_f[\aQ,\aR]} \down_{\aQ,\aR}^{q_1,r_1}$ is constructed pointwise in $\aR$. In fact,
\[
\begin{tabular}{ll}
$\down_{\aQ,\aR}^{q_0,r_0} \land_{\JSL_f[\aQ,\aR]} \down_{\aQ,\aR}^{q_1,r_1}$
&
$= \lambda q \in Q.( \down_{\aQ,\aR}^{q_0,r_0}(q) \; \land_\aR \down_{\aQ,\aR}^{q_1,r_1}(q))$
\\[1ex]&
$= \Lor_{\JSL_f[\aQ,\aR]} \{ \; \up_{\aQ,\aR}^{\bot_\aQ,r_0 \land_\aR r_1}, \;  \up_{\aQ,\aR}^{q_0,r_1}, \; \up_{\aQ,\aR}^{q_1,\top}  \}$
\end{tabular}
\]

\item
Regarding the relationship between the two different types of special morphisms:
\[
\up_{\aQ,\aR}^{q_0,r_0} \; \leq_{\JSL_f[\aQ,\aR]} \; \down_{\aQ,\aR}^{q_1,r_1}
\quad\iff\quad
q_1 \leq_\aQ q_0 \text{ or } r_0 \leq_\aR r_1
\]

\end{enumerate}
\end{lemma}

\begin{proof}
\item
\begin{enumerate}
\item
The top two equalities follow via very similar calculations:
\[
\begin{tabular}{lll}
$(\down_{\aQ,\aR}^{q_0,r_0})_*$
&
$= (\trelem{\aR}{r_0} \circ \trideal{\aQ}{q_0})_*$
\\&
$= (\trideal{\aQ}{q_0})_* \circ (\trelem{\aR}{r_0})_*$
\\&
$= (\trelem{\aQ^{\pOp}}{q_0} \circ \rot) \circ (\rot^{\bf-1} \circ \trideal{\aR^{\pOp}}{r_0})$
& by Lemma \ref{lem:trelem_trideal_well_def}
\\&
$= \trelem{\aQ^{\pOp}}{q_0} \circ \trideal{\aR^{\pOp}}{r_0}$
\\&
$= \; \down_{\aR^{\pOp},\aQ^{\pOp}}^{r_0,q_0}$
\\
\\
$(\up_{\aQ,\aR}^{q_0,r_0})_*$
&
$= (\elem{\aR}{r_0} \circ \ideal{\aQ}{q_0})_*$
\\&
$= (\ideal{\aQ}{q_0})_* \circ (\elem{\aR}{r_0})_*$
\\&
$= (\elem{\aQ^{\pOp}}{q_0} \circ \swap) \circ (\swap^{\bf-1} \circ \ideal{\aR^{\pOp}}{r_0})$
& by Lemma \ref{lem:jsl_elem_iso}
\\&
$= \elem{\aQ^{\pOp}}{q_0} \circ \ideal{\aR^{\pOp}}{r_0}$
\\&
$= \; \up_{\aR^{\pOp},\aQ^{\pOp}}^{r_0,q_0}$
\end{tabular}
\]
The other equalties also follow by considering the join-semilattice isomorphisms:
\[
\begin{tabular}{ll}
$\elem{\aQ}{-} : \aQ \to \JSL_f[\two,\aQ]$
&
$\ideal{\aQ}{-} : \aQ^{\pOp} \to \JSL_f[\aQ,\two]$
\\
$\trelem{\aQ}{-} : \aQ \to \JSL_f[\two,\aQ]$
&
$\trideal{\aQ}{-} : \aQ^{\pOp} \to \JSL_f[\aQ,\two]$
\end{tabular}
\]
which are necessarily also bounded lattice isomorphisms. Combined with the fact that $\JSL_f$-composition is bilinear we obtain all the other equalities e.g.\
\[
\begin{tabular}{lll}
$\down_{\aQ,\aR}^{q_0 \land_\aQ q_1, r_0}$
&
$= \trelem{\aR}{r_0} \circ \trideal{\aQ}{q_0 \land_\aQ q_1}$
\\&
$= \trelem{\aR}{r_0} \circ (\trideal{\aQ}{q_0} \lor_{\JSL_f[\aQ,\two]} \trideal{\aQ}{q_1})$
\\&
$= (\trelem{\aR}{r_0} \circ \trideal{\aQ}{q_0}) \lor_{\JSL_f[\aQ,\aR]} (\trelem{\aR}{r_0} \circ \trideal{\aQ}{q_1})$
\\&
$= \down_{\aQ,\aR}^{q_0,r_0} \lor_{\JSL_f[\aQ,\aR]} \down_{\aQ,\aR}^{q_0,r_1}$
\end{tabular}
\]
and the final line of equalities follows by preservation of top elements.

\item
We have $\up_{\aQ,\aR}^{q_0,r_0} \; \leq_{\JSL_f[\aQ,\aR]} \; \up_{\aQ,\aR}^{q_1,r_1}$ iff the following two statements hold:
\begin{enumerate}
\item
for all $q \nleq_\aQ q_0$ we have $q \nleq_\aQ q_1$ (since by assumption $r_0 \neq \bot_\aR$),
\item
for all $q \nleq_\aQ q_0$ we have $r_0 \leq_\aR r_1$.
\end{enumerate}
Then (a) is equivalent to $q_1 \leq_\aQ q_0$ and, since we assume $q_0 \neq \top_\aQ$, (b) implies and thus is equivalent to $r_0 \leq_\aR r_1$.

\item
It suffices to prove the left-hand equalities, since the others follow via $(\down_{\aQ,\aR}^{q,r})_* = \; \down_{\aR^{\pOp},\aQ^{\pOp}}^{r,q}$ proved in (1), recalling that taking adjoints defines a join-semilattice isomorphism $\JSL_f[\aQ,\aR] \cong \JSL_f[\aR^{\pOp},\aQ^{\pOp}]$.

\smallskip
Regarding the first equality,
\[
\down_{\aQ,\aR}^{\bot_\aQ,r_0}
= \trelem{\aR}{r_0} \circ \trideal{\aQ}{\bot_\aQ}
= \trelem{\aR}{r_0} \circ (\lambda q \in Q. (q = \bot_\aQ) \;?\; 0 : 2)
= \top_{\JSL_f[\aQ,\aR]}
\]
using the explicit description of $\trideal{\aQ}{\bot_\aQ}$ and the fact that $\trelem{\aR}{r_0}$ preserves $\top_\three$. Finally, we certainly have:
\[
\begin{tabular}{lll}
$\down_{\aQ,\aR}^{q_0 \lor_\aQ q_1, r_0}$
&
$= \trelem{\aR}{r_0} \circ \trideal{\aQ}{q_0 \lor_\aQ q_1}$
\\&
$= \trelem{\aR}{r_0} \circ (\trideal{\aQ}{q_0} \land_{\JSL_f[\aQ,\two]} \trideal{\aQ}{q_1})$
\\&
$\leq_{\JSL_f[\aQ,\aR]} (\trelem{\aR}{r_0} \circ \trideal{\aQ}{q_0}) \land_{\JSL_f[\aQ,\aR]} (\trelem{\aR}{r_0} \circ \trideal{\aQ}{q_1}) $
\\&
$= \; \down_{\aQ,\aR}^{q_0,r_0} \land_{\JSL_f[\aQ,\aR]} \down_{\aQ,\aR}^{q_1,r_0}$
\end{tabular}
\]
using monotonicity in the right parameter. To prove equality, suppose for a contradiction that:
\[
\down_{\aQ,\aR}^{q_0 \lor_\aQ q_1, r_0}
< h
\leq 
\; \down_{\aQ,\aR}^{q_0,r_0} \land_{\JSL_f[\aQ,\aR]} \down_{\aQ,\aR}^{q_1,r_0}
\]
for some morphism $h : \aQ \to \aR$. Then there must exist $q \leq_\aQ q_0 \lor_\aQ q_1$ such that $r_0 <_\aR h(q)$, so by monotonicity and join-preservation $r_0 <_\aR h(q_1) \lor_\aR h(q_2)$, which contradicts the fact that:
\[
h(q_i) 
\leq_\aR (\down_{\aQ,\aR}^{q_0,r_0} \land_{\JSL_f[\aQ,\aR]} \down_{\aQ,\aR}^{q_1,r_0})(q_i)
\leq_\aR \; \down_{\aQ,\aR}^{q_0,r_0}(q_i) \land_\aR \down_{\aQ,\aR}^{q_1,r_0}(q_i)
\leq_\aR r_0
\qquad
\text{(for $i = 0,1$)}
\]

\item
Letting $f_i := \down_{\aQ,\aR}^{q_i,r_i}$ for $i = 0,1$, then $f_0 \leq_{\JSL_f[\aQ,\aR]} f_1$ iff the following two statements hold:
\[
(a)\quad\forall q \nleq_\aQ q_0. f_1(q) = \top_\aR
\qquad
(b)\quad\forall \bot <_\aQ q \leq_\aQ q_0. \; r_0 \leq_\aR f_1(q)
\]
Since by assumption $r_1 \neq \top_\aR$, (a) is equivalent to $\forall q \in Q.(q \nleq_\aQ q_0 \To q \nleq_\aQ q_1)$ or equivalently $q_1 \leq_\aQ q_0$. Then in the presence of (a), statement (b) is equivalent to:
\[
\forall \bot_\aQ <_\aQ q \leq_\aQ q_0. \; r_0 \leq_\aQ r_1
\]
Since $q_0 \neq \bot_\aQ$ by assumption, the latter is equivalent to $r_0 \leq_\aQ r_1$ and we are done.

\item
Consider the morphism:
\[
h := \Lor_{\JSL_f[\aQ,\aR]} \{ \up_{\aQ,\aR}^{\bot_\aQ,r_0 \land_\aR r_1}, \;  \up_{\aQ,\aR}^{q_0,r_1}, \; \up_{\aQ,\aR}^{q_1,\top}  \} : \aQ \to \aR
\]
and also define the function $g : Q \to R$ as:
\[
g(x) := \; \down_{\aQ,\aR}^{q_0,r_0}(x) \, \land_\aR \, \down_{\aQ,\aR}^{q_1,r_1}(x)
\qquad
\text{for each $x \in Q$}.
\]
By Lemma \ref{lem:hom_meet_bound} it suffices to establish that $g = h$. Recalling our assumption that $q_0 \leq_\aQ q_1$, we can never have $x \leq_\aQ q_0$ and $x \nleq_\aQ q_1$. Then $g$ has action:
\[
g(x) =
\begin{cases}
\top_\aR & \text{if $x \nleq_\aQ q_1$ (hence also $x \nleq_\aQ q_0$) }
\\
r_1 & \text{if $\bot_\aQ <_\aQ x \leq_\aQ q_1$ and $x \nleq_\aQ q_0$}
\\
r_0 \land_\aR r_1
& \text{if $\bot_\aQ <_\aQ x \leq_\aQ q_0$ (hence also $\bot_\aQ <_\aQ x \leq_\aQ q_1$) }
\\
\bot_\aR & \text{if $x = \bot_\aQ$}
\end{cases}
\]
One can directly verify that $g(x) = h(x)$ for each of the four disjoint cases.

\item
We calculate:
\[
\begin{tabular}{lll}
$\up_{\aQ,\aR}^{q_0,r_0} \; \nleq_{\JSL_f[\aQ,\aR]} \; \down_{\aQ,\aR}^{q_1,r_1}$
&
$\iff \exists j_q \in J(\aQ).[ j_q \nleq_\aQ q_0 \text{ and } r_0 \nleq_\aR \; \down_{\aQ,\aR}^{q_1,r_1}(j_q) ]$
\\&
$\iff \exists j_q \in J(\aQ).[ j_q \nleq_\aQ q_0 \text{ and } j_q \leq_\aQ q_1 \text{ and } r_0 \nleq_\aR r_1 ]$
\\&
$\iff \exists j_q \in J(\aQ).[ j_q \nleq_\aQ q_0 \text{ and } j_q \leq_\aQ q_1 ] \text{ and } r_0 \nleq_\aR r_1$
\\&
$\iff \neg\forall j_q \in J(\aQ).[ j_q \leq_\aQ q_1 \To j_q \leq_\aQ q_0 ]  \text{ and } r_0 \nleq_\aR r_1$
\\&
$\iff q_1 \nleq_\aQ q_0 \text{ and } r_0 \nleq_\aR r_1$
\end{tabular}
\]

\end{enumerate}
\end{proof}

We now describe the meet-irreducible elements of $\JSL_f[\aQ,\aR]$, and also some of its join-irreducibles. In the special case that either $\aQ$ or $\aR$ is distributive then the latter will be precisely the join-irreducible morphisms.

\begin{lemma}[Meet and join-irreducible homomorphisms]
\label{lem:hom_meet_join_irr}
\item
Let $\aQ$ and $\aR$ be finite join-semilattices.
\begin{enumerate}
\item
$\down_{\aQ,\aR}^{j,m}$ is meet-irreducible in $\JSL_f[\aQ,\aR]$ whenever $j \in J(\aQ)$ and $m \in M(\aR)$, in fact:
\[
M(\JSL_f[\aQ,\aR]) = \{ \down_{\aQ,\aR}^{j,m} : j \in J(\aQ), \, m \in M(\aR) \}
\]
so that $|M(\JSL_f[\aQ,\aR])| = |J(\aQ)| \cdot |M(\aR)|$.

\item
Concerning the morphisms $\up_{\aQ,\aR}^{m,j} : \aQ \to \aR$ where $m \in M(\aQ)$ and $j \in J(\aR)$.
\begin{enumerate}
\item
They are always join-irreducible in $\JSL_f[\aQ,\aR]$.

\item
If $\aQ$ is distributive then:
\[
J(\JSL_f[\aQ,\aR]) = \{ \up_{\aQ,\aR}^{m,j} : m \in M(\aQ), \, j \in J(\aR) \}
\quad
(\star)
\]
so that $|J(\JSL_f[\aQ,\aR])| = |J(\aQ)| \cdot |J(\aR)|$.
\item
If $\aR$ is distributive then $(\star)$ again holds, so that $|J(\JSL_f[\aQ,\aR])| = |M(\aQ)| \cdot |M(\aR)|$.
\item
If neither $\aQ$ nor $\aR$ are distributive then these morphisms needn't  join-generate $\JSL_f[\aQ,\aR]$. We may have $|J(\JSL_f[\aQ,\aQ])| > |J(\aQ)|^2$ where $|J(\aQ)| = |M(\aQ)|$.

\end{enumerate}

\end{enumerate}
\end{lemma}

\begin{proof}
\item
\begin{enumerate}
\item
Let $\aS := \JSL_f[\aQ,\aR]$. We first show that every morphism $f : \aQ \to \aR$ arises as an $\aS$-meet of the morphisms $\down_{\aQ,\aR}^{j,m} : \aQ \to \aR$ where $(j,m) \in J(\aQ) \times M(\aR)$. Indeed, consider:
\[
g := \Land_\aS \{ \; \down_{\aQ,\aR}^{j,m}  \; : (j,m) \in J(\aQ) \times M(\aR), \, f(j) \leq_\aR m \}
\]
First of all, $f \leq_\aS g$ because $f \leq_\aS \; \down_{\aQ,\aR}^{j,m}$ for each summand $\down_{\aQ,\aR}^{j,m}$ above. To see this, observe that if $\bot_\aQ <_\aQ q \leq_\aQ j$ then $f(q) \leq_\aR f(j) \leq_\aR m$ using the monotonicity of $f$. Now, by Lemma \ref{lem:hom_meet_bound} we know that:
\[
g(q) \leq_\aR \Land_\aR \{ \; \down_{\aQ,\aR}^{j,m}(q) \; : (j,m) \in J(\aQ) \times M(\aQ) , \, f(j) \leq_\aR m \} 
\qquad
\text{for each $q \in Q$}
\]
and consequently for every $j_0 \in J(\aQ)$ we have:
\[
g(j_0) 
\leq_\aR \Land_\aR \{ \; \down_{\aQ,\aR}^{j_0,m}(j_0) \; : f(j_0) \leq_\aR m \in M(\aR) \} 
= \Land_\aR \{ m \in M(\aR) : f(j_0) \leq_\aR m \} 
= f(j_0)
\]
and it follows that $g \leq_\aR f$. Thus every morphism $\aQ \to \aR$ arises as the $\aS$-meet of these special morphisms, and hence every meet-irreducible in $\aS$ is one of these morphisms. Then to show that every $\down_{\aQ,\aR}^{j,m}$ is meet-irreducible, it suffices to establish that they are not meets of other such special morphisms. To this end, first observe that $\up_{\aQ,\aR}^{j_1,m_1} \; \leq \; \up_{\aQ,\aR}^{j_2,m_2}$ if and only if $j_2 \leq_\aQ j_1$ and $m_1 \leq_\aR m_2$ by Lemma \ref{lem:special_jsl_morphisms}.4. Now, fix any $f  := \; \down_{\aQ,\aR}^{j,m}$ where $(j,m) \in J(\aQ) \times M(\aR)$ and consider the morphisms:
\[
\begin{tabular}{lll}
$g_1 := \; \down_{\aQ,\aR}^{q_j,m}$
&
where $q_j $ & $:= \Lor_\aQ \{ j' \in J(\aQ) : j' <_\aQ j \}$,
\\[1ex]
$g_2 := \; \down_{\aQ,\aR}^{j,r_m}$
&
where $r_m$ & $:= \Land_\aR \{ m' \in M(\aR) : m <_\aR m' \}$.
\end{tabular}
\]
Then we have $q_j <_\aQ j$ by join-irreducibility and $m <_\aR r_m$ by meet-irreducibility. Using Lemma \ref{lem:special_jsl_morphisms}.4:
\begin{enumerate}
\item
$f <_\aS g_1,\,g_2$ and hence $f \leq_\aS g_1 \land_\aS  g_2$.
\item
Whenever $f <_\aS \; \down_{\aQ,\aR}^{j_i,m_i}$ for some $(j_i,m_i) \in J(\aQ) \times M(\aR)$ then either ($j_i <_\aQ j$ and $m \leq_\aR m_i$) or ($j_i \leq_\aQ j$ and $m <_\aR m_i$), and consequently $g_1 \land_\aS g_2 \leq_\aS \; \down_{\aQ,\aR}^{j_i,m_i}$.
\end{enumerate}

 It follows that to establish the meet-irreducibility of $f$ we can show that $f \neq g_1 \land_\aS g_2$. Since $q_j <_\aQ j$ we may apply Lemma \ref{lem:special_jsl_morphisms}.5 to deduce that $g_1 \land_\aS g_2$ is constructed pointwise, hence:
\[
(g_1 \land_\aS g_2)(j)
= g_1(j) \land_\aR g_2(j)
= \top_\aR \land_\aR r_m
= r_m >_{\aR} r = f(j)
\]
as required. Finally, these maps are pairwise distinct so $|M(\aS)| = |J(\aQ)| \cdot |M(\aR)|$.

\item
Again let $\aS := \JSL_f[\aQ,\aR]$ and now consider the special morphisms $\up_{\aQ,\aR}^{m,j} : \aQ \to \aR$ where $m \in M(\aQ)$ and $j \in J(\aR)$.
\begin{enumerate}
\item
To see that they are join-irreducible, suppose that $\up_{\aQ,\aR}^{m,j} \; = f \lor_{\JSL_f[\aQ,\aR]} g$. Since $m$ is meet-irreducible it has a unique cover $m \prec_\aQ x$, and since $j = f(x) \lor_\aR g(x)$ is join-irreducible we may assume w.l.o.g.\ that $f(x) = j$. Seeing as $f \leq_{\JSL_f[\aQ,\aR]} \; \up_{\aQ,\aR}^{m,j}$ it follows that for any $q \leq_\aQ m$ we have $f(q) = \bot_\aR$, and for any $q \nleq_\aQ m$ we have $f(q) \leq_\aQ j$. Now, fix any $q \nleq_\aQ m$ and observe that $m <_\aQ q \lor_\aQ m$ because equality yields a contradiction. Thus $x \leq_\aQ q \lor_\aQ m$ and hence:
\[
j 
= f(x) \leq_\aQ f(q \lor_\aQ m) 
= f(q) \lor_\aR f(m) 
= f(q) \lor_\aR \bot_\aR
= f(q)
\]
using the monotonicity of $f$, preservation of joins and also $f(m) = \bot_\aR$. Hence $f = \; \up_{\aQ,\aR}^{m,j}$ and we are done.

\item
Assuming that $\aQ$ is distributive, let us show that the $\up_{\aQ,\aR}^{m,j}$ join-generate $\aS$. Given any join-semilattice morphism $f : \aQ \to \aR$ define the morphism:
\[
g := \Lor_\aS \{ \up_{\aQ,\aR}^{m_j, j_0}  \; : j \in J(\aQ), \, j_0 \in J(\aR), \, j_0 \leq_\aR f(j) \}
\]
where $m_j := \Lor_\aQ \{ q \in Q : j \nleq_\aQ q \} \in M(\aQ)$ is the meet-irreducible corresponding to $j$ under the canonical bijection from Lemma \ref{lem:std_order_theory}.13. To establish $g \leq_\aS f$ we'll show that every summand $\up_{\aQ,\aR}^{m_j,j_0} \; \leq_\aS f$ i.e.\ whenever $q \nleq_\aQ m_j$ we must show that $j_0 \leq_\aQ f(q)$. By construction $j_0 \leq_\aQ f(j)$ and the canonical bijection informs us that $j = \Land_\aQ \{ q \in Q : q \nleq_\aQ m_j \}$, hence $j \leq_\aQ q$ and thus $j_0 \leq_\aQ f(j) \leq_\aQ f(q)$ using the monotonicity of $f$. To establish the converse $f \leq_\aS g$ it suffices to show that:
\[
f(j_q) 
\leq_\aR \Lor_\aR \{ \up_{\aQ,\aR}^{m_{j_q},j_0}(j_q) : j_0 \in J(\aR), \, j_0 \leq_\aR f(j_q) \}
= \Lor_\aR \{ j_0 \in J(\aR) : j_0 \leq_\aR f(j_q) \}
\]
which follows because $f(j_q)$ is the $\aR$-join of those join-irreducibles beneath it. Then the $\up_{\aQ,\aR}^{m,j}$ are precisely the join-irreducibles in $\aS$. Since they are pairwise distinct the number of join-irreducibles is exactly $|M(\aQ)| \cdot |J(\aR)| = |J(\aQ)| \cdot |J(\aR)|$, recalling that $|J(\aQ)| = |M(\aQ)|$ in a distributive lattice via the canonical bijection.

\item
Now instead assume that $\aR$ is distributive. By Lemma \ref{lem:jsl_mor_iso} we know that $\JSL_f[\aQ,\aR] \cong \JSL_f[\aR^{\pOp},\aQ^{\pOp}]$ where the action of this join-semilattice isomorphism takes the adjoint. Since distributive lattices are closed under taking the order-dual, we may apply the previous statement. This then translates back to the desired statement via Lemma \ref{lem:special_jsl_morphisms}.4. We finally deduce that:
\[
|J(\JSL_f[\aQ,\aR])| 
= |J(\JSL_f[\aR^{\pOp},\aQ^{\pOp}])|
= |J(\aR^{\pOp})| \cdot |J(\aQ^{\pOp})|
= |M(\aR)| \cdot |M(\aQ)|
\]

\item
Let $\aQ = \aR$ be $M_3$ with three atoms $x_1,\, x_2, \, x_3$. Then the identity morphism $id_\aQ : \aQ \to \aQ$ does not arise as a join of the special morphisms $\up_{\aQ,\aQ}^{m,j}$. To see this, observe that the latter sends $m$ to $\bot_\aQ$, and the other two atoms to $j$. Thus none of them are pointwise below $id_\aQ$, so it cannot arise as a join of them. In fact, none of the six isomorphisms of $\aQ$ are join-generated by these special morphisms. By (a) every $\up_{\aQ,\aQ}^{j,m}$ is join-irreducible, in fact they are atoms: if $f : \aQ \to \aQ$ sends more than one atom to $\bot_\aQ$ then it sends everything to $\bot_\aQ$. The remaining join-irreducibles are also atoms: send one atom to $\bot_\aQ$ and the others to distinct atoms. Thus $|J(\JSL_f[\aQ,\aQ])| = 3 \cdot (3^2) = 27 > 3 \cdot 3 = |J(\aQ)|^2$, where $|J(\aQ)| = |M(\aQ)|$ by symmetry.

\end{enumerate}

\end{enumerate}
\end{proof}

In the rest of this subsection we define the tensor product functor, the associated notion of bimorphism and prove some basic properties. The tensor product $\aQ \tenp \aR$ is defined as a composite functor built from one copy of $\JSL_f[-,-] : \JSL_f^{op} \times \JSL_f \to \JSL_f$ and two copies of the self-duality functor $\OD_j : \JSL_f^{op} \to \JSL_f$. The associated bimorphisms are actually mappings $(q,r) \mapsto \; \down_{\aQ,\aR^{\pOp}}^{q,r}$, so these special morphisms play a prominent role. In particular, the above Lemmas concerning irreducible homomorphisms directly provide descriptions of irreducible elements inside $\aQ \tenp \aR$. In the next subsection we'll describe the tensor product in a different way i.e.\ in terms of so-called bi-ideals, a concept that arises naturally from $\BiCliq$.

\begin{definition}[Tensor product of finite join-semilattices]
\label{def:tenp}
\item
The \emph{tensor product} functor $\tenp : \JSL_f \times \JSL_f \to \JSL_f$ is the composite functor:
\[
\JSL_f \times \JSL_f \xto{(\JSL_f[-,\OD_j(-)])^{op}} \JSL_f^{op} \xto{\OD_j} \JSL_f
\]
It also has canonically associated functions for each pair $(\aQ,\aR)$,
\[
\begin{tabular}{c}
$\beta_{\aQ,\aR} : Q \times R \to \JSL_f(\aQ,\aR^{\pOp})$
\\
$\quad\text{where}\quad
\beta_{\aQ,\aR}(q_0,r_0) 
:= \; \down_{\aQ,\aR^{\pOp}}^{q_0,r_0} \; 
= \trelem{\aR^{\pOp}}{r_0} \circ \trideal{\aQ}{q_0}
= \lambda q \in Q.
\begin{cases} 
\top_\aR & \text{if $q = \bot_\aQ$}
\\
r_0 & \text{if $\bot_\aQ <_\aQ q \leq_\aQ q_0$}
\\
\bot_\aR & \text{if $q \nleq_\aQ q_0$}
\end{cases}$
\end{tabular}
\]
observing that the $\top_\aR$ and $\bot_\aR$ are `switched' because we work with $\aR^{\pOp}$. \endbox
\end{definition}

\begin{note}[The tensor product in more detail]
\label{note:tenp_detail}
Regarding its action on objects,
\[
\begin{tabular}{c}
$\aQ \tenp \aR 
= (\JSL_f[\aQ,\aR^{\pOp}])^{\pOp}
= (\JSL_f(\aQ,\aR^{\pOp}),\lor_{\aQ \tenp \aR},\bot_{\aQ \tenp \aR})$
\end{tabular}
\]
where $\lor_{\aQ \otimes \aR}$ is defined as the binary meet in $\JSL_f[\aQ,\aR^{\pOp}]$, and 
\[
\bot_{\aQ \tenp \aR} = \top_{\JSL_f[\aQ,\aR^{\pOp}]}  =  \lambda q \in Q. (q = \bot_\aQ) : \top_\aR : \bot_\aR.
\]
Observe $\bot_{\aQ \tenp \aR} = \beta_{\aQ,\aR}(\bot_\aQ,r) = \beta_{\aQ,\aR}(q,\bot_\aR)$ for any $q \in Q$, $r \in R$ by Lemma \ref{lem:special_jsl_morphisms}.3, this being the bilinearity condition for bottom elements. \endbox

\takeout{
 Furthermore $f \otimes g = (\JSL_f[f^{op},g_*])_*$, whose computation can be broken down as follows:
\begin{enumerate}
\item
We start with homomorphisms  $f : \aQ_1 \to \aQ_2$ and $g : \aR_1 \to \aR_2$.
\item
Then $\JSL_f[f^{op},g_*] : \JSL_f[\aQ_2,\aR_2^{\pOp}] \to \JSL_f[\aQ_1,\aR_1^{\pOp}]$ acts as follows:
\[
\aQ_2 \xto{h} \aR_2^{\pOp} 
\quad \mapsto \quad
\aQ_1 \xto{f} \aQ_2 \xto{h} \aR_2^{\pOp} \xto{g_*} \aR_1^{\pOp}
\]
\item
$f \otimes g$ is the adjoint of the above morphism, so has typing $(\JSL_f[\aQ_1,\aR_1^{\pOp}])^{\pOp} \to (\JSL_f[\aQ_2,\aR_2^{\pOp}])^{\pOp}$ with action:
\[
h : \aQ_1 \to \aR_1^{\pOp}
\quad\mapsto\quad
\Lor_{\JSL_f[\aQ_2,\aR_2^{\pOp}]} \{ \theta : \aQ_2 \to \aR_2^{\pOp} : g_* \circ \theta \circ f \leq_{\JSL_f[\aQ_1,\aR_1^{\pOp}]} h  \}
\]
That is, we take the pointwise-join inside $\aR^{\pOp}$ (i.e.\ pointwise-meet inside $\aR$) of all those homomorphisms $\theta : \aQ_2 \to \aR_2^{\pOp}$ such that $g_* \circ \theta \circ f$ is pointwise-below $h$ inside $\aR_1^{\pOp}$ (i.e.\ pointwise-above inside $\aR_1$).
\end{enumerate}
Our later usage of bi-ideals will make tensor products easier to deal with.
}
\end{note}

\bigskip

Since $-\tenp-$ is defined as a composite of well-defined functors, we have:

\smallskip

\begin{lemma}
$\tenp : \JSL_f \times \JSL_f \to \JSL_f$ is a well-defined functor.
\end{lemma}

\smallskip

Each function $\beta_{\aQ,\aR}$ is `well-defined' in the sense that it defines a bilinear mapping i.e.\ a \emph{bimorphism}.

\bigskip

\begin{definition}[Bimorphisms]
\label{def:bimorphism}
For any triple of finite join-semilattices $(\aQ,\aR,\aS)$, a \emph{bimorphism} (or \emph{bilinear mapping}) from $(\aQ,\aR)$ to $\aS$ is a function $\beta : Q \times R \to S$ such that:
\begin{enumerate}
\item
$\beta(\bot_\aQ,r) = \beta(q,\bot_\aR) = \bot_\aS$ for any $q \in Q$ and $r \in R$.
\item
$\beta(q_1 \lor_\aQ q_2,r) = \beta(q_1,r) \lor_\aS \beta(q_2,r)$ for any $q_1\,,q_2 \in Q$ and $r \in R$.
\item
$\beta(q,r_1 \lor_\aR r_2) = \beta(q,r_1) \lor_\aS \beta(q,r_2)$ for any $q \in Q$ and $r_1\,,r_2 \in R$.
\end{enumerate}
Each $\JSL_f$-morphism $f : \aQ \otimes \aR \to \aS$ \emph{induces} the bimorphism $\beta_f : Q \times R \to S$ from $(\aQ,\aR)$ to $\aS$ with action:
\[
\beta_f (q,r) := f(\beta_{\aQ \times \aR}(q,r))
\]
Finally, let $\BiMor{\aQ,\aR,\aS}$ be the set of all bimorphisms from  $(\aQ,\aR)$ to $\aS$. \endbox
\end{definition}

\bigskip

\begin{lemma}[Basic properties of $\aQ \tenp \aR$]
\label{lem:tenp_basic}
\item
Let $(\aQ,\aR)$ be finite join-semilattices, and recall that $\beta_{\aQ,\aR}(q,r) = \;\down_{\aQ,\aR^{\pOp}}^{q,r}$.
\begin{enumerate}
\item
Each function $\beta_{\aQ,\aR} : Q \times R \to \JSL_f(\aQ,\aR^{\pOp})$ is a well-defined bilinear mapping from $(\aQ,\aR)$ to $\aQ \tenp \aR$.

\item
Concerning irreducibles.

\begin{enumerate}
\item
$J(\aQ \tenp \aR) = \{ \beta_{\aQ,\aR}(j_q,j_r) : j_q \in J(\aQ),\, j_r \in J(\aR)  \}$ hence $|J(\aQ \tenp \aR)| = |J(\aQ)| \cdot |J(\aR)|$.
\item
If $\aQ$ or $\aR$ are distributive then:
\[
M(\aQ \tenp \aR) = \{ \up_{\aQ,\aR^{\pOp}}^{m_q,m_r} : m_q \in M(\aQ), \, m_r \in M(\aR) \}
\]
hence $|M(\aQ \tenp \aR)| = |M(\aQ)| \cdot |M(\aR)|$.
\end{enumerate}

Thus by (a) the images $\beta_{\aQ,\aR}[J(\aQ)\times J(\aR)] \subseteq \beta_{\aQ,\aR}[Q \times R]$ both join-generate $\aQ \tenp \aR$.

\item
$\beta_{\aQ,\aR}$ almost defines an order-embedding of $(Q \times R,\leq_{\aQ \times \aR})$ into $\aQ \tenp \aR$. That is, for any $(q_1,r_1),\,(q_2,r_2) \in Q \times R$ such that $q_1 \neq \bot_\aQ$ and $r_1 \neq \bot_\aR$,
\[
\beta_{\aQ,\aR}(q_1,r_1) \leq_{\aQ \tenp \aR} \beta_{\aQ,\aR}(q_2,r_2)
\quad\iff\quad
(q_1,r_1) \leq_{\aQ \times \aR} (q_2,r_2)
\]
Also, the implication $\oT$ holds without restriction i.e.\ $\beta_{\aQ, \aR}$ defines a monotone map from $\aQ \times \aR$ to $\aQ \tenp \aR$.

\item
For any join-semilattice morphism $f : \aQ \tenp \aR \to \aS$,  $\beta_f$ is a well-defined bilinear mapping from $(\aQ,\aR)$ to $\aS$.

\end{enumerate}
\end{lemma}

\begin{proof}
\item
\begin{enumerate}
\item
$\beta_{\aQ,\aR}$ is a well-defined function because each $\down_{\aQ,\aR^{\pOp}}^{q_0,r_0}$ is a well-defined join-semilattice morphism of type $\aQ \to \aR^{\pOp}$ by Lemma \ref{lem:special_jsl_morphisms}.1. Concerning bilinearity, we already observed that $\beta(\bot_\aQ,r) = \beta(q,\bot_\aR) = \bot_{\aQ \otimes \aR}$ in Note \ref{note:tenp_detail}. The other conditions follow directly from Lemma \ref{lem:special_jsl_morphisms}.3, since $\lor_{\aQ \tenp \aR} \, = \, \land_{\JSL_f[\aQ,\aR]}$ and $\land_{\aR^{\pOp}} \; = \; \lor_\aR$.

\item
These statements follow directly from Lemma \ref{lem:hom_meet_join_irr} i.e.\ our description of join-irreducibles and meet-irreducibles in $\JSL_f[\aQ,\aR]$. 

\begin{enumerate}
\item
This is the first statement, since $J(\aQ \tenp \aR) = M(\JSL_f[\aQ,\aR^{\pOp}])$ and also  $M(\aR^{\pOp}) = J(\aR)$.
\item
We are using the second statement, since $M(\aQ \tenp \aR) = J(\JSL_f[\aQ,\aR^{\pOp}])$. If $\aQ$ is distributive their cardinality is $|J(\aQ)| \cdot |J(\aR^{\pOp})| = |M(\aQ)| \cdot |M(\aR)|$ recalling that $|J(\aQ)| = |M(\aQ)|$. On the other hand, if $\aR$ is distributive their cardinality is $|M(\aQ)| \cdot |M(\aR^{\pOp})| = |M(\aQ)| \cdot |M(\aR)|$ since $|J(\aR)| = |M(\aR)|$.
\end{enumerate}

\item
Unwinding the definitions, we have:
\[
\beta_{\aQ,\aR}(q_1,r_1) \leq_{\aQ \tenp \aR} \beta_{\aQ,\aR}(q_2,r_2)
\quad\iff\quad
\down_{\aQ,\aR^{\pOp}}^{q_2,r_2} \; \leq_{\JSL_f[\aQ,\aR^{\pOp}]} \; \down_{\aQ,\aR^{\pOp}}^{q_1,r_1}
\]
Then our assumptions that $\bot_\aQ \neq q_1$ and $r \neq \top_{\aR^{\pOp}} = \bot_\aR$ are precisely those from Lemma \ref{lem:special_jsl_morphisms}.4. Thus the above holds iff $q_1 \leq_\aQ q_2$ and $r_2 \leq_{\aR^{\pOp}} r_1$ (or equivalently $r_1 \leq_\aR r_2$).

Finally if $q_1 = \bot_\aQ$ or $r_1 = \bot_\aR$ then by bilinearity $\beta_{\aQ,\aR}(q_1,r_1) = \bot_{\aQ \tenp \aR} \leq_{\aQ \tenp \aR} \beta_{\aQ,\aR}(q_2,r_2)$. Hence the original implication $\oT$ holds without restriction.

\item
Given any join-semilattice morphism $f : \aQ \tenp \aR \to \aS$ we must verify that $\beta_f := \lambda (q,r) \in Q \times R. f(\beta_{\aQ,\aR}(q,r))$ is bilinear. This follows immediately via (1) i.e.\ that $\beta_{\aQ,\aR}$ is bilinear. 
\[
\begin{tabular}{c}
$f(\beta_{\aQ,\aR}(\bot_\aQ,r))
= f(\bot_{\aQ \tenp \aR})
= \bot_\aS
\qquad
f(\beta_{\aQ,\aR}(q,\bot_\aR))
= f(\bot_{\aQ \tenp \aR})
= \bot_\aS$
\\
$f(\beta_{\aQ,\aR}(q_1 \lor_\aQ q_2,r))
= f(\beta_{\aQ,\aR}(q_1,r) \lor_{\aQ \tenp \aR} \beta_{\aQ,\aR}(q_2,r))
= f(\beta_{\aQ,\aR}(q_1,r)) \lor_\aS f(\beta_{\aQ,\aR}(q_2,r))
$
\end{tabular}
\]
where preservation of joins in the right parameter follows symmetrically.

\end{enumerate}
\end{proof}

\subsubsection{Universality of the tensor product via $\BiCliq$ and bi-ideals}

In order to prove the universality of the tensor product, we'll describe the latter in terms of the category $\BiCliq$. This amounts to (and explains) the `bi-ideals' approach of Fraser \cite{Fraser1978} and more recently of Gr\"{a}tzer and Wehrung \cite{GratzerTensorSemilattices2005}. We proceed as follows.
\begin{enumerate}
\item
One has the inclusion-ordered join-semilattice of $\BiCliq$-morphisms:
\[
\BiCliq[\rG,\rH] := (\BiCliq(\rG,\rH),\cup,\emptyset : \rG \to \rH)
\]
and moreover $\JSL_f[\aQ,\aR]$ is isomorphic to $\BiCliq[\nleq_\aQ, \nleq_\aR]$. Here we have chosen to use $\nleq_\aQ$ rather than its restriction $\Pirr\aQ$, recalling that this is permissible via the natural isomorphism $\rE : \Pirr \To \Nleq$ -- see Lemmas \ref{lem:pirr_to_nleq_iso} and \ref{lem:nleq_e_well_def}. More importantly, it makes the connection with bi-ideals clearer.

\item
Recall that $\aQ \tenp \aR = (\JSL_f[\aQ,\aR^{\pOp}])^{\pOp}$. Let us express this in terms of $\BiCliq$.
\begin{enumerate}
\item
Start with the sub join-semilattice $\BiCliq[\nleq_\aQ,\nleq_{\aR^{\pOp}}] \subseteq \JPow (Q \times R)$.
\item
To obtain the opposite join-semilattice we take the pointwise relative complements inside $Q \times R$ and order by inclusion.
\item
Such relative complements correspond to taking the complement relation, and are necessarily closed under intersections. This inclusion-ordered join-semilattice is denoted by:
\[
\jslBId{\aQ,\aR} := (\BId{\aQ,\aR},\lor_{\jslBId{\aQ,\aR}},\bot_{\jslBId{\aQ,\aR}})
\]
\end{enumerate}

By construction we have $\aQ \tenp \aR \cong \jslBId{\aQ,\aR}$ i.e.\ another description of the tensor product.

\item
Importantly we have two different descriptions of $\BId{\aQ,\aR}$.

\begin{enumerate}
\item
The first comes directly from the equivalence between $\JSL_f$ and $\BiCliq$. That is, the  elements of $\BId{\aQ,\aR}$ are precisely the relations $\rR(q,r) \iff r \leq_\aR f(q)$ for some join-semilattice morphism $f : \aQ \to \aR^{\pOp}$.

\item
The second is the pre-existing notion of \emph{bi-ideal} \cite{Fraser1978, GratzerTensorSemilattices2005}: a subset $\rR \subseteq Q \times R$ lies in $\BId{\aQ,\aR}$ iff it is closed under the following rules:
\begin{enumerate}
\item
$\rR(\bot_\aQ,r)$ for all $r \in R$, and $\rR(q,\bot_\aR)$ for all $q \in Q$.
\item
$\rR$ is downwards-closed inside $\aQ \times \aR$.
\item
$\rR$ is closed under `lateral joins' i.e.\ 
\[
\rR(q,r_1) \, \land \, \rR(q,r_2) \implies \rR(q,r_1 \lor_\aR r_2)
\qquad
\rR(q_1,r) \, \land \, \rR(q_2,r) \implies \rR(q_1 \lor_\aQ q_2,r)
\]
\end{enumerate}

\end{enumerate}

\item
Having established the latter correspondence, we'll prove that the tensor product of finite join-semilattices is universal \cite{Fraser1978}.

\end{enumerate}

\smallskip
So let us begin by describing the join-semilattice structure of $\BiCliq$'s hom-sets, and also its top element.

\begin{lemma}
\label{lem:bicliq_mor_jsl_struct}
Let $\rG$ and $\rH$ be relations between finite sets.
\begin{enumerate}
\item
We have the $\BiCliq$-morphism $\emptyset : \rG \to \rH$ with associated components:
\[
\begin{tabular}{lll}
$\emptyset_-$ & $:= \BC{\rG_s,\cl_\rH(\emptyset)} \subseteq \rG_s \times \rH_s$
\\
$\emptyset_+$ & $:= \BC{\rH_t,\cl_{\breve{\rG}}(\emptyset)} \subseteq \rH_t \times \rG_t$
\end{tabular}
\]
i.e.\ we connect everything to the respective isolated elements.

\item
Given $\BiCliq$-morphisms $\rR,\,\rS : \rG \to \rH$, their union defines a $\BiCliq$-morphism $\rR \cup \rS : \rG \to \rH$, and:
\[
\begin{tabular}{lll}
$(\rR \cup \rS)_-$ & $:= \{ (g_s,h_s) \in \rG_s \times \rH_s : h_s \in \cl_\rH((\rR_- \cup \rS_-)[g_s]) \}$
\\
$(\rR \cup \rS)_+$ & $:= \{ (h_t,g_t) \in \rH_t \times \rG_t : g_t \in \cl_{\breve{\rG}}((\rR_+ \cup \rS_+)[h_t]) \}$
\end{tabular}
\]
are its associated component relations.

\item
We have the $\BiCliq$-morphism $\top_{\BiCliq[\aQ,\aR]} = \BC{\breve{\rG}[\rG_t]\,,\, \rH[\rH_s]} : \rG \to \rH$ with associated components:
\[
\begin{tabular}{lll}
$(\top_{\BiCliq[\aQ,\aR]})_-$ & $:= \emptyset_- \cup \BC{\breve{\rG}[\rG_t],\,\breve{\rH}[\rH_t]}$ & $\subseteq \rG_s \times \rH_s$
\\
$(\top_{\BiCliq[\aQ,\aR]})_+$ & $:= \emptyset_+ \cup \BC{\rH[\rH_s],\,\rG[\rG_s]}$ & $\subseteq \rH_t \times \rG_t$
\end{tabular}
\]

\end{enumerate}
\end{lemma}

\begin{proof}
\item
\begin{enumerate}
\item
$\emptyset : \rG \to \rH$ is a $\BiCliq$-morphism via the witnesses $\emptyset ; \rH = \emptyset = \emptyset = \rG ; \emptyset$. Closing these witnesses, the negative component $\emptyset_-$ sends every $g_s$ to $\cl_\rH(\emptyset)$, whereas the positive component sends every $h_t$ to $\cl_{\breve{\rH}}(\emptyset)$.

\item
We have $\rR_- ; \rH = \rR = \rG ; \rR_+\spbreve$ and $\rS_- ; \rH = \rS = \rG ; \rS_+\spbreve$. Then since (i) relational composition preserves unions separately in each component, (ii) relational converse preserves unions, we deduce that $\rR \cup \rS$ is a well-defined $\BiCliq$-morphism via the witnesses $(\rR_- \cup \rS_-,\rR_+ \cup \rS_+)$. Closing these witnesses point-image-wise yields the associated components.

\item
$\top_{\BiCliq[\rG,\rH]} := \breve{\rG}[\rG_t] \times \rH[\rH_s]$ defines a $\BiCliq$-morphism of type $\rG \to \rH$ via the witnesses:
\[
\xymatrix@=20pt{
\rG_t \ar[rr]^{\BC{\rG_t,\rH[\rH_s]}} && \rH_t
\\
\rG_s \ar[u]^{\rG} \ar[rr]_{\BC{\breve{\rG}[\rG_t], \rH_s}} && \rH_s \ar[u]_{\rH}
}
\]
so let us compute the negative component:
\[
\begin{tabular}{lll}
$(\top_{\BiCliq[\rG,\rH]})_-(g_s,h_s)$
&
$\iff h_s \in \rH^\down(\top_{\BiCliq[\rG,\rH]}[g_s])$
\\&
$\iff \rH[h_s] \subseteq (\breve{\rG}[\rG_t] \times \rH[\rH_s])[g_s]$
\\&
$\iff (g_s \in \cl_\rG(\emptyset) \text{ and } \rH[h_s] \subseteq \emptyset) \text{ or } (g_s \nin \cl_\rG(\emptyset) \text{ and } \rH[h_s] \subseteq \rH[\rH_s])$
\\&
$\iff (g_s,h_s) \in \cl_\rG(\emptyset) \times \cl_\rH(\emptyset) \,\cup\, \breve{\rG}[\rG_t] \times \rH_s$
\\&
$\iff (g_s,h_s) \in \emptyset_- \cup \BC{\breve{\rG}[\rG_t],\,\breve{\rH}[\rH_t]}$
\end{tabular}
\]
As for the positive component, recall that it is the negative component of the dual morphism. Since relational converse preserves inclusions it also preserves the largest morphism, thus:
\[
(\top_{\BiCliq[\rG,\rH]})_+
= (\top_{\BiCliq[\breve{\rH},\breve{\rG}]})_-
= (\emptyset : \breve{\rH} \to \breve{\rG})_- \cup \BC{\rH[\rH_s],\,\rG[\rG_s]}
= \emptyset_+ \cup \BC{\rH[\rH_s],\,\rG[\rG_s]}
\]

\end{enumerate}
\end{proof}

Now we repackage the preceding Lemma as a Definition.

\begin{definition}[Join-semilattice structure on $\BiCliq$'s hom-sets]
\label{def:bicliq_hom_functor_jsl}
\item
For each bipartite graph $\rG$ and $\rH$ define the finite join-semilattice:
\[
\BiCliq[\rG,\rH] := (\BiCliq(\rG,\rH),\cup,\emptyset : \rG \to \rH ) 
\]
which is well-defined by Lemma \ref{lem:bicliq_mor_jsl_struct}. Observe that it is ordered by inclusion. It extends to a functor $\BiCliq[-,-] : \BiCliq^{op} \times \BiCliq \to \JSL_f$, whose action on morphisms is as follows:
\[
\dfrac{\rR : \rG \to \rH \quad \rS : \rG' \to \rH'}
{\BiCliq[\rR^{op},\rS] := \lambda \rT. \rR \fatsemi \rT \fatsemi \rS : \BiCliq[\rH,\rG'] \to \BiCliq[\rG,\rH']}
\]
\endbox
\end{definition}

\begin{lemma}
\label{lem:bicliq_hom_functor}
$\BiCliq[-,-] : \BiCliq^{op} \times \BiCliq \to \BiCliq$ is a well-defined functor. In particular:
\[
\begin{tabular}{ll}
$(\emptyset : \rG \to \rG_1) \fatsemi \rR = \emptyset : \rG \to \rG_2$
&
$\rR \fatsemi (\emptyset : \rG_2 \to \rH) = \emptyset : \rG_1 \to \rH$
\\[1ex]
$(\rR_1 \cup \rR_2) \fatsemi \rR = (\rR_1 \fatsemi \rR) \cup (\rR_2 \fatsemi \rR)$
&
$\rR \fatsemi (\rS_1 \cup \rS_2) = (\rR \fatsemi \rS_1) \cup (\rR \fatsemi \rS_2)$
\end{tabular}
\]
for any $\BiCliq$-morphisms $\rR : \rG_1 \to \rG_2$, $(\rR_i : \rG \to \rG_1)_{i = 1,2}$ and $(\rS_i : \rG_2 \to \rH)_{i = 1,2}$.
\end{lemma}

\begin{proof}
Each $\BiCliq[\rG,\rH]$ is a well-defined join-semilattice by Lemma \ref{lem:bicliq_mor_jsl_struct}. Take any $\BiCliq$-morphisms $\rR : \rG \to \rH$ and $\rS : \rG' \to \rH'$. Then functorality follows if $\BiCliq[\rR^{op},\rS] = \lambda \rT. \rR \fatsemi \rT \fatsemi \rS$ is a well-defined join-semilattice morphism. The bottom element is preserved:
\[
\rR \fatsemi (\emptyset \fatsemi \rS)
= \rR \fatsemi (\emptyset ; \rS_+\spbreve)
= \rR \fatsemi \emptyset
= \rR ; \emptyset_+\spbreve
= (\rR_- ; \rH) ; \BC{\cl_{\breve{\rH}}(\emptyset),\rH'_t}
= \rR_- ; \emptyset
= \emptyset
\]
recalling that $\cl_{\breve{\rH}}(\emptyset)$ is the set of isolated elements in $\rH_t$. Next, $\rR \fatsemi (\rT_1 \cup \rT_2) = (\rR \fatsemi \rT_1) \cup (\rR \fatsemi \rT_2)$ because:
\[
(\rR \fatsemi (\rT_1 \cup \rT_2))^\up(X)
= (\rT_1 \cup \rT_2)^\up \circ (\rG')^\down \circ \rR^\up(X)
= \bigcup_{i = 1,2} \rT_i^\up \circ (\rG')^\down \circ \rR^\up(X)
= \bigcup_{i = 1,2} (\rR \fatsemi \rT_i)^\up (X)
\]
Since the self-duality of $\BiCliq$ is relational converse (which preserves unions), we immediately deduce that $(\rT_1 \cup \rT_2) \fatsemi \rS = (\rT_1 \fatsemi \rS) \cup (\rT_2 \fatsemi \rS)$. Thus $\BiCliq[\rR^{op},\rS]$ preserves binary unions, as desired.
\end{proof}

\begin{note}[Isomorphisms between join-semilattices of morphisms]
\item

Using the equivalence functors $\Pirr$ and $\Open$ and also their respective natural isomorphisms $rep_\aQ$ and $red_\rG$, one can describe explicit join-semilattice isomorphisms:
\[
\JSL_f[\aQ,\aR] \cong \BiCliq[\Pirr\aQ,\Pirr\aR]
\quad\text{and}\quad
\BiCliq[\rG,\rH] \cong \JSL_f[\Open\rG,\Open\rH]
\]
By Theorem \ref{thm:jsl_bicliq_equiv_without_irr} we also know that $\Nleq$ and $\Open$ define an equivalence of categories, yielding the join-semilattice isomorphisms $\JSL_f[\aQ,\aR] \cong \BiCliq[\Nleq\aQ,\Nleq\aR]$ described directly below. \endbox
\end{note}

\smallskip

\begin{lemma}
\label{lem:jsl_bicliq_hom_spec_iso}
For each pair of finite join-semilattices $(\aQ,\aR)$ we have the join-semilattice isomorphism:
\[
\begin{tabular}{c}
$\rho_{\aQ,\aR} : \JSL_f[\aQ,\aR] \to \BiCliq[\nleq_\aQ , \, \nleq_\aR ]$
\\[1ex]
$\rho_{\aQ,\aR} (f) := \{ (q,r) \in Q \times R : f(q) \nleq_\aR r \}$
\qquad
$\rho_{\aQ,\aR}^{\bf-1}(\rR) := \lambda q \in Q. \Land_\aR \overline{\rR}\,[q]$
\end{tabular}
\]
where $\nleq_\aQ \; \subseteq Q \times Q$ and $\nleq_\aR \; \subseteq R \times R$. In particular, $f \leq_{\JSL_f[\aQ,\aR]} g \iff \rho_{\aQ,\aR}(f) \subseteq \rho_{\aQ,\aR}(g)$.
\end{lemma}

\begin{proof}
Recall Theorem \ref{thm:jsl_bicliq_equiv_without_irr} i.e.\ the categorical equivalence via functors $\Nleq : \JSL_f \to \BiCliq$ and $\Open$, where $\Nleq \aQ := \; \nleq_\aQ \; \subseteq Q \times Q$ and $\Nleq f := \{ (q,r) \in Q \times R : f(q) \nleq_\aR r \}$. Fixing $(\aQ,\aR)$ then $\Nleq$ restricts to the bijective function $\rho_{\aQ,\aR}(f) = \Nleq f$, and clearly $\rho_{\aQ,\aR}^{\bf-1}$ is its functional inverse. Then it suffices to establish that this bijection $\rho_{\aQ,\aR}$ defines an order-embedding i.e.\
\[
f \leq_{\JSL_f[\aQ,\aR]} g
\iff
\Nleq f \subseteq \Nleq g
\]
Regarding $(\To)$, by assumption $f(q) \leq_\aR g(q)$ for all $q \in Q$. Then $\Nleq f(q,r)$ means that $f(q) \nleq_\aR r$ hence $g(q) \nleq_\aR r$ (else contradiction), so that $\Nleq g(q,r)$. Concerning $(\oT)$, suppose that $\Nleq f \subseteq \Nleq q$. Then:
\[
f(q) = \Land_\aR \overline{\Nleq f[q]} \leq_\aR \Land_\aR \overline{\Nleq g[q]} = g(q)
\]
because $\overline{\Nleq g[q]} \subseteq \overline{\Nleq f[q]}$ i.e.\ we have fewer summands.
\end{proof}

This permits an alternative description of the tensor product $\aQ \tenp \aR$. The name `bi-ideal' already exists in the literature, and the following definition will be shown to coincide with the pre-existing notion.

\begin{definition}[Bi-ideals over a pair of finite join-semilattices]
\label{def:bi_ideal}
\item
Given a pair of finite join-semilattices $(\aQ,\aR)$, define:
\[
\BId{\aQ,\aR}
:= \{ \overline{\rR} \subseteq Q \times R : \rR \in \BiCliq(\nleq_\aQ,\ngeq_\aR) \}
\]
and call them the \emph{bi-ideals over $(\aQ,\aR)$}. Note that $\ngeq_\aR \; = \; \nleq_{\aR^{\pOp}}$, so we are taking relative complements of elements of the join-semilattice $\BiCliq[\nleq_\aQ,\nleq_{\aR^{\pOp}}]$. In words, a \emph{bi-ideal over $(\aQ,\aR)$} is the complement relation of a $\BiCliq$-morphism of type $\nleq_\aQ \; \to \; \ngeq_\aR$. Ordering them by inclusion uniquely determines a join-semilattice:
\[
\jslBId{\aQ,\aR} := (\BId{\aQ,\aR},\lor_{\jslBId{\aQ,\aR}},\bot_{\jslBId{\aQ,\aR}})
\]
where $\land_{\jslBId{\aQ,\aR}} \, = \, \cap$ and $\top_{\jslBId{\aQ,\aR}} = Q \times R$. Then the join constructs the intersection of all bi-ideals containing the summands as subsets, and we also have the explicit description $\bot_{\jslBId{\aQ,\aR}} = \{ \bot_\aQ \} \times R \, \cup \, Q \times \{ \bot_\aR \}$. \endbox
\end{definition}

\begin{note}
$\jslBId{\aQ,\aR}$'s meets are intersections because $\BiCliq(\nleq_\aQ,\ngeq_\aR)$ is closed under arbitrary unions by Lemma \ref{lem:bicliq_mor_jsl_struct}. Concerning the bottom element, it is necessarily the relative complement of:
\[
\begin{tabular}{lll}
$\top_{\BiCliq[\nleq_\aQ,\ngeq_\aR]}$
&
$= \Nleq \top_{\JSL_f[\aQ,\aR^{\pOp}]}$
& by Lemma \ref{lem:jsl_bicliq_hom_spec_iso}
\\&
$= \{ (q,r) \in Q \times R : \top_{\JSL_f[\aQ,\aR^{\pOp}]}(q) \nleq_{\aR^{\pOp}} r \}$
\\&
$= \{ (q,r) \in Q \times R : r \nleq_\aR \top_{\JSL_f[\aQ,\aR^{\pOp}]}(q) \}$
\end{tabular}
\]
Concerning the condition $r \nleq_\aR \top_{\JSL_f[\aQ,\aR^{\pOp}]}(q)$,
\begin{quote}
If $q = \bot_\aR$ then $r \nleq_\aR \bot_{\aR^{\pOp}} = \top_\aR$ which never holds.
\qquad
If $q \neq \bot_\aR$ then $r \nleq_\aR \top_{\aR^{\pOp}} = \bot_\aR$ which holds iff $r \neq \bot_\aR$.
\end{quote}
Thus $\bot_{\jslBId{\aQ,\aR}} = \{ \bot_\aQ \} \times R \, \cup\, Q \times \{ \bot_\aR \}$, as previously stated. \endbox
\end{note}

We may reinterpret the tensor product $\aQ \tenp \aR$ as the collection of bi-ideals over $(\aQ,\aR)$ ordered by inclusion.

\begin{lemma}[Tensor product as join-semilattice of bi-ideals]
\label{lem:tenp_as_bi_ideals}
\item
We have the join-semilattice isomorphism:
\[
\begin{tabular}{c}
$bid_{\aQ,\aR} : \aQ \tenp \aR \to \jslBId{\aQ,\aR}$
\\[1ex]
\begin{tabular}{ll}
$bid_{\aQ,\aR}(h : \aQ \to \aR^{\pOp})$ & $:= \{ (q,r) \in Q \times R : r \leq_\aR h(q) \}$
\\[0.5ex]
$bid_{\aQ,\aR}^{\bf-1}(\rR \subseteq Q \times R)$ & $:= \lambda q. \Lor_\aR \rR[q]$
\end{tabular}
\end{tabular}
\]
\end{lemma}

\begin{proof}
Recalling that $\aQ \tenp \aR := (\JSL_f[\aQ,\aR^{\pOp}])^{\pOp}$, first observe:
\[
bid_{\aQ,\aR} = \neg_{Q \times R} \circ \rho_{\aQ,\aR^{\pOp}}
\]
where $\rho_{\aQ,\aR^{\pOp}} : \JSL_f[\aQ,\aR^{\pOp}] \to \BiCliq[\nleq_\aQ,\nleq_{\aR^{\pOp}}]$ is the isomorphism from Lemma \ref{lem:jsl_bicliq_hom_spec_iso}.  Thus $bid_{\aQ,\aR}$ is bijective, and:
\[
\begin{tabular}{lll}
$h_1 \leq_{\aQ \tenp \aR} h_2$
&
$\iff h_2 \leq_{\JSL_f[\aQ,\aR^{\pOp}]} h_1$
\\&
$\iff \rho_{\aQ, \aR}(h_2) \subseteq \rho_{\aQ,\aR}(h_1)$
& by Lemma \ref{lem:jsl_bicliq_hom_spec_iso}
\\[1ex]&
$\iff \overline{\rho_{\aQ, \aR}(h_1)} \subseteq \overline{\rho_{\aQ,\aR}(h_2)}$
\end{tabular}
\]
so it is an order-isomorphism. The description of its inverse is immediate.
\end{proof}

Let us now prove that bi-ideals correspond to the classical concept.

\begin{lemma}[Inductive description of bi-ideals]
\label{lem:bi_ideal_classical}
\item
A relation $\rR \subseteq Q \times R$ defines a bi-ideal over $(\aQ,\aR)$ iff the following three statements hold:
\begin{enumerate}[(a)]
\item
$\bot_{\jslBId{\aQ,\aR}} \subseteq \rR$.
\item
$\rR$ is down-closed inside $\aQ \times \aR$ i.e.\
\[
\dfrac{\rR(q_1,r_1) \quad q_2 \leq_\aQ q_1 \quad r_2 \leq_\aR r_1}
{\rR(q_2,r_2)}
\]

\item
$\rR$ is closed under `lateral joins' i.e.\
\[
\dfrac{\rR(q,r_1) \quad \rR(q,r_2)}{\rR(q,r_1 \lor_\aR r_2)}
\qquad
\dfrac{\rR(q_1,r) \quad \rR(q_2,r)}{\rR(q_1 \lor_\aQ q_2,r)}
\]
\end{enumerate}
\end{lemma}

\begin{proof}
\item
\begin{enumerate}
\item
We first show that every bi-ideal $\rR \in \BId{\aQ,\aR}$ satisfies the above three statements. (a) is immediate by well-definedness of the inclusion-ordered join-semilattice $\jslBId{\aQ,\aR} = (\BId{\aQ,\aR},\lor_{\jslBId{\aQ,\aR}},\bot_{\jslBId{\aQ,\aR}})$. Next, by Lemma \ref{lem:tenp_as_bi_ideals} there exists a join-semilattice morphism $h : \aQ \to \aR^{\pOp}$ such that $\rR = \{ (q,r) \in Q \times R : r \leq_\aR h(q) \}$. Thus (b) holds because if $(q_2,r_2) \leq_{\aQ \times \aR} (q_1,r_1)$ then:
\[
r_2 \leq_\aR r_1 \leq_\aR h(q_1) \leq_\aR h(q_2)
\]
noting that $q_2 \leq_\aQ q_1$ implies $h(q_2) \leq_{\aR^{\pOp}} h(q_1)$. Finally, (c) also follows easily. That is:
\[
\rR(q,r_1)\,\land\rR(q,r_2)
\implies
r_1,\,r_2 \leq_\aR f(q)
\implies
r_1 \lor_\aR r_2 \leq_\aR h(q)
\implies
\rR(q,r_1 \lor_\aR r_2)
\]
\[
\rR(q_1,r)\,\land\,\rR(q_2,r)
\implies
r \leq_\aQ f(q_1),\,f(q_2)
\implies
r \leq_\aQ f(q_1 \lor_\aQ q_2)
\implies
\rR(r,q_1 \lor_\aQ q_2) 
\]

\item
Conversely take any relation $\rR \subseteq Q \times R$ satisfying the three statements above. By Lemma \ref{lem:tenp_as_bi_ideals} it suffices to construct a join-semilattice morphism $h : \aQ \to \aR^{\pOp}$ such that $\rR(q,r) \iff r \leq_\aR h(q)$, so define:
\[
h : Q \to R
\qquad
h(q) := \Lor_\aR \rR[q]
\]
Then using (a) we have $h(\bot_\aQ) = \Lor_\aR R = \top_\aR = \bot_{\aR^{\pOp}}$, so it remains to prove preservation of joins i.e.\
\[
x := \Lor_\aR \rR[q_1 \lor_\aQ q_2] \stackrel{?}{=} (\Lor_\aR \rR[q_1]) \land_\aR (\Lor_\aR \rR[q_2]) =: y
\]
for any fixed $q_1$, $q_2 \in Q$. First observe that:
\[
\rR[q_1 \lor_\aQ q_2] 
= \{ r \in R : \rR(q_1 \lor_\aQ q_2, r) \}
= \{ r \in R : \rR(q_1,r) \,\land\,\rR(q_2,r) \}
= \rR[q_1] \cap \rR[q_2]
\]
because $\rR(q_1 \lor_\aQ q_2,r) \iff  \rR(q_1,r) \,\land\,\rR(q_2,r)$ follows by downwards-closure (b), and closure under lateral joins (c). Then since $\rR[q_1] \cap \rR[q_2] \subseteq \rR[q_i]$ for $i = 1,2$ we deduce that $x \leq_\aR y$.

\smallskip
In order to prove $y \leq_\aR x$, observe that $\rR[q] \in \Lor_\aR \rR[q]$ for every $q \in Q$. This follows because $\rR[q] = \{ r \in R : \rR(q,r) \}$ is finite, so we can apply closure under lateral joins in the second component iteratively to deduce $\rR(q,\Lor_\aR \rR[q])$. Thus $\rR(q_i,\Lor_\aR \rR[q_i])$ and hence $\rR(q_i,\Lor_\aR \rR[q_1] \, \land_\aR \, \Lor_\aR \rR[q_2])$ for $i = 1,2$ by downwards-closure. Applying closure under lateral joins in the first component yields:
\[
\rR(q_1 \lor_\aQ q_2,\Lor_\aR \rR[q_1] \, \land_\aR \, \Lor_\aR \rR[q_2])
\]
and hence $\Lor_\aR \rR[q_1] \, \land_\aR \, \Lor_\aR \rR[q_2] \leq_\aR \Lor_\aR \rR[q_1 \lor_\aR q_2]$ as required.
\end{enumerate}
\end{proof}

\begin{corollary}
\label{cor:bi_ideal_jsl_join}
For any collection of bi-ideals $S \subseteq \BId{\aQ,\aR}$ we have:
\[
\Lor_{\jslBId{\aQ,\aR}} S = \bot_{\jslBId{\aQ,\aR}} \,\cup\, \bigcup_{n \geq 0} S_n
\]
where $S_0 := \bigcup S$ and, for each $n \geq 0$, $S_{n+1}$ is the downwards-closure in $\aQ \times \aR$ of all lateral-joins of $S_n$.
\end{corollary}

\begin{proof}
Denote the left-hand-side by $X$ and the right-hand-side by $Y \subseteq Q \times R$. Then since $\bot_{\jslBId{\aQ,\aR}} \subseteq X$ and $\forall \rR \in S. \rR \subseteq X$ (i.e.\ $\bigcup S \subseteq X$), it follows by Theorem \ref{lem:bi_ideal_classical} that every $S_n \subseteq X$, so that $\bigcup S \subseteq Y \subseteq X$. Then it remains to show that $Y$ is a bi-ideal, and we'll again use Theorem \ref{lem:bi_ideal_classical}. Certainly $\bot_{\jslBId{\aQ,\aR}} \subseteq Y$, and the union of  down-closed sets is down-closed. Since each $S_n \subseteq S_{n+1}$ we also deduce closure under lateral joins e.g.\ given $(q,r_1) \in S_m$ and $(q,r_2) \in S_n$ then $(q,r_i) \in S_{max(m,n)}$ for $i = 1,2$, hence $(q,r_1 \lor_\aR r_2) \in S_{max(m,n)} \subseteq Y$.
\end{proof}

\begin{theorem}[Universal property of tensor product]
\label{thm:tenp_universal}
\item
For each bimorphism $\beta$ from $(\aQ,\aR)$ to $\aS$, there is a $\JSL_f$-morphism $f : \aQ \otimes \aR \to \aS$ with action:
\[
\begin{tabular}{lll}
$f(\aQ \xto{h} \aR^{\pOp})$
&
$:=$
&
$\Lor_\aS \{ \beta(q,r) : (q,r) \in bid_{\aQ,\aR}(h) \}$
\\[0.5ex] &=&
$\Lor_\aS \{ \beta(q,r) : r \leq_\aR h(q)\}$
\\[0.5ex] &=&
$\Lor_\aS \{ \beta(q,r) : \beta_{\aQ,\aR}(q,r) \leq_{\aQ \tenp \aR} h \}$
\\[0.5ex] &=&
$\Lor_\aS \{ \beta(j_q,j_r) : \beta_{\aQ,\aR}(j_q,j_r) \leq_{\aQ \tenp \aR} h, \, (j_q,j_r) \in J(\aQ) \times J(\aR) \}$.
\end{tabular}
\]
It is the unique join-semilattice morphism $f$ with typing $\aQ \tenp \aR \to \aS$ such that $\beta_f = \beta$.
\end{theorem}


\begin{proof}
We first explain why the four descriptions of $f$'s action are equivalent. The first equality follows by Lemma \ref{lem:tenp_as_bi_ideals} i.e.\ the definition of $bid_{\aQ,\aR}$. The second equality follows by the calculation:
\[
\begin{tabular}{lll}
$\beta_{\aQ,\aR}(q,r) \leq_{\aQ \tenp \aR} h$
&
$\iff \; \down_{\aQ,\aR^{\pOp}}^{q,r} \;\; \leq_{(\JSL_f[\aQ,\aR^{\pOp}])^{\pOp}} h$
\\&
$\iff  h \leq_{\JSL_f[\aQ,\aR^{\pOp}]} \; \down_{\aQ,\aR^{\pOp}}^{q,r}$
\\&
$\iff \forall q' \in Q.\; h(q') \leq_{\aR^{\pOp}} \; \down_{\aQ,\aR^{\pOp}}^{q,r}(q') $
\\&
$\iff \forall q' \in Q. \; \down_{\aQ,\aR^{\pOp}}^{q,r}(q')  \leq_\aR h(q') $
\\&
$\iff \forall q' \leq_\aQ q. \; r \leq_\aR h(q')$
& since $\forall q' \in Q.\top_{\aR^{\pOp}} = \bot_\aR \leq_\aR h(q')$
\\&
$\iff  r \leq_\aR h(q)$
& since $h : \aQ \to \aR^{\pOp}$ monotonic
\end{tabular}
\]
The third equality follows via the bilinearity of both $\beta$ and $\beta_{\aQ,\aR}$, see Lemma \ref{lem:tenp_basic}.

\smallskip
Using the first description of $f$, we'll show that it preserves all joins. Take any collection of morphisms $H = \{h_i : \aQ \to \aR^{\pOp}: i \in I \}$ and let $S := \{ bid_{\aQ,\aR}(h_i) : i \in I \}$. Then we have:
\[
\begin{tabular}{lll}
$f(\Lor_{\aQ \tenp \aR} H)$
&
$= \Lor_\aS \{ \beta(q,r) : (q,r) \in bid_{\aQ,\aR}(\Lor_{\aQ \tenp \aR} H) \}$
\\&
$= \Lor_\aS \{ \beta(q,r) : (q,r) \in \Lor_{\jslBId{\aQ,\aR}} S \}$
& by Lemma \ref{lem:tenp_as_bi_ideals}
\\&
$= \Lor_\aS \{ \beta(q,r) : (q,r) \in \bot_{\jslBId{\aQ,\aR}} \cup \bigcup_{n \geq 0} S_n \}$
& by Corollary \ref{cor:bi_ideal_jsl_join}
\\&
$= \Lor_\aS \{ \beta(q,r) : (q,r) \in \bigcup_{n \geq 0} S_n \}$
& via bilinearity of $\beta$
\\&
$= \Lor_\aS \{ \beta(q,r) : (q,r) \in S_0 \}$
& see below
\\&
$= \Lor_\aS \{ \beta(q,r) : \exists i \in I. (q,r) \in bid_{\aQ,\aR}(h_i)  \}$
& since $S_0 = \bigcup S$
\\&
$= \Lor_\aS \{ f(h_i) : i \in I \}$
& by associativity
\\&
$= \Lor_\aS f[H]$.
\end{tabular}
\]
The marked equality follows because if $x_n := \Lor_\aS \{ \beta(q,r) : (q,r) \in S_n \}$ then we have $x_n \leq_\aS x_0$ for every $n \geq 0$. That is, adding lateral joins and taking the down-closure can be `mirrored' inside $x_0$ using the bilinearity of $\beta$, since we may add the appropriate summands without altering the value of $x_0$.

\smallskip
Finally, using the third description of $f$ we'll show that $f(\beta_{\aQ,\aR}(q_0,r_0)) = \beta(q_0,r_0)$:
\[
\begin{tabular}{lll}
$f(\beta_{\aQ,\aR}(q_0,r_0))$
&
$= \Lor_\aS \{ \beta(q,r) : \beta_{\aQ,\aR}(q,r) \leq_{\aQ \tenp \aR} \beta_{\aQ,\aR}(q_0,r_0) \}$
\\&
$= \Lor_\aS \{ \beta(q,r) : \beta_{\aQ,\aR}(q,r) \leq_{\aQ \tenp \aR} \beta_{\aQ,\aR}(q_0,r_0), \, q \neq \bot_\aQ, \, r \neq \bot_\aR \}$
& bilinearity of $\beta$
\\&
$= \Lor_\aS \{ \beta(q,r) : (q,r) \leq_{\aQ \times \aR} (q_0,r_0),  \, q \neq \bot_\aQ, \, r \neq \bot_\aR \}$
& by Lemma \ref{lem:tenp_basic}.3
\\&
$= \beta(q_0,r_0)$
& bilinearity of $\beta$.
\end{tabular}
\]
Note that since $\beta_{\aQ,\aR}[Q \times R]$ join-generates $\aQ \tenp \aR$ by Lemma \ref{lem:tenp_basic}.2, there can only be one join-semilattice morphism extending $\beta$ in this way.
\end{proof}

\begin{lemma}
\label{lem:tenp_more_basic}
\item
\begin{enumerate}
\item
We have the following natural isomorphisms.
\begin{enumerate}
\item
$i_\aQ :\aQ \to \two \tenp \aQ$ i.e.\ `the unit' arises by applying $(-)^\pOp$ to the element-morphism:
\[
\aQ^{\pOp} \xto{\elem{\aQ^{\pOp}}{-}} \JSL_f[\two,\aQ^{\pOp}] = (\two \tenp \aQ)^{\pOp}
\]

\item
$\pi_{\aQ,\aR,\aS} : (\aQ \tenp \aR) \tenp \aS \to \aQ \tenp (\aR \tenp \aS)$  i.e.\ `associativity' arises by  applying $(-)^{\pOp}$ to the universality of the tensor product:
\[
((\aQ \tenp \aR) \tenp \aS)^{\pOp}
= \JSL_f[\aQ \tenp \aR,\aS^{\pOp}]
\to
\JSL_f[\aQ,\JSL_f[\aR,\aS^{\pOp}]]
= (\aQ \tenp (\aR \tenp \aS))^{\pOp}
\]

\item
$\tau_{\aQ,\aR} : \aQ \tenp \aR \to \aR \tenp \aQ$ i.e.\ `commutativity' arises by applying $(-)^{\pOp}$ to the duality isomorphism between  internal-homs:
\[
(\aQ \tenp \aR)^{\pOp} =  \JSL_f[\aQ,\aR^{\pOp}] \to \JSL_f[\aQ^{\pOp},\aR] = (\aR \tenp \aQ)^{\pOp}
\qquad\qquad
\aQ \xto{h} \aR^{\pOp} \; \mapsto \; \aR \xto{h_*} \aQ^{\pOp}
\]

\item
$d_{\aQ,\aR,\aS} : (\aQ \times \aR) \tenp \aS \to (\aQ \tenp \aS) \times (\aR \tenp \aS)$ i.e.\ `distributivity' arises by applying  $(-)^{\pOp}$ to the universality of the (co)product:
\[
((\aQ \times \aR) \tenp \aS)^{\pOp}
= \JSL_f[\aQ \times \aR,\aS^{\pOp}]
\to \JSL_f[\aQ,\aS^{\pOp}] \times \JSL_f[\aR,\aS^{\pOp}]
= ((\aQ \tenp \aS) \times (\aR \tenp \aS))^{\pOp}
\]
\end{enumerate}

\item
Given $|Q|$, $|R| \geq 2$, then:
\begin{enumerate}
\item
 $\aQ$ and $\aR$ are boolean join-semilattices iff $\aQ \tenp \aR$ is a boolean join-semilattice,
\item
$\aQ$ and $\aR$ are distributive join-semilattices iff $\aQ \tenp \aR$ is a distributive join-semilattice.
\end{enumerate}
\end{enumerate}
\end{lemma}

\begin{proof}
\item
\begin{enumerate}
\item
ok
\item
\begin{enumerate}
\item
If $\aQ$ and $\aR$ are boolean then iteratively apply $d_{\aQ,\aR,\aS}$. Conversely, if (w.l.o.g.) $\aQ$ is not boolean then there exists $j_1 \in J(\aQ)$ which is not an atom, and by assumption some $j_2 \in J(\aR)$. Then since $\bot_\aQ <_\aQ x <_\aQ j_1$, 
\[
\bot_{\aQ \tenp \aR} \; <_{\aQ \tenp \aR} \; \down_{\aQ^{\pOp},\aR}^{x,j_2} \; <_{\aQ \tenp \aR} \; \down_{\aQ^{\pOp},\aR}^{j_1,j_2}
\]
thus the latter join-irreducible element is not an atom.

\item
If $\aQ$ and $\aR$ are distributive then their join-irreducibles are join-prime, and since:
\[
\down_{\aQ^{\pOp},\aR}^{j,j'} \; \leq_{\aQ \tenp \aR} 
\Lor_{\aQ \tenp \aR} \{ \down_{\aQ^{\pOp},\aR}^{j_i,j'_i} \; : i \in I  \}
\iff 
j \leq_\aQ \Lor_\aQ \{ j_i : i \in I \}
\text{ and }
j' \leq_\aR \Lor_\aR \{ j'_i : i \in I \}
\]
it follows that every join-irreducible in $\aQ \tenp \aR$ is join-prime, hence the latter is distributive. Conversely if (w.l.o.g.) $\aQ$ is not distributive then there exists $j \in J(\aQ)$ which is not join-prime. By fixing $j' \in J(\aR)$ one can show that the join-irreducible $\down_{\aQ^{\pOp},\aR}^{j,j'}$ is not join-prime, hence  $\aQ \tenp \aR$ is not distributive.
\end{enumerate}
\end{enumerate}
\end{proof}

\begin{example}[Morphisms obtained via bilinearity]
\item
\begin{enumerate}
\item
Evaluation map $evl : \JSL_f[\aQ,\aR] \tenp \aQ \to \aR$.
\item
Internal composition $cmp : \JSL_f[\aR,\aS] \tenp \JSL_f[\aQ,\aR] \to \JSL_f[\aQ,\aS]$.
\item
Approximation from above $tig : \aR \tenp \aQ^{\pOp} \to \JSL_f[\aQ,\aR]$. \endbox
\end{enumerate}
\end{example}

\smallskip

\begin{note}[Addendum]
  \item
  \begin{enumerate}
  \item
  Theorem \ref{thm:tenp_universal} actually defines a natural isomorphism:
  \[
  \JSL_f[\aQ \tenp \aR,\aS] \cong \JSL_f[\aQ,\JSL_f[\aR,\aS]].
  \]
  
  \item
  Fraser also has a characterisation of:
  \[
  \beta_{\aQ,\aR}(q,r) \leq_{\aQ \tenp \aR} \Lor_{\aQ \tenp \aR} 
  \{ \beta_{\aQ,\aR}(q_i,r_i) : i \in I \}
  \]
  i.e.\ it holds iff there exists a lattice term $\phi$ in variables $I$ such that:
  \[
  q \leq_\aQ \sem{\phi[i \mapsto q_i]}_\aQ
  \qquad\text{and}\qquad
  r \leq_\aR \sem{\phi^d[i \mapsto r_i]}_\aR
  \]
  where $\phi^d$ is obtained from $\phi$ by swapping the joins/meets. \endbox
  \end{enumerate}
\end{note}

\subsection{Tight morphisms and tight tensors}

In this subsection we define \emph{tight join-semilattice morphisms}. We describe their join/meet-irreducibles and define:
\begin{enumerate}
\item
the \emph{tight hom-functor} $\jslTight{-,-}$ which is a subfunctor of $\JSL_f[-,-]$.
\item
the \emph{tight tensor product} $\ttenp : \JSL_f \times \JSL_f \to \JSL_f$.
\end{enumerate}

In the next subsection we'll describe the synchronous product functor $\syncp : \BiCliq \times \BiCliq \to \BiCliq$, which may also be viewed as the Kronecker product of binary matrices over the boolean semiring. We shall then prove that the tight tensor product and the synchronous product are essentially the same concepts. In the final subsection we'll consider the notions of `tightness' inside $\BiCliq$ and prove the universal property of the synchronous product, and hence also of the tight tensor product.

\begin{definition}[Tight join-semilattice morphisms]
\label{def:tight_jsl_mor}
\item
A $\JSL_f$-morphism $f : \aQ \to \aR$ is \emph{tight} if it factors through some $\JPow Z \in \JSL_f$ i.e.\
\[
f \quad = \quad
\aQ \xto{\alpha} \JPow Z \xto{\beta} \aR
\]
for $\JSL_f$-morphisms $\alpha$, $\beta$. Equivalently, $f$ factors through some boolean join-semilattice inside $\JSL_f$. \endbox
\end{definition}

\smallskip
Then every morphism from or to $\JPow Z \cong \two^Z \cong \two^{|Z|}$ is tight, as are the special morphisms:
\[
\up_{\aQ,\aR}^{q,r} \quad = \quad \aQ \xto{\ideal{\aQ}{q}} \two \xto{\elem{\aR}{r}} \aR.
\]
Furthermore each special morphism $\down_{\aQ,\aR}^{q,r}$ is also tight, see Corollary \ref{cor:tight_special_morphisms} below. Before characterising tight morphisms, the following Lemma provides plenty of non-examples.

\begin{lemma}
\label{lem:jsl_iso_not_tight}
A $\JSL_f$-isomorphism $f : \aQ \to \aR$ is tight iff both $\aQ$ and $\aR$ are distributive.
\end{lemma}

\begin{proof}
Given a tight $\JSL_f$-isomorphism $f : \aQ \to \aR$ then:
\[
id_\aQ 
= f^{\bf-1} \circ f
\quad = \quad
\aQ \xto{\alpha} \JPow Z \xto{\beta} \aR \xto{f^{\bf-1}} \aQ
\]
for some morphisms $\alpha$, $\beta$. Thus $\aQ$ is a join-semilattice retract of a boolean join-semilattice, so by Lemma \ref{lem:std_order_theory}.15 we deduce that $\aQ$ is distributive. Hence $\aR$ is also distributive, since $\JSL_f$-isomorphisms are also lattice isomorphisms. Conversely, suppose that $\aQ$ and $\aR$ are distributive. Then again by Lemma \ref{lem:std_order_theory}.15 we know that $\aQ$ is a join-semilattice retract of some $\JPow Z$, so that $f = f \circ id_\aQ = f \circ r \circ e$ is tight.
\end{proof}

We also briefly observe that tight morphisms are closed under tensor products.

\begin{lemma}
If $(f_i : \aQ_i \to \aR_i)_{i = 1,2}$ are tight then $f_1 \tenp f_2 : \aQ_1 \tenp \aQ_2 \to \aR_1 \tenp \aR_2$ is tight.
\end{lemma}

\begin{proof}
Follows because the tensor product of boolean join-semilattices is boolean, see Lemma \ref{lem:tenp_more_basic}.2.
\end{proof}

\begin{note}
Let us recall some basic terminology, used in the proof of Lemma \ref{lem:tight_mor_char} directly below. For any finite set $Z$ we have the join-semilattice:
\[
\two^Z = (\Set(Z,2),\lor_{\two^Z},\bot_{\two^Z})
\]
whose elements are all functions $Z \to 2 = \{0,1\}$, whose join is the pointwise join inside $\two$, and whose bottom element is necessarily $\lambda z \in Z. 0$. Every finite boolean join-semilattice is isomorphic to such an algebra, since $\JPow Z = (\Pow Z,\cup,\emptyset) \cong \two^Z$ via the mapping:
\[
S \subseteq Z \; \mapsto \; \lambda z \in Z.(z \in S) \;?\; 1 : 0
\]
i.e.\ a subset is sent to its indicator function. \endbox
\end{note}

\begin{lemma}[Characterisation of tight morphisms]
\label{lem:tight_mor_char}
\item
For any $\JSL_f$-morphism $f : \aQ \to \aR$, the following statements are equivalent.
\begin{enumerate}
\item
$f$ is tight.
\item
$f$ factors through some distributive join-semilattice inside $\JSL_f$.
\item
$f$ is a $\JSL_f[\aQ,\aR]$-join of morphisms $\up_{\aQ,\aR}^{q,r} : \aQ \to \aR$ where $q \in Q$ and $r \in R$.
\item
$f$ is a $\JSL_f[\aQ,\aR]$-join of morphisms $\up_{\aQ,\aR}^{m,j} : \aQ \to \aR$ where $m \in M(\aQ)$ and $j \in J(\aR)$.

\end{enumerate}
\end{lemma}

\begin{proof}
\item
\begin{enumerate}
\item
$(1 \iff 2)$:

Certainly (1) implies (2). Conversely, suppose $f =\; \aQ \xto{\alpha} \aD \xto{\beta} \aR$ for some finite join-semilattice $\aD$ which is distributive. By Lemma \ref{lem:std_order_theory}.15 we know every finite distributive lattice arises as the join-semilattice retract of a finite boolean join-semilattice, so that:
\[
id_\aD \quad = \quad \aD \xto{s} \JPow Z \xto{r} \aD
\]
and thus $f = (\beta \circ r) \circ (s \circ \alpha)$ implies (1).

\item
$(1 \iff 3)$:

$f : \aQ \to \aR$ is tight iff we have morphisms $\alpha$, $\beta$ such that:
\[
f 
\quad = \quad
\aQ \xto{\alpha} \two^Z \xto{\beta} \aR
\]
for some finite set $Z$. Since the coproduct and product coincide in $\JSL_f$ (in fact also in $\JSL$), we equivalently have morphisms $(\alpha_z : \aQ \to \two)_{z \in Z}$ and $(\beta_z : \two \to \aR)_{z \in Z}$ such that:
\[
f = [\beta_z]_{z \in Z} \circ \ang{\alpha_z}_{z \in Z}
\qquad
\text{where}\quad
\begin{tabular}{ll}
$\ang{\alpha_z}_{z \in Z}(q)$ & $:= \lambda z \in Z. \alpha_z(q)$
\\
$[\beta_z]_{z \in Z}(\delta : Z \to 2)$ & $:= \Lor_\aR \{ \beta_z(\delta(z)) : z\in Z \}$
\end{tabular}
\]
so that:
\[
\begin{tabular}{lll}
$f(q)$
& 
$= \Lor_\aR \{ \beta_z \circ \alpha_z(q) : z \in Z \}$
\\&
$= \Lor_\aR \{ \elem{\aR}{r_z} \circ \ideal{\aR}{q_z}(q) : z \in Z \}$
& see Definition \ref{def:elem_ideal_mor}
\\&
$= \Lor_\aR \{ \up_{\aQ,\aR}^{q_z,r_z}(q) : z \in Z \}$
& see Definition \ref{def:spec_hom_morphisms}
\end{tabular}
\]
Thus tight morphisms are precisely the joins of the special morphisms $\up_{\aQ,\aR}^{q,r}$.

\item
$(3 \iff 4)$:

It suffices to show that every special morphism $\up_{\aQ,\aR}^{q,r}$ arises as a possibly-empty join of the special morphisms $\up_{\aQ,\aR}^{m,j}$ where $m \in M(\aQ)$ and $j \in J(\aR)$. Recall the equalities from Lemma \ref{lem:special_jsl_morphisms}.1. First of all:
\[
\up_{\aQ,\aR}^{\top_\aQ,r} \; = \bot_{\JSL_f[\aQ,\aR]} = \; \up_{\aQ,\aR}^{q,\bot_\aR}
\]
so that these morphisms arise as the empty-join. Finally if $q \neq \top_\aQ$ then it arises as a non-empty meet of meet-irreducibles, and if $r \neq \bot_\aR$ then it arises as a non-empty join of join-irreducibles, so by Lemma \ref{lem:special_jsl_morphisms}.1,
\[
\begin{tabular}{lll}
$\up_{\aQ,\aR}^{q,r}$
&
$= \up_{\aQ,\aR}^{\Land_\aQ \{m \in M(\aQ) : q \leq_\aR m\},r}$
\\&
$= \Lor_{\JSL_f[\aQ,\aR]} \{ \up_{\aQ,\aR}^{m,r} \; : m \in M(\aQ), \, q \leq_\aQ m \}$
\\&
$= \Lor_{\JSL_f[\aQ,\aR]} \{ \up_{\aQ,\aR}^{m,\Lor_\aR \{ j \in J(\aQ) : j \leq_\aQ r \} } \; : m \in M(\aQ), \, q \leq_\aQ m \}$
\\&
$= \Lor_{\JSL_f[\aQ,\aR]} \{ \up_{\aQ,\aR}^{m,j} \; : m \in M(\aQ), \, j \in J(\aR), \, q \leq_\aQ m , \, j \leq_\aR r \}$
\end{tabular}
\]
\end{enumerate}
\end{proof}

\begin{corollary}
\label{cor:tight_special_morphisms}
The special morphisms $\up_{\aQ,\aR}^{q_0,r_0}, \, \down_{\aQ,\aR}^{q_0,r_0} : \aQ \to \aR$ are always tight. In fact,
\[
\down_{\aQ,\aR}^{q_0,r_0} 
\; = \; \up_{\aQ,\aR}^{\bot_\aQ,r_0} \; \lor_{\JSL_f[\aQ,\aR]} \; \up_{\aQ,\aR}^{q_0,\top_\aR}
\]
\end{corollary}

\begin{proof}
Each $\up_{\aQ,\aR}^{q_0,r_0}$ is certainly tight, since by definition it factors through $\two$. Each $\down_{\aQ,\aR}^{q_0,r_0}$ is tight because by definition it factorises through the distributive lattice $\three$, so we may apply Lemma \ref{lem:tight_mor_char}. In particular, viewing $\three$ as a join-semilattice retract of $\two^{\{1,2\}}$ leads to the above equality, which we now verify directly.
\begin{enumerate}
\item
$\up_{\aQ,\aR}^{q_0,\top_\aR}(q)$ equals $\bot_\aR$ whenever $q \leq_\aQ q_0$, otherwise it is $\top_\aR$.
\item
$\up_{\aQ,\aR}^{\bot_\aQ,r_0}(q)$ equals $\bot_\aR$ if $q = \bot_\aQ$, otherwise it is $r_0$.
\end{enumerate}
Thus their join is precisely the morphism $\down_{\aQ,\aR}^{q_0,r_0} : \aQ \to \aR$.
\end{proof}

To understand why the tight morphisms are a particularly natural subclass of the $\JSL_f$-morphisms, first observe that they determine a subfunctor of $\JSL_f[-,-]$.

\bigskip

\begin{definition}[Tight hom-functor]
\item
Given finite join-semilattices $\aQ$, $\aR$, first let $\Tight{\aQ,\aR} \subseteq \JSL_f(\aQ,\aR)$ be the subset of tight morphisms. Then we define the finite join-semilattice:
\[
\jslTight{\aQ,\aR} 
:= (\Tight{\aQ,\aR},\lor_{\jslTight{\aQ,\aR}},\bot_{\jslTight{\aQ,\aR}})
\subseteq \JSL_f[\aQ,\aR]
\]
whose join is necessarily the pointwise-join and whose bottom is necessarily $\bot_{\JSL_f[\aQ,\aR]} = \lambda q \in Q.\bot_\aR$. This extends to a functor $\jslTight{-,-} : \JSL_f^{op} \times \JSL_f \to \JSL_f$ as follows:
\[
\dfrac{f : \aQ_2 \to \aQ_1 \quad g : \aR_1 \to \aR_2}
{\jslTight{f^{op},g} := \lambda h. g \circ h \circ f    : \jslTight{\aQ_1,\aR_1} \to \jslTight{\aQ_2,\aR_2} }
\]
this being precisely the same way that $\JSL_f[-,-]$ acts, see Definition \ref{def:jsl_internal_hom}. \endbox
\end{definition}

\bigskip

\begin{lemma}
$\jslTight{-,-} : \JSL_f^{op} \times \JSL_f \to \JSL_f$ is a well-defined functor.
\end{lemma}

\begin{proof}
This follows from the well-definedness of $\JSL_f[-,-]$ and the following two observations.
\begin{enumerate}
\item
Each $\jslTight{\aQ,\aR}$ is well-defined sub join-semilattice via Lemma \ref{lem:tight_mor_char}.3 noting also that  $\bot_{\JSL_f[\aQ,\aR]} = \; \up_{\aQ,\aR}^{\top_\aQ,r}$.
\item
Tight morphisms are closed under pre/post-composition by arbitrary $\JSL_f$-morphisms, since the factorisation through a boolean join-semilattice is preserved.
\end{enumerate}
\end{proof}

Then we immediately have the following important fact:

\begin{corollary}
\label{cor:jsl_equ_tight_distributive}
Whenever $\aQ$ or $\aR$ are distributive then:
\[
\jslTight{\aQ,\aR} = \JSL_f[\aQ,\aR]
\]
\end{corollary}

\begin{proof}
Every $\JSL_f$-morphism $\aQ \to \aR$ such that either $\aQ$ or $\aR$ are distributive is tight by Lemma \ref{lem:tight_mor_char}.
\end{proof}

That tight morphisms are closed under composition with arbitrary morphisms is now further clarifed.

\begin{lemma}[Composing special morphisms with arbitrary morphisms]
\label{lem:compose_spec_gen_mor}
\item
Take any $\JSL_f$-morphisms $f : \aQ \to \aR$ and $g : \aR \to \aS$ and fix any elements $(q,r,s)  \in Q \times R \times S$.

\begin{enumerate}
\item
We have the equalities:
\[
\begin{tabular}{lll}
$\up_{\aR,\aS}^{r,s}  \circ \; f = \; \up_{\aQ,\aS}^{f_*(r),s}$
&&
$g \; \circ \up_{\aQ,\aR}^{q,r} \; = \; \up_{\aQ,\aS}^{q,g(r)}$
\end{tabular}
\]

\item
If additionally $f^{-1}(\{\bot_\aR\}) = \{\bot_\aQ\}$ and $g(\top_\aR) = \top_\aS$ then we have the equalities:
\[
\begin{tabular}{lll}
$\down_{\aR,\aS}^{r,s}  \circ \; f = \; \down_{\aQ,\aS}^{f_*(r),s}$
&&
$g \; \circ \down_{\aQ,\aR}^{q,r} \; = \; \down_{\aQ,\aS}^{q,g(r)}$
\end{tabular}
\]

\end{enumerate}
\end{lemma}

\begin{proof}
\item
\begin{enumerate}
\item
To see that the left equality holds, consider the action:
\[
\up_{\aR,\aS}^{r,s}  \circ \; f (q) 
= \begin{cases} 
\bot_{\aS} & \text{if $f(q) \leq_\aR r$}
\\
s & \text{otherwise}
\end{cases}
\qquad
\text{for each $q \in Q$}
\]
and recall that $f(q) \leq_\aR r \iff q \leq_\aQ f_*(r)$. Regarding the right equality:
\[
\begin{tabular}{lll}
$g \; \circ \up_{\aQ,\aR}^{q,r}$
&
$= ((\up_{\aQ,\aR}^{q,r})_* \circ g_*)_*$
\\&
$= (\up_{\aR^{\pOp},\aQ^{\pOp}}^{r,q} \circ g_*)_* $
& by Lemma \ref{lem:special_jsl_morphisms}.1
\\&
$= (\up_{\aS^{\pOp},\aQ^{\pOp}}^{(g_*)_*(r),q})_*$
& by left equality
\\&
$= (\up_{\aS^{\pOp},\aQ^{\pOp}}^{g(r),q})_*$
\\&
$= \; \up_{\aQ,\aS}^{q,g(r)}$
& by Lemma \ref{lem:special_jsl_morphisms}.1
\end{tabular}
\]

\item
Concerning the left equality,
\[
\down_{\aR,\aS}^{r,s}  \circ \; f
= (\up_{\aR,\aS}^{\bot_\aR,s} \;\lor\; \up_{\aR,\aS}^{r,\top_\aS}) \circ f
= (\up_{\aR,\aS}^{\bot_\aR,s} \circ \; f) \lor (\up_{\aR,\aS}^{r,\top_\aS} \circ \; f)
= \; \up_{\aR,\aS}^{f_*(\bot_\aR),s} \lor \up_{\aR,\aS}^{f_*(r),\top_\aS}
\]
Now, since $f_*(\bot_\aR) = \Lor_\aQ \{ q \in Q : f(q) \leq_\aR \bot_\aR \} = \Lor_\aQ \{\bot_\aQ\} = \bot_\aQ$ by assumption, the above equals $\down_{\aR,\aS}^{f_*(r),s}$ as desired. Finally, a similar argument yields the right equality -- this time using $g(\top_\aQ) = \top_\aR$.

\end{enumerate}
\end{proof}

The irreducible tight morphisms are easier to describe than the irreducible morphisms (Lemma \ref{lem:hom_meet_join_irr}).

\begin{lemma}[Irreducible tight morphisms]
\label{lem:irr_tight_morphisms}
For all finite join-semilattices $\aQ$, $\aR$,
\[
J(\jslTight{\aQ,\aR}) = \{ \up_{\aQ,\aR}^{m,j} \; : m \in M(\aQ), \, j \in J(\aR) \}
\qquad\qquad
M(\jslTight{\aQ,\aR}) = \{ \down_{\aQ,\aR}^{j,m} \; : j \in J(\aQ), \, m \in M(\aR) \}
\]
and hence $|J(\jslTight{\aQ,\aR})| = |M(\aQ)| \cdot |J(\aR)|$ and $|M(\jslTight{\aQ,\aR})| = |J(\aQ)| \cdot |M(\aR)|$.
\end{lemma}

\begin{proof}
\item
\begin{enumerate}
\item
Regarding join-irreducibles, Lemma \ref{lem:tight_mor_char}.4 informs us that every join-irreducible tight morphism takes the form $\up_{\aQ,\aR}^{m,j}$ where $m \in M(\aQ)$ and $j \in J(\aR)$. Finally by Lemma \ref{lem:hom_meet_join_irr}.2 we know that each such morphism is join-irreducible in $\JSL_f[\aQ,\aR]$, and hence also in the sub join-semilattice $\jslTight{\aQ,\aR}$.

\item
Concerning meet-irreducibles, recall that every $\down_{\aQ,\aR}^{q,r}$ lies in $\jslTight{\aQ,\aR}$ by Corollary \ref{cor:tight_special_morphisms}. It turns out we can completely reuse the proof of Lemma \ref{lem:hom_meet_join_irr}.1. That is, every tight morphism $f : \aQ \to \aR$ arises as the meet:
\[
\Land_{\jslTight{\aQ,\aR}} \{ \down_{\aQ,\aR}^{j,m} : j \in J(\aQ), \, m \in M(\aR), \, f(j) \leq_\aR m \}
\]
because the proof only used (i) the pointwise-ordering (again $\jslTight{\aQ,\aR}$ order-embeds into $\aR^Q$), (ii) the fact that $f$ is a join-semilattice morphism, and (iii) the usual properties of join/meet-irreducibles in $\aQ$ and $\aR$. The proof that these special morphisms do not arise as meets of other such morphisms uses only (i) their relative pointwise ordering, and (ii) the fact that if $q_0 \leq_\aQ q_1$ then the $\JSL_f[\aQ,\aR]$-meet of $\down_{\aQ,\aR}^{q_0,r_0}$ and $\down_{\aQ,\aR}^{q_1,r_1}$ is constructed pointwise. The latter point continues to hold in our setting i.e.\ their $\jslTight{\aQ,\aR}$-meet is constructed pointwise. To see this, observe that Lemma \ref{lem:special_jsl_morphisms}.5 actually shows that the pointwise meet is:
\[
\Lor_{\JSL_f[\aQ,\aR]} \{ \up_{\aQ,\aR}^{\bot_\aQ,r_0 \land_\aR r_1}, \, \up_{\aQ,\aR}^{q_0,r_0}, \, \up_{\aQ,\aR}^{q_1,\top_\aR} \}
\]
this being a tight morphism.
\end{enumerate}
\end{proof}

\begin{definition}[Tight tensor product]
\label{def:ttenp}
\item
The \emph{tight tensor product} functor $- \ttenp - : \JSL_f \times \JSL_f \to \JSL_f$ is defined as the composite functor:
\[
\JSL_f \times \JSL_f \xto{\OD_j^{op} \times \Id_{\JSL_f}} \JSL_f^{op} \times \JSL_f \xto{\jslTight{-,-}} \JSL_f
\]
There are associated canonical functions:
\[
\begin{tabular}{c}
$\beta_{\aQ,\aR}^t : Q \times R \to \Tight{\aQ^{\pOp},\aR}$
\\
where $\beta_{\aQ,\aR}^t(q_0,r_0) := \; \up_{\aQ^{\pOp},\aR}^{q_0,r_0} \; = \elem{\aR}{r_0} \circ \ideal{\aQ^{\pOp}}{q_0} = 
\begin{cases}
\bot_\aR & \text{if $q_0 \leq_\aQ q$}
\\
r_0 & \text{if $q_0 \nleq_\aQ q$}.
\end{cases}$
\end{tabular}
\]
\endbox
\end{definition}

Since $\aQ \ttenp \aR = \jslTight{\aQ^{\pOp},\aR}$, observe that Lemma \ref{lem:irr_tight_morphisms} immediately implies the following important statement.

\begin{lemma}[Irreducibles in tight tensor products]
\label{lem:tight_tensor_irr}
\item
For all finite join-semilattices $\aQ$, $\aR$,
\[
J(\aQ \ttenp \aR) = \{ \up_{\aQ^{\pOp},\aR}^{j_1,j_2} \; : j_1 \in J(\aQ), \, j_2 \in J(\aR) \}
\qquad\qquad
M(\aQ \ttenp \aR) = \{ \down_{\aQ^{\pOp},\aR}^{m_1,m_2} \; : m_1 \in M(\aQ), \, m_2 \in M(\aR) \}.
\]
Therefore $|J(\aQ \ttenp \aR)| = |J(\aQ)| \cdot |J(\aR)|$ and $|M(\aQ \ttenp \aR)| = |M(\aQ)| \cdot |M(\aR)|$.
\end{lemma}

\smallskip
We also have the following basic result.

\begin{lemma}
\label{lem:ttenp_pres_monos}
$\ttenp : \JSL_f \times \JSL_f \to \JSL_f$ preserves embeddings: $f \ttenp g$ is injective whenever both $f$ and $g$ are.
\end{lemma}

\begin{proof}
Given morphisms $f : \aQ_1 \to \aQ_2$ and $g : \aR_1 \to \aR_2$ then:
\[
f \ttenp g : \jslTight{\aQ_1^{\pOp},\aR_1} \to \jslTight{\aQ_2^{\pOp},\aR_2}
\qquad
f \ttenp g(h) := g \circ h \circ f_*.
\]
Recall that $g$ is injective iff it is $\JSL_f$-monic, and $f$ is injective iff $f_*$ is $\JSL_f$-epic. Thus if $f \ttenp g(h_1) = f \ttenp g(h_2)$ we immediately deduce that $h_1 = h_2$.
\end{proof}

\takeout{
\[
\begin{tabular}{c}
$\beta_{\aQ \otimes \aR} : Q \times R \to O(\Pirr\aQ \syncp \Pirr\aR)$
\\[0.5ex]
where $\beta_{\aQ \otimes \aR}(q,r) := \{ (m_q,m_r) \in M(\aQ) \times M(\aR) : q \nleq_\aQ m_q \text{ and } r \nleq_\aR m_r \}$
\end{tabular}
\]
}

\subsubsection{Tight morphisms: some more examples}

\begin{lemma}[Constant morphisms are tight]
\label{lem:const_mor_are_tight}
\item
For each pair of finite join-semilattices $(\aQ,\aR)$ and element $r_0 \in R$, the constant morphism:
\[
\lambda q \in Q. 
\begin{cases}
\bot_\aR & \text{if $q = \bot_\aR$}
\\
r_0 & \text{otherwise}
\end{cases} : \aQ \to \aR
\]
is a tight morphism.
\end{lemma}

\begin{proof}
This is simply the special morphism $\up_{\aQ,\aR}^{\bot_\aQ,r_0}$.
\end{proof}

Recall that for every finite distributive join-semilattice $\aQ$ we have the canonical order-isomorphism $\tau_\aQ : J(\aQ) \to M(\aQ)$ between join/meet-irreducibles, see Lemma \ref{lem:std_order_theory}.13. It extends naturally to a (tight) endomorphism of $\aQ$.

\smallskip

\begin{lemma}[Special endomorphisms of distributive join-semilattices]
\label{lem:spec_finite_distrib_endo_mor}
\item
If $\aQ$ is a finite distributive join-semilattice,
\[
\Lor_{\jslTight{\aQ,\aQ}} \, \{ \up_{\aQ,\aQ}^{q,q} : q \in Q \} \;\; |_{J(\aQ) \times M(\aQ)} \;=\; \tau_\aQ
\qquad\qquad
\Land_{\jslTight{\aQ,\aQ}} \, \{ \down_{\aQ,\aQ}^{q,q} : q \in Q \} = id_\aQ
\]
recalling that $\jslTight{\aQ,\aQ} = \JSL_f[\aQ,\aQ]$ because $\aQ$ is distributive (Corollary \ref{cor:jsl_equ_tight_distributive}).
\end{lemma}

\begin{proof}
For any join-irreducible $j \in J(\aQ)$ we have:
\[
\begin{tabular}{lll}
$(\Lor_{\jslTight{\aQ,\aQ}} \, \{ \up_{\aQ,\aQ}^{q,q} : q \in Q \})(j)$
&
$= \Lor_\aQ \{ \up_{\aQ,\aQ}^{q,q}(j) : q \in Q \}$
& join is pointwise
\\&
$= \Lor_\aQ \{ q \in Q : j \nleq_\aQ q \}$
& by definition of $\up_{\aQ,\aQ}^{q,q}$
\\&
$= \tau_\aQ(j)$
& see Lemma \ref{lem:std_order_theory}.13
\end{tabular}
\]
Regarding the second equality, first observe that $id_\aQ \leq \; \down_{\aQ,\aQ}^{q,q}$ for every $q \in Q$. The converse follows because:
\[
\Land_{\jslTight{\aQ,\aQ}} \{ \down_{\aQ,\aQ}^{q,q} \; : q \in Q \} (q') \leq_\aQ \Land_\aQ \{ \down_{\aQ,\aQ}^{q,q}(q') : q \in Q\}
= \Land_\aQ \{ q \in Q : q' \leq_\aQ q \}
= q'
\]
for each $q' \in Q$, where the first inequality follows by Lemma \ref{lem:hom_meet_bound}.
\end{proof}

\begin{lemma}[Comparing tight morphisms to arbitrary morphisms]
\label{lem:spec_mor_q_q}
\item
Take any finite join-semilattice $\aQ$ and any pair $(m,j) \in M(\aQ) \times J(\aQ)$.
\item
\begin{enumerate}
\item
The following statements hold:
\[
\exists q \in Q\backslash\{\top_\aQ\}.(\up_{\aQ,\aQ}^{q,q} \;\leq\; \up_{\aQ,\aQ}^{m,j})
\iff
m \leq_\aQ j
\qquad\qquad
\exists q \in Q\backslash\{\bot_\aQ\}.(\down_{\aQ,\aQ}^{q,q} \;\leq\; \down_{\aQ,\aQ}^{j,m})
\iff
j\leq_\aQ m
\]

\item
For any $\JSL_f$-morphism $f : \aQ \to \aQ$ we have:
\[
f \;\leq \; \down_{\aQ,\aQ}^{j,m} \;
\iff
f(j) \leq_\aQ m
\iff
\overline{\Pirr f}(j,m)
\]
and for any tight $\JSL_f$-morphism $g : \aQ \to \aQ$ we have:
\[
\up_{\aQ,\aQ}^{m,j} \; \leq g
\iff
(\nu_{\aQ,\aQ^{\pOp}}(g))_* \leq \; \down_{\aQ,\aQ}^{m,j} \;
\iff
(\nu_{\aQ,\aQ^{\pOp}}(g))_*(j) \leq_\aQ m
\iff
\overline{\Pirr \nu_{\aQ,\aQ^{\pOp}}(g)}(m,j)
\]


\end{enumerate}
\end{lemma}

\begin{proof}
\item
\begin{enumerate}
\item
Consider the left-hand equality and assume its left-hand side. Since $q \neq \top_\aQ$ by assumption and also $j \neq \bot_\aQ$ by join-irreducibility, we may apply Lemma \ref{lem:special_jsl_morphisms}.2. Thus $\up_{\aQ,\aQ}^{q,q} \;\leq\; \up_{\aQ,\aQ}^{m,j}$ if and only if $m \leq_\aQ q$ and $q \leq_\aQ j$, which certainly implies $m \leq_\aQ j$. Conversely, if $m \leq_\aQ j$ then again by Lemma \ref{lem:special_jsl_morphisms}.2 we have $\up_{\aQ,\aQ}^{m,m} \;\leq\; \up_{\aQ,\aQ}^{m,j}$, where the former is applicable because $m \neq_\aQ \top_\aQ$ by meet-irreducibility. The right-hand equality follows by a symmetric argument, using Lemma \ref{lem:special_jsl_morphisms}.4. For the second part of the argument one finds that $j \leq_\aQ m$ implies $\down_{\aQ,\aQ}^{j,j} \;\leq\; \down_{\aQ,\aQ}^{j,m}$.

\item
We calculate:
\[
\begin{tabular}{lll}
$f \; \leq \; \down_{\aQ,\aQ}^{j,m}$
&
$\iff \forall j' \in J(Q).(j' \leq_\aQ j \To f(j') \leq_\aQ m)$
& 
\\&
$\iff f(j) \leq_\aQ m$
& using monotonicity of $f$
\\&
$\iff \overline{\Pirr f}(j,m)$
& by definition of $\Pirr$
\end{tabular}
\]
recalling that the pointwise ordering is determined by the restriction to join-irreducibles. Regarding the final claim, take any tight morphism $g : \aQ \to \aQ$ and consider the composite isomorphism:
\[
\alpha :=
\qquad
(\jslTight{\aQ,\aQ})^{\pOp}
\xto{\nu_{\aQ,\aQ^{\pOp}}}
\jslTight{\aQ^{\pOp},\aQ^{\pOp}}
\xto{(-)_*}
\jslTight{\aQ,\aQ}
\]
using $\nu_{\aQ,\aQ^{\pOp}}$ from Theorem \ref{thm:tight_self_dual}. The latter informs us that $\alpha(\down_{\aQ,\aQ}^{j,m}) = (\nu_{\aQ,\aQ^{\pOp}}(\down_{\aQ,\aQ}^{j,m}))_* =\; (\up_{\aQ^{\pOp},\aQ^{\pOp}}^{j,m})_* \;=\; \up_{\aQ,\aQ}^{m,j}$, hence:
\[
\up_{\aQ,\aQ}^{m,j} \; \leq g
\iff
(\nu_{\aQ,\aQ^{\pOp}}(g))_* \leq \; \down_{\aQ,\aQ}^{m,j} \;
\iff
(\nu_{\aQ,\aQ^{\pOp}}(g))_*(j) \leq_\aQ m
\]
using the order-isomorphism and the previous claim. The final equivalence follows by definition of $\Pirr$, recalling that $\Pirr(h_*) = (\Pirr h)\spbreve$ holds generally.
\end{enumerate}
\end{proof}

\subsection{Tight tensors are essentially synchronous products}

In order to better understand the tight tensor product, we'll describe essentially the same functor inside $\BiCliq$. This turns out to be the synchronous product of binary relations, and corresponds to the Kronecker product of binary matrices over the boolean semiring \cite{WattsBooleanRankKronecker2001}.

\begin{definition}[Synchronous product functor]
\label{def:sync_product}
The \emph{synchronous product functor} $\syncp : \BiCliq \times \BiCliq \to \BiCliq$ is defined on objects as follows:
\[
\rG \syncp \rH \subseteq (\rG_s \times \rH_s) \times (\rG_t \times \rH_t)
\qquad
\rG \syncp \rH ((g_s, h_s), (g_t, h_t)) :\iff \rG(g_s, g_t) \,\land\, \rH(h_s, h_t).
\]
Its action on morphisms is the same i.e.\ given $\BiCliq$-morphisms $\rR : \rG \to \rH$ and $\rR : \rG' \to \rH'$ then:
\[
\rR \syncp \rR' : \rG \syncp \rG' \to \rH \syncp \rH'
\]
views the parameters $\rR \subseteq \rG_s \times \rH_t$ and $\rR' \subseteq \rG'_s \times \rH'_t$ as binary relations and constructs the relation $\rR \syncp \rR' \subseteq (\rG_s \times \rG'_s) \times (\rH'_t \times \rH_t)$ as above. Similarly, the associated component morphisms are:
\[
\begin{tabular}{llll}
$(\rR \syncp \rR')_-$ 
& $:= \rR_- \syncp \rR'_-$
& $\subseteq (\rG_s \times \rG'_s) \times (\rH_s \times \rH'_s)$
\\[1ex]
$(\rR \syncp \rR')_+$
& $:= \rR_+ \syncp \rR'_+$
& $\subseteq (\rH_t \times \rH'_t) \times (\rG_t \times \rG'_t)$
\end{tabular}
\]
\endbox
\end{definition}

\begin{note}[Kronecker product of binary matrices]
\item
Given an $m \times n$ binary matrix $M$, and also an $m' \times n'$ binary matrix $N$, then their Kronecker product (over the boolean semiring) is obtained by replacing each $1$ in $M$ by a copy of $N$, and each $0$ in $M$ by the $m' \times n'$ zero-matrix. More formally, it is the $(m \times m') \times (n \times n')$ binary matrix $M \syncp N$ where the indices are ordered lexicographically, and:
\[
(M \syncp N)_{(i,i'),(j,j')} := M_{i,j} \land_\two N_{i',j'}.
\]
Then the Kronecker product of binary matrices is the synchronous product of their corresponding indicator relations, endowed with the lexicographic ordering. \endbox
\end{note}

Before proving that this functor is well-defined we will prove a number of basic properties e.g.\ synchronous products preserve bicliques (Cartesian-products), and also $\rR \syncp \rS$ is reduced iff both $\rR$ and $\rS$ are reduced, as long as none of the domains/codomains of $\rR$ and $\rS$ are empty.

\begin{lemma}
\label{lem:sync_functor_basic}
Let $\rR \subseteq X \times Y$ and $\rR' \subseteq X' \times Y'$ be any relations between finite sets.
\begin{enumerate}
\item
Given any biclique $A \times A' \subseteq X \times X'$, then:
\[
(\rR \syncp \rR')^\up (A \times A') 
= \rR \syncp \rR' [ A \times A'] 
= \rR[A] \times \rR'[A']
\]
and we have the special case $\rR \syncp \rR' [(x,x')] = \rR[x] \times \rR'[x']$.

\item
Given any biclique $B \times B' \subseteq Y \times Y'$ then:
\[
(\rR \syncp \rR')^\down (B \times B') = \rR^\down(B) \times (\rR')^\down(B')
\]

\item
$(\rR \syncp \rR')\spbreve = \rR\spbreve \syncp (\rR')\spbreve$

\item
$\rR \syncp \rR'$ is strict iff both $\rR$ and $\rR'$ are strict.

\end{enumerate}
\end{lemma}

\begin{proof}
\item
\begin{enumerate}
\item
We calculate:
\[
\begin{tabular}{lll}
$\rR \syncp \rR'[A \times A']$
& $= \{ (y,y') \in Y \times Y' : \exists (x,x') \in A \times A'. \rR(x,y) \,\land\, \rR'(x',y') \}$
\\&
$= \{ (y,y') \in Y \times Y' : (\exists x \in A. \rR(x,y)) \land (\exists x' \in A'. \rR'(x',y')) \}$
\\&
$= \{ (y,y') \in Y \times Y' : y \in \rR[A] \,\land\, y \in \rR'[A'] \}$
\\&
$=  \rR[A] \times \rR'[A']$
\end{tabular}
\]

\item
We calculate:
\[
\begin{tabular}{lll}
$(\rR \syncp \rR')^\down(B \times B')$
&
$= \{ (x,x') \in X \times X' : \rR \syncp \rR' [(x,x')] \subseteq B \times B' \}$
\\&
$= \{ (x,x') \in X \times X' : \rR[x] \times \rR'[x] \subseteq B \times B' \}$
& by (1)
\\&
$= \{ (x,x') \in X \times X' : \rR[x] \subseteq B \,\land\, \rR'[x'] \subseteq B' \}$
\\&
$= \rR^\down(B) \times (\rR')^\down(B')$
\end{tabular}
\]

\item
Follows directly from the definitions.
\item
Follows by (1) and (3).

\end{enumerate}
\end{proof}

The following Lemma is rather basic, but we write it out in full.

\begin{lemma}
\label{lem:syncp_reduced_char}
Take any relations $\rR \subseteq X_1 \times Y_1$ and $\rS \subseteq X_2 \times Y_2$ such that each of the four sets $X_1,\,Y_1,\,X_2,\,Y_2$ defining the domains/codomains are non-empty. Then $\rR \syncp \rS$ is reduced iff both $\rR$ and $\rS$ are reduced.
\end{lemma}

\begin{proof}
In the first half of the proof we do not use the non-emptiness assumption. Suppose that $\rR \subseteq X_1 \times Y_1$ and $\rS \subseteq X_2 \times Y_2$ are reduced i.e.\ satisfy the two statements from Lemma \ref{lem:reduce_char}. If either of them is $\emptyset \subseteq \emptyset \times \emptyset$ then so is $\rR \syncp \rS$ and thus is reduced. Otherwise, given any $(x_1,x_2) \in X_1 \times X_2$ and any subset $A \subseteq X_1 \times X_2$ we must show that:
\[
\rR[x_1] \times \rS[x_2] = \rR \syncp \rS[A] = \bigcup_{(a_1,a_2) \in A} \rR[a_1] \times \rS[a_2]
\qquad
\text{implies that $(x_1,x_2) \in A$}
\]
Firstly, for each $y_2 \in \rS[x_2]$ let $C_{y_2} := \{ (a_1,a_2) \in A : y_2 \in \rS[a_2] \}$. The union of their respective bicliques must contain the subset $\rR[x_1] \times \{y_2\}$, so that $\rR[x_1] = \bigcup \{ \rR[a_1] : \exists a_2.(a_1,a_2) \in C_{y_2} \}$.  Then since $\rR$ is reduced there exists $(a_1,a_2)$ in $C_{y_2}$ such that $\rR[a_1] = \rR[x_1]$, so the induced biclique equals $\rR[x_1] \times Z_{y_2}$ for some $Z_{y_2} \subseteq \rS[x_2]$ containing $y_2$. Taking the union of these `horizontal strips' yields:
\[
\rR[x_1] \times \rS[x_2] 
= \bigcup_{y_2 \in \rS[x_2]} \rR[x_1] \times Z_{y_2}
\]
Then since $\rS$ is reduced there exists $y_2 \in \rS[x_2]$ such that $Z_{y_2} = \rS[x_2]$, so that $\rR[x_1] \times \rS[x_2] = \rR[a_1] \times \rS[a_2]$ for some $(a_1,a_2) \in C_{y_2} \subseteq A$. Finally since $\rR$ and $\rS$ are reduced we have $x_1 = a_1$ and $x_2 = a_2$. 

\smallskip
Regarding the converse, assuming that $\rR$ is not reduced we'll show that $\rR \syncp \rS$ is not reduced. First observe that if $X_1 = Y_2 = \emptyset$ and $X_2,\, Y_1 \neq \emptyset$ then both $\rR$ and $\rS$ fail to be strict (hence cannot be reduced by Lemma \ref{lem:reduce_char}), whereas $\rR \syncp \rS = \emptyset \subseteq \emptyset \times \emptyset$ is reduced. This explains our assumption that every set $X_1,\,Y_1,\,X_2,\,Y_2$ is non-empty. If $\rS = \emptyset \subseteq X_2 \times Y_2$ then $\rR \syncp \rS = \emptyset$ with non-empty domain/codomain, so it is not strict and thus is not reduced. Otherwise we fix some $(x_2,y_2) \in \rS$. Since $\rR$ is not reduced one of two statements in Lemma \ref{lem:reduce_char} fails.
\begin{enumerate}
\item
If the first statement fails we can find $x_1 \nin Z \subseteq X_1$ such that  $\rR[x_1] = \rR[Z]$, so that:
\[
\rR \syncp \rS[ Z \times \{x_2\}] 
= \rR[Z] \times \rS[x_2] 
= \rR[x_1] \times \rS[x_2] 
= \rR \syncp \rS[(x_1,x_2)]
\]
whereas $(x_1,x_2) \nin Z \times \{x_2\}$.
\item
If the second statement fails we can find $y_1 \nin Z \subseteq Y_1$ such that $\breve{\rR}[y_1] = \breve{\rR}[Z]$, so that:
\[
(\rR \syncp \rS)\spbreve[Z \times \{y_2\}]
= \breve{\rR} \syncp \breve{\rS}[Z \times \{y_2\}]
= \breve{\rR}[Z] \times \breve{\rS}[y_2]
= \breve{\rR}[y_1] \times \breve{\rS}[y_2]
= \breve{\rR} \syncp \breve{\rS}[(y_1,y_2)]
\]
whereas $(y_1,y_2) \nin Z \times \{y_2\}$.

\end{enumerate} 
Thus in either case we deduce that $\rR \syncp \rS$ is not reduced.
\end{proof}

\smallskip

\begin{lemma}
$\syncp : \BiCliq \times \BiCliq \to \BiCliq$ is a well-defined functor.
\end{lemma}

\begin{proof}
Certainly its action on objects is well-defined. Given $\BiCliq$-morphisms $\rR : \rG \to \rH$ and $\rR' : \rG' \to \rH'$ then $\rR \syncp \rR'$ is a well-defined $\BiCliq$-morphism of type $\rG \syncp \rG' \to \rH \syncp \rH'$ via the witnesses:
\[
\xymatrix@=20pt{
\rG_t \times \rG'_t \ar[rr]^{(\rR_+ \syncp (\rR')_+)\spbreve} && \rH_t \times \rH'_t
\\
\rG_s \times \rG'_s \ar[urr]^{\rR \syncp \rR'} \ar[u]^{\rG \syncp \rG'} \ar[rr]_{\rR_- \syncp \rR'_-} && \rH_s \ar[u]_{\rH \syncp \rH'} \times \rH'_s
}
\]
That is, consider the following basic calculations:
\[
\begin{tabular}{lll}
$\rH \syncp \rH'[ \rR_- \syncp \rR'_- [(g_s,g'_s)]]$
&
$= \rH \syncp \rH'[ \rR_-[g_s] \times \rR'_-[g'_s] ]$
& preserves biclique
\\&
$= \rH[\rR_-[g_s]] \times \rH'[\rR'_-[g'_s]]$
& preserves biclique
\\&
$= \rR[g_s] \times \rR'[g'_s]$
& components are witnesses
\\&
$= (\rR \syncp \rR')[(g_s,g'_s)]$
& preserves biclique
\\
\\
$(\rR_+ \syncp \rR'_+)\spbreve[ \rG \syncp \rG'[(g_s,g'_s)]]$
&
$= (\rR_+ \syncp \rR'_+)\spbreve[ \rG[g_s] \times \rG'[g'_s] ]$
& preserves biclique
\\&
$= \rR_+\spbreve \syncp (\rR'_+)\spbreve [ \rG[g_s] \times \rG'[g'_s] ]$
& by Lemma \ref{lem:sync_functor_basic}.3
\\&
$= \rR_+\spbreve[\rG[g_s]] \times (\rR'_+)\spbreve[\rG'[g'_s]]$
& preserves biclique
\\&
$= \rR[g_s] \times \rR'[g'_s]$
& components are witnesses
\\&
$= (\rR \syncp \rR')[g'_s]$
& preserves biclique
\end{tabular}
\]
Next we establish that these witnesses are the associated components.
\[
\begin{tabular}{lll}
$(\rR \syncp \rR')_-[(g_s,g'_s)]$
&
$=  \cl_{\rH \syncp \rH'}(\rR_- \syncp \rR'_-[(g_s,g'_s)])$
& close witness
\\&
$= (\rH \syncp \rH')^\down \circ (\rH \syncp \rH')^\up(\rR_-[g_s] \times \rR'_-[g'_s])$
& preserves biclique
\\&
$= (\rH \syncp \rH')^\down (\rH[\rR_-[g_s]] \times \rH'[\rR'_-[g'_s]])$
& preserves biclique
\\&
$= (\rH \syncp \rH')^\down (\rR[g_s] \times \rR'[g'_s])$
& components are witnesses
\\&
$= \rH^\down(\rR[g_s]) \times (\rH')^\down(\rR'[g'_s])$
& preserves biclique
\\&
$= \rR_-[g_s] \times \rR'_-[g'_s]$
& by definition
\\&
$= (\rR_- \syncp \rR'_-)[(g_s,g'_s)]$
& preserves biclique
\end{tabular}
\]
and the proof that $(\rR \syncp \rR')_+ = \rR_+ \syncp \rR'_+$ is similar. Regarding preservation of identity morphisms:
\[
id_\rG \syncp id_\rH
= (\rG : \rG \to \rG) \syncp (\rH : \rH \to \rH)
= \rG \syncp \rH : \rG \syncp \rH \to \rG \syncp \rH
= id_{\rG \syncp \rH}
\]
To prove preservation of $\BiCliq$-composition, we first establish that:
\[
\begin{tabular}{lll}
$\vcenter{\vbox{\xymatrix@=20pt{
\rG \syncp \rG' \ar[dr]_{\rR \syncp \rR'} \ar[rr]^{(\rR \fatsemi \rS) \syncp \rR'} && \rI \syncp \rH'
\\
& \rH \syncp \rH' \ar[ur]_{S \syncp id_{\rH'}} &
}}}$
&&
\begin{tabular}{l}
for all $\BiCliq$-morphisms
\\
$\rR : \rG \to \rH$, $\rS : \rH \to \rI$
\\
and $\rR' : \rG' \to \rH'$
\end{tabular}
\end{tabular}
\]
We prove this using the characterisation of $\BiCliq$-morphisms from Lemma \ref{lem:bicliq_mor_char_max_witness}, and also the functional description of $\BiCliq$-composition from Corollary \ref{cor:bicliq_func_comp}.
\[
\begin{tabular}{lll}
$(\rR \syncp \rR') \fatsemi (\rS \syncp id_{\rH'})[(g_s,g'_s)]$
&
$= (\rS \syncp id_{\rH'})^\up \circ (\rH \syncp \rH')^\down \circ (\rR \syncp \rR')^\up(\{(g_s,g'_s)\})$
& 
\\&
$= (\rS \syncp id_{\rH'})^\up \circ (\rH \syncp \rH')^\down (\rR[g_s] \times \rR'[g'_s])$
& preserves biclique
\\&
$= (\rS \syncp \rH')^\up ( \rH^\down(\rR[g_s]) \times (\rH')^\down(\rR'[g'_s]))$
& preserves biclique
\\&
$= \rS^\up \circ \rH^\down \circ \rR^\up(\{g_s\}) \times (\rH')^\up \circ (\rH')^\down(\rR'[g'_s]))$
& preserves biclique
\\&
$= \rS^\up \circ \rH^\down \circ \rR^\up(\{g_s\}) \times \rR'[g'_s]$
& $\inte_{\rH'} \circ \rR' = \rR'$
\\&
$= (\rR \fatsemi \rS)[g_s] \times \rR'[g'_s]$
\\&
$= ((\rR \fatsemi \rS) \syncp \rR')[(g_s,g'_s)]$
\end{tabular}
\]
Thus we also have the symmetric statement $\rR \syncp (\rR' \fatsemi \rS') = (\rR \syncp \rR') \fatsemi (id_{\rH} \syncp \rS')$. Then we calculate:
\[
\begin{tabular}{lll}
$(\rR \fatsemi \rS) \syncp (\rR' \fatsemi \rS')$
&
$= ((\rR \fatsemi \rS) \syncp \rR') \fatsemi (id_{\rH} \syncp \rS')$
& right preservation
\\&
$= ((\rR \syncp \rR') \fatsemi (\rS \syncp id_{\rH'})) \fatsemi (id_\rI \syncp \rS')$
& left preservation
\\&
$= (\rR \syncp \rR') \fatsemi ((\rS \syncp id_{\rH'}) \fatsemi (id_\rI \syncp \rS'))$
& associativity
\\&
$= (\rR \syncp \rR') \fatsemi (\rS \syncp \rS')$
& see below
\end{tabular}
\]
Regarding the final statement, we have:
\[
\begin{tabular}{lll}
$(\rS \syncp id_{\rH'}) \fatsemi (id_\rI \syncp \rS' )[(h_s,h'_s)]$
&
$= (id_\rI \syncp \rS')^\up \circ (\rI \syncp \rH')^\down \circ (\rS \syncp id_{\rH'})^\up)(\{(h_s,h'_s)\})$
\\&
$= (id_\rI \syncp \rS')^\up \circ (\rI \syncp \rH')^\down(\rS[h_s] \times \rH'[h'_s])$
\\&
$= (id_\rI \syncp \rS')^\up (\rI^\down[\rS[h_s]] \times (\rH')^\down(\rH'[h'_s]))$
\\&
$= (\inte_\rI \circ \rS^\up(\{h_s\})) \times ((\rS')^\up \circ \cl_{\rH'}(\{h'_s\}))$
\\&
$= \rS[h_s] \times \rS'[h'_s]$
\\&
$= (\rS \syncp \rS')[(h_s,h'_s)]$
\end{tabular}
\]
\end{proof}

We now prove the main result of this subsection. It is further clarified via its corollaries further below.

\begin{theorem}[The synchronous product is essentially the tight tensor product]
\label{thm:sync_is_tight_tensor}
\item
We have the natural isomorphism:
\[
\begin{tabular}{c}
$\rTS : \Pirr(- \ttenp -) \To (\Pirr - )\syncp (\Pirr -)$
\qquad
$\rTS_{\aQ,\aR} : \Pirr(\aQ \ttenp \aR) \to \Pirr \aQ \syncp \Pirr \aR$
\\[0.5ex]
\begin{tabular}{ll}
$\rTS_{\aQ,\aR}(\up_{\aQ^{\pOp},\aR}^{j_1,j_2},\,(m_1,m_2))$ &
$\iff (\Pirr \aQ) \syncp (\Pirr \aR)((j_1,j_2),(m_1,m_2))$
\\[0.5ex] &
$\iff j_1 \nleq_\aQ m_1 \text{ \emph{and} }  j_2 \nleq_\aR m_2$
\end{tabular}
\end{tabular}
\]
with associated components:
\[
\begin{tabular}{lll}
$(\rTS_{\aQ,\aR})_- (\up_{\aQ^{\pOp},\aR}^{j_1,j_2},(j_3,j_4))$ 
& $\iff  (\Pirr\aQ)_- \syncp (\Pirr\aR)_- ((j_1,j_2),(j_3,j_4))$
\\[0.5ex]
& $\iff  j_3 \leq_\aQ j_1 \text{ \emph{and} } j_4 \leq_\aR j_2$
\\[1ex]
$(\rTS_{\aQ,\aR})_+((m_1,m_2),\down_{\aQ^{\pOp},\aR}^{m_3,m_4}) $ 
& $\iff (\Pirr\aQ)_+ \syncp (\Pirr\aR)_+ ((m_1,m_2),(m_3,m_4))$
\\[0.5ex]&
$\iff m_1 \leq_\aQ m_3 \text{ \emph{and} } m_2 \leq_\aR m_4$
\end{tabular}
\]
Its inverse is described in Note \ref{note:ts_nat_inverse} directly below.
\end{theorem}

\begin{proof}
Although the notation is somewhat cumbersome, the proof that each $\rTS_{\aQ,\aR}$ is a $\BiCliq$-isomorphism is relatively simple. Importantly, we shall show that $\Pirr(\aQ \ttenp \aR)$ is bipartite graph isomorphic to $\Pirr \aQ \syncp \Pirr \aR$. Then one can see that $\rTS_{\aQ,\aR}$ and its components are really just the $\BiCliq$-identity-morphism $id_{\Pirr \aQ \syncp \Pirr \aR} = \Pirr \aQ \syncp \Pirr \aR$ modulo relabelling, recalling that $(\Pirr \aQ \syncp \Pirr \aR)_- = (\Pirr \aQ)_- \syncp (\Pirr \aR)_-$ and similar for the positive component (see Definition \ref{def:sync_product}). First let:
\[
\rG := \Pirr(\aQ \ttenp \aR) = \Pirr\jslTight{\aQ^{\pOp},\aR}
\qquad\text{and}\qquad
\rH := \Pirr\aQ \syncp \Pirr\aR
\]
and recall that by Lemma \ref{lem:tight_tensor_irr}:
\[
\begin{tabular}{lll}
$\rG_s$ & $= J(\jslTight{\aQ^{\pOp},\aR})$ & $= \{ \up_{\aQ^{\pOp},\aR}^{j_1,j_2} \; : (j_1,j_2) \in J(\aQ) \times J(\aR) \}$
\\[0.5ex]
$\rG_t$ & $= M(\jslTight{\aQ^{\pOp},\aR})$ & $= \{ \down_{\aQ^{\pOp},\aR}^{m_1,m_2} \; : (m_1,m_2) \in M(\aQ) \times M(\aR) \}$
\end{tabular}
\]
\[
\rH_s = J(\aQ) \times J(\aR)
\qquad
\rH_t = M(\aQ) \times M(\aR)
\]
Clearly $|\rG_s| = |\rH_s|$ and $|\rG_t| = |\rH_t|$, and moreover:
\[
\begin{tabular}{lll}
$\rG(\up_{\aQ^{\pOp},\aR}^{j_1,j_2},\down_{\aQ^{\pOp},\aR}^{m_1,m_2})$
& $\iff$ &
$\up_{\aQ^{\pOp},\aR}^{j_1,j_2} \; \nleq_{\JSL_f[\aQ^{\pOp},\aR]} \; \down_{\aQ^{\pOp},\aR}^{m_1,m_2}$
\\[1ex]
$\rH((j_1,j_2),(m_1,m_2))$ & $\iff$ & ($j_1 \nleq_\aQ m_1$ and $j_2 \nleq_\aR m_2$)
\end{tabular}
\]
recalling that $\jslTight{\aQ^{\pOp},\aR}$ is a sub join-semilattice of $\JSL_f[\aQ^{\pOp},\aR]$ and hence inherits the pointwise ordering. There is an obvious candidate for a bipartite graph isomorphism i.e.\ send $\up_{\aQ^{\pOp},\aR}^{j_1,j_2}$ to $(j_1,j_2)$, and send $\down_{\aQ^{\pOp},\aR}^{m_1,m_2}$ to $(m_1,m_2)$. To verify its correctness, we need to show that for any fixed $(j_1,j_2) \in J(\aQ) \times J(\aR)$ and $(m_1,m_2) \in M(\aQ) \times M(\aR)$:
\[
\quad \up_{\aQ^{\pOp},\aR}^{j_1,j_2} \; \leq_{\JSL_f[\aQ^{\pOp},\aR]} \; \down_{\aQ^{\pOp},\aR}^{m_1,m_2}
\qquad\iff\qquad
(j_1 \leq_\aQ m_1 \text{ or } j_2 \leq_\aR m_2)
\]
which follows immediately by Lemma \ref{lem:special_jsl_morphisms}.6.

Having established this bipartite graph isomorphism between $\Pirr(\aQ \ttenp \aR)$ and $\Pirr\aQ \syncp \Pirr\aR$, it follows that each $\rTS_{\aQ,\aR}$ is a well-defined $\BiCliq$-isomorphism. That is, $\rTS_{\aQ,\aR}$ is constructed by starting with the $\BiCliq$-isomorphism $id_{\Pirr\aQ \syncp \Pirr\aR} = \Pirr\aQ \syncp \Pirr\aR$ and applying a bipartite graph isomorphism to its domain. The description of the associated components also follow from this.

\smallskip
It remains to show naturality i.e.\ given morphisms $(f_i : \aQ_i \to \aR_i)_{i = 1,2}$ we must show that:
\[
\xymatrix@=15pt{
\Pirr(\aQ_1 \ttenp \aQ_2) \ar[d]_-{\Pirr(f_1 \ttenp f_2)} \ar[rr]^-{\rTS_{\aQ_1,\aQ_2}} && (\Pirr\aQ_1) \syncp (\Pirr\aQ_2) \ar[d]^-{\Pirr f_1 \syncp \Pirr f_2}
\\
\Pirr(\aR_1 \ttenp \aR_2) \ar[rr]_-{\rTS_{\aR_1,\aR_2}} && (\Pirr\aR_1) \syncp (\Pirr\aR_2)
}
\]
Let us first calculate:
\[
\begin{tabular}{lll}
$\rTS_{\aQ_1,\aQ_2} \fatsemi \Pirr f_1 \syncp \Pirr f_2 $
&
$= \rTS_{\aQ_1,\aQ_2} ; (\Pirr f_1 \syncp \Pirr f_2)_+\spbreve$
\\&
$= \rTS_{\aQ_1,\aQ_2} ; ((\Pirr f_1)_+ \syncp (\Pirr f_2)_+)\spbreve$
\\&
$= \rTS_{\aQ_1,\aQ_2} ; (\Pirr f_1)_+\spbreve \syncp (\Pirr f_2)_+\spbreve$
\end{tabular}
\]
so that:
\[
\begin{tabular}{lll}
& $\rTS_{\aQ_1,\aQ_2} \fatsemi \Pirr f_1 \syncp \Pirr f_2(\up_{\aQ_1^{\pOp},\aQ_2}^{j_1,j_2},\,(m_1,m_2))$
\\ $\iff$ &
$\forall i = 1,2. \exists m_q^i \in M(\aQ_i) . (j_i \nleq_{\aQ_i} m_q^i \text{ and } (f_i)_*(m_i) \leq_{\aQ_i} m_q^i) $
\\ $\iff$ &
$\forall i = 1,2. \neg\forall m_q^i \in M(\aQ_i). ( (f_i)_*(m_i) \leq_{\aQ_i} m_q^i \To j_i \leq_{\aQ_i} m_q^i)$
\\ $\iff$ &
$\forall i =1,2.\neg(j_i \leq_{\aQ_i} (f_i)_*(m_i))$
\\ $\iff$ &
$f_1(j_1) \nleq_{\aR_1} m_1 \text{ and } f_2(j_2) \nleq_{\aR_2} m_2$
\end{tabular}
\]
We now compute the other composite $\BiCliq$-morphism:
\[
\Pirr(f_1 \ttenp f_2) \fatsemi \rTS_{\aR_1,\aR_2}
= \Pirr(f_1 \ttenp f_2) ; (\rTS_{\aR_1,\aR_2})_+\spbreve
\]
in three steps.
\begin{enumerate}
\item
The first relation $\Pirr(f_1 \ttenp f_2) \subseteq J(\aQ_1 \ttenp \aQ_2) \times M(\aR_1 \ttenp \aR_2)$ has definition:
\[
\begin{tabular}{lll}
$\Pirr(f_1 \ttenp f_2) (\up_{\aQ^{\pOp},\aQ_2}^{j_1,j_2},\,\down_{\aR_1^{\pOp},\aR_2}^{m'_1,m'_2})$
&
$\iff f_1 \ttenp f_2(\up_{\aQ_1^{\pOp},\aQ_2}^{j_1,j_2}) \nleq_{\aR_1 \ttenp \aR_2} \; \down_{\aR_1^{\pOp},\aR_2}^{m'_1,m'_2}$
\\&
$\iff f_2 \circ \up_{\aQ_1^{\pOp},\aQ_2}^{j_1,j_2} \circ (f_1)_*  \nleq_{\JSL_f[\aR_1^{\pOp},\aR_2]} \; \down_{\aR_1^{\pOp},\aR_2}^{m'_1,m'_2}$
\end{tabular}
\]
\item
The second relation $(\rTS_{\aR_1,\aR_2})_+\spbreve \subseteq M(\aR_1 \ttenp \aR_2) \times (M(\aR_1) \times M(\aR_2))$ has definition:
\[
(\rTS_{\aR_1,\aR_2})_+\spbreve \, (\down_{\aR_1^{\pOp},\aR_2}^{m'_1,m'_2},\, (m_1,m_2))
\iff m_1 \leq_{\aR_1} m'_1 \text{ and } m_2 \leq_{\aR_2} m'_2
\]
as per the statement of this theorem.

\item
Composing yields all pairs $(\up_{\aQ_1^{\pOp},\aQ_2}^{j_1,j_2},\,(m_1,m_2))$ s.t.\ $\exists (m'_1,m'_2) \in M(\aR_1) \times M(\aR_2)$ satisfying:
\begin{enumerate}
\item
$f_1 \circ \up_{\aQ_1^{\pOp},\aQ_2}^{j_1,j_2} \circ (f_1)_* \nleq_{\JSL_f[\aR_1^{\pOp},\aR_2]} \; \down_{\aR_1^{\pOp},\aR_2}^{m'_1,m'_2}$

\item
$m_1 \leq_{\aR_1} m'_1 \text{ and } m_2 \leq_{\aR_2} m'_2$
\end{enumerate}

By Lemma \ref{lem:special_jsl_morphisms}.4 the latter condition is equivalent to
\[
\down_{\aR_1^{\pOp},\aR_2}^{m_1,m_2} \; \leq_{\JSL_f[\aR_1^{\pOp},\aR_2]} \; \down_{\aR_1^{\pOp},\aR_2}^{m'_1,m'_2}
\]
where it is important that $\aR_1^{\pOp}$ reverses the ordering. Consequently:
\[
\begin{tabular}{lll}
& $\Pirr(f_1 \ttenp f_2) ; (\rTS_{\aR_1,\aR_2})_+\spbreve(\up_{\aQ^{\pOp},\aQ_2}^{j_1,j_2},(m_1,m_2))$
\\[1ex]
$\iff$ &
$\exists (m'_1,m'_2) \in M(\aR_1) \times M(\aR_2).(
\down_{\aR_1^{\pOp},\aR_2}^{m_1,m_2} \; \leq \; \down_{\aR_1^{\pOp},\aR_2}^{m'_1,m'_2}
\text{ and }
f_2 \circ \up_{\aQ_1^{\pOp},\aQ_2}^{j_1,j_2} \circ (f_1)_*  \nleq \; \down_{\aR_1^{\pOp},\aR_2}^{m'_1,m'_2}
)$
\\[1ex]
$\iff$ &
$\neg\forall m'_1 \in M(\aR_1), m'_2 \in M(\aR_2).(
\down_{\aR_1^{\pOp},\aR_2}^{m_1,m_2} \; \leq \; \down_{\aR_1^{\pOp},\aR_2}^{m'_1,m'_2}
\; \implies 
f_2 \circ \up_{\aQ_1^{\pOp},\aQ_2}^{j_1,j_2} \circ (f_1)_*  \leq \; \down_{\aR_1^{\pOp},\aR_2}^{m'_1,m'_2}
)$
\\[1ex]
$\iff$ &
$f_2 \circ \up_{\aQ_1^{\pOp},\aQ_2}^{j_1,j_2} \circ (f_1)_* \nleq_{\JSL_f[\aR_1^{\pOp},\aR_2]} \; \down_{\aR_1^{\pOp},\aR_2}^{m_1,m_2} $
\end{tabular}
\]
where the final step uses the fact that $M(\aR_1 \ttenp \aR_2)$ consists of the morphisms $\down_{\aR_1^{\pOp},\aR_2}^{m'_1,m'_2}$.
\end{enumerate}

Having described the two sides of the naturality square, we need to prove their equality. By the above descriptions it suffices to prove that:
\[
f_2 \circ \up_{\aQ_1^{\pOp},\aQ_2}^{j_1,j_2} \circ (f_1)_* \leq \; \down_{\aR_1^{\pOp},\aR_2}^{m_1,m_2}
\iff
(f_1(j_1) \leq_{\aR_1} m_1 \text{ or } f_2(j_2) \leq_{\aR_2} m_2)
\]
for any $(j_1,j_2) \in J(\aQ_1) \times J(\aQ_2)$ and $(m_1,m_2) \in M(\aR_1) \times M(\aR_2)$. By Lemma \ref{lem:compose_spec_gen_mor}.1 this amounts to:
\[
\up_{\aR_1^{\pOp},\aR_2}^{f_1(j_1),f_2(j_2)} \; \leq \; \down_{\aR_1^{\pOp},\aR_2}^{m_1,m_2}
\iff
(f_1(j_1) \leq_{\aR_1} m_1 \text{ or } f_2(j_2) \leq_{\aR_2} m_2)
\]
which follows immediately by Lemma \ref{lem:special_jsl_morphisms}.6.
\end{proof}

\begin{note}
\label{note:ts_nat_inverse}
The natural inverse $\rTS^{\bf-1} : (\Pirr -) \syncp (\Pirr -) \To \Pirr(- \ttenp -)$ and its associated components are defined:
\[
\begin{tabular}{lll}
$\rTS_{\aQ,\aR}^{\bf-1}((j_1,j_2),\down_{\aQ^{\pOp},\aR}^{m_1,m_2})$
& $\iff$ &
$j_1 \nleq_\aQ m_1$ and $j_2 \nleq_\aR m_2$
\\[1ex]
$(\rTS_{\aQ,\aR}^{\bf-1})_-((j_1,j_2),\up_{\aQ^{\pOp},\aR}^{j_3,j_4})$
& $\iff$ &
$j_3 \leq_\aQ j_1$ and $j_4 \leq_\aR j_2$
\\[1ex]
$(\rTS_{\aQ,\aR}^{\bf-1})_+(\down_{\aQ^{\pOp},\aR}^{m_1,m_2},(m_3,m_4))$
& $\iff$ &
$m_1 \leq_\aQ m_3$ and $m_2 \leq_\aR m_4$
\end{tabular}
\]
This follows from the proof above i.e.\ apply the bipartite graph isomorphism to the \emph{codomain} of $id_{\Pirr\aQ \syncp \Pirr\aR}$. \endbox
\end{note}

That the synchronous product and the tight tensor product are `equivalent concepts' is now further clarified i.e.\ we describe certain composite natural isomorphisms. Recall that by definition $\aQ \ttenp \aR = \jslTight{\aQ^{\pOp},\aR}$.

\begin{corollary}
\label{cor:tight_tensor_syncp}
\item
We have the composite natural isomorphisms:
\begin{enumerate}
\item
$\aQ \ttenp \aR \xto{\rep_{\aQ \ttenp \aR}} \Open\Pirr(\aQ \ttenp \aR) \xto{\Open \rTS_{\aQ,\aR}} \Open(\Pirr\aQ \syncp \Pirr\aR)$
\[
\begin{tabular}{rcll}
with action
& $f$
& $\mapsto$ & $\{ (m_q,m_r) \in M(\aQ) \times M(\aR) : f(m_q) \nleq_\aR m_r \}$
\\[0.5ex]
and inverse action
& $Y$ & $\mapsto$ & $\lambda q \in Q. \Lor_\aR \{ \Land_\aR \{ m_r \in M(\aR) :  (m_q,m_r) \nin Y \} : q \leq_\aQ m_q \in M(\aQ) \}$.
\end{tabular}
\]
Moreover the action on join/meet-irreducibles is as follows:
\[
\begin{tabular}{lll}
$\up_{\aQ^{\pOp},\aR}^{j_q,j_r}$
& $\mapsto$
& $\Pirr\aQ \syncp \Pirr\aR[(j_q,j_r)] = \Pirr\aQ[j_q] \times \Pirr\aR[j_r]$
\\[1ex]
$\down_{\aQ^{\pOp},\aR}^{m_q,m_r}$
& $\mapsto$
& $\inte_{\Pirr\aQ \syncp \Pirr\aR}(\overline{(m_q,m_r)})$
\end{tabular}
\]

\item
$\rG \syncp \rH \xto{\red_\rG \syncp \red_\rH} (\Pirr\Open\rG) \syncp (\Pirr\Open\rH) \xto{\rTS_{\Open\rG,\Open\rH}^{\bf-1}} \Pirr(\Open\rG \ttenp \Open\rH)$
\[
\begin{tabular}{c}
where we relate $(g_s,h_s)$ to $\down_{(\Open\rG)^{\pOp},\Open\rH}^{\inte_\rG(\overline{g_t}) , \inte_\rH(\overline{h_t})}$ iff we have $\rG(g_s,g_t)$ and $\rH(h_s,h_t)$.
\end{tabular}
\]
Regarding the relation directly above, recall that every meet-irreducible in $\Open\rG$ arises as $\inte_\rG(\overline{g_t})$ for some $g_t \in \rG_t$. However not all $g_t \in \rG_t$ need yield a meet-irreducible in this way, unless $\rG$ is reduced.
\end{enumerate}
\end{corollary}

\begin{proof}
\item
\begin{enumerate}
\item
We begin by showing that:
\[
rep_{\aQ \ttenp \aR}(f : \aQ^{\pOp} \to \aR)
= \{ \down_{\aQ^{\pOp},\aR}^{m_1,m_2} \; : (m_1,m_2) \in M(\aQ) \times M(\aR), \, f(m_1) \nleq_\aR m_2 \}
\]
Recall by definition that $rep_{\aQ \ttenp \aR}(f)$ contains all those meet-irreducibles $m \in M(\aQ \ttenp \aR)$ such that $f \nleq_{\aQ \ttenp \aR} m$. Now, by Lemma \ref{lem:irr_tight_morphisms} we know these meet-irreducibles are precisely $(\down_{\aQ^{\pOp},\aR}^{m_1,m_2})_{(m_1,m_2) \in M(
\aQ) \times M(\aR)}$. Then it only remains to show that $f \nleq_{\aQ \ttenp \aR} \down_{\aQ^{\pOp},\aR}^{m_1,m_2}$ iff $f(m_1) \nleq_\aR m_2$. We calculate:
\[
\begin{tabular}{lll}
$f \leq_{\aQ \ttenp \aR} \; \down_{\aQ^{\pOp},\aR}^{m_1,m_2}$
&
$\iff \forall q \in Q.( q \leq_{\aQ^{\pOp}} m_1 \To f(q) \leq_\aR m_2)$
& since $q \leq_\aR \top_\aR$
\\&
$\iff  \forall q \in Q.( m_1 \leq_\aQ q \To f(q) \leq_\aR m_2)$
\\&
$\iff  f(m_1) \leq_\aR m_2$
& since $f : \aQ^{\pOp} \to \aR$ monotonic
\end{tabular}
\] 
In order to apply $\Open\rTS_{\aQ,\aR}$, first recall that $(\rTS_{\aQ,\aR})_+\spbreve \subseteq M(\aQ \ttenp \aR) \times (M(\aQ) \times M(\aR))$ has definition:
\[
(\rTS_{\aQ,\aR})_+\spbreve(\down_{\aQ^{\pOp},\aR}^{m_1,m_2},\,(m_q,m_r))
\iff
m_q \leq_\aQ m_1 \text{ and } m_r \leq_\aR m_2
\]
Then to finally understand why:
\[
\Open\rTS_{\aQ,\aR}[rep_{\aQ \ttenp \aR}(f)]
= \{ (m_q,m_r) \in M(\aQ) \times M(\aR) : f(m_q) \nleq_\aR m_r \}
\]
observe that $f(m_1) \nleq_\aR m_2$ and $m_q \leq_\aQ m_1$ and $m_r \leq_\aR m_2$ imply that $f(m_q) \nleq_\aR m_r$, by making use of the `order-reversing' monotonicity of $f : \aQ^{\pOp} \to \aR$.

\bigskip
The inverse action follows because every $q \in Q$ arises as a meet of those meet-irreducibles above it and $f : \aQ^{\pOp} \to \aR$ sends $\aQ$-meets to $\aR$-joins. The description of the action on join/meet-irreducibles is `the natural one' in the sense that (i) we know the join/meet-irreducibles of $\aQ \ttenp \aR = \jslTight{\aQ^{\pOp},\aR}$ via Lemma \ref{lem:tight_tensor_irr}, and (ii) we know the join/meet-irreducibles of $\Open(\Pirr\aQ \syncp \Pirr\aR)$ via Lemma \ref{lem:lat_op_cl}.3 and the fact that $\Pirr\aQ \syncp \Pirr\aR$ is reduced via Lemma \ref{lem:syncp_reduced_char}. Nevertheless, let us directly verify these claims.
\begin{enumerate}
\item
First the action on join-irreducibles:
\[
\begin{tabular}{lll}
$\alpha_{\aQ,\aR}(\up_{\aQ^{\pOp},\aR}^{j_q,j_r})$
&
$= \{ (m_q,m_r) \in M(\aQ) \times M(\aR) : \; \up_{\aQ^{\pOp},\aR}^{j_q,j_r}(m_q) \nleq_\aR m_r \}$
& by definition
\\&
$= \{ (m_q,m_r) \in M(\aQ) \times M(\aR) : m_q \nleq_{\aQ^{\pOp}} j_q , \; j_r \nleq_\aR m_r \}$
\\&
$= \{ (m_q,m_r) \in M(\aQ) \times M(\aR) : j_q \nleq_\aQ m_q , \; j_r \nleq_\aR m_r \}$
\\&
$= \Pirr\aQ[j_q] \times \Pirr\aR[j_r]$
\\&
$= (\Pirr\aQ \syncp \Pirr\aR)[(j_q,j_r)]$
& see Lemma \ref{lem:sync_functor_basic}.1
\end{tabular}
\]
\item
Finally we verify the action on meet-irreducibles in terms of the inverse action. For brevity let $\rG_\aQ = \Pirr\aQ$ and $\rG_\aR = \Pirr\aR$.
\[
\begin{tabular}{lll}
$\alpha_{\aQ,\aR}^{\bf-1}(\inte_{\rG_\aQ \syncp \rG_\aR}(\overline{(m_q,m_r)}))$
\\
$= \lambda q \in Q. \Lor_\aR \{ \Land_\aR \{ m'_r \in M(\aR) : (m'_q,m'_r) \nin \inte_{\rG_\aQ \syncp \rG_\aR}(\overline{(m_q,m_r)}) \} : q \leq_\aQ m'_q \in M(\aQ) \}$
\\
$= \lambda q \in Q. \Lor_\aR \{ \Land_\aR \{ m'_r \in M(\aR) : (m'_q,m'_r) \in \cl_{\breve{\rG}_\aQ \syncp \breve{\rG}_\aR}(\{(m_q,m_r)\}) \} : q \leq_\aQ m'_q \in M(\aQ) \}$
\\
$= \lambda q \in Q. \Lor_\aR \{ \Land_\aR \{ m'_r \in M(\aR) : (m'_q,m'_r) \in \cl_{\breve{\rG}_\aQ}(\{m_q\}) \times \cl_{\breve{\rG}_\aR}(\{m_r\}) \} : q \leq_\aQ m'_q \in M(\aQ) \}$
\\
$= \lambda q \in Q. \Lor_\aR \{ \Land_\aR \{ m'_r \in M(\aR) : m_q \leq_\aQ m'_q, \; m_r \leq_\aR m'_r \} : q \leq_\aQ m'_q \in M(\aQ) \}$
\\
$= \lambda q \in Q.
\begin{cases}
\bot_\aR & \text{if $q = \top_\aQ$}
\\
\top_\aR & \text{if $m_q \nleq_\aQ q$}
\\
m_r & \text{otherwise}
\end{cases}$
\\
$= \; \down_{\aQ^{\pOp},\aR}^{m_q,m_r}$
\end{tabular}
\]
Here we have used De Morgan duality applied to interior/closure operators, the fact that synchronous products preserve bicliques (see Lemma \ref{lem:sync_functor_basic}), and also that e.g.\ $\breve{\rG}_\aQ = (\Pirr\aQ)\spbreve = \Pirr\aQ^{\pOp}$ so that:
\[
m'_q \in \cl_{\breve{\rG}_\aQ}(\{ m_q \})
\iff
m'_q \in \cl_{\Pirr\aQ^{\pOp}}(\{ m_q \})
\iff 
m'_q \leq_{\aQ^{\pOp}} m_q
\iff 
m_q \leq_\aQ m'_q
\]
using Lemma \ref{lem:cl_inte_of_pirr}.2 in the second equivalence.
\end{enumerate}

\item
First recall the canonical isomorphism:
\[
red_\rG : \rG \to \Pirr\Open\rG
\qquad
red_\rG \subseteq \rG_s \times M(\Open\rG)
\qquad
red_\rG(g_s,Y) :\iff \rG[g_s] \nsubseteq Y
\]
and consequently:
\[
\begin{tabular}{c}
$red_\rG \syncp \red_\rH \subseteq (\rG_s \times \rH_s) \times (M(\Open\rG) \times M(\Open\rH))$
\\[0.5ex]
$red_\rG \syncp \red_\rH((g_s,h_s),(Y_1,Y_2)) \iff \rG[g_s] \nsubseteq Y_1 \text{ and } \rH[h_s] \nsubseteq Y_2$
\end{tabular}
\]
We already described the positive component of $\rTS^{\bf-1}$ in Note \ref{note:ts_nat_inverse} above,
\[
\begin{tabular}{c}
$(\rTS_{\Open\rG,\Open\rH}^{\bf-1})_+\spbreve \subseteq (M(\Open\rG) \times M(\Open\rH)) \times M(\Open\rG \ttenp \Open\rH)$
\\[0.5ex]
$(\rTS_{\Open\rG,\Open\rH}^{\bf-1})_+\spbreve((Y_1,Y_2),(\down_{(\Open\rG)^{\pOp},\Open\rH}^{Y_\rG,Y_\rH}))
\iff Y_\rG \subseteq Y_1 \text{ and } Y_\rH \subseteq Y_2$
\end{tabular}
\]
also recalling that $\leq_{\Open\rG}$ is inclusion of sets. Then:
\[
\begin{tabular}{lll}
& $(red_\rG \syncp red_\rH) \fatsemi \rTS_{\Open\rG,\Open\rH}^{\bf-1} \, ((g_s,h_s),\down_{(\Open\rG)^{\pOp},\Open\rH}^{Y_\rG, Y_\rH})$
\\[0.5ex] $\iff$ &
$(red_\rG \syncp red_\rH) ; (\rTS_{\Open\rG,\Open\rH}^{\bf-1})_+\spbreve ((g_s,h_s),\down_{(\Open\rG)^{\pOp},\Open\rH}^{Y_\rG, Y_\rH})$
\\[0.5ex] $\iff$ &
$\exists (Y_1,Y_2) \in M(\Open\rG) \times M(\Open\rH).(\rG[g_s] \nsubseteq Y_1 \text{ and } \rH[h_s] \nsubseteq Y_2 \text{ and } Y_\rG \subseteq Y_1 \text{ and } Y_\rH \subseteq Y_2 )$
\\[0.5ex] $\iff$ &
$\rG[g_s] \nsubseteq Y_\rG \text{ and } \rH[h_s] \nsubseteq Y_\rH$
\end{tabular}
\]
where the final equivalence follows easily using basic properties of sets. Finally, recall by Lemma \ref{lem:lat_op_cl}.3 that every meet-irreducible $Y_\rG \in M(\Open\rG)$ equals $\inte_\rG(\overline{g_t})$ for some $g_t \in \rG_t$, and similarly $Y_\rH = \inte_\rH(\overline{h_t})$ for some $h_t \in \rH_t$. Then since:
\[
\begin{tabular}{lll}
$\rG[g_s] \nsubseteq \inte_\rG(\overline{g_t})$
&
$\iff g_s \nin \rG^\down \circ \rG^\up \circ \rG^\down(\overline{g_t})$
& by adjointness
\\&
$\iff g_s \nin \rG^\down(\overline{g_t})$
& by $(\down\up\down)$
\\&
$\iff \rG[g_s] \nsubseteq \overline{g_t}$
& by adjointness
\\&
$\iff \rG(g_s,g_t)$
\end{tabular}
\]
we are finished.

\end{enumerate}
\end{proof}

We also have the following basic result.

\begin{corollary}
\label{cor:sync_prod_pres_monos}
The synchronous product functor preserves monos.
\end{corollary}

\begin{proof}
By Corollary \ref{cor:tight_tensor_syncp}.2 we have the natural isomorphism $- \syncp - \cong \Pirr(\Open - \ttenp \Open -)$. Equivalence functors preserve monos, and by Lemma \ref{lem:ttenp_pres_monos} the tight tensor product preserves them too.
\end{proof}

The following Theorem describes the equality $(\breve{\rG} \syncp \breve{\rH})\spbreve = \rG \syncp \rH$ inside $\JSL_f$, where it becomes a non-trivial natural isomorphism. Importantly, it also provides a method to compute meets inside tight tensor products.

\begin{theorem}[Tight tensor products are isomorphic to their De Morgan dual]
\label{thm:tight_self_dual}
\item
We have the natural isomorphism:
\[
\begin{tabular}{lll}
$(\aQ^{\pOp} \ttenp \aR^{\pOp})^{\pOp}  \xto{\nu_{\aQ,\aR}} \aQ \ttenp \aR$
& $\nu_{\aQ,\aR}(f) := \lambda q \in Q. \Lor_\aR \{ j \in J(\aR) : f_*(j) \nleq_\aQ q \}$
\\[1ex]
& $\nu_{\aQ,\aR}^{\bf-1}(g) := \lambda q \in Q. \Land_\aR \{ m \in M(\aR) : q \nleq_\aQ g_*(m) \}$
\end{tabular}
\]
In particular, $\nu_{\aQ,\aR}^{\bf-1} = \nu_{\aQ^{\pOp},\aR^{\pOp}}^{\pOp}$ and furthermore:
\[
\begin{tabular}{c}
$\nu_{\aQ,\aR}(\Land_{\jslTight{\aQ,\aR^{\pOp}}} \{ \down_{\aQ,\aR^{\pOp}}^{q_i,r_i} : i \in I \})
=  \Lor_{\jslTight{\aQ^{\pOp},\aR}} \{ \up_{\aQ^{\pOp},\aR}^{q_i,r_i} : i \in I \}$
\\[2ex]
$\nu_{\aQ,\aR}(\Lor_{\jslTight{\aQ,\aR^{\pOp}}} \{ \up_{\aQ,\aR^{\pOp}}^{q_i,r_i} : i \in I \})
=  \Land_{\jslTight{\aQ^{\pOp},\aR}} \{ \down_{\aQ^{\pOp},\aR}^{q_i,r_i} : i \in I \}$
\end{tabular}
\]
for any $(q_i,r_i) \in Q \times R$ and index set $I$.
\end{theorem}

\begin{proof}
We define $\nu_{\aQ,\aR}$ as the following composite natural isomorphism:
\[
\xymatrix@=10pt{
(\aQ^{\pOp} \ttenp \aR^{\pOp})^{\pOp}
\ar[rrr]^-{\alpha_{\aQ^{\pOp},\aR^{\pOp}}^{\pOp}}_-{\cong} 
\ar[ddddrrrrr]_{\nu_{\aQ,\aR}}
&&& (\Open(\Pirr\aQ^{\pOp} \syncp \Pirr\aR^{\pOp}))^{\pOp} \ar@{=}[rr] 
&& (\Open(\Pirr\aQ \syncp \Pirr \aR)\spbreve)^{\pOp} \ar[dd]_{\cong}^{\partial_{(\Pirr\aQ \syncp \Pirr\aR)\spbreve}}
\\
\\ &&& && \Open(\Pirr\aQ \syncp \Pirr\aR) 
\ar[dd]_{\cong}^{\alpha_{\aQ,\aR}^{\bf-1}}
\\
\\ &&& && \aQ \ttenp \aR
}
\]
where:
\[
\begin{tabular}{rll}
$\alpha_{\aQ^{\pOp},\aR^{\pOp}}(h)$ 
& $:= \{ (j_q,j_r) \in J(\aQ) \times J(\aR) : j_r \nleq_\aR  h(j_q) \}$
& is from Corollary \ref{cor:tight_tensor_syncp}.1
\\[0.5ex]
$\alpha_{\aQ,\aR}^{\bf-1}(Y)$ 
& $:= \lambda q \in Q. \Lor_\aR \{ \Land_\aR \{ m_r \in M(\aR) : (m_q,m_r) \nin Y \} : q \leq_\aQ m_q \in M(\aQ)  \}$
& see above
\\[0.5ex]
$\partial_{(\Pirr\aQ \syncp \Pirr\aR)\spbreve}(X)$ 
& $:= \Pirr\aQ \syncp \Pirr\aR[\overline{X}]$
& is from Definition \ref{def:open_dual_iso}
\end{tabular}
\]
Then $\nu_{\aQ,\aR}$ is a well-defined natural isomorphism, and moreover:
\[
\begin{tabular}{lll}
$\nu_{\aQ^{\pOp},\aR^{\pOp}}^{\pOp}$
&
$= (\alpha_{\aQ^{\pOp},\aR^{\pOp}}^{\bf-1})^{\pOp}
\circ \partial_{(\Pirr\aQ^{\pOp} \syncp \Pirr\aR^{\pOp})\spbreve}^{\pOp} \circ \alpha_{\aQ,\aR}$
& apply $(-)^{\pOp}$
\\&
$= (\alpha_{\aQ^{\pOp},\aR^{\pOp}}^{\bf-1})^{\pOp}
\circ \partial_{\Pirr\aQ \syncp \Pirr\aR}^{\pOp} \circ \alpha_{\aQ,\aR}$
& 
\\&
$= (\alpha_{\aQ^{\pOp},\aR^{\pOp}}^{\bf-1})^{\pOp}
\circ \partial_{(\Pirr\aQ \syncp \Pirr\aR)\spbreve}^{\bf-1} \circ \alpha_{\aQ,\aR}$
& see below
\\&
$= (\alpha_{\aQ^{\pOp},\aR^{\pOp}}^{\pOp})^{\bf-1}
\circ \partial_{\Pirr\aQ^{\pOp} \syncp \Pirr\aR^{\pOp}}^{\bf-1} \circ \alpha_{\aQ,\aR}$
\\&
$= \nu_{\aQ,\aR}^{\bf-1}$
& by rule of composite inverses
\end{tabular}
\]
The marked equality follows because for any relation $\rG$ we have:
\[
\delta_\rG^{\pOp} = ((\delta_\rG)_*)^{\bf-1}
\quad\text{by Lemma \ref{lem:adj_obs}.2}
\qquad\text{and \qquad $(\delta_\rG)_* = \delta_{\breve{\rG}}$ by Lemma \ref{lem:open_dual_iso_adjoints}.1}.
\]
Using this fact, one can readily verify that the description of $\nu_{\aQ,\aR}^{\bf-1}$'s action follows from that of $\nu_{\aQ,\aR}$, so it remains to prove that the latter is correct. To this end, we first show that:
\[
\nu_{\aQ,\aR}(\down_{\aQ,\aR^{\pOp}}^{j_q,j_r}) 
\quad = \quad \up_{\aQ^{\pOp},\aR}^{j_q,j_r}
\qquad
\text{for each pair of join-irreducibles $(j_q,j_r) \in J(\aQ) \times J(\aR)$.}
\]
So let $h := \; \down_{\aQ,\aR^{\pOp}}^{j_q,j_r} : \aQ \to \aR^{\pOp}$, and also define $\rG := \Pirr\aQ \syncp \Pirr\aR$. To compute $\nu_{\aQ,\aR}(h)$ consider the action of the first two composites:
\[
\begin{tabular}{lll}
$\partial_{\breve{\rG}} \circ \alpha_{\aQ^{\pOp},\aR^{\pOp}}^{\pOp}(h)$
&
$= \partial_{\breve{\rG}}(\{ (j'_q,j'_r) : j'_r \nleq_\aR \; \down_{\aQ,\aR^{\pOp}}^{j_q,j_r}(j'_q) \})$
\\[0.5ex]&
$= \rG[ \{ (j'_q,j'_r) : j'_r \leq_\aR \; \down_{\aQ,\aR^{\pOp}}^{j_q,j_r}(j'_q) \} ]$
\\[0.5ex]&
$= \rG[\{ (j'_q,j'_r)  :  j'_q \leq_\aQ j_q \text{ and } j'_r \leq_\aR j_r \}]$
& $\top_{\aR^{\pOp}} = \bot_\aR \neq j_r$
\\&
$=  \{ (m_q,m_r) : \exists (j'_q,j'_r).[ j'_q \nleq_\aQ m_q \text{ and } j'_r \nleq_\aR m_r \text{ and } j'_q \leq_\aQ j_q \text{ and } j'_r \leq_\aR j_r    ] \}$
\\&
$= \{ (m_q,m_r) : j_q \nleq_\aQ m_q \text{ and } j_r \nleq_\aR m_r  \}$
\\&
$= \{ (m_q,m_r) : \; \up_{\aQ^{\pOp},\aR}^{j_q,j_r}(m_q) \nleq_\aR m_r \}$
\\&
$= \alpha_{\aQ,\aR}(\up_{\aQ^{\pOp},\aR}^{j_q,j_r})$
\end{tabular}
\]
where the final step follows from Corollary \ref{cor:tight_tensor_syncp}.1. Thus $\nu_{\aQ,\aR}(h) = \; \up_{\aQ^{\pOp},\aR}^{j_q,j_r}$, and we now use this to derive the action on an arbitrary tight morphism $f : \aQ \to \aR^{\pOp}$.
\[
\begin{tabular}{lll}
$\nu_{\aQ,\aR}(f)$
&
$= \nu_{\aQ,\aR}(\Land_{\jslTight{\aQ,\aR^{\pOp}}} \{ \down_{\aQ,\aR^{\pOp}}^{j_q,j_r} \; : f \leq \; \down_{\aQ,\aR^{\pOp}}^{j_q,j_r}  \})$
\\&
$= \nu_{\aQ,\aR}(\Land_{\jslTight{\aQ,\aR^{\pOp}}} \{ \down_{\aQ,\aR^{\pOp}}^{j_q,j_r} \; :  j_r \leq_\aR f(j_q)  \})$
\\&
$= \Lor_{\JSL_f[\aQ^{\pOp},\aR]} \{ \up_{\aQ^{\pOp},\aR}^{j_q,j_r} \; :  j_r \leq_\aR f(j_q) \}$
\\&
$= \lambda q \in Q. \Lor_\aR \{ j_r \in J(\aR) : \exists j_q \in J(\aQ).(j_q \nleq_\aQ q \text{ and } j_r \leq_\aR f(j_q) )  \}$
\\&
$= \lambda q \in Q. \Lor_\aR \{ j_r : \exists j_q.(j_q \nleq_\aQ q \text{ and } f(j_q) \leq_{\aR^{\pOp}} j_r )  \}$
\\&
$= \lambda q \in Q. \Lor_\aR \{ j_r : \exists j_q.(j_q \nleq_\aQ q \text{ and } j_q \leq_\aQ f_*(j_r) )  \}$
\\&
$= \lambda q \in Q. \Lor_\aR \{ j_r  \in J(\aR) : f_*(j_r) \nleq_\aQ q \}$
\end{tabular}
\]
Regarding the description of $\nu_{\aQ,\aR}$'s action on meets of special morphisms, given any $(q,r) \in Q \times R$ we have:
\[
\begin{tabular}{lll}
$\nu_{\aQ,\aR}(\down_{\aQ,\aR^{\pOp}}^{q,r})$
&
$= \lambda q' \in Q. \Lor_\aR \{ j \in J(\aR) : (\down_{\aQ,\aR^{\pOp}}^{q,r})_*(j) \nleq_\aQ q' \}$
& see above
\\&
$= \lambda q' \in Q. \Lor_\aR \{ j \in J(\aR) : \down_{\aR,\aQ^{\pOp}}^{r,q}(j) \nleq_\aQ q' \}$
& by Lemma \ref{lem:special_jsl_morphisms}.1
\\&
$= \lambda q' \in Q. \Lor_\aR \{ j \in J(\aR) : j \leq_\aR r \text{ and } q \leq_\aR q'  \}$
& since $\top_{\aQ^{\pOp}} = \bot_\aQ \leq_\aQ q'$
\\&
$= \lambda q' \in Q. \begin{cases} r & \text{if $q \leq_\aQ q'$} \\ \bot_\aR & \text{otherwise} \end{cases}$
\\&
$= \; \up_{\aQ^{\pOp},\aR}^{q,r}$
\end{tabular}
\]
and the first claim follows because meets are sent to joins. Applying $\nu_{\aQ,\aR}^{\bf-1} = \nu_{\aQ^{\pOp},\aR^{\pOp}}^{\pOp}$ yields:
\[
\Land_{\jslTight{\aQ,\aR^{\pOp}}} \{ \down_{\aQ,\aR^{\pOp}}^{q_i,r_i} \; : i \in I \}
= \nu_{\aQ^{\pOp},\aR^{\pOp}}(\Lor_{\jslTight{\aQ^{\pOp},\aR}} \{ \up_{\aQ^{\pOp},\aR}^{q_i,r_i} : i \in I \})
\]
and relabelling yields the final claim.
\end{proof}

We now use the canonical bijection $\Tight{\aQ^{\pOp},\aR} \cong \Tight{\aQ,\aR^{\pOp}}$ to explicitly describe the isomorphism between $\Pirr\aQ \syncp \Pirr\aR$'s open and closed sets. It has a simpler description in the presence of distributivity, as described in the Corollary following the Lemma.

\begin{lemma}
\label{lem:sync_of_pirrs_open_closed}
Given any finite join-semilattices $\aQ$, $\aR$ consider the relation $\rG := \Pirr\aQ \syncp \Pirr\aR$ . Then:
\[
\begin{tabular}{lllll}
$O(\rG)$ 
& $= \{ o_\rG(f) : f \in \jslTight{\aQ^{\pOp},\aR} \}$
& where 
& $o_\rG(f)$ & $:= \{ (m_q,m_r) \in M(\aQ) \times M(\aR) : f(m_q) \nleq_\aR m_r \}$,
\\[0.5ex]
$C(\rG)$
& $= \{  c_\rG(g) : g \in \jslTight{\aQ,\aR^{\pOp}} \}$
& where 
& $c_\rG(g)$ & $:= \{ (j_q,j_r) \in J(\aQ) \times J(\aR) : j_r \leq_\aR g(j_q)\}$,
\end{tabular}
\]
and the generic isomorphism $\theta_\rG : \latCl{\rG} \to \latOp{\rG}$ has action:
\[
\begin{tabular}{lll}
$\theta_\rG(c_\rG(g))$
& $= o_\rG(\nu_{\aQ,\aR}(g))$
& $= \{ (m_q,m_r) \in M(\aQ) \times M(\aR) : \; \up_{\aQ,\aR^{\pOp}}^{m_q,m_r}  \; \nleq_{\JSL_f[\aQ,\aR^{\pOp}]} g  \}$
\\[1ex]
$\theta_\rG^{\bf-1}(o_\rG(f))$
& $= c_\rG(\nu_{\aQ,\aR}^{\bf-1}(f))$
& $= \{ (j_q,j_r) \in J(\aQ) \times J(\aR) : \; \up_{\aQ^{\pOp},\aR}^{j_q,j_r} \; \leq_{\JSL_f[\aQ^{\pOp},\aR]} f \}$
\end{tabular}
\]
where $\nu_{\aQ,\aR}$ is the natural isomorphism from Theorem \ref{thm:tight_self_dual}.
\end{lemma}

\begin{proof}
\item
\begin{enumerate}
\item
The description of $O(\rG)$ follows directly from Corollary \ref{cor:tight_tensor_syncp}.1. Using the bounded lattice isomorphism $\kappa_\rG$, the elements of $C(\rG)$ are precisely the relative complements of the $\breve{\rG}$-open sets. Now, since:
\[
\breve{\rG} 
= (\Pirr\aQ \syncp \Pirr\aR)\spbreve
= (\Pirr\aQ)\spbreve \syncp (\Pirr\aR)\spbreve
= \Pirr(\aQ^{\pOp}) \syncp \Pirr(\aR^{\pOp})
\]
we deduce that $O(\breve{\rG})$ consists precisely of those sets $\{ (j_q,j_r) \in M(\aQ^{\pOp}) \times M(\aR^{\pOp}) : g(j_q) \nleq_{\aR^{\pOp}} j_r \}$ where $g : \aQ \to \aR^{\pOp}$ is tight, so the $\rG$-closed sets are those of the form:
\[
\{(j_q,j_r) \in J(\aQ) \times J(\aR) : j_r \leq_\aR g(j_q) \}
\]
as required.

\item
We finally verify the description of the bounded lattice isomorphism $\theta_\rG = \lambda X \in C(\rG).\rG[X]$ and also its inverse. Recall that elements of $C(\rG)$ take the form $c_\rG(g)$ where $g : \aQ \to \aR^{\pOp}$, and also that:
\[
\nu_{\aQ,\aR}(g) = \lambda q \in Q.\Lor_\aR \{ j_r \in J(\aR) : g_*(j_r) \nleq_\aQ q \}
\]
see Theorem \ref{thm:tight_self_dual}. Then we calculate:
\[
\begin{tabular}{lll}
$\theta_\rG(c_\rG(g))$
&
$= \rG[\{ (j_q,j_r) : j_r \leq_\aR g(j_q) \}]$
\\&
$= \rG[\{ (j_q,j_r) : j_q \leq_\aQ g_*(j_r) \}]$
& take adjoint
\\&
$= \{ (m_q,m_r) : \exists (j_q,j_r).( j_q \nleq_\aQ m_q \text{ and } j_r \nleq_\aR m_r \text{ and } j_q \leq_\aQ g_*(j_r)   )\}$
\\&
$= \{ (m_q,m_r) : \exists j_r.( j_r \nleq_\aR m_r \text{ and } g_*(j_r) \nleq_\aQ m_q ) \}$
& see below
\\&
$= \{ (m_q,m_r) : \nu_{\aQ,\aR}(g)(m_q) \nleq_\aR m_r \}$
& see above
\\&
$= o_\rG(\nu_{\aQ,\aR}(g))$
& by definition
\end{tabular}
\]
Regarding the marked equality, $(\To)$ follows because if $g_*(j_r) \leq_\aQ m_q$ we derive the contradiction $j_q \leq_\aQ g_*(j_r) \leq_\aQ m_q$. Also, $(\oT)$ follows because if $\forall j_q \in J(\aQ).(j_q \nleq_\aQ m_q \To j_q \nleq_\aQ g_*(j_r))$ then by converting to contrapositives we derive the contradiction $g_*(j_r) \leq_\aQ m_q$. Concerning the alternative description of $\theta_\rG(c(g))$,
\[
\exists j_r.( j_r \nleq_\aR m_r \text{ and } g_*(j_r) \nleq_\aQ m_q )
\iff
\; \up_{\aR,\aQ^{\pOp}}^{m_r,m_q}  \; \nleq_{\JSL_f[\aR,\aQ^{\pOp}]} g_*
\iff
\; \up_{\aQ,\aR^{\pOp}}^{m_q,m_r}  \; \nleq_{\JSL_f[\aQ,\aR^{\pOp}]} g 
\]
where the final equivalence uses the generic order-isomorphism $\JSL_f[\aR,\aQ^{\pOp}] \cong \JSL_f[\aQ,\aR^{\pOp}]$ i.e.\ take the adjoint, recalling that $(\up_{\aR,\aQ^{\pOp}}^{m_r,m_q})_* = \; \up_{\aQ,\aR^{\pOp}}^{m_q,m_r}$ by Lemma \ref{lem:special_jsl_morphisms}.1.

\smallskip
It remains to describe the action of the inverse $\theta_\rG^{\bf-1} = \lambda Y \in O(\rG).\rG^\down(Y)$. First recall that by De Morgan duality $\rG^\down = \neg_{\rG_s} \circ \breve{\rG}^\up \circ \neg_{\rG_t}$. Then for any tight morphism $f : \aQ^{\pOp} \to \aR$ we calculate:
\[
\begin{tabular}{lll}
$\theta_\rG^{\bf-1}(o_\rG(f))$
&
$= \overline{\theta_{\breve{\rG}}(\overline{o_\rG(f)})}$
\\&
$= \overline{\theta_{\breve{\rG}}(\{ (m_q,m_r) : f(m_q) \leq_\aR m_r  \})}$
\\&
$= \overline{\theta_{\breve{\rG}}(\{ (m_q,m_r) : m_r \leq_{\aR^{\pOp}} f(m_q)  \})}$
\\&
$= \overline{\theta_{\breve{\rG}}(c_{\breve{\rG}}(f))}$
& see below
\\&
$= \overline{o_{\breve{\rG}}(\nu_{\aQ^{\pOp},\aR^{\pOp}}(f))}$
& see above
\\&
$= \{ (j_q,j_r) : \nu_{\aQ^{\pOp},\aR^{\pOp}}(f)(j_q) \leq_{\aR^{\pOp}} j_r  \}$
\\&
$= \{ (j_q,j_r) :  j_r \leq_\aR \nu_{\aQ^{\pOp},\aR^{\pOp}}^{\pOp}(f)(j_q)  \}$
& recalling $(-)^{\pOp}$ has same action
\\&
$= \{ (j_q,j_r) :  j_r \leq_\aR \nu_{\aQ,\aR}^{\bf-1}(f)(j_q)  \}$
& see Theorem \ref{thm:tight_self_dual}
\\&
$= c_\rG(\nu_{\aQ,\aR}^{\bf-1}(f))$
& by definition
\end{tabular}
\]
To understand the marked equality, use the fact that $\breve{\rG} = \Pirr(\aQ^{\pOp}) \syncp \Pirr(\aR^{\pOp})$. Regarding the alternative description, we'll make use of the previous alternative description:
\[
\begin{tabular}{lll}
$\theta_\rG^{\bf-1}(o_\rG(f))$
&
$= \overline{\theta_{\breve{\rG}}(c_{\breve{\rG}}(f))}$
& see above
\\&
$= \overline{\{ (j_q,j_r) \in M(\aQ^{\pOp}) \times M(\aR^{\pOp}) : \; \up_{\aQ^{\pOp},\aR}^{j_q,j_r} \nleq_{\JSL_f[\aQ^{\pOp},\aR]} f \}}$
& see above
\\&
$= \{ (j_q,j_r) \in J(\aQ) \times J(\aR) : \; \up_{\aQ^{\pOp},\aR}^{j_q,j_r} \leq_{\JSL_f[\aQ^{\pOp},\aR]} f \}$
\end{tabular}
\] 

\end{enumerate}
\end{proof}

\begin{corollary}
\label{cor:sync_of_pirrs_distrib}
In the notation of Lemma \ref{lem:sync_of_pirrs_open_closed} where $\rG := \Pirr\aQ \syncp \Pirr\aR$, the following statements hold.
\begin{enumerate}
\item
If $\aQ$ is distributive then:
\[
\begin{tabular}{lll}
$\theta_\rG(c_\rG(g))$
& $= \{ (m_q,m_r) \in M(\aQ) \times M(\aR) : g(\tau_\aQ^{\bf-1}(m_q)) \nleq_\aR m_r \}$
\\[0.5ex]
$\theta_\rG^{\bf-1}(o_\rG(f))$
& $= \{ (j_q,j_r) \in J(\aQ) \times J(\aR) : j_r \leq_\aR f(\tau_\aQ(j_q)) \}$
\end{tabular}
\]
recalling that $\tau_\aQ = \lambda j \in J(\aQ).(\Lor_\aQ \overline{\up_\aQ j} \; \in M(\aQ))$ is the canonical order-isomorphism from Lemma \ref{lem:std_order_theory}.13.

\item
If $\aR$ is distributive then:
\[
\begin{tabular}{lll}
$\theta_\rG(c_\rG(g))$
& $= \{ (m_q,m_r) \in M(\aQ) \times M(\aR) : g_*(\tau_\aR^{\bf-1}(m_r)) \nleq_\aQ m_q \}$
\\[0.5ex]
$\theta_\rG^{\bf-1}(o_\rG(f))$
& $= \{ (j_q,j_r) \in J(\aQ) \times J(\aR) : j_q \leq_\aQ f_*(\tau_\aR(j_r)) \}$
\end{tabular}
\]

\item
If both $\aQ$, $\aR$ are distributive and $f : \aQ \to \aR$ is any $\JSL_f$-morphism, we deduce the equivalence:
\[
\begin{tabular}{lll}
$j_r \leq_\aR f(\tau_\aQ^{\bf-1}(m_q)) \iff f_*(\tau_\aR(j_r)) \leq_\aQ m_q$
& for every $(m_q,j_r) \in M(\aQ) \times J(\aR)$.
\end{tabular}
\]

\end{enumerate}
\end{corollary}

\begin{proof}
\item
\begin{enumerate}
\item
Regarding the description of $\theta_\rG^{\bf-1}$, by Lemma \ref{lem:sync_of_pirrs_open_closed} we know that:
\[
\begin{tabular}{lll}
$\theta_\rG^{\bf-1}(o_\rG(f))$
&
$= \{ (j_q,j_r) \in J(\aQ) \times J(\aR) : \;  \up_{\aQ^{\pOp},\aR}^{j_q,j_r} \; \leq f \}$
\\&
$= \{ (j_q,j_r) \in J(\aQ) \times J(\aR) : \forall q \in Q.[ j_q \nleq_\aQ q \To j_r \leq_\aR f(q)] \}$
\\&
$= \{ (j_q,j_r) \in J(\aQ) \times J(\aR) : \forall m_q \in M(\aQ).[ j_q \nleq_\aQ m_q \To j_r \leq_\aR f(m_q)] $
\end{tabular}
\]
where the final equality follows because morphisms $\aQ^{\pOp} \to \aR$ are determined by their action on $J(\aQ^{\pOp}) = M(\aQ)$. Let us prove that:
\[
\forall m_q \in M(\aQ).[ j_q \nleq_\aQ m_q \To j_r \leq_\aR f(m_q)]
\quad\iff\quad
j_q \leq_\aR f(\tau_\aQ(j_q))
\]
Then $(\To)$ follows because $j_q \nleq_\aQ \tau_\aQ(j_q) = \Lor_\aQ \overline{\up_\aQ j} \in M(\aQ)$, since the subset $\overline{\up_\aQ j} \subseteq Q$ is both down-closed and closed under joins (using join-primeness). Conversely, if $j_q \leq_\aQ m_q$ then certainly $m_q \leq_\aQ \tau_\aQ(j_q)$ and hence $j_q \leq_\aR f(\tau_\aQ(j_q)) \leq_\aR f(m_q)$ using monotonicity of $f : \aQ^{\pOp} \to \aR$.

To understand $\theta_\rG(c_\rG(g))$, recall that $\breve{\rG} = \Pirr(\aQ^{\pOp}) \syncp \Pirr(\aR^{\pOp})$ where $\aQ^{\pOp}$ is also distributive. Then:
\[
\begin{tabular}{lll}
$\theta_\rG(c_\rG(g))$
&
$= \overline{\theta_{\breve{\rG}}^{\bf-1}(o_{\breve{\rG}}(g))}$
& see proof of Lemma \ref{lem:sync_of_pirrs_open_closed}
\\&
$= \overline{\{ (m_q,m_r) \in M(\aQ) \times M(\aR) : m_r \leq_{\aR^{\pOp}} f(\tau_{\aQ^{\pOp}}(m_q)) \}}$
& using above
\\&
$= \{ (m_q,m_r) : f(\tau_\aQ^{\bf-1}(m_q)) \nleq_\aR m_r \}$
\end{tabular}
\]
where in the final equality we use the fact that $\tau_{\aQ^{\pOp}}$ acts like the inverse of $\tau_\aQ$.

\item
Let $\rH := \Pirr\aR \syncp \Pirr\aQ$. Taking the converse relation defines bounded lattice isomorphisms $\latOp{\rG} \cong \latOp{\rH}$ and $\latCl{\rG} \cong \latCl{\rH}$. Using the previous statement, we have:
\[
\begin{tabular}{lllll}
$c_\rG(g)$
& $\mapsto (c_\rG(g))\spbreve = c_\rH(g_*)$
& $\mapsto \{ (m_r,m_q) : g_*(\tau_\aR^{\bf-1}(m_r)) \nleq_{\aQ} m_q \}$
& $\mapsto \{ (m_q,m_r) : g_*(\tau_\aR^{\bf-1}(m_r)) \nleq_{\aQ} m_q \}$
\\[1ex]
$o_\rG(f)$
& $\mapsto (o_\rG(f))\spbreve = o_\rH(f_*)$
& $\mapsto \{ (j_r,j_q) : j_q \leq_\aQ f_*(\tau_\aR(j_r))  \}$
& $\mapsto \{ (j_q,j_r) : j_q \leq_\aQ f_*(\tau_\aR(j_r)) \}$
\end{tabular}
\]

\item
If both $\aQ$, $\aR$ are distributive then so is $\aQ^{\pOp}$. Let $\rG = \Pirr(\aQ^{\pOp}) \syncp \Pirr\aR$. Given any join-semilattice morphism $f : \aQ \to \aR$ then by (1) and (2) we have:
\[
\{(m_q,j_r) : j_r \leq_\aR f(\tau_{\aQ^{\pOp}}(m_q)) \}
= \theta_\rG^{\bf-1}(o_\rG(g))
= \{(m_q,j_r) : m_q \leq_{\aQ^{\pOp}} f_*(\tau_\aR (j_r)) \}
\]
and rewriting yields:
\[
j_r \leq_\aR f(\tau_{\aQ}^{\bf-1}(m_q))
\quad\iff\quad
f_*(\tau_\aR (j_r)) \leq_\aQ m_q
\]
The other bijection involving closed sets yields the same equivalence.

\end{enumerate}
\end{proof}

\begin{note}
Let us provide some basic examples of the equivalence:
\[
\forall m_q \in M(\aQ),\, j_r \in J(\aR). \;\;( j_r \leq_\aR f(\tau_\aQ^{\bf-1}(m_q)) \iff f_*(\tau_\aR(j_r)) \leq_\aQ m_q)
\]
which holds for any $\JSL_f$-morphism $f : \aQ \to \aR$ between distributive join-semilattices $\aQ$ and $\aR$. Take any relation $\rR \subseteq X \times Y$ between finite sets and let $f := \rR^\up : \JPow X \to \JPow Y$, so that $f_* = \rR^\down$. Noting that $M(\aQ) = \{ \overline{x} : x \in X \}$ and $J(\aR) = \{ \{y\} : y \in Y\}$, and moreover that the canonical order-isomorphisms $\tau_{\JPow X}$ and $\tau_{\JPow Y}$ take the relative complement, the equivalence becomes:
\[
\{y\} \subseteq \rR^\up(\{x\})
\quad\iff\quad
\rR^\down(\overline{y}) \subseteq \overline{x}.
\]
Applying De Morgan duality one sees this is the equivalence $y \in \rR[x] \iff x \in \overline{\rR^\down(\overline{y})} = \breve{\rR}[y]$. \endbox
\end{note}

\bigskip

We finish off with a clean set-theoretic description of the tensor product of finite distributive join-semilattices, which is also the tight tensor product. Let us agree that a `set-theoretic bounded distributive lattice' is one which arises as a sub bounded lattice of some $\DPow Z = (\Pow Z,\cup,\emptyset,\cap,Z)$, in which case $Z$ is uniquely determined.

\bigskip

\begin{theorem}[Representing the tensor and tight tensor product of distributive join-semilattices]
\item
Let $\aD_1$ and $\aD_2$ be finite distributive join-semilattices.
\begin{enumerate}
\item
Their tight tensor product and tensor product are isomorphic:
\[
\xymatrix@=10pt{
\JSL_f[\aD_1^{\pOp},\aD_2] \ar@{=}[r]
& \aD_1 \ttenp \aD_2
\ar[rr]^{\nu_{\aD_1,\aD_2}^{\bf-1}}_{\cong}
&&
\aD_1 \tenp \aD_2 \ar@{=}[r]
& (\JSL_f[\aD_1,\aD_2^{\pOp}])^{\pOp}
}
\]
using the isomorphism from Theorem \ref{thm:tight_self_dual}.

\item
If each $\aD_i$ defines a set-theoretic bounded distributive lattice over $Z_i$, then:
\[
\begin{tabular}{l}
$\trep{\aD_1,\aD_2} : \aD_1 \ttenp \aD_2 \to \ang{\{ j_1 \times j_2 : (j_1,j_2) \in J(\aD_1) \times J(\aD_2) \}}_{\JPow(Z_1 \times Z_2)}$
\\[1ex]
$\trep{\aD_1,\aD_2}(f) 
:= \bigcup \{ j_1 \times j_2 : (j_1,j_2) \in J(\aD_1) \times J(\aD_2), \; j_2 \subseteq f(\tau_{\aD_1}(j_1)) \}$
\end{tabular}
\]
defines a join-semilattice isomorphism. Regarding $\trep{\aD_1,\aD_2}$'s codomain $\aD$:
\begin{enumerate}
\item
it defines a set-theoretic bounded distributive lattice over $Z_1 \times Z_2$,
\item
its associated canonical bilinear mapping has action $(d_1,d_2) \mapsto d_1 \times d_2$,
\item
its irreducible elements are:
\[
J(\aD) = \{ j_1 \times j_2 : (j_1,j_2) \in J(\aD_1) \times J(\aD_2) \}
\qquad
M(\aD) = \{ \overline{\overline{m_1} \times \overline{m_2}} : (m_1,m_2) \in M(\aD_1) \times M(\aD_2) \}
\]
with associated canonical order-isomorphism $\tau_\aD(j_1 \times j_2) := \tau_{\aD_1}(j_1) \times Z_2 \; \cup \; Z_1 \times \tau_{\aD_2}(j_2)$.

\end{enumerate}

\end{enumerate}
\end{theorem}

\begin{proof}
\item
\begin{enumerate}
\item
Recall one of the characterisations of tight morphisms from Lemma \ref{lem:tight_mor_char} i.e.\ they are those $\JSL_f$-morphisms which factor through a distributive join-semilattice. Then:
\[
\begin{tabular}{lll}
$\JSL_f[\aD_1^{\pOp},\aD_2]$
&
$= \jslTight{\aD_1^{\pOp},\aD_2}$
& see above
\\&
$= \aD_1 \ttenp \aD_2$
& by definition
\\&
$\cong (\aD_1^{\pOp} \ttenp \aD_2^{\pOp})^{\pOp}$
& via $\nu_{\aD_1,\aD_2}^{\bf-1}$ from Theorem \ref{thm:tight_self_dual}
\\&
$= (\jslTight{\aD_1,\aD_2^{\pOp}})^{\pOp}$
& by definition
\\&
$= (\JSL_f[\aD_1,\aD_2^{\pOp}])^{\pOp}$
& see above
\\&
$= \aD_1 \tenp \aD_2$
& by definition
\end{tabular}
\]

\item
For notational convenience, define the induced posets $\pP_i := (J(\aD_i),\leq_{\aD_i} |_{J(\aD_i) \times J(\aD_i)} )$ for $i = 1,2$. We are going to construct $\trep{\aD_1,\aD_2}$ as a composite isomorphism:
\[
\aD_1 \ttenp \aD_2
\xto{\alpha_{\aD_1, \aD_2}}
\Open(\Pirr\aD_1 \syncp \Pirr\aD_2)
\xto{\beta}
(Dn(\pP_1 \times \pP_2),\cup,\emptyset)
\xto{\delta}
\ang{ \{ j_1 \times j_2 : (j_1,j_2) \in P_1 \times P_2  \} }_{\JPow(Z_1 \times Z_2)}
\]
recalling that $Dn(-)$ constructs the collection of down-closed subsets of a poset. The first isomorphism instantiates the natural isomorphism from Corollary \ref{cor:tight_tensor_syncp}.1:
\[
\alpha_{\aD_1, \aD_2}(f) := \{ (m_1,m_2) \in M(\aD_1) \times M(\aD_2) : f(m_1) \nleq_{\aD_2} m_2 \}
\]
i.e.\ the fact that tight tensors correspond to synchronous products. Concerning $\beta$, recall that for any finite distributive join-semilattice $\aD$ we have the bipartite graph isomorphism shown below on the left.
\[
\begin{tabular}{lll}
$\vcenter{\vbox{\xymatrix@=15pt{
M(\aD) \ar[rr]^{\tau_\aD^{\bf-1}} && J(\aD)
\\
J(\aD) \ar[u]^{\Pirr\aD} \ar[rr]_{\Delta_{J(\aD)}} && J(\aD) \ar[u]_{\geq_{\aD}}
}}}$
&&
$\rR_\aD := \; \geq_\aD |_{J(\aD) \times J(\aD)} \; : \Pirr\aD \to \; \geq_\aD |_{J(\aD) \times J(\aD)}$
\end{tabular}
\]
Here we are using the canonical order-isomorphism $\tau_\aD$ between join/meet-irreducibles, see Lemma \ref{lem:std_order_theory}.14. This commutative square witnesses a $\BiCliq$-isomorphism $\rR_\aD$ shown above. To provide some clarification, $\Open\rR_\aD$ is the well-known isomorphism representing $\aD \cong \Open\Pirr\aD$ as down-closed sets of join-irreducibles. These $\BiCliq$-isomorphisms induce the $\JSL_f$-isomorphism:
\[
\beta := \Open (\rR_{\aD_1} \syncp \rR_{\aD_2})  : \Open(\Pirr\aD_1 \syncp \; \Pirr\aD_2) \to \Open(\geq_{\pP_1} \syncp \geq_{\pP_2})
\]
since functors preserve isomorphisms. To see that $\beta$'s codomain is correct,  observe that the synchronous product of two order relations is the order relation of the product of their respective posets, and also that the open sets of an order relation are precisely the up-closed subsets of its corresponding poset, hence:
\[
\Open(\geq_{\pP_1} \syncp \geq_{\pP_2}) = (Dn(\pP_1 \times \pP_2),\cup,\emptyset).
\]
To compute the action of $\beta \circ \alpha_{\aD_1, \aD_2}$, let $\rG := \Pirr\aD_1 \syncp \Pirr\aD_2$ and recall that:
\[
\begin{tabular}{lll}
$\beta(Y)
= \Open(\rR_{\aD_1} \syncp \rR_{\aD_2})(Y)
= (\rR_{\aD_1} \syncp \rR_{\aD_2})^\up \circ \rG^\down(Y)$
&&
see Definition \ref{def:open_pirr}.1.
\\[1ex]
$\alpha_{\aD_1,\aD_2}(f) = \{ (m_1,m_2) \in M(\aD_1) \times M(\aD_2) : f(m_1) \nleq_{\aD_2} m_2 \} = o_\rG(f)$
&&
using notation of Lemma \ref{lem:sync_of_pirrs_open_closed}.
\end{tabular}
\]
 Thus:
\[
\begin{tabular}{lll}
$\beta \circ \alpha_{\aD_1, \aD_2}(f)$
&
$= (\rR_{\aD_1} \syncp \rR_{\aD_2})^\up \circ \rG^\down(o_\rG(f))$
\\&
$= (\rR_{\aD_1} \syncp \rR_{\aD_2})^\up(\theta_\rG^{\bf-1}(o_\rG(f))$
& recall Definition \ref{def:bip_cl_inte_lattice}
\\&
$= (\rR_{\aD_1} \syncp \rR_{\aD_2})^\up(\{ (j_1,j_2) \in J(\aD_1) \times J(\aD_2) : j_2 \leq_{\aD_2} f(\tau_{\aD_1}(j_1))  \})$
& by Corollary \ref{cor:sync_of_pirrs_distrib}
\\&
$= \; \geq_{\pP_1 \times \pP_2} [\{ (j_1,j_2) : j_2 \leq_{\aD_2} f(\tau_{\aD_1}(j_1))  \}]$
& by definition
\\&
$= \{ (j_1,j_2)  : j_2 \leq_{\aD_2} f(\tau_{\aD_1}(j_1))\}$
& see below
\\&
$= \{ (j_1,j_2)  : j_2 \subseteq f(\tau_{\aD_1}(j_1))\}$
& $\leq_{\aD_2}$ is inclusion
\end{tabular}
\]
The marked equality follows because the set is already down-closed i.e.\ given $(j'_1,j'_2) \leq_{\pP_1 \times \pP_2} (j_1,j_2)$ then:
\[
j'_2 \leq_{\aD_1} j_2 \leq_{\aD_2} f(\tau_{\aD_1}(j_1)) \leq_{\aD_2} f(\tau_{\aD_1}(j'_1))
\qquad
\text{using the monotonicity of $\tau_{\aD_1}$ and $f : \aD_1^{\pOp} \to \aD_2$.}
\]
It remains to describe the isomorphism $\delta$. Firstly, for each $\aD_i$ consider the diagram below:
\[
\begin{tabular}{lll}
$\vcenter{\vbox{\xymatrix@=15pt{
J(\aD_i) \ar[rr]^{\breve{\in}} && Z_i
\\
J(\aD_i) \ar[u]^{\geq_{\pP_i}} \ar[rr]_{\breve{\in}} && Z_i \ar[u]_{\Delta_{Z_i}}
}}}$
&&
$\rS_i := \breve{\in} : \; \geq_{\pP_i} \; \to \Delta_{Z_i}$
\end{tabular}
\]
which commutes because $\aD_i$ is inclusion-ordered, and hence witnesses the $\BiCliq$-morphism $\rS_i$ above. Recall that $\rS_i$ is monic iff $\cl_{\rS_i} = \cl_{\geq_{\pP_i}}$ by Lemma \ref{lem:bicliq_mono_epi_char}. The closure operator induced by an order relation constructs the upwards closure in the respective poset, as one may verify. Then since:
\[
\begin{tabular}{lll}
$\cl_{\rS_i}(X)$
&
$= (\breve{\in})^\down \circ (\breve{\in})^\up (X)$
\\&
$= (\breve{\in})^\down (\bigcup X)$
\\&
$= \{ j \in J(\aD_i) : j \subseteq \bigcup X \}$
\\&
$= \; \down_{\pP_i} X$
& $j$ is join-prime, see Lemma \ref{lem:std_order_theory}.2
\end{tabular}
\]
we deduce that each $\rS_i$ is monic. Consider the induced join-semilattice morphism:
\[
\gamma := \Open(\rS_1 \syncp \rS_2) : 
(Dn(\pP_1 \times \pP_2),\cup,\emptyset)
\monoto \Open(\Delta_{Z_1} \syncp \Delta_{Z_2}) = \JPow (Z_1 \times Z_2)
\]
It is injective because the synchronous product functor preserves monos by Crollary \ref{cor:sync_prod_pres_monos}, and also $\Open$ is an equivalence functor. Then the isomorphism $\delta$ is defined by restricting $\gamma$'s codomain to its embedded image:
\[
\xymatrix@=15pt{
(Dn(\pP_1 \times \pP_2),\cup,\emptyset) \ar@{>->}[rrr]^-{\gamma} \ar[rrrd]_-{\delta}^-{\cong} &&& \JPow (Z_1 \times Z_2)
\\
&&& \ang{\{ j_1 \times j_2 : (j_1,j_2) \in \pP_1 \times \pP_2 \}}_{\JPow (Z_1 \times Z_2)} \ar@{>->}[u]
}
\]
To understand $\delta$ and its codomain, we describe $\gamma$'s action on any downset $X \in Dn(\pP_1 \times \pP_2)$:
\[
\begin{tabular}{lll}
$\gamma(X) $
&
$= (\rS_1 \syncp \rS_2)^\up \circ (\geq_{\pP_1 \times \pP_2})^\down(X)$
\\&
$= \rS_1 \syncp \rS_2 [X]$
& since $X$ is down-closed
\\&
$= \breve{\in}  \syncp \breve{\in} [X]$
\\&
$= \bigcup \{ j_1 \times j_2 : (j_1,j_2) \in X \}$
\end{tabular}
\]
Since $\delta$'s domain is union-generated by the principal downsets, its codomain is correct. Consequently, the composite isomorphism $\trep{\aD_1,\aD_2} := \delta \circ \beta \circ \alpha_{\aD_1, \aD_2}$ has the desired action:
\[
\begin{tabular}{lll}
$\trep{\aD_1,\aD_2}(f)$
&
$= \delta \circ \beta \circ \alpha_{\aD_1, \aD_2}(f)$
\\&
$= \delta(\{ (j_1,j_2) \in \pP_1 \times \pP_2 : j_2 \subseteq f(\tau_{\aD_1}(j_1)) \})$
\\&
$= \bigcup \{ j_1 \times j_2 \in \pP_1 \times \pP_2 : j_2 \subseteq f(\tau_{\aD_1}(j_1)) \}$
\end{tabular}
\]

\smallskip
Finally we verify the claimed properties of $\trep{\aD_1,\aD_2}$'s codomain distributive join-semilattice, denoted $\aD$.

\begin{enumerate}[(a)]
\item
To see that $\aD$ is set-theoretic, we'll show that $\gamma$ defines a distributive lattice morphism between the two set-theoretic distributive join-semilattices. The top element is preserved because $Z_1 \times Z_2$ equals the union over all $j_1 \times j_2$'s. It preserves binary meets iff it preserves meets of join-irreducibles (apply distributivity twice), and finally:
\[
\begin{tabular}{lll}
$\gamma(\down_{\pP_1 \times \pP_2} (j_1,j_2) \;\cap \down_{\pP_1 \times \pP_2} (j'_1,j'_2))$
&
$= \gamma(\down_{\pP_1 \times \pP_2} ((j_1 \cap j'_1),(j_2 \cap j'_2))$
\\&
$= (j_1 \cap j'_1) \times (j_2 \cap j'_2)$
\\&
$= (j_1 \times j_2) \cap (j'_1 \times j'_2)$
\end{tabular}
\]
since binary intersections and products commute.

\item
The canonical bilinear map associated to the tensor product $\aD_1 \tenp \aD_2$ is the function: 
\[
\beta_{\aD_1,\aD_2} : D_1 \times D_2 \to \JSL_f(\aD_1,\aD_2^{\pOp})
\qquad
\beta_{\aD_1,\aD_2}(d_1,d_2) := \; \down_{\aD_1,\aD_2^{\pOp}}^{d_1,d_2}
\]
recalling that $\aD_1 \tenp \aD_2 = (\JSL_f[\aD_1,\aD_2^{\pOp}])^{\pOp}$, see Definition \ref{def:tenp}. Since the tensor product of distributive join-semilattices is isomorphic to their tight tensor product by (1), we may equivalently view $\aD$ as the tensor product of $\aD_1$ and $\aD_2$. Then $\aD$'s associated canonical bilinear map arises by composing with the isomorphism as follows:
\[
\xymatrix@=5pt{
D_1 \times D_2 \ar[rr]^-{\beta_{\aD_1, \aD_2}}
&& \JSL_f(\aD_1,\aD_2^{\pOp}) \ar[rr]_-{\cong}^-{\nu_{\aD_1,\aD_2}}
&& \JSL_f(\aD_1^{\pOp},\aD_2) \ar[rr]_-{\cong}^-{\trep{\aD_1,\aD_2}}
&& D
\\
(d_1,d_2) & \mapsto & \down_{\aD_1,\aD_2^{\pOp}}^{d_1,d_2} & \mapsto & \up_{\aD_1^{\pOp},\aD_2}^{d_1,d_2} & \mapsto & \trep{\aD_1,\aD_2}(\up_{\aD_1^{\pOp},\aD_2}^{d_1,d_2})
}
\]
also using Theorem \ref{thm:tight_self_dual}. It only remains to simplify the latter:
\[
\begin{tabular}{lll}
$\trep{\aD_1,\aD_2}(\up_{\aD_1^{\pOp},\aD_2}^{d_1,d_2})$
&
$= \bigcup \{ j_1 \times j_2 : j_2 \subseteq  \; \up_{\aD_1^{\pOp},\aD_2}^{d_1,d_2}(\tau_{\aD_1}(j_1)) \}$
\\&
$= \bigcup \{ j_1 \times j_2 : d_1 \nsubseteq \tau_{\aD_1}(j_1) \text{ and } j_2 \subseteq d_2 \}$
\\&
$= \bigcup \{ j_1 \times j_2 : d_1 \nsubseteq \bigcup \{ d \in D_1 : j_1 \nsubseteq d \} \text{ and } j_2 \subseteq d_2 \}$ 
& by definition
\\&
$= \bigcup \{ j_1 \times j_2 : j_1 \subseteq d_1 \text{ and } j_2 \subseteq d_2 \}$ 
\\&
$= \bigcup \{ j_1 \times j_2 : j_1 \times j_2 \subseteq d_1 \times d_2 \}$ 
\\&
$= d_1 \times d_2$
\end{tabular}
\]

\item
By construction the sets $j_1 \times j_2$ union-generate $\aD$, so that every join-irreducible takes this form. Then since $|J(\aD_1 \ttenp \aD_2)| = |J(\aD_1)| \cdot |J(\aD_2)|$ by Lemma \ref{lem:tight_tensor_irr}.1, they are precisely $\aD$'s join-irreducibles. Since $\aD$ is distributive we have the canonical order-isomorphism between its join/meet-irreducibles, with action:
\[
\begin{tabular}{lll}
$\tau_\aD(j_1 \times j_2)$
&
$= \Lor_{\aD} \{ d \in D : j_1 \times j_2 \nleq_\aD d \}$
\\&
$= \bigcup \{ j'_1 \times j'_2 \in D : j_1 \times j_2 \nsubseteq j'_1 \times j'_2 \}$
\\&
$= \bigcup \{ j'_1 \times j'_2 \in D : j_1 \nsubseteq j'_1 \text{ or } j_2 \nsubseteq j'_2 \}$
\\&
$= \bigcup \{ j'_1 \times j'_2 \in D : j_1 \nsubseteq j'_1 \} \,\cup\, \bigcup \{ j'_1 \times j'_2 \in D : j_2 \nsubseteq j'_2 \}$
\\&
$= (\bigcup \{j'_1 : j_1 \nsubseteq j'_1)) \times Z_2 \,\cup\, Z_1 \times \bigcup \{ j'_2 : j_2 \nsubseteq j'_2 \}$ 
\\&
$= \tau_{\aD_1}(j_1) \times Z_2 \,\cup\, Z_1 \times \tau_{\aD_2}(j_2)$
\\&
$= \overline{\overline{\tau_{\aD_1}(j_1)} \times \overline{\tau_{\aD_1}(j_1)}}$
\end{tabular}
\]
Thus $M(\aD)$ contains precisely the elements $\overline{\overline{m_1} \times \overline{m_2}}$ where $(m_1,m_2) \in M(\aD_1) \times M(\aD_2)$.
\end{enumerate}
\end{enumerate}
\end{proof}

\subsection{Tightness inside $\BiCliq$ and the universality of the tight tensor product}

We now define the $\BiCliq$-correspondents of various concepts  above i.e.\ tight join-semilattice morphisms, the special join-irreducible join-semilattice morphisms $\up_{\aQ,\aR}^{m,j} : \aQ \to \aR$ which join-generate them, and also the tight hom-functor. It is also worth describing the correspondents of the special morphisms $\down_{\aQ,\aR}^{j,m} : \aQ \to \aR$, seeing as they are both the join-irreducibles of $\aQ \tenp \aR^{\pOp} = (\JSL_f[\aQ,\aR])^{\pOp}$ by Lemma \ref{lem:tenp_basic}, and also the meet-irreducibles of  $\aQ^{\pOp} \ttenp \aR = \jslTight{\aQ,\aR}$ by Lemma \ref{lem:tight_tensor_irr}. On the other hand, it is the correspondents of $\up_{\aQ,\aR}^{m,j}$ which take the leading role in this final subsection.

\smallskip

\begin{definition}[Tight $\BiCliq$-morphisms, basic bicliques and basic independents]
\label{def:tight_bicliq_mor}
\item
\begin{enumerate}
\item
A $\BiCliq$-morphism $\rR : \rG \to \rH$ is \emph{tight} if it factors through an identity relation i.e.\ $\rR = \rS \fatsemi \rT$ for some $\BiCliq$-morphisms $\rS : \rG \to \Delta_Z$, $\rT : \Delta_Z \to \rH$ and some finite set $Z$.

\item
Let $\Tight{\rG,\rH} \subseteq \BiCliq(\rG,\rH)$ be the subset of tight morphisms. Then we have the join-semilattice:
\[
\biTight{\rG,\rH} := (\Tight{\rG,\rH},\cup,\emptyset) \subseteq \BiCliq[\rG,\rH]
\]
This extends to a functor $\biTight{-,-} : \BiCliq^{op} \times \BiCliq \to \JSL_f$ whose action on morphisms is:
\[
\dfrac{\rR : \rG \to \rH \quad \rS : \rG' \to \rH'}
{\biTight{\rR^{op},\rS} := \lambda \rT. \rR \fatsemi \rT \fatsemi \rS : \biTight{\rH,\rG'} \to \biTight{\rG,\rH'}}
\]
this being precisely the way that $\BiCliq[-,-]$ acts, see Definition \ref{def:bicliq_hom_functor_jsl}.

\item
Given relations $\rG$, $\rH$ and elements $(g_t,h_s) \in \rG_t \times \rH_s$, there is an associated $\BiCliq$-morphism:
\[
\begin{tabular}{c}
$\upsp_{\rG,\rH}^{g_t,h_s} \; := \BC{\breve{\rG}[g_t],\rH[h_s]} : \rG \to \rH$
\\[1ex]
$(\upsp_{\rG,\rH}^{g_t,h_s})_- := \emptyset_- \cup \BC{\breve{\rG}[g_t],\cl_\rH(\{h_s\})}$
\qquad
$(\upsp_{\rG,\rH}^{g_t,h_s})_+ := \emptyset_+ \cup \BC{\rH[h_s],\cl_{\breve{\rG}}(\{g_t\})}$
\end{tabular}
\]
Call them \emph{basic bicliques}, where distinct pairs $(g_t,h_s) \neq (g'_t,h'_s)$ may induce the same morphism. 

\item
Given relations $\rG$, $\rH$ and elements $(g_s,h_t) \in \rG_s \times \rH_t$, there is an associated $\BiCliq$-morphism:
\[
\begin{tabular}{lll}
$\downsp_{\rG,\rH}^{g_s,h_t}$
& $:= \top_{\BiCliq[\rG,\rH]} \, \cap \, \overline{\cl_\rG(\{g_s\}) \times \cl_{\breve{\rH}}(\{h_t\})}$
\\[0.5ex]&
$= \inte_{\breve{\rG}}(\overline{g_s}) \times \rH[\rH_s] \;\cup\; \breve{\rG}[\rG_t] \times \inte_\rH(\overline{h_t})$
& $: \rG \to \rH$
\end{tabular}
\]
\[
\begin{tabular}{llll}
$(\downsp_{\rG,\rH}^{g_s,h_t})_-$
&
$:= (\top_{\BiCliq[\rG,\rH]})_- \,\cap\, \overline{\cl_\rG(\{g_s\}) \times \breve{\rH}[h_t]}$
&
$= \emptyset_- \, \cup \, \inte_{\breve{\rG}}(\overline{g_s}) \times \rH_t \, \cup \, \breve{\rG}[\rG_t] \times \rH^\down(\overline{h_t})$
\\[1ex]
$(\downsp_{\rG,\rH}^{g_s,h_t})_+$
&
$:= (\top_{\BiCliq[\rG,\rH]})_+ \,\cap\, \overline{\cl_{\breve{\rH}}(\{h_t\}) \times \rG[g_s]}$
&
$= \emptyset_+ \,\cup\, \inte_\rH(\overline{g_s}) \times \rG_s \,\cup\, \rH[\rH_s] \times \breve{\rG}^\down(\overline{g_s})$
\end{tabular}
\]
Call them \emph{basic independents}, where distinct pairs $(g_s,h_t) \neq (g'_s,h'_t)$ may induce the same morphism. \endbox

\end{enumerate}
\end{definition}

\begin{example}[Understanding the special $\BiCliq$-morphisms]
\label{ex:special_bicliq_mor}
\item
\begin{enumerate}
\item
Each basic biclique $\upsp_{\rG,\rH}^{g_t,h_s} \; = \BC{\breve{\rG}[g_t], \rH[h_s]} : \rG \to \rH$ is well-defined via the witnesses:
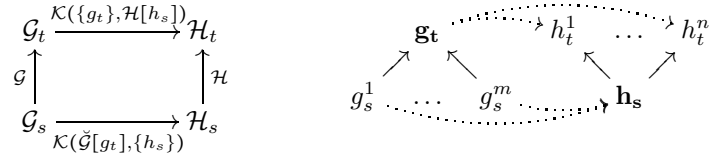
\[
\xymatrix@=20pt{
\rG_t \ar[rr]^{\BC{\{g_t\},\rH[h_s]}} && \rH_t
\\
\rG_s \ar[u]^\rG \ar[rr]_{\BC{\breve{\rG}[g_t],\{h_s\}}} && \rH_s \ar[u]_\rH
}
\qquad\qquad
\xymatrix@=8pt{
& {\bf g_t} \ar@/^5pt/@{..>}[rr] \ar@/^10pt/@{..>}[rrrr] &  & h_t^1 & \dots & h_t^n
\\
g_s^1 \ar@{..>}@/_10pt/[rrrr] \ar[ur] & \dots & g_s^m \ar@{..>}@/_5pt/[rr] \ar[ul] & & {\bf h_s} \ar[ul] \ar[ur]
}
\]
Indeed, $g_t$ may be viewed as the apex of a `cone' with base $\breve{\rG}[g_t]$, and $h_s$ may be viewed as the apex of an upside-down cone with base $\rH[h_s]$. Connecting each cone's apex to the other cone's base yields the diagram above on left, whose support is shown on the right. Closing these witnesses yields the components described earlier. As suggested by the notation, these morphisms correspond to the special morphisms $\up_{\aQ,\aR}^{m,j}$ under the equivalence functors $\Pirr$ and $\Open$, see Lemma \ref{lem:bicliq_tight_basic}.5 below. Importantly, a $\BiCliq$-morphism is tight iff it is a union of basic bicliques.

\item
Regarding the previous example, it seems natural to `flip' the two cones upside-down i.e.\ we take $g_s \in \rG_s$ as the apex of an upside-down cone with base $\rG[g_s]$, and $h_t \in \rH_t$ as the apex of a cone with base $\breve{\rH}[h_t]$.
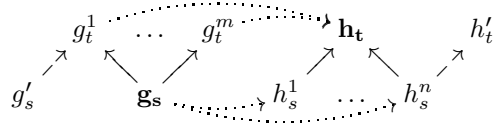
\[
\xymatrix@=8pt{
& g_t^1 \ar@/^10pt/@{..>}[rrrr] & \dots & g_t^m \ar@/^5pt/@{..>}[rr]  &  & {\bf h_t} && h'_t
\\
g'_s \ar@{-->}[ur] && {\bf g_s} \ar@{..>}@/_10pt/[rrrr] \ar@/_5pt/@{..>}[rr] \ar[ul] \ar[ur] & & h_s^1 \ar[ur] & \dots & h_s^n \ar[ul] \ar@{-->}[ur]
}
\]
However, gluing apexes to bases need not yield witnessing relations, as indicated by either of the two dashed arrows shown above.

\item
It turns out that the `natural' choice for a $\BiCliq$-morphism $\rG \to \rH$ depending on elements $(g_s,h_t) \in \rG_s \times \rH_t$ is the basic independent morphism:
\[
\downsp_{\rG,\rH}^{g_s,h_t} \;
= \top_{\BiCliq[\rG,\rH]} \cap \overline{\cl_\rG(\{g_s\}) \times \cl_{\breve{\rH}}(\{h_t\})}
\]
Let us first explain its alternate description from the Definition:
\[
\begin{tabular}{lll}
$\downsp_{\rG,\rH}^{g_s,h_t}$
&
$= \top_{\BiCliq[\rG,\rH]} \cap \overline{\cl_\rG(\{g_s\}) \times \cl_{\breve{\rH}}(\{h_t\})}$
\\&
$= \breve{\rG}[\rG_t] \times \rH[\rH_s] \,\cap\, (\overline{\cl_\rG(\{g_s\})} \times \rH_t \,\cup\, \rG_s \times \overline{\cl_{\breve{\rH}}(\{h_t\})})$
\\&
$= \breve{\rG}[\rG_t] \times \rH[\rH_s] \,\cap\, (\inte_{\breve{\rG}}(\overline{g_s}) \times \rH_t \,\cup\, \rG_s \times \inte_{\rH}(\overline{h_t}))$
\\&
$= \inte_{\breve{\rG}}(\overline{g_s}) \times \rH[\rH_s] \,\cup\, \breve{\rG}[\rG_t] \times \inte_{\rH}(\overline{h_t})$
\end{tabular}
\]
Here we have used Lemma \ref{lem:bicliq_mor_jsl_struct}.3 i.e.\ that $\top_{\BiCliq[\rG,\rH]}$ relates  everything which is not isolated, and also De Morgan duality and the fact that binary intersections commute with binary products. That the basic independent $\downsp_{\rG,\rH}^{g_s,h_t}$ is a well-defined $\BiCliq$-morphism follows because it is a union of basic bicliques:
\[
\begin{tabular}{ll}
\begin{tabular}{ll}
$\inte_{\breve{\rG}}(\overline{g_s}) \times \rH[h_s]$
&
$= \breve{\rG}[\breve{\rG}^\down(\overline{g_s})] \times \rH[\rH_s]$
\\&
$= \bigcup \{ \breve{\rG}[g_t] \times \rH[h_s] : g_t \in \breve{\rG}^\down(\overline{g_s}),\, h_s \in \rH_s \}$
\\&
$= \bigcup \{ \upsp_{\rG,\rH}^{g_s,h_t} \; : g_t \nin \breve{\rG}[g_s] ,\, h_s \in \rH_s \}$
\\
\\
$\breve{\rG}[\rG_t] \times \inte_{\rH}(\overline{h_t})$
&
$= \breve{\rG}[\rG_t] \times \rH^\up(\rH^\down(\overline{h_t}))$
\\&
$= \bigcup \{ \breve{\rG}[g_t,h_s] : g_t \in \rG_t, \, h_s \in \rH^\down(\overline{h_t}) \}$
\\&
$= \bigcup \{ \upsp_{\rG,\rH}^{g_t,h_s} \; : g_t \in \rG_t, \, h_s \nin \breve{\rH}[h_t] \}$
\end{tabular}
\end{tabular}
\]
It is slightly tedious to verify the associated components and their alternative descriptions. They simplify in the case that $\rG$ and $\rH$ are strict, which is easily enforced. Of course, $\downsp_{\rG,\rH}^{g_s,h_t}$ corresponds to the special morphisms $\down_{\aQ,\aR}^{j,m}$ as we prove below. Furthermore, their description as a binary union of  $\breve{\rG}[\rG_t] \times \inte_{\rH}(\overline{h_t})$ and $\inte_{\breve{\rG}}(\overline{g_s}) \times \rH[\rH_s]$ corresponds to an equality we've already seen i.e.\
\[
\down_{\aQ,\aR}^{j,m}
\; = \; \up_{\aQ,\aR}^{\bot_\aQ,m} \; \lor_{\JSL_f[\aQ,\aR]} \; \up_{\aQ,\aR}^{j,\top_\aR}
\]
see Corollary \ref{cor:tight_special_morphisms}. As mentioned above, we will not have much use for the basic independents in this subsection. However we make one more observation. Recall that the morphisms $(\down_{\aQ,\aR}^{j,m})_{j \in J(\aQ),m \in M(\aR)}$ are precisely the meet-irreducibles of $\JSL_f[\aQ,\aR]$ by Lemma \ref{lem:hom_meet_join_irr}. Analogously, the basic independents $\downsp_{\rG,\rH}^{g_s,h_t}$ are precisely the meet-irreducibles of $\BiCliq[\rG,\rH]$, as long as both $\rG$ and $\rH$ are reduced. \endbox
\end{enumerate}
\end{example}

\smallskip

\begin{lemma}[Tightness inside $\BiCliq$]
\label{lem:bicliq_tight_basic}
\item
Let $\rG$, $\rH$ be any relations between finite sets.
\begin{enumerate}

\item
A $\BiCliq$-morphism $\rR$ is tight iff $\Open\rR$ is tight, a $\JSL_f$-morphism $f$ is tight iff $\Pirr f$ is tight.

\item
$\biTight{\rG,\rH}$ is a well-defined join-semilattice.

\item
$\biTight{-,-} : \BiCliq^{op} \times \BiCliq \to \JSL_f$ is a well-defined functor.

\item
Basic bicliques and basic independents are well-defined tight $\BiCliq$-morphisms.

\item
We have the equalities:
\[
\begin{tabular}{c}
$(\upsp_{\rG,\rH}^{g_t,h_s})\spcheck = \; \upsp_{\breve{\rH},\breve{\rG}}^{h_s,g_t}$
\qquad
$(\downsp_{\rG,\rH}^{g_s,h_t})\spcheck = \; \downsp_{\breve{\rH},\breve{\rG}}^{h_t,g_s}$
\\[1.5ex]
$\Open\upsp_{\rG,\rH}^{g_t,h_s} \; = \;\up_{\Open\rG,\Open\rH}^{\inte_\rG(\overline{g_t}), 
\rH[h_s]}$
\qquad
$\Open\downsp_{\rG,\rH}^{g_s,h_t} \; = \; \down_{\Open\rG,\Open\rH}^{\rG[g_s],\inte_\rH(\overline{h_t})}$
\\[1.5ex]
$\Pirr\up_{\aQ,\aR}^{q,r} \; = \; \bigcup \{ \upsp_{\Pirr\aQ,\Pirr\aR}^{m,j} \; : q \leq_\aQ m \in M(\aQ),\, J(\aR) \ni j \leq_\aR r \}$
\end{tabular}
\]
for any $(g_t,h_s) \in \rG_t \times \rH_s$, $(g_s,h_t) \in \rG_s \times \rH_t$ and any finite join-semilattices $\aQ$, $\aR$ with $(q,r) \in Q \times R$.

\takeout{
\\[1ex]
$\Pirr\down_{\aQ,\aR}^{q,r} \; = \; \bigcup \{ \downsp_{\Pirr\aQ,\Pirr\aR}^{m,j} \; : q \leq_\aQ m \in M(\aQ),\, J(\aR) \ni j \leq_\aR r \}$
}

\item
A $\BiCliq$-morphism $\rR : \rG \to \rH$ is tight iff it is a union of basic bicliques of type $\rG \to \rH$.

\item
We have the equalities:
\[
\begin{tabular}{llll}
$\rR \, \fatsemi \upsp_{\rG,\rH}^{g_t,h_s}$
&
$= \BC{\breve{\rR}[g_t]\,,\,\rH[h_s]}$
&
$= \bigcup \{ \upsp_{\rF,\rH}^{f_t,h_s} : f_t \in \rR_+[g_t] \}$
&
$: \rF \to \rH$
\\[1ex]
$\upsp_{\rG,\rH}^{g_t,h_s} \fatsemi \, \rS$
&
$= \BC{\breve{\rG}[g_t]\,,\,\rS[h_s]} $
&
$= \bigcup \{ \upsp_{\rG,\rI}^{g_t,i_s} : i_t \in \rS_-[h_s]  \}$
&
$: \rG \to \rI$
\end{tabular}
\]
for any elements $(g_t,h_s) \in \rG_t \times \rH_s$, and any $\BiCliq$-morphisms $\rR : \rF \to \rG$, $\rS : \rH \to \rI$.

\item
For every $\BiCliq$-morphism $\rR : \rG \to \rH$ we have:
\[
\rR \subseteq \;\downsp_{\rG,\rH}^{g_s,h_t}
\iff \overline{\rR}(g_s,h_t).
\]

\end{enumerate}
\end{lemma}

\begin{proof}
\item
\begin{enumerate}
\item
If $\rR : \rG \to \rH$ is tight via $\rR = \rS \fatsemi \rT$ then $\Open\rR = \Open\rT \circ \Open \rS$ and hence $\Open\rR$ is tight, since $\Open\Delta_Z = \JPow Z$. On other other hand, if $f : \aQ \to \aR$ is tight via $f = h \circ g$ then $\Pirr f = \Pirr g \fatsemi \Pirr h$ is tight, since $\Pirr\JPow Z$ is bipartite graph isomorphic to $\Delta_Z$.

\item
This follows from the previous statement, since tight $\JSL_f$-morphisms are closed under joins (= pointwise joins), and $\Open$ induces a join-semilattice isomorphism $\BiCliq[\rG,\rH] \cong \JSL_f[\Open\rG,\Open\rH]$. However, we choose to provide an explicit proof.

$\emptyset : \rG \to \rH$ is tight because it equals the $\BiCliq$-composite $\rG \xto{\emptyset} \rH \xto{!} \emptyset = \Delta_\emptyset \xto{!} \rH$, using Lemma \ref{lem:bicliq_hom_functor}. It remains to show that tight morphisms are closed under binary unions. Take any tight $\BiCliq$-morphisms $\rR_i : \rG \to \rH$ so that $\rR_i = (\rS_i : \rG \to \Delta_{Z_i}) \fatsemi (\rT_i : \Delta_{Z_i} \to \rH)$ for $i = 1,2$, where may assume that $Z_1$, $Z_2$ are disjoint. Let $Z := Z_1 \cup Z_2$ and define the relations:
\[
\rS := \rS_1 \cup \rS_2 \subseteq \rG_s \times Z
\qquad
\rT := \rT_1 \cup \rT_2 \subseteq Z \times \rH_t
\]
Then $\rS$ is a $\BiCliq$-morphism of type $\rG \to \Delta_Z$ because every subset $\rS[X]$ is $\Delta_Z$-open, and:
\[
\rS[\cl_\rG(X)] 
= \rS_1[\cl_\rG(X)] \cup \rS_2[\cl_\rG(X)] 
= \rS_1[X] \cup \rS_2[X]
= \rS[X]
\]
because $\rS_1$, $\rS_2$ are $\BiCliq$-morphisms. It follows that $\rT : \Delta_Z \to \rH$ is a $\BiCliq$-morphism via duality. Observe that $\rS \fatsemi \rT = \rS ; \rT$ because $(\rS \fatsemi \rT)^\up = \rT^\up \circ \Delta_Z^\down \circ \rS^\up = \rT^\up \circ \rS^\up = (\rS ; \rT)^\up$. Then we also deduce that $\rR_i = \rS_i \fatsemi \rT_i = \rS_i ; \rT_i$. Therefore:
\[
\rS \fatsemi \rT
= \rS ; \rT
= (\rS_1 \cup \rS_2) ; (\rT_1 \cup \rT_2)
= (\rS_1 ; \rT_1) \cup (\rS_2 ; \rT_2)
= \rR_1 \cup \R_2
\]
where in the penultimate equality we have used the fact that (i) sequential composition preserves unions in each component, and (ii) $Z_1$, $Z_2$ are disjoint.

\item
This follows from the well-definedness of $\BiCliq[-,-]$ (see Lemma \ref{lem:bicliq_hom_functor}) and the following two observations.
\begin{enumerate}
\item
Each $\biTight{\rG,\rH} \subseteq \BiCliq[\rG,\rH]$ is well-defined sub join-semilattice by the previous statement.
\item
Tight $\BiCliq$-morphisms are closed under pre/post-composition by arbitrary $\BiCliq$-morphisms, since the factorisation through an identity relation is preserved.
\end{enumerate}

\item
That they are well-defined $\BiCliq$-morphisms is proved in Example \ref{ex:special_bicliq_mor} above. Since basic independents are unions of basic bicliques, it suffices to show the latter are tight. This follows because:
\[
\upsp_{\rG,\rH}^{g_t,h_s} \; 
= \; \upsp_{\rG,\Delta_{\{0\}}}^{g_t,0} \fatsemi\, \upsp_{\Delta_{\{0\}},\rH}^{0,h_s} 
\]

\item
\begin{enumerate}
\item
Regarding the topmost equalities,
\[
(\upsp_{\rG,\rH}^{g_t,h_s})\spcheck
= \BC{\breve{\rG}[g_t], \, \rH[h_s]}\spbreve
= \BC{\rH[h_s], \breve{\rG}[g_t]}
= \BC{\breve{\rH}\spbreve[h_s], \breve{\rG}[g_t]}
= \; \upsp_{\breve{\rH},\breve{\rG}}^{h_s,g_t}
\]
\[
\begin{tabular}{lll}
$(\downsp_{\rG,\rH}^{g_s,h_t})\spcheck$
&
$= (\inte_{\breve{\rG}}(\overline{g_s}) \times \rH[\rH_s] \,\cup\, \breve{\rG}[\rG_t] \times \inte_{\rH}(\overline{h_t}))\spbreve$
\\&
$= \inte_{\rH}(\overline{h_t}) \times \breve{\rG}[\rG_t] \, \cup \, \rH[\rH_s]\times\inte_{\breve{\rG}}(\overline{g_s})$
\\&
$= \inte_{\breve{\rH}\spbreve}(\overline{h_t}) \times \breve{\rG}[\breve{\rG}_s] \, \cup \, \breve{\rH}\spbreve[\breve{\rH}_t]\times\inte_{\breve{\rG}}(\overline{g_s})$
\\&
$= \; \downsp_{\breve{\rH},\breve{\rG}}^{h_t,g_s}$
\end{tabular}
\]

\item
As for the central equalities, consider the left one and let $f := \Open \upsp_{\rG,\rH}^{g_t,h_s} : \Open\rG \to \Open\rH$, recalling: 
\[
(\upsp_{\rG,\rH}^{g_t,h_s})_+\spbreve
= \emptyset_+\spbreve \cup \BC{\cl_{\breve{\rG}}(\{g_t\}),\rH[h_s]}
\]
Now recall that $\leq_{\Open\rG}$ is inclusion of sets, and a $\rG$-open set $Y$ satisfies $Y \nsubseteq \inte_\rG(\overline{g_t})$ iff $g_t \in Y$ by Lemma \ref{lem:cl_inte_of_pirr}.1. Furthermore $\emptyset_+\spbreve[Y] = \emptyset$ because $Y$ is $\rG$-open i.e.\ each element $g_t \in Y$ must be contained in some `neighourhood' $\rG[g_s]$. Consequently:
\[
f(Y) 
= (\upsp_{\rG,\rH}^{g_t,h_s})_+\spbreve[Y] =
\begin{cases}
\emptyset_+\spbreve[Y] = \emptyset = \bot_{\Open\rG} & \text{if $Y \leq_{\Open\rG} \inte_\rG(\overline{g_t})$}
\\
\rH[h_s] & \text{otherwise}
\end{cases}
\]
as required. Regarding the right equality, let $f := \Open \downsp_{\rG,\rH}^{g_s,h_t} : \Open\rG \to \Open\rH$ and recall:
\[
(\downsp_{\rG,\rH}^{g_s,h_t})_+\spbreve
= \emptyset_+\spbreve \,\cup\, \rG_t \times \inte_\rH(\overline{h_t}) \, \cup \, \breve{\rG}^\down(\overline{g_s}) \times \rH[\rH_s]
\]
Certainly $\emptyset_+\spbreve[Y] = \emptyset$ as before, and furthermore:
\[
\begin{tabular}{lll}
$Y \cap \breve{\rG}^\down(\overline{g_s}) \neq \emptyset$
&
$\iff Y \cap \overline{\rG[g_s]} \neq \emptyset$
& by De Morgan duality
\\&
$\iff Y \nsubseteq \rG[g_s]$
\end{tabular}
\]
Further recalling that $\leq_{\Open\rH}$ is inclusion, $\emptyset = \bot_{\Open\rG}$ and  $\rH[\rH_s] = \top_{\Open\rH}$, we obtain the desired action:
\[
f(Y) 
= (\downsp_{\rG,\rH}^{g_s,h_t})_+\spbreve[Y] =
\begin{cases}
\bot_{\Open\rH} & \text{if $Y = \bot_{\Open\rG}$}
\\
\inte_\rH(\overline{h_t}) & \text{if $\emptyset <_{\Open\rG} Y \leq_{\Open\rG} \rG[g_s]$}
\\
\top_{\Open\rH} & \text{if $Y \nleq_{\Open\rG} \rG[g_s]$}
\end{cases}
\]

\item
Concerning the final equality, let $\rR := \Pirr \up_{\aQ,\aR}^{q,r} \; = \{ (j,m) \in J(\aQ) \times M(\aR) : \; \up_{\aQ,\aR}^{q,r}(j) \nleq_\aR m \}$. Recall that:
\[
\up_{\aQ,\aR}^{q,r}
\; = \Lor_{\JSL_f[\aQ,\aR]} \{ \up_{\aQ,\aR}^{m,j} \; : q \leq_\aQ m \in M(\aQ), \, J(\aR) \ni j \leq_\aR r \}
\qquad
\text{by Lemma \ref{lem:special_jsl_morphisms}.1}
\]
and hence $\Pirr \up_{\aQ,\aR}^{q,r}$ is the union of the $\Pirr \up_{\aQ,\aR}^{m,j}$'s, using the generalisation of Lemma \ref{lem:jsl_bicliq_hom_spec_iso}. So fix any $m \in M(\aQ)$ and $j \in J(\aR)$, finally observing that:
\[
\begin{tabular}{lll}
$\Pirr \up_{\aQ,\aR}^{m,j}$
&
$= \{ (j',m') \in M(\aQ) \times J(\aR) : \; \up_{\aQ,\aR}^{m,j}(j') \nleq_\aR m' \}$
\\&
$= \{ (j',m') \in M(\aQ) \times J(\aR) : j' \nleq_\aQ m \text{ and } j \nleq_\aR m'  \}$
& by definition of $\up_{\aQ,\aR}^{m,j}$
\\&
$= (\Pirr \aQ)\spbreve \syncp \Pirr\aR [(m,j)]$
\\&
$= \BC{\Pirr\aQ\spbreve[m],\,\Pirr\aR[j]}$
\\&
$= \; \upsp_{\Pirr\aQ,\Pirr\aR}^{m,j}$
& see Definition \ref{def:tight_bicliq_mor}.3
\end{tabular}
\] 
\end{enumerate}


\item
By Lemma \ref{lem:tight_mor_char}, a $\JSL_f$-morphism $f : \aQ \to \aR$ is tight iff it arises as a join of the special morphisms $\up_{\aQ,\aR}^{m,j}$ where $m \in M(\aQ)$ and $j \in J(\aR)$. In particular, if $\aQ = \Open\rG$ and $\aR = \Open\rH$ then by Lemma \ref{lem:lat_op_cl}.3, $m = \inte_\rG(\overline{g_t})$ for some $g_t \in \rG_t$ and $j = \rH[h_s]$ for some $h_s \in \rH_s$. By (1) a $\BiCliq$-morphism $\rR : \rG \to \rH$ is tight iff $\Open\rR$ is tight, and also $\Open$ induces a join-semilattice isomorphism $\BiCliq[\rG,\rH] \cong \JSL_f[\Open\rG,\Open\rH]$. Then by the left-central equality in (5), $\rR$ is tight iff it arises as a union of basic bicliques.

\item
We calculate:
\[
\begin{tabular}{lll}
$\rR \fatsemi \upsp_{\rG,\rH}^{g_t,h_s}$
&
$= \rR ; (\upsp_{\rG,\rH}^{g_t,h_s})_+\spbreve$
\\&
$= \rR ; (\emptyset_+\spbreve \,\cup\, \cl_{\breve{\rG}}(\{g_t\}) \times \rH[h_s])$
& see Definition \ref{def:tight_bicliq_mor}.3
\\&
$= \rR ; \cl_{\breve{\rG}}(\{g_t\}) \times \rH[h_s])$
& since $\breve{\rG}^\down(\emptyset)$ isolated
\\&
$= \breve{\rR}[\cl_{\breve{\rG}}(\{g_t\})] \times \rH[h_s]$
\\&
$= \breve{\rR}[g_t] \times \rH[h_s]$
& since $\breve{\rR}^\up \circ \cl_{\breve{\rG}} = \breve{\rR}^\up$
\\&
$= \breve{\rF}[\rR_+[g_t]] \times \rH[h_s]$
& since $\rR = \rF ; \rR_+\spbreve$
\\&
$= \bigcup \{ \breve{\rF}[f_t] \times \rH[h_s] : f_t \in \rR_+[g_t] \}$
\\&
$= \bigcup \{  \upsp_{\rF,\rH}^{f_t,h_s} : f_t \in \rR_+[g_t] \}$
\end{tabular}
\]
The other equality follows by duality and the top-left equality in (5), recalling that $(\rS\spcheck)_+ = \rS_-$.

\item
We have:
\[
\begin{tabular}{lll}
& $\rR \subseteq \, \downsp_{\rG,\rH}^{g_s,g_t}$
\\ $\iff$ &
$\Open\rR \leq \Open \downsp_{\rG,\rH}^{g_s,g_t}$
& $\Open$ preserves ordering
\\[0.5ex] $\iff$ &
$\Open\rR \leq \,\down_{\Open\rG,\Open\rH}^{\rG[g_s],\inte_\rH(\overline{h_t})}$
& by (5)
\\[0.5ex] $\iff$ &
$\Open\rR(\rG[g_s]) \subseteq \inte_\rH(\overline{h_t})$
& by Lemma \ref{lem:spec_mor_q_q}.2
\\[0.5ex] $\iff$ &
$\rR[g_s] \subseteq \inte_\rH(\overline{h_t})$
& since $\rR = \rG ; \rR_+\spbreve$
\\[0.5ex] $\iff$ &
$h_t \nin \rR[g_s]$
& by Lemma \ref{lem:cl_inte_of_pirr}.1
\\[0.5ex] $\iff$ &
$\overline{\rR}(g_s,h_t)$.
\end{tabular}
\]

\end{enumerate}
\end{proof}

We may now use Theorem \ref{thm:tight_self_dual} to obtain a natural isomorphism $v : \OD_j \circ \biTight{-,-} \To \biTight{(-)\spcheck,-}$ between join-semilattices of tight $\BiCliq$-morphisms.

\begin{theorem}[Dual isomorphism between join-semilattices of tight $\BiCliq$-morphisms]
\item
We have the natural isomorphism:
\[
\begin{tabular}{llll}
$v_{\rG,\rH} : (\jslTight{\rG,\breve{\rH}})^{\pOp} \to \jslTight{\breve{\rG},\rH}$
&&
$v_{\rG,\rH} (\rR)$
& $:= \overline{\rR_+\spbreve} \, ; \, \rH$
\\[1ex] &&
$v_{\rG,\rH}^{\bf-1}(\rR)$
& $:= \overline{\rR_+\spbreve} \, ; \, \breve{\rH}$
\end{tabular}
\]
and in particular $v_{\rG,\rH}^{\bf-1} = v_{\breve{\rG},\breve{\rH}}^{\pOp}$. Furthermore:
\[
\begin{tabular}{lll}
$v_{\rG,\rH}(\upsp_{\rG,\breve{\rH}}^{g_t,h_t})$
&
$=  \;\downsp_{\breve{\rG},\rH}^{g_t,h_t}$
& for every $(g_t,h_t) \in \rG_t \times \rH_t$,
\\[1ex]
$v_{\rG,\rH}(\downsp_{\rG,\breve{\rH}}^{g_s,h_s})$
&
$=  \;\upsp_{\breve{\rG},\rH}^{g_s,h_s}$
& for every $(g_s,h_s) \in \rG_s \times \rH_s$.
\end{tabular}
\]
\end{theorem}

\begin{proof}
We construct $v_{\rG,\rH}$ as a composition of natural isomorphisms:
\[
\xymatrix@=15pt{
(\jslTight{\rG,\breve{\rH}})^{\pOp} \ar@{-->}[rrr]^-{v_{\rG,\rH}}_-{\cong} \ar[d]_{\alpha_{\rG,\rH}}^{\cong} &&& \jslTight{\breve{\rG},\rH}
\\
(\jslTight{\Open\rG,\Open\breve{\rH}})^{\pOp} \ar[d]_{(\jslTight{id_{\Open\rG},\partial_\rH^{\bf-1}})^{\pOp}}^\cong
&&& \jslTight{\Open\breve{\rG},\Open\rH} \ar[u]_{(\alpha_{\breve{\rG},\breve{\rH}}^{\pOp})^{\bf-1}}^{\cong}
\\
(\jslTight{\Open\rG,(\Open\rH)^{\pOp} })^{\pOp} 
\ar[rrr]_-{\nu_{\Open\rG,\Open\rH}}^-{\cong} &&& \jslTight{(\Open\rG)^{\pOp},\Open\rH} \ar[u]_{\jslTight{\partial_\rG^{\bf-1},id_{\Open\rH}}}^{\cong}
}
\]
Let us start with an arbitrary tight $\BiCliq$-morphism $\rR : \rG \to \breve{\rH}$.
\begin{enumerate}
\item
$\alpha_{\rG,\rH}$ applies $\Open$. It is a natural isomorphism because:
 \[
 \BiCliq[\rG,\breve{\rH}] \cong \JSL_f[\Open\rG,\Open\breve{\rH}]
 \qquad
\text{restricts to the respective subfunctors of tight morphisms,}
\]
see Lemma \ref{lem:bicliq_tight_basic}.1. In particular $\alpha_{\rG,\rH}(\rR) = \Open\rR$.

\item
$(\jslTight{id_{\Open\rG},\partial_\rH^{\bf-1}})^{\pOp}$ post-composes with $\partial_\rH^{\bf-1}$, and thus:
\[
\beta(\Open\rR) = \qquad
\Open\rG \xto{\Open\rR} \Open\breve{\rH} \xto{\partial_\rH^{\bf-1}} (\Open\rH)^{\pOp}
\]
recalling that $\partial_\rH^{\bf-1}(X) = \rH[\overline{X}]$ -- see Definition \ref{def:open_dual_iso}.

\item
$\nu_{\Open\rG,\Open\rH}$ instantiates the natural isomorphism from Theorem \ref{thm:tight_self_dual}, so that:
\[
f := \qquad
\nu_{\Open\rG,\Open\rH}(\partial_\rH^{\bf-1} \circ \Open\rR)
= \lambda Y \in O(\rG). \bigcup \{ \rH[h_s] : h_s \in \rH_s, \; (\partial_\rH^{\bf-1} \circ \Open\rR)_*(\rH[h_s]) \nsubseteq Y \}
\]
Here we have used a slightly modified description of $\nu_{\aQ,\aR}(f : \aQ \to \aR^{\pOp})$ from the statement of the Theorem i.e.\ one can replace $j_r \in J(\aR)$ by any join-generating set. Then let us simplify the above comprehension:
\[
\begin{tabular}{lll}
$(\partial_\rH^{\bf-1} \circ \Open\rR)_*(\rH[h_s]) \nsubseteq Y$
&
$\iff (\Open\rR)_* \circ (\partial_\rH^{\bf-1})_*(\rH[h_s]) \nsubseteq Y$
\\&
$\iff (\Open\rR)_* \circ \partial_{\breve{\rH}}^{\bf-1}(\rH[h_s]) \nsubseteq Y$
& by Lemma \ref{lem:open_dual_iso_adjoints}.1
\\&
$\iff (\partial_\rG^{\bf-1} \circ \Open\breve{\rR} \circ \partial_{\breve{\rH}} ) \circ \partial_{\breve{\rH}}^{\bf-1}(\rH[h_s]) \nsubseteq Y$
& by Lemma \ref{lem:partial_def_nat_iso}
\\&
$\iff \partial_\rG^{\bf-1} \circ \Open\breve{\rR}(\rH[h_s]) \nsubseteq Y$
\\&
$\iff \rG^\up \circ \neg_{\rG_s} \circ (\breve{\rR})_+\spbreve[\rH[h_s]] \nsubseteq Y$
& by definition
\\&
$\iff \rG^\up \circ \neg_{\rG_s} \circ \rR_-\spbreve[\rH[h_s]] \nsubseteq Y$
& since $(\breve{\rR})_+ = \rR_-$ generally
\\&
$\iff \rG^\up \circ \neg_{\rG_s} \circ (\rR_- ; \breve{\rH})\spbreve[h_s] \nsubseteq Y$
\\&
$\iff \rG^\up \circ \neg_{\rG_s} \circ \breve{\rR}[h_s] \nsubseteq Y$
& recall $\rR : \rG \to \breve{\rH}$
\\&
$\iff \rG^\up \circ \rR^\down(\overline{h_s}) \nsubseteq Y$
& by De Morgan duality
\\&
$\iff \rR^\down(\overline{h_s}) \nsubseteq \rG^\down(Y)$
& standard adjoint
\end{tabular}
\]
Consequently, we have:
\[
f = \lambda Y \in O(\rG). \bigcup \{ \rH[h_s] : h_s \in \rH_s, \; \rR^\down(\overline{h_s}) \nsubseteq \rG^\down(Y) \}
\]

\item
$\jslTight{\partial_\rG^{\bf-1},id_{\Open\rH}}$ pre-composes with $\partial_\rG^{\bf-1}$, yielding the tight morphism:
\[
\begin{tabular}{lll}
$g$
& $:= \lambda Y \in O(\breve{\rG}). \bigcup \{ \rH[h_s] : h_s \in \rH_s, \, \rR^\down(\overline{h_s}) \nsubseteq \cl_\rG(\overline{Y}) \}$
\\[1ex]
& $= \lambda Y \in O(\breve{\rG}). \bigcup \{ \rH[h_s] : h_s \in \rH_s, \, \inte_{\breve{\rG}}(Y) \nsubseteq \breve{\rR}[h_s]  \}$
&  $: \Open\breve{\rG} \to \Open\rH$
\end{tabular}
\]
using De Morgan duality at the level of closure and interior operators.

\item
Finally we need to apply $(\alpha_{\breve{\rG},\breve{\rH}}^{\pOp})^{\bf-1}$, this being the inverse of the natural isomorphism $\jslTight{\breve{\rG},\rH} \to \jslTight{\Open\breve{\rG},\Open\rH}$ whose action applies $\Open$. Then we seek the necessarily unique tight $\BiCliq$-morphism $\rS := \breve{\rG} \to \rH$ such that $\Open\rS = g$, this being a relation $\rS \subseteq \rG_t \times \rH_t$. Observing that $g(\breve{\rG}[g_s]) = \rS_+\spbreve[\breve{\rG}[g_s]] = \breve{\rG};\rS_+\spbreve[g_s] = \rS[g_s]$, it follows that for every $g_t \in \rG_t$ we have:
\[
\begin{tabular}{lll}
$\rS[g_t]$
&
$= g(\breve{\rG}[g_t])$
\\&
$= \bigcup \{ \rH[h_s] : h_s \in \rH_s, \, \inte_{\breve{\rG}}(\breve{\rG}[g_t]) \nsubseteq \breve{\rR}[h_s] \}$
\\&
$= \bigcup \{ \rH[h_s] : h_s \in \rH_s, \, \breve{\rG}[g_t] \nsubseteq \breve{\rR}[h_s] \}$
& by $(\up\down\up)$
\\&
$= \bigcup \{ \rH[h_s] : h_s \in \rH_s, \, g_t \nin \breve{\rG}^\down \circ \breve{\rR}^\up (\{h_s\}) \}$
& via usual adjoint
\end{tabular}
\]
Recalling the definition of $\rR$'s associated positive component (see Definition \ref{def:bicliq_mor_components}), we deduce that:
\[
\begin{tabular}{lll}
$\rS(g_t,h_t)$
&
$\iff \exists h_s \in \rH_s.[ \rH(h_s,h_t) \text{ and } \overline{\rR_+}(h_s,g_t) ]$
\\[1ex] &
$\iff \exists h_s \in \rH_s.[ \overline{\rR_+\spbreve}(g_t,h_s) \text{ and } \rH(h_s,h_t)]$
\\[1ex] &
$\iff \overline{\rR_+\spbreve} ; \rH(g_t,h_t)$
\end{tabular}
\]
since converse commutes with complement.
\end{enumerate}

To prove that $v_{\rG,\rH}^{\bf-1} = v_{\breve{\rG},\breve{\rH}}^{\pOp}$, observe that the latter is a well-defined join-semilattice isomorphism of the correct type. It will be helpful to verify that $v_{\rG,\rH}$ sends basic bicliques to basic independents i.e.\
\[
v_{\rG,\rH}(\upsp_{\rG,\breve{\rH}}^{g_t,h_t})
 =  \;\downsp_{\breve{\rG},\rH}^{g_t,h_t}
\qquad
\text{for every $(g_t,h_t) \in \rG_t \times \rH_t$}.
\]
Then let us consider the composite action:
\[
\begin{tabular}{lcll}
$\upsp_{\rG,\breve{\rH}}^{g_t,h_t}$
&
$\stackrel{\alpha_{\rG,\rH}}{\mapsto}$
&
$\up_{\Open\rG,\Open\breve{\rH}}^{\inte_\rG(\overline{g_t}),\breve{\rH}[h_t]}$
& by Lemma \ref{lem:bicliq_tight_basic}.5
\\[1ex]&
$\stackrel{\jslTight{id_{\Open\rG},\partial_\rH^{\bf-1}})^{\pOp}}{\mapsto}$
&
$\up_{\Open\rG,(\Open\rH)^{\pOp}}^{\inte_\rG(\overline{g_t}),\partial_\rH^{\bf-1}(\breve{\rH}[h_t]}$
& by Lemma \ref{lem:compose_spec_gen_mor}.1
\\[1ex]& $=$ &
$\up_{\Open\rG,(\Open\rH)^{\pOp}}^{\inte_\rG(\overline{g_t}) , \inte_\rH(\overline{h_t})}$
& 
\\[1ex]&
$\stackrel{\nu_{\Open\rG,\Open\rH}}{\mapsto}$
&
$\down_{(\Open\rG)^{\pOp},\Open\rH}^{\inte_\rG(\overline{g_t}) , \inte_\rH(\overline{h_t})}$
& see Theorem \ref{thm:tight_self_dual}
\\[1ex]&
$\stackrel{\jslTight{\partial_\rG^{\bf-1},id_{\Open\rH}}}{\mapsto}$
&
$\down_{\Open\breve{\rG},\Open\rH}^{ (\partial_\rG^{\bf-1})_*(\inte_\rG(\overline{g_t})) , \inte_\rH(\overline{h_t})}$
& by Lemma \ref{lem:compose_spec_gen_mor}.1
\\[1ex]& $=$ &
$\down_{\Open\breve{\rG},\Open\rH}^{ \partial_{\breve{\rG}}^{\bf-1} (\inte_\rG(\overline{g_t})) , \inte_\rH(\overline{h_t})}$
& see Lemma \ref{lem:open_dual_iso_adjoints}.1
\\[1ex]& $=$ &
$\down_{\Open\breve{\rG},\Open\rH}^{ \breve{\rG}[g_t] , \inte_\rH(\overline{h_t})}$
& using $(\up\down\up)$
\\&
$\stackrel{\alpha_{\breve{\rG},\breve{\rH}}^{\bf-1}}{\mapsto}$
&
$\downsp_{\breve{\rG},\rH}^{g_t,h_t}$
& by Lemma \ref{lem:bicliq_tight_basic}.5
\end{tabular}
\]
recalling that $\partial_\rG^{\bf-1}(X) := \rG[\overline{X}]$. By essentially the same proof we also have $v_{\rG,\rH}(\downsp_{\rG,\breve{\rH}}^{g_s,h_s}) = \;\upsp_{\breve{\rG},\rH}^{g_s,h_s}$ for every $(g_s,h_s) \in \rG_s \times \rH_s$, noting that isomorphisms satisfy the extra conditions listed in Lemma \ref{lem:compose_spec_gen_mor}.2. It follows that $v_{\breve{\rG},\breve{\rH}}^{\pOp}$ has the same action as $v_{\rG,\rH}^{\bf-1}$, so they are the same morphisms.
\end{proof}

Recall that $(-)\spcheck : \BiCliq^{op} \to \BiCliq$ is the self-duality functor, whose action on both objects and morphisms takes the relational converse. We now prove the universal property of the synchronous product i.e.\
\begin{quote}
\emph{tight morphisms $\rG \syncp \rH \to \rI$ naturally biject with tight morphisms $\rG \to \rH\spcheck \syncp \rI$.}
\end{quote}

This bijection cannot hold without the tightness assumption. To see why, let $\rG = \Delta_{\{0\}}$ so that $\rG \syncp \rH \cong \rH$. Applying $\Open$, we see that morphisms of type $\rG \syncp \rH \to \rI$ biject (also as a join-semilattice) with $\JSL_f[\Open\rH,\Open\rI]$. Similarly since $\Open\Delta_{\{0\}} \cong \two$, the morphisms of type $\rG \to \breve{\rH} \syncp \rI$ biject with $\Open(\breve{\rH} \syncp \rI) \cong \jslTight{\Open\rH,\Open\rI}$ using Theorem \ref{thm:sync_is_tight_tensor}. Since join-semilattice morphisms needn't be tight by Lemma \ref{lem:jsl_iso_not_tight}, it follows that we cannot drop the tightness assumption. Moving on, the relevant natural isomorphism has a very natural action:
\begin{quote}
 given a relation $\rR \subseteq (\rG \syncp \rH)_s \times \rI_t = (\rG_s \times \rH_s) \times \rI_t$, re-tuple each element $((g_s,h_s),i_t) \; \mapsto \; (g_s,(h_s,i_t))$, yielding a relation $\rR' \subseteq \rG_s \times (\rH_s \times \rI_t) 
= \rG_s \times (\breve{\rH} \syncp \rI)_t$ of the desired type.
\end{quote}

The interpretation of this natural isomorphism inside $\JSL_f$ is:
\begin{quote}
\emph{
 tight morphisms $\aQ \ttenp \aR \to \aS$ naturally biject with tight morphisms $\aQ \to \jslTight{\aR,\aS}$,
 }
\end{quote}
 
 which will follow by combining Theorem \ref{thm:sync_is_tight_tensor} with the result we are just about to prove.

\bigskip

\begin{theorem}[The synchronous product is universal w.r.t.\ tight morphisms]
\label{thm:syncp_universal_tight}
\item
We have the natural isomorphism:
\[
\begin{tabular}{c}
$rtup : \biTight{ - \syncp -,-} \To \biTight{-,(-)\spcheck \syncp -}$
\qquad
$rtup_{\rG,\rH,\rI} : \biTight{\rG \syncp \rH,\rI} \to \biTight{\rG,\breve{\rH} \syncp \rI}$
\\[1ex]
$rtup_{\rG,\rH,\rI}(\rR) 
:= \{ (g_s,(h_s,i_t)) \in \rG_s \times (\rH_s \times \rI_t) :  \rR((g_s,h_s),i_t) \}$
\end{tabular}
\]
The associated components of its component's action are described in Note \ref{note:rtup_component_assoc} below, and its natural inverse re-tuples in the other direction.
\end{theorem}

\begin{proof}
For the first part of the proof, fix any $\rG$, $\rH$, $\rI$, and also any tight $\BiCliq$-morphism $\rR : \rG \syncp \rH \to \rI$, and define the relation:
\[
\rS := rtup_{\rG,\rH,\rI}(\rR) = \{ (g_s,(h_s,i_t)) \in \rG_s \times (\rH_s \times \rI_t) : \rR((g_s,h_s),i_t) \}
\]
We begin by verifying that $\rS$ is a well-defined tight $\BiCliq$-morphism of type $\rG \to \breve{\rH} \syncp \rI$. Since $\rR$ is tight it is a union of basic bicliques by Lemma \ref{lem:bicliq_tight_basic}.
\begin{enumerate}
\item
First consider the base case where $\rR$ is a basic biclique:
\[
\rR = \; \upsp_{\rG\syncp\rH,\rI}^{(g_s,h_s),i_t} \; = (\breve{\rG}[g_t] \times \breve{\rH}[h_t]) \times \rI[i_s]  \subseteq (\rG_s \times \rH_s) \times \rI_t
\]
Then $\rS = \breve{\rG}[g_t] \times (\breve{\rH}[h_t] \times \rI[i_s]) = \; \upsp_{\rG,\breve{\rH}\syncp\rI}^{g_s,(h_s,i_t)}$ is a basic biclique and thus is well-defined and tight.

\item
Generally speaking, if $\rR$ is a union of basic bicliques then observe that $rtup_{\rG,\rH,\rI}$ preserves unions i.e.\ re-tupling preserves $\emptyset$ and also binary unions of relations. Then by the previous point $\rS$ is a union of basic bicliques, and  is therefore well-defined and tight.
\end{enumerate}

Then each $rtup_{\rG,\rH,\rI}$ is a well-defined function, and also preserves the join-semilattice structure: $\emptyset$ and binary union. Since re-tupling can be undone, we know that $rtup_{\rG,\rH,\rI}$ is an injective join-semilattice morphism. Furthermore the preceding argument implies surjectivity, since every basic biclique $\rS : \rG \to \breve{\rH} \syncp \rI$ arises from a basic biclique $\rR$. 

\smallskip
It remains to prove naturality i.e.\ that the following diagram commutes inside $\JSL_f$:
\[
\xymatrix@=15pt{
\biTight{ \rG \syncp \rH,\rI} \ar[d]_{\biTight{(\rR \syncp \rS)^{op},\rT}} \ar[rrr]^{rtup_{\rG,\rH,\rI}} &&& \biTight{\rG,\breve{\rH} \syncp \rI} \ar[d]^{\biTight{\rR^{op},\breve{\rS} \syncp \rT}}
\\
\biTight{ \rG' \syncp \rH',\rI'} \ar[rrr]_{rtup_{\rG',\rH',\rI'}} &&& \biTight{\rG',(\rH')\spbreve \syncp \rI'} 
}
\]
for any $\BiCliq$-morphisms $\rR : \rG' \to \rG$, $\rS : \rH' \to \rH$ and $\rT : \rI \to \rI'$. In other words, for every tight $\BiCliq$-morphism $\rR' : \rG \syncp \rH \to \rI$, we must establish that:
\[
\rR \fatsemi rtup_{\rG,\rH,\rI}(\rR') \fatsemi (\breve{\rS} \syncp \rT)
= rtup_{\rG',\rH',\rI'}((\rR \syncp \rS) \fatsemi \rR' \fatsemi \rT)
\]
Since $\fatsemi$ preserves unions of $\BiCliq$-morphisms separately in each component by Lemma \ref{lem:bicliq_hom_functor}, and each component of $rtup$ preserves such unions, it suffices to consider the special case where $\rR' = \; \upsp_{\rG \syncp \rH,\rI}^{(g_t,h_t),i_s}$ is a basic biclique. We finally use Lemma \ref{lem:bicliq_tight_basic}.7 to prove this.
\[
\begin{tabular}{lll}
& $\rR \fatsemi rtup_{\rG,\rH,\rI}(\rR') \fatsemi (\breve{\rS} \syncp \rT)$
\\[1ex]$=$ &
$\rR \fatsemi rtup_{\rG,\rH,\rI}(\upsp_{\rG \syncp \rH,\rI}^{(g_t,h_t),i_s} ) \fatsemi (\breve{\rS} \syncp \rT)$
& definition of $\rR'$
\\[1ex] $=$ &
$\rR \fatsemi \upsp_{\rG, \breve{\rH}\syncp\rI}^{g_t,(h_t,i_s)}  \fatsemi (\breve{\rS} \syncp \rT)$
& $rtup$ preserves basic bicliques
\\[1ex] $=$ &
$(\bigcup \{ \upsp_{\rG', \breve{\rH}\syncp\rI}^{g'_t,(h_t,i_s)} : \rR_+(g_t,g'_t)  \}) \fatsemi (\breve{\rS} \syncp \rT)$
& by Lemma \ref{lem:bicliq_tight_basic}.7
\\[1ex] $=$ &
$\bigcup \{ \upsp_{\rG', \breve{\rH}\syncp\rI}^{g'_t,(h_t,i_s)} \fatsemi (\breve{\rS} \syncp \rT) : \rR_+(g_t,g'_t)  \}$
& $\fatsemi$ preserves unions of morphisms
\\[1ex] $=$ &
$\bigcup \{ \upsp_{\rG', (\rH')\spbreve \syncp \rI'}^{g'_t,(h'_t,i'_s)} : \rR_+(g_t,g'_t), \, \rS_+(h_t,h'_t), \, \rT_-(i_s,i'_s) \}$
& by Lemma \ref{lem:bicliq_tight_basic}.7
\\[1ex] $=$ &
$rtup_{\rG',\rH',\rI'}(\bigcup \{ \upsp_{\rG' \syncp \rH',\rI'}^{(g'_t,h'_t),i'_s} : \rR_+(g_t,g'_t),\, \rS_+(h_t,h'_t), \, \rT_-(i_s,i'_s) \})$
& $rtup$ preserves basic bicliques, unions
\\[1ex] $=$ &
$rtup_{\rG',\rH',\rI'}((\rR \syncp \rS) \fatsemi \upsp_{\rG \syncp \rH,\rI}^{(g_t,h_t),i_s} \fatsemi \rT)$
& repeating above reasoning
\\[1ex] $=$ &
$rtup_{\rG',\rH',\rI'}((\rR \syncp \rS) \fatsemi \rR' \fatsemi \rT)$
& definition of $\rR'$
\end{tabular}
\]
\end{proof}

\begin{note}[Associated components of $rtup_{\rG,\rH,\rI}$'s action]
\label{note:rtup_component_assoc}
\item
The components of $rtup$ are join-semilattice isomorphisms $rtup_{\rG,\rH,\rI}$, sending tight $\BiCliq$-morphisms $\rR : \rG \syncp \rH \to \rI$ to tight $\BiCliq$-morphisms $\rG \to \breve{\rH} \syncp \rI$ by re-tupling each element of $\rR$. We now describe the associated components of the morphism $rtup_{\rG,\rH,\rI}(\rR)$.
\[
\begin{tabular}{ll}
$(rtup_{\rG,\rH,\rI}(\rR))_-(g_s,(h_t,i_s))$
&
$\iff \forall h_s.\forall i_t.[ \rH(h_s,h_t) \,\land\, \rI(i_s,i_t) \To \rR((g_s,h_s),i_t) ]$
\\
$(rtup_{\rG,\rH,\rI}(\rR))_+((h_s,i_t),g_t)$
&
$\iff \forall g_s.[\rG(g_s,g_t) \To \rR((g_s,h_s),i_t)]$
\end{tabular}
\]
Here we have simply used the definition of the associated components i.e.\ Definition \ref{def:bicliq_mor_components}. \endbox
\end{note}

\bigskip


\begin{theorem}[The tight tensor product is universal w.r.t.\ tight morphisms]
\item
We have the natural isomorphism:
\[
\begin{tabular}{c}
$ut : \jslTight{- \ttenp -, -} \To \jslTight{-,\jslTight{-,-}}$
\qquad
$ut_{\aQ,\aR,\aS} : \jslTight{\aQ \ttenp \aR,\aS} \to \jslTight{\aQ,\jslTight{\aR,\aS}}$
\end{tabular}
\]
with action:
\[
\begin{tabular}{lll}
$ut_{\aQ,\aR,\aS}(\jslTight{\aQ^{\pOp},\aR} \xto{f} \aS)$
& $ := \lambda q \in Q. \lambda r \in R.f(\up_{\aQ^{\pOp},\aR}^{q,r})$
\\[1ex]
$ut_{\aQ,\aR,\aS}^{\bf-1}(\aQ \xto{h} \jslTight{\aR,\aS})$
& $ := \lambda g \in \jslTight{\aQ^{\pOp},\aR}. \Lor_\aS \{ h(q)(r) : \; \up_{\aQ^{\pOp},\aR}^{q,r} \; \leq g \}$.
\end{tabular}
\]
Finally, the action on join/meet-irreducibles is as follows:
\[
\up_{\aQ \ttenp \aR, \aS}^{\down_{\aQ^{\pOp},\aR}^{m_q,m_r}, j_s}
\quad \stackrel{ut_{\aQ,\aR,\aS}}{\mapsto} \quad
\up_{\aQ,\jslTight{\aR,\aS}}^{m_q, \up_{\aR,\aS}^{m_r,j_s} }
\qquad\qquad
\down_{\aQ \ttenp \aR, \aS}^{\up_{\aQ^{\pOp},\aR}^{j_q,j_r}, m_s}
\quad \stackrel{ut_{\aQ,\aR,\aS}}{\mapsto} \quad
\down_{\aQ,\jslTight{\aR,\aS}}^{j_q, \down_{\aR,\aS}^{j_r,m_s} }
\]
\end{theorem}

\begin{proof}
We shall make use of the following natural isomorphisms.
\begin{enumerate}
\item
By Corollary \ref{cor:tight_tensor_syncp} of Theorem \ref{thm:sync_is_tight_tensor} we have the natural isomorphism:
\[
\begin{tabular}{c}
$\alpha_{\aQ,\aR} : \aQ \ttenp \aR \to \Open(\Pirr\aQ \syncp \Pirr\aR)$
\\[1ex]
\begin{tabular}{ll}
$\alpha_{\aQ,\aR}(f : \aQ^{\pOp} \to \aR)$
& $ := \{ (m_q,m_r) \in M(\aQ) \times M(\aR) : f(m_q) \nleq_\aR m_r \}$
\\[1ex]
$\alpha_{\aQ,\aR}^{\bf-1}(Y)$
& $ := \lambda q \in Q.\Lor_\aR \{ \Land_\aR \{ m_r \in M(\aR) : (m_q,m_r) \nin Y  \} :  q \leq_\aQ m_q \in M(\aQ) \} $
\end{tabular}
\end{tabular}
\]
witnessing the fact that the tight tensor product is essentially the synchronous product of binary relations.

\item
We have the inverse of the re-tupling natural isomorphism from Theorem \ref{thm:syncp_universal_tight}:
\[
\begin{tabular}{c}
$rtup_{\rG,\rH,\rI}^{\bf-1} : \jslTight{\rG,\breve{\rH} \syncp \rI} \to \jslTight{\rG \syncp \rH, \rI}$
\\[1ex]
$rtup_{\rG,\rH,\rI}^{\bf-1}(\rR) := \{ ((g_s,h_s),i_t) : \rR(g_s,(h_s,i_t)) \}$
\end{tabular}
\]
which re-tuples in the `other direction'. This is the universal property of synchronous products w.r.t.\ tight $\BiCliq$-morphisms, this being the $\BiCliq$-version of the natural isomorphism we are trying to describe.

\item
We'll also use auxiliary natural isomorphisms:
\[
\begin{tabular}{lllll}
$p_{\aQ,\rG} : \jslTight{\aQ,\Open\rG} \to \biTight{\Pirr\aQ,\rG}$
&&
$f : \aQ \to \Open\rG$ & $\mapsto$ & $\Pirr f \fatsemi red_\rG^{\bf-1} : \Pirr\aQ \to \rG$
\\[2ex]
$q_{\rG,\aS} : \biTight{\rG,\Pirr\aS} \to \jslTight{\Open\rG,\aS}$
&&
$\rR : \rG \to \Pirr\aS$ & $\mapsto$ & $rep_\aS^{\bf-1} \circ \Open\rR : \Open\rG \to \aS$
\end{tabular}
\]
which are correct because tightness is preserved by these equivalence functors, and tight morphisms are closed under composition with arbitrary morphisms. Further recall that:
\[
red_\rG^{\bf-1} := \breve{\in} = \{ (X,g_t) \in J(\Open\rG) \times \rG_t : g_t \in X \}
\qquad
rep_\aS^{\bf-1} := \lambda Y \in \Open\Pirr\aS. \Land_\aS M(\aS) \backslash Y
\]
\end{enumerate}

Then $ut_{\aQ,\aR,\aS}^{\bf-1}$ is defined as the composite natural isomorphism:
\[
\xymatrix@=15pt{
\jslTight{\aQ,\jslTight{\aR,\aS}} \ar@{=}[rr] \ar[dddd]_{ut_{\aQ,\aR,\aS}^{\bf-1}} && \jslTight{\aQ,\aR^{\pOp} \ttenp \aS} 
\ar[d]^{\jslTight{id_\aQ,\alpha_{\aR^{\pOp},\aS}}}
\\
&& \jslTight{\aQ,\Open((\Pirr\aR)\spbreve \syncp \Pirr\aS)} 
\ar[d]^-{p_{\aQ,(\Pirr\aR)\spbreve \syncp \Pirr\aS}}
\\
&& \jslTight{\Pirr\aQ,(\Pirr\aR)\spbreve \syncp \Pirr\aS} 
\ar[d]^-{rtup_{\Pirr\aQ,\Pirr\aR,\Pirr\aS}^{\bf-1}}
\\ && \jslTight{\Pirr\aQ \syncp \Pirr\aR,\Pirr\aS} 
\ar[d]^-{q_{\Pirr\aQ \syncp \Pirr\aR,\aS}}
\\
\jslTight{\aQ \ttenp \aR, \aS}   &&
\jslTight{\Open(\Pirr\aQ \syncp \Pirr\aR),\aS} \ar[ll]^-{\jslTight{\alpha_{\aQ,\aR},id_\aS}}
}
\]
which has action:
\[
h : \aQ \to \jslTight{\aR,\aS}
\quad\mapsto\quad
rep_\aS^{\bf-1} \circ \Open(rtup_{\Pirr\aQ,\Pirr\aR,\Pirr\aS}^{\bf-1}(\Pirr (\alpha_{\aR^{\pOp},\aS} \circ h)  \fatsemi red_{(\Pirr\aR)\spbreve \syncp \Pirr\aS}^{\bf-1} )) \circ \alpha_{\aQ,\aR}
\]
Now, since $h$ is tight it arises as the pointwise-join of special morphisms:
\[
\up_{\aQ,\jslTight{\aR,\aS}}^{m_q,\up_{\aR,\aS}^{m_r,j_s}}
\qquad
\text{where $m_q \in M(\aQ)$ and $\up_{\aR,\aS}^{m_r,j_s} \in J(\jslTight{\aR,\aS})$}.
\]
Then let us consider the action of the mapping $ut_{\aQ,\aR,\aS}^{\bf-1}$ upon these specific morphisms. For brevity it will be helpful to first set some basic notation:
\[
\rG_\aQ := \Pirr\aQ
\qquad \rG_\aR := \Pirr\aR
\qquad \rG_\aS := \Pirr\aS
\]
and we are going to split the computation of $ut_{\aQ,\aR,\aS}^{\bf-1}(\up_{\aQ,\jslTight{\aR,\aS}}^{m_q,\up_{\aR,\aS}^{m_r,j_s}})$ into parts.
\begin{enumerate}
\item
Let us begin with the following simplification:
\[
\begin{tabular}{lll}
$\Pirr(\alpha_{\aR^{\pOp},\aS} \circ \up_{\aQ,\jslTight{\aR,\aS}}^{m_q,\up_{\aR,\aS}^{m_r,j_s}})$
\\
$= \Pirr (\up_{\aQ, \Open(\breve{\rG}_\aR \syncp \rG_\aS)  }^{m_q,\alpha_{\aR^{\pOp},\aS}(\up_{\aR,\aS}^{m_r,j_s})})$
\\[1ex] \qquad
by Lemma \ref{lem:compose_spec_gen_mor}
\\[1ex]
$= \bigcup \{ \upsp_{\rG_\aQ,\Pirr\Open(\breve{\rG}_\aR \syncp \rG_\aS) }^{m'_q , \breve{\rG}_\aR[m'_r] \times \rG_\aS[j'_s] } : \; m_q \leq_\aQ m'_q, \; \breve{\rG}_\aR[m'_r] \times \rG_\aS[j'_s] \subseteq \alpha_{\aR^{\pOp},\aS}(\up_{\aR,\aS}^{m_r,j_s}) \}$
\\[1ex] \qquad
by Lemma \ref{lem:bicliq_tight_basic}.5
\\[1ex]
$= \bigcup \{ \upsp_{\rG_\aQ,\Pirr\Open(\breve{\rG}_\aR \syncp \rG_\aS) }^{m'_q , \breve{\rG}_\aR[m'_r] \times \rG_\aS[j'_s] } : \; m_q \leq_\aQ m'_q, \; \breve{\rG}_\aR[m'_r] \times \rG_\aS[j'_s] \subseteq \breve{\rG}_\aR[m_r] \times \rG_\aS[j_s]  \}$
\\[1ex] \qquad
by Corollary \ref{cor:tight_tensor_syncp}.1
\\[1ex]
$= \; \upsp_{\rG_\aQ,\Pirr\Open(\breve{\rG}_\aR \syncp \rG_\aS) }^{m_q , \breve{\rG}_\aR[m_r] \times \rG_\aS[j_s] }$
\\[1ex]\qquad
using the definition of basic bicliques
\end{tabular}
\]

\item
Continuing, we have:
\[
\begin{tabular}{lll}
$\Open(rtup_{\rG_\aQ,\rG_\aR,\rG_\aS}^{\bf-1}(\Pirr (\alpha_{\aR^{\pOp},\aS} \circ h)  \fatsemi red_{\breve{\rG}_\aR \syncp \rG_\aS}^{\bf-1} ))$
&
$= \Open(rtup_{\rG_\aQ,\rG_\aR,\rG_\aS}^{\bf-1}( \; \upsp_{\rG_\aQ,\breve{\rG}_\aR \syncp \rG_\aS }^{m_q , (m_r,j_s) } ))$
& see below
\\&
$= \Open(\upsp_{\rG_\aQ \syncp \rG_\aR , \rG_\aS }^{(m_q,m_r),j_s} )$
& using definitions
\\[1ex]&
$= \; \up_{\Open(\rG_\aQ \syncp \rG_\aR ),\Open\rG_\aS }^{ \inte_{\rG_\aQ \syncp \rG_\aR}(\overline{(m_q,m_r)}), \rG_\aS[j_s] }$
& by Lemma \ref{lem:bicliq_tight_basic}.5
\end{tabular}
\]
Concerning the marked equality, $\breve{\rG}_\aR \syncp \rG_\aS = \Pirr\aR^{\pOp} \syncp \Pirr\aS$ is reduced so that $red_{\breve{\rG}_\aR \syncp \rG_\aS}^{\bf-1}$ is the closure of a bipartite graph isomorphism. It turns out that post-composing with it bijectively relabels $\breve{\rG}_\aR[m_r] \times \rG_\aS[j_s]$ with $(m_r,j_s)$.

\item
Then it remains to simplify:
\[
rep_\aS^{\bf-1} \circ
\up_{\Open(\rG_\aQ \syncp \rG_\aR ),\Open\rG_\aS }^{ \inte_{\rG_\aQ \syncp \rG_\aR}(\overline{(m_q,m_r)}), \rG_\aS[j_s] }
\circ  \alpha_{\aQ,\aR}
\]
which equals:
\[
\up_{ \aQ \ttenp \aR , \aS }^{ (\alpha_{\aQ,\aR})_*(\inte_{\rG_\aQ \syncp \rG_\aR}(\overline{(m_q,m_r)})) , rep_\aS^{\bf-1}(\rG_\aS[j_s]) }
\qquad
\text{by Lemma \ref{lem:compose_spec_gen_mor}.}
\]
Now, the first parameter equals $\down_{\aQ^{\pOp},\aR}^{m_q,m_r}$ by Corollary \ref{cor:tight_tensor_syncp}.1 because the adjoint of an isomorphism acts as the inverse. Moreover the second parameter simplifies as follows:
\[
rep_\aS^{\bf-1}(\rG_\aS[j_s])
= \Land_\aS M(\aS) \backslash \rG_\aS[j_s]
= \Land_\aS \{ m \in M(\aS) : j_s \leq_\aS m \}
= j_s.
\]
\end{enumerate}

In conclusion, we have shown that:
\[
\up_{ \aQ \ttenp \aR , \aS }^{ \down_{\aQ^{\pOp},\aR}^{m_q,m_r} , j_s } 
\quad \stackrel{ut_{\aQ,\aR,\aS}}{\mapsto} \quad
\up_{\aQ,\jslTight{\aR,\aS}}^{m_q,\up_{\aR,\aS}^{m_r,j_s}}
\]
and have thus established its action on join-irreducibles as desired. Now, to verify its general action:
\[
ut_{\aQ,\aR,\aS}(f) \stackrel{?}{=} \lambda q \in Q.\lambda r \in R.f(\up_{\aQ^{\pOp},\aR}^{q,r})
\]
it suffices to establish this when $f$ is join-irreducible because (i) $ut_{\aQ,\aR,\aS}$ preserves joins $\Lor_i f_i$, and (ii) we can absorb the $\Lor_\aR$-joins into $\JSL_f[\jslTight{\aQ^{\pOp},\aR},\aS]$-joins of the $f_i$'s. Thus we need to prove that:
\[
\up_{\aQ,\jslTight{\aR,\aS}}^{m_q,\up_{\aR,\aS}^{m_r,j_s}}(q)(r)
\; = \;
\up_{ \aQ \ttenp \aR , \aS }^{ \down_{\aQ^{\pOp},\aR}^{m_q,m_r} , j_s } (\up_{\aQ^{\pOp},\aR}^{q,r})
\qquad
\text{for every $(q,r) \in Q \times R$}.
\]
Indeed:
\[
\up_{ \aQ \ttenp \aR , \aS }^{ \down_{\aQ^{\pOp},\aR}^{m_q,m_r} , j_s } (\up_{\aQ^{\pOp},\aR}^{q,r})
\; = 
\begin{cases}
j_s & \text{if $\up_{\aQ^{\pOp},\aR}^{q,r} \; \nleq \; \down_{\aQ^{\pOp},\aR}^{m_q,m_r}$, or equivalently ($q \nleq_\aQ m_q$ and $r \nleq_\aR m_r$) }
\\
\bot_\aS & \text{otherwise}
\end{cases}
\]
using Lemma \ref{lem:special_jsl_morphisms}.6, whereas:
\[
\up_{\aQ,\jslTight{\aR,\aS}}^{m_q,\up_{\aR,\aS}^{m_r,j_s}}(q)(r)
\; =
(\begin{cases}
\up_{\aR,\aS}^{m_r,j_s} & \text{if $q \nleq_\aQ m_q$}
\\
\bot_{\jslTight{\aR,\aS}} & \text{otherwise}
\end{cases})(r)
=
\begin{cases}
j_s & \text{if $q \nleq_\aQ m_q$ and $r \nleq_\aQ m_r$}
\\
\bot_\aS & \text{otherwise}
\end{cases}
\]
as required. Having verified the action of $ut_{\aQ,\aR,\aS}$, the action of the inverse $ut_{\aQ,\aR,\aS}^{\bf-1}$ follows using the fact that $f : \jslTight{\aQ^{\pOp},\aR} \to \aS$ preserves joins. Then it only remains to verify the action of $ut_{\aQ,\aR,\aS}$ on meet-irreducibles:
\[
\begin{tabular}{lll}
$ut_{\aQ,\aR,\aS}(\down_{\aQ \ttenp \aR, \aS}^{\up_{\aQ^{\pOp},\aR}^{j_q,j_r}, m_s})$
&
$= \lambda q.\lambda r. \down_{\aQ \ttenp \aR, \aS}^{\up_{\aQ^{\pOp},\aR}^{j_q,j_r}, m_s}(\up_{\aQ^{\pOp},\aR}^{q,r})$
\\&
$= \lambda q . \lambda r.
\begin{cases}
\bot_\aS & \text{if $\up_{\aQ^{\pOp},\aR}^{q,r} = \bot_{\aQ \ttenp \aR}$}
\\
m_s & \text{if $\bot_{\aQ \ttenp \aR} < \; \up_{\aQ^{\pOp},\aR}^{q,r} \;\leq\; \up_{\aQ^{\pOp},\aR}^{j_q,j_r}$}
\\
\top_\aS & \text{if $\up_{\aQ^{\pOp},\aR}^{q,r} \;\nleq\; \up_{\aQ^{\pOp},\aR}^{j_q,j_r}$}
\end{cases}$
\\[1ex]&
$= \lambda q . \lambda r.
\begin{cases}
\bot_\aS & \text{if $q = \bot_\aQ$ or $r = \bot_\aR$}
\\
m_s & \text{if $\bot_\aQ <_\aQ q \leq_\aQ j_q$ and $\bot_\aR < r \leq_\aR j_r$}
\\
\top_\aS & \text{if $q \nleq_\aQ j_q$ or $r \nleq_\aR j_r$}
\end{cases}$
& by Lemma \ref{lem:special_jsl_morphisms}.2
\end{tabular}
\]
whereas:
\[
\down_{\aQ,\jslTight{\aR,\aS}}^{j_q,\down_{\aR,\aS}^{j_r,m_s}}(q)(r)
\; =
(\begin{cases}
\bot_{\jslTight{\aR,\aS}} & \text{if $q = \bot_\aQ$}
\\
\down_{\aR,\aS}^{j_r,m_s} & \text{if $\bot_\aQ < q \leq_\aQ j_q$}
\\
\top_{\jslTight{\aR,\aS}} & \text{if $q \nleq_\aQ j_q$}
\end{cases})(r)
=
\begin{cases}
\bot_\aS & \text{if $q = \bot_\aQ$ or $r = \bot_\aR$}
\\
m_s & \text{if $\bot_\aQ < q \leq_\aQ j_q$ and $\bot_\aR < r \leq_\aR j_r$}
\\
\top_\aS & \text{if $q \nleq_\aQ j_q$ and $r \nleq_\aR j_r$}
\end{cases}
\]
and we are finally finished.
\end{proof}


\section{Reduced undirected graphs and De Morgan algebras}
\label{sec:graphs_and_de_morgan}

So far, our main result amounts to a categorical equivalence:
\[
\begin{tabular}{c}
binary relations  \qquad $\approx$ \qquad finite join-semilattices
\\[1ex]
\begin{tabular}{c}
where isomorphism classes of \emph{reduced} relations
\\
correspond to isomorphism classes of finite lattices.
\end{tabular}
\end{tabular}
\]
We are going to extend this to a categorical equivalence:
\[
\begin{tabular}{c}
undirected graphs \qquad $\approx$ \qquad finite de morgan algebras
\\[1ex]
\begin{tabular}{c}
where isomorphism classes of \emph{reduced} undirected graphs
\\
correspond to isomorphism classes of finite de morgan algebras.
\end{tabular}
\end{tabular}
\]
Denoting the latter categories $\UG$ and $\SAI_f$ respectively, the variety $\SAI$ consists of join-semilattices equipped with a {\bf S}elf-{\bf A}djoint {\bf I}nvolutive morphism i.e.\ De Morgan algebras where distributivity is not assumed.

\smallskip
\noindent
Here's a brief summary of our approach:

\begin{itemize}
\renewcommand\labelitemi{--}
\item
$\UGJ$ is a category whose objects are pairs of relations $(\rG, \rE)$ between finite sets where $\rE$ is symmetric and satisfies $\rE = \rG ; \rH$ for some $\rH$. The category $\UGM$ has objects $(\rG, \rE)$ where instead $\rE = \rH ; \rG$ for some $\rH$.

\item
There are respective equivalent categories of algebras $\SAJ_f$ and $\SAM_f$ i.e.\ the two different ways of extending $\JSL_f$ with a single self-adjoint morphism.
\item
$\SAJ$ and $\SAM$ are varieties whose intersection is the variety of De Morgan algebras $\SAI$.
\item
The corresponding intersection of $\UGJ$ and $\UGM$ amounts to $\UG$ i.e.\ a category whose isomorphism classes are those undirected graphs $(V, \rE)$ where $\rE[v] = \rE[S]$ implies $v \in S$.
\end{itemize}

\subsection{Preliminary definitions}

\begin{definition}[Undirected graphs]
  \label{def:ugraphs_bipartite_ugraphs}
  \item
  \begin{enumerate}
  \item
  A relation $\rR \subseteq X \times X$ is \emph{symmetric} if $\rR = \breve{\rR}$ i.e.\ it equals its converse relation.
  \item
  An \emph{undirected graph} (or just \emph{graph}) is a pair $(V,\rE)$ where $V$ is a possibly-empty finite set and $\rE \subseteq V \times V$ is a symmetric relation. Then vertices may have self-loops i.e.\ $\rE(v,v)$ is permissable.
  
  \item
  An undirected-graph $(V,\rE)$ is \emph{irreflexive} if $\rE \cap \Delta_V = \emptyset$, and \emph{reflexive} if $\Delta_V \subseteq \rE$.

  \item
  A \emph{bipartite} undirected graph $(V,\rE)$ satisfies:
  \[
  \rE = \rE |_{U \times \overline{U}} \,\cup\, \rE |_{\overline{U} \times U}
  \qquad
  \text{for some subset $U \subseteq V$}.
  \]
  Then the set $\{ U, \overline{U} \}$ is called a \emph{bipartition for $\rE$} and the pair $(U,\overline{U})$ is called an \emph{ordered bipartition for $\rE$}. \endbox
  \end{enumerate}
\end{definition}

\smallskip

\begin{example}[Visualising undirected-graphs]
  Here are 3 examples, depicted classically and as a typed relation.
  \[
  \begin{tabular}{clclc}
  \txt{undirected graph} &&  \txt{symmetric relation $\rE \subseteq V \times V$} && \txt{binary relation}
  \\[-2ex]
  \\ \hline
  \\
  $\vcenter{\vbox{\xymatrix@=5pt{
  x \ar@{-}@(dl,ul) \ar@{-}[rr] && y && z \ar@{-}[ll] \ar@{-}@(dr,ur)
  }}}$
  &&
  \begin{tabular}{c}
  $\{ (x,x), (x,y), (y,x), (z,z) \}$
  \\
  where $V = \{x,y,z\}$
  \end{tabular}
  && 
  $\vcenter{\vbox{\xymatrix@=15pt{
  x & y & z
  \\
  x \ar[u] \ar[ur] & y \ar[ul] \ar[ur] & z \ar[u] \ar[ul]
  }}}$
  \\ \\
  $\vcenter{\vbox{\xymatrix@=10pt{
  & x
  \\
  y \ar@{-}[ur] \ar@{-}[rr] && z \ar@{-}[ul]
  }}}$
  &&
  \begin{tabular}{c}
  $\{ (u,v) \in V \times V : u \neq v \}$
  \\
  where $V = \{x,y,z\}$
  \end{tabular}
  &&
  $\vcenter{\vbox{\xymatrix@=15pt{
  z & y & x
  \\
  x \ar[u] \ar[ur] & y \ar[ul] \ar[ur] & z \ar[u] \ar[ul]
  }}}$
  \\\\
  $\vcenter{\vbox{\xymatrix@=5pt{
  & x_1 \ar@{-}[r] & x_2 \ar@{-}[dr] 
  \\
  x_0 \ar@{-}[ur] & & & x_3 \ar@{-}[dl]
  \\
  & x_5 \ar@{-}[ul] & x_4 \ar@{-}[l]
  }}}$
  &&
  \begin{tabular}{c}
  $\{ (x_i,x_j) \in V \times V : j = i \pm 1 \text{ (mod 6)} \}$
  \\
  where $V = \{ x_i : 0 \leq i < 6\}$
  \end{tabular}
  &&
  $\vcenter{\vbox{\xymatrix@=15pt{
  x_5 & x_1 & x_3
  \\
  x_0 \ar[u] \ar[ur] & x_2 \ar[ul] \ar[ur] & x_4 \ar[u] \ar[ul]
  }}}$
  $\vcenter{\vbox{\xymatrix@=15pt{
  x_0 & x_2 & x_4
  \\
  x_5 \ar[u] \ar[ur] & x_1 \ar[ul] \ar[ur] & x_3 \ar[u] \ar[ul]
  }}}$
  \end{tabular}
  \]
  The 2nd and 3rd examples are irreflexive and the latter is also bipartite as witnessed by the bipartition:
  \[
    \{\{x_0,x_2,x_4\},\{x_1,x_3,x_5\}\}.
  \]
  Finally, observe that every bipartite graph is irreflexive. \endbox
\end{example}

\begin{definition}[Basic graph-theoretic notions]
  \label{def:u_graphs_basic_notions}
  Let $(V,\rE)$ be a graph.
  \begin{enumerate}
  \item
  A graph $(S,\rE_0)$ is a \emph{subgraph of} $(V,\rE)$ if $S \subseteq V$ and $\rE_0 \subseteq \rE$, in which case we write $(S,\rE_0) \subseteq (V,\rE)$. Such a subgraph is \emph{induced} if $\rE_0 = \rE |_{S \times S}$.
  
  \item
  Given any finite sets $X$ and $Y$ we define specific sets constructed from them.
  \begin{enumerate}
  \item
  $\UBC{X,Y} := X \times Y \,\cup\, Y \times X$ is called an \emph{undirected biclique}.
  \item
  If $X \cap Y = \emptyset$ then $\UBC{X,Y}$ is called a \emph{irreflexive} undirected biclique.
  \item
  $\URC{X} := X \times X$ is called a \emph{reflexive undirected clique}.
  \item
  $\UIC{X} := (X \times X) \backslash \Delta_X$ is called an \emph{irreflexive undirected clique}.
  \end{enumerate}
  Irreflexive undirected bicliques and irreflexive undirected cliques are the standard notion of `biclique' and `clique' in an undirected graph without self-loops. Since we permit self-loops we have additional concepts.
  
  \item
  The equivalence classes $S \subseteq V$ of the reflexive transitive closure of $\rE$ are the \emph{connected components}. A graph is \emph{connected} if it has precisely one connected component, and thus cannot be empty.
   
  \item
  The \emph{neighbourhood function} $N_\rE : V \to \Pow V$ is defined $N_\rE(v) := \rE[v]$, and the \emph{degree function} $deg_\rE : V \to \Nat$ is defined $deg_\rE(v) := |N_\rE[v]|$. Then for each vertex $v \in V$, its \emph{neighbourhood} is $N_\rE(v)$ and its \emph{degree} is $deg_\rE(v)$.
  
  \item
  A graph's associated \emph{adjacency matrix} $Adj(V,\rE)$ is the function $f : V \times V \to 2$ where $2 := \{0,1\}$ and:
  \[
  f := \lambda (u,v) \in V \times V.\, \rE(u,v) \, ? \, 1 : 0.
  \]
  It is usually depicted as a $|V| \times |V|$ binary matrix whose rows and columns are indexed by $V$.
  
  \item
  There are also some important special types of graphs.
  
  \begin{enumerate}
  \item
  A \emph{path} is a connected irreflexive graph $(V,\rE)$ with $v_0 \neq v_1 \in V$ such that $deg_\rE = \lambda v \in V. v \in \{v_0,v_1\} \, ? \, 1 : 2$.
  
  \smallskip
  This is the usual notion of a path with at least two vertices and no repetitions. Its \emph{length} is the number of edges $|\rE| = |V|-1$. Finally, for each $n \geq 2$ we have the path $P_n := (\{0,1,\dots,n-1\},\rE)$ where $\rE(i,j) :\iff j = i \pm 1$, so that $P_n$ has length $n-1$.
  
  \item
  A \emph{cycle} is a connected irreflexive graph $(V,\rE)$ where every vertex has degree $2$.
  
  \smallskip
  This corresponds to the usual notion of a cycle with at least three distinct vertices. Its \emph{length}  is the number of edges $|\rE| = |V|$. A cycle is \emph{odd} if its length is, otherwise it is \emph{even}. For each $n \geq 3$ we have the cycle $C_n := (\{0,1,2,\dots,n-1\},\rE)$ where $\rE(i,j) :\iff j = i \pm 1 \; \text{(mod $n$)}$, so that $C_n$ has length $n$. \endbox
  \end{enumerate}
  
  \end{enumerate}
  \end{definition}

  Recall the standard notion of morphism between undirected graphs.
  
  \begin{definition}[Graph morphisms and isomorphisms]
  \item
  \begin{enumerate}
  \item
  Given graphs $(V_i,\rE_i)$ for $i = 1,2$, an \emph{undirected graph morphism} (or \emph{graph morphism}) $f : (V_1,\rE_1) \to (V_2,\rE_2)$ is a function $f : V_1 \to V_2$ such that $\rE_1 ; f \subseteq f ; \rE_2$, or equivalently: $\rE_1(v_1,v_2) \To \rE_2(f(v_1),f(v_2))$ for all $v_1,v_2 \in V$.
  
  \item
  A \emph{graph isomorphism} is a graph morphism $f : (V_1,\rE_1) \to (V_2,\rE_2)$ such that $f : V_1 \to V_2$ is bijective and satisfies $\rE_1 ; f = f ; \rE_2$. Equivalently, $f$ is bijective and $\rE_1(v_1,v_2) \iff \rE_2(f(v_1),f(v_2))$ for all $v_1,v_2 \in V$. \endbox
  
  \end{enumerate}
  \end{definition}

\smallskip
Now for our first non-standard concept.

\begin{definition}[Reduced undirected graph]
  $(V,\rE)$ is \emph{reduced} if its edge-relation is reduced i.e.\
  \[
  \forall v \in V,\ S \subseteq V. \;\; \rE[v] = \rE[S] \To v \in S.
  \]
  This coincides with the previous notion because $\rE$ is symmetric. \endbox
\end{definition}

\begin{example}[Reduced graphs]
  \item
  \begin{enumerate}[1.]
  \item
  The complete graph $K_V = (V,\overline{\Delta_V})$ is reduced iff $|V| \neq 1$. The case $|V| = 0$ is the empty graph and is reduced. If $|V| \geq 2$ then the neighbourhoods $\{ \overline{v} : v \in V \}$ are not  unions of others. The case $|V| = 1$ is an isolated point and is not reduced.
  
  \item
  The complete bipartite graph $K_{X,Y}$ is reduced iff $|X| = |Y| \leq 1$ i.e.\ if it is the empty graph or a single-edge. In all other cases we have two distinct vertices with the same neighbourhood.
  
  \item
  For each finite set $V$, the reflexive graph $(V, V \times V)$ is reduced iff $|V| \leq 1$. That is, the only reduced examples are the empty graph and a single self-loop.

  \item
  The $0$-regular graphs are disjoint unions of isolated vertices, and only the empty disjoint union is reduced i.e.\ the empty graph.
  \item
  The $1$-regular graphs are disjoint unions of self-loops and single-edges. They are all reduced, and  correspond to the finite boolean De Morgan algebras (see below).
  \item
  The $2$-regular graphs are disjoint unions of cycles $(C_n)_{n \geq 3}$ (see Definition \ref{def:u_graphs_basic_notions}) and also paths $(P_n)_{n \geq 2}$ with additional self-loops at each distinct end. There are precisely two connected $2$-regular graphs which are not reduced i.e.\
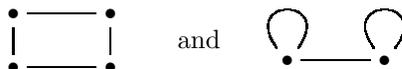
  \[
  \vcenter{\vbox{\xymatrix@=10pt{
  \bullet \ar@{-}[rr] && \bullet
  \\
  \bullet \ar@{-}[u] \ar@{-}[rr] && \bullet \ar@{-}[u]
  }}}
  \qquad\text{and}\qquad
  \vcenter{\vbox{\xymatrix@=10pt{
  \\
  \bullet \ar@{-}@(ul,ur) \ar@{-}[rr] && \bullet \ar@{-}@(ul,ur)
  }}}
  \]
  That is, $C_4$ is not reduced because diagonally opposite elements have the same neighbourhoods, and the two vertices of the other graph have the same neighbourhood. We have already seen the two smallest non-empty reduced $2$-regular graphs:
  \[
  \vcenter
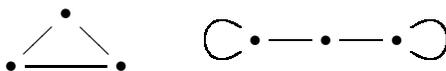
{\vbox{\xymatrix@=10pt{
  & \bullet
  \\
  \bullet \ar@{-}[ur] \ar@{-}[rr] && \bullet \ar@{-}[ul]
  }}}
  \qquad\qquad
  \vcenter{\vbox{\xymatrix@=5pt{
  \bullet \ar@{-}@(dl,ul) \ar@{-}[rr] && \bullet && \bullet \ar@{-}@(dr,ur) \ar@{-}[ll] 
  }}}
  \]
  Their corresponding De Morgan algebras have the same lattice structure ($\jslM{3}$) yet with different involutions, see Example \ref{ex:sai_equiv_ug_some_examples} below.
  
  \item
  Let us restrict to irreflexive graphs (forbidding self-loops) and fix $m \geq 2$. Then an $m$-regular irreflexive graph $G$ is reduced iff the complete bipartite graph $K_{2,m}$ is not an induced subgraph. Indeed, the neighbourhoods are $m$-element sets, so the graph can only fail to be reduced if two distinct vertices have the same neighbourhood.
  \begin{itemize}
  \item
  $K_{m,m}$ is an $m$-regular irreflexive graph which is not reduced, the case $m = 2$ yields $K_{2,2} \cong C_4$ as above.
  \item
  Since $K_{2,m}$ has a $4$-cycle, any $m$-regular graph with strictly greater \emph{girth} is reduced. For example, the girth of the $3$-regular Petersen graph is known to be $5$.
  \end{itemize}
   
  \item
  Recall that $P_n$ denotes the path with $n \geq 0$ edges. A path $P_n$ is reduced iff $n = 1$ or $n \geq 3$. That is, the only non-reduced paths are the isolated point and the path with two edges and three vertices. In the latter case, the two endpoints have the same neighbourhood i.e.\ the central point. \endbox

  \end{enumerate}
  \end{example}

Here is another non-standard notion.

\begin{definition}[Self-dual bipartite graphs]
  \item
  A connected bipartite graph $(V,\rE)$ is \emph{self-dual} if there exists a graph isomorphism $\theta  :(V,\rE) \to (V,\rE)$ such that $\theta[U] = \overline{U}$, where $\{U,\overline{U}\}$ is its unique bipartition. Then an arbitrary bipartite graph is \emph{self-dual} if all of its connected components are. \endbox
\end{definition}

\begin{example}[Self-dual bipartite graphs]
  \item
  \begin{enumerate}
  \item
  Every even cycle $C_{2n}$ is a connected bipartite graph and is also self-dual via the graph isomorphism $\theta(i) := i + 1 \, \text{(mod $2n$)}$.
  
  \item
  Although every path $P_n$ is a connected bipartite graph, it is self-dual iff $n$ is even. Indeed, the only non-identity graph isomorphism $\theta : P_n \to P_n$ has action $\theta(i) = (n - 1) - i$, so that $\theta$ switches the parity iff $n - 1$ is odd iff $n$ is even.
  
  \item
  Any graph $(V,\emptyset)$ with no edges is a bipartite graph. If $V = \emptyset$ then it is self-dual because it has no connected components. However if $V \neq \emptyset$ then it is not self-dual: each connected component $(\{*\},\emptyset)$ has unique bipartition $\{\{*\},\emptyset\}$ and we cannot have $f[\emptyset] = \{*\}$.
  \item
  Consider the following connected bipartite graph:
  
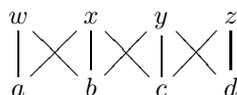
\[
  \xymatrix@=15pt{
  w & x & y & z
  \\
  a \ar@{-}[u] \ar@{-}[ur] & b \ar@{-}[ul] \ar@{-}[u] \ar@{-}[ur] & c \ar@{-}[u] \ar@{-}[ul] \ar@{-}[ur] & d \ar@{-}[ul] \ar@{-}[u]
  }
  \]
  It is self-dual and has two distinct witnessing graph isomorphisms. The first reflects along the centered horizontal axis, whereas the second additionally reflects along the centered vertical axis.
  
  \item
  Later on we'll see that the `self-dual bipartite graphs' corresponds to the symmetric finite lattices i.e.\ one which is isomorphic to its order-dual. For example, we derived the previous example from the lattice $\latL$:
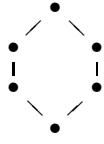
  \[
  \xymatrix@=5pt{
  & \bullet
  \\
  \bullet \ar@{-}[ur] && \bullet \ar@{-}[ul]
  \\
  \bullet \ar@{-}[u] && \bullet \ar@{-}[u]
  \\
  & \bullet \ar@{-}[ul] \ar@{-}[ur] & 
  }
  \]
  which has two distinct automorphisms. \endbox
  
  \end{enumerate}
\end{example}

\smallskip
Later we'll also need \emph{polarities} which are defined in terms of the operators $(-)^\up$ and $(-)^\down$.

\begin{definition}[Polarities]
  \label{def:polarities}
  Each relation $\rR \subseteq \rR_s \times \rR_t$ between finite sets yields functions:
  \[
  \rR^\Upa := \neg_{\rR_t} \circ \overline{\rR}^\up : \Pow \rR_s \to \Pow \rR_t
  \qquad
  \rR^\Dna := \overline{\rR}^\down \circ \neg_{\rR_t} : \Pow \rR_t \to \Pow \rR_s
  \]
  recalling that $\overline{\rR} \subseteq \rR_s \times \rR_t$ is the complement relation, and $\neg_X : \Pow X \to \Pow X$ constructs the relative complement.  We refer to these functions as the \emph{polarities of $\rR$}. \endbox
  \end{definition}
  
  The polarity $(-)^\Upa$ is a `De Morgan dual' of $(-)^\up$ involving the \emph{complement} relation rather than the converse. The up/down arrows are not flipped because the complement of a relation does not alter its type. We now show that polarities correspond to the classical concept, and prove some basic related equalities.

  \begin{lemma}[Basic properties of polarities]
  \label{lem:polarities_basic}
  \item
  \label{lem:polar_basic}
  Let $\rR \subseteq \rR_s \times \rR_t$ be any relation between finite sets.
  \begin{enumerate}
  \item
  The mappings $(-)^\Upa$ and $(-)^\Dna$ correspond to the `standard polarities':
  \[
  \begin{tabular}{lll}
  $\rR^\Upa(X)$
  & $= \bigcap_{x \in X} \rR[x]$
  & $= \{ y \in \rR_t : \forall x \in X.\rR(x,y)\}$
  \\[0.5ex]
  $\rR^\Dna(Y)$
  & $= \bigcap_{y \in Y} \breve{\rR}[y]$
  & $= \{ x \in \rR_s : \forall y \in Y.\rR(x,y) \}$
  \end{tabular}
  \]
  \item
  The polarities define adjoint join-semilattice morphisms:
  \[
  \begin{tabular}{c}
  $\rR^\Upa = (\rR^\Dna)_* : \JPow \rR_s \to (\JPow \rR_t)^{\pOp}
  \qquad
  \rR^\Dna = (\rR^\Upa)_* : \JPow \rR_t \to (\JPow \rR_s)^{\pOp}$
  \\[0.5ex]
  $Y \subseteq \rR^\Upa(X) \iff \rR^\Upa (X) \leq_{(\JPow \rR_t)^{\pOp}} Y
  \iff X \leq_{\JPow \rR_s} \rR^\Dna (Y) \iff X \subseteq \rR^\Dna (Y)$
  \end{tabular}
  \]
  for all $X \subseteq \rR_s$, $Y \subseteq \rR_t$. Then both $\rR^\Upa$ and $\rR^\Dna$ send arbitrary unions to intersections.
  
  \item
  We have the equalities:
  \[
  \begin{tabular}{lll}
  $\rR^\Dna \circ \rR^\Upa$
  & $= \cl_{\overline{\rR}}$
  & $= \lambda X \subseteq \rR_s. \{ r_s \in \rR_s : \bigcap_{x \in X} \rR[x] \subseteq \rR[r_s] \}$
  \\[0.5ex]
  $\rR^\Upa \circ \rR^\Dna$
  & $= \cl_{\overline{\rR}\spbreve}$
  & $= \lambda Y \subseteq \rR_t. \{ r_t \in \rR_t : \bigcap_{y \in Y} \breve{\rR}[y] \subseteq \breve{\rR}[r_t] \}$.
  \end{tabular}
  \]
  \end{enumerate}
  \end{lemma}
  
  \begin{proof}
  \item
  \begin{enumerate}
  \item
  We calculate:
  \[
  \rR^\Upa(X)
  = \neg_{\rR_t} \circ \overline{\rR}^\up(X)
  = \overline{\bigcup_{x \in x} \overline{\rR}[x]}
  =  \bigcap_{x \in X} \rR[x]
  \]
  since $\overline{\rR}[x] = \overline{\rR[x]}$. Furthermore:
  \[
  \begin{tabular}{lll}
  $\rR^\Dna(Y)$
  & $= \overline{\rR}^\down \circ \neg_{\rR_t}(Y)$
  \\&
  $= \neg_{\rR_s} \circ (\overline{\rR}\,\spbreve)^\up(Y)$
  & by DeMorgan duality
  \\&
  $= \breve{\rR}^\Upa(Y)$
  & by definition
  \\&
  $= \bigcap_{y \in Y} \breve{\rR}[y]$
  & by previous equality
  \end{tabular}
  \]
  
  \item
  The polarities actually define composite join-semilattice morphisms:
  \[
  \begin{tabular}{lll}
  $\rR^\Upa =\; \JPow \rG_s \xto{\overline{\rR}^\up} \JPow \rG_t \xto{\neg_{\rG_t}^{\pOp}} (\JPow \rG_t)^{\pOp}$
  &&
  $\rR^\Dna = \JPow \rG_t \xto{\neg_{\rG_t}^{\pOp}} \JPow \rG_t \xto{\overline{\rR}^\down} (\JPow \rG_s)^{\pOp}$
  \end{tabular}
  \]
  They are adjoint because ($\overline{\rR}^\up$,$\overline{\rR}^\down$) are adjoint by Lemma \ref{lem:up_down_basic}, and each $\neg_X^{\pOp} : \JPow X \to (\JPow X)^{\pOp}$ is self-adjoint.
  
  \item
  That $\rR^\Dna \circ \rR^\Upa = \cl_{\overline{\rR}}$ follows because $\neg_{\rR_t}$ is involutive, and the second description follows by (1). Finally:
  \[
  \begin{tabular}{lll}
  $\rR^\Upa \circ \rR^\Dna$
  &
  $= \neg_{\rR_t} \circ \overline{\rR}^\up \circ \overline{\rR}^\down \circ \neg_{\rR_t}$
  & by definition
  \\&
  $= (\overline{\rR}\,\spbreve)^\down \circ (\overline{\rR}\,\spbreve)^\up$
  & by DeMorgan duality
  \\&
  $= \cl_{\overline{\rR}\,\spbreve}$
  & by definition
  \end{tabular}
  \]
  and the second description again follows by (1).
  
  \end{enumerate}
  \end{proof}

\subsection{The Varieties $\SAJ$, $\SAM$ and $\SAI$}

We will soon define three varieties (equationally-defined classes of algebras) extending $\JSL$.

\begin{itemize}
\item[--]
the \emph{finite algebras of} $\SAJ$ amount to a finite join-semilattice $\aQ$ with a self-adjoint morphism $\aQ \to \aQ^{\pOp}$.
\item[--]
the \emph{finite algebras of} $\SAM$ amount to a finite join-semilattice $\aQ$ with a self-adjoint morphism $\aQ^{\pOp} \to \aQ$.
\item[--]
$\SAI = \SAJ \cap \SAM$ is essentially the variety of De Morgan algebras i.e.\ bounded lattices equipped with an involutive endofunction satisfying the De Morgan laws.
\end{itemize}

\begin{definition}[The three varieties $\SAJ$, $\SAM$ and $\SAI$]
\item
In each case we'll extend $\JSL$'s signature $\{ \bot : 0, \lor : 2 \}$ with a unary operation $\sigma$ satisfying the equation:
\[
\eqnOR
\qquad
\sigma(x \lor y) \preccurlyeq \sigma(x)
\]
where $\phi \preccurlyeq \psi$ is syntactic sugar for the equation $\phi \lor \psi \approx \psi$.\footnote{(Rev) stands for \emph{order-reversing} because it is equivalent to the rule
$
x \leq_\aQ y
\To 
\sigma(y) \leq_\aQ \sigma(x).
$}

\begin{enumerate}
\item
$\SAJ$ extends $\JSL$ with a single unary operation satisfying $\eqnOR$ and:
\[
\eqnEx
\qquad
x \preccurlyeq \sigma\sigma(x)
\]
where `Ex' stands for \emph{extensive}.

\item
$\SAM$ extends $\JSL$ with a single unary operation satisfying $\eqnOR$ and:
\[
\eqnCEx
\qquad
\sigma\sigma(x) \preccurlyeq x
\]
where `Cx' stands for \emph{co-extensive}.

\item
$\SAI$ extends $\JSL$ with a single unary operation satisfying $\eqnOR$ and:
\[
\eqnInv
\qquad
\sigma\sigma(x) \approx x
\]
where `Inv' stands for \emph{involutive}.
\end{enumerate}

We view them as categories in the usual sense: the objects are the algebras and a morphism $f : (\aQ_1,\sigma_1) \to (\aQ_2,\sigma_2)$ is function $f : Q_1 \to Q_2$ which preserves the three basic operations. Equivalently, $f$ defines a $\JSL$-morphism $\aQ_1 \to \aQ_2$ such that $f(\sigma_1(q)) = \sigma_2(f(q))$ for every $q \in Q$.  \endbox
\end{definition}

\bigskip

\noindent
These three varieties are related to one another as follows.

\begin{lemma}[Basic observations concerning $\SAJ$, $\SAM$ and $\SAI$]
\label{lem:sa_vars_basic_relationship}
\item
\begin{enumerate}
\item
$(\aQ,\sigma) \in \SAJ_f$ iff $(\aQ^{\pOp},\sigma) \in \SAM_f$.
\item
$\SAI = \SAJ \cap \SAM$.
\item
In $\SAJ$, $\SAM$ and $\SAI$ the equation $\sigma\sigma\sigma(x) \approx \sigma(x)$ holds.
\item
The equation $x \preccurlyeq \sigma(\bot)$ holds in $\SAJ$ but not $\SAM$. The equation $\sigma\sigma(\bot) \preccurlyeq x$ holds in $\SAM$ but not $\SAJ$.
\item
$(\aQ,\sigma) \in \SAJ$ iff $\sigma \circ \sigma$ defines a closure operator on $(Q,\leq_\aQ)$. Similarly, $(\aQ,\sigma) \in \SAM$ iff $\sigma \circ \sigma$ defines an interior operator on $(Q,\leq_\aQ)$.
\end{enumerate}
\end{lemma}

\begin{proof}
\item
\begin{enumerate}
\item
$\eqnEx$ is the order-dual of $\eqnCEx$.
\item
$\eqnInv$ holds iff both $\eqnEx$ and $\eqnCEx$ hold.
\item
In $\SAJ$ we can apply $\eqnOR$ to $\eqnEx$ to deduce that $\sigma(\sigma \circ \sigma(x)) \leq \sigma(x)$, whereas $\sigma(x) \leq \sigma \circ \sigma(\sigma(x))$ arises from $\eqnEx$ and the substitution rule. We have the order-dual argument in $\SAM$, and in $\SAI$ we apply substitution to $\eqnInv$.
\item
In $\SAJ$ we have $x \preccurlyeq \sigma\sigma(x)$ by $\eqnEx$ and applying $\eqnOR$ to $\bot \preccurlyeq \sigma(x)$ yields $\sigma\sigma(x) \preccurlyeq \sigma(\bot)$, so that $x \preccurlyeq \sigma(\bot)$. This fails in $\SAM$ e.g.\ take any $\aQ \in \JSL$ with at least two elements and define $\sigma := \lambda q \in Q.\bot_\aQ$. Finally, in $\SAM$ applying $\eqnOR$ twice yields $\sigma\sigma(\bot) \preccurlyeq \sigma\sigma(x)$ and applying  $\eqnInv$ yields $\sigma\sigma(\bot) \preccurlyeq x$. This fails in $\SAJ$ e.g.\ take any join-semilattice $\aQ$ with a distinct bottom and top element and define $\sigma := \lambda q \in Q.\top_\aQ$.
\item
Given $(\aQ,\sigma) \in \SAJ$ then certainly $x \leq_\aQ \sigma\sigma(x)$ holds by $\eqnEx$ i.e.\ $\sigma \circ \sigma$ is extensive. Monotonicity follows by applying $\eqnOR$ twice (viewed as a rule), whereas idempotence follows using $\sigma\sigma\sigma(x) \approx \sigma(x)$ from (3). The proof for $\SAM$ is completely analogous.
\end{enumerate}
\end{proof}

\smallskip

\begin{example}[$\BA$ forms a full subcategory of $\SAI$]
\label{ex:ba_full_subcat_sai}
Observe that every possibly infinite boolean algebra $\aA$ arises as an $\SAI$-algebra $((A,\lor_\aA,\bot_\aA),\neg_\aA)$, and moreover the $\SAI$-morphisms between such algebras are precisely the boolean algebra morphisms. That is, $\BA$ forms a full subcategory of $\SAI$, and also of $\SAJ$ and $\SAM$. \endbox
\end{example}

\smallskip

\begin{example}[Characterising algebras built on a finite boolean join-semilattice]
  \label{ex:char_sa_jmi_fin_boolean}
  Let $Z$ be a finite set.
  \begin{enumerate}
  \item
  {\bf $(\JPow Z,\sigma) \in \SAI$ iff $\sigma = \neg_Z \circ \theta^\up$ for some involutive function $\theta : Z \to Z$}.
  
  \smallskip
  These algebras are well-defined i.e.\ $\eqnOR$ holds because $\theta^\up$ preserves the inclusion-ordering and $\neg_Z$ flips it, whereas $\eqnInv$ holds because:
  \[
  \begin{tabular}{lll}
  $\sigma \circ \sigma$
  &
  $= \neg_Z \circ \theta^\up \circ \neg_Z \circ \theta^\up$
  \\&
  $= \breve{\theta}^\down \circ \neg_Z \circ \neg_Z \circ \theta^\up$
  & by De Morgan duality
  \\&
  $= \theta^\down \circ \theta^\up$
  & $\theta = \theta^{\bf-1} = \breve{\theta}$ by involutivity
  \\&
  $= (\theta^{\bf-1})^\up \circ \theta^\up$
  & since $\theta$ bijective
  \\&
  $= id_{\Pow Z}$.
  \end{tabular}
  \]
  Conversely, take any $\SAI$-algebra $(\JPow Z,\sigma)$. By the characterisation in Lemma \ref{lem:interpret_infinite_sai} further below:
  \begin{quote}
   $\sigma$ defines a self-adjoint $\JSL$-isomorphism $\JPow Z \to (\JPow Z)^{\pOp}$.
  \end{quote}
   The $\JSL$-morphisms of type $\JPow Z \to \JPow Z$ are precisely the functions $\rR^\up$ where $\rR \subseteq Z \times Z$ is an arbitrary relation. Recalling the self-inverse $\JSL$-isomorphism $\neg_Z : \JPow Z \to (\JPow Z)^{\pOp}$, the $\JSL$-morphisms of type $\JPow Z \to (\JPow Z)^{\pOp}$ are precisely the functions $\neg_Z \circ \rR^\up$. Thus $\sigma = \neg_Z \circ \rR^\up$ where $\rR^\up$ is bijective because $\sigma$ and $\neg_Z$ are, so that $\rR$ is a bijective function (each singleton must be seen). Since $\neg_Z$ is self-adjoint,
  \[
  \sigma_*
  = (\neg_Z \circ \rR^\up)_*
  = \rR^\down \circ (\neg_Z)_*
  = \rR^\down \circ \neg_Z
  \qquad\text{so that}\qquad
  \rR^\down \circ \neg_Z = \neg_Z \circ \rR^\up.
  \]
  Thus $\rR^\up = \neg_Z \circ \rR^\down \circ \neg_Z = \breve{\rR}^\up$ by De Morgan duality, so $\rR = \breve{\rR}$. Then $\rR$ is a self-inverse bijection i.e.\ an involutive function $\theta : Z \to Z$.
  
  \smallskip

  \item
  {\bf $(\JPow Z,\sigma) \in \SAJ$ iff $\sigma = \neg_Z \circ \rR^\up$ for some symmetric relation $\rR \subseteq Z \times Z$}.
  
  \smallskip
  $\eqnOR$ is satisfied because $\rR^\up$ preserves the inclusion-ordering and $\neg_Z$ flips it. As for $\eqnEx$,
  \[
  \neg_Z \circ \rR^\up \circ \neg_Z \circ \rR^\up
  = \breve{\rR}^\down \circ \neg_Z \circ \neg_Z \circ \rR^\up
  = \rR^\down \circ \rR^\up
  = \cl_\rR
  \]
  using De Morgan duality and symmetry, which suffices because closure operators are extensive. Conversely, take any $\SAJ$-algebra $(\JPow Z,\sigma)$. By the characterization in Lemma \ref{lem:interpret_finite_saj_sam}.1 below:
  \begin{quote}
  $\sigma$ defines a self-adjoint $\JSL$-morphism $\JPow Z \to (\JPow Z)^{\pOp}$.
  \end{quote}
  Repeating the reasoning in the previous example, we know that $\sigma = \neg_Z \circ \rR^\up$ for some relation $\rR \subseteq Z \times Z$. Then by $\sigma$'s self-adjointness we deduce that $\rR^\down \circ \neg_Z = \neg_Z \circ \rR^\up$ and thus $\rR^\up = \breve{\rR}^\up$ by De Morgan duality, so that $\rR$ is symmetric as required.
  
  \smallskip

  \item
  {\bf $(\JPow Z,\sigma) \in \SAM$ iff $\sigma = \rR^\up \circ \neg_Z$ for some symmetric relation $\rR \subseteq Z \times Z$}.
  
  \smallskip
  This follows by the previous example and $(\JPow Z,\sigma) \in \SAM \iff ((\JPow Z)^{\pOp},\sigma) \in \SAJ$. In more detail, the latter $\SAJ$-algebras necessarily take the form:
  \[
  \sigma = \qquad (\JPow Z)^{\pOp} \xto{\neg_Z} \JPow Z \xto{\sigma_0} (\JPow Z)^{\pOp} \xto{\neg_Z} \JPow Z
  \]
  where $(\JPow Z,\sigma_0)$ is a $\SAJ$-algebra. Then we immediately deduce that:
  \[
  \sigma
  = \neg_Z \circ (\neg_Z \circ \rR^\up) \circ \neg_Z
  = \rR^\up \circ \neg_Z
  \]
  where $\rR$ is symmetric, and every symmetric relation is permissible.
  
  \item
  We explain how the above $\SAI$-algebras correspond to undirected graphs i.e.\ $(\JPow Z,\sigma) \in \SAI$ where $\sigma = \neg_Z \circ \theta^\up$ for some involutive function $\theta : Z \to Z$. The `equivalent' undirected graph is the relation $\Pirr\sigma \subseteq J(\JPow Z) \times J(\JPow Z)$. It is symmetric because:
  \[
  \begin{tabular}{lll}
  $\Pirr\sigma(\{z_1\},\{z_2\})$
  &
  $\iff \sigma(\{z_1\}) \nleq_{(\JPow Z)^{\pOp}} \{z_2\}$
  \\&
  $\iff \{z_2\} \nsubseteq \neg_Z \circ \theta^\up(\{z_1\})$
  \\&
  $\iff z_2 \in \theta[z_1]$
  \\&
  $\iff z_2 = \theta(z_1)$
  \end{tabular}
  \]
  and $\theta$ is an involutive function i.e.\ a functional relation which is bijective and symmetric. Here are three examples of these undirected graphs:
  \bigskip
  \[
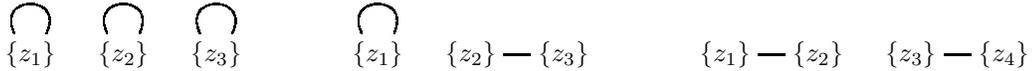
  \begin{tabular}{lllllll}
  $\vcenter{\vbox{\xymatrix@=10pt{
  \{z_1\} \ar@{-}@(ul,ur) & \{z_2\} \ar@{-}@(ul,ur) & \{z_3\} \ar@{-}@(ul,ur)
  }}}$
  &&&
  $\vcenter{\vbox{\xymatrix@=10pt{
  \{z_1\} \ar@{-}@(ul,ur) & \{z_2\} \ar@{-}[r] & \{z_3\} \ar@{-}[l]
  }}}$
  &&&
  $\vcenter{\vbox{\xymatrix@=10pt{
  \{z_1\} \ar@{-}[r] & \{z_2\} \ar@{-}[l] & \{z_3\} \ar@{-}[r] & \{z_4\} \ar@{-}[l]
  }}}$
  \end{tabular}
  \]
  That is, they are disjoint unions of self-loops and single-edge-paths. Note that $(\JPow Z,\neg_Z)$ correponds to the graph consisting of $|Z|$ self-loops i.e.\ $\Delta_{J(\JPow Z)}$.

  \item
  We'll characterise the \emph{distributive} $\SAJ_f$, $\SAM_f$ and $\SAI_f$-algebras in Theorem \ref{thm:char_distrib_saj_sam_sai} further below. The respective undirected graphs $(V, \rE)$ are precisely those satisfying $\rE = \theta ; \leq_\pP$ for some finite poset $\pP$ and involutive order-isomorphism $\theta : \pP \to \pP^\pOp$. \endbox

  \end{enumerate}
\end{example}

\subsection{Adjointness and self-adjointness}

Before reinterpreting the above finite algebras, let us first clarify the notion of adjoint morphism. Recall that adjoint $\JSL_f$-morphisms $\OD_j f^{op} := f_*$ arise via the action of the self-duality functor $\OD_j : \JSL_f^{op} \to \JSL_f$ from Theorem \ref{thm:jsl_self_dual}. We'll also consider the infinite case. The notion of \emph{self-adjoint morphism} in $\JSL_f$ and $\BiCliq$ will also be defined and compared.

\begin{definition}[Adjoints of $\JSL$-morphisms between possibly infinite algebras]
\label{def:jsl_adjoints_infinite_case}
\item
Given a $\JSL$-morphism $f : \aQ \to \aR$ where $\aQ$ and $\aR$ define bounded lattices (so that $\aQ^{\pOp}$ and $\aR^{\pOp}$ are well-defined), then we say that \emph{$f$ has an adjoint} if there exists a $\JSL$-morphism $g : \aR^{\pOp} \to \aQ^{\pOp}$ such that:
\[
f(q) \leq_\aR r \iff q \leq_\aQ g(r)
\qquad
\text{for all $q \in Q$ and $r \in R$}.
\]
We also say that $f$ \emph{has adjoint} $g$. \endbox
\end{definition}

\begin{lemma}
\label{lem:jsl_adjoints_infinite_case}
If $f : \aQ \to \aR$ is a $\JSL$-morphism between bounded lattices then t.f.a.e.\
\begin{enumerate}[(a)]
\item
$f$ has an adjoint.
\item
The function $f_* := \lambda r \in R. \Lor_\aQ \{ q \in Q : f(q) \leq_\aR r \} : R \to Q$ is well-defined.

\item
$f$ has the unique adjoint $f_* : \aR^{\pOp} \to \aQ^{\pOp}$.

\end{enumerate}
\end{lemma}

\begin{proof}
\item
\begin{itemize}
\item
$(a \To b)$: Suppose $f$ has an adjoint i.e.\ we have a $\JSL$-morphism $g : \aR^{\pOp} \to \aQ^{\pOp}$ satisfying the adjoint relationship between the two orderings. Then for all $r \in R$ we have:
\[
g(r) 
= \Lor_\aR \{ q \in Q : q \leq_\aQ g(r) \}
= \Lor_\aR \{ q \in Q : f(q) \leq_\aR r \}
= f_*(r)
\]
using the adjoint relationship, so $f_* = g$ is a well-defined function.

\item
$(b \To a)$: Given $f_*$ is a well-defined function we first  establish  $\forall q \in Q. r \in R.\,(f(q) \leq_\aR r \iff q \leq_\aQ f_*(r))$. The implication $(\To)$ is immediate. Conversely if $q \leq_\aQ f_*(r)$ then by monotonicity and join-preservation:
\[
f(q) \leq_\aR f(f_*(r)) 
= f(\Lor_\aQ \{ q_0 \in Q : f(q_0) \leq_\aR r \})
= \Lor_\aQ \{ f(q_0) : q_0 \in Q, f(q_0) \leq_\aR r  \}
\leq_\aR r
\]
as required. To see that $f_*$ defines a $\JSL$-morphism of type $\aR^{\pOp} \to \aQ^{\pOp}$, observe $f_*(\top_\aR) = \Lor_\aQ Q = \top_\aQ$ and:
\[
\begin{tabular}{lll}
$f_*(r_1 \land_\aR r_2)$
&
$= \Lor_\aQ \{ q \in Q : f(q) \leq_\aR r_1 \land_\aR r_2 \}$
\\&
$= \Lor_\aQ \{ q \in Q : f(q) \leq_\aR r_1 \text{ and } f(q) \leq_\aR r_2 \}$
\\&
$= \Lor_\aQ \{ q \in Q : q \leq_\aQ f_*(r_1) \text{ and } q \leq_\aQ f_*(r_2) \}$
& by adjoint relationship
\\&
$= f_*(r_1) \land_\aQ f_*(r_2)$.
\end{tabular}
\]

\item
$(a \iff c)$: Inspecting the proof of $(a \To b)$ we see that if $f$ has adjoint $g$ then $g = f_*$ and hence is unique. The converse is immediate.
 \end{itemize}
\end{proof}

\smallskip

\begin{definition}[Self-adjoint morphisms in $\JSL_f$ and $\BiCliq$]
\item
A $\JSL_f$-morphism $f : \aQ \to \aR$ is \emph{self-adjoint} if $f = f_*$, so we must have $\aR = \aQ^{\pOp}$. Likewise, a $\BiCliq$-morphism $\rR : \rG \to \rH$ is \emph{self-adjoint} if $\rR\spcheck = \rR$, which means precisely that $\rR$ is a symmetric relation, and also $\rH = \breve{\rG}$. \endbox
\end{definition}

\smallskip
These two concepts are two sides of the same coin.
\smallskip

\begin{lemma}[Self-adjointness in $\BiCliq$ and $\JSL_f$]
\label{lem:self_adjointness_bicliq_vs_jsl}
\item
\begin{enumerate}
\item
A $\BiCliq$-morphism $\rR : \rG \to \rH$ is self-adjoint iff $\rH = \breve{\rG}$ and $\rR = \breve{\rR}$.

\item
Given any $\BiCliq$-morphism $\rR : \rG \to \breve{\rG}$ t.f.a.e.\
\begin{enumerate}[a.]
\item
$\rR : \rG \to \breve{\rG}$ is self-adjoint.
\item
$\rR_- = \rR_+$.
\item
$\partial_\rG^{\bf-1} \circ \Open\rR$ is a self-adjoint $\JSL_f$-morphism.
\item
$\Open\rR \circ \partial_{\breve{\rG}}$ is a self-adjoint $\JSL_f$-morphism.
\end{enumerate}

\item
Given any $\JSL_f$-morphism $f : \aQ \to \aQ^{\pOp}$ t.f.a.e.\
\begin{enumerate}
\item
$f$ is self-adjoint.
\item
$\Nleq f$ is a self-adjoint $\BiCliq$-morphism.
\item
$\Pirr f$ is a self-adjoint $\BiCliq$-morphism.
\item
$(\Pirr f)_- 
= (\Pirr f)_+
= \{ (j,m) \in J(\aQ) \times M(\aQ) : f(j) \leq_\aQ m \}$.
\end{enumerate}

\end{enumerate}
\end{lemma}

\begin{proof}
\item
\begin{enumerate}
\item
A $\BiCliq$-morphism $\rR : \rG \to \rH$ is self-adjoint if $\rR = \rR\spcheck$ recalling that $(-)\spcheck : \BiCliq^{op} \to \BiCliq$ is the self-duality functor. Since $\rR\spcheck = \breve{\rR} : \breve{\rH} \to \breve{\rG}$, this holds iff $\rR = \breve{\rR}$ and $\rG = \breve{\rH}$, in which case $\breve{\rG} = \rH$ follows.

\item
\begin{itemize}
\item
$(a \iff b)$: If $\rR$ is self-adjoint then recall that $\rR_- = (\rR\spcheck)_+$ holds generally. Conversely, recall the associated components always satisfy $\rR_- ; \breve{\rG} = \rR = \rG ; \rR_+\spbreve$, so that if $\rR_- = \rR_+$ we deduce $\rR = \breve{\rR}$.

\item
$(a \iff c)$: Given any $\BiCliq$-morphism $\rR : \rG \to \breve{\rG}$ which is self-adjoint i.e.\ $\rR = \breve{\rR}$, then:
\[
\begin{tabular}{lll}
$(\partial_\rG^{\bf-1} \circ \Open\rR)_*$
&
$= (\Open\rR)_* \circ (\partial_\rG^{\bf-1})_*$
\\&
$= (\Open\rR)_* \circ \partial_{\breve{\rG}}^{\bf-1}$
& by Lemma \ref{lem:open_dual_iso_adjoints}.1
\\&
$= \partial_{\rG}^{\bf-1} \circ \Open \breve{\rR} \circ \partial_{\breve{\rG}} \circ \partial_{\breve{\rG}}^{\bf-1}$
& by Lemma \ref{lem:partial_nat_iso}
\\&
$= \partial_{\rG}^{\bf-1} \circ \Open \breve{\rR}$
\\&
$= \partial_{\rG}^{\bf-1} \circ \Open\rR$
& since $\breve{\rR} = \rR$.
\end{tabular}
\]
Conversely, if $\partial_\rG^{\bf-1} \circ \Open\rR$ is self-adjoint we can reuse the above calculation to deduce that:
\[
\partial_\rG^{\bf-1} \circ \Open\rR = \partial_{\rG}^{\bf-1} \circ \Open \rR\spcheck.
\]
Cancelling the isomorphism yields $\Open\rR = \Open\rR\spcheck$ and hence $\rR = \rR\spcheck$ by faithfulness.

\item
$(a \iff d)$: Again suppose that $\rR = \breve{\rR}$ and calculate:
\[
\begin{tabular}{lll}
$(\Open\rR \circ \partial_{\breve{\rG}})_*$
&
$= (\partial_{\breve{\rG}})_* \circ (\Open\rR)_*$
\\&
$= \partial_{\rG} \circ (\Open\rR)_*$
& by Lemma \ref{lem:open_dual_iso_adjoints}.1
\\&
$=  \partial_{\rG} \circ \partial_\rG^{\bf-1} \circ \Open \breve{\rR} \circ \partial_{\breve{\rG}}$
& by Lemma \ref{lem:partial_nat_iso}
\\&
$= \Open \breve{\rR} \circ \partial_{\breve{\rG}}$
\\&
$= \Open \rR \circ \partial_{\breve{\rG}}$
& since $\breve{\rR} = \rR$
\end{tabular}
\]
Conversely we deduce as in the previous item that $\Open\rR = \Open\rR\spcheck$ so that $\rR = \rR\spcheck$ by faithfulness.

\end{itemize}

\item
\begin{itemize}
\item
$(a \iff b)$: Given any self-adjoint $\JSL_f$-morphism $f : \aQ \to \aQ^{\pOp}$, first observe the typing $\Nleq f : \Nleq\aQ \to \Nleq\aQ^{\pOp} = (\Nleq\aQ)\spcheck$. It is a symmetric relation because:
\[
\begin{tabular}{lll}
$\Nleq f(q_1,q_2)$
&
$\iff f(q_1) \nleq_{\aQ^{\pOp}} q_2$
& by definition
\\&
$\iff q_1 \nleq_{\aQ} f(q_2)$
& since $f = f_*$
\\&
$\iff f(q_2) \nleq_{\aQ^{\pOp}} q_1$
\\&
$\iff \Nleq f(q_2,q_1)$.
\end{tabular}
\]
Conversely, suppose that $\Nleq f = (\Nleq f)\spcheck$ so that:
\[
q_2 \leq_\aQ f_1(q_1)
\iff
f_1(q_1) \leq_{\aQ^{\pOp}} q_2
\iff
\overline{\Nleq f_1}(q_1,q_2)
\iff
\overline{\Nleq f_2}(q_2,q_1)
\iff
q_1 \leq_\aQ f_2(q_2).
\]
Then it follows that $f_* = f$ e.g.\ by Lemma \ref{lem:jsl_adjoints_infinite_case}.

\item
$(b \iff c)$: Follows because we have the natural isomorphism $\rE : \Pirr \To \Nleq$, see Lemma \ref{lem:pirr_to_nleq_iso}.2.

\item
$(c \iff d)$: Follows by the equivalence of (2).a and (2).b, observing that $(\Pirr f)_-(j,m) \iff m \leq_{\aQ^{\pOp}} f(j)$ holds generally, see Definition \ref{def:open_pirr}.

\end{itemize}
\end{enumerate}
\end{proof}

\begin{note}[Explicit description of the associated component of a self-adjoint morphism]
For any self-adjoint $\BiCliq$-morphism $\rE : \rG \to \breve{\rG}$ we have $\rE_- = \rE_+$ by the Lemma above. In Lemma \ref{lem:self_adjoint_component_description} further below we'll prove a more explicit description i.e.\ if $\rE : \rG \to \breve{\rG}$ is self-adjoint then $\rE_- = \rE_+ = \overline{\overline{\rE} ; \rG} \; \subseteq \; \rG_s \times \rG_t$. \endbox
\end{note}

\subsection{Interpreting the finite algebras of the three varieties}

We now reinterpret the finite algebras of $\SAJ$ and $\SAM$, and also all $\SAI$ algebras.

\smallskip

\begin{lemma}[Interpretation of finite $\SAJ$ and $\SAM$ algebras]
\label{lem:interpret_finite_saj_sam}
\item
Fix any finite join-semilattice $\aQ \in \JSL_f$ and endofunction $\sigma : Q \to Q$.
\begin{enumerate}
\item
$(\aQ,\sigma) \in \SAJ_f$ iff $\sigma$ defines a self-adjoint $\JSL_f$-morphism of type $\aQ \to \aQ^{\pOp}$, or equivalently:
\[
q_2 \leq_\aQ \sigma(q_1) \iff q_1 \leq_\aQ \sigma(q_2)
\qquad
\text{for all $q_1,q_2 \in Q$}.
\]

\item
$(\aQ,\sigma) \in \SAM_f$ iff $\sigma$ defines a self-adjoint $\JSL_f$-morphism of type $\aQ^{\pOp} \to \aQ$, or equivalently:
\[
\sigma(q_1) \leq_\aQ q_2 \iff \sigma(q_2) \leq_\aQ q_1
\qquad
\text{for all $q_1,q_2 \in Q$}.
\]
\end{enumerate}
\end{lemma}

\begin{proof}
\item
\begin{enumerate}
\item
Given a self-adjoint $\JSL_f$-morphism $\sigma : \aQ \to \aQ^{\pOp}$ then it certainly defines a monotone morphism $(Q,\leq_\aQ) \to (Q,\geq_\aQ)$, hence $\eqnOR$ holds. The self-adjoint relationship informs us that:
\[
q_2 \leq_\aQ \sigma(q_1)
\iff \sigma(q_1) \leq_{\aQ^{\pOp}} q_2
\iff q_1 \leq_\aQ \sigma(q_2).
\]
Consequently $\eqnEx$ holds, because for every $x \in Q$ we have $\sigma(x) \leq_\aQ \sigma(x) \iff x \leq_\aQ \sigma(\sigma(x))$. Conversely, suppose we have a function $\sigma : Q \to Q$ satisfying $\eqnOR$ and $\eqnEx$. Then $\forall q_1, q_2 \in Q$ we have:
\[
q_2 \leq_\aQ \sigma(q_1)
\implies
\sigma(\sigma(q_1)) \leq_\aQ \sigma(q_2)
\implies
q_1 \leq_\aQ \sigma(q_2)
\]
using the order-reversing monotonicity of $\sigma$ and also $q_1 \leq_\aQ \sigma(\sigma(q_1))$. Then by symmetry we have:
\[
\sigma(q_1) \leq_{\aQ^{\pOp}} q_2 \iff q_1 \leq_\aQ \sigma(q_2)
\qquad
\text{for all $q_1, q_2 \in Q$}.
\]
By Lemma \ref{lem:adjoint_cl_in}.2 and the fact that $\aQ$ has all finite joins, we deduce that $\sigma$ defines a $\JSL_f$-morphism of type $\aQ \to \aQ^{\pOp}$. Finally, the above equivalence informs us that $\sigma$ is self-adjoint i.e.\ $\sigma = \sigma_*$.

\item
We have $(\aQ,\sigma) \in \SAM_f$ if and only if $(\aQ^{\pOp},\sigma) \in \SAJ_f$ by Lemma \ref{lem:sa_vars_basic_relationship}, so apply (1).

\end{enumerate}
\end{proof}

\begin{note}[Interpretation of infinite $\SAJ$ and $\SAM$ algebras]
\item
The above interpretation does not extend to infinite algebras e.g.\ because $\aQ^{\pOp}$ needn't be a well-defined join-semilattice in the infinite case.
\begin{enumerate}
\item
Given any join-semilattice $\aQ \in \JSL$ then we have the $\SAM$-algebra $(\aQ,\sigma)$ where $\sigma(q) := \bot_\aQ$ for all $q \in Q$. Indeed, $\eqnOR$ holds trivially for this constant map, as does $\eqnCEx$ because $\sigma(\sigma(x)) = \bot_\aQ \leq_\aQ x$. Thus there are $\SAM$-algebras whose join-semilattice has no top and/or fails to have binary meets, see Definition \ref{def:std_order_theory}.12.d.

\item
Concerning $\SAJ$-algebras $(\aQ,\sigma)$, it so happens that $\aQ$ always has the top element $\sigma(\bot_\aQ)$ because:
\[
\bot_\aQ \leq_\aQ \sigma(q) \iff q \leq_\aQ \sigma(\bot_\aQ)
\]
via the adjoint relationship. However, $\aQ$ needn't have binary meets e.g.\ let $\aQ$ be the join-semilattice depicted in Definition \ref{def:std_order_theory}.12.d and define $\sigma(q) := \top$ for all $q \in Q$. \endbox

\end{enumerate}
\end{note}

\smallskip

\begin{lemma}[Interpretation of arbitrary $\SAI$-algebras]
\label{lem:interpret_infinite_sai}
\item
\begin{enumerate}
\item
If $(\aQ,\sigma) \in \SAI$ then $\aQ$ is a bounded lattice and $\sigma$ defines a self-adjoint $\JSL$-isomorphism $\aQ \to \aQ^{\pOp}$ (hence bounded lattice isomorphism) and also its inverse. Furthermore, $\aQ$ has the following meet structure:
\[
\top_\aQ = \sigma(\bot_\aQ)
\qquad
q_1 \land_\aQ q_2 = \sigma(\sigma(q_1) \lor_\aQ \sigma(q_2)).
\]

\item
Given any $\aQ \in \JSL$ and function $\sigma : Q \to Q$ then t.f.a.e.\

\begin{enumerate}
\item
$(\aQ,\sigma) \in \SAI$.
\item
$\sigma$ defines a self-adjoint $\JSL$-isomorphism $\aQ \to \aQ^{\pOp}$.
\item
$\sigma$ defines a self-adjoint $\JSL$-morphism of type $\aQ \to \aQ^{\pOp}$ and $\aQ^{\pOp} \to \aQ$.
\item
$\sigma$ is involutive and defines a $\JSL$-morphism of type $\aQ \to \aQ^{\pOp}$.
\end{enumerate}

Furthermore, in (b) and (d) one may replace $\aQ \to \aQ^{\pOp}$ with $\aQ^{\pOp} \to \aQ$.

\item
Every $\SAI$-morphism defines a bounded lattice morphism.

\end{enumerate}
\end{lemma}

\begin{proof}
\item
\begin{enumerate}
\item
Let $(\aQ,\sigma)$ be an $\SAI$-algebra. Given $q_2 \leq_\aQ \sigma(q_1)$ then applying $\sigma$ yields $q_1 = \sigma(\sigma(q_1)) \leq_\aQ \sigma(q_2)$, so that:
\[
(\star) \qquad
\sigma(q_1) \leq_{\aQ^{\pOp}} q_2
\iff
q_2 \leq_\aQ \sigma(q_1) \iff q_1 \leq_\aQ \sigma(q_2)
\qquad
\text{for all $q_1,q_2 \in Q$}.
\]
Since $\aQ$ has all finite joins, we may apply Lemma \ref{lem:adjoint_cl_in}.2 to deduce that $\sigma(\Lor_\aQ X) = \Land_\aQ \sigma[X]$ for every finite subset $X \subseteq_\omega Q$ i.e.\ these particular meets exist in $\aQ$. But since $\sigma$ is involutive it is bijective, hence $\aQ$ has all finite meets. It follows that $\aQ$ is a bounded lattice and $\sigma$ defines a join-semilattice morphism $\sigma : \aQ \to \aQ^{\pOp}$.

Now, since $\sigma$ is involutive it is bijective and thus a $\JSL$-isomorphism by universal algebra, and also a bounded lattice isomorphism because $\aQ$ and $\aQ^{\pOp}$ are bounded lattices. Again by involutiveness $\sigma^{\bf-1}(q) = \sigma(q)$. Observe that the join-semilattice isomorphism $\sigma : \aQ \to \aQ^{\pOp}$  between bounded lattices is self-adjoint by $(\star)$. Concerning $\aQ$'s meet structure, since $\sigma : \aQ \to \aQ^{\pOp}$ is a bounded lattice isomorphism we have $\top_\aQ = \sigma(\bot_\aQ)$, and finally:
\[
\sigma(\sigma(q_1) \lor_\aQ \sigma(q_2))
= \sigma(\sigma(q_1)) \land_\aQ \sigma(\sigma(q_2))
= q_1 \land_\aQ q_2.
\]

\item
\begin{itemize}
\item 
$(a \iff b)$: The implication $\To$ follows by (1). Conversely, $\eqnOR$ follows by taking the underlying monotone map of $\sigma : \aQ \to \aQ^{\pOp}$ whereas $\eqnInv$ holds because for every $q \in Q$ we have:
\[
\begin{tabular}{lll}
$\sigma(q)$
&
$= \sigma_*(q)$
& by self-adjointness
\\&
$= \Lor_\aQ \{ q' \in Q : \sigma(q') \leq_{\aQ^{\pOp}} q \}$
& by definition
\\&
$= \Lor_\aQ \{ q' \in Q : q \leq_\aQ \sigma(q') \}$
\\&
$= \Lor_\aQ \{ q' \in Q : q' \leq_\aQ \sigma^{\bf-1}(q) \}$
& apply isomorphism
\\&
$= \sigma^{\bf-1}(q)$.
\end{tabular}
\]

\item
$(a \iff c)$: Regarding $\To$, this follows because $a \To b$, and the inverse of a self-adjoint isomorphism is itself self-adjoint. Conversely, taking an underlying monotone map yields $\eqnOR$, and concerning involutiveness we can use the two adjoint relations to deduce that for every $q \in Q$:
\[
\sigma(q) \leq_\aQ \sigma(q) \iff q \leq_\aQ \sigma(\sigma(q))
\qquad\text{and}\qquad
\sigma(q) \leq_\aQ \sigma(q) \iff \sigma(\sigma(q)) \leq_\aQ q.
\]

\item
$(a \iff d)$: First of all, $\To$ follows because $a \To b$. The converse is immediate by taking the underlying monotone morphism.

\end{itemize}

\item
Follows because the top and binary meet are definable in terms of $\bot$, $\lor$ and $\sigma$ by (1), hence are preserved by algebra homomorphisms.

\end{enumerate}
\end{proof}

\begin{corollary}[$\SAI$ is the variety of De Morgan algebras]
\item
By identifying $(\aQ,\sigma) \in \SAI$ with the tuple $(Q,\lor_\aQ,\bot_\aQ,\land_\aQ,\top_\aQ,\sigma)$, the category $\SAI$ is precisely the variety of De Morgan algebras i.e.\ bounded lattices $\latL$ equipped with a unary operation $\sigma : L \to L$ satisfying the equations:
\[
\sigma\sigma(x) \approx x
\qquad
\sigma(x \lor y) \approx \sigma(x) \land \sigma(y)
\qquad
\sigma(x \land y) \approx \sigma(x) \lor \sigma(y).
\]
\end{corollary}

\begin{proof}
Given $(\aQ,\sigma) \in \SAI$ then the induced tuple $(Q,\lor_\aQ,\bot_\aQ,\land_\aQ,\top_\aQ,\sigma)$ is a de morgan algebra by $\eqnInv$ and Lemma \ref{lem:interpret_infinite_sai}.1. Conversely, any de morgan algebra defines an $\SAI$-algebra because $\eqnInv$ by assumption, and moreover $\sigma(x \lor y) \approx \sigma(x) \land \sigma(y)$ implies $\eqnOR$:
\[
x \preccurlyeq y \iff x \lor y \approx y \implies \sigma(x \lor y) \approx \sigma(y) \stackrel{!}{\implies} \sigma(x) \land \sigma(y) \approx \sigma(y) \iff \sigma(y) \preccurlyeq \sigma(x).
\]
The homomorphisms of De Morgan algebras are the bounded lattice morphisms which preserve the unary operation, and using Lemma \ref{lem:interpret_infinite_sai}.3 they are precisely the $\SAI$-morphisms.
\end{proof}

\begin{corollary}[Order-dual algebras]
\item
\begin{enumerate}
\item
Given $(\aQ,\sigma) \in \SAJ_f$ then $(\aQ^{\pOp},\sigma) \in \SAJ_f$ iff $(\aQ,\sigma) \in \SAI_f$. 
\item
Given $(\aQ,\sigma) \in \SAM_f$ then $(\aQ^{\pOp},\sigma) \in \SAM_f$ iff $(\aQ,\sigma) \in \SAI_f$. 
\item
Given $(\aQ,\sigma) \in \SAI$ then $\sigma$ defines an $\SAI$-isomorphism $(\aQ,\sigma) \to (\aQ^{\pOp},\sigma)$ and also its inverse.

\end{enumerate}
\end{corollary}

\begin{proof}
\item
\begin{enumerate}
\item
Given $(\aQ,\sigma) \in \SAI_f$ then $(\aQ^{\pOp},\sigma) \in \SAI_f \subseteq \SAJ_f$ because $\eqnInv$ continues to hold, as does $\eqnOR$ by considering the opposite monotone morphism. Conversely, if $(\aQ^{\pOp},\sigma) \in \SAJ_f$ then $\sigma$ defines a self-adjoint morphism of type $\aQ \to \aQ^{\pOp}$ and $\aQ^{\pOp} \to \aQ$ and hence an $\SAI_f$-algebra by Lemma \ref{lem:interpret_infinite_sai}.2.

\item
Follows because $(\aQ,\sigma) \in \SAM_f$ iff $(\aQ^{\pOp},\sigma) \in \SAJ_f$.

\item
$\sigma$ defines a $\JSL_f$-isomorphism $\aQ \to \aQ^{\pOp}$ by Lemma \ref{lem:interpret_infinite_sai}, and also an $\SAI_f$-morphism because $\sigma \circ \sigma = \sigma \circ \sigma$.

\end{enumerate}
\end{proof}

\smallskip
We can also characterise the finite $\SAJ_f$ and $\SAM_f$-algebras in terms of the finite de morgan algebras.

\smallskip

\begin{lemma}[$\SAJ_f$ and $\SAM_f$-algebras as extensions of finite De Morgan algebras]
\label{lem:saj_sam_f_as_sai_extensions}
\item
\begin{enumerate}
\item
Given any $(\aQ,\sigma) \in \SAJ_f$ then we have $(\sigma[\aQ],\sigma |_{\sigma[Q] \times \sigma[Q]}) \in \SAI_f$, in fact:
\[
\sigma = \quad
\aQ \stackrel{\sigma|_{Q \times \sigma[Q]}}{\epito} \sigma[\aQ] \xto{\sigma |_{\sigma[Q] \times \sigma[Q]} } (\sigma[\aQ])^{\pOp} \stackrel{(\sigma|_{Q \times \sigma[Q]})_*}{\monoto} \aQ^{\pOp}.
\]

\item
Given any $(\aQ,\sigma) \in \SAM_f$ then $(\sigma[\aQ],\sigma |_{\sigma[Q] \times \sigma[Q]}) \in \SAI_f$ and moreover:
\[
\sigma = \quad
\aQ^{\pOp} \stackrel{\sigma|_{Q \times \sigma[Q]}}{\epito} \sigma[\aQ^{\pOp}] \xto{\sigma |_{\sigma[Q] \times \sigma[Q]} } (\sigma[\aQ^{\pOp}])^{\pOp} \stackrel{(\sigma|_{Q \times \sigma[Q]})_*}{\monoto} \aQ.
\]

\item
Consequently,
\begin{itemize}
\item[--]
$(\aQ,\sigma) \in \SAJ_f$ iff there exists $(\aR,\sigma_0) \in \SAI_f$ and a $\JSL_f$-morphism $\alpha : \aQ \to \aR$ such that $\sigma = \alpha_* \circ \sigma_0 \circ \alpha$.

\item[--]
$(\aQ,\sigma) \in \SAM_f$ iff there exists $(\aR,\sigma_0) \in \SAI_f$ and a $\JSL_f$-morphism $\beta : \aR \to \aQ$ such that $\sigma = \beta \circ \sigma_0 \circ \beta_*$.
\end{itemize}

\end{enumerate}
\end{lemma}

\begin{proof}
\item
\begin{enumerate}
\item
By Lemma \ref{lem:interpret_finite_saj_sam} we know $\sigma$ defines a join-semilattice morphism $\sigma : \aQ \to \aQ^{\pOp}$. By the (surjection,inclusion) factorisation we have the join-semilattice inclusion-morphism $\sigma[\aQ] \hookto \aQ^{\pOp}$ which restricts to a join-semilattice endomorphism $\sigma_{\sigma[Q] \times \sigma[Q]}$ of type $\sigma[\aQ] \to \sigma[\aQ]$. By restriction it satisfies $\eqnOR$, whereas for any $\sigma(x) \in \sigma[Q]$,
\[
\sigma_{\sigma[Q] \times \sigma[Q]} \circ \sigma_{\sigma[Q] \times \sigma[Q]}(\sigma(x)) 
= \sigma\sigma\sigma(x) = \sigma(x)
\]
by Lemma \ref{lem:sa_vars_basic_relationship}.3, so that $\eqnInv$ holds. Consequently $(\sigma[\aQ],\sigma |_{\sigma[Q] \times \sigma[Q]})$ is a finite de morgan algebra.

Next, the surjective join-semilattice morphism $\sigma|_{Q \times \sigma[Q]}$ arises from the other part of $\sigma$'s (surjection,inclusion) factorisation. Then the respective composite is a well-defined join-semilattice morphism of type $\aQ \to \aQ^{\pOp}$. Its action is that of $\sigma$ because:
\[
\begin{tabular}{lll}
$(\sigma |_{\sigma[Q] \times Q})_* \circ \sigma_{\sigma[Q] \times \sigma[Q]} \circ \sigma |_{\sigma[Q] \times Q}(q_0))$
&
$= (\sigma |_{\sigma[Q] \times Q})_*(\sigma\sigma(q_0))$
\\&
$= \Lor_\aQ \{  q \in Q : \sigma |_{\sigma[Q] \times Q}(q) \leq_{\aQ^{\pOp}} \sigma\sigma(q_0) \}$
\\&
$= \Lor_\aQ \{  q \in Q : \sigma\sigma(q_0) \leq_\aQ \sigma(q) \}$
\\&
$= \Lor_\aQ \{ q \in Q : q \leq_\aQ \sigma\sigma\sigma(q_0)  \}$
& by adjoint relationship
\\&
$= \sigma\sigma\sigma(q_0)$
\\&
$= \sigma(q_0)$
& by Lemma \ref{lem:sa_vars_basic_relationship}.3.
\end{tabular}
\]

\item
Follows because $(\aQ,\sigma) \in \SAM_f$ iff $(\aQ^{\pOp},\sigma) \in \SAI_f$.

\item
Take any $(\aR,\sigma_0) \in \SAI_f$ and any $\JSL_f$-morphism $\alpha : \aQ \to \aR$. Then $(\aQ,\alpha_* \circ \sigma_0 \circ \alpha) \in \SAJ_f$ because $\alpha_* \circ \sigma_0 \circ \alpha : \aQ \to \aQ^{\pOp}$ is a self-adjoint morphism because $\sigma_0 : \aR \to \aR^{\pOp}$ is. Conversely by (1) every $\SAJ_f$-algebra arises in this way, in fact we may assume $\alpha$ is surjective. The second item follows analogously.

\end{enumerate}
\end{proof}

\smallskip
It is worth mentioning a related result.
\smallskip

\begin{lemma}[Lifting $\JSL_f$-quotients and embeddings to $\SAJ_f$ and $\SAM_f$]
\item
Fix any finite de morgan algebra $(\aQ,\sigma)$.
\begin{enumerate}
\item
Each surjective $\JSL_f$-morphism $\psi : \aR \epito \aQ$ defines a $\SAJ$-morphism:
\[
\psi : (\aR, \psi_* \circ \sigma \circ \psi) \epito (\aQ,\sigma)
\]

\item
Each injective $\JSL_f$-morphism $e : \aQ \monoto \aR$ defines a $\SAM$-morphism:
\[
e : (\aQ,\sigma) \monoto (\aR,e \circ \sigma \circ e_*)
\]

\end{enumerate}
\end{lemma}

\begin{proof}
\item
\begin{enumerate}


\item
$(\aR, \psi_* \circ \sigma \circ \psi) \in \SAJ_f$ by Lemma \ref{lem:saj_sam_f_as_sai_extensions}.3. To establish that $\psi : (\aR,\sigma_\aR) \epito (\aQ,\sigma)$ is a well-defined $\SAJ$-morphism we must show that $\psi(\psi_* \circ \sigma \circ \psi(r)) = \sigma(\psi(r))$ for each $r \in R$. This follows because for every $q \in Q$,
\[
\begin{tabular}{lll}
$\psi(\psi_*(q))$
&
$= \psi(\Lor_\aR \{ r \in R : \psi(r) \leq_\aQ q \})$
& by definition
\\&
$= \Lor_\aQ \{ \psi(r) : r \in R, \, \psi(r) \leq_\aQ q \}$
& by join-preservation
\\&
$= q$
& by surjectivity.
\end{tabular}
\]

\item
$(\aR, e \circ \sigma \circ e_*) \in \SAM_f$ by Lemma \ref{lem:saj_sam_f_as_sai_extensions}.3. It remains to establish that $e : (\aQ,\sigma) \monoto (\aR,\sigma_\aR)$ is a well-defined $\SAM$-morphism. Then for every $q \in Q$ we must show that $e(\sigma(q)) = e \circ \sigma \circ e_*(e(q))$, which follows because:
\[
\begin{tabular}{lll}
$e_*(e(q))$
&
$= \Lor_\aQ \{ q' \in Q : e(q') \leq_\aR e(q) \}$
& by definition
\\&
$= \Lor_\aQ \{ q' \in Q : q' \leq_\aQ q \}$
& injective $\JSL$-morphisms are order embeddings
\\&
$= q$.
\end{tabular}
\]

\end{enumerate}
\end{proof}

\smallskip
We finish off with an explicit description of the free one-generated algebras. They are finite in each case, whereas the two-generated algebras are already infinite. In fact, a free De Morgan algebra on $X$ amounts to a free bounded lattice on $X + X$ equipped with a natural involution.
\smallskip

\begin{proposition}[Free one-generated $\SAJ$, $\SAM$ and $\SAI$-algebras]
\label{prop:free_one_gen_saj_sam_sai}
\item
\begin{enumerate}
\item
The free one-generated $\SAI$-algebra may be depicted as follows:
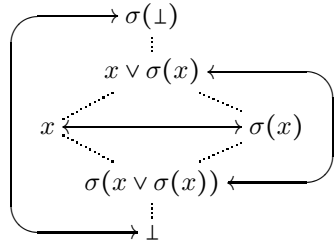
\[
\xymatrix@=5pt{
&& \sigma(\bot)
\\
&& x \lor \sigma(x) \ar@{..}[u] &&
\\
& x \ar@{..}[ur] \ar@{<->}[rr] & & \sigma(x)  \ar@{..}[ul] &
\\
&& \sigma(x \lor \sigma(x)) \ar@{..}[ul] \ar@{..}[ur] \ar@{<->}`[rru]`[uu][uu] &&
\\
&& \bot \ar@{..}[u] \ar@{<->}`[llu]`[uuuu][uuuu] &&
}
\]
More generally, given any set $X$ then the free $X$-generated $\SAI$-algebra arises as the free bounded lattice on generators $X + X$ with inductively defined unary operation:
\[
\begin{tabular}{c}
$\sigma_X(l(x)) := r(x)
\quad
\sigma_X(r(x)) := l(x)$
\\[1ex]
$\sigma_X(\phi \land \psi) := \sigma_X(\phi) \lor \sigma_X(\psi)
\quad
\sigma_X(\phi \lor \psi) := \sigma_X(\phi) \land \sigma_X(\psi)$.
\end{tabular}
\]

\item
The free one-generated $\SAJ$-algebra $(\aQ,\sigma)$ is depicted below:
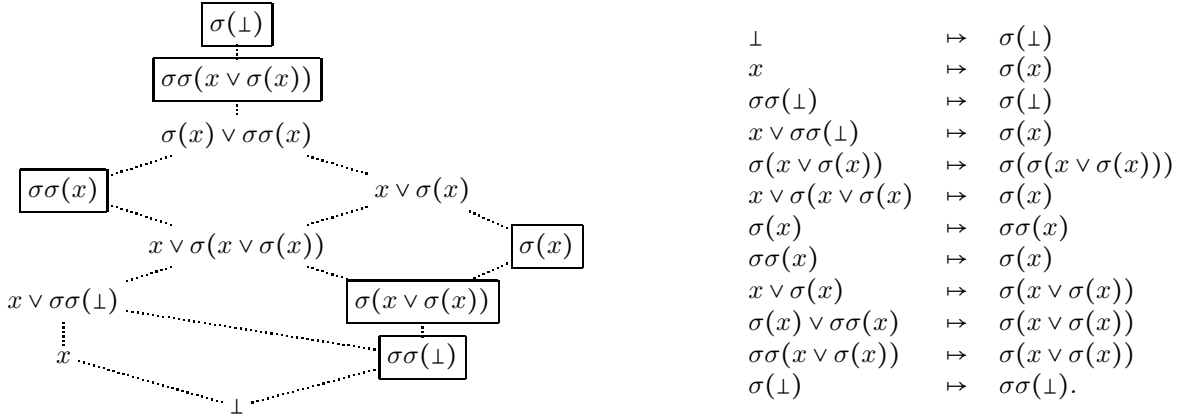
\[
\begin{tabular}{ccc}
$\vcenter{\vbox{\xymatrix@=5pt{
& *+[F-]{\sigma(\bot)}
\\
& *+[F-]{\sigma\sigma(x \lor \sigma(x))} \ar@{..}[u] &&&&
\\
& \sigma(x) \lor \sigma\sigma(x) \ar@{..}[u] &&&&
\\
*+[F-]{\sigma\sigma(x)} \ar@{..}[ur] && x \lor \sigma(x) \ar@{..}[ul]  &&
\\
& x \lor \sigma(x \lor \sigma(x)) \ar@{..}[ul] \ar@{..}[ur]   && *+[F-]{\sigma(x)} \ar@{..}[ul] &&
\\
x \lor \sigma\sigma(\bot)  \ar@{..}[ur] && *+[F-]{\sigma(x \lor \sigma(x))} \ar@{..}[ul] \ar@{..}[ur]  &&&&
\\
x \ar@{..}[u] && *+[F-]{\sigma\sigma(\bot)} \ar@{..}[u] \ar@{..}[ull] &&&
\\
& \bot \ar@{..}[ul] \ar@{..}[ur] & &&&&
}}}$
&&
\begin{tabular}{lll}
$\bot$ & $\mapsto$ & $\sigma(\bot)$
\\
$x$ & $\mapsto$ & $\sigma(x)$
\\
$\sigma\sigma(\bot)$ & $\mapsto$ & $\sigma(\bot)$
\\
$x \lor \sigma\sigma(\bot)$ & $\mapsto$ & $\sigma(x)$
\\
$\sigma(x \lor \sigma(x))$ & $\mapsto$ & $\sigma(\sigma(x \lor \sigma(x) ))$
\\
$x \lor \sigma(x \lor \sigma(x)$ & $\mapsto$ & $\sigma(x)$
\\
$\sigma(x)$ & $\mapsto$ & $\sigma\sigma(x)$
\\
$\sigma\sigma(x)$ & $\mapsto$ & $\sigma(x)$
\\
$x \lor \sigma(x)$ & $\mapsto$ & $\sigma(x \lor \sigma(x))$
\\
$\sigma(x) \lor \sigma\sigma(x)$ & $\mapsto$ & $\sigma(x \lor \sigma(x))$
\\
$\sigma\sigma(x \lor \sigma(x))$ & $\mapsto$ & $\sigma(x \lor \sigma(x))$
\\
$\sigma(\bot)$ & $\mapsto$ & $\sigma\sigma(\bot)$.
\end{tabular}
\end{tabular}
\]
The boxed elements show that the image $\sigma[\aQ] \subseteq \aQ^{\pOp}$ is a free $\SAI$-algebra on the generator $\sigma(x)$.

\item
The free one-generated $\SAM$-algebra $(\aQ,\sigma)$ may be depicted as follows:
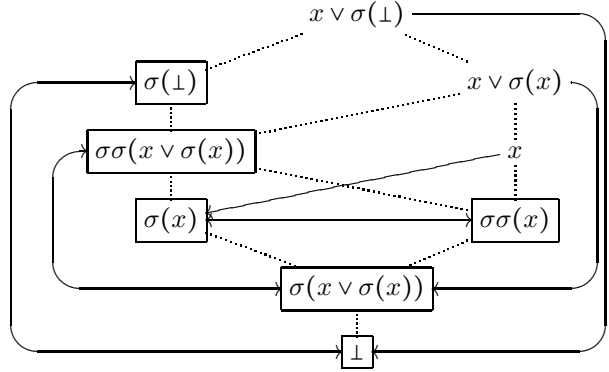
\[
\vcenter{\vbox{\xymatrix@=10pt{
&&& x \lor \sigma(\bot) \ar`r[rrd]`[ddddd][ddddd] &&
\\
&& *+[F-]{\sigma(\bot)} \ar@{..}[ur] && x \lor \sigma(x) \ar`r[d]`[dddl][dddl] \ar@{..}[ul] &
\\
&& *+[F-]{\sigma\sigma(x \lor \sigma(x))} \ar@{..}[u] \ar@{..}[urr] && x \ar[dll] \ar@{..}[u] &
\\
&& *+[F-]{\sigma(x)} \ar@{<->}[rr] \ar@{..}[u] && *+[F-]{\sigma\sigma(x)}  \ar@{..}[ull] \ar@{..}[u] &
\\
&&& *+[F-]{\sigma(x \lor \sigma(x))} \ar@{<->}`l[ll]`[uul][uul] \ar@{..}[ul] \ar@{..}[ur] &&
\\
&&& *+[F-]{\bot} \ar@{..}[u] \ar@{<->}`l[ulll]`[uuuul][uuuul] &
}}}
\]
The boxed elements show that the image $\sigma[\aQ] \subseteq \aQ$ is a free $\SAI$-algebra on the generator $\sigma(x)$.

\end{enumerate}
\end{proposition}

\begin{proof}
\item
\begin{enumerate}
\item
The depicted finite de morgan algebra is a well-defined bounded lattice because:
\[
\sigma(\bot) = \top
\qquad\qquad
\sigma(x \lor \sigma(x)) = x \land \sigma(x)
\]
i.e.\ we have the bounded lattice structure by Lemma \ref{lem:interpret_infinite_sai}. Then it is closed under the involution $\sigma$ and defines a finite de morgan algebra. Since no additional relations were assumed this is a free one-generated algebra.

\smallskip
Regarding the more general statement, take any set $X$ and let:
\[
\aQ_X := F(X + X)
\qquad
\text{be a free bounded lattice on generators $X + X$ i.e.\ two copies of $X$.}
\]
We may view its elements as equivalence classes of bounded lattice terms in variables $l(x), r(x)$ for $x \in X$. Then $\sigma_X$ is a well-defined involutive bounded lattice isomorphism $\aQ_X \to \aQ_X^{\pOp}$. This follows by the symmetry of the usual equational presentation of bounded lattices, and the fact that we may bijectively relabel variables, so that $\phi \approx \psi \iff \sigma_X(\phi) \approx \sigma_X(\psi)$. Then given $(\aQ,\sigma) \in \SAI$ and elements $el : X \to Q$, we have a unique bounded lattice morphism $\alpha : \aQ_X \to \aQ$ where:
\[
\alpha(l(x)) := el(x)
\qquad
\alpha(r(x)) := \sigma(el(x)),
\]
via the universal property of free bounded lattices since $\aQ$ is a bounded lattice. It remains to establish that $\alpha$ preserves the unary operation. First consider the base case:
\[
\alpha(\sigma_X(l(x))) = \alpha(r(x)) = \sigma(el(x)) = \sigma(\alpha(l(x)))
\qquad
\text{for each $x \in X$.}
\]
As for the inductive case, assuming that $\alpha(\sigma_X(\phi)) = \sigma(\alpha(\phi))$ holds for all $\phi \in \Phi$, then:
\[
\begin{tabular}{lll}
$\alpha(\sigma_X(\sigma_X(\phi)))$
&
$= \alpha(\phi)$
& $\sigma_X$ involutive
\\&
$= \sigma(\alpha(\sigma_X(\phi))$
& by induction, $\sigma$ involutive.
\end{tabular}
\]
\[
\begin{tabular}{lll}
$\alpha(\sigma_X(\Lor_{\aQ_X} \Phi))$
&
$= \alpha(\Land_{\aQ_X} \sigma_X[\Phi])$
& $\sigma_X : \aQ_X \to \aQ_X^{\pOp}$ a bounded lattice morphism
\\&
$= \Lor_\aQ \alpha \circ \sigma_X [\Phi])$
& $\alpha : \aQ_X \to \aQ$ a bounded lattice morphim
\\&
$= \Lor_\aQ \sigma(\alpha[\Phi])$
& by induction
\\&
$= \sigma(\alpha(\Lor_{\aQ_X} \Phi)$
& repeating reasoning in reverse.
\end{tabular}
\]

\item
(Sketch) Ignoring the terms we have a well-defined join-semilattice, which is actually distributive. One may verify that $\sigma$ satisfies the rules $\eqnOR$ and $\eqnEx$, so we have a well-defined finite $\SAJ$-algebra. Now view the elements as their respective term modulo the equational axioms of $\SAJ$. The join-structure is compatible using the fact that $\sigma(\bot)$ is the top element by Lemma \ref{lem:sa_vars_basic_relationship}.4. To see that the unary operation is compatible one verifies that the action is derivable from the equational laws. It suffices to verify a subset of them.
\begin{itemize}
\item
$\sigma(\bot) \approx \sigma(\bot)$ trivially.
\item
$\sigma(\sigma\sigma(\bot)) \approx \sigma(\bot)$ by Lemma \ref{lem:sa_vars_basic_relationship}.3.
\item
$\sigma(x \lor \sigma\sigma(\bot)) \approx \sigma(x)$. Indeed, since $x \preccurlyeq x \lor \sigma\sigma(\bot)$ we obtain $\sigma(x \lor \sigma\sigma(\bot)) \preccurlyeq \sigma(x)$ via $\eqnOR$. Conversely, $\sigma(x) \preccurlyeq \sigma(x \lor \sigma\sigma(\bot)) \iff x \lor \sigma\sigma(\bot) \preccurlyeq \sigma\sigma(x)$ via the adjoint relationship, and hence holds using $\eqnEx$ and by applying $\eqnOR$ twice.

\item
$\sigma(x \lor \sigma(x \lor \sigma(x))) \approx \sigma(x)$. Firstly since $x \preccurlyeq x \lor \sigma(x \lor \sigma(x))$, applying $\eqnOR$ yields half of the desired equality. Conversely, we need to establish that:
\[
\sigma(x)
\preccurlyeq
\sigma(x \lor \sigma(x \lor \sigma(x))).
\]
Applying the adjoint relationship this is equivalent to $x \lor \sigma(x \lor \sigma(x)) \preccurlyeq \sigma\sigma(x)$. Then $x \preccurlyeq \sigma\sigma(x)$ is $\eqnEx$ and finally $\sigma(x \lor \sigma(x)) \preccurlyeq \sigma\sigma(x)$ follows by applying $\eqnOR$ to $\sigma(x) \preccurlyeq x \lor \sigma(x)$.

\end{itemize}

\item
(Sketch) Follows by the method used in (2), noting that $\sigma\sigma(\bot) \cong \bot$ in $\SAM$ by Lemma \ref{lem:sa_vars_basic_relationship}.4. That the action of $\sigma$ is witnessed by various equational proofs is easier than in (2). One only needs to use $\sigma\sigma\sigma(x) \cong \sigma(x)$ and $\sigma(x \lor \sigma(\bot)) \preccurlyeq \sigma\sigma(\bot) \cong \bot$ via $\eqnOR$.

\end{enumerate}
\end{proof}

\begin{corollary}
If $|X| > 1$ then the free $\SAJ$, $\SAM$ and $\SAI$-algebra on $X$ are infinite.
\end{corollary}

\begin{proof}
By Proposition \ref{prop:free_one_gen_saj_sam_sai}.1, a free $\SAI$-algebra on $X$ is a free bounded lattice on $X + X$ generators equipped with an involution. It is well-known that the free bounded lattice on $3$ generators $\{x,y,z\}$ is infinite e.g.\ we have the strict $<$-chain where $\phi_0 := x$ and $\phi_{n+1} := x \lor (y \land (z \lor (x \land (y \lor (z \land \phi_n)))))$ for all $n \geq 0$. Then if $|X| > 1$ it follows that the free $\SAI$-algebra on $X$ is infinite. Finally, the free $X$-generated $\SAJ$ and $\SAM$-algebra have the free $X$-generated $\SAI$-algebra as a quotient, so they are themselves infinite.
\end{proof}

\subsection{The categories $\UGJ$, $\UGM$ and $\UG$}




\begin{definition}[The three categories corresponding to the varieties]
\label{def:ulg_j_m}
\item
The compositional structure of the categories below is inherited from $\BiCliq$.

\begin{enumerate}
\item
$\UGJ$'s objects are pairs $(\rG,\rE)$ where:
\begin{enumerate}
\item
$\rG \subseteq \rG_s \times \rG_t$ is an arbitrary relation between finite sets,
\item
$\rE \subseteq \rG_s \times \rG_s$ is symmetric and defines a $\BiCliq$-morphism of type $\rG \to \breve{\rG}$.
\end{enumerate}
Its morphisms $\rR : (\rG,\rE_1) \to (\rH,\rE_2)$ are those $\BiCliq$-morphisms $\rR : \rG \to \rH$ such that:
\[
\rR^\up \circ \rE_1^\down = \rH^\up \circ (\rE_2 \fatsemi \breve{\rR})^\down.
\]

\item
$\UGM$'s objects are pairs $(\rG,\rE)$ where:
\begin{enumerate}
\item
$\rG \subseteq \rG_s \times \rG_t$ is an arbitrary relation between finite sets,
\item
$\rE \subseteq \rG_t \times \rG_t$ is symmetric and defines a $\BiCliq$-morphism of type $\breve{\rG} \to \rG$.
\end{enumerate}
Its morphisms $\rR : (\rG,\rE_1) \to (\rH,\rE_2)$ are those $\BiCliq$-morphisms $\rR : \rG \to \rH$ such that:
\[
\rE_2^\down \circ \rR^\up
= (\breve{\rR} \fatsemi \rE_1)^\down \circ \rG^\up.
\]

\item
$\UG$'s objects are the undirected graphs $(V,\rE)$ i.e.\ $V$ is a finite set and $\rE \subseteq V \times V$ is a symmetric relation. Its morphisms $\rR : (V_1,\rE_1) \to (V_2,\rE_2)$ are those $\BiCliq$-morphisms $\rR : \rE_1 \to \rE_2$ such that:
\[
\rR^\up \circ \rE_1^\down = \rE_2^\up \circ \breve{\rR}^\down
\qquad \text{or equivalently} \qquad
\rE_2^\down \circ \rR^\up = \breve{\rR}^\down \circ \rE_1^\up
\]
by De Morgan duality. \endbox
\end{enumerate}

\end{definition}

\smallskip

\begin{note}
  By Lemma \ref{lem:self_adjointness_bicliq_vs_jsl}.1, for any relation $\rE$, requiring it is \emph{symmetric} and defines a $\BiCliq$-morphism $\rG \to \breve{\rG}$ is equivalent to requiring it defines a \emph{self-adjoint} $\BiCliq$-morphism $\rG \to \breve{\rG}$. \endbox
\end{note}

\smallskip

\begin{note}[Concerning the additional constraints on $\BiCliq$-morphisms]
These constraints will be seen to capture the preservation of the unary operation at the algebraic level. We do not \emph{in general} know how to interpret these conditions in a more intuitive fashion. \endbox
\end{note}

\smallskip
Before proving well-definedness we describe $\UG$-morphisms via a single equation i.e.\ without the underlying assumption they are $\Dep$-morphisms.

\begin{lemma}[Characterisation of $\UG$-morphisms]
\item
Given undirected graphs $(V_i,\rE_i)_{i = 1,2}$, a relation $\rR \subseteq V_1 \times V_2$ defines a $\UG$-morphism $\rR : (V_1,\rE_1) \to (V_2,\rE_2)$ iff:
\[
\rR^\up = \rE_2^\up \circ \breve{\rR}^\down \circ \rE_1^\up.
\]
\end{lemma}

\begin{proof}
A $\UG$-morphism $\rR : (V_1,\rE_1) \to (V_2,\rE_2)$ is a $\Dep$-morphism hence $\rR^\up = \rR^\up \circ \cl_{\rE_1}$ by Lemma \ref{lem:bicliq_mor_char_max_witness}. Since $\rR^\up \circ \rE_1^\down = \rE_2^\up \circ \breve{\rR}^\down$, precomposing with $\rE_1^\up$ yields the desired equality. Conversely suppose $\rR \subseteq V_1 \times V_2$ satisfies $\rR^\up = \rE_2^\up \circ \breve{\rR}^\down \circ \rE_1^\up$. It follows that $\rR^\up \circ \cl_{\rE_1} = \rR^\up = \inte_{\rE_2} \circ \rR^\up$ by using $(\up\down\up)$ twice i.e.\ $\rR : \rE_1 \to \rE_2$ is a $\BiCliq$-morphism. Precomposing the assumed equality yields $\rR^\up \circ \rE_1^\up = \rE_2^\up \circ \breve{\rR}^\down \circ \inte_{\rE_1} = \rE_2^\up \circ \breve{\rR}^\down$ via Lemma \ref{lem:bicliq_mor_char_max_witness} and Lemma \ref{lem:up_down_basic}.4.
\end{proof}




\smallskip

\begin{lemma}
\label{lem:ugj_ugm_ug_well_def}
$\UGJ$, $\UGM$ and $\UG$ are well-defined categories.
\end{lemma}

\begin{proof}
\item
\begin{enumerate}
\item
We'll show that the $\UGJ$-morphisms are closed under the compositional structure of $\BiCliq$. For each $(\rG,\rE) \in \UGJ$, the $\BiCliq$ identity-morphism $id_\rG = \rG : \rG \to \rG$ defines a $\UGJ$-morphism $(\rG,\rE_\rG) \to (\rG,\rE_\rG)$ because:
\[
\rG^\up \circ (\rE_\rG \fatsemi (id_\rG)\spcheck )^\down
= \rG^\up \circ (\rE_\rG \fatsemi id_{ \breve{\rG} })^\down
= \rG^\up \circ \rE_\rG^\down.
\]
Finally, given any composite $(\rG,\rE_\rG) \xto{\rR} (\rH,\rE_2) \xto{\rS} (\rI,\rE_\rI)$ we calculate:
\[
\begin{tabular}{lll}
$\rI^\up \circ (\rE_\rI \fatsemi (\rR \fatsemi \rS)\spcheck)^\down$
&
$= \rI^\up \circ (\rE_\rI \fatsemi \breve{\rS} \fatsemi \breve{\rR})^\down$
& by functorality
\\&
$= \rI^\up \circ (\rE_\rI \fatsemi \breve{\rS})^\down \circ \breve{\rH}^\up \circ \breve{\rR}^\down$
& by $(\down\fatsemi)$
\\&
$= \rS^\up \circ \rE_2^\down \circ \breve{\rH}^\up \circ \breve{\rR}^\down$
& by assumption
\\&
$= \rS^\up \circ (\rE_2 \fatsemi \rR)^\down$
& by $(\down\fatsemi)$
\\&
$= \rS^\up \circ \rH^\down \circ \rH^\up \circ (\rE_2 \fatsemi \rR)^\down$
& by Lemma \ref{lem:bicliq_mor_char_max_witness}
\\&
$= \rS^\up \circ \rH^\down \circ \rR^\up \circ \rE_\rG^\down$
& by assumption
\\&
$= (\rR \fatsemi \rS)^\up \circ \rE_\rG^\down$
& by $(\up\fatsemi)$.
\end{tabular}
\]

\item
Given any $\UGM$-object $(\rG,\rE_\rG)$, the $\BiCliq$ identity-morphism $id_\rG = \rG$ defines a $\UGM$-morphism $(\rG,\rE_\rG) \to (\rG,\rE_\rG)$:
\[
((id_\rG)\spcheck \fatsemi \rE_\rG)^\down \circ \rG^\up
= (id_{\breve{\rG}} \fatsemi \rE_\rG)^\down \circ \rG^\up
= \rE_\rG^\down \circ \rG^\up.
\]
Given any composite $(\rG,\rE_\rG) \xto{\rR} (\rH,\rE_2) \xto{\rS} (\rI,\rE_\rI)$ we calculate:
\[
\begin{tabular}{lll}
$((\rR \fatsemi \rS)\spcheck \fatsemi \rE_\rG)^\down \circ \rG^\up$
&
$= (\breve{\rS} \fatsemi \breve{\rR} \fatsemi \rE_\rG)^\down \circ \rG^\up$
& by functorality
\\&
$=  \breve{\rS}^\down \circ \breve{\rH}^\up \circ (\breve{\rR}\fatsemi \rE_\rG)^\down \circ \rG^\up$
& by $(\down\fatsemi)$
\\&
$= \breve{\rS}^\down \circ \breve{\rH}^\up \circ \rE_2^\down \circ \rR^\up$
& by assumption
\\&
$= (\breve{\rS} \fatsemi \rE_2)^\down \circ \rR^\up$
& by $(\down\fatsemi)$
\\&
$= (\breve{\rS} \fatsemi \rE_2)^\down \circ \rH^\up \circ \rH^\down \circ \rR^\up$
& using Lemma \ref{lem:up_down_basic}.4
\\&
$= \rE_\rI^\down \circ \rS^\up \circ \rH^\down \circ \rR^\up$
& by assumption
\\&
$= \rE_\rI^\down \circ (\rR \fatsemi \rS)^\up$.
& by $(\up\fatsemi)$
\end{tabular}
\]

\item
Each $\UG$-object $(V,\rE)$ induces a well-defined $\UG_j$-object $(\rE,\rE)$ because $id_\rE = \rE : \rE \to \rE = \breve{\rE}$. In fact, a $\BiCliq$-morphism $\rR : \rE_1 \to \rE_2$ defines a $\UG$-morphism iff $\rR : (\rE_1,\rE_1) \to (\rE_2,\rE_2)$ is a $\UGJ$-morphism, which follows because the constraint on $\UGJ$-morphisms:
\[
\rR^\up \circ \rE_\rG^\down = \rH^\up \circ (\rE_2 \fatsemi \breve{\rR})^\down
\qquad\text{becomes}\qquad
\rR^\up \circ \rE_1^\down = \rE_2^\up \circ (id_{\rE_2} \fatsemi \breve{\rR})^\down = \rE_2^\up \circ \breve{\rR}^\down.
\]
Thus $\UG$ is isomorphic to a full subcategory of the well-defined category $\UGJ$, and hence is itself a well-defined category.
\end{enumerate}
\end{proof}

\smallskip

\begin{lemma}[Basic observations concerning the $\UG$, $\UGJ$ and $\UGM$-objects]
\item
\begin{enumerate}
\item
The $\UG$-objects are precisely the undirected graphs.
\item
$(\rG,\rE) \in \UGJ$ iff $(\breve{\rG},\rE) \in \UGM$.
\item
The $\UGJ$-objects are precisely the pairs $(\rG,\rE)$ where $\rE$ defines a self-adjoint $\BiCliq$-morphism $\rG \to \breve{\rG}$.
\item
The $\UGM$-objects are precisely the pairs $(\rG,\rE)$ where $\rE$ defines a self-adjoint $\BiCliq$-morphism $\breve{\rG} \to \rG$.
\end{enumerate}
\end{lemma}

\smallskip
Next a simple yet important characterisation.
Recall the standard polarity $\rE^\Upa : \Pow V \to \Pow V$ from Definition \ref{def:polarities}. It has action $\rE^\Upa(X) = \bigcap_{x \in X} \rE[x]$ by Lemma \ref{lem:polarities_basic}.
\smallskip

\begin{lemma}[Characterisation of the $\UGJ$-objects and the $\UGM$-objects]
\label{lem:char_ug_jm}
\item
\begin{enumerate}
\item
Given any bipartite graph $\rG \subseteq V \times \rG_t$ and symmetric relation $\rE \subseteq V \times V$ t.f.a.e.\
\begin{enumerate}[a.]
\item
$(\rG,\rE) \in \UGJ$.
\item
$\rE = \rG ; \rH$ for some relation $\rH \subseteq \rG_t \times V$.
\item
$\forall (v_1,v_2) \in \rE$ there exists $g_t \in \rG_t$ such that $v_1 \in \breve{\rG}[g_t]$ and $\forall v \in \breve{\rG}[g_t].\rE(v,v_2)$.
\item
$\rE = \bigcup_{g_t \in \rG_t} \, \breve{\rG}[g_t] \times \rE^\Upa(\breve{\rG}[g_t])$.
\item
$\rE^\up = \inte_{\breve{\rG}} \circ \rE^\up$.

\end{enumerate}
\item
Given any bipartite graph $\rG \subseteq \rG_s \times V$ and symmetric relation $\rE \subseteq V \times V$ t.f.a.e.\
\begin{enumerate}[a.]
\item
$(\rG,\rE) \in \UGM$.
\item
$\rE = \rH ; \rG$ for some relation $\rH \subseteq V \times \rG_s$.
\item
$\forall (v_1,v_2) \in \rE$ there exists $g_s \in \rG_s$ such that $v_1 \in \rG[g_s]$ and $\forall v \in \rG[i].\rE(v,v_2)$.
\item
$\rE = \bigcup_{g_s \in \rG_s} \, \rG[g_s] \times \rE^\Upa(\rG[g_s])$.
\item
$\rE^\up = \inte_\rG \circ \rE^\up$.
\end{enumerate}
\end{enumerate}
\end{lemma}

\begin{proof}
The second collection of equivalent statements follows from the first because $(\rG,\rE) \in \UGM \iff (\breve{\rG},\rE) \in \UGJ$, and moreover $\rE = \breve{\rG};\rH \iff \rE = \breve{\rH};\rG$ since $\rE$ is symmetric. We verify the first collection of equivalences.
\begin{itemize}
\item
$(a \iff b)$: Given a $\BiCliq$-morphism $\rE : \rG \to \breve{\rG}$ then $\rE = \rG;\rE_+\spbreve$ so we may choose $\rH := \rE_+\spbreve$. Conversely, if $\rE = \rG ; \rH$ then since $\rE$ is symmetric we have witnesses $\breve{\rH};\breve{\rG} = \rE = \rG;\rH$, hence $\rE$ defines a $\BiCliq$-morphism $\rG \to \breve{\rG}$.

\item
$(b \iff c)$: Suppose that $\rE = \rG ; \rH$. Then given $(v_1,v_2) \in \rE$ there exists $g_t \in \rG_t$ such that $\rG(v_1,g_t)$ and $\rH(g_t,v_2)$. Thus $v_1 \in \breve{\rG}[g_t]$ and for any $v \in \breve{\rG}[g_t]$ we have $\rG(v,g_t) \, \land \, \rH(g_t,v_2)$ and hence $\rE(v,v_2)$.

For the other implication, suppose that (c) holds and define:
\[
\rH := \{ (g_t,v) \in \rG_t \times V : \forall u \in \breve{\rG}[g_t].\rE(u,v)  \}.
\]
Then whenever $\rG(v_1,g_t) \, \land \, \rH(g_t,v_2)$ we deduce $\rE(v_1,v_2)$ by instantiating $u := v_1$, so that $\rG ; \rH \subseteq \rE$. For the converse inclusion, if $\rE(v_1,v_2)$ then by assumption there exists $g_t \in \rG_t$ such that $\rG(v_1,g_t) \, \land \, \rH(g_t,v_2)$.

$(c \iff d)$: For any $\rG \subseteq V \times \rG_t$ we have $\bigcup_{g_t \in \rG_t} \breve{\rG}[g_t] \times \rE^\Upa(\breve{\rG}[g_t]) \subseteq \rE$, seeing as $\rE^\Upa(\breve{\rG}[g_t])$ consists precisely of those vertices $v \in V$ which are adjacent in $\rE$ to every $u \in \breve{\rG}[g_t]$. To see that $(c)$ is equivalent to the converse inclusion, observe that $\forall v \in \breve{\rG}[g_t].\rE(v,v_2)$ holds iff $v_2 \in \rE^\Upa(\breve{\rG}[g_t]) = \bigcap_{v \in \breve{\rG}[g_t]} \rE[v]$.

\item
$(b \iff e)$: Since open sets are closed under unions, (e) is equivalent to $\forall v \in V.\rE[v] \in O(\breve{\rG})$. But this in turn is equivalent to assuming $\rE = \rH ; \breve{\rG}$ for some relation $\rH$.

\end{itemize}
\end{proof}

Recall $\rE_- = \rE_+$ for any self-adjoint $\BiCliq$-morphism $\rE : \rG \to \breve{\rG}$.

\begin{lemma}[Associated component of self-adjoint $\BiCliq$-morphisms]
\label{lem:self_adjoint_component_description}
If $\rE \subseteq V \times V$ is a self-adjoint morphism $\rE : \rG \to \breve{\rG}$ then:
\[
\rE_- = \rE_+ = \overline{\overline{\rE} ; \rG} \; \subseteq \; V \times \rG_t
\]
and also $\rE_-\spbreve[g_t] = \rE^\Upa(\breve{\rG}[g_t])$ for every $g_t \in \rG_t$.
\end{lemma}

\begin{proof}
By Lemma \ref{lem:self_adjointness_bicliq_vs_jsl}.2 it suffices to establish the equality $\rE_- = \overline{\overline{\rE} ; \rG}$; the other claim will follow on the way. First recall that for any $\BiCliq$-morphism $\rR : \rH_1 \to \rH_2$, its associated components $(\rR_-,\rR_+)$ are the maximum witnesses by Lemma \ref{lem:bicliq_mor_char_max_witness}.2. That is, whenever $\rR_l ; \rH_2 = \rR = \rH_1 ; \rR_r\spbreve$ then $(\rR_l,\rR_r)$ pairwise include into $(\rR_-,\rR_+)$. Applied to $\rR := \rE$ we deduce that $\rE_- = \rE_+$ is the largest relation $\rS \subseteq \rG_s \times \rG_t$ such that $\rE = \rS ; \breve{\rG}$. Since $(\rG,\rE) \in \UGJ$ via our assumption, it follows by Lemma \ref{lem:char_ug_jm}.1 that:
\[
\rE = \bigcup_{g_t \in \rG_t} \rE^\Upa(\breve{\rG}[g_t]) \times \breve{\rG}[g_t]
\quad\text{or equivalently}\quad
\text{$\rE = \rS ; \breve{\rG}$ where $\rS(v,g_t) :\iff v \in \rE^\Upa(\breve{\rG}[g_t])$}.
\]
Then by maximality we have $\rE^\Upa(\breve{\rG}[g_t]) \subseteq \rE_-\spbreve[g_t]$ for every $g_t \in \rG_t$, whereas the converse inclusions follow because if $v \nin \rE^\Upa(\breve{\rG}[g_t])$ then $\{v\} \times \breve{\rG}[g_t] \nsubseteq \rE$. Finally, we can rewrite these equalities by recalling the original definition of polarities i.e.\ as the `de morgan dual' $\rE^\Upa = \neg_V \circ \overline{\rE}^\up$.
\[
\begin{tabular}{lll}
$\rE_-(v,g_t)$
&
$\iff v \in \rE_-\spbreve[g_t]$
\\&
$\iff v \in \rE^\Upa(\breve{\rG}[g_t])$
& by above reasoning
\\&
$\iff v \in \neg_V \circ \overline{\rE}^\up(\breve{\rG}[g_t])$
& by definition of $(-)^\Upa$
\\&
$\iff v \in \neg_V (\breve{\rG} ; \overline{\rE}[g_t])$
& using $(; \, \up)$
\\&
$\iff v \in \overline{\breve{\rG} ; \overline{\rE}}[g_t]$
& property of complement relations
\\&
$\iff \overline{\breve{\rG} ; \overline{\rE}}(g_t,v)$
\\&
$\iff \overline{\overline{\rE} ; \rG}(v,g_t)$
& take converse, see below.
\end{tabular}
\]
For the final step recall that the complement and converse of arbitrary relations commute, and $\rE$ is symmetric.
\end{proof}

\subsection{$\UGJ$, $\UGM$ and $\UG$ -- some structural lemmas}

We now prove a number of useful lemmas. These results mirror certain properties of the finite algebras of $\SAJ$, $\SAM$ and $\SAI$. They will be easier to understand once the categorical equivalences have been proved.
\smallskip

\begin{lemma}[The diagonals of $\UGJ$ and $\UGM$ are equal and isomorphic to $\UG$]
\label{lem:ug_iso_to_diagonals}
\item
\begin{enumerate}
\item
$\UG$ is isomorphic to the full subcategory of $\UGJ$ with objects $(\rE,\rE)$ where $\breve{\rE} = \rE$. The witnessing identity-on-morphisms functor has action $\rE \mapsto (\rE,\rE)$.

\item
$\UG$ is isomorphic to the full subcategory of $\UGM$ with objects $(\rE,\rE)$ where $\breve{\rE} = \rE$. The witnessing identity-on-morphisms functor has action $\rE \mapsto (\rE,\rE)$.
\end{enumerate}
\end{lemma}

\begin{proof}
\item
\begin{enumerate}
\item
We already observed this in the proof of Lemma \ref{lem:ugj_ugm_ug_well_def}.3.

\item
As above, also because the constraint on $\UGM$-morphisms:
\[
\rE_2^\down \circ \rR^\up
= (\breve{\rR} \fatsemi \rE_\rG)^\down \circ \rG^\up.
\qquad\text{becomes}\qquad
\rE_2^\down \circ \rR^\up
= (\breve{\rR} \fatsemi id_{\rE_1})^\down \circ \rE_1^\up
= \breve{\rR}^\down \circ \rE_1^\up
\]
which is one of the two equivalent constraints on $\UG$-morphisms.
\end{enumerate}
\end{proof}

\begin{lemma}[Reflection of $\BiCliq$-isomorphisms]
\label{lem:ug_reflect_iso}
\item
The three forgetful functors from $\UGJ$, $\UGM$ and $\UG$ to $\BiCliq$ reflect isomorphisms.
\begin{enumerate}
\item
If $\rR : (\rG,\rE_1) \to (\rH,\rE_2)$ is a $\UGJ$-morphism and a $\BiCliq$-isomorphism then its inverse is a $\UGJ$-morphism.
\item
If $\rR : (\rG,\rE_1) \to (\rH,\rE_2)$ is a $\UGM$-morphism and a $\BiCliq$-isomorphism then its inverse is a $\UGM$-morphism.
\item
If $\rR : (V_1,\rE_1) \to (V_2,\rE_2)$ is a $\UG$-morphism and a $\BiCliq$-isomorphism then its inverse is a $\UG$-morphism.

\end{enumerate}
\end{lemma}

\begin{proof}
\item
\begin{enumerate}
\item
By assumption we have a $\BiCliq$-isomorphism $\rR : \rG \to \rH$ with inverse $\rS : \rH \to \rG$, and:
\[
\underbrace{\rR^\up \circ \rE_1^\down}_A
= \rH^\up \circ (\rE_2 \fatsemi \breve{\rR})^\down
= \underbrace{\rH^\up \circ \rE_2^\down \circ \breve{\rH}^\up \circ \breve{\rR}^\down}_B.
\]
Then we have:
\[
\begin{tabular}{c}
$(\rS_+\spbreve)^\up \circ \mathrm{A}
= (\rR ; \rS_+\spbreve)^\up \circ \rE_1^\down
= \rG^\up \circ \rE_1^\down$
\\[1ex]
$(\rS_+\spbreve)^\up \circ \mathrm{B}
= (\rH;\rS_+\spbreve)^\up \circ \rE_2^\down \circ \breve{\rH}^\up \circ \breve{\rR}^\down
= \rS^\up \circ \rE_2^\down \circ \breve{\rH}^\up \circ \breve{\rR}^\down$
\end{tabular}
\]
using Lemma \ref{lem:bicliq_iso_two_mor} and an associated component of $\rS$. Hence:
\[
\begin{tabular}{lll}
$\rG^\up \circ \rE_1^\down \circ \breve{\rR}^\up \circ \breve{\rH}^\down$
&
$= \rS^\up \circ \rE_2^\down \circ \breve{\rH}^\up \circ \cl_{\breve{\rR}} \circ \breve{\rH}^\down$
& see above
\\&
$= \rS^\up \circ \rE_2^\down \circ \breve{\rH}^\up \circ \cl_{\breve{\rH}} \circ \breve{\rH}^\down$
& since $\rR\spcheck$ monic, see Lemma \ref{lem:bicliq_mono_epi_char}
\\&
$= \rS^\up \circ \rE_2^\down \circ \breve{\rH}^\up \circ \breve{\rH}^\down$
& by $(\up\down\up)$
\\&
$= \rS^\up \circ \rE_2^\down \circ \inte_{\breve{\rH}}$
\\&
$= \rS^\up \circ \rE_2^\down$
& see Lemma \ref{lem:bicliq_mor_char_max_witness}.1 and Lemma \ref{lem:up_down_basic}.4.
\end{tabular}
\]
Furthermore since $\rR$ is an isomorphism we deduce that $\rR^\down \circ \rH^\up = \rG^\down \circ \rS^\up$ by Lemma \ref{lem:bicliq_iso_from_rel}, or equivalently $\breve{\rR}^\up \circ \breve{\rH}^\down = \breve{\rG}^\up \circ \breve{\rS}^\down$ by de morgan duality. Thus:
\[
\rS^\up \circ \rE_2^\down
= \rG^\up \circ \rE_1^\down \circ \breve{\rG}^\up \circ \breve{\rS}^\down
= \rG^\up \circ (\rE_1 \fatsemi \breve{\rS})^\down
\]
as required.

\item
We have a $\BiCliq$-isomorphism $\rR : \rG \to \rH$ with inverse $\rS : \rH \to \rG$, so:
\[
\underbrace{\rE_2^\down \circ \rR^\up}_A
= (\breve{\rR} \fatsemi \rE_1)^\down \circ \rG^\up
= \underbrace{\breve{\rR}^\down \circ \breve{\rG}^\up \circ \rE_1^\down \circ \rG^\up}_B.
\]
Then we have:
\[
\begin{tabular}{c}
$A \circ \rS_-^\up
= \rE_2^\down \circ (\rS_- ; \rR)^\up
= \rE_2^\down \circ \rH^\up$
\\[1ex]
$B \circ \rS_-^\up
= \breve{\rR}^\down \circ \breve{\rG}^\up \circ \rE_1^\down \circ (\rS_-;\rG)^\up
= \breve{\rR}^\down \circ \breve{\rG}^\up \circ \rE_1^\down \circ \rS^\up$
\end{tabular}
\]
using Lemma \ref{lem:bicliq_iso_two_mor} and an associated component of $\rS$. Hence:
\[
\begin{tabular}{lll}
$\breve{\rG}^\down \circ \breve{\rR}^\up \circ \rE_2^\down \circ \rH^\up$
&
$= \breve{\rG}^\down \circ \inte_{\breve{\rR}} \circ \breve{\rG}^\up \circ \rE_1^\down \circ \rS^\up$
& see above
\\&
$= \breve{\rG}^\down \circ \inte_{\breve{\rG}} \circ \breve{\rG}^\up \circ \rE_1^\down \circ \rS^\up$
& since $\rR\spcheck$ epic, see Lemma \ref{lem:bicliq_mono_epi_char}
\\&
$= \breve{\rG}^\down \circ \breve{\rG}^\up \circ \rE_1^\down \circ \rS^\up$
& by $(\up\down\up)$
\\&
$= \cl_{\breve{\rG}} \circ \rE_1^\down \circ \rS^\up$
\\&
$= \rE_1^\down \circ \rS^\up$
& see Lemma \ref{lem:bicliq_mor_char_max_witness}.1 and Lemma \ref{lem:up_down_basic}.4.
\end{tabular}
\]
Moreover since $\rG^\up \circ \rR^\down = \rS^\up \circ \rH^\down$ by Lemma \ref{lem:bicliq_iso_from_rel}, or equivalently $\breve{\rG}^\down \circ \breve{\rR}^\up =  \breve{\rS}^\down \circ \breve{\rH}^\up$ by De Morgan duality,
\[
\rE_1^\down \circ \rS^\up
= \breve{\rS}^\down \circ \breve{\rH}^\up \circ \rE_2^\down \circ \rH^\up
= (\breve{\rS} \fatsemi \rE_2)^\down \circ \rH^\up
\]
as required.

\item
Follows because $\UG$ is isomorphic to the full subcategory of $\UGJ$ with objects $(\rE,\rE)$ where $\rE$ is a symmetric relation, so we can apply (1).

\end{enumerate}
\end{proof}

\smallskip


\begin{lemma}[Graph isomorphisms induce $\UG$-isomorphisms]
\label{lem:ugraph_iso_induce_ug_iso}
\item
Each undirected graph isomorphism $f : (V,\rE_1) \to (V,\rE_2)$ induces the $\UG$-isomorphism:
\[
f ; \rE_2 = \rE_1 ; f : (V_1,\rE_1) \to (V,\rE_2).
\]
\end{lemma}

\begin{proof}
The equality $f ; \rE_2 = \rE_1 ; f$ provides a $\BiCliq$-isomorphism $\rR := f ; \rE_2 = \rE_1 ; f$ of type $\rE_1 \to \rE_2$. By Lemma \ref{lem:ug_reflect_iso} it suffices to show that $\rR$ defines a $\UG$-morphism $(V,\rE_1) \to (V,\rE_2)$ i.e.\ $f^\up \circ \rE_1^\down \;\stackrel{?}{=}\; \rE_2^\up \circ \breve{f}^\down$. We certainly know $f^\up \circ \rE_1^\up = \rE_2^\up \circ f^\up$, and applying De morgan Muality yields:
\[
\breve{f}^\down \circ (\rE_1\spbreve)^\down = (\rE_2\spbreve)^\down \circ \breve{f}^\down.
\]
The desired equality follows because each $\rE_i$ is symmetric and moreover $\breve{f}^\down = (f^{\bf-1})^\down = f^\up$ by bijectivity.
\end{proof}

\smallskip

\begin{lemma}[$\UG$-isomorphisms of reduced graphs]
\label{lem:reduced_ug_isos_graph_isos_correspondence}
Given reduced graphs $(V_i,\rE_i)$,
\begin{quote}
$\rR : (V_1,\rE_1) \to (V_2,\rE_2)$ is a $\UG$-isomorphism iff there exists a graph isomorphism $f : (V_1,\rE_1) \to (V_2,\rE_2)$ such that $f ; \rE_2 = \rR = \rE_1 ; f$.
\end{quote}
\end{lemma}

\begin{proof}
Recall that the usual graph isomorphisms $f : (V_1,\rE_1) \to (V_2,\rE_2)$ are precisely those bijective functions $f : V_1 \to V_2$ such that $f ; \rE_2 = \rE_1 ; f$. Given such an $f$ we obtain the $\UG$-isomorphism $\rR := f ; \rE_2$ by Lemma \ref{lem:ugraph_iso_induce_ug_iso}. Conversely, suppose that $\rR : (V_1,\rE_1) \to (V_2,\rE_2)$ is a $\UG$-isomorphism between reduced graphs. Since $\UG$ inherits the compositional structure of $\BiCliq$ we know that $\rR : \rE_1 \to \rE_2$ is a $\BiCliq$-isomorphism between reduced relations. Then by Lemma \ref{lem:bip_restrict_to_reduced}:
\[
\xymatrix@=15pt{
V_1 \ar[rr]^{f_u}_-\cong && V_2
\\
V_1 \ar[rr]_{f_l}^-\cong \ar[urr]^-<<<<\rR \ar[u]^{\rE_1} && V_2 \ar[u]_{\rE_2}
}
\]
for some bijections $f_u$ and $f_l$. Now, since $\rR$ is a $\UG$-morphism we have $\rR^\up \circ \rE_1^\down = \rE_2^\up \circ \breve{\rR}^\down$. Moreover:
\[
\rR^\up 
= (\rE_1 ; f_u)^\up 
\;\stackrel{(;\up)}{=}\; f_u^\up \circ \rE_1^\up
\qquad
\breve{\rR}^\down 
\;=\; ((f_l ; \rE_2)\spbreve)^\down 
\;\stackrel{(;\down)}{=}\; \rE_2^\down \circ (\breve{f}_l)^\down
\;\stackrel{!}{=}\; \rE_2^\down \circ f_l^\up
\]
where the marked equality follows because $(\breve{f}_l)^\down = (f_l^{\bf-1})^\down = f_l^\up$ since $f_l$ is bijective. Substituting into the known equality yields:
\[
f_u^\up
\;\geq\; f_u^\up \circ \inte_{\rE_1} 
\;=\; \cl_{\rE_2} \circ f_l^\up
\;\geq\; f_l^\up
\]
using the pointwise inclusion-ordering. But since $f_l^\up$ and $f_u^\up$ preserve singleton sets this implies $f_u = f_l$.
\end{proof}

\smallskip
\begin{corollary}[Automorphism groups of reduced graphs]
\item
\begin{enumerate}
\item
Two reduced graphs are $\UG$-isomorphic iff they are graph isomorphic.
\item
The $\UG$-automorphism group of a reduced graph is isomorphic to its classical automorphism group.
\end{enumerate}
\end{corollary}

\begin{proof}
\item
\begin{enumerate}
\item
Immediate by Lemma \ref{lem:reduced_ug_isos_graph_isos_correspondence}.

\item
Fix any reduced graph $(V,\rE)$. The elements of the two automorphism groups biject via Lemma \ref{lem:reduced_ug_isos_graph_isos_correspondence}, via $f \;\mapsto\; f ; \rE$. The identity function $id_V$ and is sent to $id_V ; \rE = \rE$ i.e.\ the $\UG$ identity morphism. Concerning composition:
\[
\xymatrix@=15pt{
V \ar[rr]^f && V \ar[rr]^g && V
\\
V \ar[rr]_f \ar[u]^{\rE} && V \ar[rr]_g \ar[u]_{\rE} && V \ar[u]_{\rE}
}
\]
we have $f ; g \;\mapsto\; f;g;\rE = (f ; \rE) \fatsemi (g ; \rE)$ by the usual rules of $\BiCliq$-composition.
\end{enumerate}
\end{proof}

\smallskip

\begin{lemma}[Isomorphism correspondence between $\UGJ$ and $\UGM$]
\label{lem:ugj_ugm_iso_correspondence}
\item
$\rR : (\rG,\rE_1) \to (\rH,\rE_2)$ is a $\UGJ$-isomorphism iff $\rR\spcheck : (\breve{\rH},\rE_1) \to (\breve{\rG},\rE_2)$ is a $\UGM$-isomorphism.
\end{lemma}

\begin{proof}
Let $\rR : (\rG,\rE_1) \to (\rH,\rE_2)$ be a $\UGJ$-isomorphism. $\rR : \rG \to \rH$ is a $\BiCliq$-isomorphism because the compositional structure of $\UGJ$ is inherited from $\BiCliq$. Then $\rR\spcheck : \rH\spcheck \to \rG\spcheck$ is also a $\BiCliq$-isomorphism because the self-duality functor $(-)\spcheck : \BiCliq^{op} \to \BiCliq$ preserves isos (as do all functors). By Lemma \ref{lem:ug_reflect_iso} it remains to show  $\rR\spcheck$ defines a $\UGM$-morphism of type $(\breve{\rH},\rE_1) \to (\breve{\rG},\rE_2)$  i.e.\
\[
\rE_1^\down \circ \breve{\rR}^\up 
\quad \stackrel{?}{=} \quad (\rR \fatsemi \rE_2)^\down \circ \breve{\rH}^\up
\quad \stackrel{(\fatsemi\down)}{=} \quad \rR^\down \circ \rH^\up \circ \rE_2^\down \circ \breve{\rH}^\up.
\]
Since $\rR$ is a $\UGJ$-morphism by assumption,
\[
\begin{tabular}{lll}
&
$\rR^\up \circ \rE_1^\down = \rH^\up \circ (\rE_2 \fatsemi \breve{\rR})^\down$
\\ $\iff$ &
$\rR^\up \circ \rE_1^\down = \rH^\up \circ \rE_2^\down \circ \breve{\rH}^\up \circ \breve{\rR}^\down$
& by $(\fatsemi\down)$
\\ $\iff$ &
$\cl_\rR \circ \rE_1^\down \circ \breve{\rR}^\up = \rR^\down \circ \rH^\up \circ \rE_2^\down \circ \breve{\rH}^\up \circ \cl_{\breve{\rR}}$
& pre/post compose with $\breve{\rR}^\up$/$\rR^\down$
\\ $\iff$ &
$\cl_\rG \circ \rE_1^\down \circ \breve{\rR}^\up = \rR^\down \circ \rH^\up \circ \rE_2^\down \circ \breve{\rH}^\up \circ \cl_{\breve{\rH}}$
& $\rR$ and $\rR\spcheck$ monic, see Lemma \ref{lem:bicliq_mono_epi_char}
\\ $\iff$ &
$\rE_1^\down \circ \breve{\rR}^\up = \rR^\down \circ \rH^\up \circ \rE_2^\down \circ \breve{\rH}^\up$
& since $\rE_1 : \rG \to \breve{\rG}$, also $(\up\down\up)$.
\end{tabular}
\]
Conversely given any $\UGM$-isomorphism $\rR : (\rG,\rE_1) \to (\rH,\rE_2)$ it suffices to show $\rR\spcheck$ defines a $\UGJ$-isomorphism of type $(\breve{\rH},\rE_2) \to (\breve{\rG},\rE_1)$. Reusing previous reasoning, we need only show that the $\BiCliq$-isomorphism $\rR\spcheck$ is a $\UGM$-morphism i.e.\
\[
\breve{\rR}^\up \circ \rE_2^\down
\quad \stackrel{?}{=} \quad
\breve{\rG}^\up \circ (\rE_1 \fatsemi \rR)^\down
\quad = \quad \breve{\rG}^\up \circ \rE_1^\down \circ \rG^\up \circ \rR^\down.
\]
where now $\rR$ is a $\UGM$-morphism by assumption:
\[
\begin{tabular}{lll}
&
$\rE_2^\down \circ \rR^\up = (\rR\spcheck \fatsemi \rE_1)^\down \circ \rG^\up$
\\ $\iff$ &
$\rE_2^\down \circ \rR^\up = \breve{\rR}^\down \circ \breve{\rG}^\up \circ \rE_1^\down \circ \rG^\up$
& by $(\fatsemi\down)$
\\ $\iff$ &
$\breve{\rR}^\up \circ \rE_2^\down \circ \inte_\rR = \inte_{\breve{\rR}} \circ \breve{\rG}^\up \circ \rE_1^\down \circ \rG^\up \circ \rR^\down$
& pre/post compose by $\rR^\down$/$\breve{\rR}^\up$
\\ $\iff$ &
$\breve{\rR}^\up \circ \rE_2^\down \circ \inte_\rH = \inte_{\breve{\rG}} \circ \breve{\rG}^\up \circ \rE_1^\down \circ \rG^\up \circ \rR^\down$
& $\rR$ and $\rR\spcheck$ epic, see Lemma \ref{lem:bicliq_mono_epi_char}
\\ $\iff$ &
$\breve{\rR}^\up \circ \rE_2^\down = \breve{\rG}^\up \circ \rE_1^\down \circ \rG^\up \circ \rR^\down$
& since $\rE_2 : \breve{\rH} \to \rH$, also $(\up\down\up)$.
\end{tabular}
\]
\end{proof}

\smallskip


\begin{lemma}[The inverse of a $\UG$-isomorphism is its converse]
\label{lem:inverse_ug_iso_is_converse}
\item
\begin{enumerate}
\item
$\rR : (V_1,\rE_1) \to (V_2,\rE_2)$ is a $\UG$-isomorphism iff $\breve{\rR} : (V_2,\rE_2) \to (V_1,\rE_1)$ is.
\item
If $\rR$ is a $\UG$-isomorphism then $\rR^{\bf-1} = \breve{\rR}$.
\end{enumerate}
\end{lemma}

\begin{proof}
\item
\begin{enumerate}
\item
By Lemma \ref{lem:ug_iso_to_diagonals} the diagonals of $\UGJ$ and $\UGM$ are (i) the same full subcategory, and (ii) categorically isomorphic to $\UG$ via the identity-on-morphisms functor where $(V,\rE) \mapsto (\rE,\rE)$. Then a $\UG$-isomorphism $\rR : (V_1,\rE_1) \to (V_2,\rE_2)$ defines a $\UGJ$-isomorphism $\rR : (\rE_1,\rE_1) \to (\rE_2,\rE_2)$. Applying Lemma \ref{lem:ugj_ugm_iso_correspondence} we obtain the $\UGM$-isomorphism $\breve{\rR} : (\rE_2,\rE_2) \to (\rE_1,\rE_1)$ (since $\rE_i\spbreve = \rE_i$), yielding a $\UG$-isomorphism $\breve{\rR} : (V_2,\rE_2) \to (V_1,\rE_1)$.

\item
Let $\rR : (V_1,\rE_1) \to (V_2,\rE_2)$ be a $\UG$-isomorphism. By (1) we have the $\UG$-isomorphism $\breve{\rR} : (V_2,\rE_2) \to (V_1,\rE_1)$. Then:
\[
\begin{tabular}{lll}
$(\rR \fatsemi \breve{\rR})^\up$
&
$= \breve{\rR}^\up \circ \rE_2^\down \circ \rR^\up$
& by $(\fatsemi\up)$
\\&
$= \breve{\rR}^\up \circ \breve{\rR}^\down \circ \rE_1^\up$
& since $\rR$ a $\UG$-morphism
\\&
$= \inte_{\breve{\rR}} \circ \rE_1^\up$
\\&
$= \inte_{\rE_1} \circ \rE_1^\up$
& since $\breve{\rR} : \rE_2 \to \rE_1$ epic
\\&
$= \rE_1^\up$
& by $(\up\down\up)$,
\end{tabular}
\]
and consequently $\rR \fatsemi \breve{\rR} = id_{\rE_1}$. By a symmetric argument one can prove $\breve{\rR} \fatsemi \rR = id_{\rE_2}$ too.

\end{enumerate}
\end{proof}

\smallskip

\begin{lemma}[Isomorphic graphs induce $\UGJ$ and $\UGM$-isomorphisms]
\item
Fix any graph isomorphism $f : (V_1,\rE_1) \to (V_2,\rE_2)$.
\begin{enumerate}
\item
Each $(\rG,\rE_1) \in \UGJ$ has an associated $\UGJ$-isomorphism $\rG : (\rG,\rE_1) \to (f^{\bf-1} ; \rG, \rE_2)$.
\item
Each $(\rG,\rE_2) \in \UGM$ has an associated $\UGM$-isomorphism $\rG : (\rG ; f,\rE_1) \to (\rG , \rE_2)$.
\end{enumerate}
\end{lemma}

\begin{proof}
\item
\begin{enumerate}
\item
To see $(f^{\bf-1} ; \rG,\rE_2) \in \UGJ$ observe that $\rE_2 = f^{\bf-1} ; \rE_1 ; f = f^{\bf-1} ; \rG ; \rE_+\spbreve ; f$ and apply Lemma \ref{lem:char_ug_jm}. Next, $\rG$ defines a $\BiCliq$-isomorphism $\rG \to f^{\bf-1} ; \rG$ via the following commuting diagram with bijective witnesses:
\[
\xymatrix@=15pt{
\rG_t \ar[rr]^{\Delta_{\rG_t}} && \rG_t
\\
V_1 \ar[u]^{\rG} \ar[rr]_{f} && V_2 \ar[u]_-{f^{\bf-1} ; \rG}
}
\]
To show that $\rG$ defines a $\UGJ$-isomorphism of the desired type, it suffices to establish that it is a $\UGJ$-morphism by Lemma \ref{lem:ug_reflect_iso}. Since $f$ is a graph isomorphism we deduce $\rE_2^\down \circ f^\up = f^\up \circ \rE_1^\down$ by Lemma \ref{lem:ugraph_iso_induce_ug_iso}, noting that $\breve{f}^\down = f^\up$. Then we calculate:
\[
\begin{tabular}{lll}
$(f^{\bf-1} ; \rG)^\up \circ (\rE_2 \fatsemi \breve{\rR})^\down$
&
$= \rG^\up \circ (f^{\bf-1})^\up \circ \rE_2^\down \circ ((f^{\bf-1} ; \rG)\spbreve)^\up \circ \breve{\rR}^\down$
& by $(;\up)$ and $(\fatsemi\down)$
\\&
$= \rG^\up \circ (f^{\bf-1})^\up \circ \rE_2^\down \circ  (\breve{\rG} ; f)^\up \circ \breve{\rG}^\down$
& $\breve{f}^{\bf-1} = f$ and $\rR = \rG$
\\&
$= \rG^\up \circ (f^{\bf-1})^\up \circ \rE_2^\down \circ f^\up \circ \breve{\rG}^\up \circ \breve{\rG}^\down$
\\&
$= \rG^\up \circ (f^{\bf-1})^\up \circ f^\up \circ \rE_1^\down \circ \breve{\rG}^\up \circ \breve{\rG}^\down$
& by earlier equality
\\&
$= \rG^\up \circ \rE_1^\down$
& $f ; f^{\bf-1} = \Delta_{V_1}$ and $\rE_1^\down \circ \inte_{\breve{\rG}} = \rE_1^\down$.
\end{tabular}
\]

\item
Given $(\rG,\rE_2) \in \UGM$ then $(\breve{\rG},\rE_2) \in \UGJ$ so by (1) we have the $\UGJ$-isomorphism $\breve{\rG} : (\breve{\rG},\rE_2) \to (f^{\bf-1} ; \breve{\rG},\rE_1)$. Then by Lemma \ref{lem:ugj_ugm_iso_correspondence} we obtain the desired $\UGM$-isomorphism $\rG : (\rG;f,\rE_1) \to (\rG,\rE_2)$ since $(f^{\bf-1};\breve{\rG})\spbreve = \rG ; f$.

\end{enumerate}
\end{proof}

\smallskip

\begin{lemma}[Lifting certain $\BiCliq$-epis and monos to $\UGJ$ and $\UGM$]
\label{lem:lift_bicliq_epi_mono_ugj_ugm}
\item
Let $(V,\rE)$ be an undirected graph.
\begin{enumerate}
\item
Given $(\rH,\rE) \in \UGJ$ then any $\BiCliq$-morphism $\rH : \rG \to \rH$ defines a $\UGJ$-morphism $(\rG,\rE) \to (\rH,\rE)$.
\item
Given $(\rG,\rE) \in \UGM$ then any $\BiCliq$-morphism $\rG : \rG \to \rH$ defines a $\UGM$-morphism $(\rG,\rE) \to (\rH,\rE)$.
\end{enumerate}
\end{lemma}

\begin{proof}
\item
\begin{enumerate}
\item
For clarity let $\rR : \rG \to \rH$ where $\rR = \rH$. Then given that $(\rH,\rE) \in \UGJ$ we deduce that $(\rG,\rE) \in \UGJ$ because $\rE = \rH ; \rE_+\spbreve = \rR ; \rE_+\spbreve = \rG ; \rR_+\spbreve ; \rE_+\spbreve$, so we can apply Lemma \ref{lem:char_ug_jm}. Finally we calculate:
\[
\begin{tabular}{lll}
$\rH^\up \circ (\rE \fatsemi \breve{\rR})^\down$
&
$= \rH^\up \circ \rE^\down \circ \breve{\rH}^\up \circ \breve{\rR}^\down$
& since $\rE : \rH \to \breve{\rH}$
\\&
$= \rH^\up \circ \rE^\down \circ \breve{\rH}^\up \circ \breve{\rH}^\down$
& since $\rR = \rH$
\\&
$= \rH^\up \circ \rE^\down$
& since $\rE^\down \circ \inte_{\breve{\rH}} = \rE^\down$
\\&
$= \rR^\up \circ \rE^\down$
& since $\rR = \rH$.
\end{tabular}
\]

\item
For clarity let $\rR : \rG \to \rH$ where $\rR = \rG$. Given that $(\rG,\rE) \in \UGM$ then we deduce $(\rH,\rE) \in \UGM$ because $\rE = \rG ; \rE_+\spbreve = \rR ; \rE_+\spbreve = \rG ; \rR_+\spbreve ; \rE_+\spbreve$. Finally we calculate:

\[
\begin{tabular}{lll}
$(\breve{\rR} \fatsemi \rE)^\down \circ \rG^\up$
&
$= \breve{\rR}^\down \circ \breve{\rG}^\up \circ \rE^\down \circ \rG^\up$
& since $\rE : \breve{\rG} \to \rG$
\\&
$= \breve{\rG}^\down \circ \breve{\rG}^\up \circ \rE^\down \circ \rG^\up$
& since $\rR = \rG$
\\&
$= \rE^\down \circ \rG^\up$
& since $\cl_{\breve{\rG}} \circ \rE^\down = \rE^\down$
\\&
$= \rE^\down \circ \rR^\up$
& since $\rR = \rG$.
\end{tabular}
\]

\end{enumerate}
\end{proof}


\subsection{The three categorical equivalences}

\begin{definition}[The equivalence functors]
\label{def:equiv_functors_ulg_all}
\item
\[
\begin{tabular}{lll}
  $\vcenter{\vbox{\xymatrix@=15pt{	
  \UGJ \ar@/^10pt/[rr]^{\GJOpen} & \cong & \SAJ_f \ar@/^10pt/[ll]^{\GJPirr}
  }}}$
  &
  $\vcenter{\vbox{\xymatrix@=15pt{	
  \UGM \ar@/^10pt/[rr]^{\GMOpen} & \cong & \SAM_f \ar@/^10pt/[ll]^{\GMPirr}
  }}}$
  &
  $\vcenter{\vbox{\xymatrix@=15pt{	
  \UG \ar@/^10pt/[rr]^{\GOpen} & \cong & \SAI_f \ar@/^10pt/[ll]^{\GPirr} &
  }}}$
\end{tabular}
\]
\begin{enumerate}
\item
Action on objects:
\[
\begin{tabular}{llllll}
  $\GJOpen (\rG,\rE)$ & $:= (\Open\rG, \partial_\rG^{\bf-1} \circ \Open\rE)$
  &&
  $\GJPirr (\aQ,\sigma)$ & $:= (\Pirr\aQ,\Pirr\sigma)$
  \\[1ex]
  $\GMOpen (\rG,\rE)$ & $:= (\Open\rG, \Open\rE \circ \partial_\rG)$
  &&
  $\GMPirr (\aQ,\sigma)$ & $:= (\Pirr\aQ,\Pirr\sigma)$
  \\[1ex]
  $\GOpen (V,\rE)$ & $:= (\Open\rE, \partial_\rE)$
  &&
  $\GPirr (\aQ,\sigma)$ & $:= (J(\aQ),\Pirr\sigma)$.
\end{tabular}
\]
For clarity,
\[
\begin{tabular}{l}
Regarding the functors on the left,
\\[1ex]
\begin{tabular}{lll}
  for $\GJOpen$
& $\rE$ is viewed as a (self-adjoint) $\BiCliq$-morphism of type $\rG \to \breve{\rG}$ when applying $\Open$.
\\[1ex]
for $\GMOpen$
& $\rE$ is viewed as a (self-adjoint) $\BiCliq$-morphism of type $\breve{\rG} \to \rG$ when applying $\Open$.
\\[1ex]
for $\GOpen$
& $\rE$ is a viewed as a (symmetric) binary relation when applying $\Open$.
\end{tabular}
\\
\\
Regarding the functors on the right,
\\[1ex]
\begin{tabular}{lll}
for $\GJPirr$
& $\sigma$ is viewed as a (self-adjoint) $\JSL_f$-morphism of type $\aQ \to \aQ^{\pOp}$ when applying $\Pirr$.
\\[1ex]
for $\GMPirr$
& $\sigma$ is viewed as a (self-adjoint) $\JSL_f$-morphism of type $\aQ^{\pOp} \to \aQ$ when applying $\Pirr$.
\\[1ex]
for $\GPirr$
& $\sigma$ is viewed as a (self-adjoint) $\JSL_f$-morphism of type $\aQ \to \aQ^{\pOp}$ when applying $\Pirr$.
\end{tabular}
\end{tabular}
\]

\item
Action on morphisms:
\begin{itemize}
\item[--]
$\GJOpen$, $\GMOpen$ and $\GOpen$ act as $\Open$ on the underlying $\BiCliq$-morphism. 
\item[--]
$\GJPirr$ and $\GMPirr$ act as $\Pirr$ on the underlying join-semilattice morphism. 
\item[--]
Finally, for any $\SAI_f$-morphism $f : (\aQ,\sigma_1) \to (\aR,\sigma_2)$,
\[
\GPirr f
:= \Pirr(\sigma_2 \circ f)
\;\stackrel{!}{=}\; \Pirr f ; \; \sigma_2 |_{M(\aR) \times J(\aR)} : (J(\aQ),\Pirr\sigma_1) \to (J(\aR),\Pirr\sigma_2)
\]
where the asserted equality is proved below. \endbox
\end{itemize}
\end{enumerate}
\end{definition}

\smallskip

\begin{example}[Complete graphs]
Consider $(V,\rE)$ where $\rE := \overline{\Delta_V}$. Applying $\GOpen$ yields the De Morgan algebra $(\Open\rE,\partial_\rE)$. Recall $\Open\rE = (O(\rE),\cup,\emptyset)$
where $O(\rE) := \{ \rE[X] : X \subseteq V \}$ = $\ang{\{ \overline{v} : v \in V \}}_{\JPow V}$, and:
\[
\partial_\rE
= \lambda Y. \rE[\overline{Y}] =
\lambda Y.
\begin{cases} 
\overline{v} & \text{if $Y = \overline{v}$}
\\
\emptyset & \text{if $Y = V$}
\\
V & \text{if $Y = \emptyset$}.
\end{cases}
\]
Then $O(\rE) = \{ \emptyset \}$ if $|V| \leq 1$ and is $\{ \emptyset, V\} \cup \{ \overline{v} : v \in V \}$ otherwise. Graphically:
\smallskip
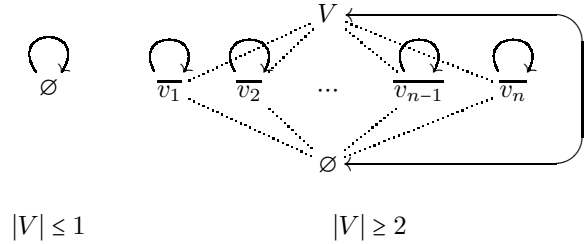
\[
\begin{tabular}{ccc}
$\vcenter{\vbox{\xymatrix@=15pt{
\emptyset \ar@(ul,ur)
}}}$
&&
$\vcenter{\vbox{\xymatrix@=15pt{
&& V && &
\\
\overline{v_1} \ar@(ul,ur) \ar@{..}[urr]  & \overline{v_2} \ar@{..}[ur] \ar@(ul,ur) & \cdots &  \overline{v_{n-1}} \ar@{..}[ul] \ar@(ul,ur)  & \overline{v_n} \ar@(ul,ur) \ar@{..}[ull] &
\\
& & \emptyset \ar@{..}[ull] \ar@{..}[ul] \ar@{..}[ur] \ar@{..}[urr] \ar@{<->}`[rrru]`[uu][uu] && &
}}}$
\\ \\
$|V| \leq 1$
&& 
$|V| \geq 2$
\end{tabular}
\]
They are well-defined De Morgan algebras and non-distributive whenever $|V| \geq 3$. Applying $\GPirr$ yields the graph $(J(\jslM{V}),\Pirr\sigma)$ with vertices $J(\jslM{V}) = \{ \overline{v} : v \in V \} \subseteq M_V$ and symmetric relation $\Pirr\sigma \subseteq J(\jslM{V}) \times J(\jslM{V})$,
\[
\Pirr\sigma(\overline{v_1},\overline{v_2}) 
:\iff \overline{v_2} \nsubseteq \partial_\rE(\overline{v_1})
\iff \overline{v_2} \nsubseteq \overline{v_1}
\iff v_1 \neq v_2
\iff \rE(v_1, v_2).
\]
That is, these De Morgan algebras lead back to the complete graphs. \endbox
\end{example}

\smallskip


\begin{example}[Chains as undirected graphs]
Chains are important examples of distributive lattices. Recall:
\[
\jslCh{n} := (C_n,max,0)
\qquad\text{where}\qquad
C_n := \{0,\dots,n\}
\]
has $n+1$ elements whereas its Hasse diagram has $n$ edges, and by definition its \emph{length} is $n$. We denote its underlying poset by $\Ch{n} := (C_n,\leq_{\jslCh{n}})$. The join-semilattice morphisms $\jslCh{m+1} \to \jslCh{n}$ naturally biject with the monotone morphisms $\Ch{m} \to \Ch{n}$ via the free construction $F_\lor : \Poset_f \to \JSL_f$, see Definition \ref{def:free_jsl_on_poset} in the Appendix. Every chain $\jslCh{n}$ extends to a finite De Morgan algebra in precisely one way:
\[
\sigma : C_n \to C_n
\qquad
\sigma(x) :=  n - x.
\]
The two equational axioms defining $\SAI_f$ are satisfied because $x \leq_{\jslCh{n}} y$ implies $n - y \leq_{\jslCh{n}} n - x$, and moreover $\sigma(\sigma(x)) = n - (n - x) = x$. It is unique because $\jslCh{n}$ has only one automorphism. Here is $\jslCh{n}$ for $0 \leq n < 5$,
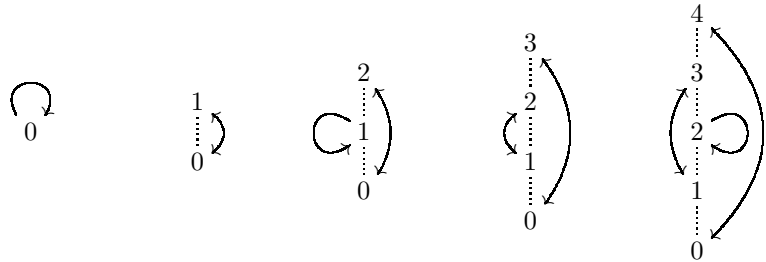
\[
\begin{tabular}{ccccc}
$\vcenter{\vbox{
\xymatrix@=10pt{
0 \ar@(ul,ur)
}
}}$
& \qquad\qquad
$\vcenter{\vbox{
\xymatrix@=10pt{
1
\\
0 \ar@{..}[u] \ar@/_10pt/@{<->}[u]
}
}}$
& \qquad\qquad
$\vcenter{\vbox{
\xymatrix@=10pt{
2
\\
1 \ar@{..}[u] \ar@(ul,dl)
\\
0 \ar@{..}[u] \ar@/_10pt/@{<->}[uu]
}
}}$
& \qquad\qquad
$\vcenter{\vbox{
\xymatrix@=10pt{
3
\\
2 \ar@{..}[u]
\\
1 \ar@{..}[u] \ar@/^10pt/@{<->}[u]
\\
0 \ar@{..}[u] \ar@/_15pt/@{<->}[uuu]
}
}}$
& \qquad\qquad
$\vcenter{\vbox{
\xymatrix@=10pt{
4
\\
3 \ar@{..}[u]
\\
2 \ar@{..}[u] \ar@(ur,dr)
\\
1 \ar@{..}[u] \ar@/^10pt/@{<->}[uu]
\\
0 \ar@{..}[u] \ar@/_25pt/@{<->}[uuuu]
}
}}$
\end{tabular}
\]
Concerning their equivalent undirected graphs:
\[
\GPirr (\jslCh{n},\sigma) = (C_n \backslash \{0\},\rE)
\qquad\text{where}\qquad
\rE(x,y) 
:\iff y \nleq_{\jslCh{n}} \sigma(x)
\iff y > n - x
\iff x + y > n.
\]
We now depict $\GPirr\jslCh{n}$ for $0 \leq n < 7$,
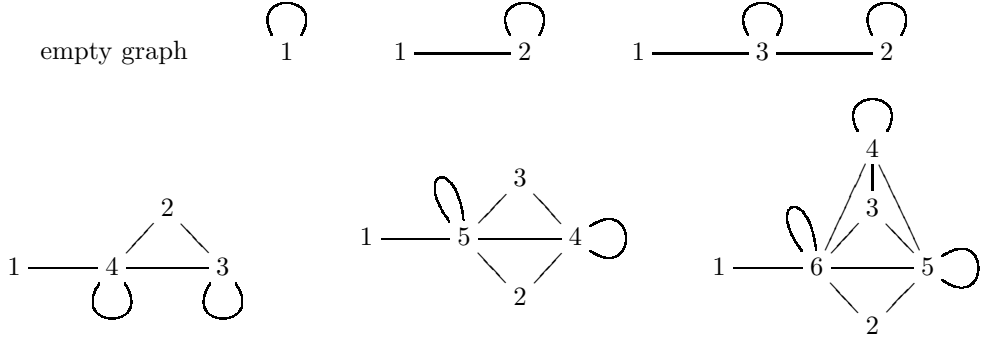
\[
\begin{tabular}{c}
\\
\begin{tabular}{ccccc}
empty graph
& \qquad
$\vcenter{\vbox{
\xymatrix@=15pt{
1 \ar@{-}@(ul,ur)
}
}}$
& \qquad
$\vcenter{\vbox{
\xymatrix@=15pt{
1 \ar@{-}[rr] && 2 \ar@{-}@(ul,ur)
}
}}$
& \qquad
$\vcenter{\vbox{
\xymatrix@=15pt{
1 \ar@{-}[rr] && 3 \ar@{-}@(ul,ur) \ar@{-}[rr] && 2 \ar@{-}@(ul,ur)
}
}}$
\end{tabular}
\\ \\ \\
$\vcenter{\vbox{\xymatrix@=10pt{
&& & 2 \ar@{-}[dr]
\\
1 \ar@{-}[rr] && 4 \ar@{-}@(dl,dr) \ar@{-}[rr] \ar@{-}[ur]  &  & 3 \ar@{-}@(dl,dr)
}}}$
\qquad\qquad
$\vcenter{\vbox{\xymatrix@=10pt{
&& & 3 \ar@{-}[dr]
\\
1 \ar@{-}[rr] && 5 \ar@{-}@(ul,u) \ar@{-}[rr] \ar@{-}[ur]  &  & 4 \ar@{-}[ll] \ar@{-}@(ur,dr)
\\
&& & 2 \ar@{-}[ur] \ar@{-}[ul]
}}}$
\qquad\qquad
$\vcenter{\vbox{\xymatrix@=10pt{
&& & 4 \ar@{-}[ddl] \ar@{-}[ddr] \ar@{-}@(ul,ur)
\\
&& & 3 \ar@{-}[dr] \ar@{-}[u]
\\
1 \ar@{-}[rr] && 6 \ar@{-}@(ul,u) \ar@{-}[rr] \ar@{-}[ur]  &  & 5 \ar@{-}[ll] \ar@{-}@(ur,dr)
\\
&& & 2 \ar@{-}[ur] \ar@{-}[ul]
}}}$
\end{tabular}
\]
They are planar graphs, and so is the next graph in the sequence. However, $\GPirr\jslCh{n}$ is non-planar for all $n \geq 8$ because $|\{ \frac{n}{2}, \dots, n \}| \geq 5$ forms a clique, so we may apply Kuratowski's theorem.

\smallskip
For brevity let $\dmCh{n} := (\jslCh{n},\lambda x.n - x)$ for each $n \geq 0$. Whenever $m = \alpha \cdot n$ i.e.\ $m$ divides $n$, there is an associated injective de morgan algebra morphism:
\[
f_{m,n} : \dmCh{m} \to \dmCh{n}
\qquad
f_{m,n}(k) := k \cdot \frac{n}{m} \;.
\]
i.e. it defines a join-semilattice morphism and preserves the involution:
\[
0 \cdot \frac{n}{m} = 0
\qquad
k \cdot max(x,y) = max(k \cdot x,k \cdot y)
\qquad
(m - k) \cdot \frac{n}{m} = n - k \cdot \frac{n}{m}.
\]
Then the corresponding $\UG$-monomorphism $\GPirr f_{m,n} : (C_m \backslash \{0\},\rE_m) \to (C_n \backslash \{0\},\rE_n)$ is the relation:
\[
\GPirr f_{m,n} \subseteq (C_m \backslash \{0\}) \times (C_n \backslash \{0\})
\qquad\text{where}\qquad
\GPirr f_{m,n} (x,y) \iff 1 <_\Nat \frac{x}{m} + \frac{y}{n}
\]
which follows by unwinding the definitions.\endbox
\end{example}


\bigskip
To prove well-definedness of the functors we'll make use of the following Lemma. Recall that the diagonals of $\UGJ$ and $\UGM$ are equal and isomorphic to $\UG$ by Lemma \ref{lem:ug_iso_to_diagonals}. We now provide isomorphisms between the two images $\GJPirr[\SAI_f \hookto \SAJ_f]$ and $\GMPirr[\SAI_f \hookto  \SAM_f]$ and this diagonal. After proving functoriality we'll be able to rephrase this result as two natural isomorphisms.
\bigskip

\begin{lemma}[The diagonals are isomorphic to the images of $\SAI_f \subseteq \SAJ_f, \, \SAM_f$]
\label{lem:sai_induces_ugjm_iso}
\item
Take any finite de morgan algebra $(\aQ,\sigma) \in \SAI_f$.
\begin{enumerate}
\item
Viewing $\sigma$ as a $\JSL$-isomorphism $\aQ \to \aQ^{\pOp}$ we have the $\UGJ$-isomorphism $\Pirr\aQ : (\Pirr\sigma,\Pirr\sigma) \to (\Pirr\aQ,\Pirr\sigma)$ with inverse $\Pirr\sigma$.
\item
Viewing $\sigma$ as a $\JSL$-isomorphism $\aQ^{\pOp} \to \aQ$ we have the $\UGM$-isomorphism $\Pirr\sigma : (\Pirr\sigma,\Pirr\sigma) \to (\Pirr\aQ,\Pirr\sigma)$ with inverse $\Pirr\aQ$.
\end{enumerate}
\end{lemma}

\begin{proof}
\item
\begin{enumerate}
\item
Given $(\aQ,\sigma) \in \SAI_f$ then since $\sigma_* = \sigma$ we deduce that $\Pirr\sigma$  defines a self-adjoint $\BiCliq$-morphism $\Pirr\aQ \to (\Pirr\aQ)\spbreve$ by Lemma \ref{lem:self_adjointness_bicliq_vs_jsl}.3. Thus $(\Pirr\aQ,\Pirr\sigma) \in \UGJ$ by definition, and also $(\Pirr\sigma,\Pirr\sigma) \in \UGJ$ because $\Pirr\sigma$ is symmetric. To see that $\Pirr\aQ$ defines a $\BiCliq$-morphism of type $\Pirr\sigma \to \Pirr\aQ$, first recall:
\[
\Pirr\sigma(j_1,j_2) :\iff \sigma(j_1) \nleq_{\aQ^{\pOp}} j_2 \iff j_2 \nleq_\aQ \sigma(j_1)
\qquad\text{and}\qquad
\Pirr\aQ(j,m) :\iff j \nleq_\aQ m.
\]
Since $\sigma : \aQ \to \aQ^{\pOp}$ is a join-semilattice isomorphism, it restricts to bijections $\sigma |_{J(\aQ) \times M(\aQ)}$ and $\sigma |_{M(\aQ) \times J(\aQ)}$, which are the inverse of one another because $\sigma$ is involutive. Then we have: 
\[
\Pirr\sigma ; \sigma |_{J(\aQ) \times M(\aQ)} (j,m)
\iff \Pirr\sigma(j,\sigma(m))
\iff \sigma(m) \nleq_\aQ \sigma(j)
\stackrel{!}{\iff} j \nleq_\aQ m
\iff \Pirr\aQ(j,m),
\]
where the marked equality follows because $\sigma$ defines an order-isomorphism $(Q,\leq_\aQ) \to (Q,\geq_\aQ)$. Then the following diagram of relations commutes:
\[
\xymatrix@=15pt{
J(\aQ) \ar[rr]^-{\sigma |_{J(\aQ) \times M(\aQ)}} && M(\aQ) \ar[rr]^-{\sigma |_{M(\aQ) \times J(\aQ)}} && J(\aQ)
\\
J(\aQ) \ar[rr]_-{\Delta_{J(\aQ)}} \ar[u]^{\Pirr\sigma} && J(\aQ) \ar[rr]_-{\Delta_{J(\aQ)}} \ar[u]^{\Pirr\aQ} && J(\aQ)  \ar[u]_{\Pirr\sigma}
}
\]
It follows that $\Pirr\aQ : \Pirr\sigma \to \Pirr\aQ$ is a $\BiCliq$-isomorphism with inverse $\Pirr\sigma$. Then by Lemma \ref{lem:lift_bicliq_epi_mono_ugj_ugm} and also Lemma \ref{lem:ug_reflect_iso} it defines a $\UGJ$-isomorphism $(\Pirr\sigma,\Pirr\sigma) \to (\Pirr\aQ,\Pirr\sigma)$ with the same inverse.

\item
Let $(\aQ,\sigma) \in \SAI_f$ and view $\sigma$ as a self-adjoint isomorphism $\aQ^{\pOp} \to \aQ$. Since $(\aQ^{\pOp},\sigma) \in \SAI_f$ we may apply (1), yielding the $\UGJ$-isomorphism:
\[
\Pirr\sigma : (\Pirr\aQ^{\pOp},\Pirr\sigma) \to (\Pirr\sigma,\Pirr\sigma)
\qquad\text{with inverse $\Pirr\aQ^{\pOp}$}.
\]
By Lemma \ref{lem:ugj_ugm_iso_correspondence} we obtain the $\UGM$-isomorphism:
\[
\Pirr\sigma : (\Pirr\sigma,\Pirr\sigma) \to (\Pirr\aQ,\Pirr\sigma)
\qquad\text{with inverse $\Pirr\aQ$}.
\]
also using the fact that $(\Pirr\sigma)\spbreve = \Pirr\sigma$ and $(\Pirr\aQ^{\pOp})\spbreve = \Pirr\aQ$.

\end{enumerate}
\end{proof}

\smallskip
We now prove well-definedness of the functors under consideration. That their action on objects is well-defined follows via Lemma \ref{lem:self_adjointness_bicliq_vs_jsl} i.e.\ the correspondence between self-adjointness in $\JSL_f$ and $\BiCliq$. Concerning their action on morphisms, well-definedness follows via mostly mindless computations. However, in the case of $\GPirr$ we make crucial use of the above Lemma. Notice that this is the only functor whose action on morphisms is not inherited from the underlying equivalence functors $\Pirr$ and $\Open$.
\smallskip

\begin{lemma}
The six functors from Definition \ref{def:equiv_functors_ulg_all} above are well-defined.
\end{lemma}

\begin{proof}
\item
\begin{enumerate}
\item
We first show that:
\[
\GJOpen : \UGJ \to \SAJ_f
\quad\text{and}\quad
\GMOpen : \UGM \to \SAM_f
\quad\text{and}\quad
\GOpen : \UG \to \SAI_f
\]
are well-defined. Consider their action on objects:
\[
\begin{tabular}{c}
$\GJOpen (\rG,\rE)
:= (\Open\rG, \partial_\rG^{\bf-1} \circ \Open\rE)
\quad
\GMOpen (\rG,\rE) 
:= (\Open\rG, \Open\rE \circ \partial_\rG)
\quad
\GOpen(V,\rE) := (\Open\rE, \partial_\rE)$
\end{tabular}
\]
The left and central actions are well-defined by Lemma \ref{lem:self_adjointness_bicliq_vs_jsl}.2 parts (c) and (d). Regarding the rightmost, first apply the identity-on-morphisms categorical isomorphism $(V,\rE) \mapsto (\rE,\rE)$ from Lemma \ref{lem:ug_iso_to_diagonals}.1, and subsequently $\GJOpen$. Then observe that $\partial_\rE^{\bf-1} \circ \Open id_\rE = \partial_\rE^{\bf-1}$ acts the same as $\partial_\rE$ because $\rE = \breve{\rE}$. Concerning the action on morphisms, we consider each functor in turn.

\smallskip
\item
Concerning $\GJOpen$, take any $\UGJ$-morphism $\rR : (\rG,\rE_1) \to (\rH,\rE_2)$ and consider the well-defined join-semilattice morphism $\GJOpen\rR = \Open\rR : \Open\rG \to \Open\rH$. To see that it is a $\SAJ$-morphism we must establish that:
\[
\Open\rR(\sigma_{\GJOpen(\rG,\rE_1)}(Y))
= \sigma_{\GJOpen(\rH,\rE_2)}(\Open\rR(Y))
\qquad
\text{for every $Y \in O(\rG_1) \subseteq \Pow \rG_t$},
\]
or more explicitly:
\[
\Open\rR(\partial_\rG^{\bf-1} \circ \Open\rE_1(Y))
=
\partial_\rH^{\bf-1} \circ \Open\rE_2(\Open\rR(Y)).
\]
Since the $\rG$-open sets are precisely those of the form $\rG[X]$, we may equivalently show that $\forall X \subseteq \rG_s$,
\[
\begin{tabular}{llcll}
&
$\Open\rR(\partial_\rG^{\bf-1}(\Open\rE_1(\rG^\up(X))$
& $\stackrel{?}{=}$ &
$\partial_\rH^{\bf-1} \circ \Open\rE_2(\Open\rR(\rG^\up(X))$
&
\\[0.5ex]
by defn &
$=\Open\rR(\partial_\rG^{\bf-1}((\rE_1)_+\spbreve[\rG[X]])$
&&
$=\partial_\rH^{\bf-1}(\Open\rE_2( \rR_+\spbreve[\rG[X]] ))$
& by defn
\\[0.5ex]
$\rE_1 = \rG ; (\rE_1)_+\spbreve$ &
$=\Open\rR(\partial_\rG^{\bf-1}(\rE_1^\up(X))$
&&
$=\partial_\rH^{\bf-1}(\Open\rE_2(\rR^\up(X)))$
& $\rR = \rG ; \rR_+\spbreve$
\\[0.5ex]
by defn &
$=\Open\rR(\rG^\up \circ \neg_{\rG_s} \circ \rE_1^\up(X))$
&&
$=\partial_\rH^{\bf-1}(\Open\rE_2(\rH^\up \circ \rH^\down \circ \rR^\up(X)))$
& $\rR^\up = \inte_\rH \circ \rR^\up$
\\[0.5ex]
by defn &
$=\rR_+\spbreve[\rG[\neg_{\rG_s} \circ \rE_1^\up(X)]]$
&&
$= \rH^\up \circ \neg_{\rH_s} \circ \rE_2^\up \circ \rH^\down \circ \rR^\up(X) $
& $\rE_2 = \rH;(\rE_2)_+\spbreve$
\\[0.5ex]
$\rR = \rG ; \rR_+\spbreve$ &
$=\rR^\up \circ \neg_{\rG_s} \circ \rE_1^\up(X)$
&&
$=\rH^\up \circ \rE_2^\down \circ \breve{\rH}^\up \circ \breve{\rR}^\down(\overline{X})$
& $(\neg\down/\up\neg)$
\\[0.5ex]
$(\neg\up\neg)$ &
$=\rR^\up \circ \rE_1^\down (\overline{X})$
&&
$=\rH^\up \circ (\rE_2 \fatsemi \breve{\rR})^\down(\overline{X})$
& $(\fatsemi\down)$
\end{tabular}
\]
Then since $X$ is an arbitrary subset this amounts to our assumption $\rR^\up \circ \rE_1^\down = \rH^\up \circ (\rE_2 \fatsemi \breve{\rR})^\down$. Hence $\GJOpen$'s action on both objects and morphisms is well-defined. It preserves the compositional structure because it acts in the same way as $\Open : \BiCliq \to \JSL_f$, and the compositional structure in both $\JSL_f$ and $\SAJ_f$ is functional.

\item
Next consider $\GMOpen$ i.e.\ take any $\UGM$-morphism $\rR : (\rG,\rE_1) \to (\rH,\rE_2)$ and consider the well-defined $\JSL$-morphism $\GMOpen\rR = \Open\rR : \Open\rG \to \Open\rH$. To see that it is a $\SAM$-morphism we must establish:
\[
\Open\rR(\sigma_{\GMOpen(\rG,\rE_1)}(Y))
= \sigma_{\GMOpen(\rH,\rE_2)}(\Open\rR(Y))
\qquad
\text{for every $Y \in O(\rG_1) \subseteq \Pow\rG_t$},
\]
or more explicitly:
\[
\Open\rR(\Open\rE_1 \circ \partial_\rG(Y))
=
\Open\rE_2 \circ \partial_\rH(\Open\rR(Y)).
\]
Since the $\rG$-open sets are precisely those of the form $\rG[X]$, we may equivalently show that $\forall X \subseteq \rG_s$,
\[
\begin{tabular}{lllll}
&
$\Open\rR(\Open\rE_1 \circ \partial_\rG(\rG[X]))$
& $\stackrel{?}{=}$ &
$\Open\rE_2 \circ \partial_\rH(\Open\rR(\rG[X]]))$
&
\\
by defn &
$= \Open\rR(\Open\rE_1( \breve{\rG}^\up \circ \neg_{\rG_t} \circ \rG^\up(X)))$
&&
$= \Open\rE_2 \circ \partial_\rH(\rR[X])$
& $\rR = \rG ; \rR_+\spbreve$
\\
$\rE_1 = \breve{\rG};(\rE_1)_+\spbreve$  & 
$= \Open\rR(\rE_1^\up \circ  \neg_{\rG_t} \circ \rG^\up(X))$
&&
$= \Open\rE_2 \circ \breve{\rH}^\up \circ \neg_{\rH_t} \circ \rR^\up(X)$
& by defn
\\
$(\neg\up\neg)$ & 
$= \Open\rR(\rE_1^\up \circ \breve{\rG}^\down(\overline{X}))$
&&
$= \Open\rE_2 \circ \breve{\rH}^\up \circ \breve{\rR}^\down(\overline{X})$
& $(\neg\up\neg)$
\\
$\inte_\rG \circ \rE_1^\up  = \inte_\rG$ &
$= \Open\rR(\rG^\up \circ \rG^\down \circ \rE_1^\up \circ \breve{\rG}^\down(\overline{X})$
&&
$= \rE_2^\up \circ \breve{\rR}^\down(\overline{X})$
& $\rE_2 = \breve{\rH} ; (\rE_2)_+\spbreve$
\\
$\rR = \rG ; \rR_+\spbreve$ &
$= \rR^\up \circ \rG^\down \circ \rE_1^\up \circ \breve{\rG}^\down(\overline{X})$
&&
\\
$(\fatsemi\up)$ &
$= (\rE_1 \fatsemi \rR)^\up \circ \breve{\rG}^\down(\overline{X})$
&&
\end{tabular}
\]
Since $X$ was arbitrary this amounts to our assumed condition. Then $\GMOpen$'s action on both objects and morphisms is well-defined. It preserves the compositional structure because it acts in the same way as $\Open$, and the compositional structure in both $\JSL_f$ and $\SAM_f$ is functional.

\smallskip
\item
Finally consider $\GOpen$ i.e.\ take any $\UG$-morphism $\rR : (V_1,\rE_1) \to (V_2,\rE_2)$. Then this relation defines a $\UGJ$-morphism $\rR : (\rE_1,\rE_1) \to (\rE_2,\rE_2)$, so by (2) we deduce that $\Open\rR$ defines a $\SAJ_f$-morphism of type:
\[
(\Open\rG,\partial_{\rE_1}) 
= (\Open\rG,\partial_{\rE_1}^{\bf-1} \circ \Open id_{\rE_1})
\to 
(\Open\rG,\partial_{\rE_2}^{\bf-1} \circ \Open id_{\rE_2})
= (\Open\rG,\partial_{\rE_2}) 
\]
recalling that $\partial_\rE$ and $\partial_\rE^{\bf-1}$ have the same action whenever $\breve{\rE} = \rE$. Then since each $\partial_{\rE_i}$ is an isomorphism this is actually a $\SAI_f$ morphism by Lemma \ref{lem:interpret_infinite_sai}. As before, functorality follows from that of $\Open$.

\smallskip
\item
It remains to show that the three functors:
\[
\GJPirr : \SAJ_f \to \UGJ
\quad\text{and}\quad
\GMPirr : \SAM_f \to \UGM
\quad\text{and}\quad
\GPirr : \SAI_f \to \UG
\]
are well-defined. Their action on objects:
\[
\begin{tabular}{c}
$\GJPirr (\aQ,\sigma)
:= (\Pirr\aQ,\Pirr\sigma)
\qquad
\GMPirr (\aQ,\sigma)
:= (\Pirr\aQ,\Pirr\sigma)
\qquad
\GPirr (\aQ,\sigma) := (J(\aQ),\Pirr\sigma)$
\end{tabular}
\]
is well-defined by Lemma \ref{lem:self_adjointness_bicliq_vs_jsl}.3.c, recalling that if $\sigma : \aQ \to \aQ^{\pOp}$ then $\Pirr\sigma : \Pirr\aQ \to \Pirr(\aQ^{\pOp}) = (\Pirr\aQ)\spbreve$. So now consider their action on morphisms.

\smallskip
\item
Take any $\SAJ_f$-morphism $f : (\aQ,\sigma_\aQ) \to (\aR,\sigma_\aR)$ and consider $\GJPirr f := \Pirr f$. We know $f(\sigma_\aQ(q)) = \sigma_\aR(f(q))$ for every $q \in Q$, and must establish the equality:
\[
(A) := \qquad
\rR^\up \circ (\Pirr\sigma_\aQ)^\down 
\quad \stackrel{?}{=} \quad
(\Pirr\aR)^\up \circ (\Pirr\sigma_\aR \fatsemi \breve{\rR})^\down
\qquad =: (B)
\]
where we define $\rR := \Pirr f$. We'll achieve this by showing that these two functions have the same action.
\begin{enumerate}[(A)]
\item
Given any subset $X \subseteq J(\aQ)$, we calculate:
\[
\begin{tabular}{lll}
$\rR^\up \circ (\Pirr\sigma_\aQ)^\down(X)$
\\ \quad
$= \rR^\up( \{ j \in J(\aQ) : \Pirr\sigma_\aQ[j] \subseteq X \}$
& by definition of $(-)^\down$
\\ \quad
$= \rR^\up(\{ j : \forall j' \in J(\aQ).[j' \nleq_\aQ \sigma_\aQ(j) \To j' \in X ] \})$
& by definition of $\Pirr\sigma_\aQ$
\\ \quad
$= \rR^\up(\{ j : \forall j' \in J(\aQ).[j' \in \overline{X} \To j' \leq_\aQ \sigma_\aQ(j)  ] \})$
\\ \quad
$= \rR^\up(\{ j : \Lor_\aQ \overline{X} \leq_\aQ \sigma_\aQ(j) \})$
\\ \quad
$= \{ m \in M(\aR) : \exists j \in J(\aQ).[ f(j) \nleq_\aR m \text{ and }  \Lor_\aQ \overline{X} \leq_\aQ \sigma_\aQ(j) ] \}$
& by definition of $\rR^\up$
\\ \quad
$= \{ m \in M(\aR) : \neg\forall j \in J(\aQ).[ \Lor_\aQ \overline{X} \leq_\aQ \sigma_\aQ(j) \To f(j) \leq_\aR m  ] \}$
\\ \quad
$= \{ m \in M(\aR) : \neg\forall j \in J(\aQ).[ j \leq_\aQ \sigma_\aQ(\Lor_\aQ \overline{X}) \To j \leq_\aQ f_*(m)  ] \}$
& take adjoints, $\sigma_\aQ$ self-adjoint 
\\ \quad
$= \{ m \in M(\aR) : \sigma_\aQ(\Lor_\aQ \overline{X}) \nleq_\aQ f_*(m) \}$
\\ \quad
$= \{ m \in M(\aR) : f(\sigma_\aQ(\Lor_\aQ \overline{X})) \nleq_\aR m \}$
& take adjoint
\\ \quad
$= \{ m \in M(\aR) : \sigma_\aR(f(\Lor_\aQ \overline{X})) \nleq_\aR m \}$
& $f$ a $\SAJ$-morphism.
\end{tabular}
\]

\item
Then let us consider the other action:
\[
\begin{tabular}{lll}
$(\Pirr\aR)^\up \circ (\Pirr\sigma_\aR \fatsemi \breve{\rR})^\down(X)$
\\ \quad
$= (\Pirr\aR)^\up \circ (\Pirr\sigma_\aR \fatsemi \Pirr f_*)^\down(X)$
& since $\breve{\rR} = (\Pirr f)\spbreve = \Pirr f_*$
\\ \quad
$= (\Pirr\aR)^\up \circ (\Pirr (f_* \circ \sigma_\aR) )^\down(X)$
& functorality of $\Pirr$
\\ \quad
$= (\Pirr\aR)^\up(\{ j_r \in J(\aR) : \Pirr(f_* \circ \sigma_\aR) [j_r] \subseteq X \})$
& by definition of $(-)^\down$
\\ \quad
$= (\Pirr\aR)^\up(\{ j_r : \forall j_q \in J(\aQ).[f_* \circ \sigma_\aR(j_r) \nleq_{\aQ^{\pOp}} j_q \To j_q \in X  ] \})$
& by definition of $\Pirr$
\\ \quad
$= (\Pirr\aR)^\up(\{ j_r : \forall j_q \in J(\aQ).[j_q \in \overline{X} \To j_q \leq_\aQ f_* \circ \sigma_\aR(j_r) ]  \})$
\\ \quad
$= (\Pirr\aR)^\up(\{ j_r : \Lor_\aQ \overline{X} \leq_\aQ f_* \circ \sigma_\aR(j_r) \})$
\\ \quad
$= (\Pirr\aR)^\up(\{ j_r : \Lor_\aQ \overline{X} \leq_\aQ (\sigma_\aR \circ f)_*(j_r) \})$
& functoriality of adjoints
\\ \quad
$= (\Pirr\aR)^\up(\{ j_r : \sigma_\aR(f(\Lor_\aQ \overline{X})) \leq_{\aR^{\pOp}} j_r \})$
& take adjoint
\\ \quad
$= \{ m \in M(\aR) : \exists j_r \in J(\aR).[ j_r \nleq_\aR m \text{ and } j_r \leq_\aR \sigma_\aR(f(\Lor_\aQ \overline{X})) ] \}$
& by definition of $(\Pirr\aR)^\up$
\\ \quad
$= \{ m \in M(\aR) : \neg\forall j_r \in J(\aR).[ j_r \leq_\aR \sigma_\aR(f(\Lor_\aQ \overline{X})) \To j_r \leq_\aR m  ] \}$
\\ \quad
$= \{ m \in M(\aR) : \sigma_\aR(f(\Lor_\aQ \overline{X})) \nleq_\aR m \}$.
\end{tabular}
\]
\end{enumerate}

Thus $\GJPirr$'s action on objects and morphisms is well-defined. Then it is a well-defined functor because it acts as $\Pirr$ on morphisms, and $\UGJ$ inherits the compositional structure of $\BiCliq$.

\bigskip
\item
Next, take any $\SAM_f$-morphism $f : (\aQ,\sigma_\aQ) \to (\aR,\sigma_\aR)$ and consider $\GMPirr f := \Pirr f$. We know that $f$ preserves the unary operations and must establish:
\[
(A) := \qquad
(\Pirr\sigma_\aR)^\down \circ \rR^\up
\quad\stackrel{?}{=}\quad
(\breve{\rR} \fatsemi \Pirr\sigma_\aQ)^\down \circ (\Pirr\aQ)^\up
\qquad =: (B)
\]
where we define $\rR := \Pirr f$. Then let us simplify their actions.
\begin{enumerate}[(A)]
\item
Given any subset $Y \subseteq M(\aQ)$, we calculate:
\[
\begin{tabular}{lll}
$(\Pirr\sigma_\aR)^\down \circ \rR^\up(Y)$
\\ \quad
$= \{ m \in M(\aR) : \Pirr\sigma_\aR[m] \subseteq \rR[Y] \}$
& by definition of $(-)^\down$
\\ \quad
$= \{ m \in M(\aR) : \forall m' \in M(\aR).\,( \Pirr\sigma_\aR(m,m') \, \To \, m' \in \rR[Y]  ) \}$
\\ \quad
$= \{ m \in M(\aR) : \forall m' \in M(\aR).\,( \sigma_\aR(m) \nleq_\aR m' \To m' \in \Pirr f[Y]  )\}$
& by definition of $\Pirr\sigma_\aR$, $\rR$
\\ \quad
$= \{ m \in M(\aR) : \forall m' \in M(\aR).\,(\forall y \in Y.(f(y) \leq_\aR m')  \To \sigma_\aR(m) \leq_\aR m' )\}$
& by definition of $\Pirr f$
\\ \quad
$= \{ m \in M(\aR) : \forall m' \in M(\aR).\,(\Lor_\aR f[Y] \leq_\aR m'  \To \sigma_\aR(m) \leq_\aR m' )\}$
\\ \quad
$= \{ m \in M(\aR) : \sigma_\aR(m) \leq_\aR \Lor_\aR f[Y] \}$.
\end{tabular}
\]

\item
Let us consider the other action.
\[
\begin{tabular}{lll}
$(\breve{\rR} \fatsemi \Pirr\sigma_\aQ)^\down \circ (\Pirr\aQ)^\up(Y)$
\\ \quad
$= (\Pirr (\sigma_\aQ \circ f_*))^\down(\Pirr\aQ[Y])$
& functorality of $\Pirr$
\\ \quad
$= \{ m \in M(\aQ) : \Pirr(\sigma_\aQ \circ f_*)[m] \subseteq \Pirr\aQ[Y] \}$
& definition of $(-)^\down$
\\ \quad
$= \{ m \in M(\aQ) : \forall m' \in M(\aQ).\,( \Pirr(\sigma_\aQ \circ f_*)(m,m') \To m' \in \Pirr\aQ[Y]  ) \}$
\\ \quad
$= \{ m \in M(\aQ) : \forall m'.\,(\sigma_\aQ \circ f_*(m) \nleq_\aQ m' \To m' \in \Pirr\aQ[Y]) \}$
& definition of $\Pirr$
\\ \quad
$= \{ m \in M(\aQ) : \forall m'.\,(m \nin \Pirr\aQ[Y] \To \sigma_\aQ \circ f_*(m) \leq_\aQ m' ) \}$
\\ \quad
$= \{ m \in M(\aQ) : \forall m'.\,(\forall y \in Y.y \leq_\aQ m' \, \To  \sigma_\aQ \circ f_*(m) \leq_\aQ m') \}$
& definition of $\Pirr\aQ$
\\ \quad
$= \{ m \in M(\aQ) : \forall m'.\,(\Lor_\aQ Y \leq_\aQ m' \To  \sigma_\aQ \circ f_*(m) \leq_\aQ m') \}$
\\ \quad
$= \{ m \in M(\aQ) :  \sigma_\aQ \circ f_*(m) \leq_\aQ \Lor_\aQ Y \}$
\\ \quad
$= \{ m \in M(\aQ) :  (f \circ \sigma_\aQ)_*(m) \leq_\aQ \Lor_\aQ Y \}$
& functoriality of adjoints
\\ \quad
$= \{ m \in M(\aQ) :  \Lor_\aQ Y \leq_{\aQ^{\pOp}} (f \circ \sigma_\aQ)_*(m) \}$
& 
\\ \quad
$= \{ m \in M(\aQ) :  (f \circ \sigma_\aQ)(\Lor_\aQ Y ) \leq_\aR m \}$
& take adjoint
\\ \quad
$= \{ m \in M(\aQ) :  \sigma_\aR (f(\Lor_\aQ Y )) \leq_\aR m \}$
& $f$ a $\SAM$-morphism
\\ \quad
$= \{ m \in M(\aQ) :  \sigma_\aR (m) \leq_\aR f(\Lor_\aQ Y ) \}$
& $\sigma_\aR$ self-adjoint
\\ \quad
$= \{ m \in M(\aQ) :  \sigma_\aR (m) \leq_\aR \Lor_\aR f[Y] \}$.
& f preserves joins.
\end{tabular}
\]
\end{enumerate}

Thus $\GMPirr$'s action on objects and morphisms is well-defined. Then it is a well-defined functor because it acts as $\Pirr$ on morphisms, and $\UGM$ inherits the compositional structure of $\BiCliq$.

\smallskip
\item
Finally we consider the action of $\GPirr$ on $\SAI_f$-morphisms $f : (\aQ,\sigma_\aQ) \to (\aR,\sigma_\aR)$. The two different descriptions of its action are equivalent because:
\[
\begin{tabular}{lll}
$\Pirr (\sigma_\aR \circ f)(j_q,j_r)$
&
$\iff \sigma_\aR \circ f(j_q) \nleq_{\aR^{\pOp}} j_r$
& by definition
\\&
$\iff j_r \nleq_\aR \sigma_\aR \circ f(j_q)$
\\&
$\iff f(j_q) \nleq_\aR \sigma_\aR(j_r)$
& $\sigma_\aR : \aR \to \aR^{\pOp}$ self-adjoint
\\&
$\iff \Pirr f ; \sigma_\aR |_{M(\aR) \times J(\aR)}(j_q,j_r)$
\end{tabular}
\]
where the final step uses the fact that the $\JSL_f$-isomorphism $\sigma : \aR \to \aR^{\pOp}$ restricts to a bijection $\sigma |_{M(\aR) \times J(\aR)}$. Next, since $f$ is also a $\SAJ_f$-morphism and $\GJPirr$ is well-defined by (6), $\Pirr f : (\Pirr\aQ,\Pirr\sigma_\aQ) \to (\Pirr\aR,\sigma_\aR)$ is a well-defined $\UGJ$-morphism. Then using Lemma \ref{lem:sai_induces_ugjm_iso} we have the well-defined $\UGJ$-morphism:
\[
\xymatrix@=15pt{
(\Pirr\sigma_\aQ,\Pirr\sigma_\aQ) \ar[rr]^{\rR_{\sigma_\aQ}} && (\Pirr\aQ,\Pirr\sigma_\aQ) \ar[rr]^{\Pirr f} && (\Pirr\aR,\Pirr\sigma_\aR) \ar[rr]^{(\rR_{\sigma_\aR})^{\bf-1}} && (\Pirr\sigma_\aR,\Pirr\sigma_\aR) 
}
\]
formed by pre/post-composing with $\SAJ_f$-isomorphisms $\rR_{\sigma_\aQ} := \Pirr\aQ$ and $\rR_{\sigma_\aQ}^{\bf-1} := \Pirr\sigma_\aQ$.  This composite is actually $\GPirr f$ by the following calculation:
\[
\begin{tabular}{lll}
$(\rR_{\sigma_\aQ} \fatsemi \Pirr f \fatsemi \rR_{\sigma_\aR}^{\bf-1})^\up$
\\ \quad
$= (\rR_{\sigma_\aR}^{\bf-1})^\up \circ \Pirr\aR^\down \circ (\Pirr f)^\up \circ (\Pirr\aQ)^\down \circ (\rR_{\sigma_\aQ})^\up$
& by $(\up\fatsemi)$ twice
\\ \quad
$= (\Pirr\sigma_\aR)^\up \circ \Pirr\aR^\down \circ (\Pirr f)^\up \circ (\Pirr\aQ)^\down \circ (\Pirr\aQ)^\up$
& by definition
\\ \quad
$= (\Pirr\sigma_\aR)^\up \circ \Pirr\aR^\down \circ (\Pirr f)^\up$
& $\Pirr f : \Pirr\aQ \to (\Pirr\aQ)\spbreve$
\\ \quad
$= (\Pirr f \fatsemi \Pirr\sigma_\aR)^\up$
& by $(\up\fatsemi)$
\\ \quad
$= (\Pirr (\sigma_\aR \circ f))^\up$
& by functorality
\end{tabular}
\]
Then $\GPirr$ is a well-defined functor using the functorality of $\Pirr$, the uniform nature of the isomorphisms $\rR_\sigma$, and the fact that $\UG$ is isomorphic to the diagonal of $\UGJ$.

\end{enumerate}
\end{proof}

\smallskip
Having proved functorality we can now capture previous concepts as natural isomorphisms.
\smallskip

\begin{definition}[Natural isomorphisms involving the diagonals]
\label{def:nat_isos_involving_diagonals}
\item
\begin{enumerate}
\item
\emph{The functors $I_j$, $I_m$, $\Diagj$ and $\Diagm$}.

\smallskip
We have the functors:
\[
\begin{tabular}{lllll}
$I_j : \SAI_f \hookto \SAJ_f$
&&
$I_m : \SAI_f \hookto \SAM_f$
&&
identity on objects and morphisms,
\\[0.5ex]
$\Diagj : \UG \monoto \UGJ$
&&
$\Diagm : \UG \monoto \UGM$
&&
identity on morphisms.
\end{tabular}
\]
That is, $I_j$ and $I_m$ are the full-inclusion functors whereas the action of the full functors $\Diagj$ and $\Diagm$ on objects is $(V,\rE) \mapsto (\rE,\rE)$.

\item
\emph{The image of $(\aQ,\sigma) \in \SAI_f \subseteq \SAJ_f$ under $\GJPirr$ is naturally $\UGJ$-isomorphic to $(\Pirr\sigma,\Pirr\sigma)$}.

\smallskip
We have the natural isomorphism:
\[
\begin{tabular}{lllllll}
$rj : \Diagj \circ \GPirr \To \GJPirr \circ I_j$
&&
$rj_{(\aQ,\sigma)} := \Pirr\aQ : (\Pirr\sigma,\Pirr\sigma) \to (\Pirr\aQ,\Pirr\sigma)$
\\&&
$rj_{(\aQ,\sigma)}^{\bf-1} := \Pirr\sigma : (\Pirr\aQ,\Pirr\sigma) \to (\Pirr\sigma,\Pirr\sigma)$
\end{tabular}
\]
noting that $\sigma$ is viewed as a join-semilattice morphism $\aQ \to \aQ^{\pOp}$ when applying $\Pirr$. \endbox

\takeout{
\item
\emph{The image of $(\aQ,\sigma) \in \SAI_f \subseteq \SAM_f$ under $\GMPirr$ is naturally $\UGM$-isomorphic to $(\Pirr\sigma,\Pirr\sigma)$}.

\smallskip
We include the following natural isomorphism, although we will not actually use it.
\[
\begin{tabular}{lllllll}
$rm : \Diagm \circ \GPirr \To \GMPirr \circ I_m$
&&
$rm_{(\aQ,\sigma)} := \Pirr\aQ : (\Pirr\sigma_j,\Pirr\sigma_j) \to (\Pirr\aQ,\Pirr\sigma_m)$
\\&&
$rm_{(\aQ,\sigma)}^{\bf-1} := \Pirr\sigma_j : (\Pirr\aQ,\Pirr\sigma_m) \to (\Pirr\sigma_j,\Pirr\sigma_j)$
\end{tabular}
\]
where $\sigma_j : \aQ \to \aQ^{\pOp}$ and $\sigma_m : \aQ^{\pOp} \to \aQ$ have underlying function $\sigma$. \endbox
}

\end{enumerate}
\end{definition}

\smallskip

\begin{note}[Concerning a certain asymmetry in our approach]
\item
Just as we have the natural isomorphim $rj : \Diagj \circ \GPirr \To \GJPirr \circ I_j$ there is another   natural isomorphism $rm : \Diagm \circ \GPirr \To \GMPirr \circ I_m$ defined:
\[
rm_{(\aQ,\sigma)} := \Pirr\aQ : (\Pirr\sigma_j,\Pirr\sigma_j) \to (\Pirr\aQ,\Pirr\sigma_m)
\qquad
\text{with inverse $\Pirr\sigma_j$}.
\]
Here $\sigma_j : \aQ \to \aQ^{\pOp}$ and $\sigma_m : \aQ^{\pOp} \to \aQ$ are the two join-semilattice isomorphisms whose underlying function is $\sigma$. Both join-semilattice morphisms are necessary because we choose to view $\SAI_f$-algebras as morphisms $\aQ \to \aQ^{\pOp}$ when applying $\GPirr$, whereas the $\SAM_f$-algebras are necessarily viewed as morphisms $\aQ^{\pOp} \to \aQ$. We will not need to use $rm$ in what follows. \endbox
\end{note}

\begin{lemma}
\label{lem:diagonal_nat_isos_equalities}
\item
\begin{enumerate}
\item
The functors and natural isomorphism $rj : \Diagj \circ \GPirr \To \GJPirr \circ I_j$ from Definition \ref{def:nat_isos_involving_diagonals} are well-defined.
\item
We have the following equalities:
\[
\GJOpen \circ \Diagj = I_j \circ \GOpen
\qquad
\GMOpen \circ \Diagm = I_m \circ \GOpen.
\]
\end{enumerate}
\end{lemma}

\begin{proof}
\item
\begin{enumerate}
\item
The fully-faithful inclusion-functors $I_j$ and $I_m$ are well-defined because $\SAI_f = \SAJ_f \cap \SAM_f$. Recall that the full subcategories of $\UGJ$ and $\UGM$ consisting of objects $(\rE,\rE)$ are actually equal, and also categorically isomorphic to $\UG$ by Lemma \ref{lem:ug_iso_to_diagonals}. It follows that $\Diagj$ and $\Diagm$ are well-defined fully-faithful functors, since they act in the same way as this categorical isomorphism. 

Next, the components $rj_{(\aQ,\sigma)}$ are well-defined $\UGJ$-isomorphisms by Lemma \ref{lem:sai_induces_ugjm_iso} (which also specifies the inverses) where the typing is correct because:
\[
\begin{tabular}{lll}
$\Diagj \circ \GPirr(\aQ,\sigma) = \Diagj(J(\aQ),\Pirr\sigma) = (\Pirr\sigma,\Pirr\sigma)$
\\
$\GJPirr \circ I_j (\aQ,\sigma) = \GJPirr(\aQ,\sigma) = (\Pirr\aQ,\Pirr\sigma)$.
\end{tabular}
\]
Concerning naturality we must verify that for every $\SAJ_f$-morphism $f : (\aQ,\sigma_1) \to (\aR,\sigma_2)$ the following square commutes inside $\UGJ$:
\[
\xymatrix@=15pt{
(\Pirr\sigma_1,\Pirr\sigma_1) \ar[d]_{\Pirr(\sigma_2 \circ f)} \ar[rr]^{rj_{(\aQ,\sigma_1)}} && (\Pirr\aQ,\Pirr\sigma_1) \ar[d]^{\Pirr f}
\\
(\Pirr\sigma_2,\Pirr\sigma_2) \ar[rr]_{rj_{(\aR,\sigma_2)}} && (\Pirr\aR,\Pirr\sigma_2)
}
\]
Firstly $(\Pirr\aQ \fatsemi \Pirr f)^\up = (\Pirr f)^\up \circ (\Pirr\aQ)^\down \circ (\Pirr\aQ)^\up = (\Pirr f)^\up$ and secondly:
\[
\begin{tabular}{lll}
$(\Pirr(\sigma_2 \circ f) \fatsemi \Pirr\aR)^\up$
\\ \quad
$= (\Pirr\aR)^\up \circ (\Pirr\sigma_2)^\down \circ (\Pirr(\sigma_2 \circ f))^\up$
& by $(\fatsemi\up)$
\\ \quad
$= (\Pirr\aR)^\up \circ (\Pirr\sigma_2)^\down \circ (\Pirr f \fatsemi \Pirr\sigma_2)^\up$
& functorality of $\Pirr$
\\ \quad
$= (\Pirr\aR)^\up \circ (\Pirr\sigma_2)^\down \circ (\Pirr \sigma_2)^\up \circ (\Pirr\aR)^\down \circ (\Pirr f)^\up$
& by $(\fatsemi\up)$
\\ \quad
$= (\Pirr\aR)^\up \circ (\Pirr\aR)^\down \circ (\Pirr f)^\up$
& $(\Pirr\aR)^\down = \cl_{\Pirr\sigma_2} \circ (\Pirr\aR)^\down$
\\ \quad
$= (\Pirr f)^\up$
& $(\Pirr f)^\up = \inte_{\Pirr\aR} \circ (\Pirr f)^\up$.
\end{tabular}
\]

\takeout{
\smallskip
Finally we need to prove that $rm : \Diagm \circ \GPirr \To \GMPirr \circ I_m$ is a well-defined natural isomorphism. Firstly, for each finite de morgan algebra $(\aQ,\sigma)$ consider the following $\UGM$-composite,
\[
(\Pirr\sigma_j,\Pirr\sigma_j) \xto{\rR := \Pirr\aQ} (\Pirr\sigma_m,\Pirr\sigma_m) \xto{\rS := \Pirr\sigma_m} (\Pirr\aQ,\Pirr\sigma_m).
\]
To explain, $\rR$ is the $\UGM$-isomorphism arising from the $\UG$-isomorphism in Lemma \ref{lem:ug_iso_arising_from_the_two_typings}, whereas $\rS$ is the $\UGM$-isomorphism described in Lemma \ref{lem:sai_induces_ugjm_iso}. Then since:
\[
(\rR \fatsemi \rS)^\up
= (\Pirr\sigma_m)^\up \circ (\Pirr\sigma_m)^\down \circ (\Pirr\aQ)^\up
= \inte_{\Pirr\sigma_m} \circ (\Pirr\aQ)^\up
= (\Pirr\aQ)^\up
\]
it follows that the $\UGM$-isomorphism $\rR \fatsemi \rS$ has underlying relation $\Pirr\aQ$. All this shows that each component $rm_{(\aQ,\sigma)}$ is a well-defined $\UGM$-isomorphism. Its inverse arises at the level of $\BiCliq$ and is:
\[
\rS^{\bf-1} \fatsemi \rR^{\bf-1}
\;\stackrel{!}{=}\; \Pirr\aQ ; \sigma_{M(\aQ) \times J(\aQ)}
= \Pirr\sigma_j
\]
where the asserted equality follows by inspecting the two Lemmas whence $\rR$ and $\rS$ came. Concerning naturality, given any $\SAI_f$-morphism $f : (\aQ,\sigma_1) \to (\aR,\sigma_2)$ we must establish that the following diagram commutes inside $\UGM$.
\[
\xymatrix@=15pt{
(\Pirr(\sigma_1^j,\Pirr\sigma_1^j) \ar[d]_{\Pirr(\sigma_2^j \circ f)} \ar[rr]^{rm_{(\aQ,\sigma_1)}} && (\Pirr\aQ,\Pirr\sigma_1^m) \ar[d]^{\Pirr f}
\\
(\Pirr\sigma_2^j,\Pirr\sigma_2^j) \ar[rr]_{rm_{(\aR,\sigma_2)}} && (\Pirr\aR,\Pirr\sigma_2^m)
}
\]
This follows because its underlying $\BiCliq$-diagram is precisely the same as the one previously considered when proving that $rj$ is natural.
}

\item
One can directly verify that their action on objects and morphisms are the same, recalling that the underlying function of $\partial_\rE^{\bf-1}$ and $\partial_\rE$ are equal because $\rE$ is symmetric.

\end{enumerate}
\end{proof}

\smallskip
We now finally prove the three categorical equivalences.
\smallskip


\begin{theorem}[Categorical equivalence between $\SAJ_f$ and $\UGJ$]
\label{thm:saj_f_equiv_ugj}
\item
The functors $\GJOpen$ and $\GJPirr$ define an equivalence of categories with respective natural isomorphisms:
\[
\begin{tabular}{lll}
$\jrep : \Id_{\SAJ_f} \To \GJOpen \circ \GJPirr$
& $\jrep_{(\aQ,\sigma)} : (\aQ,\sigma) \to (\Open\Pirr\aQ,\partial_{\Pirr\aQ}^{\bf-1} \circ \Open\Pirr\sigma )$
\\[1ex]&
$\jrep_{(\aQ,\sigma)}(q) := \rep_\aQ(q) = \{ m \in M(\aQ) : q \nleq_\aQ m \}$
\\[1ex]&
$\jrep_{(\aQ,\sigma)}^{\bf-1}(Y) := \rep_\aQ^{\bf-1}(Y) = \Land_\aQ M(\aQ) \backslash Y$.
\\[2ex]
$\jred : \Id_{\UGJ} \To \GJPirr \circ \GJOpen$
& $\jred_{(\rG,\rE)} : (\rG,\rE) \to  (\Pirr\Open\rG,\Pirr(\partial_\rG^{\bf-1} \circ \Open\rE))$
\\[1ex]&
$\jred_{(\rG,\rE)} := \red_\rG = \{ (v,Y) \in V \times M(\Open\rG) :  \rG[v] \nsubseteq Y \}$
\\[1ex]&
$\jred_{(\rG,\rE)}^{\bf-1} := \red_\rG^{\bf-1} = \; \breve{\in} \; \subseteq J(\Open\rG) \times \rG_t$.
\end{tabular}
\]
The associated components of the latter natural isomorphisms follow from Theorem \ref{thm:bicliq_jirr_equivalent}.
\end{theorem}

\begin{proof}
\item
\begin{enumerate}
\item
Regarding $\jrep$, we already know that $\rep : \Id_{\JSL_f} \To \Open\Pirr$ defines a natural isomorphism by Theorem \ref{thm:bicliq_jirr_equivalent}. Then since $\GJOpen\GJPirr f = \Open\Pirr f$ as functions, it suffices to show that for each $(\aQ,\sigma) \in \SAJ_f$,
\[
\rep_\aQ(\sigma(q))
\; \stackrel{?}{=} \; \partial_{\Pirr\aQ}^{\bf-1} \circ \Open\Pirr\sigma(\rep_\aQ(q))
\qquad
\text{for every $q \in Q$},
\]
i.e.\ the respective unary operations are preserved. The LHS equals $\{ m \in M(\aQ) : \sigma(q) \nleq_\aQ m \}$, so let us consider the other side:
\[
\begin{tabular}{lll}
$\partial_{\Pirr\aQ}^{\bf-1} \circ \Open\Pirr\sigma(\rep_\aQ(q))$
\\[0.5ex] \quad
$= \partial_{\Pirr\aQ}^{\bf-1}( (\Pirr\sigma)_+\spbreve[\{ m \in M(\aQ) : q \nleq_\aQ m \}] )$
& by definition of $\Open$
\\ \quad
$= \Pirr\aQ^\up \circ \neg_{J(\aQ)} \circ ((\Pirr\sigma)_+\spbreve)^\up(\{ m \in M(\aQ) : q \nleq_\aQ m \})$
& by definition of $\partial$
\\ \quad
$= \Pirr\aQ^\up \circ ((\Pirr\sigma)_+)^\down(\{ m \in M(\aQ) : q \leq_\aQ m \}) $
& by $(\neg\up\neg)$
\\ \quad
$= \Pirr\aQ^\up(\{ j \in J(\aQ) : (\Pirr\sigma)_+[j] \subseteq \{ m \in M(\aQ) : q \leq_\aQ m \} \})$
& by definition of $(-)^\down$
\\ \quad
$= \Pirr\aQ^\up(\{ j \in J(\aQ) : \forall m \in M(\aQ).[ \sigma(j) \leq_\aQ m \To q \leq_\aQ m ] \})$
& by definition of $(-)_+$
\\ \quad
$= \Pirr\aQ^\up(\{ j \in J(\aQ) : q \leq_\aQ \sigma(j) \})$
& 
\\ \quad
$= \{ m \in M(\aQ) : \exists j \in J(\aQ). [ j \nleq_\aQ m \text{ and } q \leq_\aQ \sigma(j) ] \}$
& by definition of $\Pirr\aQ$
\\ \quad
$= \{ m \in M(\aQ) : \neg\forall j \in J(\aQ).[ q \leq_\aQ \sigma(j) \To j \leq_\aQ m ] \}$
& 
\\ \quad
$= \{ m \in M(\aQ) : \sigma(j) \nleq_\aQ m \}$
\end{tabular}
\]
which completes the proof that $\jrep$ is natural. The description of its inverse is immediate, and also instantiates $\rep^{\bf-1}$ from Theorem \ref{thm:bicliq_jirr_equivalent}.

\item
Concerning $\jred$, we know that $\red : \Id_{\BiCliq} \To \Pirr\Open$ defines a natural isomorphism by Theorem \ref{thm:bicliq_jirr_equivalent}. Since $\GJPirr\GJOpen \rR = \Pirr\Open\rR$ as $\BiCliq$-morphisms it suffices to show that for each $(\rG,\rE)$ we have:
\[
\red_\rG^\up \circ \rE^\down
\; \stackrel{?}{=} \;
(\Pirr\Open\rG)^\up \circ (\Pirr(\partial_\rG^{\bf-1} \circ \Open\rE) \fatsemi \red_\rG\spbreve)^\down.
\]
Regarding the LHS, for every subset $X \subseteq \rG_s$ we have:
\[
\begin{tabular}{lll}
$\red_\rG^\up \circ \rE^\down(X)$
\\ \quad
$= \{ Y \in M(\Open\rG) : \exists g_s \in \rG_s.[ g_s \in \rE^\down(X) \text{ and } \rG[g_s] \nsubseteq Y ] \}$
& by definition of $\red$
\\ \quad
$= \{ Y \in M(\Open\rG) : \exists g_s \in \rG_s.[ g_s \in \rE^\down(X) \text{ and } g_s \nin \rG^\down(Y) ] \}$
& by definition of $(-)^\down$
\\ \quad
$= \{ Y \in M(\Open\rG) : \rE^\down(X) \nsubseteq \rG^\down(Y) \}$
\\ \quad
$= \{ Y \in M(\Open\rG) : \rG^\up \circ \rE^\down(X) \nsubseteq Y \}$.
& by $(\up\vdash\down)$
\end{tabular}
\]
As for the RHS, we first simplify a sub-term:
\[
\begin{tabular}{lll}
$\Pirr(\partial_\rG^{\bf-1} \circ \Open\rE) \fatsemi \red_\rG\spbreve(Y,g_s)$
\\ \quad
$\iff \Pirr(\partial_\rG^{\bf-1} \circ \Open\rE) ; (\red_\rG\spbreve)_+\spbreve (Y,g_s)$
& $\BiCliq$-composition
\\ \quad
$\iff \Pirr(\partial_\rG^{\bf-1} \circ \Open\rE) ; (\red_\rG)_-\spbreve (Y,g_s)$
& $(\rS\spcheck)_+ = \rS_-$ generally
\\ \quad
$\iff \exists Y'.( \partial_\rG^{\bf-1} \circ \Open\rE(Y) \nleq_{(\Open\rG)^{\pOp}} Y' \text{ and } Y' \subseteq \rG[g_s]  )$
& definition of $\Pirr$, $\red_\rG$
\\ \quad
$\iff \exists Y'.( Y' \nsubseteq \partial_\rG^{\bf-1} \circ \Open\rE(Y) \text{ and } Y' \subseteq \rG[g_s] )$
& $\Open\rG$ inclusion-ordered
\\ \quad
$\iff \neg\forall Y'.( Y' \subseteq \rG[g_s] \To Y' \subseteq \partial_\rG^{\bf-1} \circ \Open\rE(Y)   )$
\\ \quad
$\iff \rG[g_s] \nsubseteq \partial_\rG^{\bf-1} \circ \Open\rE(Y)$
\\ \quad
$\iff \Open\rE(Y) \nsubseteq \partial_\rG(\rG[g_s]) $
&  $\partial_\rG$ reverses the ordering
\\ \quad
$\iff \Open\rE(Y) \nsubseteq \inte_{\breve{\rG}}(\overline{g_s})$
& definition of $\partial_\rG$, $(\neg\up\neg)$
\\ \quad
$\iff g_s \in \Open\rE(Y)$
& by Lemma \ref{lem:cl_inte_of_pirr}.1
\end{tabular}
\]
and finally simplify its action:
\[
\begin{tabular}{lll}
$(\Pirr\Open\rG)^\up \circ (\Pirr(\partial_\rG^{\bf-1} \circ \Open\rE) \fatsemi \red_\rG\spbreve)^\down(X)$
\\[0.5ex]
$= (\Pirr\Open\rG)^\up(\{ X' \in J(\Open\rG) : \Open\rE(X') \subseteq X  \})$
& using above
\\
$= \{ Y \in M(\Open\rG) : \exists X'.[ X' \nsubseteq Y \text{ and } \Open\rE(X') \subseteq X ] \}$
& definition of $\Pirr$
\\
$= \{ Y  : \neg\forall X'.[ \Open\rE(X') \subseteq X \To X' \subseteq Y ] \}$
\\
$= \{ Y  : \neg\forall X'.[ X' \subseteq (\Open\rE)_*(X) \To X' \subseteq Y ] \}$
\\
$= \{ Y  : (\Open\rE)_*(X) \nsubseteq Y \}$
\\
$= \{ Y : \rG^\up \circ \rE^\down(X) \nsubseteq Y \}$
& using Lemma \ref{lem:composite_adjoints}.4
\end{tabular}
\]
as required. The description of its inverse follows because $\red_\rG^{\bf-1}$ is the inverse $\BiCliq$-isomorphism by Theorem \ref{thm:bicliq_jirr_equivalent}, and thus is also the inverse $\UGJ$-isomorphism by Lemma \ref{lem:ugj_ugm_iso_correspondence}.

\end{enumerate}
\end{proof}

\smallskip


Next we prove the equivalence $\SAM_f \cong \UGM$, making use of $\SAJ_f \cong \UGJ$. As was the case with the latter equivalence, the respective natural isomorphisms lift directly from the fundamental equivalence $\JSL_f \cong \BiCliq$.

\smallskip

\begin{theorem}[Categorical equivalence between $\SAM_f$ and $\UGM$]
\item
The functors $\GMOpen$ and $\GMPirr$ define an equivalence of categories with witnessing natural isomorphisms:
\[
\begin{tabular}{lll}
$\mrep : \Id_{\SAM_f} \To \GMOpen \circ \GMPirr$
&
$\mrep_{(\aQ,\sigma)} : (\aQ,\sigma) \to (\Open\Pirr\aQ,\Open\Pirr\sigma \circ \partial_{\Pirr\aQ})$
\\[1ex]&
$\mrep_{(\aQ,\sigma)}(q) := \rep_\aQ(q) = \{ m \in M(\aQ) : q \nleq_\aQ m \}$
\\[1ex]&
$\mrep_{(\aQ,\sigma)}^{\bf-1}(Y) := \rep_\aQ^{\bf-1}(Y) = \Land_\aQ M(\aQ) \backslash Y$.
\\[2ex]
$\mred : \Id_{\UGM} \To \GMPirr \circ \GMOpen$
&
$\mred_{(\rG,\rE)} : (\rG,\rE) \to (\Pirr\Open\rG,\Pirr(\Open\rE \circ \partial_\rG ))$
\\[1ex]&
$\mred_{(\rG,\rE)} := \red_\rG := \{ (g_s,Y) \in \rG_s \times M(\Open\rG) : \rG[g_s] \subseteq Y \}$
\\[1ex]&
$\mred_{(\rG,\rE)}^{\bf-1} := \red_\rG^{\bf-1} = \; \breve{\in} \; \subseteq J(\Open\rG) \times V$
\end{tabular}
\]
The associated components of the latter natural isomorphisms follow from Theorem \ref{thm:bicliq_jirr_equivalent}.
\end{theorem}

\begin{proof}
\item
\begin{enumerate}
\item
Fix any $(\aQ,\sigma) \in \SAM_f$, so we have a respective join-semilattice morphism $\sigma : \aQ^{\pOp} \to \aQ$. Then $(\aQ^{\pOp},\sigma) \in \SAJ_f$ so by Theorem \ref{thm:saj_f_equiv_ugj} we have the $\SAJ_f$-isomorphism:
$\jrep_{(\aQ^{\pOp},\sigma)} : (\aQ^{\pOp},\sigma) \to (\Open\Pirr\aQ^{\pOp},\partial_{\Pirr\aQ^{\pOp}}^{\bf-1} \circ \Open\Pirr\sigma)$. We can apply the opposite construction to join-semilattice isomorphisms, yielding the $\SAM_f$-isomorphism:
\[
\jrep_{(\aQ^{\pOp},\sigma)}^{\pOp} : (\aQ,\sigma) \to ((\Open\Pirr\aQ^{\pOp})^{\pOp},\partial_{\Pirr\aQ^{\pOp}}^{\bf-1} \circ \Open\Pirr\sigma).
\]
Now, we are going to compose the above isomorphism with the following $\SAM$-isomorphism:
\[
\partial_{\Pirr\aQ^{\pOp}} : ((\Open\Pirr\aQ^{\pOp})^{\pOp},\partial_{\Pirr\aQ^{\pOp}}^{\bf-1} \circ \Open\Pirr\sigma) \to (\Open\Pirr\aQ, \Open\Pirr\sigma \circ \partial_{\Pirr\aQ}) = \GMOpen\GMPirr(\aQ,\sigma).
\]
Its typing is well-defined because $\GMOpen\GMPirr\aQ \in \SAM_f$ and it is a join-semilattice isomorphism by construction -- see Definition \ref{def:bip_canon_quo_incl}. The unary operation is preserved because:
\[
\begin{tabular}{c}
$\partial_{\Pirr\aQ^{\pOp}}( \partial_{\Pirr\aQ^{\pOp}}^{\bf-1} \circ \Open\Pirr\sigma (Y))
=  \Open\Pirr\sigma(Y)$
\\[1ex]
$\Open\Pirr\sigma \circ \partial_{\Pirr\aQ}( \partial_{\Pirr\aQ^{\pOp}}(Y))
= \Open\Pirr\sigma \circ \partial_{\Pirr\aQ}( \partial_{\Pirr\aQ}^{\bf-1}(Y)) 
= \Open\Pirr\sigma(Y)$
\end{tabular}
\]
where in the second equality we use the fact that $\partial_{\Pirr\aQ^{\pOp}} = \partial_{(\Pirr\aQ)\spbreve}$ acts in the same way as $\partial_{\Pirr\aQ}^{\bf-1}$. Then we define the component $\SAM_f$-isomorphisms as follows:
\[
\mrep_{(\aQ,\sigma)} := \partial_{\Pirr\aQ^{\pOp}} \circ \jrep_{(\aQ^{\pOp},\sigma)}^{\pOp} :
(\aQ,\sigma) \to \GMOpen\GMPirr(\aQ,\sigma)
\]
which actually act in the same way as $\rep_\aQ$:
\[
\begin{tabular}{lll}
$\mrep_{(\aQ,\sigma)}(q)$
&
$= \partial_{\Pirr\aQ^{\pOp}} \circ \jrep_{(\aQ^{\pOp},\sigma)}^{\pOp}(q)$
\\&
$= \partial_{\Pirr\aQ^{\pOp}} \circ \{ j \in J(\aQ) : j \nleq_\aQ q \}$
\\&
$= \Pirr\aQ[\{ j \in J(\aQ) : j \leq_\aQ q \}]$
\\&
$= \{ m \in M(\aQ) : \exists j \in J(\aQ).[ j \leq_\aQ q \text{ and } j \nleq_\aQ m ] \}$
\\&
$= \{ m \in M(\aQ) : \neg\forall j \in J(\aQ).[ j \leq_\aQ q \To j \leq_\aQ m ] \}$
\\&
$= \{ m \in M(\aQ) : q \nleq_\aQ m \}$
\\&
$= \rep_\aQ(q)$.
\end{tabular}
\]
Then since $\rep : \Id_{\JSL_f} \To \Open\Pirr$ is natural and $\SAM_f$ is built on top of $\JSL_f$, it follows that $\mrep : \Id_{\SAM_f} \To \GMOpen\GMPirr$ is a natural isomorphism. The description of the component inverses is immediate.

\item
Fixing $(\rG,\rE) \in \UGM$ we have the self-adjoint $\BiCliq$-morphism $\rE : \breve{\rG} \to \rG$. Then we have $(\breve{\rG},\rE) \in \UGJ$ and thus the $\UGJ$-isomorphism:
\[
\jred_{(\breve{\rG},\rE)}^{\bf-1} : (\Pirr\Open\breve{\rG},\Pirr(\partial_{\breve{\rG}}^{\bf-1} \circ \Open\rE )) \to (\breve{\rG},\rE)
\]
by Theorem \ref{thm:saj_f_equiv_ugj}. Consequently by Lemma \ref{lem:ugj_ugm_iso_correspondence} we have the $\UGM$-isomorphism:
\[
(\jred_{(\breve{\rG},\rE)}^{\bf-1})\spbreve : (\rG,\rE) \to ((\Pirr\Open\breve{\rG})\spbreve,\Pirr(\partial_{\breve{\rG}}^{\bf-1} \circ \Open\rE ))
= (\Pirr(\Open\breve{\rG})^\pOp,\Pirr(\partial_{\breve{\rG}}^{\bf-1} \circ \Open\rE )).
\]
We are going to compose the above isomorphism with the $\UGM$-isomorphism:
\[
\GMPirr \theta
\quad\text{where}\quad
\theta := \partial_{\breve{\rG}}
: ((\Open\breve{\rG})^\pOp,\partial_{\breve{\rG}}^{\bf-1} \circ \Open\rE ) \to (\Open\rG,\Open\rE \circ \partial_\rG).
\]
To see that $\theta$ is a $\SAM_f$-isomorphism, observe that it is a join-semilattice isomorphism by construction. It preserves the unary operation because $\partial_{\breve{\rG}}(\partial_{\breve{\rG}}^{\bf-1} \circ \Open\rE(Y)) = \Open\rE(Y)$ and also:
\[
\begin{tabular}{lll}
$\Open\rE \circ \partial_\rG(\partial_{\breve{\rG}}(Y))$
&
$= \Open\rE \circ \breve{\rG}^\up \circ \neg_V \circ \rG^\up \circ \neg_{\rG_s}(Y)$
& by definition of $\partial_-$
\\&
$= \Open\rE \circ \breve{\rG}^\up \circ \breve{\rG}^\down (Y)$
& by de morgan duality
\\&
$= \Open\rE \circ \inte_{\breve{\rG}} (Y)$
\\&
$= \Open\rE(Y)$
& since $Y \in O(\breve{\rG})$.
\end{tabular}
\]
Then we define the component $\UGM$-isomorphisms as follows:
\[
\mred_{(\rG,\rE)} := (\jred_{(\breve{\rG},\rE)}^{\bf-1})\spbreve \fatsemi \GMPirr \theta : (\rG,\rE) \to \GMPirr\GMOpen(\rG,\rE).
\]
To compute this composite, first observe that $\jred_{(\breve{\rG},\rE)}^{\bf-1} = (\breve{\in})\spbreve = \; \in \; \subseteq V \times J(\Open\breve{\rG})$ and moreover:
\[
(\Pirr\partial_{\breve{\rG}})_+ \subseteq M(\Open\breve{\rG}) \times J(\Open\breve{\rG})
\qquad
(\Pirr\partial_{\breve{\rG}})_+(Y,X) \iff (\partial_{\breve{\rG}})_*(Y) \leq_{(\Open\breve{\rG})^{\pOp}} X \iff X \subseteq \breve{\rG}[\overline{Y}],
\]
using Definition \ref{def:open_pirr} and the fact that $(\partial_{\breve{\rG}})_* = \partial_\rG$ by Lemma \ref{lem:open_dual_iso_adjoints}. We now show that the underlying relation $\mred_{(\rG,\rE)} \subseteq V \times M(\Open\breve{\rG}$ is actually $red_\rG$.
\[
\begin{tabular}{lll}
$\mred_{(\rG,\rE)}(v,Y)$
&
$\iff (\jred_{(\breve{\rG},\rE)}^{\bf-1})\spbreve \fatsemi \GMPirr \theta(v,Y)$
\\&
$\iff \in ; (\Pirr\partial_{\breve{\rG}})_+\spbreve(v,Y)$
\\&
$\iff \exists X \in J(\Open\breve{\rG}).( v \in X \text{ and } X \subseteq \breve{\rG}[\overline{Y}]  )$
\\&
$\iff \neg \forall X \in J(\Open\breve{\rG}.( X \subseteq \breve{\rG}[\overline{Y}] \To v \nin X)$
\\&
$\iff \neg \forall X \in J(\Open\breve{\rG}.( X \subseteq \breve{\rG}[\overline{Y}] \To X \subseteq \inte_{\breve{\rG}}(\overline{v}) )$
& by Lemma \ref{lem:cl_inte_of_pirr}.1
\\&
$\iff \breve{\rG}[\overline{Y}] \nsubseteq \inte_{\breve{\rG}}(\overline{v})$
\\&
$\iff \overline{Y} \nsubseteq \breve{\rG}^\down(\overline{v})$
& by $(\up\dashv\down)$ and $(\down\up\down)$
\\&
$\iff \rG[v] \nsubseteq Y$
& contrapositive, $(\neg\down\neg)$
\\&
$\iff red_\rG(v,Y)$.
\end{tabular}
\]
Then since $\red : \Id_{\BiCliq} \To \Pirr\Open$ is natural and $\UGM$ is built on top of $\BiCliq$, we deduce that $\mred : \Id_{\UGM} \To \GMPirr\GMOpen$ is a natural isomorphism. The description of the component inverses follows because $\red^{\bf-1} = \; \breve{\in}$ are the inverses in $\BiCliq$, and hence also in $\UGM$ by Lemma \ref{lem:ug_reflect_iso}.

\end{enumerate}
\end{proof}

\smallskip
Finally we use the first equivalence $\SAJ_f \cong \UGJ$ to prove the categorical equivalence between finite de morgan algebras and our category of undirected graphs. The construction of the witnessing natural isomorphisms follows quickly, whereas the explicit description of their components requires some computation. Unlike the previous two equivalences, the natural isomorphisms do not lift directly from the fundamental equivalence $\JSL_f \cong \BiCliq$.
\smallskip


\begin{theorem}[Categorical equivalence between $\SAI_f$ and $\UG$]
\label{thm:sai_equivalent_to_ug}
\item
The functors $\GOpen$ and $\GPirr$ define an equivalence of categories with respective natural isomorphisms:
\[
\begin{tabular}{lll}
$\grep : \Id_{\SAI_f} \To \GOpen \circ \GPirr$
& $\grep_{(\aQ,\sigma)} : (\aQ,\sigma) \to (\Open\Pirr\sigma,\partial_{\Pirr\sigma})$
\\[1ex]&
$\grep_{(\aQ,\sigma)}(q) := \{ j \in J(\aQ) : j \nleq_\aQ \sigma(q) \}$
\\[1ex]&
$\grep_{(\aQ,\sigma)}^{\bf-1}(Y) := \sigma(\Lor_\aQ J(\aQ) \backslash Y ) = \Land_\aQ M(\aQ) \backslash \sigma[Y]$.
\\[2ex]
$\gred : \Id_{\UG} \To \GPirr \circ \GOpen$
& $\gred_{(V,\rE)} : (V,\rE) \to  (J(\Open\rE),\Pirr\partial_\rE^{\bf-1})$
\\[1ex]&
$\gred_{(V,\rE)} := \;\in \; \subseteq V \times J(\Open\rE)$
\\[1ex]&
$\gred_{(V,\rE)}^{\bf-1} := \;\breve{\in} \; \subseteq J(\Open\rE) \times V$.
\end{tabular}
\]
Furthermore, the components of the latter natural isomorphisms have the following associated components.
\[
\begin{tabular}{lll}
$(\gred_{(V,\rE)})_- = (\gred_{(V,\rE)}^{\bf-1})_+ = \{ (v,X) \in V \times J(\Open\rE) : X \subseteq \rE[v] \}$,
\\[1ex]
$(\gred_{(V,\rE)})_+ = (\gred_{(V,\rE)}^{\bf-1})_- = \{  (X,v) \in J(\Open\rE) \times V : \rE[v] \subseteq X \}$.
\end{tabular}
\]
\end{theorem}

\begin{proof}
\item
\begin{enumerate}
\item
Recall the inclusion functor $I_j : \SAI_f \to \SAJ_f$, the natural isomorphism $rj^{\bf-1} : \Pirr_j \circ I_j \To \Diagj \circ \GPirr$ from Definition \ref{def:nat_isos_involving_diagonals}, and also the equality $\GJOpen \circ \Diagj = I_j \circ \GOpen$ from Lemma \ref{lem:diagonal_nat_isos_equalities}.2. Using the natural isomorphism $\jrep$ from Theorem \ref{thm:saj_f_equiv_ugj}, we obtain the composite natural isomorphism:
\[
I_j \xTo{\jrep_{I_j -}} \GJOpen \circ \GJPirr \circ I_j
\xTo{\GJOpen rj_-^{\bf-1}} \GJOpen \circ \Diagj \circ \GPirr
= I_j \circ \GOpen \circ \GPirr.
\]
Then since $I_j$ is a fully-faithful inclusion functor we further obtain the natural isomorphism:
\[
\begin{tabular}{c}
$\grep : \Id_{\SAI_f} \To \GOpen \circ \GPirr$
\\[1ex]
$\grep_{(\aQ,\sigma)} := \GJOpen rj_{(\aQ,\sigma)}^{\bf-1} \circ \jrep_{(\aQ,\sigma)} : (\aQ,\sigma) \to (\Open\Pirr\sigma,\partial_{\Pirr\sigma})$.
\end{tabular}
\]
It remains to explain the codomain of the above components, and also simplify their action. For each $(\aQ,\sigma) \in \SAI_f \subseteq \SAJ_f$, the first isomorphism is:
\[
\begin{tabular}{lll}
$\alpha := \jrep_{(\aQ,\sigma)} : (\aQ,\sigma) \to (\Open\Pirr\aQ,\partial_{\Pirr\aQ}^{\bf-1} \circ \Open\Pirr\aQ)$
\\[1ex]
$\alpha(q) := \rep_\aQ(q) = \{ m \in M(\aQ) : q \nleq_\aQ m \}$.
\end{tabular}
\]
Concerning the second isomorphism, let us recall the $\UGJ$-isomorphism $rj_{(\aQ,\sigma)}^{\bf-1} := \Pirr\sigma : (\Pirr\aQ,\Pirr\sigma) \to (\Pirr\sigma,\Pirr\sigma)$. Confusingly, the symmetric relation $\Pirr\sigma \subseteq J(\aQ) \times J(\aQ)$ arises in a number of different ways.
\begin{itemize}
\item
It is the underlying relation of the $\UGJ$-isomorphism $rj_{(\aQ,\sigma)}^{\bf-1}$.
\item
In the domain $(\Pirr\aQ,\Pirr\sigma)$ it is understood as a $\BiCliq$-morphism $\Pirr\sigma : \Pirr\aQ \to (\Pirr\aQ)\spbreve$.
\item
In the codomain $(\Pirr\sigma,\Pirr\sigma)$ it is the bipartite graph in the first component, whereas the second component is understood as the $\BiCliq$-morphism $\Pirr\sigma : \Pirr\sigma \to (\Pirr\sigma)\spbreve$. Recall that since $\Pirr\sigma$ is symmetric this is actually the $\BiCliq$-identity-morphism $id_{\Pirr\sigma}$.
\end{itemize}

Applying the equivalence functor $\GJOpen$ yields the second isomorphism, with typing:
\[
\begin{tabular}{lll}
$\GJOpen rj_{(\aQ,\sigma)}^{\bf-1} : (\Open\Pirr\aQ,\partial_{\Pirr\aQ}^{\bf-1} \circ \Pirr\sigma)$
& $\to$ &
$(\Open\Pirr\sigma,\partial_{\Pirr\sigma}^{\bf-1} \circ \Open id_{\Pirr\sigma})$
\\&&
$=(\Open\Pirr\sigma,\partial_{\Pirr\sigma}^{\bf-1})$
\\&&
$ = (\Open\Pirr\sigma,\partial_{\Pirr\sigma})$
\end{tabular}
\]
where the final equality follows because $\Pirr\sigma$ is symmetric. Concerning its action, since $\GJOpen$ acts as $\Open$ on the underlying $\BiCliq$-morphism we have:
\[
\beta := \GJOpen rj_{(\aQ,\sigma)}^{\bf-1} = \Open\rR
\qquad\text{where}\qquad
\rR := \Pirr\sigma : \Pirr\aQ \to \Pirr\sigma.
\]
Then by definition of $\Open : \BiCliq \to \JSL_f$, for each $Y \in O(\Pirr\aQ)$ we have $\beta(Y) = \rR_+\spbreve[Y]$ where the relation $\rR_+ \subseteq J(\aQ) \times M(\aQ)$ satisfies:
\[
\begin{tabular}{lll}
$\rR_+(j,m)$
&
$\iff (\Pirr\aQ)\spbreve[m] \subseteq \breve{\rR}[j]$
& by definition of $(-)_+$
\\&
$\iff (\Pirr\aQ)\spbreve[m] \subseteq \Pirr\sigma[j]$
& $\rR = \Pirr\sigma$ symmetric
\\&
$\iff \forall j' \in J(\aQ).( j' \nleq_\aQ m \To j' \nleq_\aQ \sigma(j))$
& definition of $\Pirr\aQ$ and $\Pirr\sigma$
\\&
$\iff \forall j' \in J(\aQ).( j' \leq_\aQ \sigma(j) \To j' \leq_\aQ m)$
\\&
$\iff \sigma(j) \leq_\aQ m$.
\end{tabular}
\]
Consequently,
\[
\begin{tabular}{lll}
$\beta(Y)$
&
$= \{ j \in J(\aQ) : \exists m \in Y. \sigma(j) \leq_\aQ m \}$
\\&
$= \{ j \in J(\aQ) : \exists m \in Y. \sigma(m) \leq_\aQ j \}$
& since $\sigma : \aQ \xto{\cong} \aQ^{\pOp}$
\\&
$= J(\aQ) \;\cap \up_\aQ \sigma[Y]$.
\end{tabular}
\]
Then we can finally compute the composite action for each $q_0 \in Q$ as follows,
\[
\begin{tabular}{lll}
$\grep_{(\aQ,\sigma)}(q_0)$
&
$= \beta \circ \alpha(q_0)$
\\&
$= \beta(\{ m \in M(\aQ) : q_0 \nleq_\aQ m \})$
\\&
$= J(\aQ) \;\cap \up_\aQ \{ \sigma(m) : m \in M(\aQ), \, q_0 \nleq_\aQ m \}$
\\&
$= J(\aQ) \;\cap \up_\aQ \{ \sigma(m) : m \in M(\aQ), \, \sigma(m) \nleq_\aQ \sigma(q_0) \}$
& since $\sigma : \aQ \xto{\cong} \aQ^{\pOp}$
\\&
$= J(\aQ) \;\cap \up_\aQ \{ j \in J(\aQ), \, j \nleq_\aQ \sigma(q_0) \}$
& since $\sigma |_{M(\aQ) \times J(\aQ)}$ bijective
\\&
$= \{ j \in J(\aQ), \, j \nleq_\aQ \sigma(q_0) \}$
& since already up-closed.
\end{tabular}
\]
The descriptions of its inverse $\grep_{(\aQ,\sigma)}^{\bf-1}$ follow immediately, recalling that $\sigma$ sends arbitrary joins to arbitrary meets and restricts to a bijection $\sigma |_{J(\aQ) \times M(\aQ)}$.

\item 
Using the natural isomorphism $\jred$ we obtain the composite natural isomorphism:
\[
\Diagj \xTo{\jred_{\Diagj -}} \GJPirr \circ \GJOpen \circ \Diagj
= \GJPirr \circ I_j \circ \GOpen
\xTo{rj_{\GOpen - }^{\bf-1}} \Diagj \circ \GPirr \circ \GOpen
\]
also using the equality $\GJOpen \circ \Diagj = I_j \circ \GOpen$ and the natural isomorphism $rj$, see Lemma \ref{lem:diagonal_nat_isos_equalities}. Recall that $\Diagj : \UG \monoto \UGJ$ is a fully-faithful identity-on-morphisms functor. Its action on objects is $(V,\rE) \mapsto (\rE,\rE)$, with inverse-action $(\rE,\rE) \mapsto (V_\rE,\rE)$ where $V_\rE$ is the source = target of the symmetric relation $\rE$. Consequently we obtain the natural isomorphism:
\[
\begin{tabular}{c}
$\gred : \Id_{\UG} \To \GPirr \circ \GOpen$
\\[0.5ex]
$\gred_{(V,\rE)} := \jred_{(\rE,\rE)} \fatsemi rj_{\GOpen(V,\rE)}^{\bf-1} : (V,\rE) \to (J(\Open\rE),\Pirr\partial_\rE^{\bf-1})$.
\end{tabular}
\]
noting that $V_{\Pirr\partial_\rE^{\bf-1}} = J(\Open\rE)$. Let us simplify its action inside $\UGJ$, so that the $\UG$-morphism $\gred_{(V,\rE)}$ will be the underlying $\BiCliq$-morphism. The component $\UGJ$-isomorphisms are:
\[
\begin{tabular}{lll}
$\rR := \jred_{(\rE,\rE)} = red_\rE : (\rE,\rE) \to (\Pirr\Open\rE,\Pirr(\partial_\rE^{\bf-1} \circ \Open id_\rE)) = (\Pirr\Open\rE,\Pirr\partial_\rE^{\bf-1})$,
\\[1.5ex]
$\rS := rj_{\GOpen(V,\rE)}^{\bf-1} = \Pirr\partial_\rE^{\bf-1} : (\Pirr\Open\rE,\Pirr\partial_\rE^{\bf-1}) \to (\Pirr\partial_\rE^{\bf-1},\Pirr\partial_\rE^{\bf-1})$.
\end{tabular}
\]
Recall the canonical relation $\rR = \red_\rE \subseteq V \times M(\Open\rE)$ and its associated components from Theorem \ref{thm:bicliq_jirr_equivalent}. Furthermore the canonical join-semilattice isomorphism $\partial_\rE^{\bf-1} : \Open\rE \to (\Open\rE)^{\pOp}$ has action $Y \mapsto \rE[\overline{Y}]$. We can now finally show that the underlying relation,
\[
\rR \fatsemi \rS 
\subseteq \rE_s \times (\Pirr\Open\rE)_t
= V \times J(\Open\rE)
\]
is in fact the element-of relation:
\[
\begin{tabular}{lll}
$\rR \fatsemi \rS(v,X)$
&
$\iff (red_\rE)_- ; \rS(v,X)$
\\&
$\iff \exists X' \in J(\Open\rE).( X' \subseteq \rE[v] \text{ and } \Pirr\partial_\rE^{\bf-1}(X',X))$
\\&
$\iff \exists X' \in J(\Open\rE).( X' \subseteq \rE[v] \text{ and } X' \nsubseteq \rE[\overline{X}] )$
\\&
$\iff \neg\forall X' \in J(\Open\rE).( X' \subseteq \rE[v] \To X' \subseteq \rE[\overline{X}] )$
\\&
$\iff \rE[v] \nsubseteq \rE[\overline{X}]$
& since $\rE[v], \, \rE[\overline{X}] \in O(\rE)$
\\&
$\iff v \nin \cl_\rE(\overline{X})$
& by $(\up\dashv\down)$
\\&
$\iff v \in \inte_\rE(X) = X$
& de morgan duality, $\breve{\rE} = \rE$, $X \in O(\rE)$.
\end{tabular}
\]
Concerning the inverses $\gred_{(V,\rE)}^{\bf-1}$, by Lemma \ref{lem:inverse_ug_iso_is_converse} they are precisely the converse-element-of relations $\breve{\in} \; \subseteq J(\Open\rE) \times V$.

\item
It remains to verify our description of the associated components. Since $\gred_{(V,\rE)}^{\bf-1} = (\gred_{(V,\rE)})\spcheck$ as $\BiCliq$-morphisms, the negative/positive component of $\gred_{(V,\rE)}^{\bf-1}$ is actually the positive/negative component of $\gred_{(V,\rE)}$ respectively. Finally,
\[
\begin{tabular}{lll}
$(\gred_{(V,\rE)})_-(v,X)$
&
$\iff \Pirr\partial_\rE^{\bf-1}[X] \subseteq {\in[v]}$
& by definition of $(-)_-$
\\&
$\iff \forall X' \in J(\Open\rE).( X' \nsubseteq \rE[\overline{X}] \To v \in X' )$
& by definition of $\Pirr$ and $\partial_\rE^{\bf-1}$
\\&
$\iff \forall X' \in J(\Open\rE).( X' \nsubseteq \rE[\overline{X}] \To X' \nsubseteq \inte_\rE(\overline{v}) )$
& by Lemma \ref{lem:cl_inte_of_pirr}.1
\\&
$\iff \forall X' \in J(\Open\rE).(  X' \subseteq \inte_\rE(\overline{v}) \To X' \subseteq \rE[\overline{X}] )$
\\&
$\iff \inte_\rE(\overline{v}) \subseteq \rE[\overline{X}]$
& $\inte_\rE(\overline{v})$, $\rE[\overline{X}]$ are $\rE$-open
\\&
$\iff \rE^\down(\overline{v}) \subseteq \cl_\rE(\overline{X})$
& by $(\up\dashv\down)$
\\&
$\iff \inte_\rE(X) \subseteq \rE[v]$
& contrapositive, de morgan duality
\\&
$\iff X \subseteq \rE[v]$
& $X$ is $\rE$-open,
\\
\\
$(\gred_{(V,\rE)})_+(v,X)$
&
$\iff \rE[v] \subseteq \breve{\in}[\{X\}]$
& by definition of $(-)_+$
\\&
$\iff \rE[v] \subseteq X$.
\end{tabular}
\]
\end{enumerate}
\end{proof}

\smallskip
Just as we described the fullness of $\Open$ explicitly in Lemma \ref{lem:open_explicit_fullness}, we now do the same for our new equivalence functors $\GJOpen$, $\GMOpen$ and $\GOpen$. Recall that each of these three functors acts in the same way as $\Open$.
\smallskip

\begin{lemma}[Explicit fullness of $\GJOpen$, $\GMOpen$ and $\GOpen$]
\label{lem:lifted_open_explicit_fullness}
\item
Fix any bipartite graph $\rG$.
\begin{enumerate}
\item
Each $(\Open\rG,\sigma) \in \SAJ_f$ arises as $\GJOpen(\rG,\rE)$ where $\rE(g_s,g'_s) :\iff \rG[g'_s] \nsubseteq \sigma(\rG[g_s])$.

\item
Each $(\Open\rG,\sigma) \in \SAM_f$ arises as $\GMOpen(\rG,\rE)$ where $\rE(g_t,g'_t) :\iff \sigma(\inte_\rG(\overline{g_t})) \nsubseteq \inte_\rG(\overline{g'_t})$.

\item
Each $(\Open\rG,\sigma) \in \SAI_f$ arises as $\GOpen(\rG_t,\rE)$ where $\rE(g_t,g'_t) :\iff \sigma(\inte_\rG(\overline{g_t})) \nsubseteq \inte_\rG(\overline{g'_t})$, as in (2).

\item
Consequently, each $\SAJ_f$, $\SAM_f$ or $\SAI_f$-morphism of type $f : (\Open\rG,\sigma_1) \to (\Open\rH,\sigma_2)$ arises by applying $\GJOpen$, $\GMOpen$ or $\GOpen$ respectively, using the explicit fullness Lemma \ref{lem:open_explicit_fullness} and the above three statements.

\end{enumerate}
\end{lemma}

\begin{proof}
\item
\begin{enumerate}
\item
Given $(\Open\rG,\sigma) \in \SAJ_f$ then consider the join-semilattice morphism:
\[
f := \quad \Open\rG \xto{\sigma} (\Open\rG)^{\pOp} \xto{\partial_\rG} \Open\breve{\rG}.
\]
By the explicit fullness Lemma \ref{lem:open_explicit_fullness}, we have $f = \Open\rE$ where the $\BiCliq$-morphism $\rE : \rG \to \breve{\rG}$ is defined:
\[
\rE(g_s,g'_s) 
:\iff g'_s \in f(\rG[g_s])
\iff g'_s \in \neg_{\rG_s} \circ \rG^\down \circ \sigma(\rG[g_s])
\iff \rG[g'_s] \nsubseteq \sigma(\rG[g_s]).
\]
To see that $(\rG,\rE) \in \UGJ$, observe $\partial_\rG^{\bf-1} \circ \Open\rE = \sigma$ is self-adjoint by assumption, and consequently $\rE : \rG \to \breve{\rG}$ is self-adjoint by Lemma \ref{lem:self_adjointness_bicliq_vs_jsl}.2.

\item
Given $(\Open\rG,\sigma) \in \SAM_f$ then consider the join-semilattice morphism:
\[
f := \quad \Open\breve{\rG} \xto{\partial_\rG^{\bf-1}} (\Open\rG)^{\pOp} \xto{\sigma} \Open\rG.
\]
By the explicit fullness Lemma \ref{lem:open_explicit_fullness} we have $f = \Open\rE$ where the $\BiCliq$-morphism $\rE : \breve{\rG} \to \rG$ is defined:
\[
\rE(g_t,g'_t) 
:\iff g'_t \in f(\breve{\rG}[g_t])
\iff g'_t \in \sigma \circ \rG^\up \circ \neg_{\rG_s} \circ \breve{\rG}[g_t] = \sigma(\inte_\rG(\overline{g_t}))
\iff \sigma(\inte_\rG(\overline{g_t})) \nsubseteq \inte_\rG(\overline{g'_t}).
\]
To see that $(\rG,\rE) \in \UGM$, observe that $\Open\rE \circ \partial_\rG = \sigma$ is self-adjoint by assumption, and consequently $\rE : \breve{\rG} \to \rG$ is self-adjoint by Lemma \ref{lem:self_adjointness_bicliq_vs_jsl}.2.

\item
Given $(\Open\rG,\sigma) \in \SAI_f \subseteq \SAM_f$ then by (2) we have the self-adjoint $\BiCliq$-morphism $\rE : \breve{\rG} \to \rG$ defined:
\[
\rE(g_t,g'_t) 
:\iff \sigma(\inte_\rG(\overline{g_t})) \nsubseteq \inte_\rG(\overline{g'_t})
\iff g'_t \in \sigma(\inte_\rG(\overline{g_t})).
\]
Then given the u-graph $(\rG_t,\rE)$ we'll show that $\GOpen(\rG_t,\rE) = (\Open\rE,\partial_\rE) \stackrel{!}{=} (\Open\rG,\sigma)$.
\begin{itemize}
\item
To establish $\Open(\rE \subseteq \rG_t \times \rG_t) = \Open\rG$, observe $\forall X \subseteq \rG_t. \rE[X] \in O(\rG)$  because $\rE : \breve{\rG} \to \rG$ is a $\BiCliq$-morphism. We have $\rE[g_t] = \sigma(\inte_\rG(\overline{g_t}))$, and every meet-irreducible in $\Open\rG$ takes the form $\inte_\rG(\overline{g_t})$. Since $\sigma : \Open\rG \to (\Open\rG)^{\pOp}$ is an isomorphism, every join-irreducible in $\Open\rG$ arises as some $\rE[g_t]$.

\item
To see that $\partial_\rE = \sigma$ as functions, by the previous item we have the typing $\partial_\rE : (\Open\rG)^{\pOp} \to \Open\rG$ recalling that $\breve{\rE} = \rE$. Then let us calculate:
\[
\begin{tabular}{lll}
$\partial_\rE(\inte_\rG(\overline{g_t}))$
&
$= \rE^\up \circ \neg_{\rG_t}(\inte_\rG(\overline{g_t}))$
\\&
$= \rE[\cl_{\breve{\rG}}(\{g_t\}))]$
& by De Morgan duality
\\&
$= \rE[g_t]$
& since $\rE : \breve{\rG} \to \rG$ a $\BiCliq$-morphism
\\&
$= \sigma(\inte_\rG(\overline{g_t}))$.
\end{tabular}
\]
Consequently $\partial_\rE$ and the isomorphism $\sigma : \Open\rG \to (\Open\rG)^{\pOp}$ have the same action on meet-irreducible elements of $\Open\rG$, hence they have the same action on all elements.

\end{itemize}

\item
Follows because $\GJOpen$, $\GMOpen$ and $\GOpen$ inherit their action from $\Open$.
\end{enumerate}
\end{proof}

\subsection{Various interesting results}

Recall the notion of \emph{tight morphism} i.e.\ Definition \ref{def:tight_jsl_mor}.

\begin{theorem}[Characterisation of self-adjoint tight morphisms]
  \item
  For each $\JSL_f$-morphism $\sigma : \aQ \to \aQ^{\pOp}$ the following statements are equivalent.
  \begin{enumerate}[a.]
  \item
  $\sigma$ is self-adjoint and tight.
  
  \item
  $\sigma$ arises as a join (= pointwise-join) of special morphisms:
  \[
  \sigma
  = \Lor_{\JSL_f[\aQ,\aQ^{\pOp}]} \; \{ \up_{\aQ,\aQ^{\pOp}}^{q_0,q_1} \; : \rR(q_0,q_1) \}
  \]
  where $\rR \subseteq Q \times Q$ is a symmetric relation.
  
  \item
  $(\aQ,\sigma)$ is a $\SAJ$-algebra and `factorises through' a boolean $\SAI_f$-algebra i.e.\
  \[
  \sigma = \;
  \aQ \xto{\alpha} \JPow Z \underset{\cong}{\xto{\sigma_0}} (\JPow Z)^{\pOp} \xto{\alpha_*} \aQ^{\pOp}
  \]
  for some $\JSL_f$-morphism $\alpha : \aQ \to \JPow Z$ and $\SAI_f$-algebra $(\JPow Z, \sigma_0)$.
  
  \item
  $(\aQ,\sigma)$ is a $\SAJ$-algebra and `factorises through' a distributive $\SAI_f$-algebra i.e.\
  \[
  \sigma
  = \;
  \aQ \xto{\alpha} \Open \leq_\pP \; \underset{\cong}{\xto{\sigma_0}} (\Open \leq_\pP)^{\pOp} \xto{\alpha_*} \aQ^{\pOp}
  \]
  for some $\JSL_f$-morphism $\alpha : \aQ \to \Open \leq_\pP$, finite poset $(P,\leq_\pP)$ and $\SAI_f$-algebra $(\Open \leq_\pP, \sigma_0)$.
  
  
  \end{enumerate}
  
  \end{theorem}
  
  \begin{proof}
  \item
  \begin{itemize}
  \item[--]
  $(a \To b)$: $\sigma$ is tight so by Lemma \ref{lem:tight_mor_char},
  \[
  \sigma = \Lor_{\JSL_f[\aQ,\aQ^{\pOp}]} \; \{ \up_{\aQ,\aQ^{\pOp}}^{q_0,q_1} \; : \; \rR(q_0,q_1)  \}
  \qquad
  \text{where $\rR := \{ (q_0,q_1) : \; \up_{\aQ,\aQ^{\pOp}}^{q_0,q_1} \; \leq \sigma \}$.}
  \]
  Finally if $\up_{\aQ,\aQ^{\pOp}}^{q_0,q_1} \; \leq \sigma$ then applying adjoints:
  \[
  \up_{\aQ,\aQ^{\pOp}}^{q_1,q_0} 
  \; = \; (\up_{\aQ,\aQ^{\pOp}}^{q_0,q_1})_* \; \leq \sigma_* = \sigma
  \]
  by Lemma \ref{lem:special_jsl_morphisms}.1 and the self-adjointness of $\sigma$. Thus $\rR$ is a symmetric relation.
  
  \item[--]
  $(b \To a)$: It is tight because joins of these special morphisms are tight by Lemma \ref{lem:tight_mor_char}. It is self-adjoint because adjoints preserve joins and $(\up_{\aQ,\aQ^{\pOp}}^{q_0,q_1})_* \; = \; \up_{\aQ,\aQ^{\pOp}}^{q_1,q_0}$.
  
  \item[--]
  $(a \To c)$: Since $(a \To b)$ we know $\sigma = \Lor \{ \up_{\aQ,\aQ^{\pOp}}^{q_0,q_1} : \rR(q_0,q_1) \}$ for some symmetric relation $\rR \subseteq Q \times Q$. Also, $\sigma : \aQ \to \aQ^{\pOp}$ is self-adjoint and hence defines a finite $\SAJ$-algebra $(\aQ,\sigma)$ (see Lemma \ref{lem:interpret_finite_saj_sam}). It suffices to establish that the following diagram commutes:
  \[
  \xymatrix@=15pt{
  \aQ \ar[d]_-{\alpha} \ar[rr]^-\sigma && \aQ^{\pOp}
  \\
  \JPow \rR \ar[rr]_-{\sigma_0}^-{\cong} && (\JPow \rR)^{\pOp} \ar[u]_-{\alpha_*}
  }
  \]
  where:
  \begin{enumerate}[1.]
  \item
  $\alpha := \Lor_{\JSL_f[\aQ, \JPow \rR]} \{ \up_{\aQ,\JPow \rR}^{q_0,\{(q_0,q_1)\}} \; : \rR(q_0,q_1)  \}$ is a join of special morphisms.
  \item
  $\sigma_0$ is the composite isomorphism:
  \[
  \JPow \rR \xto{\theta^\up} \JPow \rR \xto{\neg_\rR} (\JPow \rR)^{\pOp}
  \]
  where the involutive function $\rR \xto{\theta} \rR$ is defined $\theta(q_0,q_1) := (q_1,q_0)$, so that $\theta^\up (S) := \breve{S}$ can be viewed as constructing converse relations. Observe that it defines a finite $\SAI$-algebra $(\JPow \rR,\sigma_0)$ by Example \ref{ex:char_sa_jmi_fin_boolean}.

  \item[--]
  The adjoint of $\alpha$ has description:
  \[
  \alpha_* = \Lor_{\JSL_f[(\JPow \rR)^{\pOp}, \aQ^{\pOp} ]} \{ \up_{(\JPow \rR)^{\pOp},\aQ^{\pOp}}^{\{(q_0,q_1)\},q_0} \; : \rR(q_0,q_1)  \}
  \]
  by Lemma \ref{lem:special_jsl_morphisms}.1 and Lemma \ref{lem:jsl_mor_iso}.
  
\end{enumerate}
  
To see that the diagram commutes, first observe:
  \[
  \begin{tabular}{lll}
  $\sigma_0 \circ \alpha$
  &
  $= \sigma_0 \circ (\Lor_{\JSL_f[\aQ, \JPow \rR]} \{ \up_{\aQ,\JPow \rR}^{q_0,\{(q_0,q_1)\}} \; : \rR(q_0,q_1) \} )$
  \\[1ex]&
  $= \Lor_{\JSL_f[\aQ, (\JPow \rR)^{\pOp}]} \{ \up_{\aQ,(\JPow \rR)^{\pOp}}^{q_0,\sigma_0(\{(q_0,q_1)\})} \; : \rR(q_0,q_1)  \}$
  & by Lemma \ref{lem:compose_spec_gen_mor}.1
  \\[1ex]&
  $= \Lor_{\JSL_f[\aQ, (\JPow \rR)^{\pOp}]} \{ \up_{\aQ,(\JPow \rR)^{\pOp}}^{q_0,\overline{(q_1,q_0)}} \; : \rR(q_0,q_1)  \}$
  & by definition.
  \end{tabular}
  \]
  Then to see:
  \[
  \sigma = \qquad
  \Lor_{\JSL_f[\aQ, (\JPow \rR)^{\pOp}]} \{ \up_{\aQ,(\JPow \rR)^{\pOp}}^{q_0,\overline{(q_1,q_0)}} \; : \rR(q_0,q_1)  \}
  \circ
  \Lor_{\JSL_f[(\JPow \rR)^{\pOp}, \aQ^{\pOp} ]} \{ \up_{(\JPow \rR)^{\pOp},\aQ^{\pOp}}^{\{(q_0,q_1)\},q_0} \; : \rR(q_0,q_1)  \}
  \]
  observe that the joins distribute over composition by bilinearity, the relevant internal compositions being:
  \[
  \begin{tabular}{lll}
  $\up_{\aQ,(\JPow \rR)^{\pOp}}^{q_0,\overline{(q_1,q_0)}} \circ \up_{(\JPow \rR)^{\pOp},\aQ^{\pOp}}^{\{(q'_0,q'_1)\},q'_0}$
  &
  $= \; \up_{\aQ,\aQ^{\pOp}}^{q_0,\up_{(\JPow \rR)^{\pOp},\aQ^{\pOp}}^{\{(q'_0,q'_1)\},q'_0}(\overline{(q_1,q_0)})}$
  & by Lemma \ref{lem:compose_spec_gen_mor}.1
  \\&
  $=
  \begin{cases}
  \up_{\aQ,\aQ^{\pOp}}^{q_0,q'_0}
  & \text{if $\overline{(q_1,q_0)} \nleq_{(\JPow \rR)^{\pOp}} \{(q'_0,q'_1) \} $}
  \\
  \up_{\aQ,\aQ^{\pOp}}^{q_0,\top_\aQ}
  & \text{otherwise}
  \end{cases}$
  \\[3ex]&
  $=
  \begin{cases}
  \up_{\aQ,\aQ^{\pOp}}^{q_0,q'_0}
  & \text{if $(q'_0,q'_1) = (q_1,q_0)$}
  \\[1ex]
  \bot_{\JSL_f[\aQ,\aQ^{\pOp}]}
  & \text{otherwise}.
  \end{cases}$
  \end{tabular}
  \]
  By the symmetry of $\rR$ we deduce that their join over all $\rR(q_0,q_1)$ is indeed $\sigma$.
  
  \item[--]
  $(c \To a)$: If $(\aQ,\sigma)$ is a finite $\SAJ$-algebra then $\sigma : \aQ \to \aQ^{\pOp}$ is self-adjoint by Lemma \ref{lem:interpret_finite_saj_sam}. Furthermore $\sigma$ is tight because $\aQ$ is boolean.
  
  \item[--]
  $(c \To d)$: Immediate because boolean join-semilattices are distributive. In particular we can choose the discrete poset $\pP := (Z,\Delta_Z)$.

  \item
  $(d \To c)$: Fix the respective finite poset $(P,\leq_\pP)$ and consider the following diagram inside $\JSL_f$.
  \[
  \xymatrix@=15pt{
  \aQ \ar[d]_\alpha \ar[rr]^{\sigma} && \aQ^{\pOp} \ar@{<-}[d]^{\alpha_*}
  \\
  \Open \leq_\pP \ar[d]_\iota \ar[rr]_-{\sigma_0}^-\cong && (\Open\leq_\pP)^{\pOp} \ar@{<-}[d]^{\iota_*}
  \\
  \JPow P   \ar[rr]^-{\sigma_1 \; := \; \neg_P \, \circ \, \theta^\up } && (\JPow P)^{\pOp} 
  \\
  & \JPow P \ar[ur]_-{\neg_P}^-\cong \ar@{<-}[ul]^-{\theta^\up}_-\cong
  }
  \]
  The top square commutes by assumption. By Theorem \ref{thm:char_distrib_saj_sam_sai}.4  $\sigma_0$ has action $\lambda Y \in O(\leq_\pP). \neg_P \circ \theta^\up(Y)$ for some involutive function $\theta : P \to P$. Then the triangle makes sense and commutes by construction. Certainly $\sigma_1$ is a join-semilattice isomorphism. To see that it is involutive, first recall that for any bijection $f : P \to P$ we have $f^\up \circ \neg_P = \neg_P \circ f^\up$.\footnote{Although easily directly verified, this also follows because $f^\up : \JPow P \to \JPow P$ is a join-semilattice isomorphism, thus a bounded lattice isomorphism, and thus a boolean algebra isomorphism.} Then:
  \[
  \sigma_1 \circ \sigma_1
  = \neg_P \circ \theta \circ \neg_P \circ \theta
  = \neg_P \circ \neg_P \circ \theta \circ \theta
  = id_{P}
  \]
  since both $\neg_P$ and $\theta$ are involutive. Thus by Lemma \ref{lem:char_ug_jm} we deduce that $(\JPow P,\sigma_1)$ is a well-defined $\SAI$-algebra. It remains to show that the central square commutes. That is, for every $Y \in O(\leq_\pP)$ we must show that:
  \[
  \sigma_0(Y) 
  \; \stackrel{?}{=} \; 
  \iota_* \circ \neg_P \circ \theta^\up \circ \iota(Y)
  \]
  or equivalently that $\iota_*(\sigma_0(Y)) = \sigma_0(Y)$ as we now show:
  \[
  \begin{tabular}{lll}
  $\iota_*(\sigma_0(Y))$
  &
  $= \bigcup \{ X \in O(\leq_\pP) : \sigma_0(X) \subseteq \sigma_0(Y) \}$
  & by definition of adjoints
  \\&
  $= \bigcup \{ X \in O(\leq_\pP) : X \subseteq Y \}$
  & since $\sigma_0$ an isomorphism
  \\&
  $= Y$
  & since $Y \in O(\leq_\pP)$.
  \end{tabular}
  \]
  
  
  \end{itemize}
\end{proof}

\section{Appendix}
\label{appendix:appendix}

Consider the following standard categories:
\[
\begin{tabular}{lll}
{\bf category} & {\bf objects} & {\bf morphisms}
\\ \hline
$\Set_f$ & finite sets & functions
\\
$\Poset_f$ & finite posets & + preserves order
\\
$\JSL_f$ & finite join-semilattices with bottom & + preserve all joins
\\
$\DL_f$ & finite distributive lattices & + preserve all meets
\\
$\BA_f$ & finite boolean algebras & + preserve negation
\\ \hline
\\[-1.6ex]
$\DL_f^\lor$ & finite distributive lattices & functions preserving all joins
\\
$\BA_f^\lor$ & finite boolean algebras & functions preserving all joins
\end{tabular}
\]
where composition is functional composition in each case. Each such category is equivalent to a possibly non-full subcategory of $\JSL_f$. We now describe many dualities, including Birkhoff's between posets and distributive lattices.
\[
\begin{tabular}{ccc}
\bf{duality} & \bf{functors} & \bf{natural isomorphisms}
\\ \hline
\\[-1.8ex]
\xymatrix@=10pt{
\Set_f^{op} \ar@/^10pt/[rr]^{\Pred}  \ar@{<-}@/_10pt/[rr]_{\At^{op}} && \BA_f
}
&
\begin{tabular}{c}
$\Pred X := \BPow X$ \\
$\Pred f^{op} := \lambda S \subseteq Y. f^{-1}(S)$ \\
$\At \aA := At(\aA)$ \\
$At f^{op} := \lambda b \in At(\aB). \Land_\aA f^{-1}(\up_\aB b)$
\end{tabular}
&
\begin{tabular}{c}
$\alpha : \Id_{\Set_f} \To \At \circ \Pred^{op}$
\\
$\alpha_X : X \to At(\BPow X)$ \quad $\alpha_X (x) := \; \{x\}$
\\[0.5ex]
$\beta : \Id_{\BA_f} \To \Pred \circ \At^{op}$
\\
$\beta_\aA : \aA \to \BPow At(\aA)$
\quad
$\beta_\aA(a) := At(\aA) \; \cap \down_\aA  a$
\end{tabular}
\\ \hline
\\[-1.8ex]
\xymatrix@=10pt{
\Poset_f^{op} \ar@/^10pt/[rr]^{\Up}  \ar@{<-}@/_10pt/[rr]_{\Ji^{op}} && \DL_f
}
&
\begin{tabular}{c}
$\Up \pP := (Up(\pP),\cup,\emptyset,\cap,P)$
\\
$\Up f^{op} := \lambda X.f^{-1}(X)$
\\
$\Ji \aD := (J(\aD),\leq_{\aD^{\pOp}})$
\\
$\Ji f^{op} := \lambda j \in J(\aE).\Land_\aD f^{-1}(\up_\aE j)$
\end{tabular}
&
\begin{tabular}{c}
$\alpha : \Id_{\Poset_f} \To \Ji \circ \Up^{op}$
\\
$\alpha_\pP : \pP \to (Pr_\up(\pP),\supseteq)$
\quad
$\alpha_\pP (p) := \; \up_\pP p$
\\
$\beta : \Id_{\DL_f} \To \Up \circ \Ji^{op}$
\\
$\beta_\aD : \aD \to   (Dn(J(\aD),\leq_\aD),\cup,\emptyset,\cap,D)$
\\
$\beta_\aD(d) := J(\aD) \; \cap \down_\aD d$
\end{tabular}
\\ \hline
\\[-1.8ex]
\xymatrix@=10pt{
\JSL_f^{op} \ar@/^10pt/[rr]^{\OD_j}  \ar@{<-}@/_10pt/[rr]_{\OD_j^{op}} && \JSL_f
}
&
\begin{tabular}{c}
$\OD_j \aQ := \aQ^{\pOp}$
\\
$\OD_j f := \lambda r \in R.\Lor_\aQ f^{-1}(\down_\aR r)$
\end{tabular}
&
\begin{tabular}{c}
$\alpha : \Id_{\JSL_f} \To \OD_j \circ \OD_j^{op}$
\\
$\alpha_\aQ = id_\aQ$
\end{tabular}
\\ \hline
\\[-1.8ex]
\xymatrix@=10pt{
(\DL_f^\lor)^{op} \ar@{-<}`u[]`[rr]^-{\OD_d}[rr]
&& \DL_f^\lor \ar`d[]`[ll]^-{\OD_d^{op}}[ll]
}
& $\OD_d$ restricts $\OD_j$
&
\begin{tabular}{c}
$\alpha : \Id_{\DL_f^\lor} \To \OD_d \circ \OD_d^{op}$
\\
$\alpha_\aD = id_\aD$
\end{tabular}
\\ \hline
\\[-1.8ex]
\xymatrix@=10pt{
(\BA_f^\lor)^{op} \ar@{-<}`u[]`[rr]^-{\OD_b}[rr]
&& \BA_f^\lor \ar`d[]`[ll]^-{\OD_b^{op}}[ll]
}
& $\OD_b$ restricts $\OD_d$
& 
\begin{tabular}{c}
$\alpha : \Id_{\BA_f^\lor} \To \OD_b \circ \OD_b^{op}$
\\
$\alpha_\aA = id_\aA$
\end{tabular}
\end{tabular}
\]
The third entry is the self-duality of finite join-semilattices proved in Theorem \ref{thm:jsl_self_dual}. The fourth and fifth entries follow because distributive lattices and boolean algebras are stable under order-dualisation, see Definition \ref{def:std_order_theory}.12. The first entry is the well-known duality between finite boolean algebras and finite sets, which restricts Birkhoff's famous duality between \emph{finite distributive lattices with bounded lattice morphisms} and \emph{finite posets with monotone functions}. Lemma \ref{lem:birkhoff_funct_welldef} below proves that $\Up$ and $\Ji$ are well-defined functors, Theorem \ref{thm:birkhoff_duality} proves Birkhoff duality and Theorem \ref{thm:bool_fin_duality} restricts it to boolean algebras and finite sets.





Concerning these five dualities, exactly seven categories are mentioned at the beginning of this subsection. Modulo categorical equivalence these seven categories are closed under taking the formal dual category. They are also related to one another via the free constructions:
\[
\xymatrix@=15pt{
\Set_f \ar@<8pt>[rr]^{\SetToPos} \ar@{<-}@<-8pt>[rr]_{\PosToSet} 
& \bot & \Poset_f \ar@<8pt>[rr]^{\PosToJSL} \ar@{<-}@<-8pt>[rr]_{\JSLToPos}
& \bot & \JSL_f \ar@<8pt>[rr]^{\JSLToDL} \ar@{<-}@<-8pt>[rr]_{\DLToJSL}
& \bot & \DL_f \ar@<8pt>[rr]^{\DLToBA} \ar@{<-}@<-8pt>[rr]_{\BAToDL}
& \bot & \BA_f
}
\]

\[
\begin{tabular}{ccccc}
\bf{free construction} & \bf{functors} &  \bf{natural transformations} 
\\ \hline
\\[-1.8ex]
\xymatrix@=10pt{
\Set_f \ar@/^10pt/[rr]^{\SetToPos}  \ar@{<-}@/_10pt/[rr]_{\PosToSet} & \bot & \Poset_f
}
&
\begin{tabular}{c}
$\SetToPos X := (X,\Delta_X)$ and $\SetToPos f := f$
\\
$\PosToSet \pP := P$ and $\PosToSet f := f$
\end{tabular}
&
\begin{tabular}{c}
$\eta : \Id_{\Set_f} \To \PosToSet \SetToPos$ where $\eta_X = id_X$
\\
$\epsilon : \SetToPos \PosToSet \To \Id_{\Poset_f}$
\\
$\epsilon_\pP : (P,\Delta_P) \to \pP$ where $\eta_\pP(p) := p$
\end{tabular}
\\ \hline
&& \emph{proved in Lemma \ref{lem:free_poset_on_set}}
\\ \hline
\\[-1.8ex]
\xymatrix@=10pt{
\Poset_f \ar@/^10pt/[rr]^{\PosToJSL}  \ar@{<-}@/_10pt/[rr]_{\JSLToPos} & \bot & \JSL_f
}
&
\begin{tabular}{c}
$\PosToJSL \pP := (Dn(\pP),\cup,\emptyset)$
\\
$\PosToJSL f := \lambda X. \down_\pQ f[X] $
\\
$\JSLToPos \aQ := (Q,\leq_\aQ)$ and $\JSLToPos f := f$
\end{tabular}
&
\begin{tabular}{c}
$\eta : \Id_{\Poset_f} \To \JSLToPos \PosToJSL$ 
\\
$\eta_\pP : \pP \to (Dn(\pP),\subseteq)$ where $\eta_\pP(p) := \; \down_\pP p$
\\
$\epsilon : \PosToJSL \JSLToPos \To  \Id_{\JSL_f}$
\\
$\epsilon_\aQ : (Dn(\aQ),\cup,\emptyset) \to \aQ$ where $\epsilon_\aQ(S) := \; \Lor_\aQ S$
\end{tabular}
\\ \hline
&& \emph{proved in  Theorem \ref{thm:free_jsl_on_poset}}
\\ \hline
\\[-1.8ex]
\xymatrix@=10pt{
\JSL_f \ar@/^10pt/[rr]^{\JSLToDL}  \ar@{<-}@/_10pt/[rr]_{\DLToJSL} & \bot & \DL_f
}
&
\begin{tabular}{c}
$\JSLToDL \aQ := (Dn(\aQ),\cup,\emptyset,\cap,Q)$
\\
$\JSLToDL f := \lambda X.(f_*)^{-1}(X)$
\\
$\DLToJSL \aD := (\aD,\lor_\aD,\bot_\aD)$ and $\DLToJSL f := f$
\end{tabular}
&
\begin{tabular}{c}
$\eta : \Id_{\JSL_f} \To \DLToJSL \JSLToDL$
\\
$\eta_\aQ : \aQ \to (Dn(\aQ),\cup,\emptyset)$ where $\eta_\aQ(q) := \; \overline{\up_\aQ q}$
\\
$\epsilon : \JSLToDL \DLToJSL \To \Id_{\DL_f}$
\\
$\epsilon_\aD : (Dn(\aD),\cup,\emptyset,\cap,D) \to \aD$
\\
where $\epsilon_\aD(S) := \; \Land_\aD \overline{S} \cap M(\aD)$
\end{tabular}
\\ \hline
&& \emph{proved in Theorem \ref{thm_free_dl_on_jsl}}
\\ \hline
\\[-1.8ex]
\xymatrix@=10pt{
\DL_f \ar@/^10pt/[rr]^{\DLToBA}  \ar@{<-}@/_10pt/[rr]_{\BAToDL} & \bot & \BA_f
}
&
\begin{tabular}{c}
$\DLToBA \aD := \BPow J(\aD)$
\\
$\DLToBA f := \lambda X.(U_{dm} f)_*^{-1}(X))$
\\
$\BAToDL \aA := (A,\lor_\aA,\bot_\aA,\land_\aA,\top_\aA)$ 
\\
$\BAToDL f := f$
\end{tabular}
&
\begin{tabular}{c}
$\eta : \Id_{\DL_f} \To \BAToDL \DLToBA$ \quad $\eta_\aD : \aD \to \DPow J(\aD)$
\\
where $\eta_\aD (d) := J(\aD) \; \cap \down_\aD d$
\\
$\epsilon : \DLToBA \BAToDL \To \Id_{\BA_f}$
\\
$\epsilon_\aA : \BPow J(\aD) \to \aA$ where $\epsilon_\aA (S) := \Lor_\aS S$
\end{tabular}
\\ \hline
&& \emph{proved in Theorem \ref{thm:free_ba_on_dl}}
\\ \hline
\end{tabular}
\]

\begin{note}
Recall that for any finite set $X$:
\[
\DPow X := (\Pow X,\cup,\emptyset,\cap,X) \in \DL_f
\qquad
\BPow X := (\Pow X,\cup,\emptyset,\cap,X,\neg_X) \in \BA_f
\]
Concerning the action of $\DLToBA$ on morphisms, first observe that there are two natural forgetful functors from $\DL_f$ to $\JSL_f$. The functor $\DLToJSL$ forgets the binary meet and top, whereas:
\[
U_{dm} : \DL_f \to \JSL_f
\qquad
U_{dm} \aD := (D,\land_\aD,\top_\aD)
\qquad
U_{dm} f := f
\]
forgets the binary join and bottom. This is important when one considers the adjoint i.e.\ given a $\DL_f$-morphism $f : \aD \to \aE$ then:
\[
\begin{tabular}{ll}
$\DLToBA f (X) $
& $= (U_{dm} f)_*^{-1}(X)$
\\&
$= \{ j \in J(\aE) : (U_{dm} f)_*(j) \in X \}$
\\&
$= \{ j \in J(\aE) : \Lor_{\aD^{\pOp}} f^{-1}(\down_{\aE^{\pOp}} j ) \in X \}$
\\&
$= \{ j \in J(\aE) : \Land_\aD f^{-1}(\up_\aE j ) \in X \}$
\end{tabular}
\]
this being a more explicit description. \endbox
\end{note}

\subsection{Birkhoff duality and its restriction to boolean algebras}

\begin{definition}[Equivalence functors between $\Poset_f^{op}$ and $\DL_f$]
\[
\begin{tabular}{lll}
$\Up : \Poset_f^{op} \to \DL_f$
&
$\Up \pP := (Up(\pP),\cup,\emptyset,\cap,P)$
&
$\dfrac{f : \pP \to \pQ}{\Up f^{op} := \lambda X.f^{-1}(X): \Up\pQ \to \Up\pP}$
\\[2ex]
$\Ji : \DL_f^{op} \to \Poset_f$
&
$\Ji \aD := (J(\aD),\leq_{\aD^{\pOp}})$
&
$\dfrac{f : \aD \to \aE}{\Ji f^{op} := \lambda j.\Land_\aD f^{-1}(\up_\aE j) : \Ji\aE \to \Ji\aD}$
\end{tabular}
\]
\end{definition}

\begin{lemma}
\label{lem:birkhoff_funct_welldef}
$\Up : \Poset_f^{op} \to \DL_f$ and $\Ji : \DL_f^{op} \to \Poset_f$ are well-defined functors.
\end{lemma}

\begin{proof}
$\Up\pP$ is a set-theoretic distributive lattice and the restricted preimage function $\Up f^{op}$ preserves all unions and intersections. It preserves the compositional structure because the preimage functor does i.e.\ $(g \circ f)^{-1} = f^{-1} \circ g^{-1}$. 

$\Ji\aD$ is clearly a well-defined poset so take any $\DL_f$-morphism $f : \aD \to \aE$. Then for $\Ji f^{op}$ to be a well-defined function we need to show that $\Land_\aD f^{-1}(\up_\aE j) \in J(\aD)$ whenever $j \in J(\aE)$. By Lemma \ref{lem:std_order_theory}.12 it suffices to show that $f^{-1}(\up_\aE j) \subseteq D$ arises as $\theta^{-1}(\{1\})$ for some $\DL_f$-morphism $\aD \to \two$. Then since $d \in f^{-1}(\up_\aE j)$ iff $j \leq_\aE f(d)$ we consider $\theta := \lambda d \in D.(j \leq_\aE f(d)) \;?\; 1 : 0$ as follows.
\begin{enumerate}
\item
$\theta(\bot_\aD) = 0$ because $f(\bot_\aD) = \bot_\aE$ and $j$ is join-irreducible by assumption.
\item
$\theta(\top_\aD) = 1$ because $f(\top_\aD) = \top_\aE$.
\item
$\theta(d_1 \land_\aD d_2) = 1$ iff $j \leq_\aE f(d_1) \land_\aD f(d_2)$ iff $\theta(d_1) = 1$ and $\theta(d_2) = 1$.
\item
$\theta(d_1 \lor_\aD d_2) = 1$ iff $j \leq_\aE f(d_1) \lor_\aD f(d_2)$, iff  $\theta(d_1) = 1$ or $\theta(d_2) = 1$ by Lemma \ref{lem:std_order_theory}.10.
\end{enumerate}

Next, $\Ji f^{op}$ is monotonic because:
\[
j_1 \leq_{\Ji\aE} j_2
\implies j_2 \leq_\aE j_1
\implies \up_\aE j_1 \subseteq \; \up_\aE j_2
\implies f^{-1}(\up_\aE j_1) \subseteq f^{-1}(\up_\aE j_2)
\implies \Land_\aD f^{-1}(\up_\aE j_2) \leq_{\aD} \Land_\aD f^{-1}(\up_\aE j_1)
\]
and thus $\Ji f^{op}(j_1) \leq_{\Ji\aD} \Ji f^{op}(j_2)$ recalling that $\Ji\aD$ restricts $\aD^{\pOp}$. Regarding the compositional structure:
\[
\Ji \; id_\aD^{op} = \lambda j.\Land_\aD id_\aD^{-1}(\up_\aD j) = \lambda j.\Land_\aD \up_\aD j = \lambda j.j = id_{\Ji\aD}
\]
\[
\begin{tabular}{lll}
$\Ji (g \circ f)^{op}$
&
$= \lambda j.\Land_\aD (g \circ f)^{-1}(\up_\aF j)$
\\&
$= \lambda j.\Land_\aD f^{-1} \circ g^{-1}(\up_\aF j)$
\\&
$= \lambda j.\Land_\aD f^{-1}(\up_\aE \Land_\aE g^{-1}(\up_\aF j))$
& see below
\\&
$= \lambda j.\Land_\aD f^{-1}(\up_\aE \Ji g^{op}(j))$
\\&
$= \lambda j.\Ji f^{op} \circ \Ji g^{op} (j)$
\\&
$= \Ji f^{op} \circ \Ji g^{op}$
\end{tabular}
\]
The marked equality holds because $g^{-1}(\up_\aF j)$ is an upset one-generated by its $\aE$-meet, see the argument further above.
\end{proof}


\begin{theorem}[Birkhoff Duality]
\label{thm:birkhoff_duality}
\item
$\Up$ and $\Ji^{op}$ define an equivalence between $\Poset_f^{op}$ and $\DL_f$ with  natural isomorphisms:
\[
\begin{tabular}{lll}
$\alpha : \Id_{\Poset_f} \To \Ji \circ \Up^{op}$
&
$\alpha_\pP : \pP \to (Pr_\up(\pP),\supseteq)$
&
$\alpha_\pP (p) := \; \up_\pP p$
\\
$\beta : \Id_{\DL_f} \To \Up \circ \Ji^{op}$
&
$\beta_\aD : \aD \to (Dn(J(\aD),\leq_\aD),\cup,\emptyset,\cap,D)$
&
$\beta_\aD(d) := \{ j \in J(\aD) : j \leq_\aD d \}$
\end{tabular}
\]
\end{theorem}

\begin{proof}
Observe that $\Ji \circ \Up^{op} \pP = \Ji(Up(\pP),\cup,\emptyset,\cap,P)$ is the collection of $\pP$-principal-upsets $Pr_\up(\pP)$ ordered by reverse inclusion. Then $\alpha_\pP$ is the well-known poset isomorphism sending $p$ to its principal upset $\up_\pP p$, recalling that $p_1 \leq_\pP p_2$ iff  $\up_\pP p_2 \subseteq \;\up_\pP p_1$. Regarding naturality, we must verify that the square:
\[
\xymatrix@=15pt{
\pP \ar[d]_{\alpha_\pP} \ar[rr]^f && \pQ \ar[d]^{\alpha_\pQ}
\\
(Pr_\up(\pP),\supseteq) \ar[rr]_{\Ji \circ \Up^{op} f} && (Pr_\up(\pQ),\supseteq)
}
\]
commutes for all monotone maps $f : \pP \to \pQ$. To this end, let $g := \Up f^{op} : \Up \pQ \to \Up \pP$ recalling its action $g(Y) = f^{-1}(Y)$. Then we calculate:
\[
\begin{tabular}{lll}
$\Ji \circ \Up^{op}f (\up_\pP p)$
&
$= \Ji g^{op} (\up_\pP p)$
\\&
$= \Land_{\Up\aD} g^{-1}(\up_{\Up(\pP)} \up_\pP p)$
\\&
$= \bigcap g^{-1}(\{ X \in Up(\pP) : p \in X \})$
\\&
$= \bigcap \{ Y \in Up(\pQ) : p \in g(Y) \}$
\\&
$= \bigcap \{ Y \in Up(\pQ) : p \in f^{-1}(Y) \}$
\\&
$= \bigcap \{ Y \in Up(\pQ) : f(p) \in Y \}$
\\&
$= \; \up_\pQ f(p)$
\end{tabular}
\]
for any $p \in P$, which proves naturality.

\smallskip
Next, $\beta_\aD$ is well-typed because:
\[
\Up \circ \Ji^{op} \aD 
= \Up(J(\aD),\leq_{\aD^{\pOp}})
= (Up(J(\aD),\leq_{\aD^{\pOp}}),\cup,\emptyset,\cap,J(\aD))
= (Dn(J(\aD),\leq_{\aD}),\cup,\emptyset,\cap,J(\aD))
\]
Thus its action is well-defined. $\beta$ is injective because each element of a finite join-semilattice (or distributive lattice) is the join of those join-irreducibles beneath it, and thus is uniquely determined by them. For $\beta$ to be surjective we must show that distinct down-closed sets of join-irreducibles yield distinct elements. This follows by applying Lemma \ref{lem:std_order_theory}.10. That is, if $\Lor_\aD X = \Lor_\aD Y$ where $X$, $Y \in Dn(J(\aD))$ then for each $j \in X$ we have $j \leq_\aD \Lor_\aD Y$ and hence $\exists j' \in Y. j \leq_\aD j'$, and thus $j \in Y$ by downwards-closure. Then $X \subseteq Y$ and by the symmetric argument $X = Y$, so that $\beta$ is bijective. It is a bounded distributive lattice morphism i.e.\ preservation of bottom, top and binary meet follow easily, whereas preservation of binary join follows by Lemma \ref{lem:std_order_theory}.10. Concerning naturality, we must show the following square commutes:
\[
\xymatrix@=15pt{
\aD \ar[d]_{\beta_\aD} \ar[rr]^f && \aE \ar[d]^{\beta_\aE}
\\
(Dn(J(\aD),\leq_{\aD}),\cup,\emptyset,\cap,J(\aD)) \ar[rr]_{\Up \circ \Ji^{op} f} && (Dn(J(\aE),\leq_{\aE}),\cup,\emptyset,\cap,J(\aE))
}
\]
for every bounded distributive lattice morphism $f : \aD \to \aE$. If we let $X = \beta_\aD(d) = \down_\aD d \cap J(\aD)$, then:
\[
\begin{tabular}{lll}
$\Up \circ \Ji^{op} f(X)$
&
$= \Up (\Ji f^{op})^{op}(X)$
\\&
$= \{ j \in J(\aE) : \Ji f^{op}(j) \in X \}$
\\&
$= \{ j \in J(\aE) : \Land_\aD f^{-1}(\up_\aE j) \in \; \down_\aD d \cap J(\aD) \}$
\\&
$= \{ j \in J(\aE): \Land_\aD f^{-1}(\up_\aE j) \leq_\aD d \}$
\end{tabular}
\]
whereas $\beta_\aE \circ f(d) = \{ j \in J(\aE) : j \leq_\aE f(d) \}$. Thus it suffices to show that:
\[
j \leq_\aE f(d) \iff
\Land_\aD f^{-1}(\up_\aE j) \leq_\aD d
\]
for all $j \in J(\aE)$ and $d \in D$. This is actually an instance of an adjoint relationship inside $\JSL_f$. That is, given $f : \aD \to \aE$ then we have the underlying join-semilattice morphism $U_{dm} f : (\aD,\land_\aD,\top_\aD) \to (\aE,\land_\aE,\top_\aE)$ i.e.\ restrict to the underlying meet structure. Then observing that:
\[
\Land_\aD f^{-1}(\up_\aE j)
= \Lor_{\aD^{\pOp}} f^{-1}(\down_{\aE^{\pOp}} j)
= (U_{dm} f)_*(j)
\]
we may instantiate Lemma \ref{lem:adj_obs}.1 to obtain:
\[
j \leq_\aE f(d)
\iff
U_{dm} f(d) \leq_{\aE^{\pOp}} j
\stackrel{!}{\iff}
d \leq_{\aD^{\pOp}} (U_{dm}f)_*(j)
\iff
\Land_\aD f^{-1}(\up_\aE j) \leq_\aD d
\]
which completes the proof.
\end{proof}

\begin{definition}[Equivalence functors between $\Set_f^{op}$ and $\BA_f$]
\[
\begin{tabular}{lll}
$\Pred : \Set_f^{op} \to \BA_f$
&
$\Pred X := \BPow X$
&
$\dfrac{f : X \to Y}{\Pred f^{op} := \lambda X.f^{-1}(X): \BPow X \to \BPow Y}$
\\[2ex]
$\At : \BA_f^{op} \to \Set_f$
&
$\At \aB := At(\aB)$
&
$\dfrac{f : \aB \to \aC}{\At f^{op} := \lambda a.\Land_\aC f^{-1}(\up_\aB a) : At(\aC) \to At(\aB)}$
\end{tabular}
\]
\end{definition}

\begin{theorem}[Duality between finite boolean algebras and finite sets]
\label{thm:bool_fin_duality}
\item
$\Pred$ and $\At^{op}$ define an equivalence between $\Set_f^{op}$ and $\BA_f$ with  natural isomorphisms:
\[
\begin{tabular}{lll}
$\alpha : \Id_{\Set_f} \To \At \circ \Pred^{op}$
&
$\alpha_X : X \to At(\BPow X)$
&
$\alpha_X (x) := \; \{x\}$
\\
$\beta : \Id_{\BA_f} \To \Pred \circ \At^{op}$
&
$\beta_\aB : \aB \to \BPow At(\aB)$
&
$\beta_\aB(b) := \{ a \in At(\aB) : a \leq_\aB b \}$
\end{tabular}
\]
\end{theorem}

\begin{proof}
This follows by restricting Theorem \ref{thm:birkhoff_duality} i.e.\ Birkhoff duality. That is, we have the commuting diagram:
\[
\xymatrix@=15pt{
\Poset_f^{op} \ar[rr]^{\Up} && \DL_f \ar[rr]^{\Ji^{op}} && \Poset_f^{op}
\\
\Set_f^{op} \ar@{>->}[u]^{I^{op}} \ar[rr]_{\Pred} && \BA_f \ar@{>->}[u]^{\BAToDL} \ar[rr]_{\At^{op}} && \Set_f^{op} \ar@{>->}[u]_{I^{op}}
}
\]
where:
\begin{enumerate}
\item
$I : \Set_f \monoto \Poset_f$ is the fully faithful functor defined $IX = (X,=_X)$ and $If = f$.
\item
$\BAToDL : \BA_f \monoto \DL_f$ is the fully faithful forgetful functor.
\end{enumerate}
Certainly $I$ is fully faithful because the monotone maps from a discrete poset $(X,=_X)$ to a discrete poset $(Y,=_Y)$ are precisely the functions $f : X \to Y$, and clearly $\BAToDL$ is faithful. To see that $\BAToDL$ is \emph{full} recall that bounded distributive lattice morphisms between boolean algebras are boolean algebra morphisms, since by Lemma \ref{lem:std_order_theory}.9 complements in distributive lattices are unique whenever they exist.

That the diagram above commutes is easily verified i.e.\ observe that the definitions of $\Pred$ and $\Up$ align, as do $\At$ and $\Ji$. Then $\alpha$ and $\beta$ are natural isomorphisms because they restrict the corresponding natural isomorphisms witnessing Birkhoff duality.
\end{proof}


\subsection{Free constructions between sets, posets, join-semilattices, distributive lattices and boolean algebras}


\[
\xymatrix@=15pt{
\Set_f \ar@<8pt>[rr]^{\SetToPos} \ar@{<-}@<-8pt>[rr]_{\PosToSet} 
& \bot & \Poset_f \ar@<8pt>[rr]^{\PosToJSL} \ar@{<-}@<-8pt>[rr]_{\JSLToPos}
& \bot & \JSL_f \ar@<8pt>[rr]^{\JSLToDL} \ar@{<-}@<-8pt>[rr]_{\DLToJSL}
& \bot & \DL_f \ar@<8pt>[rr]^{\DLToBA} \ar@{<-}@<-8pt>[rr]_{\BAToDL}
& \bot & \BA_f
}
\]

\begin{definition}[Free poset on a set]
Let $\PosToSet : \Poset_f \to \Set_f$ be the forgetful functor which forgets the ordering i.e.\ $\PosToSet \pP := P$ and $\PosToSet f := f$. Further define:
\[
\SetToPos : \Set_f \to \Poset_f
\qquad
\SetToPos X := (X,\Delta_X)
\qquad
\dfrac{f : X \to Y}{\SetToPos f := \lambda x.f(x) : (X,\Delta_X) \to (Y,\Delta_X)}
\]
\end{definition}

\begin{lemma}[Free poset on a set]
\label{lem:free_poset_on_set}
\item
$\SetToPos : \Set_f \to \Poset_f$ is left adjoint to the forgetful functor $\PosToSet : \Poset_f \to \Set_f$ via natural transformations:
\[
\begin{tabular}{lll}
$\eta : \Id_{\Set_f} \To \PosToSet \circ \SetToPos$
&
$\eta_X : X \to X$
&
$\eta_X(p) := x$
\\
$\epsilon : \SetToPos \circ \PosToSet \To \Id_{\Poset_f}$
&
$\epsilon_\pP : (P,\Delta_P) \to P$
&
$\epsilon_\pP(p) := p$
\end{tabular}
\]
\end{lemma}

\begin{proof}
Each $\eta_X$ is a well-defined function and each $\epsilon_\pP$ is a well-defined monotone function. Although they are both bijective, $\epsilon_\pP$ is not a $\Poset_f$-isomorphism whenever $|P| \geq 2$. Naturality is obvious by inspecting the required commutative squares:
\[
\xymatrix@=15pt{
X \ar[d]_{\eta_X} \ar[rr]^f && Y \ar[d]^{\eta_Y}
\\
X \ar[rr]_f && Y
}
\qquad
\xymatrix@=15pt{
(P,\Delta_P) \ar[d]_{\epsilon_\pP} \ar[rr]^g && (Q,\Delta_Q) \ar[d]^{\epsilon_\pQ}
\\
\pP \ar[rr]_g && \pQ
}
\]
for all functions $f : X \to Y$ and monotone functions $g : \pP \to \pQ$. Finally, the counit-unit equations are also immediate:
\[
\epsilon_{\SetToPos X} \circ \SetToPos \eta_X = \lambda x.x = id_{\SetToPos X}
\qquad
\PosToSet \epsilon_\pP \circ \eta_{\PosToSet \pP} = \lambda p.p = id_{\PosToSet\pP}
\]
\end{proof}

\begin{definition}[Free join-semilattice on a poset]
\label{def:free_jsl_on_poset}
Let $\JSLToPos : \JSL_f \to \Poset_f$ be the forgetful functor which takes the underlying ordering i.e.\ $\JSLToPos\aQ = (Q,\leq_\aQ)$ and $\JSLToPos f := f$. Furthermore define:
\[
\PosToJSL : \Poset_f \to \JSL_f
\qquad
\PosToJSL \,\pP := (Dn(\pP),\cup,\emptyset)
\qquad
\dfrac{f : \pP \to \pQ}{\PosToJSL f := \lambda X.\down_{\pQ} f[X] : \PosToJSL \pP \to \PosToJSL \pQ}
\]
\end{definition}

\begin{theorem}[Free join-semilattice on a poset]
\label{thm:free_jsl_on_poset}
\item
$\PosToJSL : \Poset_f \to \JSL_f$ is left adjoint to the forgetful functor $\JSLToPos : \JSL_f \to \Poset_f$ via natural transformations:
\[
\begin{tabular}{lll}
$\eta : \Id_{\Poset_f} \To \JSLToPos \circ \PosToJSL$
&
$\eta_\pP : \pP \to (Dn(\aQ),\subseteq)$
&
$\eta_\pP(p) := \; \down_\pP p$
\\
$\epsilon : \PosToJSL \circ \JSLToPos \To \Id_{\JSL_f}$
&
$\epsilon_\aQ : (Dn(\aQ),\cup,\emptyset) \to \aQ$
&
$\epsilon_\aQ(S) := \Lor_\aQ S$
\end{tabular}
\]
\end{theorem}

\begin{proof}
We first verify that $\PosToJSL$ is a well-defined functor. Its action on objects is well-defined because the downsets $\pP$ contain $\emptyset$ and are union-closed. Concerning its action on morphisms:
\[
\PosToJSL f(\bot_{\PosToJSL\pP})
= \;\down_{\pQ} f[\bot_{\PosToJSL\pP}] = \; \down_{\pQ} f[\emptyset] = \; \down_{\pQ} \emptyset = \emptyset
\]
\[
\PosToJSL f(A_1 \lor_{\PosToJSL\pP} A_2)
= \; \down_{\pQ} f[A_1 \cup A_2]
= \; \down_{\pQ} f[A_1] \; \cup \; \down_\pQ f[A_2]
= \PosToJSL f(A_1) \lor_{\PosToJSL\pQ} \PosToJSL f(A_2)
\]
Each $\eta_\pP$ is monotone because $p \leq_\pP q$ implies $\down_\pP p \subseteq \; \down_\pP q$. Concerning naturality we must verify that:
\[
\xymatrix@=15pt{
\pP \ar[rr]^f \ar[d]_{\eta_\pP} && \pQ \ar[d]^{\eta_\pQ}
\\
(Dn\pP,\subseteq) \ar[rr]_{\JSLToPos\PosToJSL f} && (Dn\pQ,\subseteq)
}
\]
i.e.\ $\down_\pQ f[\down_\pP p] = \; \down_\pQ f(p)$ which follows by the monotonicity of $f$. Each $\epsilon_\aQ$ is a join-semilattice morphism:
\[
\epsilon_\aQ(\emptyset) = \Lor_\aQ \emptyset = \bot_\aQ
\qquad
\epsilon_\aQ(S_1 \cup S_2) = \Lor_\aQ S_1 \cup S_2 = \Lor_\aQ S_1 \lor_\aQ \Lor_\aQ S_2 = \epsilon_\aQ(S_1) \lor_\aQ \epsilon_\aQ(S_2)
\]
and for naturality we must verify that:
\[
\xymatrix@=15pt{
(Dn(Q,\leq_\aQ),\cup,\emptyset) \ar[rr]^{\PosToJSL \JSLToPos f} \ar[d]_{\epsilon_\aQ} &&  (Dn(R,\leq_\aR),\cup,\emptyset) \ar[d]^{\epsilon_\aR}
\\
\aQ \ar[rr]_f && \aR
}
\]
i.e.\ $f(\Lor_\aQ S) = \Lor_\aR \down_\aR f[S]$ which follows because (i) $f$ preserves arbitrary joins, (ii) adding smaller elements has no effect. Then it only remains to verify the counit-unit equations:
\[
\begin{tabular}{lll}
$\epsilon_{\PosToJSL \pP} \circ \PosToJSL \eta_{\pP} (A)$
&
$= \epsilon_{\PosToJSL\pP}( \down_{\PosToJSL \JSLToPos \PosToJSL\pP} \{ \down_\pP p : p \in A \}) $
\\&
$= \Lor_{\PosToJSL\pP} \down_{\PosToJSL \JSLToPos \PosToJSL\pP} \{ \down_\pP p : p \in A \}$
\\&
$= \bigcup \{ S \in Dn(\pP) : \exists p \in A. S \subseteq \; \down_\pP p \}$
\\&
$= \bigcup \{ \down_\pP p : p \in A \}$
\\&
$= A$
& since $A$ downclosed
\end{tabular}
\]
\[
\JSLToPos \epsilon_\aQ \circ \eta_{\JSLToPos\aQ} (q)
= \epsilon_\aQ (\down_\aQ q)
= \Lor_\aQ \down_\aQ q
= q
\]
\end{proof}

We are now going to describe the free distributive lattice on a finite join-semilattice. Let us first define the relevant functor $\JSLToDL$.

\begin{definition}[Free distributive lattice on a join-semilattice]
Let $\DLToJSL : \DL_f \to \JSL_f$ be the forgetful functor which takes the underlying join-semilattice structure i.e.\ $\DLToJSL \aD := (D,\lor_\aD,\bot_\aD)$ and $\DLToJSL f := f$. Further define:
\[
\JSLToDL : \JSL_f \to \DL_f
\qquad
\JSLToDL \aQ := (Dn(\aQ),\cup,\emptyset,\cap,Q)
\qquad
\dfrac{f : \aQ \to \aR}{\JSLToDL f := \lambda X. (f_*)^{-1}(X) : \JSLToDL\aQ \to \JSLToDL\aR }
\]
\end{definition}

\begin{lemma}
$\JSLToDL$ equals the composite functor:
\[
\JSL_f \xto{\OD_j^{op}} \JSL_f^{op} \xto{\JSLToPos^{op}} \Poset_f^{op} \xto{\Up} \DL_f
\]
and is thus a well-defined functor.
\end{lemma}

\begin{proof}
We have:
\[
\begin{tabular}{lll}
$\Up \circ \JSLToPos^{op} \circ \OD_j^{op} \aQ$
&
$= \Up \circ \JSLToPos^{op}(\aQ^{\pOp})$
\\&
$= \Up (Q,\geq_\aQ)$
\\&
$= (Up(Q,\geq_\aQ),\cup,\emptyset,\cap,Q)$
\\&
$= (Dn(Q,\leq_\aQ),\cup,\emptyset,\cap,Q)$
\\&
$= \JSLToDL \aQ$
\end{tabular}
\]
and furthermore $\Up \circ \JSLToPos^{op} \circ \OD_j^{op} f = \Up f_* = (f_*)^{-1}$ with domain $\JSLToDL \aQ$ and codomain $\JSLToDL \aR$.
\end{proof}

\begin{theorem}[Free distributive lattice on a join-semilattice]
\label{thm_free_dl_on_jsl}
\item
$\JSLToDL : \JSL_f \to \DL_f$ is left adjoint to the forgetful functor $\DLToJSL : \DL_f \to \JSL_f$ with associated natural transformations:
\[
\begin{tabular}{lll}
$\eta : \Id_{\JSL_f} \To \DLToJSL \circ \JSLToDL$
&
$\eta_\aQ : \aQ \to (Dn(\aQ),\cup,\emptyset)$
&
$\eta_\aQ(q) := \overline{\up_\aQ q}$
\\
$\epsilon : \JSLToDL \circ \DLToJSL \To \Id_{\DL_f}$
&
$\epsilon_\aD : (Dn(\aD),\cup,\emptyset,\cap,D) \to \aD$
&
$\epsilon_\aD(S) := \; \Land_\aD \overline{S} \cap M(\aD)$
\end{tabular}
\]
\end{theorem}

\begin{proof}
Each $\eta_\aQ$ is a well-defined join-semilattice morphism because:
\[
\eta_\aQ(\bot_\aQ) = \overline{Q} = \emptyset = \bot_{\DLToJSL\JSLToDL\aQ} 
\]
\[
\begin{tabular}{ll}
$\eta_\aQ(q_1 \lor_\aQ q_2)$
&
$= \overline{\up_\aQ (q_1 \lor_\aQ q_2)}$
\\&
$= \overline{\{ q \in Q : q_1 \leq_\aQ q \text{ and } q_2 \leq_\aQ q \}}$
\\&
$= \{ q \in Q: q_1 \nleq_\aQ q \text{ or } q_2 \nleq_\aQ q \}$
\\&
$= \overline{\up_\aQ q_1} \cup \overline{\up_\aQ q_2}$
\\&
$= \eta_\aQ(q_1) \lor_{\JSLToDL\aQ} \eta_\aQ(q_2)$
\end{tabular}
\]
For naturality we must show that:
\[
\xymatrix@=15pt{
\aQ \ar[rr]^f \ar[d]_{\eta_\aQ} && \aR \ar[d]^{\eta_\aR}
\\
(Dn(\aQ),\cup,\emptyset) \ar[rr]_{\DLToJSL \JSLToDL f} && (Dn(\aR),\cup,\emptyset)
}
\]
commutes for all join-semilattice morphisms $f : \aQ \to \aR$. Observing that $\eta_\aR \circ f(q) = \overline{\up_\aR f(q)}$, we calculate:
\[
\begin{tabular}{lll}
$\DLToJSL \JSLToDL f \circ \eta_\aQ (q)$
&
$= \JSLToDL f(\overline{\up_\aQ q})$
\\&
$= (f_*)^{-1}(\overline{\up_\aQ q})$
\\&
$= \overline{f_*^{-1}(\up_\aQ q)}$
\\&
$= \overline{\{ r \in R : q \leq_\aQ f_*(r)\}}$
\\&
$= \overline{\{ r \in R : f(q) \leq_\aR r \}}$
& by adjoint relationship
\\&
$= \overline{\up_\aR f(q)}$
\end{tabular}
\]
as required. 

\smallskip
Next we show that $\epsilon_\aD : (Dn(\aD),\cup,\emptyset,\cap,D) \to \aD$ is a well-defined bounded distributive lattice morphism:
\begin{enumerate}
\item
$\epsilon_\aD(\bot_{\JSLToDL \DLToJSL\aD}) = \epsilon_\aD(\emptyset) = \Land_\aD \overline{\emptyset} \cap M(\aD) = \Land_\aD M(\aD) = \bot_\aD$.

\item
$\epsilon_\aD(\top_{\JSLToDL \DLToJSL\aD}) = \epsilon_\aD(D) = \Land_\aD \emptyset = \top_\aD$.

\item
Regarding meet-preservation:
\[
\begin{tabular}{lll}
$\epsilon_\aD(X_1 \land_{\JSLToDL \DLToJSL\aD} X_2)$
& $= \epsilon_\aD(X_1 \cap X_2) $
\\& $= \Land_\aD \overline{X_1 \cap X_2} \cap M(\aD)$
\\&
$= \Land_\aD (\overline{X_1} \cup \overline{X_2}) \cap M(\aD)$
\\&
$= \Land_\aD (\overline{X_1} \cap M(\aD)) \cup (\overline{X_2} \cap M(\aD))$
\\&
$= \epsilon_\aD(X_1) \land_\aD \epsilon_\aD(X_2)$
\end{tabular}
\]

\item
Regarding join-preservation:
\[
\begin{tabular}{lll}
$\epsilon_\aD(X_1 \cup X_2)$
&
$= \Land_\aD \overline{X_1} \cap \overline{X_2} \cap M(\aD)$
& (A)
\\&
$= \Lor_\aD \{ d \in D : \forall m \in \overline{X_1} \cap \overline{X_2} \cap M(\aD). d \leq_\aD m \}$
& (A')
\\
\\
$\epsilon_\aD(X_1) \lor_\aD \epsilon_\aD(X_2)$
&
$= (\Land_\aD \overline{X_1} \cap M(\aD)) \lor_\aD (\Land_\aD \overline{X_2} \cap  M(\aD))$ & (B)
\\&
$= \Land_\aD \{ m_1 \lor_\aD m_2 :  m_i \in \overline{X_i} \cap M(\aD), \, i = 1,2 \}$ & (B')
\end{tabular}
\]
using distributivity in the final equality. Then $\mathrm{(B)} \leq \mathrm{(A)}$ because $\overline{X_1}\cap\overline{X_2} \cap M(\aD) \subseteq \overline{X_i} \cap M(\aD)$ for $i = 1,2$. To understand why $(\mathrm{A'}) \leq (\mathrm{B'})$, first observe that each $\overline{X_i}$ is up-closed inside $\aD$, as is their intersection. Thus given any elements $m_i \in \overline{X_i} \cap M(\aD)$ (where $i = 1,2$) we have $m_1 \lor_\aD m_2 \in \overline{X_1} \cap \overline{X_2}$. Furthermore any meet-irreducible above $m_1 \lor_\aD m_2$ lies in $\overline{X_1} \cap \overline{X_2} \cap M(\aD)$. Thus any $d \in D$ which lies below every meet-irreducible in $\overline{X_1}\cap\overline{X_2}$ also lies below $m_1 \lor_\aD m_2$, since the latter is the meet of those meet-irreducibles above it.

\end{enumerate}

\smallskip
Concerning the counit-unit equations, we first need to show that:
\[
\xymatrix@=15pt{
(Dn(\aQ),\cup,\emptyset,\cap,Q) \ar@{=}[d] && (Dn(Dn\aQ,\subseteq),\cup,\emptyset) \ar@{=}[d] && (Dn(\aQ),\cup,\emptyset,\cap,Q) \ar@{=}[d]
\\
\JSLToDL \aQ \ar@/_10pt/[rrrr]_-{id_{\JSLToDL\aQ}} \ar[rr]^-{\JSLToDL \eta_\aQ} && \JSLToDL \circ \DLToJSL \circ \JSLToDL \aQ \ar[rr]^-{\epsilon_{\JSLToDL\aQ}} && \JSLToDL\aQ
}
\]

\begin{enumerate}
\item
The first map has action $\JSLToDL \eta_\aQ(X) = (f_*)^{-1}(X) = \{ Y \in Dn(\aQ) : \Land_\aQ \overline{Y} \in X \}$, using the following calculation:
\[
\begin{tabular}{lll}
$(\eta_\aQ)_*(Y)$
&
$= \Lor_\aQ \{q \in Q : \eta_\aQ(q) \subseteq Y \}$
\\&
$= \Lor_\aQ \{ q \in Q : \overline{\up_\aQ q} \subseteq Y \}$
\\&
$= \Lor_\aQ \{ q \in Q : \overline{Y} \subseteq \; \up_\aQ q \}$
\\&
$= \Lor_\aQ \{ q \in Q : q \leq_\aQ \Land_\aQ \overline{Y} \}$
\\&
$= \Land_\aQ \overline{Y}$
\end{tabular}
\]

\item
Regarding the second map, we first observe that:
\[
M(\JSLToDL \aQ) = M(Dn(\aQ),\cup,\emptyset,\cap,Q) = \{ \overline{\up_\aQ q} : q \in Q \}
\]
which holds because:
\begin{enumerate}
\item
If $\overline{\up_\aQ q} = X_1 \cap X_2$ then $\up_\aQ q = \overline{X_1} \cup \overline{X_2}$. Since each $\overline{X_i}$ is $\aQ$-upclosed $\exists i.\up_\aQ q \subseteq \overline{X_i}$, hence $\overline{\up_\aQ q} \subseteq X_i \subseteq \overline{\up_\aQ q}$.
\item
Every downset is an intersection of these sets, since every upset arises as a union of principal upsets.
\end{enumerate}

Then the second map has action:
\[
\begin{tabular}{lll}
$\epsilon_{\JSLToDL \aQ}(S)$
&
$= \Land_{\JSLToDL\aQ} \overline{S} \cap M(\JSLToDL\aQ)$
\\&
$= \bigcap \{ \overline{\up_\aQ q} \in \overline{S} : q \in Q \}$
\end{tabular}
\]

\item
Composing we obtain:
\[
\begin{tabular}{lll}
$\epsilon_{\JSLToDL \aQ} \circ \JSLToDL \eta_\aQ(X)$
&
$= \epsilon_{\JSLToDL \aQ}(\{ Y \in Dn(\aQ) : \Land_\aQ \overline{Y} \in X \})$
\\&
$= \bigcap \{ \overline{\up_\aQ q} : q \in Q, \; \Land_\aQ \overline{\overline{\up_\aQ q}} \nin X \} $
\\&
$= \bigcap \{ \overline{\up_\aQ q} : q \in Q, \; \Land_\aQ \up_\aQ q \nin X \} $
\\&
$= \bigcap \{ \overline{\up_\aQ q} : q \nin X \} $
\\&
$= \bigcap \{ \overline{\up_\aQ q} : q \in \overline{X} \} $
\\&
$= \bigcap \{ \overline{\up_\aQ q} : \; \up_\aQ q \in \overline{X} \} $
& since $\overline{X}$ up-closed
\\&
$= \bigcap \{ \overline{\up_\aQ q} : X \subseteq \overline{\up_\aQ q} \} $
\\&
$= X$
\end{tabular}
\]
Regarding the final step, we already observed that every down-closed set arises as an intersection of sets $\overline{\up_\aQ q}$.
\end{enumerate}

\smallskip
Finally we show the other counit-unit equation holds:
\[
\xymatrix@=15pt{
(D,\lor_\aD,\bot_\aD) \ar@{=}[d] && (Dn(\aD),\cup,\emptyset) \ar@{=}[d] && (D,\lor_\aD,\bot_\aD) \ar@{=}[d]
\\
\DLToJSL \aD \ar[rr]^-{\eta_{\DLToJSL \aD}} \ar@/_10pt/[rrrr]_-{id_{\DLToJSL \aD}} && \DLToJSL \circ \JSLToDL \circ \DLToJSL \aD \ar[rr]^-{\DLToJSL\epsilon_\aD} && \DLToJSL\aD
}
\]
which follows because:
\[
\DLToJSL \epsilon_\aD \circ \eta_{\DLToJSL\aD} (d)
= \epsilon_\aD (\overline{\up_\aD d})
= \Land_\aD \overline{\overline{\up_\aD d}} \cap M(\aD)
= \Land_\aD \up_\aD d \cap M(\aD)
= \Land_\aD \{ m \in M(\aD) : d \leq_\aD m \}
= d
\]
since every element is the meet of those meet-irreducibles above it.
\end{proof}

\begin{definition}[Free boolean algebra on a distributive lattice]
Let $\BAToDL : \BA_f \to \DL_f$ be the forgetful functor where $\BAToDL\aB := (B,\lor_\aB,\bot_\aB,\land_\aB,\top_\aB)$ and $\BAToDL f := f$. Further define:
\[
\DLToBA : \DL_f \to \BA_f
\qquad
\DLToBA \aD := \BPow J(\aD)
\qquad
\dfrac{f : \aD \to \aE}{\DLToBA f := \lambda X.(U_{dm} f)_*^{-1}(X) : \DLToBA J(\aD) \to \DLToBA J(\aE)}
\]
where $U_{dm} f$ takes the underlying join-semilattice morphism between the meet structures i.e.\ 
\[
U_{dm} f : (D,\land_\aD,\top_\aD) \to (E,\land_\aE,\top_\aE)
\] 
so that $\DLToBA f (X) = \{ j \in J(\aE) : \Land_\aD f^{-1}(\up_\aE j) \in X \}$. 
\end{definition}

\begin{lemma}
$\DLToBA$ equals the composite functor:
\[
\DL_f \xto{\Ji^{op}} \Poset_f^{op} \xto{\PosToSet^{op}} \Set_f^{op} \xto{\Pred} \BA_f
\]
and is thus a well-defined functor.
\end{lemma}

\begin{proof}
We have:
\[
\Pred \circ \PosToSet^{op} \circ \Ji^{op} \aD
= \Pred \PosToSet (J(\aD),\leq_{\aD^{\pOp}})
= \Pred J(\aD)
= \BPow J(\aD)
= \DLToBA \aD
\]
Furthermore given any $\DL_f$-morphism $f : \aD \to \aE$ we have:
\[
\begin{tabular}{lll}
$\Pred \circ \PosToSet^{op} \circ \Ji^{op} f$
&
$= \Pred \lambda j \in J(\aE). \Land_\aD f^{-1}(\up_\aE j)$
\\&
$= \Pred \lambda j \in J(\aE). \Lor_{\aD^{\pOp}} f^{-1}(\down_{\aE^{\pOp}} j)$
\\&
$= \Pred \lambda j \in J(\aE). (U_{dm} f)_*(j)$
\\&
$= \lambda X \subseteq J(\aD).(U_{dm} f)_*)^{-1} (X)$
\\&
$= \DLToBA f$
\end{tabular}
\]
\end{proof}

\begin{theorem}[Free boolean algebra on a distributive lattice]
\label{thm:free_ba_on_dl}
\item
$\DLToBA : \DL_f \to \BA_f$ is left adjoint to the forgetful functor $\BAToDL : \BA_f \to \DL_f$ with associated natural transformations:
\[
\begin{tabular}{lll}
$\eta : Id_{\DL_f} \To \BAToDL \circ \DLToBA$
&
$\eta_\aD : \aD \to \DPow J(\aD)$
&
$\eta_\aD(d) := \; J(\aD) \;\cap \down_\aD d$
\\
$\epsilon : \DLToBA \circ \BAToDL \To Id_{\BA_f}$
&
$\epsilon_\aB : \BPow At(\aB) \to \aB$
&
$\epsilon_\aB(S) := \; \Lor_\aB S$
\end{tabular}
\]
\end{theorem}

\begin{proof}
To see that each $\eta_\aD$ is a well-defined bounded distributive lattice morphism (which needn't be an isomorphism), observe that it is a codomain extension of the canonical representation of $\aD$ from Theorem \ref{thm:birkhoff_duality} i.e.\ Birkhoff duality. In order to prove naturality:
\[
\xymatrix@=15pt{
\aD \ar[rr]^f \ar[d]_{\eta_\aD} && \aE \ar[d]^{\eta_\aE}
\\
\DPow J(\aD) \ar[rr]_{\BAToDL \DLToBA f} && \DPow J(\aE)
}
\]
we calculate as follows:
\[
\begin{tabular}{lll}
$\BAToDL \DLToBA f \circ \eta_\aD (d)$
&
$= \DLToBA f(J(\aD)\; \cap \down_\aD d)$
\\&
$= \{ j \in J(\aE) : (U_{dm} f)_*(j) \in \;\down_\aD d \}$
\\&
$= \{ j \in J(\aE) : (U_{dm} f)_*(j) \leq_\aD d \}$
\\&
$= \{ j \in J(\aE) : d \leq_\aE f(d) \}$
& see proof of Theorem \ref{thm:birkhoff_duality}
\\&
$= J(\aE) \; \cap \down_\aE f(d)$
\\&
$= \eta_\aE \circ f(d)$
\end{tabular}
\]

Each $\epsilon_\aB$ is well-defined boolean algebra morphism because it is the inverse of a canonical isomorphism from Theorem \ref{thm:bool_fin_duality} i.e.\ the duality between finite sets and finite boolean algebras. Thus naturality also follows.

Finally we verify the counit-unit equations. Firstly, for any $X \subseteq J(\aD)$ we have:
\[
\begin{tabular}{lll}
$\epsilon_{\DLToBA \aD} \circ \DLToBA \eta_\aD (X)$
&
$= \epsilon_{\DLToBA \aD} \circ (U_{dm} \eta_\aD)_*^{-1}(X)$
\\&
$= \epsilon_{\DLToBA \aD}(\{ j \in J(\DPow J(\aD)) : \Land_\aD \eta_\aD^{-1}(\up_{\DPow J(\aD)} j) \in X \})$
\\&
$= \epsilon_{\DLToBA \aD}(\{ \{j\} : j \in J(\aD), \Land_\aD \{ d \in D : \eta_\aD(d) \ni j\} \in X  \})$
\\&
$= \Lor_{\DLToBA\aD} \{ \{j\} : j \in J(\aD), \Land_\aD \{ d \in D : j \in  (J(\aD) \; \cap \down_\aD d)  \} \in X  \}$
\\&
$= \bigcup \{ \{j\} : j \in J(\aD), j \in X  \}$
\\&
$= X$
\end{tabular}
\]
and finally:
\[
\BAToDL \epsilon_\aB \circ \eta_{\BAToDL \aB} (b)
= \epsilon_\aB (J(\BAToDL \aB)\;\cap \down_{\BAToDL\aB} b)
= \Lor_\aB At(\aB)\;\cap \down_\aB b
= b
\]
\end{proof}

\bibliographystyle{alpha}
\bibliography{bib-2019}

\end{document}